\documentclass[a4paper]{article}

\usepackage{graphicx}

\usepackage[utf8]{inputenc}
\usepackage[british]{babel}
\usepackage{microtype}

\usepackage[pdfencoding=auto,pdfusetitle]{hyperref}
\hypersetup{colorlinks=true, linkcolor=black, citecolor=black, urlcolor=black}

\usepackage{amsmath,amsthm,amssymb,amsfonts,amsbsy}
\usepackage[mathcal]{euscript}
\usepackage{mathrsfs,dsfont,bbm}
\usepackage{fdsymbol,marvosym}
\usepackage{relsize}

\usepackage{mathtools}
\usepackage{thmtools} %
\usepackage[noabbrev,capitalise]{cleveref}

\usepackage{longtable}

\usepackage{graphicx}
\usepackage[all]{xy}

\usepackage[svgnames]{xcolor}
\definecolor{blue(munsell)}{rgb}{0.0, 0.5, 0.69}
\usepackage{svg}

\usepackage{fullpage}
\usepackage[skip=0.5\baselineskip plus 2.5pt]{parskip}
\makeatletter
\def\thm@space@setup{%
  \thm@preskip=\parskip \thm@postskip=0pt
}
\makeatother

\usepackage{verbatim,adjustbox,float,subfig,framed,caption,epigraph}
\usepackage{stackrel}
\usepackage{enumitem}

\usepackage{tikz}
\usepackage{quiver}

\usepackage{tikzit}
\usetikzlibrary{positioning}

\tikzstyle{dot}=[fill=black, draw=black, shape=circle, minimum size=1mm, inner sep=0mm]
\tikzstyle{box}=[fill={rgb,255: red,171; green,205; blue,255}, draw=black, shape=rectangle, minimum width=1.5cm, minimum height=1.0cm, text width=1.5cm, align=center]
\tikzstyle{wide}=[fill={rgb,255: red,171; green,205; blue,255}, draw=black, shape=rectangle, minimum width=3.25cm, minimum height=1.0cm, text width=2.5cm, align=center]

\tikzstyle{line}=[-]
\tikzstyle{dashed_line}=[-, dashed]

\usepackage{circledsteps}

\numberwithin{equation}{section}

\theoremstyle{plain}
\newtheorem{theorem}{Theorem}[section]
\newtheorem{lemma}[theorem]{Lemma}
\newtheorem{proposition}[theorem]{Proposition}
\newtheorem{corollary}[theorem]{Corollary}

\theoremstyle{definition}
\newtheorem{definition}[theorem]{Definition}
\newtheorem{notation}[theorem]{Notation}
\newtheorem{assumption}[theorem]{Assumption}
\newtheorem{remark}[theorem]{Remark}

\newtheorem{example}[theorem]{Example}



\newcommand{\catx}{\mathcal{X}}

\newcommand{\cata}{\mathcal{A}}
\newcommand{\catc}{\mathcal{C}}
\newcommand{\caty}{\mathcal{Y}}

\newcommand{\Cat}{\mathbf{Cat}}
\newcommand{\Mon}{\mathbf{Mon}}

\newcommand{\id}{\mathrm{id}}
\def\op{\mathrm{op}}
\def\co{\mathrm{co}}

\newcommand{\SIGMA}{{\text{\Large $\Sigma$}}}

\newcommand{\SIGMAc}{\stackrel{\bullet}{\text{\Large $\Sigma$}}}
\newcommand{\SIGMAcid}{\stackrel{\bullet}{\text{\Large $\Sigma$}^{\mathrlap{\smash{\id}}}}}
\newcommand{\SIGMAcnamed}[1]{\stackrel{\bullet}{\text{\Large $\Sigma$}^{\mathrlap{\smash{#1}}}}}

\newcommand{\Assoc}{\mathrm{Assoc}}
\newcommand{\Can}{\mathsf{Can}}
\newcommand{\bd}{\mathbf{d}}
\newcommand{\bu}{\mathbf{u}}
\newcommand{\bs}{\mathbf{s}}

\newcommand{\Sg}{\Sigma}
\newcommand{\dS}{\dot{\SIGMA}}

\renewcommand{\epsilon}{\varepsilon}
\renewcommand{\phi}{\varphi}

\usepackage{leftindex}
\usepackage{umoline}
\usepackage{enumitem}

\title{Bicategories of Lax Fractions}
\author{Graham Manuell and Lurdes Sousa\footnote{The authors acknowledge financial support by the Centre for Mathematics of the University of Coimbra (CMUC, https://doi.org/10.54499/UID/00324/2025) under the Portuguese Foundation for Science and Technology (FCT), Grants UIDB/00324/2020, UID/00324/2025 and UID/PRR/00324/2025.}}

\date{February 2026} %

\begin{document}

\maketitle

\begin{abstract} The well-known calculus of fractions of Gabriel and Zisman provides a convenient way to formally invert morphisms in a category. This was generalised to bicategories by Pronk.  We extend these constructions by presenting a calculus of lax fractions for 2-categories that formally turns given morphisms into left adjoint right inverses and given pseudo-commutative squares into Beck--Chevalley squares.
\end{abstract}

\setcounter{tocdepth}{2}
\tableofcontents

\section{Introduction}

The calculus of fractions construction by Gabriel and Zisman \cite{gabriel1967calculus} provides a universal way to formally invert a class of morphisms in a category. In \cite{pronk1996etendues} Pronk presents a 2-categorical generalisation which freely turns morphisms from a given class into equivalences. On the other hand, in \cite{sousa2017calculus} the second author introduces a calculus of \emph{lax} fractions for order-enriched categories in order to formally add right adjoint retractions to morphisms in a class, while also controlling when the Beck--Chevalley condition holds.
In this paper we provide a synthesis of these two approaches by providing a calculus of lax fractions for 2-categories.

One application of our calculus, which we will explore in the paper \cite{manuellsousa2}, is a construction of the bicategory of (strict) monoidal categories and lax monoidal functors from the 2-category of strict monoidal categories and strict monoidal functors by formally adding right adjoints to the morphisms whose underlying functors have fully faithful right adjoints.

The data for the construction involves the original 2-category $\catx$ and a collection $\Sg$ of squares,
commuting up to isomorphism, whose horizontal morphisms are to become left adjoint right inverses,
and which will themselves will become Beck–Chevalley squares. This collection of squares
must satisfy a number of axioms (Definition \ref{def:TheCalculus}), which provide the calculus of lax fractions. In particular, the horizontal morphisms of these squares together with the squares themselves form a subcategory of the arrow category $\catx^{\to\cong}$ (that is, of the category of 1-cells of $\catx$ with morphisms given by pseudo-commutative squares). The resulting bicategory of lax fractions $\catx[\Sg_*]$ will have the
same objects as $\catx$ and 1-cells given by cospans $A\xrightarrow{f}I\xleftarrow{s}B$ where $s$ is a horizontal morphism from $\Sigma$.
 The 2-cells are certain equivalence classes of diagrams of the form
	\begin{equation*}
		\begin{tikzcd}
	A && {I_1} && B \\
	&& X && B \\
	A && {I_2} && B
	\arrow["{f_1}", from=1-1, to=1-3]
	\arrow[equals, from=1-1, to=3-1]
	\arrow[""{name=0, anchor=center, inner sep=0}, "{{{{x_1}}}}"', from=1-3, to=2-3]
	\arrow["\alpha"', between={0.3}{0.7}, Rightarrow, from=1-3, to=3-1]
	\arrow["{r_1}"', from=1-5, to=1-3]
	\arrow[""{name=1, anchor=center, inner sep=0}, equals, from=1-5, to=2-5]
	\arrow["{{x_3}}"', from=2-5, to=2-3]
	\arrow["{f_2}"', from=3-1, to=3-3]
	\arrow[""{name=2, anchor=center, inner sep=0}, "{{{{x_2}}}}", from=3-3, to=2-3]
	\arrow[""{name=3, anchor=center, inner sep=0}, equals, from=3-5, to=2-5]
	\arrow["{r_2}", from=3-5, to=3-3]
	\arrow["{{{\SIGMA^{\delta_1}}}}"{description}, draw=none, from=0, to=1]
	\arrow["{{{\SIGMA^{\delta_2}}}}"{description}, draw=none, from=2, to=3]
\end{tikzcd}
	\end{equation*}
where $\Sg^{\delta_i}$ indicates that the squares are morphisms in $\Sg$ and contain 2-cells $\delta_i\colon x_3\Rightarrow x_ir_i$.
This bicategory is universal in the sense that, for each bicategory $\caty$, there is a biequivalence
between the bicategory of pseudofunctors from $\catx[\Sg_*]$ to $\caty$ and a bicategory whose objects are
pseudofunctors from $\catx$ to $\caty$ that send the horizontal morphisms of $\Sg$ to left adjoint right inverses
and the squares of $\Sg$ to Beck–Chevalley squares.

Roughly speaking, the above diagram may be viewed as a 2-cell
$$\, (\delta_{2*}^{-1}\circ f_2)\cdot (x_{3*}\circ \alpha)\cdot(\delta_{1*}\circ f_1)\colon r_{1*}f_1\Rightarrow r_{2*}f_2\,$$
\noindent where $r_{i*}$ and $x_{3*}$ denote the right adjoint left inverses of $r_i$ and $x_3$, and $\delta_{i*}$ stands for the mate of $\delta_{i}$.

When the subcategory $\Sigma$ of $\catx^{\to \cong}$ is full, our calculus is equivalent to Pronk's calculus and the resulting bicategory of lax fractions takes the objects of $\Sigma$ to equivalences (see \cref{cor:universal_prop}).

To simplify and systematise the proof of the necessary coherence conditions we develop a theory of $\Sigma$-schemes, $\Sigma$-paths and $\Omega$ 2-cells in \cref{sec:bicategory}.

The two appendices cover the more technical aspects of some of the proofs. Appendix~\ref{sec:appendix} presents the proofs of the three propositions of Subsection~\ref{sec:sigma_schemes} concerning $\Sigma$-schemes. Appendix~\ref{sec:appendixB} gives a detailed account of the proof of \cref{thm:universal_prop} using string diagrams.

In the paper \cite{manuellsousa2} we will present several special examples of bicategories of lax fractions and explore the relationship between the calculus of lax fractions, lax idempotent monads and Kan extensions.

We note some related work in the literature. Bourke and Garner \cite{bourkegarner} addressed to the construction of categories by freely adding a `section' to certain morphisms. Our approach here is different, but it would be interesting to study how their work relates to our bicategory of lax fractions.

We also recall the work of Dawson, Par\'e and Pronk \cite{dawson_pare_pronk1} on the problem of freely adding adjoints, although from a perspective entirely different from ours. They construct a 2-category by equipping every morphism of a given category with an adjoint. However, they do not impose any Beck--Chevalley conditions, nor do they provide a calculus of fractions for their construction.

Also compare the result mentioned in \cite[Theorem A.2]{hermida2000representable} stating that if $\catx$ is a 1-category with pullbacks, the bicategory of spans in $\catx$ is the universal bicategory making every morphism of $\catx$ into a left adjoint and every pullback square into a Beck--Chevalley square.

For the basic notions on 2-categories and bicategories we refer to  \cite{johnsonyau2D} and the more informal introduction \cite{lack2009}.

\section{A calculus of lax fractions}\label{sec:calculus}

Let $\catx$ be a 2-category. Denote by  $\catx^{\rightarrow\cong}$ the \textbf{\em arrow category}, whose objects are the 1-cells of $\catx$ and  whose morphisms from $f\colon X\to Y$ to $g\colon Z\to W$ are triples  $(u,v,\delta)\colon f\to g$ where $u\colon X\to Z$ and $v\colon Y\to W$ are 1-cells and $\delta\colon gu\to vf$ is an invertible 2-cell:
\begin{equation}\label{Eq-s}
\begin{tikzcd}
	X & Y \\
	Z & W
	\arrow["f", from=1-1, to=1-2]
	\arrow["g"', from=2-1, to=2-2]
	\arrow["u"', from=1-1, to=2-1]
	\arrow["v", from=1-2, to=2-2]
	\arrow["\delta", shorten <=4pt, shorten >=4pt, Rightarrow, from=2-1, to=1-2]
\end{tikzcd}\, .
\end{equation}
The identity morphisms are just identity 2-cells $(1,1,\id)\colon f\to f$. Composition is vertical composition of squares --- that is, for $\xymatrix{f\ar[r]^{(u,v,\delta)}&g\ar[r]^{(u',v',\delta^{\prime})}&h}$, the composite is given by \newline $(u',v',\delta')\cdot (u,v,\delta)=(u'u,v'v, (v'\circ \delta) \cdot (\delta'\circ u))$.

Let $\Sigma$ be a subcategory of  $\catx^{\rightarrow\cong}$.
In the following we will use a square
\[\begin{tikzcd}
	\bullet & \bullet \\
	\bullet & \bullet
	\arrow["r", from=1-1, to=1-2]
	\arrow["s"', from=2-1, to=2-2]
	\arrow[""{name=0, anchor=center, inner sep=0}, "u"', from=1-1, to=2-1]
	\arrow[""{name=1, anchor=center, inner sep=0}, "v", from=1-2, to=2-2]
	\arrow["{\text{\Large $\Sigma^{\delta}$}}"{description}, draw=none, from=0, to=1]
\end{tikzcd}\, ,\]
 with  $\Sigma^{\delta}$  in its center, to indicate that $(u,v, \delta)\colon r\to s$ is a morphism in $\Sigma$. If no danger of confusion exists, we may use $\Sigma$ alone without the superscript. We call these squares
 \textbf{\em $\Sigma$-squares}. Sometimes we reverse or flip them, but the horizontal arrows always refer to objects of $\Sigma$, and the vertical ones to the 1-cell part of the represented morphism.

 In this way, our calculus of lax fractions becomes essentially a calculus of $\Sigma$-squares.
 As we will see in Example~\ref{exas:calculus:laris}, the rules of this calculus are satisfied by left adjoint left inverses and Beck-Chevalley squares (see \cref{rem:lari}).

\begin{definition}\label{def:TheCalculus}{\textbf{A calculus of lax fractions.}}
	We say that a subcategory $\Sigma$ of $\catx^{\rightarrow\cong}$ admits a \textbf{\em calculus of left lax fractions}\footnote{In \cite{sousa2017calculus}, we say that $\Sigma$  admits a \textbf{\em left calculus of lax fractions}.} provided that the following conditions are statisfied.
	\begin{enumerate}[label={(\arabic*)}]
		\item \textbf{Identity.} Every identity 1-cell of $\catx$ is an object of $\Sigma$, and for every $\Sigma$-object $s\colon X\to Y$  we have the $\Sigma$-square
		\[\begin{tikzcd}
			X & X \\
			X & Y
			\arrow["{1_X}", from=1-1, to=1-2]
			\arrow[""{name=0, anchor=center, inner sep=0}, "{1_X}"', from=1-1, to=2-1]
			\arrow["s"', from=2-1, to=2-2]
			\arrow[""{name=1, anchor=center, inner sep=0}, "s", from=1-2, to=2-2]
			\arrow["{\SIGMA^{\id}}"{description}, draw=none, from=0, to=1]
		\end{tikzcd}\, .\]
		\item  \textbf{Repletion.}
		
		(a) \textbf{Vertical Repletion.} If $\delta\colon r\Rightarrow s$ is an invertible 2-cell and $r\colon X\to Y$ is in $\Sg$, then $s$ also belongs to $\Sigma$ and we have the $\Sigma$-square
		\[\begin{tikzcd}
			X & Y \\
			X & Y
			\arrow["s", from=1-1, to=1-2]
			\arrow[""{name=0, anchor=center, inner sep=0}, "{1_X}"', from=1-1, to=2-1]
			\arrow["r"', from=2-1, to=2-2]
			\arrow[""{name=1, anchor=center, inner sep=0}, "1_Y", from=1-2, to=2-2]
			\arrow["{\SIGMA^{\delta}}"{description}, draw=none, from=0, to=1]
		\end{tikzcd}\; .\]

		(b) \textbf{Horizontal Repletion.}  For every pair of morphisms $f,g\colon X\to Y$ and every invertible 2-cell $\gamma \colon f\Rightarrow g$, we have the $\Sigma$-square
		\[\begin{tikzcd}
			X & X \\
			Y & Y
			\arrow["1_X", from=1-1, to=1-2]
			\arrow[""{name=0, anchor=center, inner sep=0}, "{f}"', from=1-1, to=2-1]
			\arrow["1_Y"', from=2-1, to=2-2]
			\arrow[""{name=1, anchor=center, inner sep=0}, "g", from=1-2, to=2-2]
			\arrow["{\SIGMA^{\gamma}}"{description}, draw=none, from=0, to=1]
		\end{tikzcd}\; .\]
		
		\item \textbf{Composition.} If in the diagram
		\[\begin{tikzcd}
			\bullet & \bullet & \bullet \\
			\bullet & \bullet & \bullet
			\arrow["r", from=1-1, to=1-2]
			\arrow["s", from=1-2, to=1-3]
			\arrow["{r'}"', from=2-1, to=2-2]
			\arrow["{s'}"', from=2-2, to=2-3]
			\arrow[""{name=0, anchor=center, inner sep=0}, "f"', from=1-1, to=2-1]
			\arrow[""{name=1, anchor=center, inner sep=0}, "g", from=1-2, to=2-2]
			\arrow[""{name=2, anchor=center, inner sep=0}, "h", from=1-3, to=2-3]
			\arrow["{\Circled{1}}"{description}, draw=none, from=0, to=1]
			\arrow["{\Circled{2}}"{description}, draw=none, from=1, to=2]
		\end{tikzcd}\]
		\Circled{1} and \Circled{2} are both $\Sigma$-squares, then  the pasting diagram  \Circled{1}$+$\Circled{2} is also a $\Sigma$-square.
		\item \textbf{Square.} For every span
		$\begin{tikzcd}
			\bullet & \bullet & \bullet
			\arrow["f"', from=1-2, to=1-1]
			\arrow["s", from=1-2, to=1-3]
		\end{tikzcd}$   with $s\in \Sigma$, there is a $\Sigma$-square of the form
		\[\begin{tikzcd}
			\bullet & \bullet \\
			\bullet & \bullet
			\arrow["s", from=1-1, to=1-2]
			\arrow[""{name=0, anchor=center, inner sep=0}, "f"', from=1-1, to=2-1]
			\arrow[""{name=1, anchor=center, inner sep=0}, "{f'}", from=1-2, to=2-2]
			\arrow["{s'}"', from=2-1, to=2-2]
			\arrow["{\text{\Large $\Sigma$}}"{description}, draw=none, from=0, to=1]
		\end{tikzcd}\, .\]
		\item \textbf{Equi-insertion.}
		For every $\Sigma$-square and every (not-necessarily invertible) 2-cell $\alpha\colon f'r\Rightarrow gr$ as  in the diagram
		\[\begin{tikzcd}
	&& B \\
	A & B \\
	C & D
	\arrow["g", curve={height=-12pt}, from=1-3, to=3-2]
	\arrow["r", curve={height=-6pt}, from=2-1, to=1-3]
	\arrow["r", from=2-1, to=2-2]
	\arrow[""{name=0, anchor=center, inner sep=0}, "f"', from=2-1, to=3-1]
	\arrow["\alpha", shift right, shorten <=4pt, shorten >=4pt, Rightarrow, from=2-2, to=1-3]
	\arrow[""{name=1, anchor=center, inner sep=0}, "{{f'}}", from=2-2, to=3-2]
	\arrow["s"', from=3-1, to=3-2]
	\arrow["{\SIGMA^{\delta}}"{description}, draw=none, from=0, to=1]
\end{tikzcd}\]
		there is a 1-cell $d\colon D\to E$ and a 2-cell $\alpha'\colon df'\Rightarrow dg$ such that
		\[\begin{tikzcd}
			C & D \\
			C & E
			\arrow["s", from=1-1, to=1-2]
			\arrow[""{name=0, anchor=center, inner sep=0}, "d", from=1-2, to=2-2]
			\arrow["ds"', from=2-1, to=2-2]
			\arrow[""{name=1, anchor=center, inner sep=0}, Rightarrow, no head, from=1-1, to=2-1]
			\arrow["{\SIGMA^{\id}}"{description}, draw=none, from=1, to=0]
		\end{tikzcd}\] and $d\alpha=\alpha'r$.
		\item \textbf{Equification.} For every $\Sigma$-square and two 2-cells as in the diagram
		\[\begin{tikzcd}
	A && B \\
	\\
	C && D
	\arrow["r", from=1-1, to=1-3]
	\arrow[""{name=0, anchor=center, inner sep=0}, "f"', from=1-1, to=3-1]
	\arrow[""{name=1, anchor=center, inner sep=0}, "{{f'}}"', curve={height=18pt}, from=1-3, to=3-3]
	\arrow[""{name=2, anchor=center, inner sep=0}, "g", curve={height=-18pt}, from=1-3, to=3-3]
	\arrow["s"', from=3-1, to=3-3]
	\arrow["{\SIGMA^{\delta}}"{description, pos=0.6}, draw=none, from=0, to=1]
	\arrow["\beta", shift right=4, shorten <=7pt, shorten >=7pt, Rightarrow, from=1, to=2]
	\arrow["\alpha", shift left=4, shorten <=7pt, shorten >=7pt, Rightarrow, from=1, to=2]
\end{tikzcd}\]
		with $\alpha r=\beta r$
(or equivalently, $(\alpha\circ r)\cdot \delta=(\beta\circ r)\cdot \delta$), %
		there is a 1-cell $d\colon D\to E$ such that
		\[\begin{tikzcd}
			C & D \\
			C & E
			\arrow["s", from=1-1, to=1-2]
			\arrow[""{name=0, anchor=center, inner sep=0}, "d", from=1-2, to=2-2]
			\arrow["ds"', from=2-1, to=2-2]
			\arrow[""{name=1, anchor=center, inner sep=0}, Rightarrow, no head, from=1-1, to=2-1]
			\arrow["{\SIGMA^{\id}}"{description}, draw=none, from=1, to=0]
		\end{tikzcd}\]
		and $d\alpha=d\beta$.
	\end{enumerate}

Dually, given a subcategory $\Sigma$ of the arrow category $\catx^{\to\cong}$, we say that it admits a \textbf{\em calculus of right lax fractions} if the corresponding subcategory of the arrow category of $\catx^\op$ admits a calculus of left lax fractions.
We say such a $\Sigma$ admits a \textbf{\em calculus of left colax fractions} or a \textbf{\em calculus of right colax fractions} if its respective counterpart in $\catx^\co$ or $\catx^{\co\op}$ admits a calculus of left lax fractions.
\end{definition}

For the above four types of calculus, our bicategory of lax fractions construction will transform the morphisms of $\Sigma$ into left adjoint right inverses, left adjoint left inverses, right adjoint right inverses and right adjoint left inverses, respectively.

\begin{definition}\label{rem:lari} 1. Recall that a 1-cell $f\colon A\to B$ is said to be a \textbf{\em lari} (or \textbf{\em left adjoint right inverse}), if there is an adjunction

	\[\adjustbox{scale=0.9}{\begin{tikzcd}
		A && B
		\arrow[""{name=0, anchor=center, inner sep=0}, "f", curve={height=-18pt}, from=1-1, to=1-3]
		\arrow[""{name=1, anchor=center, inner sep=0}, "g", curve={height=-18pt}, from=1-3, to=1-1]
		\arrow["{\bot}", shift right=2, draw=none, from=0, to=1]
	\end{tikzcd}}\]
	with unit $\eta$ and counit $\epsilon$
	such that $\eta$ is a invertible.
 Thus, in this case, the triangle identities defining the adjunction
	\[\begin{tikzcd}
	A & B && B & A \\
	&& {\; \; \quad\text{ and}} \\
	A & B && B & A
	\arrow["f", from=1-1, to=1-2]
	\arrow[""{name=0, anchor=center, inner sep=0}, equals, from=1-1, to=3-1]
	\arrow["g"{description}, from=1-2, to=3-1]
	\arrow[""{name=1, anchor=center, inner sep=0}, "{{\;\; =\; \; \id_f}}", equals, from=1-2, to=3-2]
	\arrow["g", from=1-4, to=1-5]
	\arrow[""{name=2, anchor=center, inner sep=0}, equals, from=1-4, to=3-4]
	\arrow["f"{description}, from=1-5, to=3-4]
	\arrow[""{name=3, anchor=center, inner sep=0}, "{{\; \; =\; \;  \id_g}}", equals, from=1-5, to=3-5]
	\arrow["f"', from=3-1, to=3-2]
	\arrow["g", from=3-4, to=3-5]
	\arrow["\eta"{description}, shorten <=6pt, shorten >=3pt, Rightarrow, from=0, to=1-2]
	\arrow["\epsilon"{description}, shorten <=3pt, shorten >=6pt, Rightarrow, from=1-5, to=2]
	\arrow["\eta"{description}, shorten <=6pt, shorten >=3pt, Rightarrow, from=3, to=3-4]
	\arrow["\epsilon"{description}, shorten <=6pt, shorten >=6pt, Rightarrow, from=3-1, to=1]
\end{tikzcd}\]
	lead to
	\[f\circ \eta^{-1} = \epsilon \circ f\; \; \text{ and }\; \; \eta^{-1}\circ g = g\circ \epsilon.\]
	In the following, we use the notation $f_{\ast}$ to denote a right adjoint to $f$.

		2. A  diagram  of the form
	\[\begin{tikzcd}
		\bullet & \bullet \\
		\bullet & \bullet
		\arrow["r", from=1-1, to=1-2]
		\arrow["s"', from=2-1, to=2-2]
		\arrow["f"', from=1-1, to=2-1]
		\arrow["g", from=1-2, to=2-2]
		\arrow["\delta", shorten <=6pt, shorten >=6pt, Rightarrow, from=2-1, to=1-2]
	\end{tikzcd}\]
	with $r$ and $s$ lari 1-cells and $\delta$ an invertible 2-cell
	has the \textbf{\em Beck--Chevalley condition} if its mate is an isomorphism --- that is,
the 2-cell $\delta_*\colon fr_\ast \Rightarrow s_\ast g$ given by
	\[\begin{tikzcd}
		\bullet & \bullet \\
		\bullet & \bullet \rlap{\,,}
		\arrow["f"', from=1-1, to=2-1]
		\arrow["\delta_{\ast}", shorten <=4pt, shorten >=4pt, Rightarrow, from=1-1, to=2-2]
		\arrow["{{r_{\ast}}}"', from=1-2, to=1-1]
		\arrow["g", from=1-2, to=2-2]
		\arrow["{{s_{\ast}}}", from=2-2, to=2-1]
	\end{tikzcd}
\quad
=
\quad
\begin{tikzcd}
	& \bullet \\
	\bullet & \bullet \\
	\bullet & {\bullet } \\
	\bullet
	\arrow[""{name=0, anchor=center, inner sep=0}, "{r_{\ast}}"', from=1-2, to=2-1]
	\arrow[""{name=1, anchor=center, inner sep=0}, equals, from=1-2, to=2-2]
	\arrow["r", from=2-1, to=2-2]
	\arrow["f"', from=2-1, to=3-1]
	\arrow["g", from=2-2, to=3-2]
	\arrow["\delta", between={0.3}{0.7}, Rightarrow, from=3-1, to=2-2]
	\arrow["s", from=3-1, to=3-2]
	\arrow[""{name=2, anchor=center, inner sep=0}, equals, from=3-1, to=4-1]
	\arrow[""{name=3, anchor=center, inner sep=0}, "{s_{\ast}}", from=3-2, to=4-1]
	\arrow["{\epsilon^r}", shift right=3, between={0.2}{0.8}, Rightarrow, from=0, to=1]
	\arrow["{\eta^s}", shift left, between={0.2}{0.8}, Rightarrow, from=2, to=3]
\end{tikzcd}\]
is  invertible.
\end{definition}

An important motivating example for the axioms above is given by the class of laris themselves together with Beck--Chevalley squares. We will explore this in more detail in Example~\ref{exas:calculus:laris}, but for now it suffices to mention that they do satisfy all of the conditions above. In particular, the square from the Identity axiom is a Beck--Chevalley square precisely when $s$ is a lari.

One can also consider this calculus of left lax fractions in the more general context of bicategories. We opt for starting with a 2-category to simplify the exposition, and also because our examples are all 2-categories. %

\begin{remark}
The Horizontal Repletion axiom can actually be omitted from \cref{def:TheCalculus}, but including this natural rule simplifies the construction. %
Also note that the Composition axiom together with Horizontal Repletion allows us to view $\Sigma$ as a double category (indeed, a sub-double category of the double category of quintents).
\end{remark}

\begin{notation}\label{nota:comp-squares} For composable $\Sigma$-squares
\[\adjustbox{scale=0.90}{\begin{tikzcd}
	\bullet & \bullet && \bullet & \bullet & \bullet \\
	\bullet & \bullet && \bullet & \bullet & \bullet \\
	\bullet & \bullet
	\arrow[from=1-1, to=1-2]
	\arrow[""{name=0, anchor=center, inner sep=0}, from=1-1, to=2-1]
	\arrow[""{name=1, anchor=center, inner sep=0}, from=1-2, to=2-2]
	\arrow[from=1-4, to=1-5]
	\arrow[""{name=2, anchor=center, inner sep=0}, from=1-4, to=2-4]
	\arrow[from=1-5, to=1-6]
	\arrow[""{name=3, anchor=center, inner sep=0}, from=1-5, to=2-5]
	\arrow[""{name=4, anchor=center, inner sep=0}, from=1-6, to=2-6]
	\arrow[from=2-1, to=2-2]
	\arrow[""{name=5, anchor=center, inner sep=0}, from=2-1, to=3-1]
	\arrow["{\text{\normalsize and}}"{marking, allow upside down}, draw=none, from=2-2, to=2-4]
	\arrow[""{name=6, anchor=center, inner sep=0}, from=2-2, to=3-2]
	\arrow[from=2-4, to=2-5]
	\arrow[from=2-5, to=2-6]
	\arrow[from=3-1, to=3-2]
	\arrow["{{\SIGMA^{\alpha}}}"{description}, draw=none, from=0, to=1]
	\arrow["{{\SIGMA^{\alpha}}}"{description}, draw=none, from=2, to=3]
	\arrow["{{\SIGMA^{\gamma}}}"{description}, draw=none, from=3, to=4]
	\arrow["{{\SIGMA^{\beta}}}"{description}, draw=none, from=5, to=6]
\end{tikzcd}}\]
	we sometimes refer to the $\Sigma$-squares which result by vertical and horizontal composition as
	$\Sigma^{\beta\odot\alpha}$  and $\Sigma^{\alpha\oplus \gamma}$, respectively. %
\end{notation}

When $\Sigma$ is a full subcategory of $\catx^{\to \cong}$, the calculus of  left lax fractions yields an interesting special case. Indeed, as we show in Example~\ref{exas:calculus:pronks}, it is equivalent to Pronk's bicalculus of fractions (see also \cref{cor:universal_prop}). Moreover, in this case the  rules of the calculus simplify as follows (see  \cref{lem:Full-Sigma}):

\begin{definition}\label{def:Full-Sigma}{\textbf{Full calculus of lax fractions.}} Let $\Sigma$ be a class of morphisms of $\catx$. Then $\Sigma$ is said to admit a {\em \textbf{full calculus of left lax fractions}} if it fulfils the following rules:
\begin{enumerate}
\item[(Id)] All identity 1-cells belong to $\Sg$.
\item[(Rep)] For  every invertible 2-cell $\delta\colon r\Rightarrow s$ with $r\in \Sg$, also  $s$ belongs to $\Sg$.
    \item[(Comp)]  $\Sg$ is closed under composition.
    \item[(Sq)]  For every span
		$\begin{tikzcd}
			X & Y & Z
			\arrow["f"', from=1-2, to=1-1]
			\arrow["s", from=1-2, to=1-3]
		\end{tikzcd}$   with $s\in \Sigma$, there are 1-cells $s'\colon Z \to W$ and $f'\colon Y\to W$ with $s'$ in $\Sg$, and an invertible  2-cell
\[\begin{tikzcd}
	\bullet & \bullet \\
	\bullet & \bullet
	\arrow["s", from=1-1, to=1-2]
	\arrow["f"', from=1-1, to=2-1]
	\arrow["{{f'}}", from=1-2, to=2-2]
	\arrow["\delta", shorten <=4pt, shorten >=4pt, Rightarrow, from=2-1, to=1-2]
	\arrow["{{s'}}"', from=2-1, to=2-2]
\end{tikzcd}\, .\]
\item[(Eq1)] For every 2-cell
\(\begin{tikzcd}
	& Y \\
	X && Z \\
	& Y
	\arrow["{g_1}", from=1-2, to=2-3]
	\arrow["\alpha", shorten <=10pt, shorten >=10pt, Rightarrow, from=1-2, to=3-2]
	\arrow["r", from=2-1, to=1-2]
	\arrow["r"', from=2-1, to=3-2]
	\arrow["{g_2}"', from=3-2, to=2-3]
\end{tikzcd}\)
with $r\in \Sg$, there is $q\colon Z\to W$ in $\Sg$ and a 2-cell $\alpha'\colon qg_1\Rightarrow qg_2$ such that $\alpha'\circ r=q\circ \alpha$.
\item[(Eq2)] For every diagram
\(\begin{tikzcd}
	X & Y && Z
	\arrow["r", from=1-1, to=1-2]
	\arrow[""{name=0, anchor=center, inner sep=0}, "{g_1}", curve={height=-24pt}, from=1-2, to=1-4]
	\arrow[""{name=1, anchor=center, inner sep=0}, "{g_2}"', curve={height=24pt}, from=1-2, to=1-4]
	\arrow["\alpha"', shift right=3, shorten <=6pt, shorten >=6pt, Rightarrow, from=0, to=1]
	\arrow["\beta", shift left=3, shorten <=6pt, shorten >=6pt, Rightarrow, from=0, to=1]
\end{tikzcd}\)
with $r\in \Sg$ and $\alpha\circ r=\beta\circ r$, there is $q\colon Z\to W$ in $\Sg$  such that $q\circ \alpha=q\circ \beta$.
\end{enumerate}
\end{definition}

\begin{lemma}\label{lem:Full-Sigma} Let $\Sigma$ be a class of morphisms of $\catx$. Then $\Sigma$ admits a calculus of left lax  fractions, seen as a full subcategory of $\catx^{\to \cong}$,  if and only if it fulfils the  rules of the full calculus of left lax  fractions.
\end{lemma}

\begin{proof} The equivalence for most of the conditions is immediate. For obtaining (Eq1) from Equi-insertion, observe that
\(\begin{tikzcd}
	{} & {} \\
	{} & {}
	\arrow["r", from=1-1, to=1-2]
	\arrow["{g_1r}"', from=1-1, to=2-1]
	\arrow["{g_1}", from=1-2, to=2-2]
	\arrow["{=}"{description}, draw=none, from=2-1, to=1-2]
	\arrow[equals, from=2-1, to=2-2]
\end{tikzcd}\)
is a $\Sg$-square. Conversely, from (Eq1) and (Comp) we easily obtain Equi-insertion. A similar analysis works for (Eq2).
\end{proof}

We now give some examples of classes $\Sigma$ that admit calculi of lax fractions.
\begin{example}\label{exas:calculus}  \begin{enumerate}[ref=\thetheorem.\arabic*]
	\item\label{exas:calculus:laris} \textbf{Laris.}
	For a 2-category $\catx$, let $\Sigma$ be the subcategory of the  category  $\catx^{\to\cong}$ given by all laris and those morphisms of $\catx^{\to\cong}$ forming Beck--Chevalley squares.

	This subcategory $\Sigma$ of laris admits a calculus of left lax fractions. Indeed, Identity, Repletion  and Composition are clear. For Square, observe that, since $s$ is a lari, the square

\[\begin{tikzcd}
	X & Y \\
	Z & Z
	\arrow["s", from=1-1, to=1-2]
	\arrow["f"', from=1-1, to=2-1]
	\arrow["{{fs_{\ast}}}", from=1-2, to=2-2]
	\arrow["{f\circ \eta}", shorten <=4pt, shorten >=4pt, Rightarrow, from=2-1, to=1-2]
	\arrow[Rightarrow, no head, from=2-1, to=2-2]
\end{tikzcd}\]

	satisfies the Beck--Chevalley condition. For Equi-insertion, suppose we have

	
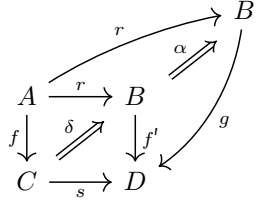
\[\begin{tikzcd}
		&& B \\
		A & B \\
		C & D
		\arrow["r", from=2-1, to=2-2]
		\arrow[""{name=0, anchor=center, inner sep=0}, "f"', from=2-1, to=3-1]
		\arrow[""{name=1, anchor=center, inner sep=0}, "{f'}", from=2-2, to=3-2]
		\arrow["s"', from=3-1, to=3-2]
		\arrow["g", curve={height=-12pt}, from=1-3, to=3-2]
		\arrow["r", curve={height=-6pt}, from=2-1, to=1-3]
		\arrow["\alpha", shift right=1, shorten <=4pt, shorten >=4pt, Rightarrow, from=2-2, to=1-3]
		\arrow["\delta", shorten <=4pt, shorten >=4pt, Rightarrow, from=3-1, to=2-2]
	\end{tikzcd}\]

	where $\delta$ is invertible and forms a Beck--Chevalley square. Let $\delta_*$ denote the mate of $\delta$. The Equi-insertion condition is fulfilled by putting $d = s_*$ and defining $\alpha'$ to be the composite 2-morphism $s_* f' \xrightarrow{\delta_*^{-1}} fr_* \xrightarrow{(\alpha\delta)_*} s_* g$ where $(\alpha\delta)_*$ denotes the mate of the composite $\alpha \cdot \delta$.
Note that
	\[\begin{tikzcd}
		C & D \\
		C & E
		\arrow["s", from=1-1, to=1-2]
		\arrow[""{name=0, anchor=center, inner sep=0}, "s_*", from=1-2, to=2-2]
		\arrow["s_*s"', from=2-1, to=2-2]
		\arrow[""{name=1, anchor=center, inner sep=0}, Rightarrow, no head, from=1-1, to=2-1]
		\arrow["{\id}", shorten <=4pt, shorten >=4pt, Rightarrow, from=2-1, to=1-2]
	\end{tikzcd}\]
	is indeed a Beck--Chevalley square. %

	To see that $s_* \alpha = \alpha' r$, it suffices to show $s_* \alpha \cdot \delta_*r \cdot f \eta^r = \alpha' r \cdot \delta_*r \cdot f \eta^r$,  since $\delta_*$ and $\eta^r$ are invertible. But $ \alpha' r \cdot \delta_*r \cdot f \eta^r= (\alpha\delta)_*r \cdot f \eta^r$ by the definition of $\alpha'$. So we can just show that $s_* \alpha \cdot \delta_*r \cdot f \eta^r=(\alpha\delta)_*r \cdot f \eta^r$.
Expanding the definition of the mates we have
\begin{align*}
 & s_* \alpha \cdot s_\ast f' \epsilon^r r \cdot s_\ast \delta r_\ast r \cdot \eta^s f r_\ast r \cdot f \eta^r \\
 = {} & s_* \alpha \cdot s_\ast f' \epsilon^r r \cdot s_\ast \delta r_\ast r \cdot s_* s f \eta^r \cdot \eta^s f \\
 = {} & s_* \alpha \cdot s_\ast f' \epsilon^r r \cdot s_* f' r \eta^r \cdot s_\ast \delta \cdot \eta^s f \\
 = {} & s_* \alpha \cdot s_\ast \delta \cdot \eta^s f \\
\end{align*}
on the left-hand side and
\begin{align*}
 & s_\ast g \epsilon^r r \cdot s_\ast \alpha r_\ast r \cdot s_\ast \delta r_\ast r \cdot \eta^s f r_\ast r \cdot f \eta^r \\
 = {} & s_\ast g \epsilon^r r \cdot s_\ast \alpha r_\ast r \cdot s_* f' r \eta^r \cdot s_\ast \delta \cdot \eta^s f \\
 = {} & s_\ast g \epsilon^r r \cdot s_* g r \eta^r \cdot s_\ast \alpha \cdot s_\ast \delta \cdot \eta^s f \\
 = {} & s_\ast \alpha \cdot s_\ast \delta \cdot \eta^s f
\end{align*}
on the right-hand side.
Thus, they are indeed equal.

For Equification, we show that again we may take $d=s_{\ast}$. Let $\delta_*:fr_{\ast}\Rightarrow s_{\ast}f'$ be the mate of $\delta$. Let us compose $\alpha$ and $\beta$ with $\epsilon^r\colon rr_{\ast}\Rightarrow 1_B$ and $\delta_*$ as in the following diagram.

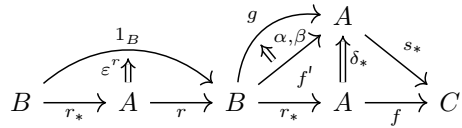
\[\begin{tikzcd}
	&&& A \\
	B & A & B & A & C
	\arrow["{{{s_{\ast}}}}", from=1-4, to=2-5]
	\arrow["{{{{{r_{\ast}}}}}}"', from=2-1, to=2-2]
	\arrow[""{name=0, anchor=center, inner sep=0}, "{{{{{1_B}}}}}", curve={height=-24pt}, from=2-1, to=2-3]
	\arrow["r"', from=2-2, to=2-3]
	\arrow[""{name=1, anchor=center, inner sep=0}, "{{{f'}}}"', from=2-3, to=1-4]
	\arrow[""{name=2, anchor=center, inner sep=0}, "g", curve={height=-18pt}, from=2-3, to=1-4]
	\arrow["{{{{{r_{\ast}}}}}}"', from=2-3, to=2-4]
	\arrow["{{\delta_*}}"', Rightarrow, from=2-4, to=1-4]
	\arrow["f"', from=2-4, to=2-5]
	\arrow["{{\epsilon^r}}", between={0}{0.8}, Rightarrow, from=2-2, to=0]
	\arrow["{\alpha,\beta}"', shift left=2, between={0.2}{0.8}, Rightarrow, from=1, to=2]
\end{tikzcd}\]
	As $s_{\ast}\circ \alpha \circ r=s_{\ast}\circ \beta \circ r$, we have $s_{\ast}\circ \alpha \circ \epsilon^r=s_{\ast}\circ \beta \circ \epsilon^r$, and hence we obtain that $(s_{\ast}\circ \alpha)\cdot (f\circ r_{\ast}\circ \epsilon^r)\cdot (\delta_*\circ r\circ r_{\ast})
	=(s_{\ast}\circ \beta)\cdot (f\circ r_{\ast}\circ \epsilon^r)\cdot (\delta_*\circ r\circ r_{\ast})$. Since $(f\circ r_{\ast}\circ \epsilon^r)\cdot (\delta_*\circ r\circ r_{\ast})$ is invertible, we conclude $s_{\ast}\circ \alpha=s_{\ast}\circ \beta$, as desired.

\item \textbf{The ordinary calculus of left fractions.} Given a 1-category $\catx$ and a class $\Sigma$ of morphisms of $\catx$, let us look at $\catx$ as a 2-category with trivial 2-cells (in particular, laris are just isomorphisms) and at $\Sigma$ as a full subcategory of the arrow category  $\catx^{\to}$. Then, for $\Sigma$ to admit a calculus of left lax fractions just means to admit a calculus of left fractions in the classical sense (\cite{gabriel1967calculus}).

\item\label{exas:calculus:pronks} \textbf{Pronk's calculus.} %
In \cite{pronk1996etendues}, Dorette Pronk introduced a {\em  bicalculus of right fractions} for a class $\Sigma$ of 1-cells generalising the classical calculus to bicategories (see also \cite{pronkscull2022}). With this calculus, the localization process yields a bicategory where morphisms in $\Sigma$ become equivalences.  Here we show that in a 2-category $\catx$, a class $\Sigma$ of 1-cells admits a  bicalculus of left fractions, in the sense of Pronk, if and only if $\Sigma$, viewed as a full subcategory of $\catx^{\to \cong}$, admits a  calculus of left lax fractions ---
that is, $\Sigma$ admits a full calculus of left lax fractions as described in \cref{def:Full-Sigma}.

Comparing our calculus with the {\em  bicalculus of left fractions} of Pronk, we see that rules (Id), (Rep), (Comp) and (Sq) are common to Pronk's calculus (except that in (Id) we only ask for identities to belong to $\Sg$ instead of all equivalences). The remaining rule of the  bicalculus of left fractions, simplified according to \cite{pronkscull2022}, states that:

\begin{enumerate}
  \item[(PR)] Given $X\xrightarrow{r}Y$ in $\Sg$, 1-cells $g_1,g_2\colon Y\to Z$ and a 2-cell $\alpha\colon g_1\circ r\Rightarrow  g_2\circ r$, we have that:

  (i) There is $q\colon Z\to Q$ in $\Sg$ and $\alpha'\colon q\circ g_1\Rightarrow  q\circ g_2$ with $\alpha'\circ r=q\circ \alpha$.

  (ii) Suppose we have $q'$ and $\alpha''$ satisfying the same conditions as $q$ and $\alpha'$  in (i) --- that is, $q'\colon Z\to Q'$ belongs to $\Sg$ and $\alpha''\colon q'\circ g_1\Rightarrow q'\circ g_2$ with $\alpha''\circ r=q'\circ \alpha$. Then there are 1-cells $u$ and $u'$ and an invertible 2-cell $\epsilon\colon u\circ q \Rightarrow  u'\circ q'$ such that $u\circ q$ and $u'\circ q'$ belong to $\Sigma$ and  $(\epsilon\circ g_2)\cdot (u\circ \alpha')=(u'\circ \alpha'')\cdot (\epsilon \circ g_1)$.
\end{enumerate}

The rule (PR)(i) is just (Eq1). Thus, in order to show that Pronk's bicalculus is equivalent to our calculus of lax fractions, we only need to prove that, in the presence of the rules (Id), (Rep), (Comp), (Sq) and (Eq1), we have that (PR)(ii)  $\Longleftrightarrow$ (Eq2).

 (PR)(ii) $\Longrightarrow$ (Eq2). Given a 1-cell $r\in \Sigma$ and  2-cells $\alpha, \beta\colon g_1\Rightarrow g_2$  as in (Eq2), put $\gamma=\alpha\circ r= \beta\circ r$. Then,
 both  $(\alpha, 1_Z)$ and $(\beta, 1_Z)$ play the same role as $(\gamma',q)$ determined by (Eq1).
 Consequently, by (PR)(ii), there are $u,u'\colon  Z\to W$ in $\Sg$ and an invertible 2-cell $\epsilon\colon u\to u'$ such that $(\epsilon \circ g_2) \cdot (u \circ \alpha) = (u' \circ \beta) \cdot (\epsilon \circ g_1)$, and hence $u\circ \alpha = u\circ \beta$ by interchange.

(Eq2) $\Longrightarrow$ (PR)(ii). Let $(q, \alpha')$ be as in (Eq1) and let another pair $(q', \alpha'')$ play the same role. Apply (Sq) to $q'$ and $q$ obtaining
\[\begin{tikzcd}
	{\bullet} & {\bullet} \\
	{\bullet} & {\bullet} \rlap{\,.}
	\arrow["q'", from=1-1, to=1-2]
	\arrow["{q}"', from=1-1, to=2-1]
	\arrow["p'", from=1-2, to=2-2]
	\arrow["\theta", between={0.2}{0.8}, Rightarrow, from=2-1, to=1-2]
	\arrow["{p}"', from=2-1, to=2-2]
\end{tikzcd}\]
 Then
 \[[(\theta\circ g_2)\cdot (p\circ \alpha')]\circ r = \theta\circ \alpha = [(p'\circ \alpha'')\cdot (\theta \circ g_1)]\circ r.\]

  By (Eq2) there is a 1-cell $d$ belonging to $\Sg$ such that $d\circ [(\theta\circ g_2)\cdot (p\circ \alpha')]= d\circ [(p'\circ \alpha'')\cdot (\theta \circ g_1)]$.
  Hence, $u=dp$, $u'=dp'$ and
  $\epsilon=d\circ \theta$ fulfil the conditions of (PR)(ii). In particular, $uq\in \Sigma$ because it is a composition of morphisms of $\Sigma$ and $u'q'\in \Sigma$ because $d\circ \theta\colon uq\Rightarrow u'q'$ is invertible.

\item \textbf{Order-enriched categories.} For an order-enriched category $\catx$ (i.e.\ a 2-category where all the hom-categories are simply posets) we can omit the Equification rule because it holds trivially.

 This way, our calculus of lax fractions becomes the calculus introduced in \cite{sousa2017calculus}  by the second author, except that Horizontal Repletion, ensuring that identity squares of the form
$\adjustbox{scale=0.60}{\begin{tikzcd}
	{} & {} \\
	{} & {}
	\arrow[equals, from=1-1, to=1-2]
	\arrow["f"', from=1-1, to=2-1]
	\arrow["f", from=1-2, to=2-2]
	\arrow[equals, from=2-1, to=2-2]
\end{tikzcd}}$
are $\Sg$-squares, was not used there. (Note that in \cite{sousa2017calculus} Equi-insertion is given by the rule called Coinsertion.)

The following examples in order-enriched categories can be found in \cite{sousa2017calculus}.
\begin{enumerate}
\item \textbf{Embbedings in $\mathbf{Pos}$.} Let $D$ be the contravariant endofunctor on $\mathbf{Pos}$ taking each poset X into the poset of lower sets of $X$, and every monotone map $f\colon X\to Y$ to the preimage map $Df\colon DY\to DX$.
    Let $\Sg$ consist of all embbeddings of $\mathbf{Pos}$ and commutative squares
     $\adjustbox{scale=0.70}{\begin{tikzcd}
	X & Y \\
	Z & W
	\arrow["m", from=1-1, to=1-2]
	\arrow["u"', from=1-1, to=2-1]
	\arrow["v", from=1-2, to=2-2]
	\arrow["n", from=2-1, to=2-2]
\end{tikzcd}}$
 such that $(Du)^*\cdot Dm=Dn\cdot (Dv)^*$, where $(-)^*$ stands for the left adjoint.
 Equivalently, $\Sigma$ consists of the squares such that, for every $y\in Y$ and $z\in Z$, if $n(z)\leq v(y)$ then there is some $x\in X$ with $z\leq u(x)$ and $m(x)\leq y$.
 Then $\Sg$ admits a  calculus of left lax fractions.
 \item \textbf{Embeddings in $\mathbf{Loc}$.}  Let $\mathbf{Loc}$ be the category of locales (i.e., frames) and localic maps, i.e., maps preserving all infima and whose left adjoints preserve finite meets. Recall that embeddings in $\mathbf{Loc}$ are precisely the localic maps $h$ that are split monomorphisms by their left adjoint: $h^*h = \mathrm{id}$.
 Let $\Sg$ consist of all embeddings and commutative squares
     $$\adjustbox{scale=0.80}{\begin{tikzcd}
	X & Y \\
	Z & W
	\arrow["m", from=1-1, to=1-2]
	\arrow["u"', from=1-1, to=2-1]
	\arrow["v", from=1-2, to=2-2]
	\arrow["n", from=2-1, to=2-2]
\end{tikzcd}}$$
 satisfying the Beck–Chevalley condition $v^*n =mu^*$. Then $\Sg$ admits a  calculus of left lax fractions.

 \item \textbf{Flat embeddings in $\mathbf{Loc}$.} In the following two cases we have also a subcategory of $\mathbf{Loc}^{\to}$ which admits a calculus of left  lax fractions:

 \begin{itemize} \item All dense embeddings and squares as above.
 \item All flat embeddings and squares as above.
 \end{itemize}
 \end{enumerate}

\item \textbf{Lax epimorphisms.}  Recall that in a 2-category $\catx$ a 1-cell $f\colon X\to Y$ is said to be a {\em lax epimorphism} (or {\em co-fully faithful}) if, for every object $Z$, the functor $\catx(Y,Z)\xrightarrow{(-)\circ f} \catx(X,Z)$ is fully faithful, i.e., every 2-cell $\alpha\colon g_1f\Rightarrow g_2f$ factors uniquely through $f$. Lax epimorphisms are stable under bi-pushouts (see \cite{lucatellisousa}).

Let $\catx$ have bi-pushouts and let $\Sg$ be the full subcategory of $\catx^{\to\cong}$ of all lax epimorphisms. Then, $\Sg$ admits a calculus of left  lax fractions --- that is, it fulfills rules (Id), (Rep), (Comp), (Sq), (Eq1) and (Eq2) of \cref{def:Full-Sigma}
above. It is easy to see all the rules hold. In particular, (Sq) is obtained as a bi-pushout.

	\item \textbf{Fully faithful functors.}  Let $\Cat$ be the 2-category of small categories. Let $\Sg$ consist of all fully faithful functors and squares
     $$\adjustbox{scale=0.80}{\begin{tikzcd}
	{\mathbb{A}} & {\mathbb{B}} \\
	{\mathbb{C}} & {\mathbb{D}}
	\arrow["M", from=1-1, to=1-2]
	\arrow["F"', from=1-1, to=2-1]
	\arrow["G", from=1-2, to=2-2]
	\arrow["\delta", shorten <=4pt, shorten >=4pt, Rightarrow, from=2-1, to=1-2]
	\arrow["N", from=2-1, to=2-2]
\end{tikzcd}}$$
  (with $M$ and $N$ fully faithful and $\delta$ invertible)  such that if $(\bar{Y}, \kappa)$ is a left Kan extension of the Yoneda embedding $Y\colon \mathbb{C} \to [\mathbb{C}^{op}, \mathbf{Set}]$ along $N$ then $(\bar{Y}G, (\bar{Y}\circ \delta)\cdot (\kappa\circ F))$ is a left Kan extension of $YF$ along $M$. This class admits a calculus of left lax fractions.

\item \textbf{Strict monoidal functors.} Let $\Cat(\Mon)$ be the 2-category of categories internal to the category $\Mon$ of monoids.  A category in $\Mon$ is the same thing as a strict monoidal category, while an internal functor is a strict monoidal functor. The class of strict monoidal functors whose underlying functors have fully faithful right adjoints and the pseudo-commutative squares whose underlying functors form Beck--Chevalley squares admits a calculus of right lax fractions.
\end{enumerate}

The last two examples, as well as other examples and the corresponding bicategories of lax fractions, will be studied in detail in the paper \cite{manuellsousa2}, where we will also explore the relation between the calculus of lax fractions and lax-idempotent monads and Kan extensions.
\end{example}

We end this section by giving some rules for working with $\Sigma$-squares that will be very ueful in what follows.
\begin{proposition}\label{pro:useful_rules}
	Let $\Sigma$  be a subcategory of $\catx^{\to\cong}$ admitting a  calculus of left lax fractions. Then it satisfies the following rules:
	\begin{enumerate}

		\item[{\rm Rule 1.}] Every square obtained as a finite (horizontal and vertical) composition of $\Sigma$-squares is a $\Sigma$-square.
		\item[{\rm Rule 2a.}] For composable $r,s\in \Sigma$, we have the $\Sigma$-square
		\[\begin{tikzcd}
			\bullet & \bullet \\
			\bullet & \bullet \rlap{\,.}
			\arrow["r", from=1-1, to=1-2]
			\arrow[""{name=0, anchor=center, inner sep=0}, "s", from=1-2, to=2-2]
			\arrow["sr"', from=2-1, to=2-2]
			\arrow[""{name=1, anchor=center, inner sep=0}, Rightarrow, no head, from=1-1, to=2-1]
			\arrow["{\SIGMA^{\id}}"{description}, draw=none, from=1, to=0]
		\end{tikzcd}\]

		\item[{\rm Rule 2b.}] Given
		$\begin{tikzcd}
			\bullet & \bullet \\
			\bullet & \bullet
			\arrow["r", from=1-1, to=1-2]
			\arrow[""{name=0, anchor=center, inner sep=0}, "s"', from=1-1, to=2-1]
			\arrow[""{name=1, anchor=center, inner sep=0}, "u", from=1-2, to=2-2]
			\arrow["t"', from=2-1, to=2-2]
			\arrow["{\SIGMA^{\delta}}"{description}, draw=none, from=0, to=1]
		\end{tikzcd} \;\; $  with $\; s\in \Sigma$,
		we have
		$\begin{tikzcd}
			\bullet & \bullet \\
			\bullet & \bullet
			\arrow["r", from=1-1, to=1-2]
			\arrow[""{name=0, anchor=center, inner sep=0}, Rightarrow, no head, from=1-1, to=2-1]
			\arrow[""{name=1, anchor=center, inner sep=0}, "u", from=1-2, to=2-2]
			\arrow["ts"', from=2-1, to=2-2]
			\arrow["{\SIGMA^{\delta}}"{description}, draw=none, from=0, to=1]
		\end{tikzcd}\; \; \; $ and $\; \; \;  %
		\begin{tikzcd}
			\bullet & \bullet \\
			\bullet & \bullet
			\arrow["s", from=1-1, to=1-2]
			\arrow[""{name=0, anchor=center, inner sep=0}, Rightarrow, no head, from=1-1, to=2-1]
			\arrow["ts"', from=2-1, to=2-2]
			\arrow[""{name=1, anchor=center, inner sep=0}, "t", from=1-2, to=2-2]
			\arrow["{\SIGMA^{\id}}"{description}, draw=none, from=0, to=1]
		\end{tikzcd}$.

		\item[{\rm Rule 3a.}] If we have diagrams of the form

		\[\begin{tikzcd}
			A && B \\
			\\
			C && D
			\arrow["r", from=1-1, to=1-3]
			\arrow[""{name=0, anchor=center, inner sep=0}, "f"', from=1-1, to=3-1]
			\arrow[""{name=1, anchor=center, inner sep=0}, "a"', curve={height=18pt}, from=1-3, to=3-3]
			\arrow[""{name=2, anchor=center, inner sep=0}, "b", curve={height=-18pt}, from=1-3, to=3-3]
			\arrow["s"', from=3-1, to=3-3]
			\arrow["{{{{\text{\Large $\Sigma$}}}}}"{description, pos=0.6}, draw=none, from=0, to=1]
			\arrow["\alpha", between={0.2}{0.8}, Rightarrow, from=1, to=2]
		\end{tikzcd}
		\; \; \quad \text{and}\;\; \quad
		\begin{tikzcd}
			A && B \\
			\\
			C && D
			\arrow["r", from=1-1, to=1-3]
			\arrow[""{name=0, anchor=center, inner sep=0}, "f"', from=1-1, to=3-1]
			\arrow[""{name=1, anchor=center, inner sep=0}, "b"', curve={height=18pt}, from=1-3, to=3-3]
			\arrow[""{name=2, anchor=center, inner sep=0}, "a", curve={height=-18pt}, from=1-3, to=3-3]
			\arrow["s"', from=3-1, to=3-3]
			\arrow["{{{{{\text{\Large $\Sigma$}}}}}}"{description, pos=0.6}, draw=none, from=0, to=1]
			\arrow["\beta", between={0.2}{0.8}, Rightarrow, from=1, to=2]
		\end{tikzcd}
		\]
		with $\alpha \circ r=(\beta\circ r)^{-1}$, then there is a 1-cell $d\colon D\to E$  such that
		\[\begin{tikzcd}
			C & D \\
			C & E
			\arrow["s", from=1-1, to=1-2]
			\arrow[""{name=0, anchor=center, inner sep=0}, "d", from=1-2, to=2-2]
			\arrow["ds"', from=2-1, to=2-2]
			\arrow[""{name=1, anchor=center, inner sep=0}, Rightarrow, no head, from=1-1, to=2-1]
			\arrow["{\text{\Large $\Sigma$}}"{description}, draw=none, from=1, to=0]
		\end{tikzcd}\] and $d \circ \alpha=(d\circ \beta)^{-1}$.

		\item[{\rm Rule 3b.}] Given $\Sigma$-squares
		\[\begin{tikzcd}
			A & B \\
			C & D
			\arrow["r", from=1-1, to=1-2]
			\arrow[from=1-1, to=2-1]
			\arrow[""{name=0, anchor=center, inner sep=0}, "f"', from=1-1, to=2-1]
			\arrow[""{name=1, anchor=center, inner sep=0}, "a", from=1-2, to=2-2]
			\arrow["s"', from=2-1, to=2-2]
			\arrow["{{\SIGMA^{\delta}}}"{description}, draw=none, from=0, to=1]
		\end{tikzcd}
		\; \; \quad \text{and}\;\; \quad
		\begin{tikzcd}
			A & B \\
			C & D
			\arrow["r", from=1-1, to=1-2]
			\arrow[""{name=0, anchor=center, inner sep=0}, "f"', from=1-1, to=2-1]
			\arrow[""{name=1, anchor=center, inner sep=0}, "b", from=1-2, to=2-2]
			\arrow["s"', from=2-1, to=2-2]
			\arrow["{{\SIGMA^{\epsilon}}}"{description}, draw=none, from=0, to=1]
		\end{tikzcd}\]
		there is a 1-cell $d\colon D\to E$ and an invertible 2-cell $\gamma\colon da\Rightarrow db$ such that  \[\begin{tikzcd}
			C & D \\
			C & E
			\arrow["s", from=1-1, to=1-2]
			\arrow[""{name=0, anchor=center, inner sep=0}, "d", from=1-2, to=2-2]
			\arrow["ds"', from=2-1, to=2-2]
			\arrow[""{name=1, anchor=center, inner sep=0}, Rightarrow, no head, from=1-1, to=2-1]
			\arrow["{\SIGMA^{\id}}"{description}, draw=none, from=1, to=0]
		\end{tikzcd}\] and  $(\gamma\circ r)\cdot (d\circ \delta)=d\circ \epsilon$.

		\item[{\rm Rule 4.}]
		If we have two  diagrams of the form
\[\begin{tikzcd}
	{B_1} & {I_1} && {B_1} & I_1 \\
	A & X && A & Y \\
	{B_2} & {I_2} && {B_2} & {I_2}
	\arrow["{r_1}", from=1-1, to=1-2]
	\arrow[""{name=0, anchor=center, inner sep=0}, "{b_1}"', from=1-1, to=2-1]
	\arrow[""{name=1, anchor=center, inner sep=0}, "{{{x_1}}}", from=1-2, to=2-2]
	\arrow["{r_1}", from=1-4, to=1-5]
	\arrow[""{name=2, anchor=center, inner sep=0}, "{b_1}"', from=1-4, to=2-4]
	\arrow[""{name=3, anchor=center, inner sep=0}, "{{{y_1}}}", from=1-5, to=2-5]
	\arrow["{x_3}", from=2-1, to=2-2]
	\arrow["{y_3}", from=2-4, to=2-5]
	\arrow[""{name=4, anchor=center, inner sep=0}, "{b_2}", from=3-1, to=2-1]
	\arrow["{r_2}"', from=3-1, to=3-2]
	\arrow[""{name=5, anchor=center, inner sep=0}, "{{{x_2}}}"', from=3-2, to=2-2]
	\arrow[""{name=6, anchor=center, inner sep=0}, "{b_2}", from=3-4, to=2-4]
	\arrow["{r_2}"', from=3-4, to=3-5]
	\arrow[""{name=7, anchor=center, inner sep=0}, "{{{y_2}}}"', from=3-5, to=2-5]
	\arrow["{\SIGMA^{\delta_1}}"{description}, draw=none, from=0, to=1]
	\arrow["{\SIGMA^{\epsilon_1}}"{description}, draw=none, from=2, to=3]
	\arrow["{\SIGMA^{\delta_2}}"{description}, draw=none, from=4, to=5]
	\arrow["{\SIGMA^{\epsilon_2}}"{description}, draw=none, from=6, to=7]
\end{tikzcd}\]
		then there are morphisms
  $X\xrightarrow{d_x}D\xleftarrow{d_y}Y$ and invertible 2-cells $\gamma_i\colon d_xx_i\Rightarrow d_yy_i$, for $i=1,2$, such that we have the following $\Sigma$-squares formed from $d_x$ and $d_y$ and an equality of pasting diagrams for the top and bottom parts of the following diagram.

\[\begin{tikzcd}
	{B_1} & {I_1} && {B_1} & {I_1} \\
	A & X && A & Y \\
	A & D && A & D \\
	A & X && A & D \\
	{B_2} & {I_2} && {B_2} & {I_2}
	\arrow["{{r_1}}", from=1-1, to=1-2]
	\arrow[""{name=0, anchor=center, inner sep=0}, "{b_1}"', from=1-1, to=2-1]
	\arrow[""{name=1, anchor=center, inner sep=0}, "{{{{{{x_1}}}}}}", from=1-2, to=2-2]
	\arrow[""{name=2, anchor=center, inner sep=0}, "{{{{d_yy_1}}}}", curve={height=-24pt}, from=1-2, to=3-2]
	\arrow["{{r_1}}", from=1-4, to=1-5]
	\arrow[""{name=3, anchor=center, inner sep=0}, "{b_1}"', from=1-4, to=2-4]
	\arrow[""{name=4, anchor=center, inner sep=0}, "{{{{{{y_1}}}}}}", from=1-5, to=2-5]
	\arrow["{{x_3}}", from=2-1, to=2-2]
	\arrow[""{name=5, anchor=center, inner sep=0}, equals, from=2-1, to=3-1]
	\arrow[""{name=6, anchor=center, inner sep=0}, "{{{{{{d_x}}}}}}", from=2-2, to=3-2]
	\arrow["{{y_3}}", from=2-4, to=2-5]
	\arrow[""{name=7, anchor=center, inner sep=0}, equals, from=2-4, to=3-4]
	\arrow[""{name=8, anchor=center, inner sep=0}, "{{{{{{d_y}}}}}}", from=2-5, to=3-5]
	\arrow["u", from=3-1, to=3-2]
	\arrow[""{name=9, anchor=center, inner sep=0}, equals, from=3-1, to=4-1]
	\arrow["{\text{\large $=$}}"{description, pos=0.6}, draw=none, from=3-2, to=3-4]
	\arrow["u", from=3-4, to=3-5]
	\arrow[""{name=10, anchor=center, inner sep=0}, equals, from=3-4, to=4-4]
	\arrow["{{x_3}}", from=4-1, to=4-2]
	\arrow[""{name=11, anchor=center, inner sep=0}, "{{{{{{d_x}}}}}}"', from=4-2, to=3-2]
	\arrow["{{y_3}}", from=4-4, to=4-5]
	\arrow[""{name=12, anchor=center, inner sep=0}, "{{d_y}}"', from=4-5, to=3-5]
	\arrow[""{name=13, anchor=center, inner sep=0}, "{{b_2}}", from=5-1, to=4-1]
	\arrow["{{r_2}}"', from=5-1, to=5-2]
	\arrow[""{name=14, anchor=center, inner sep=0}, "{{{{d_yy_2}}}}"', curve={height=24pt}, from=5-2, to=3-2]
	\arrow[""{name=15, anchor=center, inner sep=0}, "{{{{{{x_2}}}}}}"', from=5-2, to=4-2]
	\arrow[""{name=16, anchor=center, inner sep=0}, "{{b_2}}", from=5-4, to=4-4]
	\arrow["{{r_2}}"', from=5-4, to=5-5]
	\arrow[""{name=17, anchor=center, inner sep=0}, "{{{{{{y_2}}}}}}"', from=5-5, to=4-5]
	\arrow["{{\SIGMA^{\delta_1}}}"{description}, draw=none, from=0, to=1]
	\arrow["{{\SIGMA^{\epsilon_1}}}"{description}, draw=none, from=3, to=4]
	\arrow["{{\SIGMA^{\phi}}}"{description}, draw=none, from=5, to=6]
	\arrow["{{\gamma_1}}", between={0}{0.8}, Rightarrow, from=2-2, to=2]
	\arrow["{{\SIGMA^{\chi}}}"{description}, draw=none, from=7, to=8]
	\arrow["{{\SIGMA^{\phi}}}"{description}, draw=none, from=9, to=11]
	\arrow["{{\SIGMA^{\chi}}}"{description}, draw=none, from=10, to=12]
	\arrow["{{\gamma_2}}", between={0}{0.8}, Rightarrow, from=4-2, to=14]
	\arrow["{{\SIGMA^{\delta_2}}}"{description}, draw=none, from=13, to=15]
	\arrow["{{\SIGMA^{\epsilon_2}}}"{description}, draw=none, from=16, to=17]
\end{tikzcd}\]
\item[{\rm Rule 4'.}] If we have $\Sigma$-squares as the two ones on the top of the diagrams
\[\begin{tikzcd}
	B & I && B & I \\
	A & X && A & Y \\
	A & D & {} & A & D
	\arrow["r", from=1-1, to=1-2]
	\arrow[""{name=0, anchor=center, inner sep=0}, "b"', from=1-1, to=2-1]
	\arrow[""{name=1, anchor=center, inner sep=0}, "x", from=1-2, to=2-2]
	\arrow[""{name=2, anchor=center, inner sep=0}, "{{{{{d_yy}}}}}", curve={height=-24pt}, from=1-2, to=3-2]
	\arrow["r", from=1-4, to=1-5]
	\arrow[""{name=3, anchor=center, inner sep=0}, "b"', from=1-4, to=2-4]
	\arrow[""{name=4, anchor=center, inner sep=0}, "{{y}}", from=1-5, to=2-5]
	\arrow[from=2-1, to=2-2]
	\arrow[""{name=5, anchor=center, inner sep=0}, equals, from=2-1, to=3-1]
	\arrow[""{name=6, anchor=center, inner sep=0}, "{{{{{{{d_x}}}}}}}", from=2-2, to=3-2]
	\arrow[from=2-4, to=2-5]
	\arrow[""{name=7, anchor=center, inner sep=0}, equals, from=2-4, to=3-4]
	\arrow[""{name=8, anchor=center, inner sep=0}, "{{{{{{{d_y}}}}}}}", from=2-5, to=3-5]
	\arrow["u", from=3-1, to=3-2]
	\arrow["u", from=3-4, to=3-5]
	\arrow["{{\SIGMA^{\delta}}}"{description}, draw=none, from=0, to=1]
	\arrow["{\text{\large $=$}}"{pos=0.7}, shift right, draw=none, from=2, to=2-4]
	\arrow["{{\SIGMA^{\epsilon}}}"{description}, draw=none, from=3, to=4]
	\arrow["{{\SIGMA^{\phi}}}"{description}, draw=none, from=5, to=6]
	\arrow["{{\gamma}}", between={0}{0.8}, Rightarrow, from=2-2, to=2]
	\arrow["{{\SIGMA^{\chi}}}"{description}, draw=none, from=7, to=8]
\end{tikzcd}\]
		then there are 1-cells
  $X\xrightarrow{d_x}D\xleftarrow{d_y}Y$ forming $\Sigma$-squares as in the bottom of the diagrams and an invertible 2-cell $\gamma: d_xx\Rightarrow d_yy$ forming the above equality of pasting diagrams.

 \item[{\rm Rule 5.}] For every two spans $X\xleftarrow{v}B\xrightarrow{f}C$ and $X\xleftarrow{v}B\xrightarrow{g}C$ with $v\in \Sigma$, there is a morphism $w\colon C\to D$ and $\Sigma$-squares of the form
	\begin{center}
	\begin{minipage}{.33\textwidth}
	\[\begin{tikzcd}
	B & X \\
	C & D
	\arrow["v", from=1-1, to=1-2]
	\arrow[""{name=0, anchor=center, inner sep=0}, "f"', from=1-1, to=2-1]
	\arrow[""{name=1, anchor=center, inner sep=0}, "{{f'}}", from=1-2, to=2-2]
	\arrow["w"', from=2-1, to=2-2]
	\arrow["{{\SIGMA^{\delta}}}"{description}, draw=none, from=0, to=1]
	\end{tikzcd}\]
	\end{minipage}
	\quad \text{and} \quad
	\begin{minipage}{.33\textwidth}
	\[\begin{tikzcd}
	B & X \\
	C & {D \rlap{\,.}}
	\arrow["v", from=1-1, to=1-2]
	\arrow[""{name=0, anchor=center, inner sep=0}, "g"', from=1-1, to=2-1]
	\arrow[""{name=1, anchor=center, inner sep=0}, "{{g'}}", from=1-2, to=2-2]
	\arrow["w"', from=2-1, to=2-2]
	\arrow["{{\SIGMA^{\epsilon}}}"{description}, draw=none, from=0, to=1]
	\end{tikzcd}\]
	\end{minipage}
	\end{center}

		\item[{\rm Rule 6.}] Given a diagram
		\[\begin{tikzcd}
			X & B & C
			\arrow["v"', from=1-2, to=1-1]
			\arrow[""{name=0, anchor=center, inner sep=0}, "g"', curve={height=12pt}, from=1-2, to=1-3]
			\arrow[""{name=1, anchor=center, inner sep=0}, "f", curve={height=-12pt}, from=1-2, to=1-3]
			\arrow["\beta", shorten <=3pt, shorten >=3pt, Rightarrow, from=1, to=0]
		\end{tikzcd}\]
		with $v\in \Sigma$, there is a 1-cell $w\colon C\to D$,  $\Sigma$-squares of the form
		\begin{center}
		\begin{minipage}{.33\textwidth}
		\[\begin{tikzcd}
		B & X \\
		C & D
		\arrow["v", from=1-1, to=1-2]
		\arrow[""{name=0, anchor=center, inner sep=0}, "f"', from=1-1, to=2-1]
		\arrow[""{name=1, anchor=center, inner sep=0}, "{{f'}}", from=1-2, to=2-2]
		\arrow["w"', from=2-1, to=2-2]
		\arrow["{{\SIGMA^{\delta}}}"{description}, draw=none, from=0, to=1]
		\end{tikzcd}\]
		\end{minipage}
		\quad  and \quad
		\begin{minipage}{.33\textwidth}
		\[\begin{tikzcd}
		B & X \\
		C & {D \rlap{\,,}}
		\arrow["v", from=1-1, to=1-2]
		\arrow[""{name=0, anchor=center, inner sep=0}, "g"', from=1-1, to=2-1]
		\arrow[""{name=1, anchor=center, inner sep=0}, "{{g'}}", from=1-2, to=2-2]
		\arrow["w"', from=2-1, to=2-2]
		\arrow["{{\SIGMA^{\epsilon}}}"{description}, draw=none, from=0, to=1]
		\end{tikzcd}\]
		\end{minipage}
		\end{center}
		and a 2-cell $\beta'\colon f'\Rightarrow g'$ satisfying the following equality of pasting diagrams.
		\[\begin{tikzcd}
		B & X &&& B & X \\
		C & D &&& C & D
		\arrow["v", from=1-1, to=1-2]
		\arrow["f"', from=1-1, to=2-1]
		\arrow[""{name=0, anchor=center, inner sep=0}, "{{{f'}}}"', from=1-2, to=2-2]
		\arrow[""{name=1, anchor=center, inner sep=0}, "{{{g'}}}", curve={height=-30pt}, from=1-2, to=2-2]
		\arrow["v", from=1-5, to=1-6]
		\arrow[""{name=2, anchor=center, inner sep=0}, "g", from=1-5, to=2-5]
		\arrow[""{name=3, anchor=center, inner sep=0}, "f"', curve={height=30pt}, from=1-5, to=2-5]
		\arrow["{{{g'}}}", from=1-6, to=2-6]
		\arrow["\delta", between={0.2}{0.8}, Rightarrow, from=2-1, to=1-2]
		\arrow["w"', from=2-1, to=2-2]
		\arrow["\epsilon", between={0.2}{0.8}, Rightarrow, from=2-5, to=1-6]
		\arrow["w"', from=2-5, to=2-6]
		\arrow["{{{\beta'}}}"'{pos=0.4}, shift left, between={0.2}{0.8}, Rightarrow, from=0, to=1]
		\arrow["{{{=}}}"{marking, allow upside down}, draw=none, from=1, to=3]
		\arrow["\beta"', shift left, between={0.2}{0.8}, Rightarrow, from=3, to=2]
		\end{tikzcd}\; .\]
	\end{enumerate}
\end{proposition}

\begin{proof}
	\begin{enumerate}
		\item[1.] Horizontal composition of $\Sigma$-squares is given by Composition, the vertical one is the composition in the subcategory $\Sigma$.
		\item[2a.] This is obtained by using Vertical Repletion, Identity and Composition:
		\[\begin{tikzcd}
			\bullet & \bullet & \bullet & \bullet & \bullet \\
			\bullet & \bullet & \bullet & \bullet & \bullet \rlap{\,.}
			\arrow[""{name=0, anchor=center, inner sep=0}, Rightarrow, no head, from=1-1, to=2-1]
			\arrow["r"', from=2-1, to=2-2]
			\arrow[""{name=1, anchor=center, inner sep=0}, "s", from=1-3, to=2-3]
			\arrow["s"', from=2-2, to=2-3]
			\arrow[""{name=2, anchor=center, inner sep=0}, Rightarrow, no head, from=1-4, to=2-4]
			\arrow["r", from=1-4, to=1-5]
			\arrow[""{name=3, anchor=center, inner sep=0}, "s", from=1-5, to=2-5]
			\arrow["sr"', from=2-4, to=2-5]
			\arrow["r", from=1-1, to=1-2]
			\arrow[""{name=4, anchor=center, inner sep=0}, Rightarrow, no head, from=1-2, to=2-2]
			\arrow[Rightarrow, no head, from=1-2, to=1-3]
			\arrow["{\SIGMA^{\id}}"{description}, draw=none, from=2, to=3]
			\arrow["{=}"{description}, draw=none, from=1, to=2]
			\arrow["{\SIGMA^{\id}}"{description}, draw=none, from=0, to=4]
			\arrow["{\SIGMA^{\id}}"{description}, draw=none, from=4, to=1]
		\end{tikzcd}\]

		\item[2b.] Observe that
		\[\begin{tikzcd}
			\bullet & \bullet & \bullet & \bullet & \bullet \\
			\bullet & \bullet & \bullet & \bullet & \bullet
			\arrow[Rightarrow, no head, from=1-1, to=1-2]
			\arrow[""{name=0, anchor=center, inner sep=0}, Rightarrow, no head, from=1-1, to=2-1]
			\arrow["s"', from=2-1, to=2-2]
			\arrow[""{name=1, anchor=center, inner sep=0}, "s", from=1-2, to=2-2]
			\arrow["r", from=1-2, to=1-3]
			\arrow[""{name=2, anchor=center, inner sep=0}, "u", from=1-3, to=2-3]
			\arrow["t"', from=2-2, to=2-3]
			\arrow[""{name=3, anchor=center, inner sep=0}, Rightarrow, no head, from=1-4, to=2-4]
			\arrow["r", from=1-4, to=1-5]
			\arrow[""{name=4, anchor=center, inner sep=0}, "u", from=1-5, to=2-5]
			\arrow["ts"', from=2-4, to=2-5]
			\arrow["{=}"{description}, draw=none, from=2, to=3]
			\arrow["{\SIGMA^{\id}}"{description}, draw=none, from=0, to=1]
			\arrow["{\SIGMA^{\delta}}"{description}, draw=none, from=1, to=2]
			\arrow["{\SIGMA^{\delta}}"{description}, draw=none, from=3, to=4]
		\end{tikzcd}\, .\]
		The other $\Sigma$-square is obtained by Rule 2a, since $s,t\in \Sigma$.
		
		\item[3a.] Since we have $(\beta\cdot \alpha)\circ r=\id_a\circ r$ and $(\alpha\cdot \beta)\circ r=\id_b\circ r$, by Equification twice we obtain, successively, 1-cells $d_1\colon D\to D_1$ and $d_2\colon D_1\to D_2$ such that
			\[\begin{tikzcd}
				C & D && C & {D_1} \\
				C & {D_1} && C & {D_2}
				\arrow["s", from=1-1, to=1-2]
				\arrow[""{name=0, anchor=center, inner sep=0}, Rightarrow, no head, from=1-1, to=2-1]
				\arrow[""{name=1, anchor=center, inner sep=0}, "{d_1}", from=1-2, to=2-2]
				\arrow["{d_1s}", from=1-4, to=1-5]
				\arrow[""{name=2, anchor=center, inner sep=0}, Rightarrow, no head, from=1-4, to=2-4]
				\arrow[""{name=3, anchor=center, inner sep=0}, "{d_2}", from=1-5, to=2-5]
				\arrow["{d_1s}"', from=2-1, to=2-2]
				\arrow["{d_2d_1s}"', from=2-4, to=2-5]
				\arrow["\SIGMA"{description}, draw=none, from=0, to=1]
				\arrow["\SIGMA"{description}, draw=none, from=2, to=3]
			\end{tikzcd}\]
			and, also, first $d_1\circ (\beta\cdot\alpha)=d_1\circ \id_a$ and, secondly, $d_2\circ d_1\circ (\alpha\cdot \beta)=d_2\circ d_1\circ \id_b$.
			Thus, the 1-cell $d=d_2d_1$ is as desired.

			\item[3b.] We have the 2-cell $\xymatrix{ar\ar@{=>}[r]^{\delta^{-1}}&sf\ar@{=>}[r]^{\epsilon}&br}$. Then, using Equi-insertion, there is $d_1\colon D\to D_1$ and $\gamma_1\colon d_1a\Rightarrow d_1b$ such that $\begin{tikzcd}
					C & D \\
					C & {D_1}
					\arrow["s", from=1-1, to=1-2]
					\arrow[""{name=0, anchor=center, inner sep=0}, Rightarrow, no head, from=1-1, to=2-1]
					\arrow[""{name=1, anchor=center, inner sep=0}, "{d_1}", from=1-2, to=2-2]
					\arrow["{d_1s}"', from=2-1, to=2-2]
					\arrow["{\SIGMA^{\id}}"{description}, draw=none, from=0, to=1]
				\end{tikzcd}\,$
				and $\gamma_1\circ r=d_1\circ (\epsilon\cdot \delta^{-1})$. Analogously, departing from $\xymatrix{d_1br\ar@{=>}[r]^{d_1\epsilon^{-1}}&d_1sf\ar@{=>}[r]^{d_1\delta}&d_1ar}$, we obtain $d_2\colon D_1\to D_2$ and a 2-cell $\gamma_2\colon d_2d_1b\Rightarrow d_2d_1a$ such that
				$\begin{tikzcd}
					C & D \\
					C & {D_2}
					\arrow["d_1s", from=1-1, to=1-2]
					\arrow[""{name=0, anchor=center, inner sep=0}, Rightarrow, no head, from=1-1, to=2-1]
					\arrow[""{name=1, anchor=center, inner sep=0}, "{d_2}", from=1-2, to=2-2]
					\arrow["{d_2d_1s}"', from=2-1, to=2-2]
					\arrow["{\SIGMA^{\id}}"{description}, draw=none, from=0, to=1]
				\end{tikzcd}$
				and $\gamma_2 \circ r=d_2\circ ((d_1\circ \delta)\cdot (d_1\circ \epsilon^{-1}))$.
				This way, we have the $\Sigma$-square
				\[\begin{tikzcd}
					C & D \\
					C & {D_2}
					\arrow["s", from=1-1, to=1-2]
					\arrow[""{name=0, anchor=center, inner sep=0}, Rightarrow, no head, from=1-1, to=2-1]
					\arrow[""{name=1, anchor=center, inner sep=0}, "{d_2d_1}", from=1-2, to=2-2]
					\arrow["{d_2d_1s}"', from=2-1, to=2-2]
					\arrow["{\SIGMA^{\id}}"{description}, draw=none, from=0, to=1]
				\end{tikzcd}\]
				and the 2-cell composition $\xymatrix{d_2d_1a\ar@{=>}[r]^{d_2\gamma_1}&d_2d_1b\ar@{=>}[r]^{\gamma_2}&d_2d_1a}$
				such that $(\gamma_2\circ r)^{-1}=d_2\circ \gamma_1\circ r$. Using the Rule 3a, we conclude that there is a map $d_3\colon D_2\to D_3$ such that $\begin{tikzcd}
					C & D \\
					C & {D_2}
					\arrow["d_2d_1s", from=1-1, to=1-2]
					\arrow[""{name=0, anchor=center, inner sep=0}, Rightarrow, no head, from=1-1, to=2-1]
					\arrow[""{name=1, anchor=center, inner sep=0}, "{d_3}", from=1-2, to=2-2]
					\arrow["{d_3d_2d_1s}"', from=2-1, to=2-2]
					\arrow["{\SIGMA^{\id}}"{description}, draw=none, from=0, to=1]
				\end{tikzcd}$
				and $(d_3\circ \gamma_2)^{-1}=d_3d_2\gamma_1$. Hence, the 1-cell $d=d_3d_2d_1$ and the invertible 2-cell $\gamma=d_3d_2\gamma_1\colon da\Rightarrow db$ are as desired.
		
		\item[4.] We obtain successively:
		
	\begin{enumerate}[label=(\roman*)]
		\item
	\hskip3mm	$\begin{tikzcd}
			A & X \\
			Y & Z
			\arrow["{x_3}", from=1-1, to=1-2]
			\arrow[""{name=0, anchor=center, inner sep=0}, "{y_3}"', from=1-1, to=2-1]
			\arrow[""{name=1, anchor=center, inner sep=0}, "{a'}", from=1-2, to=2-2]
			\arrow["a"', from=2-1, to=2-2]
			\arrow["{\SIGMA^{\phi}}"{description}, draw=none, from=0, to=1]
		\end{tikzcd}$ {\hfill by Square}
		
		\vskip2mm
		
	\hskip10mm \rule{250pt}{0.5pt}
		
		\vskip2mm
		
		\item  \hskip3mm $\begin{tikzcd}
			A & X \\
			A & Z
			\arrow["{{x_3}}", from=1-1, to=1-2]
			\arrow[""{name=0, anchor=center, inner sep=0}, Rightarrow, no head, from=1-1, to=2-1]
			\arrow[""{name=1, anchor=center, inner sep=0}, "{{a'}}", from=1-2, to=2-2]
			\arrow["{ay_3}"', from=2-1, to=2-2]
			\arrow["{{\SIGMA^{\phi}}}"{description}, draw=none, from=0, to=1]
		\end{tikzcd}$ \hskip5mm and \hskip5mm $\begin{tikzcd}
		A & Y \\
		A & Z
		\arrow["{y_3}", from=1-1, to=1-2]
		\arrow[""{name=0, anchor=center, inner sep=0}, Rightarrow, no head, from=1-1, to=2-1]
		\arrow[""{name=1, anchor=center, inner sep=0}, "a", from=1-2, to=2-2]
		\arrow["{{ay_3}}"', from=2-1, to=2-2]
		\arrow["{{{\SIGMA^{\id}}}}"{description}, draw=none, from=0, to=1]
		\end{tikzcd}$  {\hfill by (i) and Rule 2b}
		
			\vskip2mm
		
		\hskip10mm \rule{250pt}{0.5pt}
		
		\vskip2mm
		
		\item
		\hskip3mm
		$\begin{tikzcd}
			B_i & I_i \\
			A & X \\
			A & Z
			\arrow["r_i", from=1-1, to=1-2]
			\arrow[""{name=0, anchor=center, inner sep=0}, "b_i"', from=1-1, to=2-1]
			\arrow[""{name=1, anchor=center, inner sep=0}, "{x_i}", from=1-2, to=2-2]
			\arrow["{x_3}"', from=2-1, to=2-2]
			\arrow[""{name=2, anchor=center, inner sep=0}, Rightarrow, no head, from=2-1, to=3-1]
			\arrow[""{name=3, anchor=center, inner sep=0}, "{a'}", from=2-2, to=3-2]
			\arrow["{ay_3}"', from=3-1, to=3-2]
			\arrow["{\SIGMA^{\delta_i}}"{description}, draw=none, from=0, to=1]
			\arrow["{\SIGMA^{\phi}}"{description}, draw=none, from=2, to=3]
			\end{tikzcd}$ \hskip2mm and \hskip2mm
			 $\begin{tikzcd}
				B_i & I_i \\
				A & Y \\
				A & Z
				\arrow["r_i", from=1-1, to=1-2]
				\arrow[""{name=0, anchor=center, inner sep=0}, "b_i"', from=1-1, to=2-1]
				\arrow[""{name=1, anchor=center, inner sep=0}, "{y_i}", from=1-2, to=2-2]
				\arrow["{y_3}"', from=2-1, to=2-2]
				\arrow[""{name=2, anchor=center, inner sep=0}, Rightarrow, no head, from=2-1, to=3-1]
				\arrow[""{name=3, anchor=center, inner sep=0}, "a", from=2-2, to=3-2]
				\arrow["{{ay_3}}"', from=3-1, to=3-2]
				\arrow["{\SIGMA^{\epsilon_{i}}}"{description}, draw=none, from=0, to=1]
					\arrow["{\SIGMA^{\id}}"{description}, draw=none, from=2, to=3]
				\end{tikzcd}\;$, $i=1,2$
				{\hfill from the initial data, (ii) and Rule 1}
				
					\vskip2mm
				
				\hskip10mm \rule{250pt}{0.5pt}
				
				\vskip2mm
				
				\item
				\hskip3mm
			$\begin{tikzcd}
	{B_1} & {I_1} && {B_1} & {I_1} \\
	A & X && A & Y \\
	A & Z && A & Z \\
	A & Q && A & Q
	\arrow["{{r_1}}", from=1-1, to=1-2]
	\arrow[""{name=0, anchor=center, inner sep=0}, "{b_1}"', from=1-1, to=2-1]
	\arrow[""{name=1, anchor=center, inner sep=0}, "{{{{{{x_1}}}}}}", from=1-2, to=2-2]
	\arrow[""{name=2, anchor=center, inner sep=0}, "{{qay_1}}", curve={height=-30pt}, from=1-2, to=4-2]
	\arrow["{{r_1}}", from=1-4, to=1-5]
	\arrow[""{name=3, anchor=center, inner sep=0}, "{b_1}"', from=1-4, to=2-4]
	\arrow[""{name=4, anchor=center, inner sep=0}, "{{y_1}}", from=1-5, to=2-5]
	\arrow["{{x_3}}", from=2-1, to=2-2]
	\arrow[""{name=5, anchor=center, inner sep=0}, equals, from=2-1, to=3-1]
	\arrow[""{name=6, anchor=center, inner sep=0}, "{{a'}}"', from=2-2, to=3-2]
	\arrow["{{y_3}}", from=2-4, to=2-5]
	\arrow[""{name=7, anchor=center, inner sep=0}, equals, from=2-4, to=3-4]
	\arrow[""{name=8, anchor=center, inner sep=0}, "{{a}}", from=2-5, to=3-5]
	\arrow["{{ay_3}}", from=3-1, to=3-2]
	\arrow[""{name=9, anchor=center, inner sep=0}, equals, from=3-1, to=4-1]
	\arrow[""{name=10, anchor=center, inner sep=0}, "q", from=3-2, to=4-2]
	\arrow["{ay_3}", from=3-4, to=3-5]
	\arrow[""{name=11, anchor=center, inner sep=0}, equals, from=3-4, to=4-4]
	\arrow[""{name=12, anchor=center, inner sep=0}, "q", from=3-5, to=4-5]
	\arrow["{{v_1}}"', from=4-1, to=4-2]
	\arrow["{{v_1}}"', from=4-4, to=4-5]
	\arrow["{{\SIGMA^{\delta_1}}}"{description}, draw=none, from=0, to=1]
	\arrow["{{=}}"{description, pos=0.6}, draw=none, from=2, to=7]
	\arrow["{{\SIGMA^{\epsilon_1}}}"{description}, draw=none, from=3, to=4]
	\arrow["{{\SIGMA^{\phi}}}"{description}, draw=none, from=5, to=6]
	\arrow["\gamma", shorten <=6pt, shorten >=6pt, Rightarrow, from=6, to=2]
	\arrow["{{\SIGMA^{\id}}}"{description}, draw=none, from=7, to=8]
	\arrow["{{\SIGMA^{\id}}}"{marking, allow upside down}, draw=none, from=9, to=10]
	\arrow["{{\SIGMA^{\id}}}"{description}, draw=none, from=11, to=12]
\end{tikzcd}$ \hskip3mm with $\gamma$ invertible
				 {\hfill by (iii) and Rule 3b. }
					\vskip2mm
				
				\hskip10mm \rule{250pt}{0.5pt}
				
				\vskip2mm
				
				\item
				\hskip3mm
			$\begin{tikzcd}
					{B_2} & {I_2} \\
					A & Q
					\arrow["{{r_2}}", from=1-1, to=1-2]
					\arrow[""{name=0, anchor=center, inner sep=0}, "{{b_2}}"', from=1-1, to=2-1]
					\arrow[""{name=1, anchor=center, inner sep=0}, "{{qa'x_2}}", from=1-2, to=2-2]
					\arrow["{{v_1}}"', from=2-1, to=2-2]
					\arrow["\SIGMA"{description}, draw=none, from=0, to=1]
				\end{tikzcd}$
				\hskip2mm and \hskip2mm
			$\begin{tikzcd}
					{B_2} & {I_2} \\
					A & Q
					\arrow["{{r_2}}", from=1-1, to=1-2]
					\arrow[""{name=0, anchor=center, inner sep=0}, "{{b_2}}"', from=1-1, to=2-1]
					\arrow[""{name=1, anchor=center, inner sep=0}, "{{qay_2}}", from=1-2, to=2-2]
					\arrow["{{v_1}}"', from=2-1, to=2-2]
					\arrow["\SIGMA"{description}, draw=none, from=0, to=1]
				\end{tikzcd}$
				{\hfill by composing $\Sigma$-squares from (iii) and (iv)}
				
					\vskip2mm
				
				\hskip10mm \rule{250pt}{0.5pt}
				
				\vskip2mm
				
				\item
				\hskip3mm
			$\begin{tikzcd}
	{B_2} & {I_2} &&& {B_2} & {I_2} \\
	A & Q &&& A & Q \\
	A & D &&& A & D
	\arrow["{{{r_2}}}", from=1-1, to=1-2]
	\arrow[""{name=0, anchor=center, inner sep=0}, "{{{b_2}}}"', from=1-1, to=2-1]
	\arrow[""{name=1, anchor=center, inner sep=0}, "{{{qa'x_2}}}", from=1-2, to=2-2]
	\arrow[""{name=2, anchor=center, inner sep=0}, "{{q'qay_2}}", curve={height=-40pt}, from=1-2, to=3-2]
	\arrow["{{{r_2}}}", from=1-5, to=1-6]
	\arrow[""{name=3, anchor=center, inner sep=0}, "{{{b_2}}}"', from=1-5, to=2-5]
	\arrow[""{name=4, anchor=center, inner sep=0}, "{{{qay_2}}}", from=1-6, to=2-6]
	\arrow["{{{v_1}}}"', from=2-1, to=2-2]
	\arrow[""{name=5, anchor=center, inner sep=0}, equals, from=2-1, to=3-1]
	\arrow[""{name=6, anchor=center, inner sep=0}, "{{q'}}", from=2-2, to=3-2]
	\arrow["{{{v_1}}}"', from=2-5, to=2-6]
	\arrow[""{name=7, anchor=center, inner sep=0}, equals, from=2-5, to=3-5]
	\arrow[""{name=8, anchor=center, inner sep=0}, "{{q'}}", from=2-6, to=3-6]
	\arrow["{{q'v_1}}"', from=3-1, to=3-2]
	\arrow["{{q'v_1}}"', from=3-5, to=3-6]
	\arrow["\SIGMA"{description}, draw=none, from=0, to=1]
	\arrow["{{\text{\large $=$}}}"{pos=0.7}, draw=none, from=2, to=2-5]
	\arrow["\SIGMA"{description}, draw=none, from=3, to=4]
	\arrow["{{\SIGMA^{\id}}}"{description}, draw=none, from=5, to=6]
	\arrow["{{\gamma_2}}"{pos=0.4}, between={0}{0.8}, Rightarrow, from=2-2, to=2]
	\arrow["{{\SIGMA^{\id}}}"{description}, draw=none, from=7, to=8]
\end{tikzcd}$
				\hskip2mm with $\gamma_2$ invertible
				 {\hfill by (v) and Rule 3b.}
					\vskip2mm
				
				\hskip10mm \rule{250pt}{0.5pt}
				
				\vskip2mm
			\end{enumerate}
		
		Thus, setting $d_x = q'qa'$, $d_y = q'qa$ and $\gamma_1=q'\gamma$, and using $\gamma_2$ as above, we obtain the desired result.
		
		\item[4'.] This is immediate from Rule 4. Indeed unfolding symmetrically each one of the $\Sigma$-squares, we get a particular case of Rule 4.
		
		\item[5.]
    Use Square to obtain successively
    \[\begin{tikzcd}
	B & X \\
	C & {D_1}
	\arrow["v", from=1-1, to=1-2]
	\arrow[""{name=0, anchor=center, inner sep=0}, "f"', from=1-1, to=2-1]
	\arrow[""{name=1, anchor=center, inner sep=0}, "{\tilde{f}}", from=1-2, to=2-2]
	\arrow["{{{w_1}}}"', from=2-1, to=2-2]
	\arrow["{{{\SIGMA^{\delta}}}}"{description}, draw=none, from=0, to=1]
\end{tikzcd}\, , \hspace{1cm}
\begin{tikzcd}
	B & X \\
	C & {D_2}
	\arrow["v", from=1-1, to=1-2]
	\arrow[""{name=0, anchor=center, inner sep=0}, "g"', from=1-1, to=2-1]
	\arrow[""{name=1, anchor=center, inner sep=0}, "{\tilde{g}}", from=1-2, to=2-2]
	\arrow["{{{w_2}}}"', from=2-1, to=2-2]
	\arrow["{{\SIGMA^{\epsilon}}}"{description}, draw=none, from=0, to=1]
\end{tikzcd}
\hspace{1cm} \text{and} \hspace{1cm}
\begin{tikzcd}
	C & {D_1} \\
	{D_2} & D
	\arrow["{{{w_1}}}", from=1-1, to=1-2]
	\arrow[""{name=0, anchor=center, inner sep=0}, "{{{w_2}}}"', from=1-1, to=2-1]
	\arrow[""{name=1, anchor=center, inner sep=0}, "{{{d_1}}}", from=1-2, to=2-2]
	\arrow["{{{d_2}}}"', from=2-1, to=2-2]
	\arrow["{{{\SIGMA^{\phi}}}}"{description}, draw=none, from=0, to=1]
\end{tikzcd}\,.
\]

Now using Rule 2b and Rule 1, we have:
\[\begin{tikzcd}
	B & X &&& B & X \\
	C & {D_1} &&& C & {D_2} \\
	C & D &&& C & D
	\arrow["v", from=1-1, to=1-2]
	\arrow[""{name=0, anchor=center, inner sep=0}, "f"', from=1-1, to=2-1]
	\arrow[""{name=1, anchor=center, inner sep=0}, "{{\tilde{f}}}", from=1-2, to=2-2]
	\arrow["v", from=1-5, to=1-6]
	\arrow[""{name=2, anchor=center, inner sep=0}, "g"', from=1-5, to=2-5]
	\arrow[""{name=3, anchor=center, inner sep=0}, "{{\tilde{g}}}", from=1-6, to=2-6]
	\arrow["{{w_1}}"', from=2-1, to=2-2]
	\arrow[""{name=4, anchor=center, inner sep=0}, Rightarrow, no head, from=2-1, to=3-1]
	\arrow["{{\text{\normalsize and}}}"{description}, draw=none, from=2-2, to=2-5]
	\arrow[""{name=5, anchor=center, inner sep=0}, "{{d_1}}", from=2-2, to=3-2]
	\arrow["{{w_2}}"', from=2-5, to=2-6]
	\arrow[""{name=6, anchor=center, inner sep=0}, Rightarrow, no head, from=2-5, to=3-5]
	\arrow[""{name=7, anchor=center, inner sep=0}, "{{d_2}}", from=2-6, to=3-6]
	\arrow["{{d_2w_2}}"', from=3-1, to=3-2]
	\arrow["{{d_2w_2}}"', from=3-5, to=3-6]
	\arrow["{\SIGMA^{\delta}}"{description}, draw=none, from=0, to=1]
	\arrow["{\SIGMA^{\epsilon}}"{description}, draw=none, from=2, to=3]
	\arrow["{\SIGMA^{\phi}}"{description}, draw=none, from=4, to=5]
	\arrow["{\SIGMA^{\id}}"{description}, draw=none, from=6, to=7]
\end{tikzcd}\, .\]
\item[6.] First, departing from $v,\, f$ and $g$, obtain the $\Sigma$-squares as in Rule 5 with
$\bar{\delta}$, $\bar{\epsilon}$, $\bar{D}$, $\bar{w}$, $\bar{f}$ and $\bar{g}$ instead of $\delta$, $\epsilon$, $D$, $w$, $f'$ and $g'$, respectively. Then we have a 2-cell
$\mu=\big(\bar{f}v\xRightarrow{\bar{\delta}^{-1}}\bar{w}f\xRightarrow{\bar{w}\beta}\bar{w}g\xRightarrow{\bar{\epsilon}}\bar{g}v)$.
  By Equi-insertion we get $d:\bar{D}\to D$, forming a $\Sigma$-square with $\bar{w}$,  and a 2-cell $d\bar{f}\xRightarrow{\mu'} d\bar{g}$ such that $\mu'\circ v=d\circ \mu$.  The desired $\Sigma$-squares $\Sigma^{\delta}$ and $\Sigma^{\epsilon}$, the 1-cell $w$  and the  2-cell $\beta'$ are then given by
$\delta=d\circ \bar{\delta}$, $\epsilon=d\circ \bar{\epsilon}$, $w=d\bar{w}$ and $\beta'=\mu'$. \qedhere
\end{enumerate}
\end{proof}

\section{The bicategory of lax fractions}\label{sec:bicategory}

Let $\catx$ be a 2-category and let $\Sigma$ be a subcategory of  $\catx^{\rightarrow \cong}$  admitting a calculus of left lax fractions. This section is devoted to the description of the \textbf{\em bicategory of lax fractions} $\catx[\Sigma_{\ast}]$.

The classical calculus of fractions with respect to a class of morphisms $\Sigma$ in a category $\catx$, introduced by Gabriel and Zisman  \cite{gabriel1967calculus} provides a nice description of a category $\catx[\Sigma^{-1}]$ and a functor ${P_{\Sigma}} \colon\catx \to \catx[\Sigma^{-1}]$ such that the images of morphisms from $\Sigma$ under $P_{\Sigma}$ are all invertible and $P_{\Sigma}$ is universal with respect to this property. Our definition of
$\catx[\Sigma_{\ast}]$ gives a generalisation of the classical case. In the next section,  we define a pseudofunctor $P_{\Sigma}\colon \catx \to \catx[\Sigma_{\ast}]$ which freely adds to each $P_{\Sigma}(s)$ with $s\in \text{ob}(\Sigma)$ a right adjoint making $P_{\Sigma}(s)$  a lari in  $\catx[\Sigma_{\ast}]$ and sends $\Sg$-squares to Beck-Chevalley squares. Moreover, $P_{\Sigma}$  is universal with respect to these properties.

The construction depends on a convenient calculus of $\Sigma$-squares based on the rules of \cref{def:TheCalculus}.

\subsection{The hom-categories}

The objects of $\catx[\Sigma_{\ast}]$ are just those of $\catx$, the 1-cells are  $\Sg$-cospans and the 2-cells are $\approx$-equivalence classes of 2-morphisms (see Definition \ref{data-0} and Definition \ref{data-1}). %
In this subsection we describe the hom-categories $\catx[\Sigma_{\ast}](A,B)$ for every pair of objects $A$ and $B$.

\begin{definition}\label{data-0}
A cospan $A\xrightarrow{f}I\xleftarrow{r}B$ with $r\in \Sigma$ is said to be a  \textbf{\em $\Sigma$-cospan} from $A$ to $B$, and is written $(f,I,r)$ or simply $(f,r)$. The \textbf{\em identity $\Sigma$-cospan} on an object $A$ is just $\xymatrix{A\ar[r]^{1_A}&A&A\ar[l]_{1_A}}$.
	
  Given two $\Sigma$-cospans $(f,I,r)$ and $(g,J,s)$, both from $A$ to $B$, a \textbf{\em 2-morphism} from $(f,I,r)$ to $(g,J,s)$ consists of two $\Sigma$-squares and a 2-cell $\alpha\colon x_1f\Rightarrow x_2g$ as in the following diagram:

	\begin{equation}\label{eq:2-morph}
		\begin{tikzcd}
		A && I && B \\
		&& X && B \\
		A && J && B
		\arrow["f", from=1-1, to=1-3]
		\arrow[Rightarrow, no head, from=1-1, to=3-1]
		\arrow[""{name=0, anchor=center, inner sep=0}, "{{{x_1}}}"', from=1-3, to=2-3]
		\arrow["\alpha"', shorten <=22pt, shorten >=22pt, Rightarrow, from=1-3, to=3-1]
		\arrow["r"', from=1-5, to=1-3]
		\arrow[""{name=1, anchor=center, inner sep=0}, Rightarrow, no head, from=1-5, to=2-5]
		\arrow["{x_3}"', from=2-5, to=2-3]
		\arrow["g"', from=3-1, to=3-3]
		\arrow[""{name=2, anchor=center, inner sep=0}, "{{{x_2}}}", from=3-3, to=2-3]
		\arrow[""{name=3, anchor=center, inner sep=0}, Rightarrow, no head, from=3-5, to=2-5]
		\arrow["s", from=3-5, to=3-3]
		\arrow["{{\SIGMA^{\delta_1}}}"{description}, draw=none, from=0, to=1]
		\arrow["{{\SIGMA^{\delta_2}}}"{description}, draw=none, from=2, to=3]
	\end{tikzcd}\; .
	\end{equation}
	We denote such a 2-morphism by
	$$(\alpha,x_1,x_2,x_3,\delta_1,\delta_2)\colon (f,I,r)\Rightarrow (g,J,s)$$
	or just $(\alpha,x_1,x_2)\colon (f,r)\Rightarrow (g,s)$, where the $\Sigma^{\delta_i}$ are obvious from the context.
\end{definition}

The $\Sg$-cospan $A\xrightarrow{f}I\xleftarrow{r}B$ should be interpreted as a composition of $f$ and a formal right adjoint retraction of $r$. The intuition behind the definition of 2-morphisms is that we can replace the two $\Sg$-cospans with isomorphic ones that have the same central object and right-hand side and then consider a 2-cell between the left-hand parts. There are different possible isomorphic replacements which should represent the same 2-cell and so the 2-cells in our bicategory are $\approx$-equivalence classes of these 2-morphisms for a convenient $\approx$-relation which we describe next.

\begin{definition}\label{data-1}
	\begin{enumerate}
		\item  A \textbf{\em $\Sigma$-extension} of a 2-morphism as in \eqref{eq:2-morph} above  is any 2-morphism of the form $((\theta_2\circ g)\cdot(d_x\circ \alpha)\cdot (\theta_1^{-1}\circ f), z_1,z_2)$ indicated by the wavy line part of the following diagram, where the 2-cells $\theta_i\colon d_xx_i\Rightarrow z_i$, $i=1,2$, are invertible.

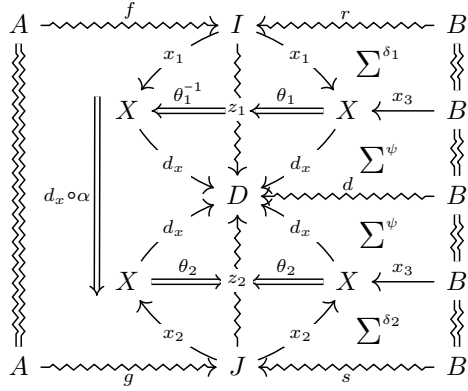
\[\begin{tikzcd}
	A && I && B \\
	& X & {} & X & B \\
	&& D && B \\
	& X & {} & X & B \\
	A && J && B
	\arrow[""{name=0, anchor=center, inner sep=0}, "f", squiggly, from=1-1, to=1-3]
	\arrow[Rightarrow, squiggly, no head, from=1-1, to=5-1]
	\arrow["{{{{{x_1}}}}}"{description}, curve={height=6pt}, from=1-3, to=2-2]
	\arrow[""{name=1, anchor=center, inner sep=0}, "{{{{x_1}}}}"{description}, curve={height=-6pt}, from=1-3, to=2-4]
	\arrow["{{{z_1}}}"{description}, squiggly, from=1-3, to=3-3]
	\arrow["{{{r}}}"', squiggly, from=1-5, to=1-3]
	\arrow[""{name=2, anchor=center, inner sep=0}, Rightarrow, squiggly, no head, from=1-5, to=2-5]
	\arrow["{d_x}"{description}, curve={height=6pt}, from=2-2, to=3-3]
	\arrow["{{{{{\theta_1^{-1}}}}}}"', Rightarrow, from=2-3, to=2-2]
	\arrow["{{{{{\theta_1}}}}}"', Rightarrow, from=2-4, to=2-3]
	\arrow[""{name=3, anchor=center, inner sep=0}, "{d_x}"{description}, curve={height=-6pt}, from=2-4, to=3-3]
	\arrow["{{{{{x_3}}}}}"', from=2-5, to=2-4]
	\arrow[""{name=4, anchor=center, inner sep=0}, Rightarrow, squiggly, no head, from=2-5, to=3-5]
	\arrow["d"', squiggly, from=3-5, to=3-3]
	\arrow[""{name=5, anchor=center, inner sep=0}, Rightarrow, squiggly, no head, from=3-5, to=4-5]
	\arrow["{d_x}"{description}, curve={height=-6pt}, from=4-2, to=3-3]
	\arrow["{{{{{\theta_2}}}}}", Rightarrow, from=4-2, to=4-3]
	\arrow[""{name=6, anchor=center, inner sep=0}, "{d_x}"{description}, curve={height=6pt}, from=4-4, to=3-3]
	\arrow["{{{{{\theta_2}}}}}"', Rightarrow, from=4-4, to=4-3]
	\arrow["{{{{{x_3}}}}}"', from=4-5, to=4-4]
	\arrow[""{name=7, anchor=center, inner sep=0}, Rightarrow, squiggly, no head, from=4-5, to=5-5]
	\arrow[""{name=8, anchor=center, inner sep=0}, "g"', squiggly, from=5-1, to=5-3]
	\arrow["{{{z_2}}}"{description}, squiggly, from=5-3, to=3-3]
	\arrow["{{{{{x_2}}}}}"{description}, curve={height=-6pt}, from=5-3, to=4-2]
	\arrow[""{name=9, anchor=center, inner sep=0}, "{{{{x_2}}}}"{description}, curve={height=6pt}, from=5-3, to=4-4]
	\arrow["{{{s}}}", squiggly, from=5-5, to=5-3]
	\arrow["{d_x\circ\alpha}"', shift right=5, shorten <=26pt, shorten >=26pt, Rightarrow, from=0, to=8]
	\arrow["{{{{{\SIGMA^{\delta_1}}}}}}"{description}, draw=none, from=1, to=2]
	\arrow["{{{{{\SIGMA^{\psi}}}}}}"{description}, draw=none, from=3, to=4]
	\arrow["{{{{\SIGMA^{\psi}}}}}"{description}, draw=none, from=6, to=5]
	\arrow["{{{{{\SIGMA^{\delta_2}}}}}}"{description}, draw=none, from=9, to=7]
\end{tikzcd}\]
		\item   We say that two 2-morphisms  with common domain and codomain are \textbf{\em $\approx$-related} if they have a common $\Sigma$-extension. This is clearly equivalent to saying that two 2-morphisms $(\alpha,x_1,x_2,x_3, \delta_1,\delta_2)$ and $(\beta, y_1,y_2,y_3, \epsilon_1,\epsilon_2)$ from $(f,r)$ to $(g,s)$ are $\approx$-related if there are $\Sigma$-squares
		\[\begin{tikzcd}
			B & X \\
			B & D
			\arrow["{x_3}", from=1-1, to=1-2]
			\arrow[""{name=0, anchor=center, inner sep=0}, Rightarrow, no head, from=1-1, to=2-1]
			\arrow[""{name=1, anchor=center, inner sep=0}, "{d_x}", from=1-2, to=2-2]
			\arrow["d"', from=2-1, to=2-2]
			\arrow["{\SIGMA^{\chi}}"{description}, draw=none, from=0, to=1]
		\end{tikzcd}
		\quad\text{and}\quad
		\begin{tikzcd}
			B & X \\
			B & D
			\arrow["{y_3}", from=1-1, to=1-2]
			\arrow[""{name=0, anchor=center, inner sep=0}, Rightarrow, no head, from=1-1, to=2-1]
			\arrow[""{name=1, anchor=center, inner sep=0}, "{d_y}", from=1-2, to=2-2]
			\arrow["d"', from=2-1, to=2-2]
			\arrow["{\SIGMA^{\psi}}"{description}, draw=none, from=0, to=1]
		\end{tikzcd}\]
		and invertible 2-cells $\gamma_i:d_xx_i\Rightarrow d_yy_i$ such that the 2-morphisms given by the wavy lines in the following diagram are equal.

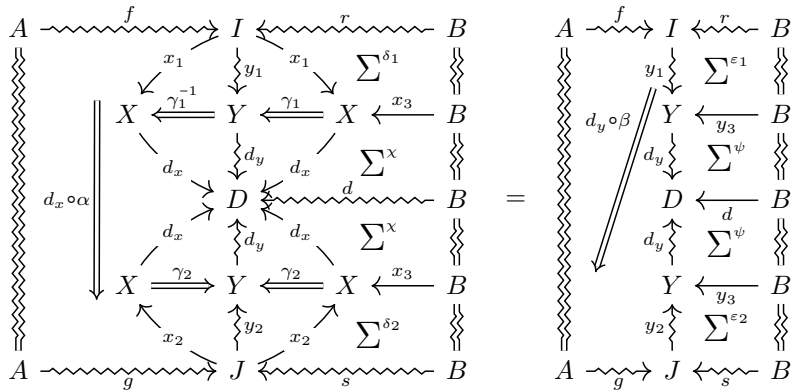
\[\begin{tikzcd}
	A && I && B & A & I & B \\
	& X & Y & X & B && Y & B \\
	&& D && B & {} & D & B &&& {} \\
	& X & Y & X & B && Y & B \\
	A && J && B & A & J & B
	\arrow[""{name=0, anchor=center, inner sep=0}, "f", squiggly, from=1-1, to=1-3]
	\arrow[Rightarrow, squiggly, no head, from=1-1, to=5-1]
	\arrow["{{{{{{{x_1}}}}}}}"{description}, curve={height=6pt}, from=1-3, to=2-2]
	\arrow["{{{{y_1}}}}", squiggly, from=1-3, to=2-3]
	\arrow[""{name=1, anchor=center, inner sep=0}, "{{{{{{x_1}}}}}}"{description}, curve={height=-6pt}, from=1-3, to=2-4]
	\arrow["{{{{{r}}}}}"', squiggly, from=1-5, to=1-3]
	\arrow[""{name=2, anchor=center, inner sep=0}, Rightarrow, squiggly, no head, from=1-5, to=2-5]
	\arrow["f", squiggly, from=1-6, to=1-7]
	\arrow[Rightarrow, squiggly, no head, from=1-6, to=5-6]
	\arrow[""{name=3, anchor=center, inner sep=0}, "{y_1}"', squiggly, from=1-7, to=2-7]
	\arrow["{{d_y\circ \beta}}"'{pos=0.3}, shorten <=16pt, shorten >=31pt, Rightarrow, from=1-7, to=5-6]
	\arrow["r"', squiggly, from=1-8, to=1-7]
	\arrow[""{name=4, anchor=center, inner sep=0}, Rightarrow, squiggly, no head, from=1-8, to=2-8]
	\arrow["{{{{{{{d_x}}}}}}}"{description}, curve={height=6pt}, from=2-2, to=3-3]
	\arrow["{{{{\gamma_1^{-1}}}}}"', Rightarrow, from=2-3, to=2-2]
	\arrow["{{{{d_y}}}}", squiggly, from=2-3, to=3-3]
	\arrow["{{{{\gamma_1}}}}"', Rightarrow, from=2-4, to=2-3]
	\arrow[""{name=5, anchor=center, inner sep=0}, "{{{{{{d_x}}}}}}"{description}, curve={height=-6pt}, from=2-4, to=3-3]
	\arrow["{{{{{{{x_3}}}}}}}"', from=2-5, to=2-4]
	\arrow[""{name=6, anchor=center, inner sep=0}, Rightarrow, squiggly, no head, from=2-5, to=3-5]
	\arrow[""{name=7, anchor=center, inner sep=0}, "{d_y}"', squiggly, from=2-7, to=3-7]
	\arrow["{{y_3}}", from=2-8, to=2-7]
	\arrow[""{name=8, anchor=center, inner sep=0}, Rightarrow, squiggly, no head, from=2-8, to=3-8]
	\arrow["d"', squiggly, from=3-5, to=3-3]
	\arrow["{{\text{\normalsize =}}}"{description}, draw=none, from=3-5, to=3-6]
	\arrow[""{name=9, anchor=center, inner sep=0}, Rightarrow, squiggly, no head, from=3-5, to=4-5]
	\arrow["d", from=3-8, to=3-7]
	\arrow["{{{{{{{d_x}}}}}}}"{description}, curve={height=-6pt}, from=4-2, to=3-3]
	\arrow["{{{{\gamma_2}}}}", Rightarrow, from=4-2, to=4-3]
	\arrow["{{{{d_y}}}}"', squiggly, from=4-3, to=3-3]
	\arrow[""{name=10, anchor=center, inner sep=0}, "{{{{{{d_x}}}}}}"{description}, curve={height=6pt}, from=4-4, to=3-3]
	\arrow["{{{{\gamma_2}}}}"', Rightarrow, from=4-4, to=4-3]
	\arrow["{{{{{{{x_3}}}}}}}"', from=4-5, to=4-4]
	\arrow[""{name=11, anchor=center, inner sep=0}, Rightarrow, squiggly, no head, from=4-5, to=5-5]
	\arrow[""{name=12, anchor=center, inner sep=0}, "{d_y}", squiggly, from=4-7, to=3-7]
	\arrow[""{name=13, anchor=center, inner sep=0}, Rightarrow, squiggly, no head, from=4-8, to=3-8]
	\arrow["{{y_3}}", from=4-8, to=4-7]
	\arrow[""{name=14, anchor=center, inner sep=0}, "g"', squiggly, from=5-1, to=5-3]
	\arrow["{{{{{{{x_2}}}}}}}"{description}, curve={height=-6pt}, from=5-3, to=4-2]
	\arrow["{{{{y_2}}}}"', squiggly, from=5-3, to=4-3]
	\arrow[""{name=15, anchor=center, inner sep=0}, "{{{{{{x_2}}}}}}"{description}, curve={height=6pt}, from=5-3, to=4-4]
	\arrow["{{{{{s}}}}}", squiggly, from=5-5, to=5-3]
	\arrow["g"', squiggly, from=5-6, to=5-7]
	\arrow[""{name=16, anchor=center, inner sep=0}, "{y_2}", squiggly, from=5-7, to=4-7]
	\arrow[""{name=17, anchor=center, inner sep=0}, Rightarrow, squiggly, no head, from=5-8, to=4-8]
	\arrow["s", squiggly, from=5-8, to=5-7]
	\arrow["{{{{{{{d_x\circ\alpha}}}}}}}"', shift right=5, shorten <=26pt, shorten >=26pt, Rightarrow, from=0, to=14]
	\arrow["{{{{{{{\SIGMA^{\delta_1}}}}}}}}"{description}, draw=none, from=1, to=2]
	\arrow["{\SIGMA^{\epsilon_1}}"{description}, draw=none, from=3, to=4]
	\arrow["{{{{\SIGMA^{\chi}}}}}"{description}, draw=none, from=5, to=6]
	\arrow["{\SIGMA^{\psi}}"{description}, draw=none, from=7, to=8]
	\arrow["{{{{\SIGMA^{\chi}}}}}"{description}, draw=none, from=10, to=9]
	\arrow["{\SIGMA^{\psi}}"{description}, draw=none, from=12, to=13]
	\arrow["{{{{{{{\SIGMA^{\delta_2}}}}}}}}"{description}, draw=none, from=15, to=11]
	\arrow["{\SIGMA^{\epsilon_2}}"{description}, draw=none, from=16, to=17]
\end{tikzcd}\]
	\end{enumerate}
\end{definition}

\begin{lemma}\label{lem:first-rel-on-2} The relation $\approx$ is an equivalence relation.
\end{lemma}

\begin{proof} Reflexivity and symmetry are obvious. To show transitivity, consider three 2-morphisms
	$\bar{\alpha}=(\alpha,x_1,x_2,x_3,\delta_1,\delta_2)$, $\bar{\beta}=(\beta, y_1,y_2,y_3,\epsilon_1,\epsilon_2)$ and $\bar{\gamma}=(\gamma, z_1,z_2,z_3, \zeta_1,\zeta_2)$ from  $(f,I_1,r_1)$ to $(g,I_2,r_2)$, with
	$\bar{\alpha}\approx \bar{\beta}$ and $\bar{\beta}\approx \bar{\gamma}$.
	Let the $\approx$-relation between $\bar{\alpha}$ and $\bar{\beta}$, and the one between $\bar{\beta}$ and $\bar{\gamma}$, be given respectively by the data represented by \eqref{eq:(a)} and \eqref{eq:(b)} below.
	\begin{equation}\label{eq:(a)}\begin{tikzcd}
			A && {I_1} && B & A & {I_1} & B \\
			& X & Y & X & B && Y & B \\
			&& D && B & {} & D & B \\
			& X & Y & X & B && Y & B \\
			A && D && B & A & {I_2} & B
			\arrow[""{name=0, anchor=center, inner sep=0}, "f", squiggly, from=1-1, to=1-3]
			\arrow[Rightarrow, squiggly, no head, from=1-1, to=5-1]
			\arrow["{{{x_1}}}"{description}, curve={height=6pt}, from=1-3, to=2-2]
			\arrow["{{{y_1}}}"', squiggly, from=1-3, to=2-3]
			\arrow[""{name=1, anchor=center, inner sep=0}, "{{{x_1}}}"{description}, curve={height=-6pt}, from=1-3, to=2-4]
			\arrow["{{{r_1}}}"', squiggly, from=1-5, to=1-3]
			\arrow[""{name=2, anchor=center, inner sep=0}, Rightarrow, squiggly, no head, from=1-5, to=2-5]
				\arrow[""{name=3, anchor=center, inner sep=0}, "f", squiggly, from=1-6, to=1-7]
			\arrow[Rightarrow, squiggly, no head, from=1-6, to=5-6]
			\arrow[""{name=4, anchor=center, inner sep=0}, "{{y_1}}"', squiggly, from=1-7, to=2-7]
			\arrow["{{r_1}}"', squiggly, from=1-8, to=1-7]
			\arrow[""{name=5, anchor=center, inner sep=0}, Rightarrow, squiggly, no head, from=1-8, to=2-8]
			\arrow["{{{d_x}}}"{description}, curve={height=6pt}, from=2-2, to=3-3]
			\arrow["{{{\lambda_{11}^{-1}}}}"', Rightarrow, from=2-3, to=2-2]
			\arrow["{{{d_y}}}"', squiggly, from=2-3, to=3-3]
			\arrow["{{{\lambda_{11}}}}"', Rightarrow, from=2-4, to=2-3]
			\arrow[""{name=6, anchor=center, inner sep=0}, "{{{d_x}}}"{description}, curve={height=-6pt}, from=2-4, to=3-3]
			\arrow["{{{x_3}}}"', from=2-5, to=2-4]
			\arrow[""{name=7, anchor=center, inner sep=0}, Rightarrow, squiggly, no head, from=2-5, to=3-5]
			\arrow[""{name=8, anchor=center, inner sep=0}, "{{d_y}}"', squiggly, from=2-7, to=3-7]
			\arrow["{{y_3}}"', from=2-8, to=2-7]
			\arrow[""{name=9, anchor=center, inner sep=0}, Rightarrow, squiggly, no head, from=2-8, to=3-8]
			\arrow["d", squiggly, from=3-5, to=3-3]
			\arrow["{{\text{\normalsize =}}}"{description}, draw=none, from=3-5, to=3-6]
			\arrow[""{name=10, anchor=center, inner sep=0}, Rightarrow, squiggly, no head, from=3-5, to=4-5]
			\arrow["d"', from=3-8, to=3-7]
			\arrow["{{{d_x}}}"{description}, curve={height=-6pt}, from=4-2, to=3-3]
			\arrow["{{{\lambda_{12}}}}", Rightarrow, from=4-2, to=4-3]
			\arrow["{{{d_y}}}", squiggly, from=4-3, to=3-3]
			\arrow[""{name=11, anchor=center, inner sep=0}, "{{{d_x}}}"{description}, curve={height=6pt}, from=4-4, to=3-3]
			\arrow["{{{\lambda_{12}}}}"', Rightarrow, from=4-4, to=4-3]
			\arrow["{{{x_3}}}"', from=4-5, to=4-4]
			\arrow[""{name=12, anchor=center, inner sep=0}, Rightarrow, squiggly, no head, from=4-5, to=5-5]
			\arrow[""{name=13, anchor=center, inner sep=0}, "{{d_y}}", squiggly, from=4-7, to=3-7]
			\arrow[""{name=14, anchor=center, inner sep=0}, Rightarrow, squiggly, no head, from=4-8, to=3-8]
			\arrow["{{y_3}}"', from=4-8, to=4-7]
			\arrow[""{name=15, anchor=center, inner sep=0}, "g"', squiggly, from=5-1, to=5-3]
			\arrow["{{{x_2}}}"{description}, curve={height=-6pt}, from=5-3, to=4-2]
			\arrow["{{{y_2}}}", squiggly, from=5-3, to=4-3]
			\arrow[""{name=16, anchor=center, inner sep=0}, "{{{x_2}}}"{description}, curve={height=6pt}, from=5-3, to=4-4]
			\arrow["{{{r_2}}}", squiggly, from=5-5, to=5-3]
			\arrow[""{name=17, anchor=center, inner sep=0}, "g"', squiggly, from=5-6, to=5-7]
			\arrow[""{name=18, anchor=center, inner sep=0}, "{{y_2}}", squiggly, from=5-7, to=4-7]
			\arrow[""{name=19, anchor=center, inner sep=0}, Rightarrow, squiggly, no head, from=5-8, to=4-8]
			\arrow["{{r_2}}", squiggly, from=5-8, to=5-7]
			\arrow["{{{d_x\circ \alpha}}}"', shift right=5, shorten <=17pt, shorten >=17pt, Rightarrow, from=0, to=15]
			\arrow["{{{\SIGMA^{\delta_1}}}}"{description}, draw=none, from=1, to=2]
			\arrow["{{{d_y\circ \beta}}}"{description}, shorten <=17pt, shorten >=17pt, Rightarrow, from=3, to=17]
			\arrow["{{{\SIGMA^{\epsilon_1}}}}"{description}, draw=none, from=4, to=5]
			\arrow["{{\SIGMA^{\phi_x}}}"{description}, draw=none, from=6, to=7]
			\arrow["{{{\SIGMA^{\phi_y}}}}"{description}, draw=none, from=8, to=9]
			\arrow["{{{\SIGMA^{\phi_x}}}}"{description}, draw=none, from=11, to=10]
			\arrow["{{{\SIGMA^{\phi_y}}}}"{description}, draw=none, from=13, to=14]
			\arrow["{{{\SIGMA^{\delta_2}}}}"{description}, draw=none, from=16, to=12]
			\arrow["{{{\SIGMA^{\epsilon_2}}}}"{description}, draw=none, from=18, to=19]
		\end{tikzcd}
	\end{equation}
	\begin{equation}\label{eq:(b)}
	\begin{tikzcd}
	A && {I_1} && B & A & {I_1} & B \\
	& Y & Z & Y & B && Z & B \\
	&& E && B & {} & E & B \\
	& Y & Z & Y & B && Z & B \\
	A && D && B & A & {I_2} & B
	\arrow[""{name=0, anchor=center, inner sep=0}, "f", squiggly, from=1-1, to=1-3]
	\arrow[Rightarrow, squiggly, no head, from=1-1, to=5-1]
	\arrow["{{{{{{y_1}}}}}}"{description}, curve={height=6pt}, from=1-3, to=2-2]
	\arrow["{{{{{{z_1}}}}}}"', squiggly, from=1-3, to=2-3]
	\arrow[""{name=1, anchor=center, inner sep=0}, "{{{{{{y_1}}}}}}"{description}, curve={height=-6pt}, from=1-3, to=2-4]
	\arrow["{{{{{{r_1}}}}}}"', squiggly, from=1-5, to=1-3]
	\arrow[""{name=2, anchor=center, inner sep=0}, Rightarrow, squiggly, no head, from=1-5, to=2-5]
	\arrow[""{name=3, anchor=center, inner sep=0}, "f", squiggly, from=1-6, to=1-7]
	\arrow[Rightarrow, squiggly, no head, from=1-6, to=5-6]
	\arrow[""{name=4, anchor=center, inner sep=0}, "{{{{{z_1}}}}}"', squiggly, from=1-7, to=2-7]
	\arrow["{{{{{r_1}}}}}"', squiggly, from=1-8, to=1-7]
	\arrow[""{name=5, anchor=center, inner sep=0}, Rightarrow, squiggly, no head, from=1-8, to=2-8]
	\arrow["{{{{{{e_y}}}}}}"{description}, curve={height=6pt}, from=2-2, to=3-3]
	\arrow["{{{{{{\lambda_{21}^{-1}}}}}}}"', Rightarrow, from=2-3, to=2-2]
	\arrow["{{{{{{e_z}}}}}}"', squiggly, from=2-3, to=3-3]
	\arrow["{{{{{{\lambda_{21}}}}}}}"', Rightarrow, from=2-4, to=2-3]
	\arrow[""{name=6, anchor=center, inner sep=0}, "{{{{{{e_y}}}}}}"{description}, curve={height=-6pt}, from=2-4, to=3-3]
	\arrow["{{{{{{y_3}}}}}}"', from=2-5, to=2-4]
	\arrow[""{name=7, anchor=center, inner sep=0}, Rightarrow, squiggly, no head, from=2-5, to=3-5]
	\arrow[""{name=8, anchor=center, inner sep=0}, "{{{{{e_z}}}}}"', squiggly, from=2-7, to=3-7]
	\arrow["{{{{{z_3}}}}}"', from=2-8, to=2-7]
	\arrow[""{name=9, anchor=center, inner sep=0}, Rightarrow, squiggly, no head, from=2-8, to=3-8]
	\arrow["e", squiggly, from=3-5, to=3-3]
	\arrow["{{{{{\text{\normalsize =}}}}}}"{description}, draw=none, from=3-5, to=3-6]
	\arrow[""{name=10, anchor=center, inner sep=0}, Rightarrow, squiggly, no head, from=3-5, to=4-5]
	\arrow["e"', from=3-8, to=3-7]
	\arrow["{{{{{{e_y}}}}}}"{description}, curve={height=-6pt}, from=4-2, to=3-3]
	\arrow["{{{{{{\lambda_{22}}}}}}}", Rightarrow, from=4-2, to=4-3]
	\arrow["{{{{{{e_z}}}}}}", squiggly, from=4-3, to=3-3]
	\arrow[""{name=11, anchor=center, inner sep=0}, "{{{{{{e_y}}}}}}"{description}, curve={height=6pt}, from=4-4, to=3-3]
	\arrow["{{{{{{\lambda_{22}}}}}}}"', Rightarrow, from=4-4, to=4-3]
	\arrow["{{{{{{y_3}}}}}}"', from=4-5, to=4-4]
	\arrow[""{name=12, anchor=center, inner sep=0}, Rightarrow, squiggly, no head, from=4-5, to=5-5]
	\arrow[""{name=13, anchor=center, inner sep=0}, "{{{{{e_z}}}}}", squiggly, from=4-7, to=3-7]
	\arrow[""{name=14, anchor=center, inner sep=0}, Rightarrow, squiggly, no head, from=4-8, to=3-8]
	\arrow["{{{{{z_3}}}}}"', from=4-8, to=4-7]
	\arrow[""{name=15, anchor=center, inner sep=0}, "g"', squiggly, from=5-1, to=5-3]
	\arrow["{{{{{{y_2}}}}}}"{description}, curve={height=-6pt}, from=5-3, to=4-2]
	\arrow["{{{{{{z_2}}}}}}", squiggly, from=5-3, to=4-3]
	\arrow[""{name=16, anchor=center, inner sep=0}, "{{{{{{y_2}}}}}}"{description}, curve={height=6pt}, from=5-3, to=4-4]
	\arrow["{{{{{{r_2}}}}}}", squiggly, from=5-5, to=5-3]
	\arrow[""{name=17, anchor=center, inner sep=0}, "g"', squiggly, from=5-6, to=5-7]
	\arrow[""{name=18, anchor=center, inner sep=0}, "{{{{{z_2}}}}}", squiggly, from=5-7, to=4-7]
	\arrow[""{name=19, anchor=center, inner sep=0}, Rightarrow, squiggly, no head, from=5-8, to=4-8]
	\arrow["{{{{{r_2}}}}}", squiggly, from=5-8, to=5-7]
	\arrow["{{{{{{e_y\circ \beta}}}}}}"', shift right=5, shorten <=17pt, shorten >=17pt, Rightarrow, from=0, to=15]
	\arrow["{{{{{{\SIGMA^{\epsilon_1}}}}}}}"{description}, draw=none, from=1, to=2]
	\arrow["{{{{{{e_z\circ \gamma}}}}}}"{description}, shorten <=17pt, shorten >=17pt, Rightarrow, from=3, to=17]
	\arrow["{{{{{{\SIGMA^{\zeta_1}}}}}}}"{description}, draw=none, from=4, to=5]
	\arrow["{{{{{\SIGMA^{\chi_y}}}}}}"{description}, draw=none, from=6, to=7]
	\arrow["{{{{{{\SIGMA^{\chi_z}}}}}}}"{description}, draw=none, from=8, to=9]
	\arrow["{{{{{{\SIGMA^{\chi_y}}}}}}}"{description}, draw=none, from=11, to=10]
	\arrow["{{{{{{\SIGMA^{\chi_z}}}}}}}"{description}, draw=none, from=13, to=14]
	\arrow["{{{{{{\SIGMA^{\epsilon_2}}}}}}}"{description}, draw=none, from=16, to=12]
	\arrow["{{{{{{\SIGMA^{\zeta_2}}}}}}}"{description}, draw=none, from=18, to=19]
\end{tikzcd}
	\end{equation}
	Using the $\Sigma$-squares $\Sigma^{\phi_y}$ and  $\Sigma^{\chi_y}$ and Rule 4' of Proposition \ref{pro:useful_rules}, we obtain $\Sigma$-squares and an invertible 2-cell $\theta$ such that
	\[
	\begin{tikzcd}
		B & Y \\
		B & D \\
		B & T
		\arrow["{y_3}", from=1-1, to=1-2]
		\arrow[""{name=0, anchor=center, inner sep=0}, Rightarrow, no head, from=1-1, to=2-1]
		\arrow[""{name=1, anchor=center, inner sep=0}, "{{d_y}}", from=1-2, to=2-2]
		\arrow[""{name=2, anchor=center, inner sep=0}, "{{t_2e_y}}", curve={height=-24pt}, from=1-2, to=3-2]
		\arrow["d"', from=2-1, to=2-2]
		\arrow[""{name=3, anchor=center, inner sep=0}, Rightarrow, no head, from=2-1, to=3-1]
		\arrow[""{name=4, anchor=center, inner sep=0}, "{{t_1}}", from=2-2, to=3-2]
		\arrow["t"', from=3-1, to=3-2]
		\arrow["{\SIGMA^{\phi_y}}"{description}, draw=none, from=0, to=1]
		\arrow["{{\SIGMA^{\eta_1}}}"{description}, draw=none, from=3, to=4]
		\arrow["\theta", shorten >=3pt, Rightarrow, from=2-2, to=2]
		\end{tikzcd}
\; \text{\normalsize $=$}\;
\begin{tikzcd}
	B & Y \\
	B & E \\
	B & T \rlap{\,.}
	\arrow["{y_3}", from=1-1, to=1-2]
	\arrow[""{name=0, anchor=center, inner sep=0}, Rightarrow, no head, from=1-1, to=2-1]
	\arrow[""{name=1, anchor=center, inner sep=0}, "{{{{e_y}}}}", from=1-2, to=2-2]
	\arrow["e"', from=2-1, to=2-2]
	\arrow[""{name=2, anchor=center, inner sep=0}, Rightarrow, no head, from=2-1, to=3-1]
	\arrow[""{name=3, anchor=center, inner sep=0}, "{t_2}", from=2-2, to=3-2]
	\arrow["t"', from=3-1, to=3-2]
	\arrow["{\SIGMA^{\chi_y}}"{description}, draw=none, from=0, to=1]
	\arrow["{{{{\SIGMA^{\eta_2}}}}}"{description}, draw=none, from=2, to=3]
\end{tikzcd}
\]

Consequently, we obtain a common $\Sigma$-extension of $\bar{\alpha}$ and $\bar{\gamma}$. Namely, for\newline
 $\mu_i=(t_2\circ \lambda_{2i})(\theta\circ y_i)(t_1\circ \lambda_{1i})$, we get the equality

\[\begin{tikzcd}
	A && {I_1} && B & A & {I_1} & B \\
	& X & Z & X & B && Z & B \\
	&& T && B & {} & T & B \\
	& X & Z & X & B && Z & B \\
	A && D && B & A & {I_2} & B
	\arrow[""{name=0, anchor=center, inner sep=0}, "f", squiggly, from=1-1, to=1-3]
	\arrow[Rightarrow, squiggly, no head, from=1-1, to=5-1]
	\arrow["{{{{{{{x_1}}}}}}}"{description}, curve={height=6pt}, from=1-3, to=2-2]
	\arrow["{{{{{z_1}}}}}"', squiggly, from=1-3, to=2-3]
	\arrow[""{name=1, anchor=center, inner sep=0}, "{{{{{{{x_1}}}}}}}"{description}, curve={height=-6pt}, from=1-3, to=2-4]
	\arrow["{{{{{{{r_1}}}}}}}"', squiggly, from=1-5, to=1-3]
	\arrow[""{name=2, anchor=center, inner sep=0}, Rightarrow, squiggly, no head, from=1-5, to=2-5]
	\arrow[""{name=3, anchor=center, inner sep=0}, "f", squiggly, from=1-6, to=1-7]
	\arrow[Rightarrow, squiggly, no head, from=1-6, to=5-6]
	\arrow[""{name=4, anchor=center, inner sep=0}, "{{{{{z_1}}}}}"', squiggly, from=1-7, to=2-7]
	\arrow["{{{{{{r_1}}}}}}"', squiggly, from=1-8, to=1-7]
	\arrow[""{name=5, anchor=center, inner sep=0}, Rightarrow, squiggly, no head, from=1-8, to=2-8]
	\arrow["{{{{{t_1d_x}}}}}"{description}, curve={height=6pt}, from=2-2, to=3-3]
	\arrow["{{{{{\mu_1^{-1}}}}}}"', Rightarrow, from=2-3, to=2-2]
	\arrow["{{{{{t_2ez}}}}}"{description}, squiggly, from=2-3, to=3-3]
	\arrow["{{{{{\mu_1}}}}}"', Rightarrow, from=2-4, to=2-3]
	\arrow[""{name=6, anchor=center, inner sep=0}, "{{{{{t_1d_x}}}}}"{description}, curve={height=-6pt}, from=2-4, to=3-3]
	\arrow["{{{{{{{x_3}}}}}}}"', from=2-5, to=2-4]
	\arrow[""{name=7, anchor=center, inner sep=0}, Rightarrow, squiggly, no head, from=2-5, to=3-5]
	\arrow[""{name=8, anchor=center, inner sep=0}, "{{{{{t_2e_z}}}}}"', squiggly, from=2-7, to=3-7]
	\arrow["{{{{{z_3}}}}}"', from=2-8, to=2-7]
	\arrow[""{name=9, anchor=center, inner sep=0}, Rightarrow, squiggly, no head, from=2-8, to=3-8]
	\arrow["t"', squiggly, from=3-5, to=3-3]
	\arrow["{{{{{{\text{\normalsize =}}}}}}}"{description}, draw=none, from=3-5, to=3-6]
	\arrow[""{name=10, anchor=center, inner sep=0}, Rightarrow, squiggly, no head, from=3-5, to=4-5]
	\arrow["t"', from=3-8, to=3-7]
	\arrow["{{{{{t_1d_x}}}}}"{description}, curve={height=-6pt}, from=4-2, to=3-3]
	\arrow["{{{{{\mu_2}}}}}", Rightarrow, from=4-2, to=4-3]
	\arrow["{{{{{t_2e_z}}}}}"{description}, squiggly, from=4-3, to=3-3]
	\arrow[""{name=11, anchor=center, inner sep=0}, "{{{{{t_1d_x}}}}}"{description}, curve={height=6pt}, from=4-4, to=3-3]
	\arrow["{{{{{\mu_2}}}}}"', Rightarrow, from=4-4, to=4-3]
	\arrow["{{{{{{{x_3}}}}}}}"', from=4-5, to=4-4]
	\arrow[""{name=12, anchor=center, inner sep=0}, Rightarrow, squiggly, no head, from=4-5, to=5-5]
	\arrow[""{name=13, anchor=center, inner sep=0}, "{{{{{t_2e_z}}}}}", squiggly, from=4-7, to=3-7]
	\arrow[""{name=14, anchor=center, inner sep=0}, Rightarrow, squiggly, no head, from=4-8, to=3-8]
	\arrow["{{{{{z_3}}}}}"', from=4-8, to=4-7]
	\arrow[""{name=15, anchor=center, inner sep=0}, "g"', squiggly, from=5-1, to=5-3]
	\arrow["{{{{{{{x_2}}}}}}}"{description}, curve={height=-6pt}, from=5-3, to=4-2]
	\arrow["{{{{{z_2}}}}}", squiggly, from=5-3, to=4-3]
	\arrow[""{name=16, anchor=center, inner sep=0}, "{{{{{{{x_2}}}}}}}"{description}, curve={height=6pt}, from=5-3, to=4-4]
	\arrow["{{{{{{{r_2}}}}}}}", squiggly, from=5-5, to=5-3]
	\arrow[""{name=17, anchor=center, inner sep=0}, "g"', squiggly, from=5-6, to=5-7]
	\arrow[""{name=18, anchor=center, inner sep=0}, "{{{{{z_2}}}}}", squiggly, from=5-7, to=4-7]
	\arrow[""{name=19, anchor=center, inner sep=0}, Rightarrow, squiggly, no head, from=5-8, to=4-8]
	\arrow["{{{{{{r_2}}}}}}", squiggly, from=5-8, to=5-7]
	\arrow["{{{{{t_1d_x\alpha}}}}}"', shift right=5, shorten <=17pt, shorten >=17pt, Rightarrow, from=0, to=15]
	\arrow["{{{{{{{\SIGMA^{\delta_1}}}}}}}}"{description}, draw=none, from=1, to=2]
	\arrow["{{{{{t_2e_z\gamma}}}}}"{description}, shorten <=17pt, shorten >=17pt, Rightarrow, from=3, to=17]
	\arrow["{{{{{\SIGMA^{\zeta_1}}}}}}"{description}, draw=none, from=4, to=5]
	\arrow["{{{\SIGMA^{\eta_1\odot\phi_x}}}}"{description}, draw=none, from=6, to=7]
	\arrow["{{{\SIGMA^{\eta_2\odot\chi_z}}}}"{description}, draw=none, from=8, to=9]
	\arrow["{{{\SIGMA^{\eta_1\odot\phi_x}}}}"{description}, draw=none, from=11, to=10]
	\arrow["{{{\SIGMA^{\eta_2\odot\chi_z}}}}"{description}, draw=none, from=13, to=14]
	\arrow["{{{{{{{\SIGMA^{\delta_2}}}}}}}}"{description}, draw=none, from=16, to=12]
	\arrow["{{{{{\SIGMA^{\zeta_2}}}}}}"{description}, draw=none, from=18, to=19]
\end{tikzcd}\]
showing that  $\bar{\alpha}\approx\bar{\gamma}$.
\end{proof}

\begin{definition}\label{def:2-cell}  A \textbf{\em 2-cell} between $\Sigma$-cospans is an $\approx$-equivalence class of 2-morphisms between them. Given a 2-morphism $\bar{\alpha}=(\alpha, x_1,x_2,x_3, \delta_1, \delta_2)$, we use the notation
$[\bar{\alpha}]$
to indicate the $\approx$-equivalence class of $\bar{\alpha}$. Sometimes we use  $[\alpha, x_1,x_2,x_3, \delta_1, \delta_2]$ or simply $[\alpha, x_1,x_2]$.
\end{definition}

We now describe the vertical composition of 2-cells between $\Sigma$-cospans. For this we first define the vertical composition between 2-morphisms and then take the corresponding $\approx$-equivalence class.

\begin{definition}\label{data-2}
	\textbf{\em  Vertical composition}.
	Let $\bar{\alpha}=(\alpha,x_1,x_2,x_3,\delta_1,\delta_2)\colon (f,r) \Rightarrow (g,s)$ and $\bar{\beta}=(\beta,y_1,y_2,y_3,\epsilon_1,\epsilon_2)\colon (g,s) \Rightarrow (h,t)$ be  2-morphisms, as illustrated  in the diagram

	\begin{equation}\label{EqV1}\begin{tikzcd}
			A & I & B \\
			& X & B \\
			A & J & B \\
			& Y & B \\
			A & K & B \rlap{\,.}
			\arrow["f", from=1-1, to=1-2]
			\arrow[Rightarrow, no head, from=1-1, to=3-1]
			\arrow[""{name=0, anchor=center, inner sep=0}, "{{{x_1}}}"', from=1-2, to=2-2]
			\arrow["\alpha"', shorten <=15pt, shorten >=15pt, Rightarrow, from=1-2, to=3-1]
			\arrow["r"', from=1-3, to=1-2]
			\arrow[""{name=1, anchor=center, inner sep=0}, Rightarrow, no head, from=1-3, to=2-3]
			\arrow["{x_3}", from=2-3, to=2-2]
			\arrow["g"', from=3-1, to=3-2]
			\arrow[Rightarrow, no head, from=3-1, to=5-1]
			\arrow[""{name=2, anchor=center, inner sep=0}, "{{{x_2}}}", from=3-2, to=2-2]
			\arrow[""{name=3, anchor=center, inner sep=0}, "{{{y_1}}}"', from=3-2, to=4-2]
			\arrow["\beta"', shorten <=15pt, shorten >=15pt, Rightarrow, from=3-2, to=5-1]
			\arrow[""{name=4, anchor=center, inner sep=0}, Rightarrow, no head, from=3-3, to=2-3]
			\arrow["s", from=3-3, to=3-2]
			\arrow[""{name=5, anchor=center, inner sep=0}, Rightarrow, no head, from=3-3, to=4-3]
			\arrow["{y_3}", from=4-3, to=4-2]
			\arrow[""{name=6, anchor=center, inner sep=0}, Rightarrow, no head, from=4-3, to=5-3]
			\arrow["h"', from=5-1, to=5-2]
			\arrow[""{name=7, anchor=center, inner sep=0}, "{{{y_2}}}", from=5-2, to=4-2]
			\arrow["t", from=5-3, to=5-2]
			\arrow["{{\SIGMA^{\delta_1}}}"{description}, draw=none, from=0, to=1]
			\arrow["{{\SIGMA^{\delta_2}}}"{description}, draw=none, from=2, to=4]
			\arrow["{{\SIGMA^{\epsilon_1}}}"{description}, draw=none, from=3, to=5]
			\arrow["{{\SIGMA^{\epsilon_2}}}"{description}, draw=none, from=7, to=6]
		\end{tikzcd}\;
	\end{equation}
	Using the $\Sigma$-squares $\Sigma^{\delta_2}$ and $\Sigma^{\epsilon_1}$ and Rule 4', we obtain  new $\Sigma$-squares and an invertible 2-cell $\gamma$ giving the following equality between pasting diagrams:
	\begin{equation}\label{EqV2}
		\begin{tikzcd}
			B & J \\
			B & X & Y \\
			B & D
			\arrow["s", from=1-1, to=1-2]
			\arrow[""{name=0, anchor=center, inner sep=0}, Rightarrow, no head, from=1-1, to=2-1]
			\arrow[""{name=1, anchor=center, inner sep=0}, "{{x_2}}", from=1-2, to=2-2]
			\arrow["{y_1}", curve={height=-6pt}, from=1-2, to=2-3]
			\arrow["{x_3}", from=2-1, to=2-2]
			\arrow[""{name=2, anchor=center, inner sep=0}, Rightarrow, no head, from=2-1, to=3-1]
			\arrow["\gamma", Rightarrow, from=2-2, to=2-3]
			\arrow["\cong"', draw=none, from=2-2, to=2-3]
			\arrow[""{name=3, anchor=center, inner sep=0}, "{{d_x}}", from=2-2, to=3-2]
			\arrow["{d_y}", curve={height=-6pt}, from=2-3, to=3-2]
			\arrow["u"', from=3-1, to=3-2]
			\arrow["{\SIGMA^{\delta_2}}"{description}, draw=none, from=0, to=1]
			\arrow["{\SIGMA^{\phi_x}}"{description}, draw=none, from=2, to=3]
		\end{tikzcd}
		\qquad=\qquad
	\begin{tikzcd}
		B & J \\
		B & Y \\
		B & D
		\arrow["s", from=1-1, to=1-2]
		\arrow[""{name=0, anchor=center, inner sep=0}, Rightarrow, no head, from=1-1, to=2-1]
		\arrow[""{name=1, anchor=center, inner sep=0}, "{{y_1}}", from=1-2, to=2-2]
		\arrow["{{y_3}}", from=2-1, to=2-2]
		\arrow[""{name=2, anchor=center, inner sep=0}, Rightarrow, no head, from=2-1, to=3-1]
		\arrow[""{name=3, anchor=center, inner sep=0}, "{{d_y}}", from=2-2, to=3-2]
		\arrow["u"', from=3-1, to=3-2]
		\arrow["{\SIGMA^{\epsilon_1}}"{description}, draw=none, from=0, to=1]
		\arrow["{\SIGMA^{\phi_y}}"{description}, draw=none, from=2, to=3]
	\end{tikzcd}\, .
	\end{equation}
Any such data gives a vertical composition $\bar{\beta}\cdot \bar{\alpha}$ represented by the following diagram.

\begin{equation}\label{eq:vert_comp}
\begin{tikzcd}
	A && I & B \\
	&& X & B \\
	A & J & D & B \\
	&& Y & B \\
	A && K & B
	\arrow["f"{description}, from=1-1, to=1-3]
	\arrow[equals, from=1-1, to=3-1]
	\arrow[""{name=0, anchor=center, inner sep=0}, "{{{x_1}}}", from=1-3, to=2-3]
	\arrow["\alpha"', between={0.3}{0.7}, Rightarrow, from=1-3, to=3-1]
	\arrow["r"', from=1-4, to=1-3]
	\arrow[""{name=1, anchor=center, inner sep=0}, equals, from=1-4, to=2-4]
	\arrow[""{name=2, anchor=center, inner sep=0}, "{{d_x}}"', from=2-3, to=3-3]
	\arrow["{{{x_3}}}"{description}, from=2-4, to=2-3]
	\arrow[""{name=3, anchor=center, inner sep=0}, equals, from=2-4, to=3-4]
	\arrow["g", from=3-1, to=3-2]
	\arrow[equals, from=3-1, to=5-1]
	\arrow["\beta"', between={0.3}{0.7}, Rightarrow, from=3-1, to=5-3]
	\arrow[""{name=4, anchor=center, inner sep=0}, "{{x_2}}", from=3-2, to=2-3]
	\arrow[""{name=5, anchor=center, inner sep=0}, "{{y_1}}"', from=3-2, to=4-3]
	\arrow["u"', from=3-4, to=3-3]
	\arrow[""{name=6, anchor=center, inner sep=0}, equals, from=3-4, to=4-4]
	\arrow[""{name=7, anchor=center, inner sep=0}, "{{d_y}}", from=4-3, to=3-3]
	\arrow["{{{y_3}}}"{description}, from=4-4, to=4-3]
	\arrow[""{name=8, anchor=center, inner sep=0}, equals, from=4-4, to=5-4]
	\arrow["h"', from=5-1, to=5-3]
	\arrow[""{name=9, anchor=center, inner sep=0}, "{{{y_2}}}"', from=5-3, to=4-3]
	\arrow["t", from=5-4, to=5-3]
	\arrow["{{{\SIGMA^{\delta_1}}}}"{description}, draw=none, from=0, to=1]
	\arrow["{{{\SIGMA^{\phi_x}}}}"{description}, draw=none, from=2, to=3]
	\arrow["\gamma", between={0.2}{0.8}, Rightarrow, from=4, to=5]
	\arrow["{{{\SIGMA^{\phi_y}}}}"{description}, draw=none, from=7, to=6]
	\arrow["{{{\SIGMA^{\epsilon_2}}}}"{description}, draw=none, from=9, to=8]
\end{tikzcd}
\end{equation}

The corresponding vertical composition of the two 2-cells is given by
	\begin{equation*}[\bar{\beta}]\cdot [\bar{\alpha}]=[\bar{\beta}\cdot \bar{\alpha}].
	\end{equation*}
	We now show that this is well-defined.
\end{definition}

\begin{proposition}\label{pro:vertical-1}The vertical composition between 2-cells is well-defined.
\end{proposition}

\begin{proof}
Consider two 2-morphisms $\bar{\alpha}=(\alpha,x_1,x_2,x_3,\delta_1,\delta_2)\colon (f,r) \Rightarrow (g,s)$ and $\bar{\beta} = (\beta,y_1,y_2,y_3,\epsilon_1,\epsilon_2)\colon \allowbreak (g,s) \Rightarrow (h,t)$ as above in Diagram~\eqref{EqV1}.

\noindent (1)
We first show that the composition does not depend on the choice of the $\Sigma$-squares and the isomorphisms $\gamma\colon d_x x_2 \Rightarrow d_y y_1$ in \cref{EqV2}.
Indeed, suppose we have two different choices of this data with $\gamma\colon d_x x_2 \Rightarrow d_y y_1$ and $\gamma'\colon d'_x x_2 \Rightarrow d'_y y_1$ and $\Sigma$-squares as in the diagrams
\[\begin{tikzcd}
	B & X && B & X \\
	B & D & {\text{and}} & B & {D'} \\
	B & Y && B & Y \rlap{\,.}
	\arrow["{x_3}", from=1-1, to=1-2]
	\arrow["u", from=2-1, to=2-2]
	\arrow[""{name=0, anchor=center, inner sep=0}, "{d_x}", from=1-2, to=2-2]
	\arrow["{x_3}", from=1-4, to=1-5]
	\arrow["{}"{name=1, anchor=center, inner sep=0}, Rightarrow, no head, from=1-4, to=2-4]
	\arrow["u'", from=2-4, to=2-5]
	\arrow[""{name=2, anchor=center, inner sep=0}, "{d'_x}", from=1-5, to=2-5]
	\arrow[""{name=3, anchor=center, inner sep=0}, "{d'_y}", tail reversed, no head, from=2-5, to=3-5]
	\arrow["{y_3}"', from=3-1, to=3-2]
	\arrow[""{name=4, anchor=center, inner sep=0}, Rightarrow, no head, from=2-1, to=3-1]
	\arrow[""{name=5, anchor=center, inner sep=0}, Rightarrow, no head, from=2-4, to=3-4]
	\arrow["{y_3}"', from=3-4, to=3-5]
	\arrow[""{name=6, anchor=center, inner sep=0}, Rightarrow, no head, from=1-1, to=2-1]
	\arrow[""{name=7, anchor=center, inner sep=0}, "{d_y}"', from=3-2, to=2-2]
	\arrow["{\SIGMA^{\phi'_x}}"{description}, draw=none, from=1, to=2]
	\arrow["{\SIGMA^{\phi'_y}}"{description}, draw=none, from=5, to=3]
	\arrow["{\SIGMA^{\phi_x}}"{description}, draw=none, from=6, to=0]
	\arrow["{\SIGMA^{\phi_y}}"{description}, draw=none, from=4, to=7]
\end{tikzcd}\]
By Rule 4 of \cref{pro:useful_rules}, we then obtain $\Sigma$-squares and invertible 2-cells $\theta_i$ (for $i=1,2$) such that
\[\begin{tikzcd}
	B & X && B & X \\
	B & D && B & {D'} \\
	B & T & {\text{\normalsize =}} & B & T \\
	B & D && B & {D'} \\
	B & Y && B & Y \rlap{\,.}
	\arrow["{x_3}", from=1-1, to=1-2]
	\arrow[""{name=0, anchor=center, inner sep=0}, Rightarrow, no head, from=1-1, to=2-1]
	\arrow[""{name=1, anchor=center, inner sep=0}, "{d_x}", from=1-2, to=2-2]
	\arrow[""{name=2, anchor=center, inner sep=0}, "{t_2d'_x}", curve={height=-24pt}, from=1-2, to=3-2]
	\arrow["{x_3}", from=1-4, to=1-5]
	\arrow[""{name=3, anchor=center, inner sep=0}, Rightarrow, no head, from=1-4, to=2-4]
	\arrow[""{name=4, anchor=center, inner sep=0}, "{d'_x}", from=1-5, to=2-5]
	\arrow["u", from=2-1, to=2-2]
	\arrow[""{name=5, anchor=center, inner sep=0}, Rightarrow, no head, from=2-1, to=3-1]
	\arrow[""{name=6, anchor=center, inner sep=0}, "{t_1}", from=2-2, to=3-2]
	\arrow["{u'}", from=2-4, to=2-5]
	\arrow[""{name=7, anchor=center, inner sep=0}, Rightarrow, no head, from=2-4, to=3-4]
	\arrow[""{name=8, anchor=center, inner sep=0}, "{t_2}", from=2-5, to=3-5]
	\arrow["v", from=3-1, to=3-2]
	\arrow[""{name=9, anchor=center, inner sep=0}, Rightarrow, no head, from=3-1, to=4-1]
	\arrow["v", from=3-4, to=3-5]
	\arrow[""{name=10, anchor=center, inner sep=0}, Rightarrow, no head, from=3-4, to=4-4]
	\arrow["u", from=4-1, to=4-2]
	\arrow[""{name=11, anchor=center, inner sep=0}, "{t_1}"', from=4-2, to=3-2]
	\arrow["{u'}", from=4-4, to=4-5]
	\arrow[""{name=12, anchor=center, inner sep=0}, "{t_2}"', from=4-5, to=3-5]
	\arrow[""{name=13, anchor=center, inner sep=0}, Rightarrow, no head, from=5-1, to=4-1]
	\arrow["{y_3}"', from=5-1, to=5-2]
	\arrow[""{name=14, anchor=center, inner sep=0}, "{t_2 d'_y}"', curve={height=24pt}, from=5-2, to=3-2]
	\arrow[""{name=15, anchor=center, inner sep=0}, "{d_y}"', from=5-2, to=4-2]
	\arrow[""{name=16, anchor=center, inner sep=0}, Rightarrow, no head, from=5-4, to=4-4]
	\arrow["{y_3}"', from=5-4, to=5-5]
	\arrow[""{name=17, anchor=center, inner sep=0}, "{d'_y}"', from=5-5, to=4-5]
	\arrow["{{\SIGMA^{\phi_x}}}"{description}, draw=none, from=0, to=1]
	\arrow["{{\SIGMA^{\phi'_x}}}"{description}, draw=none, from=3, to=4]
	\arrow["{{\SIGMA^{\psi_1}}}"{description}, draw=none, from=5, to=6]
	\arrow["{\theta_1}", shorten >=3pt, Rightarrow, from=2-2, to=2]
	\arrow["{{\SIGMA^{\psi_2}}}"{description}, draw=none, from=7, to=8]
	\arrow["{{\SIGMA^{\psi_1}}}"{description}, draw=none, from=9, to=11]
	\arrow["{{\SIGMA^{\psi_2}}}"{description}, draw=none, from=10, to=12]
	\arrow["{\theta_2}", shorten >=3pt, Rightarrow, from=4-2, to=14]
	\arrow["{{\SIGMA^{\phi_y}}}"{description}, draw=none, from=13, to=15]
	\arrow["{{\SIGMA^{\phi'_y}}}"{description}, draw=none, from=16, to=17]
\end{tikzcd}\]
These can be used to form $\Sigma$-extensions for each of the two compositions:

\[\begin{tikzcd}
	A &&& I && B & A &&& I & B \\
	&&& X && B &&&& X & B \\
	&& D & {D'} & D & B &&&& {D'} & B \\
	A & J && T && B & A & J && T & B \\
	&& D & {D'} & D & B &&&& {D'} & B \\
	&&& Y && B &&&& Y & B \\
	A &&& K && B & A &&& K & B \rlap{\,.}
	\arrow["f", from=1-1, to=1-4]
	\arrow[Rightarrow, no head, from=1-1, to=4-1]
	\arrow[""{name=0, anchor=center, inner sep=0}, "{x_1}"', from=1-4, to=2-4]
	\arrow["r"', from=1-6, to=1-4]
	\arrow[""{name=1, anchor=center, inner sep=0}, Rightarrow, no head, from=1-6, to=2-6]
	\arrow["f", from=1-7, to=1-10]
	\arrow[Rightarrow, no head, from=1-7, to=4-7]
	\arrow[""{name=2, anchor=center, inner sep=0}, "{x_1}"', from=1-10, to=2-10]
	\arrow["{t_2 d'_x \alpha}"', shorten <=37pt, shorten >=37pt, Rightarrow, from=1-10, to=4-7]
	\arrow["r"', from=1-11, to=1-10]
	\arrow[""{name=3, anchor=center, inner sep=0}, Rightarrow, no head, from=1-11, to=2-11]
	\arrow["{x_3}", from=2-4, to=2-6]
	\arrow["{d_x}"{description}, curve={height=6pt}, from=2-4, to=3-3]
	\arrow["{d'_x}", from=2-4, to=3-4]
	\arrow[""{name=4, anchor=center, inner sep=0}, "{d_x}"{description}, curve={height=-6pt}, from=2-4, to=3-5]
	\arrow[""{name=5, anchor=center, inner sep=0}, "{d'_x}"', from=2-10, to=3-10]
	\arrow["{x_3}"', from=2-11, to=2-10]
	\arrow[""{name=6, anchor=center, inner sep=0}, Rightarrow, no head, from=2-11, to=3-11]
	\arrow["{t_1}"{description}, from=3-3, to=4-4]
	\arrow["{\theta_1^{-1}}"', Rightarrow, from=3-4, to=3-3]
	\arrow["{t_2}"', from=3-4, to=4-4]
	\arrow["{\theta_1}"', Rightarrow, from=3-5, to=3-4]
	\arrow[""{name=7, anchor=center, inner sep=0}, "{t_1}"{description}, curve={height=-6pt}, from=3-5, to=4-4]
	\arrow[""{name=8, anchor=center, inner sep=0}, Rightarrow, no head, from=3-6, to=2-6]
	\arrow["u"', from=3-6, to=3-5]
	\arrow[""{name=9, anchor=center, inner sep=0}, Rightarrow, no head, from=3-6, to=4-6]
	\arrow[""{name=10, anchor=center, inner sep=0}, "{t_2}"', from=3-10, to=4-10]
	\arrow["{u'}"', from=3-11, to=3-10]
	\arrow[""{name=11, anchor=center, inner sep=0}, Rightarrow, no head, from=3-11, to=4-11]
	\arrow["g", from=4-1, to=4-2]
	\arrow[Rightarrow, no head, from=4-1, to=7-1]
	\arrow["{t_1d_y \beta}"', shorten <=37pt, shorten >=37pt, Rightarrow, from=4-1, to=7-4]
	\arrow[""{name=12, anchor=center, inner sep=0}, "{t_1 d_xx_2}"{description}, curve={height=-12pt}, from=4-2, to=4-4]
	\arrow[""{name=13, anchor=center, inner sep=0}, "{t_1 d_y y_1}"{description}, curve={height=12pt}, from=4-2, to=4-4]
	\arrow["v"', from=4-6, to=4-4]
	\arrow["{\text{ and }}"{description}, draw=none, from=4-6, to=4-7]
	\arrow[""{name=14, anchor=center, inner sep=0}, Rightarrow, no head, from=4-6, to=5-6]
	\arrow["g", from=4-7, to=4-8]
	\arrow[Rightarrow, no head, from=4-7, to=7-7]
	\arrow["{t_2 d'_y \beta}"', shorten <=37pt, shorten >=37pt, Rightarrow, from=4-7, to=7-10]
	\arrow[""{name=15, anchor=center, inner sep=0}, "{t_ 2d'_x x_2}", curve={height=-12pt}, from=4-8, to=4-10]
	\arrow[""{name=16, anchor=center, inner sep=0}, "{t_2 d'_y y_1}"', curve={height=12pt}, from=4-8, to=4-10]
	\arrow["v"', from=4-11, to=4-10]
	\arrow["{t_1}"{description}, from=5-3, to=4-4]
	\arrow["{\theta_2}"', Rightarrow, from=5-3, to=5-4]
	\arrow["{t_2}"{pos=0.4}, from=5-4, to=4-4]
	\arrow[""{name=17, anchor=center, inner sep=0}, "{t_1}"{description}, curve={height=6pt}, from=5-5, to=4-4]
	\arrow["{\theta_2}", Rightarrow, from=5-5, to=5-4]
	\arrow["u"', from=5-6, to=5-5]
	\arrow[""{name=18, anchor=center, inner sep=0}, Rightarrow, no head, from=5-6, to=6-6]
	\arrow[""{name=19, anchor=center, inner sep=0}, "{t_2}", from=5-10, to=4-10]
	\arrow[""{name=20, anchor=center, inner sep=0}, Rightarrow, no head, from=5-11, to=4-11]
	\arrow["{u'}"', from=5-11, to=5-10]
	\arrow[""{name=21, anchor=center, inner sep=0}, Rightarrow, no head, from=5-11, to=6-11]
	\arrow["{d_y}"{description}, curve={height=-6pt}, from=6-4, to=5-3]
	\arrow["{d'_y}"', from=6-4, to=5-4]
	\arrow[""{name=22, anchor=center, inner sep=0}, "{d_y}"{description}, curve={height=6pt}, from=6-4, to=5-5]
	\arrow["{y_3}"', from=6-4, to=6-6]
	\arrow[""{name=23, anchor=center, inner sep=0}, Rightarrow, no head, from=6-6, to=7-6]
	\arrow[""{name=24, anchor=center, inner sep=0}, "{d'_y}", from=6-10, to=5-10]
	\arrow["{{{y_3}}}", from=6-11, to=6-10]
	\arrow["h"', from=7-1, to=7-4]
	\arrow[""{name=25, anchor=center, inner sep=0}, "{y_2}", from=7-4, to=6-4]
	\arrow["t", from=7-6, to=7-4]
	\arrow["h"', from=7-7, to=7-10]
	\arrow[""{name=26, anchor=center, inner sep=0}, "{y_2}", from=7-10, to=6-10]
	\arrow[""{name=27, anchor=center, inner sep=0}, Rightarrow, no head, from=7-11, to=6-11]
	\arrow["t", from=7-11, to=7-10]
	\arrow["{\SIGMA^{\delta_1}}"{description}, draw=none, from=0, to=1]
	\arrow["{t_1 d_x \alpha}"', shift right=5, shorten <=38pt, shorten >=38pt, Rightarrow, from=0, to=4-1]
	\arrow["{\SIGMA^{\delta_1}}"{description}, draw=none, from=2, to=3]
	\arrow["{\SIGMA^{\phi_x}}"{description}, draw=none, from=4, to=8]
	\arrow["{\SIGMA^{\phi'_x}}"{description}, draw=none, from=6, to=5]
	\arrow["{\SIGMA^{\psi_1}}"{description}, draw=none, from=7, to=9]
	\arrow["{\SIGMA^{\psi_2}}"{description}, draw=none, from=10, to=11]
	\arrow["{t_1\gamma}"', shorten <=5pt, shorten >=5pt, Rightarrow, from=12, to=13]
	\arrow["{t_2 \gamma'}"', shorten <=3pt, shorten >=3pt, Rightarrow, from=15, to=16]
	\arrow["{\SIGMA^{\psi_1}}"{description}, draw=none, from=17, to=14]
	\arrow["{\SIGMA^{\psi_2}}"{description}, draw=none, from=20, to=19]
	\arrow["{\SIGMA^{\phi_y}}"{description}, draw=none, from=22, to=18]
	\arrow["{\SIGMA^{\phi'_y}}"{description}, draw=none, from=24, to=21]
	\arrow["{\SIGMA^{\epsilon_2}}"{description}, draw=none, from=25, to=23]
	\arrow["{\SIGMA^{\epsilon_2}}"{description}, draw=none, from=26, to=27]
\end{tikzcd}\]
For these two $\Sigma$-extensions to coincide we would need the 2-cells on the left-hand-side of each diagram to be equal.
That is, we want the outer rectangle of the following diagram to commute.

\[\begin{tikzcd}
	{t_1d_xx_1f} & {t_1d_xx_2g} & {t_1d_yy_1g} & {t_1d_yy_2h} \\
	{t_2d'_xx_1f} & {t_2d'_xx_2g} & {t_2d'_yy_1g} & {t_2d'_yy_2h}
	\arrow["{t_1d_x\alpha}", Rightarrow, from=1-1, to=1-2]
	\arrow["{t_1\gamma g}", Rightarrow, from=1-2, to=1-3]
	\arrow["{t_1d_y\beta}", Rightarrow, from=1-3, to=1-4]
	\arrow["{\theta_2y_1g}", Rightarrow, from=1-3, to=2-3]
	\arrow["{\theta_2y_2h}", Rightarrow, from=1-4, to=2-4]
	\arrow["{\theta_1^{-1}x_1f}", Rightarrow, from=2-1, to=1-1]
	\arrow["{t_2d'_x\alpha}"', Rightarrow, from=2-1, to=2-2]
	\arrow["{\theta_1^{-1}x_2g}", Rightarrow, from=2-2, to=1-2]
	\arrow["{t_2\gamma'g}"', Rightarrow, from=2-2, to=2-3]
	\arrow["{t_2d'_y\beta}"', Rightarrow, from=2-3, to=2-4]
\end{tikzcd}\]
The left and right squares commute by naturality and so it suffices to show the central square commutes.
Actually,  this square might not commute, but we can force it by passing to a further $\Sigma$-extension.
We want to force $\theta_2 y_1 \cdot t_1 \gamma \cdot \theta_1^{-1} x_2 = t_2 \gamma'$ using Equification.
Consider the diagram
\[\begin{tikzcd}
	B & J \\
	B & X & Y \\
	B & {D'} & {D'} \\
	B & T \rlap{\,,}
	\arrow["s", from=1-1, to=1-2]
	\arrow[""{name=0, anchor=center, inner sep=0}, Rightarrow, no head, from=1-1, to=2-1]
	\arrow[""{name=1, anchor=center, inner sep=0}, "{x_2}", from=1-2, to=2-2]
	\arrow["{y_1}", from=1-2, to=2-3]
	\arrow["{x_3}", from=2-1, to=2-2]
	\arrow[""{name=2, anchor=center, inner sep=0}, Rightarrow, no head, from=2-1, to=3-1]
	\arrow[""{name=3, anchor=center, inner sep=0}, "{d'_x}"', from=2-2, to=3-2]
	\arrow[""{name=4, anchor=center, inner sep=0}, "{d'_y}", from=2-3, to=3-3]
	\arrow["{u'}", from=3-1, to=3-2]
	\arrow[""{name=5, anchor=center, inner sep=0}, Rightarrow, no head, from=3-1, to=4-1]
	\arrow[""{name=6, anchor=center, inner sep=0}, "{t_2}", from=3-2, to=4-2]
	\arrow["{t_2}", from=3-3, to=4-2]
	\arrow["v"', from=4-1, to=4-2]
	\arrow["{\SIGMA^{\delta_2}}"{marking, allow upside down}, draw=none, from=0, to=1]
	\arrow["{\SIGMA^{\phi'_x}}"{marking, allow upside down}, draw=none, from=2, to=3]
	\arrow["{t_2 \gamma'}", shift left=2, shorten <=10pt, shorten >=10pt, Rightarrow, from=3, to=4]
	\arrow["{\zeta}"', shift right=2, shorten <=10pt, shorten >=10pt, Rightarrow, from=3, to=4]
	\arrow["{\SIGMA^{\psi_2}}"{marking, allow upside down}, draw=none, from=5, to=6]
\end{tikzcd}\]
where $\zeta = \theta_2 y_1 \cdot t_1 \gamma \cdot \theta_1^{-1} x_2$.
To apply Equification we must first show $t_2 \gamma' s = \zeta s$.
Note that the composite of the above squares with the map $t_2 \gamma' s$ is equal to $\SIGMA^{\psi_2 \odot \phi'_y \odot \epsilon_1}$ by the assumption on $\gamma'$.
On the other hand, the composite of the squares with $\zeta s$ is

\begin{alignat*}{2}
 &\begin{tikzcd}[ampersand replacement=\&]
	B \& J \\
	B \& X \& Y \\
	B \& {D'} \& D \& {D'} \\
	B \& T
	\arrow["s", from=1-1, to=1-2]
	\arrow[""{name=0, anchor=center, inner sep=0}, Rightarrow, no head, from=1-1, to=2-1]
	\arrow[""{name=1, anchor=center, inner sep=0}, "{x_2}", from=1-2, to=2-2]
	\arrow["{y_1}", from=1-2, to=2-3]
	\arrow["{x_3}", from=2-1, to=2-2]
	\arrow[""{name=2, anchor=center, inner sep=0}, Rightarrow, no head, from=2-1, to=3-1]
	\arrow["\gamma", shorten <=2pt, shorten >=2pt, Rightarrow, from=2-2, to=2-3]
	\arrow[""{name=3, anchor=center, inner sep=0}, "{d'_x}"', from=2-2, to=3-2]
	\arrow["{d_x}", from=2-2, to=3-3]
	\arrow["{d_y}", from=2-3, to=3-3]
	\arrow["{d'_y}", from=2-3, to=3-4]
	\arrow["{u'}", from=3-1, to=3-2]
	\arrow[""{name=4, anchor=center, inner sep=0}, Rightarrow, no head, from=3-1, to=4-1]
	\arrow["{\theta_1^{-1}}", shorten <=2pt, shorten >=2pt, Rightarrow, from=3-2, to=3-3]
	\arrow[""{name=5, anchor=center, inner sep=0}, "{t_2}"', from=3-2, to=4-2]
	\arrow["{\theta_2}", shorten <=2pt, shorten >=2pt, Rightarrow, from=3-3, to=3-4]
	\arrow["{t_1}"', from=3-3, to=4-2]
	\arrow["{t_2}", from=3-4, to=4-2]
	\arrow["v"', from=4-1, to=4-2]
	\arrow["{\SIGMA^{\delta_2}}"{marking, allow upside down}, draw=none, from=0, to=1]
	\arrow["{\SIGMA^{\phi'_x}}"{marking, allow upside down}, draw=none, from=2, to=3]
	\arrow["{\SIGMA^{\psi_2}}"{marking, allow upside down}, draw=none, from=4, to=5]
 \end{tikzcd} &\quad =& \quad
 \begin{tikzcd}[ampersand replacement=\&]
	B \& J \\
	B \& X \& Y \\
	B \& D \& {D'} \\
	B \& T
	\arrow["s", from=1-1, to=1-2]
	\arrow[""{name=0, anchor=center, inner sep=0}, Rightarrow, no head, from=1-1, to=2-1]
	\arrow[""{name=1, anchor=center, inner sep=0}, "{x_2}", from=1-2, to=2-2]
	\arrow["{y_1}", from=1-2, to=2-3]
	\arrow["{x_3}", from=2-1, to=2-2]
	\arrow[""{name=2, anchor=center, inner sep=0}, Rightarrow, no head, from=2-1, to=3-1]
	\arrow["\gamma", shorten <=2pt, shorten >=2pt, Rightarrow, from=2-2, to=2-3]
	\arrow[""{name=3, anchor=center, inner sep=0}, "{d_x}"', from=2-2, to=3-2]
	\arrow["{d_y}"', from=2-3, to=3-2]
	\arrow["{d'_y}", from=2-3, to=3-3]
	\arrow["{u'}", from=3-1, to=3-2]
	\arrow[""{name=4, anchor=center, inner sep=0}, Rightarrow, no head, from=3-1, to=4-1]
	\arrow["{\theta_2}", shorten <=2pt, shorten >=2pt, Rightarrow, from=3-2, to=3-3]
	\arrow[""{name=5, anchor=center, inner sep=0}, "{t_1}"', from=3-2, to=4-2]
	\arrow["{t_2}", from=3-3, to=4-2]
	\arrow["v"', from=4-1, to=4-2]
	\arrow["{\SIGMA^{\delta_2}}"{marking, allow upside down}, draw=none, from=0, to=1]
	\arrow["{\SIGMA^{\phi_x}}"{marking, allow upside down}, draw=none, from=2, to=3]
	\arrow["{\SIGMA^{\psi_1}}"{marking, allow upside down}, draw=none, from=4, to=5]
 \end{tikzcd} \\
= \quad
 &\begin{tikzcd}[ampersand replacement=\&]
	B \& J \\
	B \& Y \\
	B \& D \& {D'} \\
	B \& T
	\arrow["s", from=1-1, to=1-2]
	\arrow[""{name=0, anchor=center, inner sep=0}, Rightarrow, no head, from=1-1, to=2-1]
	\arrow[""{name=1, anchor=center, inner sep=0}, "{y_1}", from=1-2, to=2-2]
	\arrow["{x_3}", from=2-1, to=2-2]
	\arrow[""{name=2, anchor=center, inner sep=0}, Rightarrow, no head, from=2-1, to=3-1]
	\arrow[""{name=3, anchor=center, inner sep=0}, "{d_y}"', from=2-2, to=3-2]
	\arrow["{d'_y}", from=2-2, to=3-3]
	\arrow["{u'}", from=3-1, to=3-2]
	\arrow[""{name=4, anchor=center, inner sep=0}, Rightarrow, no head, from=3-1, to=4-1]
	\arrow["{\theta_2}", shorten <=2pt, shorten >=2pt, Rightarrow, from=3-2, to=3-3]
	\arrow[""{name=5, anchor=center, inner sep=0}, "{t_1}"', from=3-2, to=4-2]
	\arrow["{t_2}", from=3-3, to=4-2]
	\arrow["v"', from=4-1, to=4-2]
	\arrow["{\SIGMA^{\epsilon_1}}"{marking, allow upside down}, draw=none, from=0, to=1]
	\arrow["{\SIGMA^{\phi_y}}"{marking, allow upside down}, draw=none, from=2, to=3]
	\arrow["{\SIGMA^{\psi_1}}"{marking, allow upside down}, draw=none, from=4, to=5]
 \end{tikzcd}
 &\quad =& \quad
 \begin{tikzcd}[ampersand replacement=\&]
	B \& J \\
	B \& Y \\
	B \& {D'} \\
	B \& T \rlap{\,.}
	\arrow["s", from=1-1, to=1-2]
	\arrow[""{name=0, anchor=center, inner sep=0}, Rightarrow, no head, from=1-1, to=2-1]
	\arrow[""{name=1, anchor=center, inner sep=0}, "{y_1}", from=1-2, to=2-2]
	\arrow["{x_3}", from=2-1, to=2-2]
	\arrow[""{name=2, anchor=center, inner sep=0}, Rightarrow, no head, from=2-1, to=3-1]
	\arrow[""{name=3, anchor=center, inner sep=0}, "{d'_y}", from=2-2, to=3-2]
	\arrow["{u'}", from=3-1, to=3-2]
	\arrow[""{name=4, anchor=center, inner sep=0}, Rightarrow, no head, from=3-1, to=4-1]
	\arrow[""{name=5, anchor=center, inner sep=0}, "{t_2}", from=3-2, to=4-2]
	\arrow["v"', from=4-1, to=4-2]
	\arrow["{\SIGMA^{\epsilon_1}}"{marking, allow upside down}, draw=none, from=0, to=1]
	\arrow["{\SIGMA^{\phi'_y}}"{marking, allow upside down}, draw=none, from=2, to=3]
	\arrow["{\SIGMA^{\psi_2}}"{marking, allow upside down}, draw=none, from=4, to=5]
 \end{tikzcd}
\end{alignat*}
So the composites are the same, and since the 2-cells in the squares are invertible, we indeed have $t_2 \gamma' s = \zeta s$. %
So by Equification, we obtain a map $q\colon T \to Q$ and associated $\Sigma$-square that allows us to pass to a common $\Sigma$-extension of the two 2-morphisms above, as required.

\noindent (2) We now show that vertical composition of 2-morphisms respects $\approx$ --- that is, if $\overline{\alpha_1} \approx \overline{\alpha_2}$ and $\overline{\beta_1} \approx \overline{\beta_2}$ then $\overline{\beta_1} \cdot \overline{\alpha_1} \approx \overline{\beta_2} \cdot \overline{\alpha_2}$. By transitivity and symmetry, we may assume $\overline{\alpha_2}$ is a $\Sigma$-extension of $\overline{\alpha_1} = \overline{\alpha}$ and $\overline{\beta_2}$ is a $\Sigma$-extension of $\overline{\beta_1} = \overline{\beta}$ without loss of generality. Moreover, again by transitivity, we may assume $\overline{\alpha}$ and $\overline{\beta}$ are replaced in turn. We will consider the case $\overline{\alpha}$; the case of $\overline{\beta}$ is entirely analogous.

Suppose $\overline{\alpha_2}$ is the $\Sg$-extension of $\overline{\alpha}$ given by the diagram

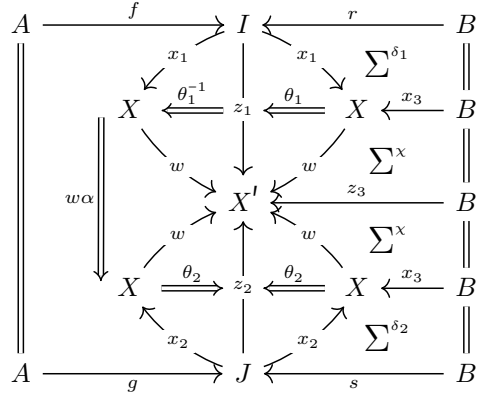
\[\begin{tikzcd}
	A && I && B \\
	& X & {} & X & B \\
	&& {X'} && B \\
	& X & {} & X & B \\
	A && J && B
	\arrow[""{name=0, anchor=center, inner sep=0}, "f", from=1-1, to=1-3]
	\arrow[Rightarrow, no head, from=1-1, to=5-1]
	\arrow["{x_1}"{description}, curve={height=6pt}, from=1-3, to=2-2]
	\arrow[""{name=1, anchor=center, inner sep=0}, "{x_1}"{description}, curve={height=-6pt}, from=1-3, to=2-4]
	\arrow["{{{z_1}}}"{description}, from=1-3, to=3-3]
	\arrow["r"', from=1-5, to=1-3]
	\arrow[""{name=2, anchor=center, inner sep=0}, Rightarrow, no head, from=1-5, to=2-5]
	\arrow["w"{description}, curve={height=6pt}, from=2-2, to=3-3]
	\arrow["{\theta_1^{-1}}"', shorten <=3pt, shorten >=3pt, Rightarrow, from=2-3, to=2-2]
	\arrow["{\theta_1}"', shorten <=3pt, shorten >=3pt, Rightarrow, from=2-4, to=2-3]
	\arrow[""{name=3, anchor=center, inner sep=0}, "w"{description}, curve={height=-6pt}, from=2-4, to=3-3]
	\arrow["{{{{{x_3}}}}}"', from=2-5, to=2-4]
	\arrow[""{name=4, anchor=center, inner sep=0}, Rightarrow, no head, from=2-5, to=3-5]
	\arrow["{z_3}"', from=3-5, to=3-3]
	\arrow[""{name=5, anchor=center, inner sep=0}, Rightarrow, no head, from=3-5, to=4-5]
	\arrow["w"{description}, curve={height=-6pt}, from=4-2, to=3-3]
	\arrow["{\theta_2}", shorten <=3pt, shorten >=3pt, Rightarrow, from=4-2, to=4-3]
	\arrow[""{name=6, anchor=center, inner sep=0}, "w"{description}, curve={height=6pt}, from=4-4, to=3-3]
	\arrow["{\theta_2}"', shorten <=3pt, shorten >=3pt, Rightarrow, from=4-4, to=4-3]
	\arrow["{{{{{x_3}}}}}"', from=4-5, to=4-4]
	\arrow[""{name=7, anchor=center, inner sep=0}, Rightarrow, no head, from=4-5, to=5-5]
	\arrow[""{name=8, anchor=center, inner sep=0}, "g"', from=5-1, to=5-3]
	\arrow["{{{z_2}}}"{description}, from=5-3, to=3-3]
	\arrow["{x_2}"{description}, curve={height=-6pt}, from=5-3, to=4-2]
	\arrow[""{name=9, anchor=center, inner sep=0}, "{x_2}"{description}, curve={height=6pt}, from=5-3, to=4-4]
	\arrow["s", from=5-5, to=5-3]
	\arrow["{w\alpha}"', shift right=5, shorten <=34pt, shorten >=34pt, Rightarrow, from=0, to=8]
	\arrow["{\SIGMA^{\delta_1}}"{description}, draw=none, from=1, to=2]
	\arrow["{\SIGMA^{\chi}}"{description}, draw=none, from=3, to=4]
	\arrow["{\SIGMA^{\chi}}"{description}, draw=none, from=6, to=5]
	\arrow["{\SIGMA^{\delta_2}}"{description}, draw=none, from=9, to=7]
\end{tikzcd}\]
and suppose the composite $\overline{\beta} \cdot \overline{\alpha_2}$ is defined using the data $d_x\colon X' \to D, d_y\colon Y \to D, u\colon B \to D, \phi_x, \phi_y$ and $\gamma$.
That is, we have

\[\begin{tikzcd}
	B & J \\
	B & X & {z_2} & {} \\
	B & {X'} && {} \\
	B & D
	\arrow["s", from=1-1, to=1-2]
	\arrow[""{name=0, anchor=center, inner sep=0}, Rightarrow, no head, from=1-1, to=2-1]
	\arrow[""{name=1, anchor=center, inner sep=0}, "{x_2}", from=1-2, to=2-2]
	\arrow[curve={height=-6pt}, no head, from=1-2, to=2-3]
	\arrow["{x_3}", from=2-1, to=2-2]
	\arrow[""{name=2, anchor=center, inner sep=0}, Rightarrow, no head, from=2-1, to=3-1]
	\arrow["{\theta_2}", Rightarrow, from=2-2, to=2-3]
	\arrow[""{name=3, anchor=center, inner sep=0}, "w", from=2-2, to=3-2]
	\arrow[""{name=4, anchor=center, inner sep=0}, curve={height=-6pt}, from=2-3, to=3-2]
	\arrow[""{name=5, anchor=center, inner sep=0}, "\text{\normalsize $Y$}"{description}, draw=none, from=2-4, to=3-4]
	\arrow["{z_3}", from=3-1, to=3-2]
	\arrow[""{name=6, anchor=center, inner sep=0}, Rightarrow, no head, from=3-1, to=4-1]
	\arrow[""{name=7, anchor=center, inner sep=0}, "{d_x}", from=3-2, to=4-2]
	\arrow["u"', from=4-1, to=4-2]
	\arrow["{\SIGMA^{\delta_2}}"{description}, draw=none, from=0, to=1]
	\arrow["{y_1}", curve={height=-12pt}, shorten >=6pt, from=1-2, to=5]
	\arrow["{\SIGMA^{\chi}}"{description}, draw=none, from=2, to=3]
	\arrow["\gamma", shorten <=13pt, shorten >=13pt, Rightarrow, from=4, to=5]
	\arrow["{d_y}", curve={height=-12pt}, shorten <=6pt, from=5, to=4-2]
	\arrow["{\SIGMA^{\phi_x}}"{description}, draw=none, from=6, to=7]
\end{tikzcd}
\quad=\quad
\begin{tikzcd}
	B & J \\
	B & Y \\
	B & D\rlap{\,,}
	\arrow["s", from=1-1, to=1-2]
	\arrow[""{name=0, anchor=center, inner sep=0}, Rightarrow, no head, from=1-1, to=2-1]
	\arrow[""{name=1, anchor=center, inner sep=0}, "{{y_1}}", from=1-2, to=2-2]
	\arrow["{{y_3}}", dashed, from=2-1, to=2-2]
	\arrow[""{name=2, anchor=center, inner sep=0}, Rightarrow, no head, from=2-1, to=3-1]
	\arrow[""{name=3, anchor=center, inner sep=0}, "{{d_y}}", from=2-2, to=3-2]
	\arrow["u"', from=3-1, to=3-2]
	\arrow["{\SIGMA^{\epsilon_1}}"{description}, draw=none, from=0, to=1]
	\arrow["{\SIGMA^{\phi_y}}"{description}, draw=none, from=2, to=3]
\end{tikzcd}\]
yielding the composite
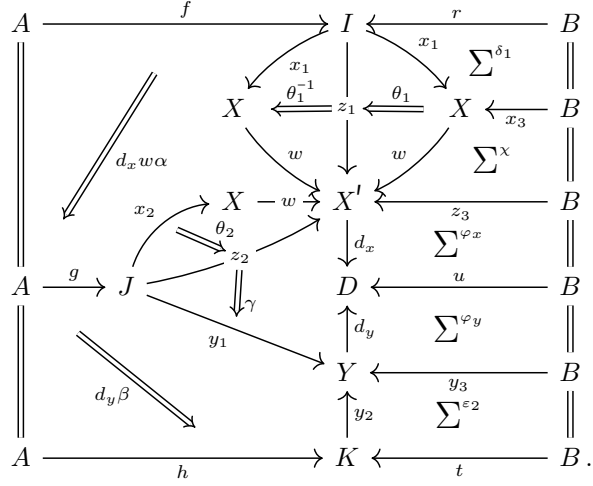
\[\begin{tikzcd}
	A &&& I && B \\
	&& X && X & B \\
	&& X & {X'} && B \\
	A & J && D && B \\
	&&& Y && B \\
	A && {} & K && B \rlap{\,.}
	\arrow[""{name=0, anchor=center, inner sep=0}, "f", from=1-1, to=1-4]
	\arrow[Rightarrow, no head, from=1-1, to=4-1]
	\arrow["{x_1}", curve={height=6pt}, from=1-4, to=2-3]
	\arrow[""{name=1, anchor=center, inner sep=0}, "{{x_1}}", curve={height=-6pt}, from=1-4, to=2-5]
	\arrow[""{name=2, anchor=center, inner sep=0}, "{z_1}"{description}, from=1-4, to=3-4]
	\arrow["r"', from=1-6, to=1-4]
	\arrow[""{name=3, anchor=center, inner sep=0}, Rightarrow, no head, from=1-6, to=2-6]
	\arrow["w", curve={height=6pt}, from=2-3, to=3-4]
	\arrow[""{name=4, anchor=center, inner sep=0}, "w"', curve={height=-6pt}, from=2-5, to=3-4]
	\arrow["{x_3}", from=2-6, to=2-5]
	\arrow[""{name=5, anchor=center, inner sep=0}, Rightarrow, no head, from=2-6, to=3-6]
	\arrow["w"{description}, from=3-3, to=3-4]
	\arrow[""{name=6, anchor=center, inner sep=0}, "{{d_x}}", from=3-4, to=4-4]
	\arrow["{z_3}", from=3-6, to=3-4]
	\arrow[""{name=7, anchor=center, inner sep=0}, Rightarrow, no head, from=3-6, to=4-6]
	\arrow["g", from=4-1, to=4-2]
	\arrow[Rightarrow, no head, from=4-1, to=6-1]
	\arrow["{d_y\beta}"', shorten <=17pt, shorten >=17pt, Rightarrow, from=4-1, to=6-3]
	\arrow[""{name=8, anchor=center, inner sep=0}, "{x_2}", curve={height=-12pt}, from=4-2, to=3-3]
	\arrow[""{name=9, anchor=center, inner sep=0}, "{z_2}"{description}, curve={height=6pt}, from=4-2, to=3-4]
	\arrow[""{name=10, anchor=center, inner sep=0}, "{y_1}"', from=4-2, to=5-4]
	\arrow["u"', from=4-6, to=4-4]
	\arrow[""{name=11, anchor=center, inner sep=0}, Rightarrow, no head, from=4-6, to=5-6]
	\arrow[""{name=12, anchor=center, inner sep=0}, "{{d_y}}"', from=5-4, to=4-4]
	\arrow["{{y_3}}", from=5-6, to=5-4]
	\arrow[""{name=13, anchor=center, inner sep=0}, Rightarrow, no head, from=5-6, to=6-6]
	\arrow["h"', from=6-1, to=6-4]
	\arrow[""{name=14, anchor=center, inner sep=0}, "{{y_2}}"', from=6-4, to=5-4]
	\arrow["t", from=6-6, to=6-4]
	\arrow["{d_x w \alpha}", shorten <=21pt, shorten >=21pt, Rightarrow, from=0, to=4-1]
	\arrow["{\theta_1^{-1}}"', shorten <=5pt, shorten >=5pt, Rightarrow, from=2, to=2-3]
	\arrow["{\SIGMA^{\delta_1}}"{description}, draw=none, from=1, to=3]
	\arrow["{\theta_1}"'{pos=0.4}, , shorten <=5pt, shorten >=5pt, Rightarrow, from=2-5, to=2]
	\arrow["{\SIGMA^{\chi}}"{description}, draw=none, from=4, to=5]
	\arrow["{{\SIGMA^{\phi_x}}}"{description}, draw=none, from=6, to=7]
	\arrow["{\theta_2}"{pos=0.6}, shorten <=5pt, shorten >=5pt, Rightarrow, from=8, to=9]
	\arrow["\gamma", shorten <=4pt, shorten >=4pt, Rightarrow, from=9, to=10]
	\arrow["{{\SIGMA^{\phi_y}}}"{description}, draw=none, from=12, to=11]
	\arrow["{{\SIGMA^{\epsilon_2}}}"{description}, draw=none, from=14, to=13]
\end{tikzcd}\]
We now observe that this is a $\Sigma$-extension (using $d_x \theta_1$ and $\id_{d_y y_2}$) of the composite for $\overline{\beta}\cdot\overline{\alpha}$ given by the data $d_x w, d_y, u, \phi_x \odot \chi$, $\phi_y$ and $\gamma \cdot d_x\theta_2$, and hence the two composites are equivalent.
\end{proof}

\begin{remark}\label{rem:identities}
	\textbf{\em Identity 2-cells.} The identity 2-cell on the $\Sigma$-cospan $(f,I,r)$ is $[\id_f,1_I,1_I,r,\id_r,\id_r]$. It can represented by any 2-morphism of the form $(\id_{df}, d,d,u,\delta,\delta)$.
\end{remark}

\begin{proposition}\label{pro:vertical-2}
The vertical composition between 2-cells is associative and the identity 2-cells indeed act as identities.
\end{proposition}

\begin{proof}
Consider the following triple composite.
\[\begin{tikzcd}
	A & {I_1} & B \\
	& X & B \\
	A & {I_2} & B \\
	& Y & B \\
	A & I_3 & B \\
	& Z & B \\
	A & {I_4} & B
	\arrow["f", from=1-1, to=1-2]
	\arrow[Rightarrow, no head, from=1-1, to=3-1]
	\arrow[""{name=0, anchor=center, inner sep=0}, "{{{{x_1}}}}"', from=1-2, to=2-2]
	\arrow["\alpha"', shorten <=15pt, shorten >=15pt, Rightarrow, from=1-2, to=3-1]
	\arrow["{r_1}"', from=1-3, to=1-2]
	\arrow[""{name=1, anchor=center, inner sep=0}, Rightarrow, no head, from=1-3, to=2-3]
	\arrow["{{x_3}}", from=2-3, to=2-2]
	\arrow["g"', from=3-1, to=3-2]
	\arrow[Rightarrow, no head, from=3-1, to=5-1]
	\arrow[""{name=2, anchor=center, inner sep=0}, "{{{{x_2}}}}", from=3-2, to=2-2]
	\arrow[""{name=3, anchor=center, inner sep=0}, "{{{{y_1}}}}"', from=3-2, to=4-2]
	\arrow["\beta"', shorten <=15pt, shorten >=15pt, Rightarrow, from=3-2, to=5-1]
	\arrow[""{name=4, anchor=center, inner sep=0}, Rightarrow, no head, from=3-3, to=2-3]
	\arrow["{r_2}", from=3-3, to=3-2]
	\arrow[""{name=5, anchor=center, inner sep=0}, Rightarrow, no head, from=3-3, to=4-3]
	\arrow["{{y_3}}", from=4-3, to=4-2]
	\arrow[""{name=6, anchor=center, inner sep=0}, Rightarrow, no head, from=4-3, to=5-3]
	\arrow["h"', from=5-1, to=5-2]
	\arrow[Rightarrow, no head, from=5-1, to=7-1]
	\arrow[""{name=7, anchor=center, inner sep=0}, "{{{{y_2}}}}", from=5-2, to=4-2]
	\arrow[""{name=8, anchor=center, inner sep=0}, "{z_1}"', from=5-2, to=6-2]
	\arrow["\gamma"', shorten <=15pt, shorten >=15pt, Rightarrow, from=5-2, to=7-1]
	\arrow["{r_3}", from=5-3, to=5-2]
	\arrow[""{name=9, anchor=center, inner sep=0}, Rightarrow, no head, from=5-3, to=6-3]
	\arrow["{z_3}", from=6-3, to=6-2]
	\arrow[""{name=10, anchor=center, inner sep=0}, Rightarrow, no head, from=6-3, to=7-3]
	\arrow["j"', from=7-1, to=7-2]
	\arrow[""{name=11, anchor=center, inner sep=0}, "{z_2}", from=7-2, to=6-2]
	\arrow["{r_4}", from=7-3, to=7-2]
	\arrow["{{{\SIGMA^{\delta_1}}}}"{description}, draw=none, from=0, to=1]
	\arrow["{{{\SIGMA^{\delta_2}}}}"{description}, draw=none, from=2, to=4]
	\arrow["{{{\SIGMA^{\epsilon_1}}}}"{description}, draw=none, from=3, to=5]
	\arrow["{{{\SIGMA^{\epsilon_2}}}}"{description}, draw=none, from=7, to=6]
	\arrow["{\SIGMA^{\zeta_1}}"{description}, draw=none, from=8, to=9]
	\arrow["{\SIGMA^{\zeta_2}}"{description}, draw=none, from=11, to=10]
\end{tikzcd}\]
We must show that the two ways of composing, $(\overline{\gamma} \cdot \overline{\beta}) \cdot \overline{\alpha}$ and $\overline{\gamma} \cdot (\overline{\beta} \cdot \overline{\alpha})$, give equivalent results. For the former of these we find $\overline{\gamma} \cdot \overline{\beta}$ is given by
\[\begin{tikzcd}
	A && {I_2} & B \\
	&& Y & B \\
	A & {I_3} & P & B \\
	&& Z & B \\
	A && {I_4} & B \rlap{\,,}
	\arrow["g", from=1-1, to=1-3]
	\arrow[Rightarrow, no head, from=1-1, to=3-1]
	\arrow[""{name=0, anchor=center, inner sep=0}, "{y_1}"', from=1-3, to=2-3]
	\arrow["\beta"', shorten <=17pt, shorten >=17pt, Rightarrow, from=1-3, to=3-1]
	\arrow["{r_2}"', from=1-4, to=1-3]
	\arrow[""{name=1, anchor=center, inner sep=0}, Rightarrow, no head, from=1-4, to=2-4]
	\arrow[""{name=2, anchor=center, inner sep=0}, "{p_y}"', from=2-3, to=3-3]
	\arrow["{y_3}"', from=2-4, to=2-3]
	\arrow[""{name=3, anchor=center, inner sep=0}, Rightarrow, no head, from=2-4, to=3-4]
	\arrow["h", from=3-1, to=3-2]
	\arrow[Rightarrow, no head, from=3-1, to=5-1]
	\arrow["\gamma"', shorten <=17pt, shorten >=17pt, Rightarrow, from=3-1, to=5-3]
	\arrow[""{name=4, anchor=center, inner sep=0}, "{y_2}", from=3-2, to=2-3]
	\arrow[""{name=5, anchor=center, inner sep=0}, "{z_1}"', from=3-2, to=4-3]
	\arrow["v"', from=3-4, to=3-3]
	\arrow[""{name=6, anchor=center, inner sep=0}, Rightarrow, no head, from=3-4, to=4-4]
	\arrow[""{name=7, anchor=center, inner sep=0}, "{p_z}", from=4-3, to=3-3]
	\arrow["{z_3}"', from=4-4, to=4-3]
	\arrow[""{name=8, anchor=center, inner sep=0}, Rightarrow, no head, from=4-4, to=5-4]
	\arrow["j"', from=5-1, to=5-3]
	\arrow[""{name=9, anchor=center, inner sep=0}, "{z_2}", from=5-3, to=4-3]
	\arrow["{r_4}", from=5-4, to=5-3]
	\arrow["{\SIGMA^{\epsilon_1}}"{description}, draw=none, from=0, to=1]
	\arrow["{\SIGMA^{\eta_y}}"{description}, draw=none, from=2, to=3]
	\arrow["\omega", shorten <=4pt, shorten >=4pt, Rightarrow, from=4, to=5]
	\arrow["{\SIGMA^{\eta_z}}"{description}, draw=none, from=7, to=6]
	\arrow["{\SIGMA^{\zeta_2}}"{description}, draw=none, from=9, to=8]
\end{tikzcd}\]
and hence $(\overline{\gamma} \cdot \overline{\beta}) \cdot \overline{\alpha}$ is given by
\begin{equation}\label{eq:gamma_beta_then_alpha}
\begin{tikzcd}
	A && {I_1} & B \\
	&& X & B \\
	A & {I_2} & Q & B \\
	& Y & P & B \\
	A & {I_3} \\
	&& Z & B \\
	A && {I_4} & B \rlap{\,,}
	\arrow["f", from=1-1, to=1-3]
	\arrow[Rightarrow, no head, from=1-1, to=3-1]
	\arrow[""{name=0, anchor=center, inner sep=0}, "{x_1}"', from=1-3, to=2-3]
	\arrow["\alpha"', shorten <=17pt, shorten >=17pt, Rightarrow, from=1-3, to=3-1]
	\arrow["{r_1}"', from=1-4, to=1-3]
	\arrow[""{name=1, anchor=center, inner sep=0}, Rightarrow, no head, from=1-4, to=2-4]
	\arrow[""{name=2, anchor=center, inner sep=0}, "{q_x}"', from=2-3, to=3-3]
	\arrow["{x_3}"', from=2-4, to=2-3]
	\arrow[""{name=3, anchor=center, inner sep=0}, Rightarrow, no head, from=2-4, to=3-4]
	\arrow["g", from=3-1, to=3-2]
	\arrow[Rightarrow, no head, from=3-1, to=5-1]
	\arrow[""{name=4, anchor=center, inner sep=0}, "{x_2}", from=3-2, to=2-3]
	\arrow["{y_1}", from=3-2, to=4-2]
	\arrow["\beta"', shorten <=11pt, shorten >=11pt, Rightarrow, from=3-2, to=5-1]
	\arrow["w"', from=3-4, to=3-3]
	\arrow[""{name=5, anchor=center, inner sep=0}, Rightarrow, no head, from=3-4, to=4-4]
	\arrow[""{name=6, anchor=center, inner sep=0}, "{p_y}", from=4-2, to=4-3]
	\arrow[""{name=7, anchor=center, inner sep=0}, "{q_y}"', from=4-3, to=3-3]
	\arrow["v"', from=4-4, to=4-3]
	\arrow[""{name=8, anchor=center, inner sep=0}, Rightarrow, no head, from=4-4, to=6-4]
	\arrow["h", from=5-1, to=5-2]
	\arrow[Rightarrow, no head, from=5-1, to=7-1]
	\arrow["\gamma"', shorten <=17pt, shorten >=17pt, Rightarrow, from=5-1, to=7-3]
	\arrow["{y_2}", from=5-2, to=4-2]
	\arrow[""{name=9, anchor=center, inner sep=0}, "{z_1}"', from=5-2, to=6-3]
	\arrow[""{name=10, anchor=center, inner sep=0}, "{p_z}"', from=6-3, to=4-3]
	\arrow["{z_3}"', from=6-4, to=6-3]
	\arrow[""{name=11, anchor=center, inner sep=0}, Rightarrow, no head, from=6-4, to=7-4]
	\arrow["j"', from=7-1, to=7-3]
	\arrow[""{name=12, anchor=center, inner sep=0}, "{z_2}", from=7-3, to=6-3]
	\arrow["{r_4}", from=7-4, to=7-3]
	\arrow["{\SIGMA^{\delta_1}}"{description}, draw=none, from=0, to=1]
	\arrow["{\SIGMA^{\theta_x}}"{description}, draw=none, from=2, to=3]
	\arrow["\psi", shorten <=10pt, shorten >=10pt, Rightarrow, from=4, to=6]
	\arrow["\omega", shorten <=10pt, shorten >=10pt, Rightarrow, from=6, to=9]
	\arrow["{\SIGMA^{\theta_y}}"{description}, draw=none, from=7, to=5]
	\arrow["{\SIGMA^{\eta_z}}"{description}, draw=none, from=10, to=8]
	\arrow["{\SIGMA^{\zeta_2}}"{description}, draw=none, from=12, to=11]
\end{tikzcd}
\end{equation}
where the data used for these compositions satisfy
\begin{equation}\label{eq:omega_equality}
\begin{tikzcd}
	B & I_3 \\
	B & Y & Z \\
	B & P
	\arrow["{r_3}", from=1-1, to=1-2]
	\arrow[""{name=0, anchor=center, inner sep=0}, Rightarrow, no head, from=1-1, to=2-1]
	\arrow[""{name=1, anchor=center, inner sep=0}, "{y_2}", from=1-2, to=2-2]
	\arrow["{z_1}", curve={height=-6pt}, from=1-2, to=2-3]
	\arrow["{y_3}", from=2-1, to=2-2]
	\arrow[""{name=2, anchor=center, inner sep=0}, Rightarrow, no head, from=2-1, to=3-1]
	\arrow["\omega", Rightarrow, from=2-2, to=2-3]
	\arrow[""{name=3, anchor=center, inner sep=0}, "{p_y}", from=2-2, to=3-2]
	\arrow["{p_z}", curve={height=-6pt}, from=2-3, to=3-2]
	\arrow["v"', from=3-1, to=3-2]
	\arrow["{\SIGMA^{\epsilon_2}}"{description}, draw=none, from=0, to=1]
	\arrow["{\SIGMA^{\eta_y}}"{description}, draw=none, from=2, to=3]
\end{tikzcd}
\quad=\quad
\begin{tikzcd}
	B & I_3 \\
	B & Z \\
	B & P
	\arrow["{r_3}", from=1-1, to=1-2]
	\arrow[""{name=0, anchor=center, inner sep=0}, Rightarrow, no head, from=1-1, to=2-1]
	\arrow[""{name=1, anchor=center, inner sep=0}, "{z_1}", from=1-2, to=2-2]
	\arrow["{z_3}", from=2-1, to=2-2]
	\arrow[""{name=2, anchor=center, inner sep=0}, Rightarrow, no head, from=2-1, to=3-1]
	\arrow[""{name=3, anchor=center, inner sep=0}, "{p_z}", from=2-2, to=3-2]
	\arrow["v"', from=3-1, to=3-2]
	\arrow["{\SIGMA^{\zeta_1}}"{description}, draw=none, from=0, to=1]
	\arrow["{\SIGMA^{\eta_z}}"{description}, draw=none, from=2, to=3]
\end{tikzcd}
\end{equation}
and
\begin{equation}\label{eq:psi_equality}
\begin{tikzcd}
	B & I_2 \\
	B & X & P \\
	B & Q
	\arrow["{r_2}", from=1-1, to=1-2]
	\arrow[""{name=0, anchor=center, inner sep=0}, Rightarrow, no head, from=1-1, to=2-1]
	\arrow[""{name=1, anchor=center, inner sep=0}, "{x_2}", from=1-2, to=2-2]
	\arrow["{p_y y_1}", curve={height=-6pt}, from=1-2, to=2-3]
	\arrow["{x_3}", from=2-1, to=2-2]
	\arrow[""{name=2, anchor=center, inner sep=0}, Rightarrow, no head, from=2-1, to=3-1]
	\arrow["\psi", Rightarrow, from=2-2, to=2-3]
	\arrow[""{name=3, anchor=center, inner sep=0}, "{q_x}", from=2-2, to=3-2]
	\arrow["{q_y}", curve={height=-6pt}, from=2-3, to=3-2]
	\arrow["w"', from=3-1, to=3-2]
	\arrow["{\SIGMA^{\delta_2}}"{description}, draw=none, from=0, to=1]
	\arrow["{\SIGMA^{\theta_x}}"{description}, draw=none, from=2, to=3]
\end{tikzcd}
\quad=\quad
\begin{tikzcd}
	B & I_2 \\
	B & Y \\
	B & P \\
	B & Q \rlap{\,.}
	\arrow["{r_2}", from=1-1, to=1-2]
	\arrow[""{name=0, anchor=center, inner sep=0}, Rightarrow, no head, from=1-1, to=2-1]
	\arrow[""{name=1, anchor=center, inner sep=0}, "{y_1}", from=1-2, to=2-2]
	\arrow["{y_3}", from=2-1, to=2-2]
	\arrow[""{name=2, anchor=center, inner sep=0}, Rightarrow, no head, from=2-1, to=3-1]
	\arrow[""{name=3, anchor=center, inner sep=0}, "{p_y}", from=2-2, to=3-2]
	\arrow["v", from=3-1, to=3-2]
	\arrow[""{name=4, anchor=center, inner sep=0}, Rightarrow, no head, from=3-1, to=4-1]
	\arrow[""{name=5, anchor=center, inner sep=0}, "{q_y}", from=3-2, to=4-2]
	\arrow["w"', from=4-1, to=4-2]
	\arrow["{\SIGMA^{\epsilon_1}}"{description}, draw=none, from=0, to=1]
	\arrow["{\SIGMA^{\eta_y}}"{description}, draw=none, from=2, to=3]
	\arrow["{\SIGMA^{\theta_y}}"{description}, draw=none, from=4, to=5]
\end{tikzcd}
\end{equation}

As for the other composition, we can form a composite of $\overline{\beta} \cdot \overline{\alpha}$ using the $\Sigma$-squares $\Sigma^{\theta_x}$ and $\Sigma^{\theta_y \odot \eta_y}$ from \cref{eq:psi_equality} above:
\[\begin{tikzcd}
	A && {I_1} & B \\
	&& X & B \\
	A & {I_2} & Q & B \\
	&& P & B \\
	&& Y & B \\
	A && {I_3} & B \rlap{\,.}
	\arrow["f", from=1-1, to=1-3]
	\arrow[Rightarrow, no head, from=1-1, to=3-1]
	\arrow[""{name=0, anchor=center, inner sep=0}, "{x_1}"', from=1-3, to=2-3]
	\arrow["\alpha"', shorten <=17pt, shorten >=17pt, Rightarrow, from=1-3, to=3-1]
	\arrow["{r_1}"', from=1-4, to=1-3]
	\arrow[""{name=1, anchor=center, inner sep=0}, Rightarrow, no head, from=1-4, to=2-4]
	\arrow[""{name=2, anchor=center, inner sep=0}, "{q_x}", from=2-3, to=3-3]
	\arrow["{x_3}"', from=2-4, to=2-3]
	\arrow[""{name=3, anchor=center, inner sep=0}, Rightarrow, no head, from=2-4, to=3-4]
	\arrow["g", from=3-1, to=3-2]
	\arrow[Rightarrow, no head, from=3-1, to=6-1]
	\arrow["\beta"', shorten <=21pt, shorten >=21pt, Rightarrow, from=3-1, to=6-3]
	\arrow[""{name=4, anchor=center, inner sep=0}, "{x_2}", from=3-2, to=2-3]
	\arrow[""{name=5, anchor=center, inner sep=0}, "{y_1}"', from=3-2, to=5-3]
	\arrow["w"', from=3-4, to=3-3]
	\arrow[""{name=6, anchor=center, inner sep=0}, Rightarrow, no head, from=3-4, to=4-4]
	\arrow[""{name=7, anchor=center, inner sep=0}, "{q_y}"', from=4-3, to=3-3]
	\arrow["v"', from=4-4, to=4-3]
	\arrow[""{name=8, anchor=center, inner sep=0}, Rightarrow, no head, from=4-4, to=5-4]
	\arrow[""{name=9, anchor=center, inner sep=0}, "{p_y}"', from=5-3, to=4-3]
	\arrow["{y_3}"', from=5-4, to=5-3]
	\arrow[""{name=10, anchor=center, inner sep=0}, Rightarrow, no head, from=5-4, to=6-4]
	\arrow["h"', from=6-1, to=6-3]
	\arrow[""{name=11, anchor=center, inner sep=0}, "{y_2}", from=6-3, to=5-3]
	\arrow["{r_3}", from=6-4, to=6-3]
	\arrow["{\SIGMA^{\delta_1}}"{description}, draw=none, from=0, to=1]
	\arrow["{\SIGMA^{\theta_x}}"{description}, draw=none, from=2, to=3]
	\arrow["\psi", shorten <=6pt, shorten >=6pt, Rightarrow, from=4, to=5]
	\arrow["{\SIGMA^{\theta_y}}"{description}, draw=none, from=7, to=6]
	\arrow["{\SIGMA^{\eta_y}}"{description}, draw=none, from=9, to=8]
	\arrow["{\SIGMA^{\epsilon_2}}"{description}, draw=none, from=11, to=10]
\end{tikzcd}\]
We then find $\overline{\gamma} \cdot (\overline{\beta} \cdot \overline{\alpha})$, using the data
\begin{equation}\label{eq:phi_equality}
\begin{tikzcd}
	B & {I_3} \\
	B & Q & Z \\
	B & E
	\arrow["{r_3}", from=1-1, to=1-2]
	\arrow[""{name=0, anchor=center, inner sep=0}, Rightarrow, no head, from=1-1, to=2-1]
	\arrow[""{name=1, anchor=center, inner sep=0}, "{q_y p_y y_2}", from=1-2, to=2-2]
	\arrow["{z_1}", curve={height=-6pt}, from=1-2, to=2-3]
	\arrow["w", from=2-1, to=2-2]
	\arrow[""{name=2, anchor=center, inner sep=0}, Rightarrow, no head, from=2-1, to=3-1]
	\arrow["\phi"', Rightarrow, from=2-2, to=2-3]
	\arrow[""{name=3, anchor=center, inner sep=0}, "{e_y}", from=2-2, to=3-2]
	\arrow["{e_z}", curve={height=-6pt}, from=2-3, to=3-2]
	\arrow["e"', from=3-1, to=3-2]
	\arrow["{\SIGMA^{(\ast)}}"{description}, draw=none, from=0, to=1]
	\arrow["{\SIGMA^{\kappa_y}}"{description}, draw=none, from=2, to=3]
\end{tikzcd}
\quad=\quad
\begin{tikzcd}
	B & {I_3} \\
	B & Z \\
	B & E
	\arrow["{r_3}", from=1-1, to=1-2]
	\arrow[""{name=0, anchor=center, inner sep=0}, Rightarrow, no head, from=1-1, to=2-1]
	\arrow[""{name=1, anchor=center, inner sep=0}, "{z_1}", from=1-2, to=2-2]
	\arrow["{z_3}", from=2-1, to=2-2]
	\arrow[""{name=2, anchor=center, inner sep=0}, Rightarrow, no head, from=2-1, to=3-1]
	\arrow[""{name=3, anchor=center, inner sep=0}, "{e_z}", from=2-2, to=3-2]
	\arrow["e"', from=3-1, to=3-2]
	\arrow["{\SIGMA^{\zeta_1}}"{description}, draw=none, from=0, to=1]
	\arrow["{\SIGMA^{\kappa_z}}"{description}, draw=none, from=2, to=3]
\end{tikzcd}
\end{equation}
(where $\Sigma^{(\ast)}$ denotes $\Sigma^{\theta_y \odot \eta_y \odot \epsilon_2}$),
to be

\begin{equation}\label{eq:gamma_then_beta_alpha}
\begin{tikzcd}
	A && {I_1} & B \\
	&& X & B \\
	A & {I_2} & Q & B \\
	& Y & E & B \\
	A & {I_3} & Z & B \\
	A && {I_4} & {B \rlap{\,.}}
	\arrow["f", from=1-1, to=1-3]
	\arrow[Rightarrow, no head, from=1-1, to=3-1]
	\arrow[""{name=0, anchor=center, inner sep=0}, "{x_1}"', from=1-3, to=2-3]
	\arrow["\alpha"', shorten <=17pt, shorten >=17pt, Rightarrow, from=1-3, to=3-1]
	\arrow["{r_1}"', from=1-4, to=1-3]
	\arrow[""{name=1, anchor=center, inner sep=0}, Rightarrow, no head, from=1-4, to=2-4]
	\arrow[""{name=2, anchor=center, inner sep=0}, "{q_x}"', from=2-3, to=3-3]
	\arrow["{x_3}"', from=2-4, to=2-3]
	\arrow[""{name=3, anchor=center, inner sep=0}, Rightarrow, no head, from=2-4, to=3-4]
	\arrow["g", from=3-1, to=3-2]
	\arrow[Rightarrow, no head, from=3-1, to=5-1]
	\arrow[""{name=4, anchor=center, inner sep=0}, "{x_2}", from=3-2, to=2-3]
	\arrow["{y_1}", from=3-2, to=4-2]
	\arrow["\beta"', shorten <=11pt, shorten >=11pt, Rightarrow, from=3-2, to=5-1]
	\arrow["w"', from=3-4, to=3-3]
	\arrow[""{name=5, anchor=center, inner sep=0}, Rightarrow, no head, from=3-4, to=4-4]
	\arrow[""{name=6, anchor=center, inner sep=0}, "{q_y p_y}"'{pos=0.4}, from=4-2, to=3-3]
	\arrow[""{name=7, anchor=center, inner sep=0}, "{e_y}"', from=3-3, to=4-3]
	\arrow["e"', from=4-4, to=4-3]
	\arrow[""{name=8, anchor=center, inner sep=0}, Rightarrow, no head, from=4-4, to=5-4]
	\arrow["h", from=5-1, to=5-2]
	\arrow[Rightarrow, no head, from=5-1, to=6-1]
	\arrow["\gamma"', shorten <=14pt, shorten >=14pt, Rightarrow, from=5-1, to=6-3]
	\arrow["{y_2}", from=5-2, to=4-2]
	\arrow[""{name=9, anchor=center, inner sep=0}, "{z_1}"', from=5-2, to=5-3]
	\arrow[""{name=10, anchor=center, inner sep=0}, "{e_z}", from=5-3, to=4-3]
	\arrow["{z_3}"', from=5-4, to=5-3]
	\arrow[""{name=11, anchor=center, inner sep=0}, Rightarrow, no head, from=5-4, to=6-4]
	\arrow["j"', from=6-1, to=6-3]
	\arrow[""{name=12, anchor=center, inner sep=0}, "{z_2}", from=6-3, to=5-3]
	\arrow["{r_4}", from=6-4, to=6-3]
	\arrow["{\SIGMA^{\delta_1}}"{description}, draw=none, from=0, to=1]
	\arrow["{\SIGMA^{\theta_x}}"{description}, draw=none, from=2, to=3]
	\arrow["\psi", shorten <=10pt, shorten >=10pt, Rightarrow, from=4, to=6]
	\arrow["\phi", shorten <=13pt, shorten >=13pt, Rightarrow, from=6, to=9]
	\arrow["{\SIGMA^{\kappa_y}}"{description}, draw=none, from=7, to=5]
	\arrow["{\SIGMA^{\kappa_z}}"{description}, draw=none, from=10, to=8]
	\arrow["{\SIGMA^{\zeta_2}}"{description}, draw=none, from=12, to=11]
\end{tikzcd}
\end{equation}
We want to show that this is an equivalent 2-morphism to the composite \eqref{eq:gamma_beta_then_alpha} above.
Using Rule 3b of \cref{pro:useful_rules} applied to $\Sigma^{\kappa_z}$ and $\Sigma^{\kappa_y \odot \theta_y \odot \eta_z}$, we obtain
\begin{equation}\label{eq:upsilon_equality}
\begin{tikzcd}
	B & Z \\
	B & E & E \\
	B & {\widehat{E}}
	\arrow["{z_3}", from=1-1, to=1-2]
	\arrow[""{name=0, anchor=center, inner sep=0}, Rightarrow, no head, from=1-1, to=2-1]
	\arrow[""{name=1, anchor=center, inner sep=0}, "{e_z}", from=1-2, to=2-2]
	\arrow["{e_y q_y p_z}", curve={height=-12pt}, from=1-2, to=2-3]
	\arrow["e", from=2-1, to=2-2]
	\arrow[""{name=2, anchor=center, inner sep=0}, Rightarrow, no head, from=2-1, to=3-1]
	\arrow["\upsilon", shorten <=3pt, shorten >=3pt, Rightarrow, from=2-2, to=2-3]
	\arrow[""{name=3, anchor=center, inner sep=0}, "{\widehat{e}}", from=2-2, to=3-2]
	\arrow["{\widehat{e}}", curve={height=-12pt}, from=2-3, to=3-2]
	\arrow["{\widehat{e}'}"', from=3-1, to=3-2]
	\arrow["{\SIGMA^{\kappa_z}}"{description}, draw=none, from=0, to=1]
	\arrow["{\SIGMA^{\lambda}}"{description}, draw=none, from=2, to=3]
\end{tikzcd}
\quad=\quad
\begin{tikzcd}[row sep=small]
	B & Z \\
	B & P \\
	B & Q \\
	B & E \\
	B & {\widehat{E}} \rlap{\,.}
	\arrow["{z_3}", from=1-1, to=1-2]
	\arrow[""{name=0, anchor=center, inner sep=0}, Rightarrow, no head, from=1-1, to=2-1]
	\arrow[""{name=1, anchor=center, inner sep=0}, "{p_z}", from=1-2, to=2-2]
	\arrow["v", from=2-1, to=2-2]
	\arrow[""{name=2, anchor=center, inner sep=0}, Rightarrow, no head, from=2-1, to=3-1]
	\arrow[""{name=3, anchor=center, inner sep=0}, "{q_y}", from=2-2, to=3-2]
	\arrow["w", from=3-1, to=3-2]
	\arrow[""{name=4, anchor=center, inner sep=0}, Rightarrow, no head, from=3-1, to=4-1]
	\arrow[""{name=5, anchor=center, inner sep=0}, "{e_y}", from=3-2, to=4-2]
	\arrow["e", from=4-1, to=4-2]
	\arrow[""{name=6, anchor=center, inner sep=0}, Rightarrow, no head, from=4-1, to=5-1]
	\arrow[""{name=7, anchor=center, inner sep=0}, "{\widehat{e}}", from=4-2, to=5-2]
	\arrow["{\widehat{e}'}"', from=5-1, to=5-2]
	\arrow["{\SIGMA^{\eta_z}}"{marking, allow upside down}, draw=none, from=0, to=1]
	\arrow["{\SIGMA^{\theta_y}}"{marking, allow upside down}, draw=none, from=2, to=3]
	\arrow["{\SIGMA^{\kappa_y}}"{marking, allow upside down}, draw=none, from=4, to=5]
	\arrow["{\SIGMA^{\lambda}}"{marking, allow upside down}, draw=none, from=6, to=7]
\end{tikzcd}
\end{equation}
Using $\widehat{e}$ and $\upsilon$ to form a $\Sigma$-extension of $\overline{\gamma} \cdot (\overline{\beta} \cdot \overline{\alpha})$ and $\widehat{e} e_y$ to form an extension of
$(\overline{\gamma} \cdot \overline{\beta}) \cdot \overline{\alpha}$, the relevant parts of the 2-morphism diagrams become, respectively,

\[\begin{tikzcd}
	A & {I_2} && X & B \\
	& Y & P & Q & B \\
	&&& E & B \\
	&&& {\widehat{E}} & B \\
	&&& E & B \\
	A & {I_3} & {} & Z & B
	\arrow[""{name=0, anchor=center, inner sep=0}, "g", from=1-1, to=1-2]
	\arrow[Rightarrow, no head, from=1-1, to=6-1]
	\arrow[""{name=1, anchor=center, inner sep=0}, "{{x_2}}", from=1-2, to=1-4]
	\arrow["{{y_1}}"', from=1-2, to=2-2]
	\arrow[""{name=2, anchor=center, inner sep=0}, "{{q_x}}"', from=1-4, to=2-4]
	\arrow["{{x_3}}"', from=1-5, to=1-4]
	\arrow[""{name=3, anchor=center, inner sep=0}, Rightarrow, no head, from=1-5, to=2-5]
	\arrow[""{name=4, anchor=center, inner sep=0}, "{{p_y}}"', from=2-2, to=2-3]
	\arrow["{{q_y}}"', from=2-3, to=2-4]
	\arrow[""{name=5, anchor=center, inner sep=0}, "{{e_y}}"', from=2-4, to=3-4]
	\arrow["w"', from=2-5, to=2-4]
	\arrow[""{name=6, anchor=center, inner sep=0}, Rightarrow, no head, from=2-5, to=3-5]
	\arrow[""{name=7, anchor=center, inner sep=0}, "{{\widehat{e}}}"', from=3-4, to=4-4]
	\arrow["e"', from=3-5, to=3-4]
	\arrow[""{name=8, anchor=center, inner sep=0}, Rightarrow, no head, from=3-5, to=4-5]
	\arrow["{{\widehat{e}'}}"', from=4-5, to=4-4]
	\arrow[""{name=9, anchor=center, inner sep=0}, Rightarrow, no head, from=4-5, to=5-5]
	\arrow[""{name=10, anchor=center, inner sep=0}, "{{\widehat{e}}}", from=5-4, to=4-4]
	\arrow["e", from=5-5, to=5-4]
	\arrow[""{name=11, anchor=center, inner sep=0}, Rightarrow, no head, from=5-5, to=6-5]
	\arrow[""{name=12, anchor=center, inner sep=0}, "h"', from=6-1, to=6-2]
	\arrow["{{y_2}}", from=6-2, to=2-2]
	\arrow[""{name=13, anchor=center, inner sep=0}, draw=none, from=6-2, to=6-3]
	\arrow["{{z_1}}"', from=6-2, to=6-4]
	\arrow[""{name=14, anchor=center, inner sep=0}, curve={height=-25pt}, from=6-4, to=4-4]
	\arrow[""{name=15, anchor=center, inner sep=0}, "{{\widehat{e}e_z}}", curve={height=-40pt}, from=6-4, to=4-4]
	\arrow[""{name=16, anchor=center, inner sep=0}, "{{e_z}}", from=6-4, to=5-4]
	\arrow["{{z_3}}", from=6-5, to=6-4]
	\arrow["\beta"', shorten <=32pt, shorten >=32pt, Rightarrow, from=0, to=12]
	\arrow["\psi"', shorten <=3pt, shorten >=3pt, Rightarrow, from=1, to=2-3]
	\arrow["{{\SIGMA^{\theta_x}}}"{description}, draw=none, from=2, to=3]
	\arrow["{{\widehat{e}\phi}}", shorten <=26pt, shorten >=26pt, Rightarrow, from=4, to=13]
	\arrow["{{\SIGMA^{\kappa_y}}}"{description}, draw=none, from=5, to=6]
	\arrow["{{\SIGMA^\lambda}}"{description}, draw=none, from=7, to=8]
	\arrow["\upsilon"'{pos=0.4,yshift=1pt}, Rightarrow, shorten <=2pt, shorten >=2pt, from=5-4, to=14]
	\arrow["{{\SIGMA^\lambda}}"{description}, draw=none, from=10, to=9]
	\arrow["\upsilon"{yshift=1pt}, Rightarrow, shorten <=2pt, shorten >=2pt, from=15, to=14]
	\arrow["{{\SIGMA^{\kappa_z}}}"{description}, draw=none, from=16, to=11]
\end{tikzcd}\]

and

\[\begin{tikzcd}[row sep=scriptsize]
	A & {I_2} & X & B \\
	&& Q & B \\
	&& E & B \\
	&& {\widehat{E}} & B \\
	&& E & B \\
	&& Q & B \\
	& Y & P & B \\
	A & {I_3} & Z & B \rlap{\,.}
	\arrow[""{name=0, anchor=center, inner sep=0}, "g", from=1-1, to=1-2]
	\arrow[Rightarrow, no head, from=1-1, to=8-1]
	\arrow[""{name=1, anchor=center, inner sep=0}, "{x_2}", from=1-2, to=1-3]
	\arrow["{y_1}"', from=1-2, to=7-2]
	\arrow[""{name=2, anchor=center, inner sep=0}, "{q_x}"', from=1-3, to=2-3]
	\arrow["{x_3}"', from=1-4, to=1-3]
	\arrow[""{name=3, anchor=center, inner sep=0}, Rightarrow, no head, from=1-4, to=2-4]
	\arrow[""{name=4, anchor=center, inner sep=0}, "{e_y}"', from=2-3, to=3-3]
	\arrow["w"', from=2-4, to=2-3]
	\arrow[""{name=5, anchor=center, inner sep=0}, Rightarrow, no head, from=2-4, to=3-4]
	\arrow[""{name=6, anchor=center, inner sep=0}, "{\widehat{e}}"', from=3-3, to=4-3]
	\arrow["e"', from=3-4, to=3-3]
	\arrow[""{name=7, anchor=center, inner sep=0}, Rightarrow, no head, from=3-4, to=4-4]
	\arrow["{\widehat{e}'}"', from=4-4, to=4-3]
	\arrow[""{name=8, anchor=center, inner sep=0}, Rightarrow, no head, from=4-4, to=5-4]
	\arrow[""{name=9, anchor=center, inner sep=0}, "{\widehat{e}}", from=5-3, to=4-3]
	\arrow["e"', from=5-4, to=5-3]
	\arrow[""{name=10, anchor=center, inner sep=0}, Rightarrow, no head, from=5-4, to=6-4]
	\arrow[""{name=11, anchor=center, inner sep=0}, "{e_y}", from=6-3, to=5-3]
	\arrow["w"', from=6-4, to=6-3]
	\arrow[""{name=12, anchor=center, inner sep=0}, Rightarrow, no head, from=6-4, to=7-4]
	\arrow[""{name=13, anchor=center, inner sep=0}, "{p_y}", from=7-2, to=7-3]
	\arrow[""{name=14, anchor=center, inner sep=0}, "{q_y}", from=7-3, to=6-3]
	\arrow["v"', from=7-4, to=7-3]
	\arrow[""{name=15, anchor=center, inner sep=0}, Rightarrow, no head, from=7-4, to=8-4]
	\arrow[""{name=16, anchor=center, inner sep=0}, "h"', from=8-1, to=8-2]
	\arrow["{y_2}", from=8-2, to=7-2]
	\arrow[""{name=17, anchor=center, inner sep=0}, "{z_1}"', from=8-2, to=8-3]
	\arrow[""{name=18, anchor=center, inner sep=0}, "{p_z}", from=8-3, to=7-3]
	\arrow["{z_3}", from=8-4, to=8-3]
	\arrow["\beta"', shorten <=45pt, shorten >=45pt, Rightarrow, from=0, to=16]
	\arrow["\widehat{e}e_y\psi"', shorten <=38pt, shorten >=38pt, Rightarrow, from=1, to=13]
	\arrow["{\SIGMA^{\theta_x}}"{description}, draw=none, from=2, to=3]
	\arrow["{\SIGMA^{\kappa_y}}"{description}, draw=none, from=4, to=5]
	\arrow["{\SIGMA^\lambda}"{description}, draw=none, from=6, to=7]
	\arrow["{\SIGMA^\lambda}"{description}, draw=none, from=9, to=8]
	\arrow["{\SIGMA^{\kappa_y}}"{description}, draw=none, from=11, to=10]
	\arrow["\omega"', shorten <=6pt, shorten >=6pt, Rightarrow, from=13, to=17]
	\arrow["{\SIGMA^{\theta_y}}"{description}, draw=none, from=14, to=12]
	\arrow["{\SIGMA^{\eta_z}}"{description}, draw=none, from=18, to=15]
\end{tikzcd}\]
Note that the $\Sigma$-squares on the right-hand side of each diagram agree by \cref{eq:upsilon_equality}.
For the remaining 2-cells, we have the composites \newline
\centerline{$\upsilon z_1 h \cdot \widehat{e}\phi h \cdot \widehat{e}e_yq_yp_y \beta \cdot \widehat{e}e_y\psi g\;\;$ and $\; \; \widehat{e}e_yq_y\omega h \cdot \widehat{e}e_yq_yp_y \beta \cdot \widehat{e}e_y\psi g$.}
 These will agree if $\upsilon z_1 \cdot \widehat{e}\phi = \widehat{e}e_yq_y\omega$.
Composing $\Sigma^{\lambda \odot \kappa_y \odot \theta_y \odot \eta_y \odot \epsilon_2}$ with $\widehat{e} \phi$ and $\upsilon z_1$ and using \cref{eq:phi_equality} and \cref{eq:upsilon_equality} obtain $\Sigma^{\lambda \odot \kappa_y \odot \theta_y \odot \eta_z \odot \zeta_1}$. Composing the same $\Sigma$-square with $\widehat{e}e_y q_y \omega$ and using \cref{eq:omega_equality} we obtain the same result. %
Since the 2-cells in the $\Sigma$-squares are invertible, this implies $\upsilon z_1 r_3 \cdot \widehat{e}\phi r_3 = \widehat{e}e_yq_y\omega r_3$ and hence we can apply Equification to obtain a $\Sg$-extension where the desired equality indeed holds. Thus, we have proved associativity.

We now show that the identity 2-cells are indeed identities with respect to vertical composition. To form the composite $\overline{\alpha} \circ \overline{\id}_{(f,I_1,r_1)}$
we use the data
\begin{equation*}
\begin{tikzcd}
	B & I_1 \\
	B & I_1 & X & {} \\
	B & X
	\arrow["r", from=1-1, to=1-2]
	\arrow[""{name=0, anchor=center, inner sep=0}, Rightarrow, no head, from=1-1, to=2-1]
	\arrow[""{name=1, anchor=center, inner sep=0}, Rightarrow, no head, from=1-2, to=2-2]
	\arrow["x_1", curve={height=-6pt}, from=1-2, to=2-3]
	\arrow["r", from=2-1, to=2-2]
	\arrow[""{name=2, anchor=center, inner sep=0}, Rightarrow, no head, from=2-1, to=3-1]
	\arrow["\id", Rightarrow, from=2-2, to=2-3]
	\arrow[""{name=3, anchor=center, inner sep=0}, "x_1", from=2-2, to=3-2]
	\arrow["\id", Rightarrow, no head, curve={height=-6pt}, from=2-3, to=3-2]
	\arrow["x_3"', from=3-1, to=3-2]
	\arrow["{\SIGMA^{\id}}"{description}, draw=none, from=0, to=1]
	\arrow["{\SIGMA^{\delta_1}}"{description}, draw=none, from=2, to=3]
\end{tikzcd}
=\qquad
\begin{tikzcd}
B & I_1 \\
B & X \\
B & X
\arrow["r", from=1-1, to=1-2]
\arrow[""{name=0, anchor=center, inner sep=0}, Rightarrow, no head, from=1-1, to=2-1]
\arrow[""{name=1, anchor=center, inner sep=0}, "x_1", from=1-2, to=2-2]
\arrow["x_3", from=2-1, to=2-2]
\arrow[""{name=2, anchor=center, inner sep=0}, Rightarrow, no head, from=2-1, to=3-1]
\arrow[""{name=3, anchor=center, inner sep=0}, Rightarrow, no head, from=2-2, to=3-2]
\arrow["x_3"', from=3-1, to=3-2]
\arrow["{\SIGMA^{\delta_1}}"{description}, draw=none, from=0, to=1]
\arrow["{\SIGMA^\id}"{description}, draw=none, from=2, to=3]
\end{tikzcd}
\end{equation*}
and easily compute the composite to be equal to $\overline{\alpha}$. The composite $\overline{\id}_{(g,I_2,r_2)} \circ \overline{\alpha}$ can be shown to also be equal to $\overline{\alpha}$ in a similar way.
\end{proof}

\subsection{\texorpdfstring{$\Sigma$}{Σ}-schemes and \texorpdfstring{$\Omega$}{Ω} 2-cells}\label{sec:sigma_schemes}

In order to define the horizontal composition and the associator, and prove the necessary properties about it, we will make use of a special kind of 2-cells between $\Sigma$-cospans that we present in this subsection. Here we define $\Sigma$-schemes, $\Sigma$-paths and $\Omega$ 2-cells, and state the properties which will have a role in the following. The proofs of the last three propositions will be provided in Appendix~\ref{sec:appendix}.

The Square, Equi-insertion and Equification axioms of a calculus of lax fractions posit the existence of certain $\Sigma$-squares without any uniqueness requirements, and so when these are used to construct $\Sigma$-cospans, the precise choice of data can lead to possibly different results. However, one would not expect these differences to be essential. $\Omega$ 2-cells are canonical isomorphisms that allow us to compare the $\Sigma$-cospans constructed using different choices of $\Sigma$-squares. $\Sigma$-schemes and $\Sigma$-paths are used to define and manipulate $\Omega$ 2-cells associated to multi-step constructions.

\subsubsection{The basic \texorpdfstring{$\Omega$}{Ω} 2-cells} Departing from $I\xleftarrow{r}B\xrightarrow{g}J$, with $r\in \Sigma$, consider two $\Sigma$-squares $\Sigma^{\alpha}$ and $\Sigma^{\alpha'}$ as below, and apply Rule 4' to obtain the equality

\begin{equation}\label{eq:basic-Omega}\begin{tikzcd}
	B & I && B & I \\
	J & {I_1} & {I_2} & J & {I_2} \\
	J & D && J & D \rlap{\,,}
	\arrow["r", from=1-1, to=1-2]
	\arrow[""{name=0, anchor=center, inner sep=0}, "g"', from=1-1, to=2-1]
	\arrow[""{name=1, anchor=center, inner sep=0}, "{{g_1}}", from=1-2, to=2-2]
	\arrow["{{g_2}}", from=1-2, to=2-3]
	\arrow["r", from=1-4, to=1-5]
	\arrow[""{name=2, anchor=center, inner sep=0}, "g"', from=1-4, to=2-4]
	\arrow[""{name=3, anchor=center, inner sep=0}, "{{g_2}}", from=1-5, to=2-5]
	\arrow["{{r_1}}", from=2-1, to=2-2]
	\arrow[""{name=4, anchor=center, inner sep=0}, Rightarrow, no head, from=2-1, to=3-1]
	\arrow["\theta", Rightarrow, from=2-2, to=2-3]
	\arrow[""{name=5, anchor=center, inner sep=0}, "{{d_1}}", from=2-2, to=3-2]
	\arrow["{{\text{\normalsize =}}}"{description}, draw=none, from=2-3, to=2-4]
	\arrow["{{d_2}}", from=2-3, to=3-2]
	\arrow["{{r_2}}", from=2-4, to=2-5]
	\arrow[""{name=6, anchor=center, inner sep=0}, Rightarrow, no head, from=2-4, to=3-4]
	\arrow[""{name=7, anchor=center, inner sep=0}, "{{d_2}}", from=2-5, to=3-5]
	\arrow["d"', from=3-1, to=3-2]
	\arrow["d"', from=3-4, to=3-5]
	\arrow["{{\SIGMA^{\alpha}}}"{description}, draw=none, from=0, to=1]
	\arrow["{{\SIGMA^{\alpha'}}}"{description}, draw=none, from=2, to=3]
	\arrow["{{\SIGMA^{\delta_1}}}"{description}, draw=none, from=4, to=5]
	\arrow["{{\SIGMA^{\delta_2}}}"{description}, draw=none, from=6, to=7]
\end{tikzcd}
\end{equation}
where $\theta$ is invertible.
    We can then form a 2-morphism as follows:
    \[\begin{tikzcd}
	I & {I_1} & J \\
	& D & J \\
	I & {I_2} & J \rlap{\,.}
	\arrow["{g_1}", from=1-1, to=1-2]
	\arrow[Rightarrow, no head, from=1-1, to=3-1]
	\arrow[""{name=0, anchor=center, inner sep=0}, "{d_1}"', from=1-2, to=2-2]
	\arrow["\theta"', shorten <=16pt, shorten >=16pt, Rightarrow, from=1-2, to=3-1]
	\arrow["{r_1}"', from=1-3, to=1-2]
	\arrow[""{name=1, anchor=center, inner sep=0}, Rightarrow, no head, from=1-3, to=2-3]
	\arrow["d"{description}, from=2-3, to=2-2]
	\arrow[""{name=2, anchor=center, inner sep=0}, Rightarrow, no head, from=2-3, to=3-3]
	\arrow["{g_2}"', from=3-1, to=3-2]
	\arrow[""{name=3, anchor=center, inner sep=0}, "{d_2}", from=3-2, to=2-2]
	\arrow["{r_2}", from=3-3, to=3-2]
	\arrow["{\SIGMA^{\delta_1}}"{description}, draw=none, from=0, to=1]
	\arrow["{\SIGMA^{\delta_2}}"{description}, draw=none, from=3, to=2]
\end{tikzcd}\]
    It is easy to see that this 2-morphism is $\approx$-independent of the data used in the application of Rule 4'. Indeed, given $e_1$, $e_2$, $e$, $\epsilon_1$, $\epsilon_2$ and $\theta'$, instead of $d_1$, $d_2$, $d$, $\delta_1$, $\delta_2$ and $\theta$, in \eqref{eq:basic-Omega}, apply Rule 4 to \[\adjustbox{scale=0.75}{\begin{tikzcd}
	{} & {} \\
	{} & {} \\
	{} & {}
	\arrow["{r_1}", from=1-1, to=1-2]
	\arrow[""{name=0, anchor=center, inner sep=0}, equals, from=1-1, to=2-1]
	\arrow[""{name=1, anchor=center, inner sep=0}, "{d_1}", from=1-2, to=2-2]
	\arrow["d"', from=2-1, to=2-2]
	\arrow[""{name=2, anchor=center, inner sep=0}, equals, from=2-1, to=3-1]
	\arrow["{r_2}"', from=3-1, to=3-2]
	\arrow[""{name=3, anchor=center, inner sep=0}, "{d_2}"', from=3-2, to=2-2]
	\arrow["{\SIGMA^{\delta_1}}"{description}, draw=none, from=0, to=1]
	\arrow["{\SIGMA^{\delta_2}}"{description}, draw=none, from=2, to=3]
\end{tikzcd}} \quad\text{and}\quad \adjustbox{scale=0.75}{\begin{tikzcd}
	{} & {} \\
	{} & {} \\
	{} & {}
	\arrow["{{r_1}}", from=1-1, to=1-2]
	\arrow[""{name=0, anchor=center, inner sep=0}, equals, from=1-1, to=2-1]
	\arrow[""{name=1, anchor=center, inner sep=0}, "{{e_1}}", from=1-2, to=2-2]
	\arrow["e"', from=2-1, to=2-2]
	\arrow[""{name=2, anchor=center, inner sep=0}, equals, from=2-1, to=3-1]
	\arrow["{{r_2}}"', from=3-1, to=3-2]
	\arrow[""{name=3, anchor=center, inner sep=0}, "{{e_2}}"', from=3-2, to=2-2]
	\arrow["{{\SIGMA^{\epsilon_1}}}"{description}, draw=none, from=0, to=1]
	\arrow["{{\SIGMA^{\epsilon_2}}}"{description}, draw=none, from=2, to=3]
\end{tikzcd}}.\] This leads to a common $\Sg$-extension of $[\theta, d_1, d_2, d, \delta_1, \delta_2]$ and $[\theta', e_1, e_2, e, \epsilon_1, \epsilon_2]$.

\begin{definition}
    We say that a 2-morphism constructed as above and the corresponding 2-cell of $\catx[\Sigma_{\ast}]$ are of \textbf{\em basic $\Omega$ type}. We denote this 2-cell by $$\Omega_{\alpha,\alpha'}\colon (g_1,r_1)\Rightarrow (g_2,r_2)$$
or simply by $\Omega$.
	For any $f\colon A\to I$ and any $s\colon B\to J$ in $\Sigma$, by composing $\Omega_{\alpha,\alpha'}$ with $f$ and $s$ on the left and on the right, respectively, we obtain the 2-cell
	$$(1_J, s)\circ \Omega_{\alpha,\alpha'} \circ (f,1_I)=[\theta \circ f, d_1,d_2,ds,\delta_1\oplus \id_s, \delta_2\oplus \id_s]\colon (g_1f,r_1s)\Rightarrow (g_2f,r_2s)\, .$$
	This 2-cell is also said to be of \textbf{\em basic $\Omega$ type}.
\end{definition}

The following lemma shows that basic $\Omega$ 2-cells behave well under composition.

\begin{lemma}\label{lem:basic-Omega} Given $\Sigma$-squares
$\; \begin{tikzcd}
	B & I \\
	J & {B_i}
	\arrow["r", from=1-1, to=1-2]
	\arrow[""{name=0, anchor=center, inner sep=0}, "g"', from=1-1, to=2-1]
	\arrow[""{name=1, anchor=center, inner sep=0}, "{g_i}", from=1-2, to=2-2]
	\arrow["{r_i}"', from=2-1, to=2-2]
	\arrow["{\SIGMA^{\alpha_i}}"{description}, draw=none, from=0, to=1]
\end{tikzcd}$ for $i=1,2,3$, we have:

\begin{enumerate}
\item[(1)] $\Omega_{\alpha_1,\alpha_2}$ is an invertible 2-cell between $\Sigma$-cospans and $\Omega_{\alpha_1,\alpha_2}^{-1}=\Omega_{\alpha_2,\alpha_1}$,
    \item[(2)]
    $\Omega_{\alpha_2, \alpha_3}\cdot \Omega_{\alpha_1,\alpha_2}=\Omega_{\alpha_1,\alpha_3}$.
    \end{enumerate}
\end{lemma}

\begin{proof} (1) is clear.

For (2), using Rule 4', we successively consider invertible 2-cells $\theta_1$ and $\theta_2$ in $\catx$ such that

\[
\begin{tikzcd}
	B & I \\
	J & {B_1} & {B_2} \\
	J & D
	\arrow["r", from=1-1, to=1-2]
	\arrow[""{name=0, anchor=center, inner sep=0}, "g"', from=1-1, to=2-1]
	\arrow[""{name=1, anchor=center, inner sep=0}, "{g_1}", from=1-2, to=2-2]
	\arrow["{g_2}", curve={height=-6pt}, from=1-2, to=2-3]
	\arrow["{r_1}"', from=2-1, to=2-2]
	\arrow[""{name=2, anchor=center, inner sep=0}, equals, from=2-1, to=3-1]
	\arrow["{\theta_1}", Rightarrow, from=2-2, to=2-3]
	\arrow[""{name=3, anchor=center, inner sep=0}, "{d_1}", from=2-2, to=3-2]
	\arrow["{d_2}", curve={height=-6pt}, from=2-3, to=3-2]
	\arrow["d"', from=3-1, to=3-2]
	\arrow["{\SIGMA^{\alpha_1}}"{description}, draw=none, from=0, to=1]
	\arrow["\SIGMA"{description}, draw=none, from=2, to=3]
\end{tikzcd}
\hspace{1.5mm} = \hspace{1.5mm}
\begin{tikzcd}
	B & I \\
	J & {B_2} \\
	J & D
	\arrow["r", from=1-1, to=1-2]
	\arrow[""{name=0, anchor=center, inner sep=0}, "g"', from=1-1, to=2-1]
	\arrow[""{name=1, anchor=center, inner sep=0}, "{{g_2}}", from=1-2, to=2-2]
	\arrow["{r_2}"', from=2-1, to=2-2]
	\arrow[""{name=2, anchor=center, inner sep=0}, equals, from=2-1, to=3-1]
	\arrow[""{name=3, anchor=center, inner sep=0}, "{d_2}", from=2-2, to=3-2]
	\arrow["d"', from=3-1, to=3-2]
	\arrow["{{\SIGMA^{\alpha_2}}}"{description}, draw=none, from=0, to=1]
	\arrow["\SIGMA"{description}, draw=none, from=2, to=3]
\end{tikzcd}
\hspace{5mm} \text{ and } \hspace{5mm}
\begin{tikzcd}
	B & I \\
	J & {B_1} \\
	J & D & {B_3} \\
	J & E
	\arrow["r", from=1-1, to=1-2]
	\arrow[""{name=0, anchor=center, inner sep=0}, "g"', from=1-1, to=2-1]
	\arrow[""{name=1, anchor=center, inner sep=0}, "{{g_2}}", from=1-2, to=2-2]
	\arrow[""{name=2, anchor=center, inner sep=0}, "{g_3}", curve={height=-12pt}, from=1-2, to=3-3]
	\arrow["{r_2}"', from=2-1, to=2-2]
	\arrow[""{name=3, anchor=center, inner sep=0}, equals, from=2-1, to=3-1]
	\arrow[""{name=4, anchor=center, inner sep=0}, "{d_2}", from=2-2, to=3-2]
	\arrow["d"', from=3-1, to=3-2]
	\arrow[""{name=5, anchor=center, inner sep=0}, equals, from=3-1, to=4-1]
	\arrow[""{name=6, anchor=center, inner sep=0}, "{e_1}", from=3-2, to=4-2]
	\arrow["{e_2}", curve={height=-6pt}, from=3-3, to=4-2]
	\arrow["e"', from=4-1, to=4-2]
	\arrow["{{\SIGMA^{\alpha_2}}}"{description}, draw=none, from=0, to=1]
	\arrow["\SIGMA"{description}, draw=none, from=3, to=4]
	\arrow["{\theta_2}", between={0.2}{0.8}, Rightarrow, from=3-2, to=3-3]
	\arrow["\SIGMA"{description}, draw=none, from=5, to=6]
\end{tikzcd}
\hspace{1.5mm} = \hspace{1.5mm}
\begin{tikzcd}
	B & I \\
	J & {B_3} \\
	J & E
	\arrow["r", from=1-1, to=1-2]
	\arrow[""{name=0, anchor=center, inner sep=0}, "g"', from=1-1, to=2-1]
	\arrow[""{name=1, anchor=center, inner sep=0}, "{g_3}", from=1-2, to=2-2]
	\arrow["{r_3}"', from=2-1, to=2-2]
	\arrow[""{name=2, anchor=center, inner sep=0}, equals, from=2-1, to=3-1]
	\arrow[""{name=3, anchor=center, inner sep=0}, "{e_2}", from=2-2, to=3-2]
	\arrow["e"', from=3-1, to=3-2]
	\arrow["{\SIGMA^{\alpha_3}}"{description}, draw=none, from=0, to=1]
	\arrow["\SIGMA"{description}, draw=none, from=2, to=3]
\end{tikzcd}
\, .\]

We obtain the diagram
\[\begin{tikzcd}
	I & {B_1} & J \\
	& E & J \\
	I & {B_2} & J \\
	& E & J \\
	I & {B_3} & J \rlap{\,.}
	\arrow["{{g_1}}", from=1-1, to=1-2]
	\arrow[equals, from=1-1, to=3-1]
	\arrow[""{name=0, anchor=center, inner sep=0}, "{{e_1d_1}}"', from=1-2, to=2-2]
	\arrow["{{e_1\circ \theta_1}}"{pos=0.6}, shift right=4, between={0.3}{0.7}, Rightarrow, from=1-2, to=3-1]
	\arrow["{{r_1}}"', from=1-3, to=1-2]
	\arrow[""{name=1, anchor=center, inner sep=0}, equals, from=1-3, to=2-3]
	\arrow["e", from=2-3, to=2-2]
	\arrow[""{name=2, anchor=center, inner sep=0}, equals, from=2-3, to=3-3]
	\arrow["{{g_2}}"', from=3-1, to=3-2]
	\arrow[equals, from=3-1, to=5-1]
	\arrow[""{name=3, anchor=center, inner sep=0}, "{{e_1d_2}}", from=3-2, to=2-2]
	\arrow[""{name=4, anchor=center, inner sep=0}, "{{e_1d_2}}"', from=3-2, to=4-2]
	\arrow["{{\theta_2}}"', shift left=2, between={0.4}{0.8}, Rightarrow, from=3-2, to=5-1]
	\arrow["{{r_2}}", from=3-3, to=3-2]
	\arrow[""{name=5, anchor=center, inner sep=0}, equals, from=3-3, to=4-3]
	\arrow["e"', from=4-3, to=4-2]
	\arrow[""{name=6, anchor=center, inner sep=0}, equals, from=4-3, to=5-3]
	\arrow["{{g_3}}", from=5-1, to=5-2]
	\arrow[""{name=7, anchor=center, inner sep=0}, "{{e_2}}", from=5-2, to=4-2]
	\arrow["{{r_3}}"', from=5-3, to=5-2]
	\arrow["\SIGMA"{description}, draw=none, from=0, to=1]
	\arrow["\SIGMA"{description}, draw=none, from=3, to=2]
	\arrow["\SIGMA"{description}, draw=none, from=4, to=5]
	\arrow["\SIGMA"{description}, draw=none, from=7, to=6]
\end{tikzcd}\]

Vertically, this diagram is the juxtaposition of two 2-morphisms which represent $\Omega_{\alpha_1, \alpha_2}$ and  $\Omega_{\alpha_2, \alpha_3}$. Observe that $(\theta_2\cdot(e_1\circ \theta_1),\, e_1d_1, e_2)$ is a representative of the vertical composition $\Omega_{\alpha_2, \alpha_3}\cdot \Omega_{\alpha_1, \alpha_2}$, and also a representative of $\Omega_{\alpha_1, \alpha_3}$.
\end{proof}

\subsubsection{\texorpdfstring{$\Sigma$}{Σ}-schemes and \texorpdfstring{$\Sigma$}{Σ}-paths} A $\Sigma$-scheme is any diagram of the form

\[\begin{tikzcd}
	&&& \bullet & \bullet \\
	&& \bullet & \bullet \\
	& {} & {} \\
	\bullet & {} \\
	\bullet &&&& \bullet
	\arrow["{{r_1}}", from=1-4, to=1-5]
	\arrow["{{g_1}}"', from=1-4, to=2-4]
	\arrow["l", from=1-5, to=5-5]
	\arrow["{{r_2}}", from=2-3, to=2-4]
	\arrow["{{g_2}}"', from=2-3, to=3-3]
	\arrow["{{\text{\Large $\ldots$}}}"{description}, draw=none, from=3-2, to=3-3]
	\arrow["{{\text{\Large $\vdots$}}}"{description}, draw=none, from=3-2, to=4-2]
	\arrow["{{g_n}}"', from=4-1, to=5-1]
	\arrow["{{r_n}}"', from=4-2, to=4-1]
	\arrow["m", from=5-1, to=5-5]
\end{tikzcd}\]
obtained by means of vertical and horizontal composition of $\Sigma$-squares. For instance, the following three diagrams are $\Sigma$-schemes.

\begin{equation}\label{eq:(1)(2)(3)}
\begin{tikzcd}
	& B & I \\
	C & J & \bullet \\
	K && \bullet \\
	{} && {}
	\arrow["r", from=1-2, to=1-3]
	\arrow[""{name=0, anchor=center, inner sep=0}, "g"', from=1-2, to=2-2]
	\arrow[""{name=1, anchor=center, inner sep=0}, "{{g'}}", from=1-3, to=2-3]
	\arrow["s", from=2-1, to=2-2]
	\arrow[""{name=2, anchor=center, inner sep=0}, "h"', from=2-1, to=3-1]
	\arrow["{r'}", from=2-2, to=2-3]
	\arrow[""{name=3, anchor=center, inner sep=0}, "{h'}", from=2-3, to=3-3]
	\arrow["{s'}", from=3-1, to=3-3]
	\arrow[""{name=4, anchor=center, inner sep=0}, draw=none, from=3-1, to=4-1]
	\arrow[""{name=5, anchor=center, inner sep=0}, draw=none, from=3-3, to=4-3]
	\arrow["\SIGMA"{description}, draw=none, from=0, to=1]
	\arrow["\SIGMA"{description}, draw=none, from=2, to=3]
	\arrow["{{(1)}}"{description}, draw=none, from=4, to=5]
\end{tikzcd}
\qquad  \quad
\begin{tikzcd}
	& B & I \\
	C & J \\
	K & \bullet & \bullet \\
	{} && {}
	\arrow["r", from=1-2, to=1-3]
	\arrow["g"', from=1-2, to=2-2]
	\arrow[""{name=0, anchor=center, inner sep=0}, "{{g''}}", from=1-3, to=3-3]
	\arrow["s", from=2-1, to=2-2]
	\arrow[""{name=1, anchor=center, inner sep=0}, "h"', from=2-1, to=3-1]
	\arrow[""{name=2, anchor=center, inner sep=0}, "{h''}", from=2-2, to=3-2]
	\arrow["{{s''}}", from=3-1, to=3-2]
	\arrow[""{name=3, anchor=center, inner sep=0}, draw=none, from=3-1, to=4-1]
	\arrow["{{r''}}", from=3-2, to=3-3]
	\arrow[""{name=4, anchor=center, inner sep=0}, draw=none, from=3-3, to=4-3]
	\arrow["\SIGMA"{description}, draw=none, from=1, to=2]
	\arrow["\SIGMA"{description}, draw=none, from=2-2, to=0]
	\arrow["{{{(2)}}}"{description}, draw=none, from=3, to=4]
\end{tikzcd}
\qquad  \quad
\begin{tikzcd}
	& B & I \\
	C & J & \bullet \\
	K & \bullet & \bullet \\
	{} && {}
	\arrow["r", from=1-2, to=1-3]
	\arrow[""{name=0, anchor=center, inner sep=0}, "g"', from=1-2, to=2-2]
	\arrow[""{name=1, anchor=center, inner sep=0}, "{{g'}}", from=1-3, to=2-3]
	\arrow["s", from=2-1, to=2-2]
	\arrow[""{name=2, anchor=center, inner sep=0}, "h"', from=2-1, to=3-1]
	\arrow["{{r'}}", from=2-2, to=2-3]
	\arrow[""{name=3, anchor=center, inner sep=0}, "{h''}", from=2-2, to=3-2]
	\arrow[""{name=4, anchor=center, inner sep=0}, "{{h'''}}", from=2-3, to=3-3]
	\arrow["{{s''}}", from=3-1, to=3-2]
	\arrow[""{name=5, anchor=center, inner sep=0}, draw=none, from=3-1, to=4-1]
	\arrow["{{r'''}}", from=3-2, to=3-3]
	\arrow[""{name=6, anchor=center, inner sep=0}, draw=none, from=3-3, to=4-3]
	\arrow["\SIGMA"{description}, draw=none, from=0, to=1]
	\arrow["\SIGMA"{description}, draw=none, from=2, to=3]
	\arrow["\SIGMA"{description}, draw=none, from=3, to=4]
	\arrow["{{{{(3)}}}}"{description}, draw=none, from=5, to=6]
\end{tikzcd}
\end{equation}
We say that $(r_1, g_1, r_2, g_2, \dots, r_n,g_n)$ is the \textbf{\em left border} of the $\Sigma$-scheme and $(l,m)$ is the \textbf{\em right border}. A $\Sigma$-scheme with left border $(r_1,\dots, g_n)$ is said to be of \textbf{\em level $n$}. A $\Sigma$-scheme of level 1 is just a $\Sigma$-square.

Let $S$ be a $\Sigma$-scheme. Any $\Sigma$-square used in the formation of  $S$ is said to be a sub-$\Sigma$-square of $S$. Amongst all sub-$\Sigma$-squares of $S$, we are interested in those whose lower right vertex coincides with the lower right vertex of $S$ --- let us call them \textbf{\em replaceable $\Sigma$-squares} in $S$.
For instance, in (1) above,
\(
\begin{tikzcd}
	{} & {} \\
	{} & {}
	\arrow["{{r's}}", from=1-1, to=1-2]
	\arrow[""{name=0, anchor=center, inner sep=0}, "h"', from=1-1, to=2-1]
	\arrow[""{name=1, anchor=center, inner sep=0}, "{{h'}}", from=1-2, to=2-2]
	\arrow["{{s'}}"', from=2-1, to=2-2]
	\arrow["\SIGMA"{description}, shift left, draw=none, from=0, to=1]
\end{tikzcd}
\)
 is replaceable; in (2),
\(
\begin{tikzcd}
	{} & {} \\
	{} & {}
	\arrow["r", from=1-1, to=1-2]
	\arrow[""{name=0, anchor=center, inner sep=0}, "{{h''g}}"', from=1-1, to=2-1]
	\arrow[""{name=1, anchor=center, inner sep=0}, "{g''}", from=1-2, to=2-2]
	\arrow["{r''}"', from=2-1, to=2-2]
	\arrow["\SIGMA"{description}, shift left, draw=none, from=0, to=1]
\end{tikzcd}
\)
is replaceable; and in (3),
\(\begin{tikzcd}
	{} & {} \\
	{} & {}
	\arrow["{r's}", from=1-1, to=1-2]
	\arrow[""{name=0, anchor=center, inner sep=0}, "{h}"', from=1-1, to=2-1]
	\arrow[""{name=1, anchor=center, inner sep=0}, "{h'''}", from=1-2, to=2-2]
	\arrow["{r'''s''}"', from=2-1, to=2-2]
	\arrow["\SIGMA"{description}, shift left, draw=none, from=0, to=1]
\end{tikzcd}\,, \;
\)
 \(
\begin{tikzcd}
	{} & {} \\
	{} & {}
	\arrow["r", from=1-1, to=1-2]
	\arrow[""{name=0, anchor=center, inner sep=0}, "{h''g}"', from=1-1, to=2-1]
	\arrow[""{name=1, anchor=center, inner sep=0}, "{h'''g'}", from=1-2, to=2-2]
	\arrow["{r'''}"', from=2-1, to=2-2]
	\arrow["\SIGMA"{description}, shift left, draw=none, from=0, to=1]
\end{tikzcd}
 \)
 and
 \(
\begin{tikzcd}
	{} & {} \\
	{} & {}
	\arrow["r'", from=1-1, to=1-2]
	\arrow[""{name=0, anchor=center, inner sep=0}, "{h''}"', from=1-1, to=2-1]
	\arrow[""{name=1, anchor=center, inner sep=0}, "{h'''}", from=1-2, to=2-2]
	\arrow["{r'''}"', from=2-1, to=2-2]
	\arrow["\SIGMA"{description}, shift left, draw=none, from=0, to=1]
\end{tikzcd}
 \)
are all replaceable.

 A \textbf{\em $\Sigma$-step} from a $\Sigma$-scheme $S_1$ to a $\Sigma$-scheme $S_2$ with the same left border is a transformation of $S_1$ into $S_2$ which substitutes a replaceable $\Sigma$-square of $S_1$ by a $\Sigma$-square with the same left border.
We will indicate a $\Sigma$-step from $S_1$ to $S_2$ by a wavy arrow
$$S_1 \rightsquigarrow S_2\,.$$
For instance, the $\Sigma$-step from (1) to (3) which replaces
\(\;\begin{tikzcd}
	{} & {} \\
	{} & {}
	\arrow["{r's}", from=1-1, to=1-2]
	\arrow[""{name=0, anchor=center, inner sep=0}, "{h}"', from=1-1, to=2-1]
	\arrow[""{name=1, anchor=center, inner sep=0}, "{h'}", from=1-2, to=2-2]
	\arrow["{s'}"', from=2-1, to=2-2]
	\arrow["\SIGMA"{description}, draw=none, from=0, to=1]
\end{tikzcd}
\;\)
by
\(\;\begin{tikzcd}
	{} & {} \\
	{} & {}
	\arrow["{r's}", from=1-1, to=1-2]
	\arrow[""{name=0, anchor=center, inner sep=0}, "{h}"', from=1-1, to=2-1]
	\arrow[""{name=1, anchor=center, inner sep=0}, "{h'''}", from=1-2, to=2-2]
	\arrow["{r'''s''}"', from=2-1, to=2-2]
	\arrow["\SIGMA"{description}, draw=none, from=0, to=1]
\end{tikzcd}
\; \)
takes a $\Sigma$-scheme of the type
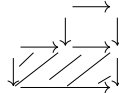
\[\adjustbox{scale=0.60}{\begin{tikzcd}
	& {} & {} \\
	{} & {} & {} \\
	{} & {} & {}
	\arrow[from=1-2, to=1-3]
	\arrow[from=1-2, to=2-2]
	\arrow[from=1-3, to=2-3]
	\arrow[""{name=0, anchor=center, inner sep=0}, from=2-1, to=2-2]
	\arrow[""{name=1, anchor=center, inner sep=0}, from=2-1, to=3-1]
	\arrow[""{name=2, anchor=center, inner sep=0}, from=2-2, to=2-3]
	\arrow[""{name=3, anchor=center, inner sep=0}, from=2-3, to=3-3]
	\arrow[no head, from=3-1, to=2-2]
	\arrow[""{name=4, anchor=center, inner sep=0}, draw=none, from=3-1, to=3-2]
	\arrow[from=3-1, to=3-3]
	\arrow[no head, from=3-2, to=2-3]
	\arrow[""{name=5, anchor=center, inner sep=0}, draw=none, from=3-2, to=3-3]
	\arrow[shorten <=4pt, shorten >=4pt, no head, from=1, to=0]
	\arrow[shorten <=7pt, shorten >=7pt, no head, from=4, to=2]
	\arrow[shorten <=4pt, shorten >=4pt, no head, from=5, to=3]
\end{tikzcd}
}\]
into another one of the same type by replacing the $\Sigma$-square corresponding to the shaded area. Observe that, as required, the two $\Sigma$-squares involved, when seen as $\Sigma$-schemes, have the same left border, namely $(r's,h)$.

As another example, consider the $\Sigma$-schemes of level 3
\begin{equation}\label{eq:(SR)}
\begin{tikzcd}
	&& {} & {} \\
	& {} & {} & {} \\
	{} & {} & {} & {} \\
	{} & {} & {} & {}
	\arrow["r", from=1-3, to=1-4]
	\arrow[""{name=0, anchor=center, inner sep=0}, "g"', from=1-3, to=2-3]
	\arrow[""{name=1, anchor=center, inner sep=0}, "{{g'}}", from=1-4, to=2-4]
	\arrow["s", from=2-2, to=2-3]
	\arrow["h", from=2-2, to=3-2]
	\arrow["{{r'}}"', from=2-3, to=2-4]
	\arrow["{{h'}}", from=2-3, to=4-3]
	\arrow["{{h''}}", from=2-4, to=4-4]
	\arrow["t", from=3-1, to=3-2]
	\arrow[""{name=2, anchor=center, inner sep=0}, "k"', from=3-1, to=4-1]
	\arrow["\SIGMA"{description}, shift left, draw=none, from=3-2, to=3-3]
	\arrow[""{name=3, anchor=center, inner sep=0}, "{k'}", from=3-2, to=4-2]
	\arrow["\SIGMA"{description}, shift left, draw=none, from=3-3, to=3-4]
	\arrow["{{t'}}"', from=4-1, to=4-2]
	\arrow["{{s'}}"', from=4-2, to=4-3]
	\arrow["{{r''}}"', from=4-3, to=4-4]
	\arrow["\SIGMA"{description}, shift left, draw=none, from=0, to=1]
	\arrow["\SIGMA"{description}, shift left, draw=none, from=2, to=3]
\end{tikzcd}
\qquad\text{and}\qquad
\begin{tikzcd}
	&& {} & {} \\
	& {} & {} & {} \\
	{} & {} & {} & {} \\
	{} & {} & {} & {}
	\arrow["r", from=1-3, to=1-4]
	\arrow[""{name=0, anchor=center, inner sep=0}, "g", from=1-3, to=2-3]
	\arrow[""{name=1, anchor=center, inner sep=0}, "{{g'}}", from=1-4, to=2-4]
	\arrow["s", from=2-2, to=2-3]
	\arrow[""{name=2, anchor=center, inner sep=0}, "h"', from=2-2, to=3-2]
	\arrow["{{r'}}"', from=2-3, to=2-4]
	\arrow[""{name=3, anchor=center, inner sep=0}, "{{{h_1}}}", from=2-3, to=3-3]
	\arrow[""{name=4, anchor=center, inner sep=0}, "{{{h_2}}}", from=2-4, to=3-4]
	\arrow["t", from=3-1, to=3-2]
	\arrow[""{name=5, anchor=center, inner sep=0}, "k", from=3-1, to=4-1]
	\arrow["{{{s_1}}}"', from=3-2, to=3-3]
	\arrow[""{name=6, anchor=center, inner sep=0}, "{{k'}}", from=3-2, to=4-2]
	\arrow["{{r_1}}"', from=3-3, to=3-4]
	\arrow[""{name=7, anchor=center, inner sep=0}, "{{k_1}}", from=3-3, to=4-3]
	\arrow[""{name=8, anchor=center, inner sep=0}, "{{k_2}}", from=3-4, to=4-4]
	\arrow["{{t'}}"', from=4-1, to=4-2]
	\arrow["{{{s_2}}}"', from=4-2, to=4-3]
	\arrow["{{r_2}}"', from=4-3, to=4-4]
	\arrow["\SIGMA"{description}, draw=none, from=0, to=1]
	\arrow["\SIGMA"{description}, draw=none, from=2, to=3]
	\arrow["\SIGMA"{description}, draw=none, from=3, to=4]
	\arrow["\SIGMA"{description}, draw=none, from=5, to=6]
	\arrow["\SIGMA"{description}, draw=none, from=6, to=7]
	\arrow["\SIGMA"{description}, draw=none, from=7, to=8]
\end{tikzcd}
\end{equation}

The $\Sigma$-step from the first of these to the second which replaces the $\Sigma$-square
\[\begin{tikzcd}
	{} & {} \\
	{} & {}
	\arrow["{r's}", from=1-1, to=1-2]
	\arrow[""{name=0, anchor=center, inner sep=0}, "{k'h}"', from=1-1, to=2-1]
	\arrow[""{name=1, anchor=center, inner sep=0}, "{h''}", from=1-2, to=2-2]
	\arrow["{r''s'}"', from=2-1, to=2-2]
	\arrow["\SIGMA"{description}, shift left, draw=none, from=0, to=1]
\end{tikzcd}
\qquad \text{by} \qquad
\begin{tikzcd}
	{} & {} \\
	{} & {}
	\arrow["{r's}", from=1-1, to=1-2]
	\arrow[""{name=0, anchor=center, inner sep=0}, "{k'h}"', from=1-1, to=2-1]
	\arrow[""{name=1, anchor=center, inner sep=0}, "{k_2h_2}", from=1-2, to=2-2]
	\arrow["{r_2s_2}"', from=2-1, to=2-2]
	\arrow["\SIGMA"{description}, shift left, draw=none, from=0, to=1]
\end{tikzcd}\]
takes a $\Sigma$-scheme of the form
\[
\adjustbox{scale=0.60}{\begin{tikzcd}
	&& {} & {} \\
	& {} & {} & {} \\
	{} & {} & {} & {} \\
	{} & {} & {} & {}
	\arrow[from=1-3, to=1-4]
	\arrow[from=1-3, to=2-3]
	\arrow[from=1-4, to=2-4]
	\arrow[""{name=0, anchor=center, inner sep=0}, from=2-2, to=2-3]
	\arrow[""{name=1, anchor=center, inner sep=0}, from=2-2, to=3-2]
	\arrow[""{name=2, anchor=center, inner sep=0}, from=2-3, to=2-4]
	\arrow[""{name=3, anchor=center, inner sep=0}, draw=none, from=2-4, to=3-4]
	\arrow[from=2-4, to=4-4]
	\arrow[from=3-1, to=3-2]
	\arrow[from=3-1, to=4-1]
	\arrow[no head, from=3-2, to=2-3]
	\arrow[""{name=4, anchor=center, inner sep=0}, from=3-2, to=4-2]
	\arrow[""{name=5, anchor=center, inner sep=0}, draw=none, from=3-4, to=4-4]
	\arrow[from=4-1, to=4-2]
	\arrow[no head, from=4-2, to=2-4]
	\arrow[""{name=6, anchor=center, inner sep=0}, draw=none, from=4-2, to=4-3]
	\arrow[from=4-2, to=4-4]
	\arrow[no head, from=4-3, to=3-4]
	\arrow[""{name=7, anchor=center, inner sep=0}, draw=none, from=4-3, to=4-4]
	\arrow[shorten <=4pt, shorten >=4pt, no head, from=1, to=0]
	\arrow[shorten <=11pt, shorten >=11pt, no head, from=4, to=2]
	\arrow[shorten >=11pt, no head, from=6, to=3]
	\arrow[shorten <=4pt, shorten >=4pt, no head, from=7, to=5]
\end{tikzcd}}
\]
into another one of the same type by again replacing the $\Sigma$-square corresponding to the shaded area.

 A \textbf{\em $\Sigma$-path} between $\Sigma$-schemes with the same left border is a finite sequence of $\Sigma$-steps (see also \cref{assumption-path}).

For $\Sigma$-schemes of level 2, we use the notations $d$ and $u$ to indicate the type of $\Sigma$-step involved: $d$ (down) means a replacement of type
\(\adjustbox{scale=0.60}{\begin{tikzcd}
	& {} & {} \\
	{} & {} & {} \\
	{} & {} & {}
	\arrow[from=1-2, to=1-3]
	\arrow[from=1-2, to=2-2]
	\arrow[from=1-3, to=2-3]
	\arrow[""{name=0, anchor=center, inner sep=0}, from=2-1, to=2-2]
	\arrow[""{name=1, anchor=center, inner sep=0}, from=2-1, to=3-1]
	\arrow[""{name=2, anchor=center, inner sep=0}, from=2-2, to=2-3]
	\arrow[""{name=3, anchor=center, inner sep=0}, from=2-3, to=3-3]
	\arrow[no head, from=3-1, to=2-2]
	\arrow[""{name=4, anchor=center, inner sep=0}, draw=none, from=3-1, to=3-2]
	\arrow[from=3-1, to=3-3]
	\arrow[no head, from=3-2, to=2-3]
	\arrow[""{name=5, anchor=center, inner sep=0}, draw=none, from=3-2, to=3-3]
	\arrow[shorten <=4pt, shorten >=4pt, no head, from=1, to=0]
	\arrow[shorten <=7pt, shorten >=7pt, no head, from=4, to=2]
	\arrow[shorten <=4pt, shorten >=4pt, no head, from=5, to=3]
\end{tikzcd}
}\)
 and $u$ (upper) means one of type
\(\adjustbox{scale=0.60}{ \begin{tikzcd}
	& {} & {} \\
	{} & {} & {} \\
	{} & {} & {}
	\arrow[""{name=0, anchor=center, inner sep=0}, from=1-2, to=1-3]
	\arrow[""{name=1, anchor=center, inner sep=0}, from=1-2, to=2-2]
	\arrow[""{name=2, anchor=center, inner sep=0}, from=1-3, to=2-3]
	\arrow[from=2-1, to=2-2]
	\arrow[from=2-1, to=3-1]
	\arrow[no head, from=2-2, to=1-3]
	\arrow[""{name=3, anchor=center, inner sep=0}, from=2-2, to=3-2]
	\arrow[""{name=4, anchor=center, inner sep=0}, from=2-3, to=3-3]
	\arrow[from=3-1, to=3-2]
	\arrow[no head, from=3-2, to=2-3]
	\arrow[""{name=5, anchor=center, inner sep=0}, from=3-2, to=3-3]
	\arrow[shorten <=4pt, shorten >=4pt, no head, from=1, to=0]
	\arrow[shorten <=7pt, shorten >=7pt, no head, from=3, to=2]
	\arrow[shorten <=4pt, shorten >=4pt, no head, from=5, to=4]
\end{tikzcd}
}\, .\)

We may for instance consider the following $\Sigma$-path from (1) to (2) of \eqref{eq:(1)(2)(3)}.
\begin{equation}\label{eq:eq-132}
\begin{tikzcd}
	{(1)} & {(3)} & {(2)} \rlap{\,.}
	\arrow["d", squiggly, from=1-1, to=1-2]
	\arrow["u", squiggly, from=1-2, to=1-3]
\end{tikzcd}
\end{equation}

\subsubsection{\texorpdfstring{$\Omega$}{Ω} 2-cells} Associated to each $\Sigma$-step from a $\Sigma$-scheme $S_1$ with right border $(l_1,m_1)$
 to another $\Sg$-scheme $S_2$  with right border $(l_2,m_2)$, we have a basic $\Omega$ 2-cell from the cospan  $(l_1,m_1)$  to the cospan $(l_2,m_2)$, which we now describe.

Assume that the $\Sg$-step replaces a $\Sg$-square $R_1$ with right border $(k_1,n_1)$  with a $\Sg$-square $R_2$ with right border $(k_2,n_2)$ such that $l_i=k_il_o$ and $m_i=n_im_0$, $i=1,2$,  as illustrated below.
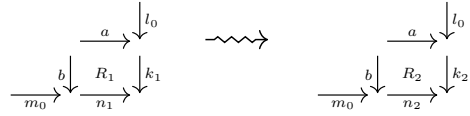
\[
\adjustbox{scale=0.80}{\begin{tikzcd}
	&& {} \\
	& {} & {} \\
	{} & {} & {}
	\arrow["{{l_0}}", from=1-3, to=2-3]
	\arrow["a", from=2-2, to=2-3]
	\arrow[""{name=0, anchor=center, inner sep=0}, "b"', from=2-2, to=3-2]
	\arrow[""{name=1, anchor=center, inner sep=0}, "{{k_1}}", from=2-3, to=3-3]
	\arrow["{{m_0}}"', from=3-1, to=3-2]
	\arrow["{{n_1}}"', from=3-2, to=3-3]
	\arrow["{{R_1}}"{description}, draw=none, from=0, to=1]
\end{tikzcd}}
\hspace{2mm}\begin{tikzcd}
	{} & {}
	\arrow[squiggly, from=1-1, to=1-2]
\end{tikzcd}\hspace{1mm}
\adjustbox{scale=0.80}{\begin{tikzcd}
	&& {} \\
	& {} & {} \\
	{} & {} & {}
	\arrow["{{l_0}}", from=1-3, to=2-3]
	\arrow["a", from=2-2, to=2-3]
	\arrow[""{name=0, anchor=center, inner sep=0}, "b"', from=2-2, to=3-2]
	\arrow[""{name=1, anchor=center, inner sep=0}, "{{k_2}}", from=2-3, to=3-3]
	\arrow["{{m_0}}"', from=3-1, to=3-2]
	\arrow["{{n_2}}"', from=3-2, to=3-3]
	\arrow["{{R_2}}"{description}, draw=none, from=0, to=1]
\end{tikzcd}}
\]

Consider the basic $\Omega$ 2-cell from $(k_1,n_1)$ to $(k_2,n_2)$ determined by the passage from $R_1$ to $R_2$. Composing on the left and on the right with $l_0$ and $m_0$, respectively, we obtain a basic $\Omega$ 2-cell from $(l_1,m_1)$ to  $(l_2,m_2)$.

For instance, consider the first $\Sg$-step, of type $d$, in \eqref{eq:eq-132}. First we take the corresponding basic $\Omega$ 2-cell
\[(h',s')\Rightarrow (h''',r'''s'').\]

The part remaining unchanged in the right border of the $\Sg$-scheme is just $g'$. Then we compose with $g'$, obtaining a basic $\Omega$ 2-cell
\[\Omega_1\colon (h'g',s')\Rightarrow (h'''g',r'''s'').\]

Analogously, for the second $\Sg$-step of \eqref{eq:eq-132}, of type $u$, we consider the basic $\Omega$ 2-cell determined by the passage from
\(\; \begin{tikzcd}
	{} & {} \\
	{} & {}
	\arrow["r", from=1-1, to=1-2]
	\arrow[""{name=0, anchor=center, inner sep=0}, "{h''g}"', from=1-1, to=2-1]
	\arrow[""{name=1, anchor=center, inner sep=0}, "{h'''g'}", from=1-2, to=2-2]
	\arrow["{r'''}"', from=2-1, to=2-2]
	\arrow["\SIGMA"{description}, shift left, draw=none, from=0, to=1]
\end{tikzcd}\;\)
 to
\(\; \begin{tikzcd}
	{} & {} \\
	{} & {}
	\arrow["r", from=1-1, to=1-2]
	\arrow[""{name=0, anchor=center, inner sep=0}, "{{h''g}}"', from=1-1, to=2-1]
	\arrow[""{name=1, anchor=center, inner sep=0}, "{g''}", from=1-2, to=2-2]
	\arrow["{{r''}}"', from=2-1, to=2-2]
	\arrow["\SIGMA"{description}, shift left, draw=none, from=0, to=1]
\end{tikzcd}
  \, , \;\)
obtaining
$$(h^{\prime\prime\prime}g', r^{\prime\prime\prime})\Rightarrow (g^{\prime\prime},r^{\prime\prime})\,;$$
then, composing with $s^{\prime\prime}$, we have the resulting basic $\Omega$ 2-cell
$$\Omega_2\colon(h^{\prime\prime\prime}g^{\prime}, r^{\prime\prime\prime}s^{\prime\prime})\Rightarrow (g^{\prime\prime},r^{\prime\prime}s^{\prime\prime})\,.$$

\begin{definition}\label{def:Omega-2-cell}
An \textbf{\em $\boldsymbol\Omega$ 2-cell} is any finite vertical composition of basic $\Omega$ 2-cells determined by $\Sg$-steps between $\Sg$-schemes with the same left border, or the result of whiskering such a 2-cell with 1-cells of the form $(f,1)$ and $(1,r)$, with $r\in \Sg$,  in the obvious way.
\end{definition}
For instance, the vertical composition $\Omega_2\cdot \Omega_1$ of the two basic $\Omega$  2-cells above gives an $\Omega$ 2-cell from $(h'g',s')$ to $(g^{\prime\prime},r^{\prime\prime}s^{\prime\prime})$.

This way, each $\Sg$-path from a $\Sigma$-scheme of right border $(l_1,m_1)$ to a $\Sigma$-scheme of right border $(l_2,m_2)$ induces an $\Omega$ 2-cell from $(l_1,m_1)$ to $(l_2,m_2)$  given by the vertical composition of the basic $\Omega$ 2-cells corresponding to the $\Sg$-steps of the $\Sg$-path.
We say that two $\Sigma$-paths are \textbf{\em equivalent} if they give rise to the same $\Omega$ 2-cell.

In accordance with \cref{def:Omega-2-cell}, sometimes we will also consider $\Omega$ 2-cells of the form $(1,r) \circ \Omega_0\circ (f, 1)\colon (l_1f,m_1r)\Rightarrow (l_2f,m_2r)$, for any appropriately composable morphisms $f$ and $r$, with $r\in \Sg$,  where $\Omega_0\colon (l_1,m_1)\Rightarrow (l_2,m_2)$ is a finite vertical composition of basic $\Omega$ 2-cells determined by a $\Sg$-path between $\Sg$-schemes with the same left border.

\subsubsection{Properties of \texorpdfstring{$\Sigma$}{Σ}-paths}  The $\Sigma$-schemes of level 3 will be of special interest in what follows.

    A $\Sigma$-scheme of level 2 can be looked as a $\Sigma$-scheme of level 3 by extending it with identities to make a third row (see \cref{nota:Sigma-steps}).
Therefore, all that we are going to conclude about $\Sigma$-schemes of level 3 has obvious consequences for $\Sigma$-schemes of level 2.

\begin{remark}\label{nota:simple-Sigma} In what follows, we frequently write
\(\adjustbox{scale=0.70}{\begin{tikzcd}
	{} & {} \\
	{} & {}
	\arrow[from=1-1, to=1-2]
	\arrow[""{name=0, anchor=center, inner sep=0}, from=1-1, to=2-1]
	\arrow[""{name=1, anchor=center, inner sep=0}, from=1-2, to=2-2]
	\arrow[from=2-1, to=2-2]
	\arrow["{\SIGMA}"{description}, shift left, draw=none, from=0, to=1]
\end{tikzcd}}\)
  instead of
   \(\adjustbox{scale=0.70}{\begin{tikzcd}
	{} & {} \\
	{} & {}
	\arrow[from=1-1, to=1-2]
	\arrow[""{name=0, anchor=center, inner sep=0}, from=1-1, to=2-1]
	\arrow[""{name=1, anchor=center, inner sep=0}, from=1-2, to=2-2]
	\arrow[from=2-1, to=2-2]
	\arrow["{\SIGMA^{\alpha}}"{description}, draw=none, from=0, to=1]
\end{tikzcd}}\)
 omitting the name of the invertible 2-cell of the $\Sigma$-square. Sometimes we will also use numbered squares to refer to $\Sigma$-squares, as for instance, in the proof of Lemma \ref{lem:Sigma-steps}.
\end{remark}

\begin{notation}\label{nota:Sigma-steps} Consider the following five types of  $\Sigma$-schemes of level 3 with the same left border, that we identify  by the letters below, namely, $\mathbf{d}$ (down), $\mathbf{u}$ (upper), $\mathbf{s}$ (square), $\mathbf{d}_1$ and $\mathbf{s}_1$.
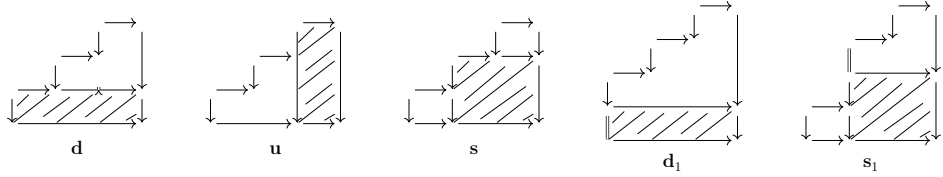
\[
\adjustbox{scale=0.50}{\begin{tikzcd}
	&& {} & {} \\
	& {} & {} \\
	{} & {} & {} & {} \\
	{} & {} & {} & {} \\
	& {} & {}
	\arrow[from=1-3, to=1-4]
	\arrow[from=1-3, to=2-3]
	\arrow[from=1-4, to=3-4]
	\arrow[from=2-2, to=2-3]
	\arrow[from=2-2, to=3-2]
	\arrow[""{name=0, anchor=center, inner sep=0}, from=3-1, to=3-2]
	\arrow[""{name=1, anchor=center, inner sep=0}, from=3-1, to=4-1]
	\arrow[""{name=2, anchor=center, inner sep=0}, from=3-2, to=3-4]
	\arrow[""{name=3, anchor=center, inner sep=0}, draw=none, from=3-3, to=3-4]
	\arrow[""{name=4, anchor=center, inner sep=0}, from=3-4, to=4-4]
	\arrow[shift left=3, draw=none, from=3-4, to=4-4]
	\arrow[no head, from=4-1, to=3-2]
	\arrow[""{name=5, anchor=center, inner sep=0}, draw=none, from=4-1, to=4-2]
	\arrow[from=4-1, to=4-4]
	\arrow[no head, from=4-2, to=3-3]
	\arrow[""{name=6, anchor=center, inner sep=0}, draw=none, from=4-2, to=4-3]
	\arrow[no head, from=4-3, to=3-4]
	\arrow[""{name=7, anchor=center, inner sep=0}, draw=none, from=4-3, to=4-4]
	\arrow["{\text{\Large $\mathbf{d}$}}", draw=none, from=5-2, to=5-3]
	\arrow[shorten <=4pt, shorten >=4pt, no head, from=1, to=0]
	\arrow[""{name=8, anchor=center, inner sep=0}, draw=none, from=3-2, to=2]
	\arrow[shorten >=0pt, Rightarrow, from=3-3, to=2]
	\arrow[shorten <=7pt, shorten >=7pt, no head, from=6, to=3]
	\arrow[shorten <=4pt, shorten >=4pt, no head, from=7, to=4]
	\arrow[shorten <=8pt, shorten >=8pt, no head, from=5, to=8]
\end{tikzcd}}
\qquad
\adjustbox{scale=0.50}{\begin{tikzcd}
	&& {} & {} \\
	& {} & {} & {} \\
	{} & {} & {} & {} \\
	{} & {} & {} & {} \\
	& {} & {}
	\arrow[""{name=0, anchor=center, inner sep=0}, from=1-3, to=1-4]
	\arrow[""{name=1, anchor=center, inner sep=0}, draw=none, from=1-3, to=2-3]
	\arrow[from=1-3, to=4-3]
	\arrow[""{name=2, anchor=center, inner sep=0}, draw=none, from=1-4, to=2-4]
	\arrow[from=1-4, to=4-4]
	\arrow[from=2-2, to=2-3]
	\arrow[from=2-2, to=3-2]
	\arrow[no head, from=2-3, to=1-4]
	\arrow[""{name=3, anchor=center, inner sep=0}, draw=none, from=2-3, to=3-3]
	\arrow[""{name=4, anchor=center, inner sep=0}, draw=none, from=2-4, to=3-4]
	\arrow[from=3-1, to=3-2]
	\arrow[from=3-1, to=4-1]
	\arrow[no head, from=3-3, to=2-4]
	\arrow[draw=none, from=3-3, to=3-4]
	\arrow[""{name=5, anchor=center, inner sep=0}, draw=none, from=3-3, to=4-3]
	\arrow[""{name=6, anchor=center, inner sep=0}, draw=none, from=3-4, to=4-4]
	\arrow[shift left=3, draw=none, from=3-4, to=4-4]
	\arrow[from=4-1, to=4-3]
	\arrow[draw=none, from=4-2, to=4-3]
	\arrow[no head, from=4-3, to=3-4]
	\arrow[""{name=7, anchor=center, inner sep=0}, from=4-3, to=4-4]
	\arrow["{\text{\Large $\mathbf{u}$}}", draw=none, from=5-2, to=5-3]
	\arrow[shorten >=4pt, no head, from=1, to=0]
	\arrow[shorten <=7pt, shorten >=7pt, no head, from=3, to=2]
	\arrow[shorten <=7pt, shorten >=7pt, no head, from=5, to=4]
	\arrow[shorten <=4pt, shorten >=4pt, no head, from=7, to=6]
\end{tikzcd}}
\qquad
\adjustbox{scale=0.50}{\begin{tikzcd}
	&& {} & {} \\
	& {} & {} & {} \\
	{} & {} & {} & {} \\
	{} & {} & {} & {} \\
	& {} & {}
	\arrow[from=1-3, to=1-4]
	\arrow[from=1-3, to=2-3]
	\arrow[from=1-4, to=2-4]
	\arrow[""{name=0, anchor=center, inner sep=0}, from=2-2, to=2-3]
	\arrow[""{name=1, anchor=center, inner sep=0}, from=2-2, to=3-2]
	\arrow[""{name=2, anchor=center, inner sep=0}, from=2-3, to=2-4]
	\arrow[""{name=3, anchor=center, inner sep=0}, draw=none, from=2-4, to=3-4]
	\arrow[from=2-4, to=4-4]
	\arrow[from=3-1, to=3-2]
	\arrow[from=3-1, to=4-1]
	\arrow[no head, from=3-2, to=2-3]
	\arrow[""{name=4, anchor=center, inner sep=0}, from=3-2, to=4-2]
	\arrow[draw=none, from=3-3, to=3-4]
	\arrow[draw=none, from=3-3, to=4-3]
	\arrow[""{name=5, anchor=center, inner sep=0}, draw=none, from=3-4, to=4-4]
	\arrow[shift left=3, draw=none, from=3-4, to=4-4]
	\arrow[from=4-1, to=4-2]
	\arrow[no head, from=4-2, to=2-4]
	\arrow[""{name=6, anchor=center, inner sep=0}, draw=none, from=4-2, to=4-3]
	\arrow[from=4-2, to=4-4]
	\arrow[no head, from=4-3, to=3-4]
	\arrow[""{name=7, anchor=center, inner sep=0}, draw=none, from=4-3, to=4-4]
	\arrow["{\text{\Large $\mathbf{s}$}}", draw=none, from=5-2, to=5-3]
	\arrow[shorten <=4pt, shorten >=4pt, no head, from=1, to=0]
	\arrow[shorten <=11pt, shorten >=11pt, no head, from=4, to=2]
	\arrow[shorten <=5pt, shorten >=11pt, no head, from=6, to=3]
	\arrow[shorten <=4pt, shorten >=4pt, no head, from=7, to=5]
\end{tikzcd}}
\qquad
\adjustbox{scale=0.50}{\begin{tikzcd}
	&& {} & {} \\
	& {} & {} & {} \\
	{} & {} & {} & {} \\
	{} & {} & {} & {} \\
	{} & {} & {} & {} \\
	& {} & {}
	\arrow[from=1-3, to=1-4]
	\arrow[from=1-3, to=2-3]
	\arrow[from=1-4, to=4-4]
	\arrow[from=2-2, to=2-3]
	\arrow[from=2-2, to=3-2]
	\arrow[draw=none, from=2-4, to=3-4]
	\arrow[from=3-1, to=3-2]
	\arrow[from=3-1, to=4-1]
	\arrow[draw=none, from=3-3, to=3-4]
	\arrow[draw=none, from=3-3, to=4-3]
	\arrow[draw=none, from=3-4, to=4-4]
	\arrow[shift left=3, draw=none, from=3-4, to=4-4]
	\arrow[""{name=0, anchor=center, inner sep=0}, draw=none, from=4-1, to=4-2]
	\arrow[from=4-1, to=4-4]
	\arrow[""{name=1, anchor=center, inner sep=0}, equals, from=4-1, to=5-1]
	\arrow[""{name=2, anchor=center, inner sep=0}, draw=none, from=4-2, to=4-3]
	\arrow[""{name=3, anchor=center, inner sep=0}, draw=none, from=4-3, to=4-4]
	\arrow[from=4-4, to=5-4]
	\arrow[no head, from=5-1, to=4-2]
	\arrow[""{name=4, anchor=center, inner sep=0}, draw=none, from=5-1, to=5-2]
	\arrow[from=5-1, to=5-4]
	\arrow[no head, from=5-2, to=4-3]
	\arrow[""{name=5, anchor=center, inner sep=0}, draw=none, from=5-2, to=5-3]
	\arrow[no head, from=5-3, to=4-4]
	\arrow[draw=none, from=5-3, to=5-4]
	\arrow["{\text{\Large $\mathbf{d}_1$}}", draw=none, from=6-2, to=6-3]
	\arrow[shorten <=4pt, shorten >=4pt, no head, from=1, to=0]
	\arrow[shorten <=7pt, shorten >=7pt, no head, from=4, to=2]
	\arrow[shorten <=7pt, shorten >=7pt, no head, from=5, to=3]
\end{tikzcd}}
\qquad
\adjustbox{scale=0.50}{\begin{tikzcd}
	&& {} & {} \\
	& {} & {} & {} \\
	& {} & {} & {} \\
	{} & {} && {} \\
	{} & {} & {} & {} \\
	& {} & {}
	\arrow[from=1-3, to=1-4]
	\arrow[from=1-3, to=2-3]
	\arrow[from=1-4, to=3-4]
	\arrow[from=2-2, to=2-3]
	\arrow[equals, from=2-2, to=3-2]
	\arrow[draw=none, from=2-4, to=3-4]
	\arrow[""{name=0, anchor=center, inner sep=0}, draw=none, from=3-2, to=3-3]
	\arrow[from=3-2, to=3-4]
	\arrow[""{name=1, anchor=center, inner sep=0}, from=3-2, to=4-2]
	\arrow[""{name=2, anchor=center, inner sep=0}, draw=none, from=3-3, to=3-4]
	\arrow[""{name=3, anchor=center, inner sep=0}, draw=none, from=3-4, to=4-4]
	\arrow[from=3-4, to=5-4]
	\arrow[from=4-1, to=4-2]
	\arrow[from=4-1, to=5-1]
	\arrow[no head, from=4-2, to=3-3]
	\arrow[""{name=4, anchor=center, inner sep=0}, from=4-2, to=5-2]
	\arrow[""{name=5, anchor=center, inner sep=0}, draw=none, from=4-4, to=5-4]
	\arrow[from=5-1, to=5-2]
	\arrow[no head, from=5-2, to=3-4]
	\arrow[""{name=6, anchor=center, inner sep=0}, draw=none, from=5-2, to=5-3]
	\arrow[from=5-2, to=5-4]
	\arrow[no head, from=5-3, to=4-4]
	\arrow[""{name=7, anchor=center, inner sep=0}, draw=none, from=5-3, to=5-4]
	\arrow["{\text{\Large $\mathbf{s}_1$}}", draw=none, from=6-2, to=6-3]
	\arrow[shorten <=4pt, shorten >=4pt, no head, from=1, to=0]
	\arrow[shorten <=11pt, shorten >=11pt, no head, from=4, to=2]
	\arrow[shorten <=11pt, shorten >=11pt, no head, from=6, to=3]
	\arrow[shorten <=4pt, shorten >=4pt, no head, from=7, to=5]
\end{tikzcd}}
\]

Between two $\Sg$-schemes of type $\bd$, which agree on the non-shaded part and on the left border of the shaded part, we may perform a $\Sg$-step by replacing just the shaded part. We then say that this $\Sg$-step is also of  type $\bd$. Analogously, we use the same terminology  for the remaining four cases.

 For $\Sigma$-steps between $\Sigma$-schemes of level 2, apart from the letters $d$ and $u$, already mentioned, we use also the letter $d_1$ for a type of the $\Sg$-step. To summarise, we use the following types.
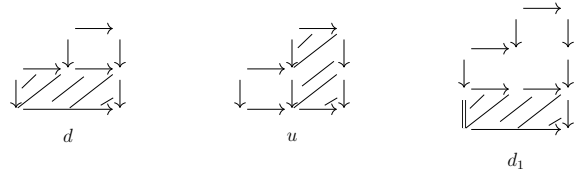
\[
\adjustbox{scale=0.60}{\begin{tikzcd}
	& {} & {} \\
	{} & {} & {} \\
	{} & {} & {}
	\arrow[from=1-2, to=1-3]
	\arrow[from=1-2, to=2-2]
	\arrow[from=1-3, to=2-3]
	\arrow[""{name=0, anchor=center, inner sep=0}, from=2-1, to=2-2]
	\arrow[""{name=1, anchor=center, inner sep=0}, from=2-1, to=3-1]
	\arrow[""{name=2, anchor=center, inner sep=0}, from=2-2, to=2-3]
	\arrow[""{name=3, anchor=center, inner sep=0}, from=2-3, to=3-3]
	\arrow[no head, from=3-1, to=2-2]
	\arrow[""{name=4, anchor=center, inner sep=0}, draw=none, from=3-1, to=3-2]
	\arrow[from=3-1, to=3-3]
	\arrow["\text{\large $d$}"', shift right=4, draw=none, from=3-1, to=3-3]
	\arrow[no head, from=3-2, to=2-3]
	\arrow[""{name=5, anchor=center, inner sep=0}, draw=none, from=3-2, to=3-3]
	\arrow[shorten <=4pt, shorten >=4pt, no head, from=1, to=0]
	\arrow[shorten <=7pt, shorten >=7pt, no head, from=4, to=2]
	\arrow[shorten <=4pt, shorten >=4pt, no head, from=5, to=3]
\end{tikzcd}}
\qquad \qquad
\adjustbox{scale=0.60}{\begin{tikzcd}
	& {} & {} \\
	{} & {} & {} \\
	{} & {} & {}
	\arrow[""{name=0, anchor=center, inner sep=0}, from=1-2, to=1-3]
	\arrow[""{name=1, anchor=center, inner sep=0}, from=1-2, to=2-2]
	\arrow[""{name=2, anchor=center, inner sep=0}, from=1-3, to=2-3]
	\arrow[from=2-1, to=2-2]
	\arrow[from=2-1, to=3-1]
	\arrow[no head, from=2-2, to=1-3]
	\arrow[""{name=3, anchor=center, inner sep=0}, from=2-2, to=3-2]
	\arrow[""{name=4, anchor=center, inner sep=0}, from=2-3, to=3-3]
	\arrow[from=3-1, to=3-2]
	\arrow["\text{\large $u$}"', shift right=5, draw=none, from=3-1, to=3-3]
	\arrow[no head, from=3-2, to=2-3]
	\arrow[""{name=5, anchor=center, inner sep=0}, from=3-2, to=3-3]
	\arrow[shorten <=4pt, shorten >=4pt, no head, from=1, to=0]
	\arrow[shorten <=7pt, shorten >=7pt, no head, from=3, to=2]
	\arrow[shorten <=4pt, shorten >=4pt, no head, from=5, to=4]
\end{tikzcd}}
\qquad \qquad
\adjustbox{scale=0.60}{\begin{tikzcd}
	& {} & {} \\
	{} & {} & {} \\
	{} & {} & {} \\
	{} & {} & {}
	\arrow[from=1-2, to=1-3]
	\arrow[from=1-2, to=2-2]
	\arrow[from=1-3, to=2-3]
	\arrow[from=2-1, to=2-2]
	\arrow[from=2-1, to=3-1]
	\arrow[from=2-3, to=3-3]
	\arrow[""{name=0, anchor=center, inner sep=0}, from=3-1, to=3-2]
	\arrow[""{name=1, anchor=center, inner sep=0}, equals, from=3-1, to=4-1]
	\arrow[""{name=2, anchor=center, inner sep=0}, from=3-2, to=3-3]
	\arrow[""{name=3, anchor=center, inner sep=0}, from=3-3, to=4-3]
	\arrow[no head, from=4-1, to=3-2]
	\arrow[""{name=4, anchor=center, inner sep=0}, draw=none, from=4-1, to=4-2]
	\arrow[from=4-1, to=4-3]
	\arrow["\text{\large $d_1$}"', shift right=5, draw=none, from=4-1, to=4-3]
	\arrow[no head, from=4-2, to=3-3]
	\arrow[""{name=5, anchor=center, inner sep=0}, draw=none, from=4-2, to=4-3]
	\arrow[shorten <=4pt, shorten >=4pt, no head, from=1, to=0]
	\arrow[shorten <=7pt, shorten >=7pt, no head, from=4, to=2]
	\arrow[shorten <=4pt, shorten >=4pt, no head, from=5, to=3]
\end{tikzcd}}
\]
By adding identity squares
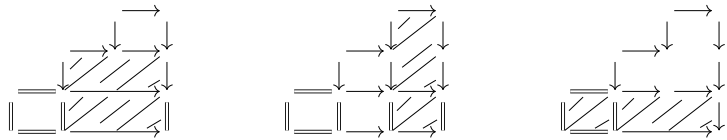
\[
\adjustbox{scale=0.60}{\begin{tikzcd}
	&& {} & {} \\
	& {} & {} & {} \\
	{} & {} & {} & {} \\
	{} & {} & {} & {}
	\arrow[from=1-3, to=1-4]
	\arrow[from=1-3, to=2-3]
	\arrow[from=1-4, to=2-4]
	\arrow[""{name=0, anchor=center, inner sep=0}, from=2-2, to=2-3]
	\arrow[""{name=1, anchor=center, inner sep=0}, from=2-2, to=3-2]
	\arrow[""{name=2, anchor=center, inner sep=0}, from=2-3, to=2-4]
	\arrow[""{name=3, anchor=center, inner sep=0}, from=2-4, to=3-4]
	\arrow[equals, from=3-1, to=3-2]
	\arrow[equals, from=3-1, to=4-1]
	\arrow[no head, from=3-2, to=2-3]
	\arrow[""{name=4, anchor=center, inner sep=0}, draw=none, from=3-2, to=3-3]
	\arrow[from=3-2, to=3-4]
	\arrow[""{name=5, anchor=center, inner sep=0}, equals, from=3-2, to=4-2]
	\arrow[no head, from=3-3, to=2-4]
	\arrow[""{name=6, anchor=center, inner sep=0}, draw=none, from=3-3, to=3-4]
	\arrow[no head, from=3-4, to=4-3]
	\arrow[""{name=7, anchor=center, inner sep=0}, equals, from=3-4, to=4-4]
	\arrow[equals, from=4-1, to=4-2]
	\arrow[no head, from=4-2, to=3-3]
	\arrow[""{name=8, anchor=center, inner sep=0}, draw=none, from=4-2, to=4-3]
	\arrow[from=4-2, to=4-4]
	\arrow[""{name=9, anchor=center, inner sep=0}, draw=none, from=4-3, to=4-4]
	\arrow[between={0.25}{0.75}, no head, from=1, to=0]
	\arrow[between={0.2}{0.8}, no head, from=4, to=2]
	\arrow[between={0.2}{0.8}, no head, from=5, to=4]
	\arrow[between={0.25}{0.75}, no head, from=6, to=3]
	\arrow[between={0.2}{0.8}, no head, from=7, to=9]
	\arrow[between={0.2}{0.8}, no head, from=8, to=6]
\end{tikzcd}}
\qquad \qquad
\adjustbox{scale=0.60}{\begin{tikzcd}
	&& {} & {} \\
	& {} & {} & {} \\
	{} & {} & {} & {} \\
	{} & {} & {} & {}
	\arrow[""{name=0, anchor=center, inner sep=0}, from=1-3, to=1-4]
	\arrow[""{name=1, anchor=center, inner sep=0}, from=1-3, to=2-3]
	\arrow[""{name=2, anchor=center, inner sep=0}, from=1-4, to=2-4]
	\arrow[from=2-2, to=2-3]
	\arrow[from=2-2, to=3-2]
	\arrow[no head, from=2-3, to=1-4]
	\arrow[""{name=3, anchor=center, inner sep=0}, from=2-3, to=3-3]
	\arrow[""{name=4, anchor=center, inner sep=0}, from=2-4, to=3-4]
	\arrow[equals, from=3-1, to=3-2]
	\arrow[equals, from=3-1, to=4-1]
	\arrow[from=3-2, to=3-3]
	\arrow[equals, from=3-2, to=4-2]
	\arrow[no head, from=3-3, to=2-4]
	\arrow[""{name=5, anchor=center, inner sep=0}, from=3-3, to=3-4]
	\arrow[""{name=6, anchor=center, inner sep=0}, equals, from=3-3, to=4-3]
	\arrow[no head, from=3-4, to=4-3]
	\arrow[""{name=7, anchor=center, inner sep=0}, equals, from=3-4, to=4-4]
	\arrow[equals, from=4-1, to=4-2]
	\arrow[from=4-2, to=4-3]
	\arrow[""{name=8, anchor=center, inner sep=0}, from=4-3, to=4-4]
	\arrow[between={0.25}{0.75}, no head, from=1, to=0]
	\arrow[between={0.2}{0.8}, no head, from=3, to=2]
	\arrow[between={0.25}{0.75}, no head, from=5, to=4]
	\arrow[between={0.2}{0.8}, no head, from=5, to=6]
	\arrow[between={0.2}{0.8}, no head, from=7, to=8]
\end{tikzcd}}
\qquad \qquad
\adjustbox{scale=0.60}{\begin{tikzcd}
	&& {} & {} \\
	& {} & {} & {} \\
	{} & {} & {} & {} \\
	{} & {} & {} & {}
	\arrow[from=1-3, to=1-4]
	\arrow[from=1-3, to=2-3]
	\arrow[from=1-4, to=2-4]
	\arrow[from=2-2, to=2-3]
	\arrow[from=2-2, to=3-2]
	\arrow[from=2-4, to=3-4]
	\arrow[""{name=0, anchor=center, inner sep=0}, equals, from=3-1, to=3-2]
	\arrow[""{name=1, anchor=center, inner sep=0}, equals, from=3-1, to=4-1]
	\arrow[""{name=2, anchor=center, inner sep=0}, from=3-2, to=3-3]
	\arrow[no head, from=3-2, to=4-1]
	\arrow[""{name=3, anchor=center, inner sep=0}, equals, from=3-2, to=4-2]
	\arrow[""{name=4, anchor=center, inner sep=0}, from=3-3, to=3-4]
	\arrow[""{name=5, anchor=center, inner sep=0}, from=3-4, to=4-4]
	\arrow[""{name=6, anchor=center, inner sep=0}, equals, from=4-1, to=4-2]
	\arrow[no head, from=4-2, to=3-3]
	\arrow[""{name=7, anchor=center, inner sep=0}, draw=none, from=4-2, to=4-3]
	\arrow[from=4-2, to=4-4]
	\arrow[no head, from=4-3, to=3-4]
	\arrow[""{name=8, anchor=center, inner sep=0}, draw=none, from=4-3, to=4-4]
	\arrow[between={0.2}{0.8}, no head, from=0, to=1]
	\arrow[between={0.25}{0.75}, no head, from=3, to=2]
	\arrow[between={0.2}{0.8}, no head, from=3, to=6]
	\arrow[between={0.2}{0.8}, no head, from=7, to=4]
	\arrow[between={0.25}{0.75}, no head, from=8, to=5]
\end{tikzcd}}
\]
we see that type $d$ is a special case of type $\bs$, type $u$ is a special case of type $\bu$ and type $d_1$ is a special case of type $\bd$.
\end{notation}

\begin{assumption}\label{assumption-path} From now on, we will assume that every $\Sg$-path between $\Sigma$-schemes of level 3 contains only $\Sg$-steps of the five types described in \cref{nota:Sigma-steps}. Analogously, for $\Sg$-schemes of level 2 we assume that the $\Sg$-paths are made only of $\Sg$-steps of the three types indicated above.
\end{assumption} %

\begin{lemma}\label{lem:Sigma-steps} We have the following properties for $\Sigma$-paths between $\Sigma$-schemes of level 3:
\begin{enumerate}[label=(\arabic*)]
	\item A $\Sigma$-path of two $\Sigma$-steps of the same type is equivalent to the $\Sigma$-path consisting of a single $\Sigma$-step of that type.
    \item A $\Sigma$-path consisting of two $\Sigma$-steps, one of type $\mathbf{s}$, the other of type $\mathbf{s}_1$, is equivalent to a $\Sigma$-step of type $\mathbf{s}$.
    \item  A $\Sigma$-path consisting of two $\Sigma$-steps, one of type $\mathbf{d}$, the other of type $\mathbf{d}_1$, is equivalent to a $\Sigma$-step of type $\mathbf{d}$.
	\item Any two possible $\Sigma$-steps between two given $\Sigma$-schemes are equivalent.  The basic $\Omega$ 2-cell corresponding to a $\Sg$-step of any type from a $\Sg$-scheme to itself is the identity 2-cell.
\end{enumerate}
\end{lemma}

\begin{proof} (1) This is clear from Lemma \ref{lem:basic-Omega}.

(2) Given a $\Sigma$-path $S_1\begin{tikzcd}
	{} & {}
	\arrow["\mathbf{s}", squiggly, from=1-1, to=1-2]
\end{tikzcd}S_3\begin{tikzcd}
	{} & {}
	\arrow["\mathbf{s}_1", squiggly, from=1-1, to=1-2]
\end{tikzcd}S_2$, we necessarily have that $S_1$ is of the form $\mathbf{s}$, $S_2$ is of the form $\mathbf{s}_1$, and $S_3$ is of both forms; that is, the $\Sigma$-path is of the following form, where we use numbers to indicate the various $\Sg$-squares involved.
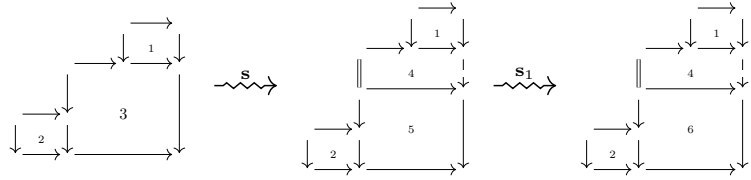
\[
\adjustbox{scale=0.60}{\begin{tikzcd}
	&& {} & {} \\
	& {} & {} & {} \\
	{} & {} & 3 & {} \\
	{} & {} & {} & {}
	\arrow[from=1-3, to=1-4]
	\arrow[""{name=0, anchor=center, inner sep=0}, from=1-3, to=2-3]
	\arrow[""{name=1, anchor=center, inner sep=0}, from=1-4, to=2-4]
	\arrow[from=2-2, to=2-3]
	\arrow[from=2-2, to=3-2]
	\arrow[from=2-3, to=2-4]
	\arrow[draw=none, from=2-4, to=3-4]
	\arrow[from=2-4, to=4-4]
	\arrow[from=3-1, to=3-2]
	\arrow[""{name=2, anchor=center, inner sep=0}, from=3-1, to=4-1]
	\arrow[""{name=3, anchor=center, inner sep=0}, from=3-2, to=4-2]
	\arrow[draw=none, from=3-3, to=3-4]
	\arrow[draw=none, from=3-3, to=4-3]
	\arrow[draw=none, from=3-4, to=4-4]
	\arrow[shift left=3, draw=none, from=3-4, to=4-4]
	\arrow[from=4-1, to=4-2]
	\arrow[draw=none, from=4-2, to=4-3]
	\arrow[from=4-2, to=4-4]
	\arrow[draw=none, from=4-3, to=4-4]
	\arrow["1"{description}, draw=none, from=0, to=1]
	\arrow["2"{description}, draw=none, from=2, to=3]
\end{tikzcd}}
\begin{tikzcd}
	{} & {}
	\arrow["\mathbf{s}", squiggly, from=1-1, to=1-2]
\end{tikzcd}
\adjustbox{scale=0.60}{\begin{tikzcd}
	&& {} & {} \\
	& {} & {} & {} \\
	& {} & {} & {} \\
	{} & {} && {} \\
	{} & {} & {} & {}
	\arrow[from=1-3, to=1-4]
	\arrow[""{name=0, anchor=center, inner sep=0}, from=1-3, to=2-3]
	\arrow[""{name=1, anchor=center, inner sep=0}, from=1-4, to=2-4]
	\arrow[from=2-2, to=2-3]
	\arrow[""{name=2, anchor=center, inner sep=0}, equals, from=2-2, to=3-2]
	\arrow[from=2-3, to=2-4]
	\arrow[""{name=3, anchor=center, inner sep=0}, ""{description}, from=2-4, to=3-4]
	\arrow[draw=none, from=3-2, to=3-3]
	\arrow[from=3-2, to=3-4]
	\arrow[from=3-2, to=4-2]
	\arrow[draw=none, from=3-3, to=3-4]
	\arrow[draw=none, from=3-4, to=4-4]
	\arrow[from=3-4, to=5-4]
	\arrow[from=4-1, to=4-2]
	\arrow[""{name=4, anchor=center, inner sep=0}, from=4-1, to=5-1]
	\arrow["5"{description}, draw=none, from=4-2, to=4-4]
	\arrow[""{name=5, anchor=center, inner sep=0}, from=4-2, to=5-2]
	\arrow[draw=none, from=4-4, to=5-4]
	\arrow[from=5-1, to=5-2]
	\arrow[draw=none, from=5-2, to=5-3]
	\arrow[from=5-2, to=5-4]
	\arrow[draw=none, from=5-3, to=5-4]
	\arrow["1"{description}, draw=none, from=0, to=1]
	\arrow["4"{description}, draw=none, from=2, to=3]
	\arrow["2"{description}, draw=none, from=4, to=5]
\end{tikzcd}}
\begin{tikzcd}
	{} & {}
	\arrow["\mathbf{s}_1", squiggly, from=1-1, to=1-2]
\end{tikzcd}
\adjustbox{scale=0.60}{\begin{tikzcd}
	&& {} & {} \\
	& {} & {} & {} \\
	& {} & {} & {} \\
	{} & {} && {} \\
	{} & {} & {} & {}
	\arrow[from=1-3, to=1-4]
	\arrow[""{name=0, anchor=center, inner sep=0}, from=1-3, to=2-3]
	\arrow[""{name=1, anchor=center, inner sep=0}, from=1-4, to=2-4]
	\arrow[from=2-2, to=2-3]
	\arrow[""{name=2, anchor=center, inner sep=0}, equals, from=2-2, to=3-2]
	\arrow[from=2-3, to=2-4]
	\arrow[""{name=3, anchor=center, inner sep=0}, ""{description}, from=2-4, to=3-4]
	\arrow[draw=none, from=3-2, to=3-3]
	\arrow[from=3-2, to=3-4]
	\arrow[from=3-2, to=4-2]
	\arrow[draw=none, from=3-3, to=3-4]
	\arrow[draw=none, from=3-4, to=4-4]
	\arrow[from=3-4, to=5-4]
	\arrow[from=4-1, to=4-2]
	\arrow[""{name=4, anchor=center, inner sep=0}, from=4-1, to=5-1]
	\arrow["6"{description}, draw=none, from=4-2, to=4-4]
	\arrow[""{name=5, anchor=center, inner sep=0}, from=4-2, to=5-2]
	\arrow[draw=none, from=4-4, to=5-4]
	\arrow[from=5-1, to=5-2]
	\arrow[draw=none, from=5-2, to=5-3]
	\arrow[from=5-2, to=5-4]
	\arrow[draw=none, from=5-3, to=5-4]
	\arrow["1"{description}, draw=none, from=0, to=1]
	\arrow["4"{description}, draw=none, from=2, to=3]
	\arrow["2"{description}, draw=none, from=4, to=5]
\end{tikzcd}}
\]

But the basic $\Omega$ 2-cell corresponding to ${\bs}_1$ obtained by the application of Rule 4'(\cref{pro:useful_rules}) to the squares
$\adjustbox{scale=0.50}{\begin{tikzcd}
	{} & {} & {} \\
	{} & {} & {} \\
	{} && {} \\
	{} & {} & {}
	\arrow[from=1-1, to=1-2]
	\arrow[""{name=0, anchor=center, inner sep=0}, equals, from=1-1, to=2-1]
	\arrow[from=1-2, to=1-3]
	\arrow[""{name=1, anchor=center, inner sep=0}, from=1-3, to=2-3]
	\arrow[draw=none, from=2-1, to=2-2]
	\arrow[from=2-1, to=2-3]
	\arrow[from=2-1, to=3-1]
	\arrow[draw=none, from=2-2, to=2-3]
	\arrow[draw=none, from=2-3, to=3-3]
	\arrow[from=2-3, to=4-3]
	\arrow["5"{description}, draw=none, from=3-1, to=3-3]
	\arrow[from=3-1, to=4-1]
	\arrow[draw=none, from=3-3, to=4-3]
	\arrow[draw=none, from=4-1, to=4-2]
	\arrow[from=4-1, to=4-3]
	\arrow[draw=none, from=4-2, to=4-3]
	\arrow["4"{description}, draw=none, from=0, to=1]
\end{tikzcd}}$
and
$\adjustbox{scale=0.50}{\begin{tikzcd}
	{} & {} & {} \\
	{} & {} & {} \\
	{} && {} \\
	{} & {} & {}
	\arrow[from=1-1, to=1-2]
	\arrow[""{name=0, anchor=center, inner sep=0}, equals, from=1-1, to=2-1]
	\arrow[from=1-2, to=1-3]
	\arrow[""{name=1, anchor=center, inner sep=0}, from=1-3, to=2-3]
	\arrow[draw=none, from=2-1, to=2-2]
	\arrow[from=2-1, to=2-3]
	\arrow[from=2-1, to=3-1]
	\arrow[draw=none, from=2-2, to=2-3]
	\arrow[draw=none, from=2-3, to=3-3]
	\arrow[from=2-3, to=4-3]
	\arrow["6"{description}, draw=none, from=3-1, to=3-3]
	\arrow[from=3-1, to=4-1]
	\arrow[draw=none, from=3-3, to=4-3]
	\arrow[draw=none, from=4-1, to=4-2]
	\arrow[from=4-1, to=4-3]
	\arrow[draw=none, from=4-2, to=4-3]
	\arrow["4"{description}, draw=none, from=0, to=1]
\end{tikzcd}}$
 is the same as the one obtained by applying Rule 4' to the squares 5 and 6 and then composing with square number 4 afterwards. Consequently, under the present circumstances, $\mathbf{s}_1$ is equivalent to $\mathbf{s}$, and the entire $\Sigma$-path is equivalent to the $\Sigma$-step of type $\mathbf{s}$.

 (3) Given a $\Sg$-path of the form
\(\begin{tikzcd}
	{S_1} & {S_3} & {S_2}\,,
	\arrow["\mathbf{d}", squiggly, from=1-1, to=1-2]
	\arrow["\mathbf{d}_1", squiggly, from=1-2, to=1-3]
\end{tikzcd}\)
the $\Sg$-scheme $S_3$ has to be simultaneously of type $\bd$ and ${\bd}_1$. Thus, it has the form
\begin{equation}\label{eq:eqSa}
\adjustbox{scale=0.60}{\begin{tikzcd}
	&& {} & {} \\
	& {} & {} \\
	{} & {} && {} \\
	{} &&& {} \\
	{} &&& {}
	\arrow[from=1-3, to=1-4]
	\arrow[from=1-3, to=2-3]
	\arrow[from=1-4, to=3-4]
	\arrow[from=2-2, to=2-3]
	\arrow[from=2-2, to=3-2]
	\arrow[from=3-1, to=3-2]
	\arrow[from=3-1, to=4-1]
	\arrow[from=3-2, to=3-4]
	\arrow[from=3-4, to=4-4]
	\arrow[from=4-1, to=4-4]
	\arrow[equals, from=4-1, to=5-1]
	\arrow[from=4-4, to=5-4]
	\arrow[from=5-1, to=5-4]
\end{tikzcd}}.
\end{equation}

Arguing similarly to (2), we conclude that the second $\Sg$-step is also of type $\bd$, thus the entire $\Sg$-path is equivalent to just a $\Sg$-step of type $\bd$.

 (4) Consider, for instance, two $\Sigma$-steps between two $\Sigma$-schemes, where one is of type $\mathbf{d}$ and the other is of type $\mathbf{u}$. Then the two $\Sigma$-schemes must be simultaneously of type $\bd$ and $\bu$, say
 \[\adjustbox{scale=0.60}
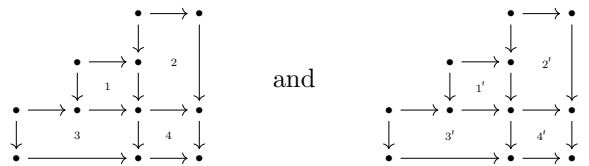
{\begin{tikzcd}
	&& \bullet & \bullet \\
	& \bullet & \bullet \\
	\bullet & \bullet & \bullet & \bullet \\
	\bullet && \bullet & \bullet
	\arrow[from=1-3, to=1-4]
	\arrow[from=1-3, to=2-3]
	\arrow[""{name=0, anchor=center, inner sep=0}, from=1-4, to=3-4]
	\arrow[from=2-2, to=2-3]
	\arrow[""{name=1, anchor=center, inner sep=0}, from=2-2, to=3-2]
	\arrow[""{name=2, anchor=center, inner sep=0}, from=2-3, to=3-3]
	\arrow[from=3-1, to=3-2]
	\arrow[""{name=3, anchor=center, inner sep=0}, from=3-1, to=4-1]
	\arrow[from=3-2, to=3-3]
	\arrow[from=3-3, to=3-4]
	\arrow[""{name=4, anchor=center, inner sep=0}, from=3-3, to=4-3]
	\arrow[""{name=5, anchor=center, inner sep=0}, from=3-4, to=4-4]
	\arrow[from=4-1, to=4-3]
	\arrow[from=4-3, to=4-4]
	\arrow["1"{description}, draw=none, from=1, to=2]
	\arrow["2"{description}, draw=none, from=2-3, to=0]
	\arrow["3"{description}, draw=none, from=3, to=4]
	\arrow["4"{description}, draw=none, from=4, to=5]
\end{tikzcd}}
\qquad  \text{ and }\qquad
 \adjustbox{scale=0.60}{\begin{tikzcd}
	&& \bullet & \bullet \\
	& \bullet & \bullet \\
	\bullet & \bullet & \bullet & \bullet \\
	\bullet && \bullet & \bullet
	\arrow[from=1-3, to=1-4]
	\arrow[from=1-3, to=2-3]
	\arrow[""{name=0, anchor=center, inner sep=0}, from=1-4, to=3-4]
	\arrow[from=2-2, to=2-3]
	\arrow[""{name=1, anchor=center, inner sep=0}, from=2-2, to=3-2]
	\arrow[""{name=2, anchor=center, inner sep=0}, from=2-3, to=3-3]
	\arrow[from=3-1, to=3-2]
	\arrow[""{name=3, anchor=center, inner sep=0}, from=3-1, to=4-1]
	\arrow[from=3-2, to=3-3]
	\arrow[from=3-3, to=3-4]
	\arrow[""{name=4, anchor=center, inner sep=0}, from=3-3, to=4-3]
	\arrow[""{name=5, anchor=center, inner sep=0}, from=3-4, to=4-4]
	\arrow[from=4-1, to=4-3]
	\arrow[from=4-3, to=4-4]
	\arrow["{1'}"{description}, draw=none, from=1, to=2]
	\arrow["{2'}"{description}, draw=none, from=2-3, to=0]
	\arrow["{3'}"{description}, draw=none, from=3, to=4]
	\arrow["{4'}"{description}, draw=none, from=4, to=5]
\end{tikzcd}}\, .\]

The $\Sigma$-step of type $\mathbf{d}$ determines that $1=1'$ and $2=2'$; the $\Sigma$-step of type $\mathbf{u}$ determines that $1=1'$ and $3=3'$. Applying Rule 4' to the $\Sigma$-squares $4$ and $4'$, and composing with the right side of 2 and with the bottom side of 3, we obtain simultaneously the $\Omega$ 2-cell corresponding to $\mathbf{d}$ and also to $\mathbf{u}$. Finally, note that the $\Omega$ 2-cell from a $\Sigma$-scheme to itself can always be chosen to be the identity.
For the other cases the argument is similar.
\end{proof}

As we have just seen, all $\Sigma$-paths of length 1 between two given $\Sigma$-schemes are equivalent. We show  in Proposition \ref{pro:du=ud} that this is also true for $\Sigma$-paths of length $2$. That is, every two $\Sigma$-paths of the form
$S_1\begin{tikzcd}
	{} & {}
	\arrow["i", squiggly, from=1-1, to=1-2]
\end{tikzcd}S_3\begin{tikzcd}
	{} & {}
	\arrow["j", squiggly, from=1-1, to=1-2]
\end{tikzcd}S_2$
\hskip2mm  and  \hskip2mm
  $S_1\begin{tikzcd}
	{} & {}
	\arrow["k", squiggly, from=1-1, to=1-2]
\end{tikzcd}S_4\begin{tikzcd}
	{} & {}
	\arrow["l", squiggly, from=1-1, to=1-2]
\end{tikzcd}S_2$,
where $i,j,k,l \in \{ \mathbf{d},\mathbf{u},\mathbf{s},\mathbf{d}_1,\mathbf{s}_1\}$, are equivalent. For $\Sigma$-paths of length greater than 2, see Remark \ref{rem:Sigma-paths-equiv} and Proposition \ref{pro:of-interest}.

\begin{remark}\label{rem:cycle} It is clear that if we have a $\Sg$-path of the form $\xymatrix{S_1\ar@{~>}[r]^{i_1}&S_2\ar@{~>}[r]^{i_2}&\dots S_{n-1}\ar@{~>}[r]^{i_{n-1}}&S_n}$ we also have a $\Sg$-path of the form $\xymatrix{S_n\ar@{~>}[r]^{i_{n-1}}&S_{n-1}\dots\ar@{~>}[r]^{i_2}&S_{2}\ar@{~>}[r]^{i_1}&S_1}$, and the two $\Sg$-paths correspond to $\Omega$ 2-cells which are inverse to each other. In particular, given a cycle of $\Sigma$-steps

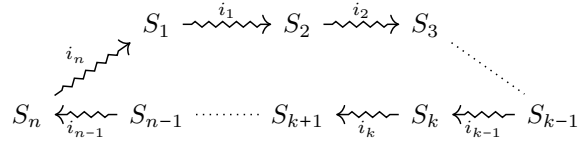
\[\begin{tikzcd}
	& {S_1} & {S_2} & {S_3} \\
	{S_n} & {S_{n-1}} & {S_{k+1}} & {S_{k}} & {S_{k-1}}
	\arrow["{{{i_1}}}", squiggly, from=1-2, to=1-3]
	\arrow["{{i_2}}", squiggly, from=1-3, to=1-4]
	\arrow[dotted, no head, from=1-4, to=2-5]
	\arrow["{{{i_n}}}", squiggly, from=2-1, to=1-2]
	\arrow["{i_{n-1}}", squiggly, from=2-2, to=2-1]
	\arrow[dotted, no head, from=2-3, to=2-2]
	\arrow["{i_k}", squiggly, from=2-4, to=2-3]
	\arrow["{i_{k-1}}", squiggly, from=2-5, to=2-4]
\end{tikzcd}\]

the corresponding $\Sigma$-path  from $S_1$ to $S_1$ is equivalent to the identity $\Sigma$-step if, and only if, for some $k$ with $1\leq k\leq n$, the $\Sigma$-paths
\(\begin{tikzcd}
	{S_1} & {S_2} & {S_{k-1}} & {S_k}
	\arrow["{i_1}", squiggly, from=1-1, to=1-2]
	\arrow[dotted, no head, from=1-2, to=1-3]
	\arrow["{i_{k-1}}", squiggly, from=1-3, to=1-4]
\end{tikzcd}\)
and
\(\begin{tikzcd}
	{S_1} & {S_n} & {S_{k+1}} & {S_k}
	\arrow["{{i_n}}", squiggly, from=1-1, to=1-2]
	\arrow[dotted, no head, from=1-2, to=1-3]
	\arrow["{{i_{k}}}", squiggly, from=1-3, to=1-4]
\end{tikzcd}\)
 are equivalent. Moreover, the $\Sigma$-path
 $\;\begin{tikzcd}
	{S_1} & {S_2} & {S_3} & \dots & {S_n} & {S_1}
	\arrow["{i_1}", squiggly, from=1-1, to=1-2]
	\arrow["{i_2}", squiggly, from=1-2, to=1-3]
	\arrow["{i_3}", squiggly, from=1-3, to=1-4]
	\arrow[squiggly, from=1-4, to=1-5]
	\arrow["{i_n}", squiggly, from=1-5, to=1-6]
\end{tikzcd}\;$ is equivalent to the identity $\Sigma$-path if and only if its any cyclic shift of it is so, and if and only if its reverse is so.
\end{remark}

\begin{proposition}\label{pro:du=ud}
Every two $\Sigma$-paths of length 2 (as in \cref{assumption-path}) between two $\Sigma$-schemes of level 3 with the same left border are equivalent. Equivalently, any cycle as in Remark \ref{rem:cycle} of length 4 corresponds to an identity 2-cell in $\catx[\Sigma_\ast]$.
\end{proposition}

\begin{proof} See Proposition \ref{pro:adend} and Corollary \ref{cor:length2} in Appendix~\ref{sec:appendix}.
\end{proof}

\begin{corollary}\label{cor:level2} For $\Sigma$-schemes $S_1$, $S_2$, $S_3$ and $S_4$ of level 2,
we have that for $i,j,k,l \in \{ d,u,{d}_1\}$, any  two $\Sigma$-paths of the form
$S_1\begin{tikzcd}
	{} & {}
	\arrow["i", squiggly, from=1-1, to=1-2]
\end{tikzcd}S_3\begin{tikzcd}
	{} & {}
	\arrow["j", squiggly, from=1-1, to=1-2]
\end{tikzcd}S_2$   and $S_1\begin{tikzcd}
	{} & {}
	\arrow["k", squiggly, from=1-1, to=1-2]
\end{tikzcd}S_4\begin{tikzcd}
	{} & {}
	\arrow["l", squiggly, from=1-1, to=1-2]
\end{tikzcd}S_2$ are equivalent.
\end{corollary}

\begin{remark}\label{rem:Sigma-paths-equiv} We believe that the property stated in \cref{pro:du=ud} is true for $\Sigma$-paths of any finite length. Equivalently, any cycle as in Remark \ref{rem:cycle} of finite length  corresponds to an identity 2-cell in $\catx[\Sigma_\ast]$. Although we do not have a complete proof involving all possible $\Sg$-paths, we have a proof of the property for $\Sg$-paths that we call {\em of interest}, which we will define next. This is stated in Proposition \ref{pro:of-interest}. This result combined with Proposition \ref{pro:du=ud} applies to all $\Sg$-paths with a role in the paper.
\end{remark}

We now define \textbf{\em $\Sg$-paths of interest}.
A $\Sg$-scheme $S$ in which a $\Sg$-step of type $\mathbf{d}$ may start (or end) is necessarily of type $\bd$ (see Notation \ref{nota:Sigma-steps}), that is, of the form
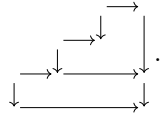
\[\adjustbox{scale=0.50}{\begin{tikzcd}
	&& {} & {} \\
	& {} & {} \\
	{} & {} && {} \\
	{} &&& {}
	\arrow[from=1-3, to=1-4]
	\arrow[from=1-3, to=2-3]
	\arrow[from=1-4, to=3-4]
	\arrow[from=2-2, to=2-3]
	\arrow[from=2-2, to=3-2]
	\arrow[from=3-1, to=3-2]
	\arrow[from=3-1, to=4-1]
	\arrow[from=3-2, to=3-4]
	\arrow[from=3-4, to=4-4]
	\arrow[from=4-1, to=4-4]
\end{tikzcd}}\, .\]

Analogously for the types $\mathbf{u}$, $\mathbf{s}$, $\mathbf{d}_1$ and $\mathbf{s}_1$.
Of course, a $\Sg$-scheme may be of various types simultaneously. And a $\Sigma$-scheme of a certain type may have different configurations. In the next table we consider some special configurations which are going to be of interest and give each one a name. For $\mathbf{s}$ and $\mathbf{s}_1$ we consider just one configuration each.

\begin{definition}\label{def:interest} A $\Sg$-scheme of level 3 has a \textbf{\em configuration of interest} if it has one of the seven configurations indicated in the right-side column of the table below. If a $\Sg$-scheme has a configuration of interest we say that it is a \textbf{\em $\Sg$-scheme of interest}. A \textbf{$\Sg$-path of interest} is a $\Sg$-path consisting only of $\Sg$-schemes of interest.
\end{definition}

\pagebreak[3]

\begin{longtable}{|c|c|c|}\hline
Type&General&Configurations of interest\\ \hline \hline
$\mathbf{d}$
&
\(\adjustbox{scale=0.60}{\begin{tikzcd}
	&& {} & {} \\
	& {} & {} \\
	{} & {} && {} \\
	{} &&& {}
	\arrow[from=1-3, to=1-4]
	\arrow[from=1-3, to=2-3]
	\arrow[from=1-4, to=3-4]
	\arrow[from=2-2, to=2-3]
	\arrow[from=2-2, to=3-2]
	\arrow[from=3-1, to=3-2]
	\arrow[from=3-1, to=4-1]
	\arrow[from=3-2, to=3-4]
	\arrow[from=3-4, to=4-4]
	\arrow[from=4-1, to=4-4]
\end{tikzcd}}\)
&
\(\adjustbox{scale=0.60}{\begin{tikzcd}
	&& {} & {} \\
	& {} & {} \\
	{} & {} & {} & {} \\
	{} &&& {}
	\arrow[from=1-3, to=1-4]
	\arrow[from=1-3, to=2-3]
	\arrow[from=1-4, to=3-4]
	\arrow[from=2-2, to=2-3]
	\arrow[from=2-2, to=3-2]
	\arrow[from=2-3, to=3-3]
	\arrow[from=3-1, to=3-2]
	\arrow[from=3-1, to=4-1]
	\arrow[from=3-2, to=3-3]
	\arrow[from=3-3, to=3-4]
	\arrow[from=3-4, to=4-4]
	\arrow[from=4-1, to=4-4]
	\arrow["\textbf{\large da}"', shift right=5, draw=none, from=4-1, to=4-4]
\end{tikzcd}}\)
\hspace{4mm}
\(\adjustbox{scale=0.60}{\begin{tikzcd}
	&& {} & {} \\
	& {} & {} & {} \\
	{} & {} && {} \\
	{} &&& {}
	\arrow[from=1-3, to=1-4]
	\arrow[from=1-3, to=2-3]
	\arrow[from=1-4, to=2-4]
	\arrow[from=2-2, to=2-3]
	\arrow[from=2-2, to=3-2]
	\arrow[from=2-3, to=2-4]
	\arrow[from=2-4, to=3-4]
	\arrow[from=3-1, to=3-2]
	\arrow[from=3-1, to=4-1]
	\arrow[from=3-2, to=3-4]
	\arrow[from=3-4, to=4-4]
	\arrow[from=4-1, to=4-4]
	\arrow["\textbf{\large db}"', shift right=5, draw=none, from=4-1, to=4-4]
\end{tikzcd}}
\)
\hspace{4mm}
\(\adjustbox{scale=0.60}{\begin{tikzcd}
	&& {} & {} \\
	& {} & {} \\
	& {} & {} & {} \\
	{} & {} && {} \\
	{} &&& {}
	\arrow[from=1-3, to=1-4]
	\arrow[from=1-3, to=2-3]
	\arrow[from=1-4, to=3-4]
	\arrow[from=2-2, to=2-3]
	\arrow[equals, from=2-2, to=3-2]
	\arrow[from=2-3, to=3-3]
	\arrow[from=3-2, to=3-3]
	\arrow[from=3-2, to=4-2]
	\arrow[from=3-3, to=3-4]
	\arrow[from=3-4, to=4-4]
	\arrow[from=4-1, to=4-2]
	\arrow[from=4-1, to=5-1]
	\arrow[from=4-2, to=4-4]
	\arrow[from=4-4, to=5-4]
	\arrow[from=5-1, to=5-4]
	\arrow["{{\textbf{\large dc}}}"{description}, shift right=5, draw=none, from=5-1, to=5-4]
\end{tikzcd}}
\)
\\ \hline
$\mathbf{u}$
&
\(\adjustbox{scale=0.60}{\begin{tikzcd}
	&& {} & {} \\
	& {} & {} \\
	{} & {} \\
	{} && {} & {}
	\arrow[from=1-3, to=1-4]
	\arrow[from=1-3, to=2-3]
	\arrow[from=1-4, to=4-4]
	\arrow[from=2-2, to=2-3]
	\arrow[from=2-2, to=3-2]
	\arrow[from=2-3, to=4-3]
	\arrow[from=3-1, to=3-2]
	\arrow[from=3-1, to=4-1]
	\arrow[from=4-1, to=4-3]
	\arrow[from=4-3, to=4-4]
\end{tikzcd}}\)
&
\(\adjustbox{scale=0.60}{\begin{tikzcd}
	&& {} & {} \\
	& {} & {} \\
	{} & {} \\
	{} & {} & {} & {}
	\arrow[from=1-3, to=1-4]
	\arrow[from=1-3, to=2-3]
	\arrow[from=1-4, to=4-4]
	\arrow[from=2-2, to=2-3]
	\arrow[from=2-2, to=3-2]
	\arrow[from=2-3, to=4-3]
	\arrow[from=3-1, to=3-2]
	\arrow[from=3-1, to=4-1]
	\arrow[from=3-2, to=4-2]
	\arrow[from=4-1, to=4-2]
	\arrow["{\textbf{\large ua}}"', shift right=5, draw=none, from=4-1, to=4-4]
	\arrow[from=4-2, to=4-3]
	\arrow[from=4-3, to=4-4]
\end{tikzcd}}\)
\hspace{8mm}
\(\adjustbox{scale=0.60}{\begin{tikzcd}
	&& {} & {} \\
	& {} & {} \\
	{} & {} & {} \\
	{} && {} & {}
	\arrow[from=1-3, to=1-4]
	\arrow[from=1-3, to=2-3]
	\arrow[from=1-4, to=4-4]
	\arrow[from=2-2, to=2-3]
	\arrow[from=2-2, to=3-2]
	\arrow[from=2-3, to=3-3]
	\arrow[from=3-1, to=3-2]
	\arrow[from=3-1, to=4-1]
	\arrow[from=3-2, to=3-3]
	\arrow[from=3-3, to=4-3]
	\arrow[from=4-1, to=4-3]
	\arrow["{{\textbf{\large ub}}}"', shift right=5, draw=none, from=4-1, to=4-4]
	\arrow[from=4-3, to=4-4]
\end{tikzcd}}
\)
\\ \hline
$\mathbf{s}$
&
\(\adjustbox{scale=0.60}{\begin{tikzcd}
	&& {} & {} \\
	& {} & {} & {} \\
	{} & {} \\
	{} & {} && {}
	\arrow[from=1-3, to=1-4]
	\arrow[from=1-3, to=2-3]
	\arrow[from=1-4, to=2-4]
	\arrow[from=2-2, to=2-3]
	\arrow[from=2-2, to=3-2]
	\arrow[from=2-3, to=2-4]
	\arrow[from=2-4, to=4-4]
	\arrow[from=3-1, to=3-2]
	\arrow[from=3-1, to=4-1]
	\arrow[from=3-2, to=4-2]
	\arrow[from=4-1, to=4-2]
	\arrow[from=4-2, to=4-4]
\end{tikzcd}}\)
&
\(\adjustbox{scale=0.60}{\begin{tikzcd}
	&& {} & {} \\
	& {} & {} & {} \\
	{} & {} \\
	{} & {} && {}
	\arrow[from=1-3, to=1-4]
	\arrow[from=1-3, to=2-3]
	\arrow[from=1-4, to=2-4]
	\arrow[from=2-2, to=2-3]
	\arrow[from=2-2, to=3-2]
	\arrow[from=2-3, to=2-4]
	\arrow[from=2-4, to=4-4]
	\arrow[from=3-1, to=3-2]
	\arrow[from=3-1, to=4-1]
	\arrow[from=3-2, to=4-2]
	\arrow[from=4-1, to=4-2]
	\arrow["{\text{\large $\mathbf{s}$}}"', shift right=4, draw=none, from=4-1, to=4-4]
	\arrow[from=4-2, to=4-4]
\end{tikzcd}}\)
\\ \hline
$\mathbf{s}_1$
&
\(\adjustbox{scale=0.60}{\begin{tikzcd}
	&& {} & {} \\
	& {} & {} \\
	& {} && {} \\
	{} & {} \\
	{} & {} && {}
	\arrow[from=1-3, to=1-4]
	\arrow[from=1-3, to=2-3]
	\arrow[from=1-4, to=3-4]
	\arrow[from=2-2, to=2-3]
	\arrow[equals, from=2-2, to=3-2]
	\arrow[from=3-2, to=3-4]
	\arrow[from=3-2, to=4-2]
	\arrow[from=3-4, to=5-4]
	\arrow[from=4-1, to=4-2]
	\arrow[from=4-1, to=5-1]
	\arrow[from=4-2, to=5-2]
	\arrow[from=5-1, to=5-2]
	\arrow[from=5-2, to=5-4]
\end{tikzcd}}\)
&
\(\adjustbox{scale=0.60}{\begin{tikzcd}
	&& {} & {} \\
	& {} & {} \\
	& {} & {} & {} \\
	{} & {} \\
	{} & {} && {}
	\arrow[from=1-3, to=1-4]
	\arrow[from=1-3, to=2-3]
	\arrow[from=1-4, to=3-4]
	\arrow[from=2-2, to=2-3]
	\arrow[equals, from=2-2, to=3-2]
	\arrow[from=2-3, to=3-3]
	\arrow[from=3-2, to=3-3]
	\arrow[from=3-2, to=4-2]
	\arrow[from=3-3, to=3-4]
	\arrow[from=3-4, to=5-4]
	\arrow[from=4-1, to=4-2]
	\arrow[from=4-1, to=5-1]
	\arrow[from=4-2, to=5-2]
	\arrow[from=5-1, to=5-2]
	\arrow["{\text{\large $\mathbf{s}_1$}}"', shift right=5, draw=none, from=5-1, to=5-4] %
	\arrow[from=5-2, to=5-4]
\end{tikzcd}}\)
\\ \hline
\end{longtable}

Observe that the configuration of type $\mathbf{s}_1$ given by
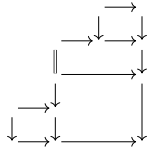
\[\adjustbox{scale=0.50}{\begin{tikzcd}
	&& {} & {} \\
	& {} & {} & {} \\
	& {} && {} \\
	{} & {} \\
	{} & {} && {}
	\arrow[from=1-3, to=1-4]
	\arrow[from=1-3, to=2-3]
	\arrow[from=1-4, to=2-4]
	\arrow[from=2-2, to=2-3]
	\arrow[equals, from=2-2, to=3-2]
	\arrow[from=2-3, to=2-4]
	\arrow[from=2-4, to=3-4]
	\arrow[from=3-2, to=3-4]
	\arrow[from=3-2, to=4-2]
	\arrow[from=3-4, to=5-4]
	\arrow[from=4-1, to=4-2]
	\arrow[from=4-1, to=5-1]
	\arrow[from=4-2, to=5-2]
	\arrow[from=5-1, to=5-2]
	\arrow[from=5-2, to=5-4]
\end{tikzcd}}\]
is also of type $\mathbf{s}$, thus we do not consider it in the row of ${\bs}_1$.

\begin{remark}
Of course, in a $\Sg$-path of interest, every $\Sigma$-step is incident to two $\Sg$-schemes with a same configuration, and the name of that configuration starts with the letter representing the type of the $\Sg$-step.
\end{remark}

\begin{proposition} \label{pro:of-interest} Every two $\Sg$-paths of interest starting and ending at the same $\Sg$-schemes are equivalent.
\end{proposition}

\begin{proof} See  \cref{thm:of-interest5} in Appendix~\ref{sec:appendix}.
\end{proof}

The following property also has a role in what follows and is proven in Appendix~\ref{sec:appendix}. It involves horizontal composition in $\catx[\Sigma_{\ast}]$, which will be defined in Subsection \ref{sec:horiz} below.

\begin{proposition}\label{pro:Omega-f} Let $A\xrightarrow{\bar{f}}B\xrightarrow{\bar{g}}C\xrightarrow{\bar{h}}D\xrightarrow{\bar{k}}E$ be $\Sg$-cospans, where $\bar{f}=(f,r)$, $\bar{g}=(g,s)$, $\bar{h}=(h,t)$ and $\bar{k}=(k,u)$.

(1) Let $\Omega\colon (l_1g,m_1u) \Rightarrow (l_2g, m_2u)$ be a basic $\Omega$ 2-cell determined by a $\Sg$-step of level 2 of the type $d$ or $u$ between two $\Sg$-schemes of left border $(s,h,t,k)$ and right border $(l_i,m_i)$, respectively. Then the 2-cell $\Omega\circ 1_{\bar{f}}$ is an $\Omega$ 2-cell corresponding to a $\Sg$-path of interest of $\Sg$-schemes of level 3 and left border $(r,g,s,h,t,k)$.

(2) Let $\Omega\colon (l_1f,m_1t) \rightarrow (l_2f,m_2t)$ be a basic $\Omega$ 2-cell determined by a $\Sg$-step of level 2 of the type $d$ or $u$ between two $\Sg$-schemes of left border $(r,g,s,h)$ and right border $(l_i,m_i)$, respectively. Then the 2-cell $1_{\bar{k}}\circ \Omega$ is an $\Omega$ 2-cell corresponding to a $\Sg$-path of interest of $\Sg$-schemes of level 3 and left border $(r,g,s,h,t,k)$.
\end{proposition}

\begin{proof}
  See Proposition~\ref{pro:alpha.f} and Corollary~\ref{cor:Omega-h} in Appendix~\ref{sec:appendix}.
\end{proof}

\subsection{Horizontal composition and the associator}\label{sec:horiz}

In the composition of $\Sigma$-cospans we are going to use certain chosen $\Sigma$-squares, called canonical $\Sg$-squares.

\begin{definition}\label{assum1}\textbf{\em Canonical $\Sigma$-squares.}
	Since $\Sigma$ admits a left calculus of lax fractions, by Square we know that, for each span $I\xleftarrow{r}B\xrightarrow{g}J$ with $r\in \Sigma$, there is some $\Sigma$-cospan forming a $\Sigma$-square with it. In order to define the composition, we assume a prespecified map which assigns to each such a pair $(r,g)$ a fixed $\Sigma$-square
	\[\begin{tikzcd}
		B & I \\
		J & {\dot{B}}
		\arrow["r", from=1-1, to=1-2]
		\arrow[""{name=0, anchor=center, inner sep=0}, "g"', from=1-1, to=2-1]
		\arrow[""{name=1, anchor=center, inner sep=0}, "{\dot{g}}", from=1-2, to=2-2]
		\arrow["{\dot{r}}"', from=2-1, to=2-2]
		\arrow["{\SIGMA^{\dot{\alpha}}}"{description}, draw=none, from=0, to=1]
	\end{tikzcd}\; \]
	called a \textbf{\em canonical $\Sigma$-square}.
	For $r\colon A\to B$ in $\Sigma$ and any $f\colon A\to B$, we assume that
	\[\begin{tikzcd}
		A & B \\
		A & B
		\arrow["r", from=1-1, to=1-2]
		\arrow[""{name=0, anchor=center, inner sep=0}, "{1_A}"', from=1-1, to=2-1]
		\arrow[""{name=1, anchor=center, inner sep=0}, "{1_B}", from=1-2, to=2-2]
		\arrow["r", from=2-1, to=2-2]
		\arrow["{\SIGMA^{\id}}"{description}, draw=none, from=0, to=1]
	\end{tikzcd}\qquad \text{and}\qquad \begin{tikzcd}
	A & A \\
	B & B
	\arrow["{1_A}", from=1-1, to=1-2]
	\arrow[""{name=0, anchor=center, inner sep=0}, "f"', from=1-1, to=2-1]
	\arrow[""{name=1, anchor=center, inner sep=0}, "f", from=1-2, to=2-2]
	\arrow["{1_A}", from=2-1, to=2-2]
	\arrow["{{\SIGMA^{\id}}}"{description}, draw=none, from=0, to=1]
	\end{tikzcd}\]
	are canonical $\Sigma$-squares (which exist by the Repletion axioms). Sometimes canonical $\Sigma$-squares will be indicated with just the symbol $\stackrel{\scriptscriptstyle{\bullet}}{\Sigma}$.

The \textbf{\em canonical $\Sg$-scheme} of level 3 with a given left border, which we denote by $\Can$, is given by
\[\adjustbox{scale=0.70}{\begin{tikzcd}
	&& {} & {} \\
	& {} & {} & {} \\
	{} & {} & {} & {} \\
	{} & {} & {} & {}
	\arrow[from=1-3, to=1-4]
	\arrow[""{name=0, anchor=center, inner sep=0}, from=1-3, to=2-3]
	\arrow[""{name=1, anchor=center, inner sep=0}, from=1-4, to=2-4]
	\arrow[from=2-2, to=2-3]
	\arrow[""{name=2, anchor=center, inner sep=0}, from=2-2, to=3-2]
	\arrow[from=2-3, to=2-4]
	\arrow[""{name=3, anchor=center, inner sep=0}, from=2-3, to=3-3]
	\arrow[""{name=4, anchor=center, inner sep=0}, from=2-4, to=3-4]
	\arrow[from=3-1, to=3-2]
	\arrow[""{name=5, anchor=center, inner sep=0}, from=3-1, to=4-1]
	\arrow[from=3-2, to=3-3]
	\arrow[""{name=6, anchor=center, inner sep=0}, from=3-2, to=4-2]
	\arrow[from=3-3, to=3-4]
	\arrow[""{name=7, anchor=center, inner sep=0}, from=3-3, to=4-3]
	\arrow[""{name=8, anchor=center, inner sep=0}, from=3-4, to=4-4]
	\arrow[from=4-1, to=4-2]
	\arrow[from=4-2, to=4-3]
	\arrow[from=4-3, to=4-4]
	\arrow["\SIGMAc"{description}, shift left, draw=none, from=0, to=1]
	\arrow["\SIGMAc"{description}, shift left, draw=none, from=2, to=3]
	\arrow["\SIGMAc"{description}, shift left, draw=none, from=3, to=4]
	\arrow["\SIGMAc"{description}, shift left, draw=none, from=5, to=6]
	\arrow["\SIGMAc"{description}, shift left, draw=none, from=6, to=7]
	\arrow["\SIGMAc"{description}, shift left, draw=none, from=7, to=8]
\end{tikzcd}}\]
and likewise for $\Sg$-schemes of level 2.
\end{definition}
\begin{definition}\label{def:comp-cospans} \textbf{\em (Horizontal) composition of $\Sigma$-cospans}. Given $\Sigma$-cospans $(f,I,r)\colon A\to B$ and $(g,J,s)\colon B\to C$, their composition is the $\Sigma$-cospan $(\dot{g}f, \dot{B}, \dot{r}s)\colon A\to C$, obtained by taking the canonical $\Sg$-square of $r$ and $g$:
\[\begin{tikzcd}
		A & I & B \\
		& \dot{B} & J & C\,.
		\arrow["f", from=1-1, to=1-2]
		\arrow["r"', from=1-3, to=1-2]
		\arrow[""{name=0, anchor=center, inner sep=0}, "g", from=1-3, to=2-3]
		\arrow["{\dot{r}}", from=2-3, to=2-2]
		\arrow[""{name=1, anchor=center, inner sep=0}, "{\dot{g}}"', from=1-2, to=2-2]
		\arrow["s", from=2-4, to=2-3]
		\arrow["{\SIGMA^{\dot{\alpha}}}"{description}, draw=none, from=1, to=0]
	\end{tikzcd}\]
\end{definition}

\begin{definition}\label{def:horizontal} \textbf{\em Horizontal composition of 2-cells.} Given two horizontally composable 2-morphisms as in the diagram
\begin{equation}\label{eq:(A11)}
	\begin{tikzcd}
	A & {I_1} & B & {J_1} & C \\
	& X & B & Y & C \\
	A & {I_2} & B & {J_2} & C \rlap{\,,}
	\arrow["{f_1}", from=1-1, to=1-2]
	\arrow[Rightarrow, no head, from=1-1, to=3-1]
	\arrow[""{name=0, anchor=center, inner sep=0}, "{x_1}", from=1-2, to=2-2]
	\arrow["\alpha"', shorten <=11pt, shorten >=11pt, Rightarrow, from=1-2, to=3-1]
	\arrow["{r_1}"', from=1-3, to=1-2]
	\arrow["{g_1}", from=1-3, to=1-4]
	\arrow[""{name=1, anchor=center, inner sep=0}, Rightarrow, no head, from=1-3, to=2-3]
	\arrow[""{name=2, anchor=center, inner sep=0}, "{y_1}", from=1-4, to=2-4]
	\arrow["\beta"', shorten <=11pt, shorten >=11pt, Rightarrow, from=1-4, to=3-3]
	\arrow["{s_1}"', from=1-5, to=1-4]
	\arrow[""{name=3, anchor=center, inner sep=0}, Rightarrow, no head, from=1-5, to=2-5]
	\arrow["{x_3}", from=2-3, to=2-2]
	\arrow[""{name=4, anchor=center, inner sep=0}, Rightarrow, no head, from=2-3, to=3-3]
	\arrow["{y_3}"', from=2-5, to=2-4]
	\arrow[""{name=5, anchor=center, inner sep=0}, Rightarrow, no head, from=2-5, to=3-5]
	\arrow["{f_2}"', from=3-1, to=3-2]
	\arrow[""{name=6, anchor=center, inner sep=0}, "{x_2}"', from=3-2, to=2-2]
	\arrow["{r_2}", from=3-3, to=3-2]
	\arrow["{g_2}"', from=3-3, to=3-4]
	\arrow[""{name=7, anchor=center, inner sep=0}, "{y_2}"', from=3-4, to=2-4]
	\arrow["{s_2}", from=3-5, to=3-4]
	\arrow["{\SIGMA^{\delta_1}}"{description}, draw=none, from=0, to=1]
	\arrow["{\SIGMA^{\epsilon_1}}"{description}, draw=none, from=2, to=3]
	\arrow["{\SIGMA^{\delta_2}}"{description}, draw=none, from=6, to=4]
	\arrow["{\SIGMA^{\epsilon_2}}"{description}, draw=none, from=7, to=5]
\end{tikzcd}
\end{equation}
use Rule 6 (\cref{pro:useful_rules}) to obtain the following equality of pasting diagrams:
\begin{equation}\label{eq:(A12)}
	\begin{tikzcd}
		B & X &&& B & X \\
		{J_1} &&& {J_1} & {J_2} \\
		Y & V &&& Y & V
		\arrow["{{x_3}}", from=1-1, to=1-2]
		\arrow["{{g_1}}"', from=1-1, to=2-1]
		\arrow[""{name=0, anchor=center, inner sep=0}, "{{y'_1}}", from=1-2, to=3-2]
		\arrow[""{name=1, anchor=center, inner sep=0}, "{{y'_2}}"{pos=0.6}, curve={height=-40pt}, from=1-2, to=3-2]
		\arrow["{{x_3}}", from=1-5, to=1-6]
		\arrow["{{g_1}}"', curve={height=6pt}, from=1-5, to=2-4]
		\arrow["{{g_2}}", from=1-5, to=2-5]
		\arrow[""{name=2, anchor=center, inner sep=0}, "{{y'_2}}", from=1-6, to=3-6]
		\arrow["{{y_1}}"', from=2-1, to=3-1]
		\arrow["\beta", shorten <=4pt, shorten >=4pt, Rightarrow, from=2-4, to=2-5]
		\arrow["{{y_1}}"', curve={height=6pt}, from=2-4, to=3-5]
		\arrow["{{y_2}}", from=2-5, to=3-5]
		\arrow["v", from=3-1, to=3-2]
		\arrow["v", from=3-5, to=3-6]
		\arrow["{{\text{\normalsize =}}}"{description, pos=0.6}, draw=none, from=1, to=2-4]
		\arrow["{{\beta'}}"{pos=0.6}, shorten <=10pt, shorten >=6pt, Rightarrow, from=0, to=1]
		\arrow["{{\SIGMA^{\xi_1}}}"{description}, draw=none, from=2-1, to=0]
		\arrow["{{\SIGMA^{\xi_2}}}"{description}, draw=none, from=2-5, to=2]
	\end{tikzcd}
	\end{equation}
	We can now form the following 2-morphism:
	\begin{equation}\label{eq:(A13)}
	\begin{tikzcd}
		A & {I_1} & X & V & Y & C \\
		&&& V & Y & C \\
		A & {I_2} & X & V & Y & C
		\arrow["{f_1}", from=1-1, to=1-2]
		\arrow[Rightarrow, no head, from=1-1, to=3-1]
		\arrow["{{x_1}}", from=1-2, to=1-3]
		\arrow["{{y_1'}}", from=1-3, to=1-4]
		\arrow["\alpha"', shorten <=11pt, shorten >=11pt, Rightarrow, from=1-3, to=3-1]
		\arrow[Rightarrow, no head, from=1-3, to=3-3]
		\arrow["v", tail reversed, no head, from=1-4, to=1-5]
		\arrow[""{name=0, anchor=center, inner sep=0}, Rightarrow, no head, from=1-4, to=2-4]
		\arrow["{{\beta'}}"', shorten <=11pt, shorten >=8pt, Rightarrow, from=1-4, to=3-3]
		\arrow["{y_3}", tail reversed, no head, from=1-5, to=1-6]
		\arrow[""{name=1, anchor=center, inner sep=0}, Rightarrow, no head, from=1-6, to=2-6]
		\arrow["v", from=2-5, to=2-4]
		\arrow["{y_3}", from=2-6, to=2-5]
		\arrow[""{name=2, anchor=center, inner sep=0}, Rightarrow, no head, from=2-6, to=3-6]
		\arrow["{f_2}"', from=3-1, to=3-2]
		\arrow["{{x_2}}"', from=3-2, to=3-3]
		\arrow["{{y'_2}}"', from=3-3, to=3-4]
		\arrow[""{name=3, anchor=center, inner sep=0}, Rightarrow, no head, from=3-4, to=2-4]
		\arrow["v", from=3-5, to=3-4]
		\arrow["{y_3}", from=3-6, to=3-5]
		\arrow["{\SIGMA^{\id}}"{description}, draw=none, from=0, to=1]
		\arrow["{\SIGMA^{\id}}"{description}, draw=none, from=3, to=2]
	\end{tikzcd}
\end{equation}
Consider the following $\Sg$-path between $\Sg$-schemes of level 2, where $d$ and $u$ refer to the type of the $\Sg$-steps (see Notation \ref{nota:Sigma-steps}):
\begin{equation}\label{eq:Omega-i}
  \begin{tikzcd}
	& B & {I_i} &&& B & {I_i} &&& B & {I_i} \\
	C & {J_i} & {\dot{B}_i} && C & {J_i} & {\dot{B}_i} && C & {J_i} \\
	C & {J_i} & {\dot{B}_i} && C & Y & {\tilde{J}_i} && C & Y & V
	\arrow["{{r_i}}", from=1-2, to=1-3]
	\arrow[""{name=0, anchor=center, inner sep=0}, "{{g_i}}"', from=1-2, to=2-2]
	\arrow[""{name=1, anchor=center, inner sep=0}, "{{\dot{g}_i}}", from=1-3, to=2-3]
	\arrow["{{r_i}}", from=1-6, to=1-7]
	\arrow[""{name=2, anchor=center, inner sep=0}, "{{g_i}}"', from=1-6, to=2-6]
	\arrow[""{name=3, anchor=center, inner sep=0}, "{{\dot{g}_i}}", from=1-7, to=2-7]
	\arrow["{{r_i}}", from=1-10, to=1-11]
	\arrow["{{g_i}}"', from=1-10, to=2-10]
	\arrow[""{name=4, anchor=center, inner sep=0}, "{{{y'_ix_i}}}", from=1-11, to=3-11]
	\arrow["{{s_i}}", from=2-1, to=2-2]
	\arrow[""{name=5, anchor=center, inner sep=0}, equals, from=2-1, to=3-1]
	\arrow["{{\dot{r}_i}}"', from=2-2, to=2-3]
	\arrow[""{name=6, anchor=center, inner sep=0}, equals, from=2-2, to=3-2]
	\arrow["d"{description}, between={0.2}{0.9}, squiggly, from=2-3, to=2-5]
	\arrow[""{name=7, anchor=center, inner sep=0}, equals, from=2-3, to=3-3]
	\arrow["{{s_i}}", from=2-5, to=2-6]
	\arrow[""{name=8, anchor=center, inner sep=0}, equals, from=2-5, to=3-5]
	\arrow["{{\dot{r}_i}}"', from=2-6, to=2-7]
	\arrow[""{name=9, anchor=center, inner sep=0}, "{{y_i}}", from=2-6, to=3-6]
	\arrow["u"{description}, between={0.2}{0.8}, squiggly, from=2-7, to=2-9]
	\arrow[""{name=10, anchor=center, inner sep=0}, "{{\tilde{y}_i}}", from=2-7, to=3-7]
	\arrow["{{s_i}}", from=2-9, to=2-10]
	\arrow[""{name=11, anchor=center, inner sep=0}, equals, from=2-9, to=3-9]
	\arrow[""{name=12, anchor=center, inner sep=0}, "{{y_i}}", from=2-10, to=3-10]
	\arrow["{{s_i}}"', from=3-1, to=3-2]
	\arrow["{{\dot{r}_i}}"', from=3-2, to=3-3]
	\arrow["{{y_3}}"', from=3-5, to=3-6]
	\arrow["{{\tilde{r}_i}}"', from=3-6, to=3-7]
	\arrow["{{y_3}}"', from=3-9, to=3-10]
	\arrow["v"', from=3-10, to=3-11]
	\arrow["{{\SIGMAc}}"{description}, draw=none, from=0, to=1]
	\arrow["{{\SIGMAc}}"{description}, draw=none, from=2, to=3]
	\arrow["{{{\SIGMAcid}}}"{description}, draw=none, from=5, to=6]
	\arrow["{{{\SIGMAcid}}}"{description}, draw=none, from=6, to=7]
	\arrow["{{{\SIGMA^{\epsilon_i}}}}"{description}, draw=none, from=8, to=9]
	\arrow["{{\SIGMAcnamed{\tilde{\alpha}_i}}}"{description}, draw=none, from=9, to=10]
	\arrow["{{\SIGMA^{\epsilon_i}}}"{description}, draw=none, from=11, to=12]
	\arrow["{{\SIGMA^{\xi_i\odot\delta_i}}}"{description}, draw=none, from=2-10, to=4]
\end{tikzcd}\, .
\end{equation} %

Let
$\Omega_i\colon (g_i,s_i)\circ (f_i,r_i)=(\dot{g}_if_i,\dot{r}_is_i) \implies (y'_ix_if_i,vy_3)$
be the $\Omega$ 2-cell determined by this $\Sg$-path.
We now have three vertically composable 2-cells
\[\begin{tikzcd}
	{(g_1,s_1)\circ (f_1,r_1)} & {(y'_1x_1f_1,vy_3)} && {(y'_2x_2f_2,vy_3)} & {(g_2,s_2)\circ (f_2,r_2)}
	\arrow["{\Omega_1}", Rightarrow, from=1-1, to=1-2]
	\arrow["{{[\beta'\circ \alpha, 1_V,1_V]}}", Rightarrow, from=1-2, to=1-4]
	\arrow["{\Omega_2^{-1}}", Rightarrow, from=1-4, to=1-5]
\end{tikzcd}.\]
The horizontal composition $[\beta,y_1,y_2]\circ [\alpha,x_1,x_2]$ is given by this vertical composition, i.e.
$$\Omega_2^{-1}\cdot [\beta'\circ \alpha, 1_V,1_V]\cdot \Omega_1\, .$$
\end{definition}

In \cref{pro:horizontal} we will show that the horizontal composition is well-defined --- that is, \cref{def:horizontal} does not depend on the choice of the $\Sg$-squares and $\beta'$ in \eqref{eq:(A12)}, and it respects the $\approx$-relation.
In \cref{pro:comp}, we will prove that horizontal composition is functorial.

\begin{definition}\label{def:assoc}\textbf{\em Associator.} Suppose we are given three composable $\Sigma$-cospans:
$$A\xrightarrow{\bar{f}}B\xrightarrow{\bar{g}}C\xrightarrow{\bar{h}}D$$
with $\bar{f}=(f,I,r)$, $\bar{g}=(g,J,s)$ and $\bar{h}=(h,K,t)$.
Consider the $\Sigma$-schemes of level 2 as below, where all $\Sigma$-squares are canonical and the middle $\Sigma$-scheme is the canonical one.
\[\begin{tikzcd}
	& \bullet & \bullet \\
	\bullet & \bullet \\
	\bullet & \bullet & \bullet
	\arrow["r", from=1-2, to=1-3]
	\arrow["g"', from=1-2, to=2-2]
	\arrow[""{name=0, anchor=center, inner sep=0}, from=1-3, to=3-3]
	\arrow["s", from=2-1, to=2-2]
	\arrow[""{name=1, anchor=center, inner sep=0}, "h"', from=2-1, to=3-1]
	\arrow[""{name=2, anchor=center, inner sep=0}, from=2-2, to=3-2]
	\arrow[from=3-1, to=3-2]
	\arrow["{(1)}"', shift right=5, draw=none, from=3-1, to=3-3]
	\arrow[from=3-2, to=3-3]
	\arrow["\SIGMAc"{description}, draw=none, from=1, to=2]
	\arrow["\SIGMAc"{description}, draw=none, from=2-2, to=0]
\end{tikzcd}
\hspace{15mm}
\begin{tikzcd}
	& \bullet & \bullet \\
	\bullet & \bullet & \bullet \\
	\bullet & \bullet & \bullet
	\arrow["r", from=1-2, to=1-3]
	\arrow[""{name=0, anchor=center, inner sep=0}, "g"', from=1-2, to=2-2]
	\arrow[""{name=1, anchor=center, inner sep=0}, from=1-3, to=2-3]
	\arrow["s", from=2-1, to=2-2]
	\arrow[""{name=2, anchor=center, inner sep=0}, "h"', from=2-1, to=3-1]
	\arrow[from=2-2, to=2-3]
	\arrow[""{name=3, anchor=center, inner sep=0}, from=2-2, to=3-2]
	\arrow[""{name=4, anchor=center, inner sep=0}, from=2-3, to=3-3]
	\arrow[from=3-1, to=3-2]
	\arrow["{(2)}"', shift right=5, draw=none, from=3-1, to=3-3]
	\arrow[from=3-2, to=3-3]
	\arrow["\SIGMAc"{description}, draw=none, from=0, to=1]
	\arrow["\SIGMAc"{description}, draw=none, from=2, to=3]
	\arrow["\SIGMAc"{description}, draw=none, from=3, to=4]
\end{tikzcd}
\hspace{15mm}
\begin{tikzcd}
	& \bullet & \bullet \\
	\bullet & \bullet & \bullet \\
	\bullet && \bullet
	\arrow["r", from=1-2, to=1-3]
	\arrow[""{name=0, anchor=center, inner sep=0}, "g"', from=1-2, to=2-2]
	\arrow[""{name=1, anchor=center, inner sep=0}, from=1-3, to=2-3]
	\arrow["s", from=2-1, to=2-2]
	\arrow[""{name=2, anchor=center, inner sep=0}, "h"', from=2-1, to=3-1]
	\arrow[from=2-2, to=2-3]
	\arrow[""{name=3, anchor=center, inner sep=0}, from=2-3, to=3-3]
	\arrow["{{(3)}}"', shift right=5, draw=none, from=3-1, to=3-3]
	\arrow[from=3-1, to=3-3]
	\arrow["\SIGMAc"{description}, draw=none, from=0, to=1]
	\arrow["\SIGMAc"{description}, draw=none, from=2, to=3]
\end{tikzcd}\]
The right borders of the $\Sigma$-schemes  (1) and (3), after performing the composition with $f$ on the left and with $t$ on the right, give the compositions $\left(\left(h,t\right)\circ\left(g,s\right)\right)\circ\left(f,r\right)$ and $\left(h,t\right)\circ\left(\left(g,s\right)\circ\left(f,r\right)\right)$, respectively.

Each component of the associator
$$\Assoc_{\bar{h}, \bar{g}, \bar{f}}\colon \left(\left(h,t\right)\circ\left(g,s\right)\right)\circ\left(f,r\right) \Rightarrow \left(h,t\right)\circ\left(\left(g,s\right)\circ\left(f,r\right)\right)$$
is the $\Omega$ 2-cell corresponding to the $\Sigma$-path $\xymatrix{(1)\ar@{~>}[r]^{u}&(2) \ar@{~>}[r]^{d}&(3)}$.
\end{definition}
In Proposition \ref{pro:natural-iso} we will prove that the associator, as defined above, is indeed a natural isomorphism from $(-\circ -)\circ -$ to $-\circ(-\circ -)$.

\begin{definition}\label{def:uni}\textbf{\em Identities and unitors.}
The identity 1-cell on an object $A$ is the cospan $A\xrightarrow{\id}A\xleftarrow{\id}A$.
From \cref{def:comp-cospans} and the canonical identity $\Sigma$-squares from \cref{assum1}, we see the identity 1-cells act strictly as identities for the horizontal composition. Thus, we define the \textbf{\em left  unitors} and the \textbf{\em right  unitors} to simply be given by identity 2-cells.
\end{definition}

\begin{proposition}\label{pro:horizontal} The horizontal composition is well-defined.
\end{proposition}

\begin{proof}
(1) First we observe that, up to the $\approx$-relation, the horizontal composition does not depend on the choice of the $\Sigma$-squares and the 2-cell $\beta'$ of \eqref{eq:(A12)}.  Indeed, suppose that in addition to the data of \eqref{eq:(A12)} we have the equality of pasting diagrams
\begin{equation}\label{eq:(A14)}
	\begin{tikzcd}
	B & X &&& B & X \\
	{J_1} &&& {J_1} & {J_2} \\
	Y & {\bar{V}} &&& Y & {\bar{V}}
	\arrow["{{{x_3}}}", from=1-1, to=1-2]
	\arrow["{{{g_1}}}"', from=1-1, to=2-1]
	\arrow[""{name=0, anchor=center, inner sep=0}, "{{\bar{y}_1}}", from=1-2, to=3-2]
	\arrow[""{name=1, anchor=center, inner sep=0}, "{{\bar{y}_2}}", curve={height=-40pt}, from=1-2, to=3-2]
	\arrow["{{{x_3}}}", from=1-5, to=1-6]
	\arrow["{{{g_1}}}"', curve={height=6pt}, from=1-5, to=2-4]
	\arrow["{{{g_2}}}", from=1-5, to=2-5]
	\arrow[""{name=2, anchor=center, inner sep=0}, "{{\bar{y}_2}}", from=1-6, to=3-6]
	\arrow["{{{y_1}}}"', from=2-1, to=3-1]
	\arrow["\beta", shorten <=4pt, shorten >=4pt, Rightarrow, from=2-4, to=2-5]
	\arrow["{{{y_1}}}"', curve={height=6pt}, from=2-4, to=3-5]
	\arrow["{{{y_2}}}", from=2-5, to=3-5]
	\arrow["{{\bar{v}}}", from=3-1, to=3-2]
	\arrow["{{\bar{v}}}", from=3-5, to=3-6]
	\arrow["{{{\text{\normalsize =}}}}"{description, pos=0.7}, draw=none, from=1, to=2-4]
	\arrow["{{\bar{\beta}}}", shorten <=12pt, shorten >=6pt, Rightarrow, from=0, to=1]
	\arrow["{{\SIGMA^{\bar{\xi}_1}}}"{description}, draw=none, from=2-1, to=0]
	\arrow["{{\SIGMA^{\bar{\xi}_2}}}"{description}, draw=none, from=2-5, to=2]
	\end{tikzcd}
\end{equation}
leading to the 2-morphism $(\bar{\beta}\circ \alpha, 1_{\bar{V}}, 1_{\bar{V}})\colon (\bar{y}_1x_1f_1,\bar{v}y_3)\Rightarrow (\bar{y}_2x_2f_2,\bar{v}y_3)$.
Using Rule 4 of Proposition \ref{pro:useful_rules}, there are morphisms
$\begin{tikzcd}
	Q & T & {\bar{Q}}
	\arrow["p", from=1-1, to=1-2]
	\arrow["{\bar{p}}"', from=1-3, to=1-2]
\end{tikzcd}$, $\Sigma$-squares $\Sigma^{\chi}$ and $\Sigma^{\bar{\chi}}$, and  invertible 2-cells $\theta_i$, such that
\begin{equation}\label{eq:(A15)}
\begin{tikzcd}
	B & X && B & X \\
	Y & {\bar{V}} & V & Y & V & {} \\
	Y & T && Y & T \\
	Y & {\bar{V}} & V & Y & V \\
	B & X && B & X
	\arrow["{{{{x_3}}}}", from=1-1, to=1-2]
	\arrow[""{name=0, anchor=center, inner sep=0}, "{{{{y_1g_1}}}}"', from=1-1, to=2-1]
	\arrow[""{name=1, anchor=center, inner sep=0}, "{{\bar{y}_1}}", from=1-2, to=2-2]
	\arrow["{y_1'}", curve={height=-6pt}, from=1-2, to=2-3]
	\arrow["{{{{x_3}}}}", from=1-4, to=1-5]
	\arrow[""{name=2, anchor=center, inner sep=0}, "{{{{y_1g_1}}}}"', from=1-4, to=2-4]
	\arrow[""{name=3, anchor=center, inner sep=0}, "{y'_1}", from=1-5, to=2-5]
	\arrow["{\bar{v}}"', dotted, from=2-1, to=2-2]
	\arrow[""{name=4, anchor=center, inner sep=0}, equals, from=2-1, to=3-1]
	\arrow["{{\theta_1}}", Rightarrow, from=2-2, to=2-3]
	\arrow[""{name=5, anchor=center, inner sep=0}, "{\bar{p}}", from=2-2, to=3-2]
	\arrow["p", curve={height=-6pt}, from=2-3, to=3-2]
	\arrow["v"', dotted, from=2-4, to=2-5]
	\arrow[""{name=6, anchor=center, inner sep=0}, equals, from=2-4, to=3-4]
	\arrow[""{name=7, anchor=center, inner sep=0}, "p", from=2-5, to=3-5]
	\arrow["t", from=3-1, to=3-2]
	\arrow[""{name=8, anchor=center, inner sep=0}, equals, from=3-1, to=4-1]
	\arrow["{{{\text{\normalsize =}}}}"{description, pos=0.725}, draw=none, from=3-2, to=3-4]
	\arrow["t", from=3-4, to=3-5]
	\arrow[""{name=9, anchor=center, inner sep=0}, equals, from=3-4, to=4-4]
	\arrow["{\bar{v}}", dotted, from=4-1, to=4-2]
	\arrow[""{name=10, anchor=center, inner sep=0}, "{\bar{p}}"', from=4-2, to=3-2]
	\arrow["{{\theta_2}}", Rightarrow, from=4-2, to=4-3]
	\arrow["p"', curve={height=6pt}, from=4-3, to=3-2]
	\arrow["v", dotted, from=4-4, to=4-5]
	\arrow[""{name=11, anchor=center, inner sep=0}, "p"', from=4-5, to=3-5]
	\arrow[""{name=12, anchor=center, inner sep=0}, "{{{{y_2g_2}}}}", from=5-1, to=4-1]
	\arrow["{{{{x_3}}}}"', from=5-1, to=5-2]
	\arrow[""{name=13, anchor=center, inner sep=0}, "{\bar{y}_2}"', from=5-2, to=4-2]
	\arrow["{y'_2}"', curve={height=6pt}, from=5-2, to=4-3]
	\arrow[""{name=14, anchor=center, inner sep=0}, "{{{{y_2g_2}}}}", from=5-4, to=4-4]
	\arrow["{{{{x_3}}}}"', from=5-4, to=5-5]
	\arrow[""{name=15, anchor=center, inner sep=0}, "{y'_2}"', from=5-5, to=4-5]
	\arrow["{{{{\SIGMA^{\bar{\xi}_1}}}}}"{description}, draw=none, from=0, to=1]
	\arrow["{{{{\SIGMA^{{\xi}_1}}}}}"{description}, draw=none, from=2, to=3]
	\arrow["{{{{\SIGMA^{\bar{\chi}}}}}}"{description}, draw=none, from=5, to=4]
	\arrow["{{{{\SIGMA^{{\chi}}}}}}"{description}, draw=none, from=6, to=7]
	\arrow["{{{{\SIGMA^{\bar{\chi}}}}}}"{description}, draw=none, from=8, to=10]
	\arrow["{{{{\SIGMA^{{\chi}}}}}}"{description}, draw=none, from=9, to=11]
	\arrow["{{{{\SIGMA^{\bar{\xi}_2}}}}}"{marking, allow upside down}, shift right, draw=none, from=12, to=13]
	\arrow["{{{{\SIGMA^{{\xi}_2}}}}}"{marking, allow upside down}, shift right, draw=none, from=14, to=15]
\end{tikzcd}
\end{equation}
Using the equalities \eqref{eq:(A12)}, \eqref{eq:(A14)} and \eqref{eq:(A15)}, we see that the pasting of the $\Sigma$-square
$$\begin{tikzcd}
	B & X \\
	Y & T
	\arrow["{x_3}", from=1-1, to=1-2]
	\arrow[""{name=0, anchor=center, inner sep=0}, "{y_1g_1}"', from=1-1, to=2-1]
	\arrow[""{name=1, anchor=center, inner sep=0}, "{\bar{p}\bar{y}_1}", from=1-2, to=2-2]
	\arrow["t", from=2-1, to=2-2]
	\arrow["\SIGMA^{\bar{\chi}\odot \bar{\xi}_1}"{description}, draw=none, from=0, to=1]
\end{tikzcd} $$
with $\theta_2^{-1}\cdot(p\circ \beta')\cdot \theta_1$ is equal to its pasting with $\bar{p}\circ \bar{\beta}$.
	Using Equification on these, we obtain a $\Sigma$-square
	\[\begin{tikzcd}
	Y & T \\
	Y & Q
	\arrow["t", from=1-1, to=1-2]
	\arrow[""{name=0, anchor=center, inner sep=0}, Rightarrow, no head, from=1-1, to=2-1]
	\arrow[""{name=1, anchor=center, inner sep=0}, "q", from=1-2, to=2-2]
	\arrow["u"', from=2-1, to=2-2]
	\arrow["{\SIGMA^{\psi}}"{description}, draw=none, from=0, to=1]
	\end{tikzcd}\]
	such that  $q\circ \big( \theta_2^{-1}\cdot(p\circ \beta')\cdot \theta_1\big) =q\circ \bar{p}\circ \bar{\beta}$.
	Thus, the 2-cell $[\bar{\beta}\circ \alpha, 1_{\bar{V}}, 1_{\bar{V}}]$ is the vertical composite of the three composable 2-cells represented in the following diagram.

\[\begin{tikzcd}
	A & {I_1} & X & {\bar{V}} & Y & C \\
	&&& Q & Y & C \\
	A & {I_1} & X & V & Y & C \\
	&&& V & Y & C \\
	A & {I_2} & X & V & Y & C \\
	&&& Q & Y & C \\
	A & {I_2} & X & {\bar{V}} & Y & C
	\arrow["{{{f_1}}}", from=1-1, to=1-2]
	\arrow[equals, from=1-1, to=3-1]
	\arrow["{{{x_1}}}", from=1-2, to=1-3]
	\arrow["{{{\bar{y}_1}}}", from=1-3, to=1-4]
	\arrow[equals, from=1-3, to=3-3]
	\arrow[""{name=0, anchor=center, inner sep=0}, "{{{q\bar{p}}}}"', from=1-4, to=2-4]
	\arrow["{{{q\theta_1}}}"{description}, between={0.3}{0.8}, Rightarrow, from=1-4, to=3-3]
	\arrow["{{{\bar{v}}}}"', from=1-5, to=1-4]
	\arrow[""{name=1, anchor=center, inner sep=0}, equals, from=1-5, to=2-5]
	\arrow["{{{y_3}}}"', from=1-6, to=1-5]
	\arrow[""{name=2, anchor=center, inner sep=0}, equals, from=1-6, to=2-6]
	\arrow["u"', from=2-5, to=2-4]
	\arrow[""{name=3, anchor=center, inner sep=0}, equals, from=2-5, to=3-5]
	\arrow["{{{y_3}}}"', from=2-6, to=2-5]
	\arrow[""{name=4, anchor=center, inner sep=0}, equals, from=2-6, to=3-6]
	\arrow["{{{{f_1}}}}", from=3-1, to=3-2]
	\arrow[equals, from=3-1, to=5-1]
	\arrow["{{{{{x_1}}}}}", from=3-2, to=3-3]
	\arrow["{{{{{y_1'}}}}}", from=3-3, to=3-4]
	\arrow["\alpha"', between={0.2}{0.8}, Rightarrow, from=3-3, to=5-1]
	\arrow[equals, from=3-3, to=5-3]
	\arrow[""{name=5, anchor=center, inner sep=0}, "qp", from=3-4, to=2-4]
	\arrow["v", tail reversed, no head, from=3-4, to=3-5]
	\arrow[""{name=6, anchor=center, inner sep=0}, equals, from=3-4, to=4-4]
	\arrow["{{{{{\beta'}}}}}"', between={0.3}{0.8}, Rightarrow, from=3-4, to=5-3]
	\arrow["{{{{y_3}}}}", tail reversed, no head, from=3-5, to=3-6]
	\arrow[""{name=7, anchor=center, inner sep=0}, equals, from=3-6, to=4-6]
	\arrow["v", from=4-5, to=4-4]
	\arrow["{{{{y_3}}}}", from=4-6, to=4-5]
	\arrow[""{name=8, anchor=center, inner sep=0}, equals, from=4-6, to=5-6]
	\arrow["{{{{f_2}}}}"', from=5-1, to=5-2]
	\arrow[equals, from=5-1, to=7-1]
	\arrow["{{{{{x_2}}}}}"', from=5-2, to=5-3]
	\arrow["{{{{{y'_2}}}}}"', from=5-3, to=5-4]
	\arrow[equals, from=5-3, to=7-3]
	\arrow[""{name=9, anchor=center, inner sep=0}, equals, from=5-4, to=4-4]
	\arrow[""{name=10, anchor=center, inner sep=0}, "qp"', from=5-4, to=6-4]
	\arrow["{{q\theta_2^{-1}}}"{description}, between={0.3}{0.8}, Rightarrow, from=5-4, to=7-3]
	\arrow["v", from=5-5, to=5-4]
	\arrow[""{name=11, anchor=center, inner sep=0}, equals, from=5-5, to=6-5]
	\arrow["{{{{y_3}}}}", from=5-6, to=5-5]
	\arrow[""{name=12, anchor=center, inner sep=0}, equals, from=5-6, to=6-6]
	\arrow["u", from=6-5, to=6-4]
	\arrow[""{name=13, anchor=center, inner sep=0}, equals, from=6-5, to=7-5]
	\arrow["{{{y_3}}}", from=6-6, to=6-5]
	\arrow[""{name=14, anchor=center, inner sep=0}, equals, from=6-6, to=7-6]
	\arrow["{{{f_2}}}"', from=7-1, to=7-2]
	\arrow["{{{x_2}}}"', from=7-2, to=7-3]
	\arrow["{{{\bar{y}_2}}}"', from=7-3, to=7-4]
	\arrow[""{name=15, anchor=center, inner sep=0}, "{{{q\bar{p}}}}", from=7-4, to=6-4]
	\arrow["{{{\bar{v}}}}", from=7-5, to=7-4]
	\arrow["{{{y_3}}}", from=7-6, to=7-5]
	\arrow["{{\SIGMA^{\psi\odot\bar{\chi}}}}"{description}, draw=none, from=0, to=1]
	\arrow["{{{\SIGMA^{\id}}}}"{description}, draw=none, from=1, to=2]
	\arrow["{{{\SIGMA^{\id}}}}"{description}, draw=none, from=3, to=4]
	\arrow["{{\SIGMA^{\psi\odot\chi}}}"{description}, draw=none, from=5, to=3]
	\arrow["{{{{\SIGMA^{\id}}}}}"{description}, draw=none, from=6, to=7]
	\arrow["{{{{\SIGMA^{\id}}}}}"{description}, draw=none, from=9, to=8]
	\arrow["{{\SIGMA^{\psi\odot\chi}}}"{description}, draw=none, from=10, to=11]
	\arrow["{{{\SIGMA^{\id}}}}"{description}, draw=none, from=11, to=12]
	\arrow["{{{\SIGMA^{\id}}}}"{description}, draw=none, from=13, to=14]
	\arrow["{{\SIGMA^{\psi\odot\bar{\chi}}}}"{description}, draw=none, from=15, to=13]
\end{tikzcd}\]
The top and the bottom 2-morphisms of this diagram are of $\Omega$ type. More precisely, we have the $\Omega$ 2-cells
\[\begin{tikzcd}
	{\tilde{\Omega}_i=[q\theta_ix_if_i,qp, q\bar{p}]\colon (\bar{y}_ix_if_i,\bar{v}y_3)} & {(y^{\prime}_ix_if_i,vy_3)}
	\arrow[Rightarrow, from=1-1, to=1-2]
\end{tikzcd}\]
whose inverses $\tilde{\Omega}_i^{-1}$ correspond to the $\Sg$-step
\[\begin{tikzcd}
	& B & {I_i} &&& B & {I_i} \\
	C & {J_i} & {} && C & {J_i} \\
	C & Y & V && C & Y & {\bar{V}}
	\arrow["{{{{{{r_i}}}}}}", from=1-2, to=1-3]
	\arrow["{{{{{{g_i}}}}}}"', from=1-2, to=2-2]
	\arrow[""{name=0, anchor=center, inner sep=0}, "{{y'_ix_i}}", from=1-3, to=3-3]
	\arrow["{{r_i}}", from=1-6, to=1-7]
	\arrow["{{g_i}}"', from=1-6, to=2-6]
	\arrow[""{name=1, anchor=center, inner sep=0}, "{{\bar{y}_ix_i}}", from=1-7, to=3-7]
	\arrow["{{{{{{s_i}}}}}}", from=2-1, to=2-2]
	\arrow[""{name=2, anchor=center, inner sep=0}, equals, from=2-1, to=3-1]
	\arrow[""{name=3, anchor=center, inner sep=0}, "{{{{{{y_i}}}}}}", from=2-2, to=3-2]
	\arrow["u"{pos=0.65}, between={0.4}{0.9}, squiggly, from=2-3, to=2-5]
	\arrow["{{s_i}}", from=2-5, to=2-6]
	\arrow[""{name=4, anchor=center, inner sep=0}, equals, from=2-5, to=3-5]
	\arrow[""{name=5, anchor=center, inner sep=0}, "{{y_i}}", from=2-6, to=3-6]
	\arrow["{{{{{{y_3}}}}}}"', from=3-1, to=3-2]
	\arrow["v"', from=3-2, to=3-3]
	\arrow["{{y_3}}"', from=3-5, to=3-6]
	\arrow["{{\bar{v}}}"', from=3-6, to=3-7]
	\arrow["{{\SIGMA^{\epsilon_i}}}"{description}, draw=none, from=2, to=3]
	\arrow["{{\SIGMA^{\xi_i\odot\delta_i}}}"{description}, draw=none, from=2-2, to=0]
	\arrow["{{\SIGMA^{\epsilon_i}}}"{description}, draw=none, from=4, to=5]
	\arrow["{{\SIGMA^{\bar{\xi}_i\odot\delta_i}}}"{description}, draw=none, from=2-6, to=1]
\end{tikzcd}\, .\]
Adding this $\Sg$-step to the end of the $\Sg$-path \eqref{eq:Omega-i}, and applying \cref{lem:Sigma-steps}(1), we obtain a $\Sg$-path of the form
$\begin{tikzcd}
	{\bullet} & {\bullet} & {\bullet}
	\arrow["d", squiggly, from=1-1, to=1-2]
	\arrow["u", squiggly, from=1-2, to=1-3]
\end{tikzcd}$
which determines, for $i=1,2$, the $\Omega$ 2-cells	
	\[\begin{tikzcd}
	{\bar{\Omega}_i\colon (g_i,s_i)\circ (f_i,r_i)} & {(\bar{y}_ix_if_i,\bar{v}y_3)}
	\arrow[Rightarrow, from=1-1, to=1-2]
	\end{tikzcd}\]
required in Definition~\ref{def:horizontal}.
Hence, $\bar{\Omega}_i = \tilde{\Omega}_i^{-1} \cdot \Omega_i$.
Now since $[\bar{\beta}\circ \alpha, 1_{\bar{V}}, 1_{\bar{V}}] = \tilde{\Omega}_2^{-1} \cdot [\beta'\circ \alpha, 1_{V}, 1_{V}] \cdot \tilde{\Omega}_1$, we can conclude that
$$\begin{array}{rl}
\bar{\Omega}_2^{-1}\cdot [\bar{\beta}\circ \alpha,1_{\bar{V}}, 1_{\bar{V}}]\cdot \bar{\Omega}_1
&=\Omega_2^{-1}\cdot [\beta'\circ \alpha,1_V,1_V]\cdot \Omega_1,
\end{array}$$
as desired.

\noindent (2) Secondly, we show that this composition respects the $\approx$-relation. That is, given $(\hat{\alpha},\hat{x}_1,\hat{x}_2)\approx (\alpha,x_1,x_2)\colon (f_1,r_1)\Rightarrow (f_2,r_2)\colon A\to B$ and $(\hat{\beta},\hat{y}_1,\hat{y}_2)\approx (\beta,y_1,y_2)\colon (g_1,s_1)\Rightarrow (g_2,s_2)\colon B\to C$, we have $(\hat{\beta},\hat{y}_1,\hat{y}_2)\circ (\hat{\alpha},\hat{x}_1,\hat{x}_2)\approx (\beta,y_1,y_2)\circ (\alpha,x_1,x_2)$.
By transitivity it is enough to consider the following two special cases:
\begin{enumerate}
	\item[(2a)] \ $(\hat{\alpha},\hat{x}_1,\hat{x}_2)$ is a $\Sigma$-extension of  $(\alpha,x_1,x_2)$ and $(\beta,y_1,y_2)=(\hat{\beta},\hat{y}_1,\hat{y}_2)$;
	\item[(2b)] \ $(\hat{\beta},\hat{y}_1,\hat{y}_2)$ is a $\Sigma$-extension of $(\beta,y_1,y_2)$ and $(\hat{\alpha},\hat{x}_1,\hat{x}_2)=(\alpha,x_1,x_2)$.
\end{enumerate}

\noindent (2a) In the first case, let $(\hat{\alpha},\hat{x}_1,\hat{x}_2)$ be a $\Sigma$-extension of $(\alpha,x_1,x_2)$ by means of the following equality:
\[\begin{tikzcd}
	A && {I_1} && B \\
	& X && X & B & A & {I_1} & B \\
	&& D && B & {} & D & B \\
	& X && X & B & A & {I_2} & B \\
	A && {I_2} && B \rlap{\,.}
	\arrow[""{name=0, anchor=center, inner sep=0}, "{f_1}", squiggly, from=1-1, to=1-3]
	\arrow[Rightarrow, squiggly, no head, from=1-1, to=5-1]
	\arrow["{{{{{x_1}}}}}"{description}, curve={height=6pt}, from=1-3, to=2-2]
	\arrow[""{name=1, anchor=center, inner sep=0}, "{{{{{x_1}}}}}"{description}, curve={height=-6pt}, from=1-3, to=2-4]
	\arrow[""{name=2, anchor=center, inner sep=0}, "{{\hat{x}_1}}"{description}, squiggly, from=1-3, to=3-3]
	\arrow["{{{{{r_1}}}}}"', squiggly, from=1-5, to=1-3]
	\arrow[""{name=3, anchor=center, inner sep=0}, Rightarrow, squiggly, no head, from=1-5, to=2-5]
	\arrow["d"{description}, curve={height=6pt}, from=2-2, to=3-3]
	\arrow[""{name=4, anchor=center, inner sep=0}, "d"{description}, curve={height=-6pt}, from=2-4, to=3-3]
	\arrow["{{{{{x_3}}}}}"', from=2-5, to=2-4]
	\arrow[""{name=5, anchor=center, inner sep=0}, Rightarrow, squiggly, no head, from=2-5, to=3-5]
	\arrow["{f_1}", squiggly, from=2-6, to=2-7]
	\arrow[Rightarrow, squiggly, no head, from=2-6, to=4-6]
	\arrow[""{name=6, anchor=center, inner sep=0}, "{\hat{x}_1}"', squiggly, from=2-7, to=3-7]
	\arrow["{\hat{\alpha}}"{description}, shorten <=11pt, shorten >=11pt, Rightarrow, from=2-7, to=4-6]
	\arrow["{r_1}"', squiggly, from=2-8, to=2-7]
	\arrow[""{name=7, anchor=center, inner sep=0}, Rightarrow, squiggly, no head, from=2-8, to=3-8]
	\arrow["u"', squiggly, from=3-5, to=3-3]
	\arrow["{\text{\normalsize =}}"{description}, draw=none, from=3-5, to=3-6]
	\arrow[""{name=8, anchor=center, inner sep=0}, Rightarrow, squiggly, no head, from=3-5, to=4-5]
	\arrow["u"', squiggly, from=3-8, to=3-7]
	\arrow[""{name=9, anchor=center, inner sep=0}, Rightarrow, squiggly, no head, from=3-8, to=4-8]
	\arrow["d"{description}, curve={height=-6pt}, from=4-2, to=3-3]
	\arrow[""{name=10, anchor=center, inner sep=0}, "d"{description}, curve={height=6pt}, from=4-4, to=3-3]
	\arrow["{{{{{x_3}}}}}"', from=4-5, to=4-4]
	\arrow[""{name=11, anchor=center, inner sep=0}, Rightarrow, squiggly, no head, from=4-5, to=5-5]
	\arrow["{f_2}"', squiggly, from=4-6, to=4-7]
	\arrow[""{name=12, anchor=center, inner sep=0}, "{\hat{x}_2}", squiggly, from=4-7, to=3-7]
	\arrow["{r_2}", squiggly, from=4-8, to=4-7]
	\arrow[""{name=13, anchor=center, inner sep=0}, "{{f_2}}"', squiggly, from=5-1, to=5-3]
	\arrow[""{name=14, anchor=center, inner sep=0}, "{{\hat{x}_2}}"{description}, squiggly, from=5-3, to=3-3]
	\arrow["{{{{{x_2}}}}}"{description}, curve={height=-6pt}, from=5-3, to=4-2]
	\arrow[""{name=15, anchor=center, inner sep=0}, "{{{{{x_2}}}}}"{description}, curve={height=6pt}, from=5-3, to=4-4]
	\arrow["{{{{{r_2}}}}}", squiggly, from=5-5, to=5-3]
	\arrow["{{d\circ \alpha}}"', shift right=5, shorten <=17pt, shorten >=17pt, Rightarrow, from=0, to=13]
	\arrow["{{\theta_1^{-1}}}"', shorten <=7pt, Rightarrow, from=2, to=2-2]
	\arrow["{{{{{\SIGMA^{\delta_1}}}}}}"{description}, draw=none, from=1, to=3]
	\arrow["{{\theta_1}}"', shorten >=5pt, Rightarrow, from=2-4, to=2]
	\arrow["{\SIGMA^{\phi}}"{description}, draw=none, from=4, to=5]
	\arrow["{\SIGMA^{\hat{\delta}_1}}"{description}, draw=none, from=6, to=7]
	\arrow["{{\theta_2}}", shorten >=7pt, Rightarrow, from=4-2, to=14]
	\arrow["{{\theta_2}}"', shorten >=7pt, Rightarrow, from=4-4, to=14]
	\arrow["{\SIGMA^{\phi}}"{description}, draw=none, from=10, to=8]
	\arrow["{\SIGMA^{\hat{\delta}_2}}"{description}, draw=none, from=12, to=9]
	\arrow["{{{{{\SIGMA^{\delta_2}}}}}}"{description}, draw=none, from=15, to=11]
\end{tikzcd}\]

In order to get $(\beta,y_1,y_2)\circ (\hat{\alpha},\hat{x}_1,\hat{x}_2)$, we apply Rule 6:
\[\begin{tikzcd}
	B & D &&& B & D \\
	{J_1} &&& {J_1} & {J_2} \\
	Y & V &&& Y & V \rlap{\,.}
	\arrow["u", from=1-1, to=1-2]
	\arrow["{{{g_1}}}"', from=1-1, to=2-1]
	\arrow[""{name=0, anchor=center, inner sep=0}, "{{y''_1}}", from=1-2, to=3-2]
	\arrow[""{name=1, anchor=center, inner sep=0}, "{{y''_2}}", curve={height=-40pt}, from=1-2, to=3-2]
	\arrow["u", from=1-5, to=1-6]
	\arrow["{{{g_1}}}"', curve={height=6pt}, from=1-5, to=2-4]
	\arrow["{{{g_2}}}", from=1-5, to=2-5]
	\arrow[""{name=2, anchor=center, inner sep=0}, "{{y''_2}}", from=1-6, to=3-6]
	\arrow["{{{y_1}}}"', from=2-1, to=3-1]
	\arrow["\beta", shorten <=4pt, shorten >=4pt, Rightarrow, from=2-4, to=2-5]
	\arrow["{{{y_1}}}"', curve={height=6pt}, from=2-4, to=3-5]
	\arrow["{{{y_2}}}", from=2-5, to=3-5]
	\arrow["v", from=3-1, to=3-2]
	\arrow["v", from=3-5, to=3-6]
	\arrow["{\text{\normalsize =}}"{description, pos=0.7}, draw=none, from=1, to=2-4]
	\arrow["{\beta^{\prime\prime}}"{pos=0.6}, shorten <=10pt, shorten >=4pt, Rightarrow, from=0, to=1]
	\arrow["{{{\SIGMA^{\xi'_1}}}}"{description}, draw=none, from=2-1, to=0]
	\arrow["{{{\SIGMA^{\xi'_2}}}}"{description}, draw=none, from=2-5, to=2]
\end{tikzcd}\]
By considering
$$y'_i=y^{\prime\prime}_id, \;\xi_i=\xi^{\prime}_i\odot\phi \text{ (for $i=1,2$)},  \; \text{and}\;  \beta'=\beta^{\prime\prime}\circ d$$ we obtain the data to define $[\beta,y_1,y_2]\circ [\alpha, x_1,x_2]$, as in \eqref{eq:(A12)}.
Let
$$\begin{array}{l}\Omega_i\colon (g_i,s_i)\circ (f_i,r_i)\Rightarrow(y'_ix_if_i, vy_3)\text{ \hskip2mm and}\\
	\hat{\Omega}_i\colon (g_i,s_i)\circ (f_i,r_i)\Rightarrow(y^{\prime\prime}_i\hat{x}_if_i, vy_3)
\end{array}$$
be the $\Omega$ 2-cells obtained as in Definition \ref{def:horizontal} to arrive at $[\beta,y_1,y_2]\circ [\alpha, x_1,x_2]$
 and $[\beta,y_1,y_2]\circ [\hat{\alpha}, \hat{x}_1,\hat{x}_2]$, respectively.
 In particular, $\hat{\Omega}_i$ is the $\Omega$ 2-cell corresponding to the $\Sg$-path
\begin{equation}\label{eq:hat-Omega_i}
\begin{tikzcd}
	& {} & {} &&& {} & {} &&& {} & {} \\
	{} & {} & {} && {} & {} & {} && {} & {} & {} \\
	{} & {} & {} && {} & {} & {} && {} & {} & {}
	\arrow["{{r_i}}", from=1-2, to=1-3]
	\arrow[""{name=0, anchor=center, inner sep=0}, "{{g_i}}"', from=1-2, to=2-2]
	\arrow[""{name=1, anchor=center, inner sep=0}, from=1-3, to=2-3]
	\arrow["{{r_i}}", from=1-6, to=1-7]
	\arrow[""{name=2, anchor=center, inner sep=0}, "{{g_i}}"', from=1-6, to=2-6]
	\arrow[""{name=3, anchor=center, inner sep=0}, from=1-7, to=2-7]
	\arrow["{{r_i}}", from=1-10, to=1-11]
	\arrow["{{g_i}}"', from=1-10, to=2-10]
	\arrow["{{y''_i\hat{x}_i}}", from=1-11, to=3-11]
	\arrow["{{s_i}}", from=2-1, to=2-2]
	\arrow[""{name=4, anchor=center, inner sep=0}, equals, from=2-1, to=3-1]
	\arrow[from=2-2, to=2-3]
	\arrow[""{name=5, anchor=center, inner sep=0}, equals, from=2-2, to=3-2]
	\arrow["d", squiggly, between={0.2}{0.8}, from=2-3, to=2-5]
	\arrow[""{name=6, anchor=center, inner sep=0}, equals, from=2-3, to=3-3]
	\arrow["{{s_i}}", from=2-5, to=2-6]
	\arrow[""{name=7, anchor=center, inner sep=0}, equals, from=2-5, to=3-5]
	\arrow[from=2-6, to=2-7]
	\arrow[""{name=8, anchor=center, inner sep=0}, "{{y_i}}"', from=2-6, to=3-6]
	\arrow["u", squiggly, between={0.2}{0.8}, from=2-7, to=2-9]
	\arrow[""{name=9, anchor=center, inner sep=0}, from=2-7, to=3-7]
	\arrow["{{s_i}}", from=2-9, to=2-10]
	\arrow[""{name=10, anchor=center, inner sep=0}, equals, from=2-9, to=3-9]
	\arrow["\SIGMA"{description}, draw=none, from=2-10, to=2-11]
	\arrow[""{name=11, anchor=center, inner sep=0}, "{{y_i}}", from=2-10, to=3-10]
	\arrow["{{s_i}}", from=3-1, to=3-2]
	\arrow[from=3-2, to=3-3]
	\arrow["{{y_3}}"', from=3-5, to=3-6]
	\arrow[from=3-6, to=3-7]
	\arrow["{{y_3}}"', from=3-9, to=3-10]
	\arrow["v"', from=3-10, to=3-11]
	\arrow["\dS"{description}, draw=none, from=0, to=1]
	\arrow["\dS"{description}, draw=none, from=2, to=3]
	\arrow["{{\SIGMA^{\id}}}"{description}, draw=none, from=4, to=5]
	\arrow["{{\SIGMA^{\id}}}"{description}, draw=none, from=5, to=6]
	\arrow["\SIGMA"{description}, draw=none, from=7, to=8]
	\arrow["\dS"{description}, draw=none, from=8, to=9]
	\arrow["\SIGMA"{description}, draw=none, from=10, to=11]
\end{tikzcd}\; .
\end{equation}
Observe that $\tilde{\Omega}_i=[y''_i\circ \theta_i^{-1}\circ f_i,1,1]\colon (y''_i\hat{x}_if_i,vy_3)\Rightarrow (y'_ix_if_i,vy_3)$ is a basic $\Omega$ 2-cell produced by a $\Sg$-step of type $u$. Hence, using Lemma \ref{lem:Sigma-steps}(1), we see that $\tilde{\Omega}_i\cdot \hat{\Omega}_i=\Omega_i$.

Consider the following diagram, which represents the composition $[\beta,y_1,y_2]\circ [\hat{\alpha},\hat{x}_1,\hat{x}_2]$.

\[\begin{tikzcd}
	A & {I_1} &&& {\dot{B}_1} & {J_1} & C \\
	\\
	A & {I_1} & X & D & V & Y & C \\
	&&&& V & Y & C \\
	A & {I_2} & X & D & V & Y & C \\
	\\
	A & {I_2} &&& {\dot{B}_2} & {J_2} & C
	\arrow["{{f_1}}"', from=1-1, to=1-2]
	\arrow[Rightarrow, no head, from=1-1, to=3-1]
	\arrow["{{\dot{g}_1}}"', from=1-2, to=1-5]
	\arrow["{{\hat{\Omega}_1}}", shorten <=6pt, shorten >=6pt, Rightarrow, from=1-5, to=3-5]
	\arrow["{{\dot{r}_1}}", from=1-6, to=1-5]
	\arrow["{{s_1}}", from=1-7, to=1-6]
	\arrow[Rightarrow, no head, from=1-7, to=3-7]
	\arrow["{{{f_1}}}", from=3-1, to=3-2]
	\arrow[Rightarrow, no head, from=3-1, to=5-1]
	\arrow["{{{x_1}}}", from=3-2, to=3-3]
	\arrow[""{name=0, anchor=center, inner sep=0}, "{{\hat{x}_1}}", curve={height=-30pt}, from=3-2, to=3-4]
	\arrow["d", from=3-3, to=3-4]
	\arrow["{{{y^{\prime\prime}_1}}}", from=3-4, to=3-5]
	\arrow["{{d\circ \alpha}}"', shorten <=25pt, shorten >=25pt, Rightarrow, from=3-4, to=5-1]
	\arrow[Rightarrow, no head, from=3-4, to=5-4]
	\arrow[""{name=1, anchor=center, inner sep=0}, Rightarrow, no head, from=3-5, to=4-5]
	\arrow["{{\beta^{\prime\prime}}}"', Rightarrow, from=3-5, to=5-4]
	\arrow["v"', from=3-6, to=3-5]
	\arrow["{{{y_3}}}"', from=3-7, to=3-6]
	\arrow[""{name=2, anchor=center, inner sep=0}, Rightarrow, no head, from=3-7, to=4-7]
	\arrow[""{name=3, anchor=center, inner sep=0}, Rightarrow, no head, from=4-5, to=5-5]
	\arrow["v"', from=4-6, to=4-5]
	\arrow["{{{y_3}}}"', from=4-7, to=4-6]
	\arrow[""{name=4, anchor=center, inner sep=0}, Rightarrow, no head, from=4-7, to=5-7]
	\arrow["{{{f_2}}}"', from=5-1, to=5-2]
	\arrow[Rightarrow, no head, from=5-1, to=7-1]
	\arrow["{{{x_2}}}"', from=5-2, to=5-3]
	\arrow[""{name=5, anchor=center, inner sep=0}, "{{\hat{x}_2}}"', curve={height=30pt}, from=5-2, to=5-4]
	\arrow["d"', from=5-3, to=5-4]
	\arrow["{{{y^{\prime\prime}_2}}}"', from=5-4, to=5-5]
	\arrow["{\hat{\Omega}^{-1}_2}", shorten <=6pt, shorten >=6pt, Rightarrow, from=5-5, to=7-5]
	\arrow["v", from=5-6, to=5-5]
	\arrow["{{{y_3}}}", from=5-7, to=5-6]
	\arrow[Rightarrow, no head, from=5-7, to=7-7]
	\arrow["{{f_2}}", from=7-1, to=7-2]
	\arrow["{{\dot{g}_2}}", from=7-2, to=7-5]
	\arrow["{{\dot{r}_2}}"', from=7-6, to=7-5]
	\arrow["{{s_2}}"', from=7-7, to=7-6]
	\arrow["{{\theta_1^{-1}}}", shorten <=3pt, Rightarrow, from=0, to=3-3]
	\arrow["{{{\SIGMA^{\id}}}}"{description}, draw=none, from=1, to=2]
	\arrow["{{{\SIGMA^{\id}}}}"{description}, draw=none, from=3, to=4]
	\arrow["{{\theta_2}}", shorten >=3pt, Rightarrow, from=5-3, to=5]
\end{tikzcd}\]
We obtain
\begin{align*}
[(\beta,y_1,y_2)\circ (\hat{\alpha},\hat{x}_1,\hat{x}_2)]&=\hat{\Omega}_2^{-1}\cdot[\beta^{\prime\prime}\circ\hat{\alpha},1_V,1_V]\cdot \hat{\Omega}_1          \\
&=\hat{\Omega}_2^{-1}\cdot[\beta^{\prime\prime}\circ\big((\theta_2\circ f_2)\cdot (d\circ \alpha)\cdot (\theta_1^{-1}\circ f_1)\big),1_V,1_V])\cdot \hat{\Omega}_1          \\
&=\hat{\Omega}_2^{-1}\cdot\tilde{\Omega}^{-1}_2\cdot [\beta'\circ \alpha,1_V,1_V]\cdot \tilde{\Omega}_1\cdot \hat{\Omega}_1       \\
&=
\Omega_2^{-1}\cdot [\beta'\circ \alpha,1_V,1_V]\cdot \Omega_1\\
&=[(\beta,y_1,y_2)\circ (\alpha,x_1,x_2)],
\end{align*}
as desired.

\noindent (2b) For the second case, let $(\hat{\beta},\hat{y}_1, \hat{y}_2)$ be a $\Sigma$-extension of $(\beta,y_1,y_2)$ through the following equalities, where the 2-cells $\theta_i$ are invertible:

\begin{equation}\label{eq:D1}
\begin{tikzcd}
	C & {J_1} \\
	C & Y && C & {J_1} \\
	C & {\hat{Y}} && C & {\hat{Y}} \\
	C & Y && C & {J_2} \\
	C & {J_2}
	\arrow["{{{s_1}}}", from=1-1, to=1-2]
	\arrow[""{name=0, anchor=center, inner sep=0}, equals, from=1-1, to=2-1]
	\arrow[""{name=1, anchor=center, inner sep=0}, "{{{y_1}}}", from=1-2, to=2-2]
	\arrow[""{name=2, anchor=center, inner sep=0}, "{{{\hat{y}_1}}}", curve={height=-24pt}, from=1-2, to=3-2]
	\arrow["{{{y_3}}}"', from=2-1, to=2-2]
	\arrow[""{name=3, anchor=center, inner sep=0}, equals, from=2-1, to=3-1]
	\arrow[""{name=4, anchor=center, inner sep=0}, "d", from=2-2, to=3-2]
	\arrow["{{{s_1}}}", from=2-4, to=2-5]
	\arrow[""{name=5, anchor=center, inner sep=0}, equals, from=2-4, to=3-4]
	\arrow[""{name=6, anchor=center, inner sep=0}, "{{{\hat{y}_1}}}", from=2-5, to=3-5]
	\arrow["{{\hat{y}_3}}"{description}, from=3-1, to=3-2]
	\arrow[""{name=7, anchor=center, inner sep=0}, equals, from=3-1, to=4-1]
	\arrow["{{{\text{\normalsize =}}}}"{description, pos=0.7}, draw=none, from=3-2, to=3-4]
	\arrow["{{{\hat{y}_3}}}", from=3-4, to=3-5]
	\arrow[""{name=8, anchor=center, inner sep=0}, equals, from=3-4, to=4-4]
	\arrow["{{{y_3}}}"', from=4-1, to=4-2]
	\arrow[""{name=9, anchor=center, inner sep=0}, equals, from=4-1, to=5-1]
	\arrow[""{name=10, anchor=center, inner sep=0}, "d"', from=4-2, to=3-2]
	\arrow["{{{s_2}}}"', from=4-4, to=4-5]
	\arrow[""{name=11, anchor=center, inner sep=0}, "{{{\hat{y}_2}}}"', from=4-5, to=3-5]
	\arrow["{{{s_2}}}"', from=5-1, to=5-2]
	\arrow[""{name=12, anchor=center, inner sep=0}, "{{{\hat{y}_2}}}"', curve={height=24pt}, from=5-2, to=3-2]
	\arrow[""{name=13, anchor=center, inner sep=0}, "{{{y_2}}}"', from=5-2, to=4-2]
	\arrow["{{{\SIGMA^{\epsilon_1}}}}"{description}, draw=none, from=0, to=1]
	\arrow["{{\SIGMA^{\lambda}}}"{description}, draw=none, from=3, to=4]
	\arrow["{{{\theta_1}}}"'{pos=0.4}, shorten >=3pt, Rightarrow, from=2-2, to=2]
	\arrow["{{{\SIGMA^{\hat{\epsilon}_1}}}}"{description}, draw=none, from=5, to=6]
	\arrow["{{\SIGMA^{\lambda}}}"{description}, draw=none, from=7, to=10]
	\arrow["{{{\SIGMA^{\hat{\epsilon}_2}}}}"{description}, draw=none, from=8, to=11]
	\arrow["{{{\SIGMA^{\epsilon_2}}}}"{description}, draw=none, from=9, to=13]
	\arrow["{{{\theta_2}}}"'{pos=0.4}, shorten >=3pt, Rightarrow, from=4-2, to=12]
\end{tikzcd}\, ; \ \
\hat{\beta}=(\theta_2\circ g_2)\cdot (d\circ \beta)\cdot (\theta_1^{-1}\circ g_1).
\end{equation}

Consider the $\Sigma$-squares $\Sigma^{\xi_i}$ of Equation~\eqref{eq:(A12)} of the definition of horizontal composition for $(\alpha,x_1,x_2)$ and $(\beta,y_1,y_2)$, and use Square and the 2-cells $\theta_i$ to get
\[\begin{tikzcd}
& B & X \\
{J_i} & Y & V \\
& {\hat{Y}} & Q
\arrow["{x_3}", from=1-2, to=1-3]
\arrow["{g_i}"', curve={height=6pt}, from=1-2, to=2-1]
\arrow[""{name=0, anchor=center, inner sep=0}, "{y_ig_i}"', from=1-2, to=2-2]
\arrow[""{name=1, anchor=center, inner sep=0}, "{y'_i}", from=1-3, to=2-3]
\arrow["{\theta_i^{-1}\circ g_i}", Rightarrow, from=2-1, to=2-2]
\arrow["{\hat{y}_i}"', curve={height=6pt}, from=2-1, to=3-2]
\arrow["v"', from=2-2, to=2-3]
\arrow[""{name=2, anchor=center, inner sep=0}, "d"', from=2-2, to=3-2]
\arrow[""{name=3, anchor=center, inner sep=0}, "{d'}", from=2-3, to=3-3]
\arrow["u"', from=3-2, to=3-3]
\arrow["{\SIGMA^{\xi_i}}"{description}, draw=none, from=0, to=1]
\arrow["{\SIGMA^{\mu}}"{description}, draw=none, from=2, to=3]
\end{tikzcd}\, .\]
By Horizontal Repletion, we have the $\Sigma$-square
\[\begin{tikzcd}
	B & B \\
	{\hat{Y}} & {\hat{Y}}
	\arrow[equals, from=1-1, to=1-2]
	\arrow[""{name=0, anchor=center, inner sep=0}, "{\hat{y}_ig_i}"', from=1-1, to=2-1]
	\arrow[""{name=1, anchor=center, inner sep=0}, "{dy_ig_i}", from=1-2, to=2-2]
	\arrow[equals, from=2-1, to=2-2]
	\arrow["{\SIGMA^{\theta_i^{-1}g_i}}"{description}, draw=none, from=0, to=1]
\end{tikzcd}\]
Put
$$\nu_i=(\theta_i^{-1}g_i)\oplus (\mu\odot \xi_i),\;
y^{\prime\prime}_i=d'y'_i\;  \text{and}\;  \beta^{\prime\prime}=d'\circ \beta'.$$
A simple calculation shows that
\[\begin{tikzcd}
	B & X && B & X \\
	{\hat{Y}} & Q && {\hat{Y}} & Q
	\arrow["{{{{{{x_3}}}}}}", from=1-1, to=1-2]
	\arrow[""{name=0, anchor=center, inner sep=0}, "{{{{\hat{y}_2g_2}}}}"{description}, from=1-1, to=2-1]
	\arrow[""{name=1, anchor=center, inner sep=0}, "{{{{\hat{y}_1g_1}}}}"', curve={height=30pt}, from=1-1, to=2-1]
	\arrow[""{name=2, anchor=center, inner sep=0}, "{{{y^{\prime\prime}_2}}}", from=1-2, to=2-2]
	\arrow["{{{{x_3}}}}", from=1-4, to=1-5]
	\arrow[""{name=3, anchor=center, inner sep=0}, "{{{{\hat{y}_1g_1}}}}"', from=1-4, to=2-4]
	\arrow[""{name=4, anchor=center, inner sep=0}, "{{{y^{\prime\prime}_1}}}"{description}, from=1-5, to=2-5]
	\arrow[""{name=5, anchor=center, inner sep=0}, "{{{y^{\prime\prime}_2}}}", curve={height=-30pt}, from=1-5, to=2-5]
	\arrow["u"', from=2-1, to=2-2]
	\arrow["u"', from=2-4, to=2-5]
	\arrow["{{{{\hat{\beta}}}}}"'{pos=0.2}, shift left, between={0.3}{0.7}, Rightarrow, from=1, to=0]
	\arrow["{{{{\SIGMA^{\nu_2}}}}}"{description}, draw=none, from=0, to=2]
	\arrow["{{\text{\normalsize =}}}"{description, pos=0.4}, draw=none, from=2, to=3]
	\arrow["{{{{\SIGMA^{\nu_1}}}}}"{description}, draw=none, from=3, to=4]
	\arrow["{{{\beta^{\prime\prime}}}}", shift right, between={0.3}{0.7}, Rightarrow, from=4, to=5]
\end{tikzcd}\, .\]
Therefore, the $\Sigma$-squares $\Sigma^{\nu_i}$ and the 2-cell $\beta^{\prime\prime}$ satisfy the conditions corresponding to \eqref{eq:(A12)} in the process of obtaining  the horizontal composition of $(\alpha, x_1,x_2)$ with $(\hat{\beta},\hat{x}_1, \hat{x}_2)$. In this way, we get the 2-morphism in the middle of  diagram \eqref{eq:D5} below. This diagram illustrates the desired $\approx$-equivalence. Indeed, as in Definition \ref{def:horizontal}, consider
\[\begin{tikzcd}
	{\Omega_i\colon (g_i,s_i)\circ (f_i,r_i)} & {(y'_ix_if_i,vy_3)}
	\arrow[Rightarrow, from=1-1, to=1-2]
\end{tikzcd}\]
and
\[\begin{tikzcd}
	{\hat{\Omega}_i\colon (g_i,s_i)\circ (f_i,r_i)} & {(d'y'_ix_if_i,u\hat{y}_3)}
	\arrow[Rightarrow, from=1-1, to=1-2]
\end{tikzcd} \]
corresponding to the compositions $[\beta,y_1,y_2]\circ [\alpha,x_1,x_2]$ and $[\hat{\beta},\hat{y}_1,\hat{y}_2]\circ [\alpha,x_1,x_2]$, respectively.  Now let
	\[\begin{tikzcd}
	{\tilde{\Omega}_i\colon (y'_ix_if_i,vy_3)} & {(d'y'_ix_if_i,u\hat{y}_3)}
	\arrow[Rightarrow, from=1-1, to=1-2]
\end{tikzcd}\]
be the basic $\Omega$ 2-cell corresponding to the following  $\Sg$-step of type $d_1$ (see Notation \ref{nota:Sigma-steps}):
\[\begin{tikzcd}
	& {} & {} && {} & {} & {} \\
	{} & {} & {} && {} & {} & {} \\
	{} & {} & {} && {} & {} & {} \\
	{} & {} & {} && {} & {} & {}
	\arrow["{{{{r_i}}}}", from=1-2, to=1-3]
	\arrow["{{{{g_i}}}}"', from=1-2, to=2-2]
	\arrow[draw=none, from=1-3, to=2-3]
	\arrow["{{{{y'_ix_i}}}}", from=1-3, to=3-3]
	\arrow[draw=none, from=1-5, to=2-5]
	\arrow["{{{{r_i}}}}", from=1-6, to=1-7]
	\arrow["{{{{g_i}}}}"', from=1-6, to=2-6]
	\arrow["{{{{y'_ix_i}}}}", from=1-7, to=3-7]
	\arrow["{{{{s_i}}}}", from=2-1, to=2-2]
	\arrow[""{name=0, anchor=center, inner sep=0}, equals, from=2-1, to=3-1]
	\arrow["\SIGMA"{description}, draw=none, from=2-2, to=2-3]
	\arrow[""{name=1, anchor=center, inner sep=0}, "{{{{y_i}}}}", from=2-2, to=3-2]
	\arrow["{{{d_1}}}", between={0.3}{0.8}, squiggly, from=2-3, to=2-5]
	\arrow["{{{{s_i}}}}", from=2-5, to=2-6]
	\arrow[""{name=2, anchor=center, inner sep=0}, equals, from=2-5, to=3-5]
	\arrow["\SIGMA"{description}, draw=none, from=2-6, to=2-7]
	\arrow[""{name=3, anchor=center, inner sep=0}, "{{{{y_i}}}}", from=2-6, to=3-6]
	\arrow["{{y_3}}"', from=3-1, to=3-2]
	\arrow[""{name=4, anchor=center, inner sep=0}, equals, from=3-1, to=4-1]
	\arrow["v"', from=3-2, to=3-3]
	\arrow[""{name=5, anchor=center, inner sep=0}, equals, from=3-2, to=4-2]
	\arrow[""{name=6, anchor=center, inner sep=0}, equals, from=3-3, to=4-3]
	\arrow["{{{{y_3}}}}"', from=3-5, to=3-6]
	\arrow[""{name=7, anchor=center, inner sep=0}, equals, from=3-5, to=4-5]
	\arrow["v"', from=3-6, to=3-7]
	\arrow[""{name=8, anchor=center, inner sep=0}, "d", from=3-6, to=4-6]
	\arrow[""{name=9, anchor=center, inner sep=0}, "{{d'}}", from=3-7, to=4-7]
	\arrow["{{y_3}}"', from=4-1, to=4-2]
	\arrow["v"', from=4-2, to=4-3]
	\arrow["{{{{\hat{y}_3}}}}"', from=4-5, to=4-6]
	\arrow["u"', from=4-6, to=4-7]
	\arrow["\SIGMA"{description}, shift left, draw=none, from=0, to=1]
	\arrow["\SIGMA"{description}, shift left, draw=none, from=2, to=3]
	\arrow["{\SIGMA^\id}"{marking, allow upside down}, draw=none, from=4, to=5]
	\arrow["{\SIGMA^\id}"{marking, allow upside down}, draw=none, from=5, to=6]
	\arrow["\SIGMA"{description}, draw=none, from=7, to=8]
	\arrow["\SIGMA"{description}, draw=none, from=8, to=9]
\end{tikzcd}\, .\]

Recall that $\Omega_i$ is given by the $\Sg$-step \eqref{eq:Omega-i} as in Definition \ref{def:horizontal}. Comparing $\tilde{\Omega}_i\circ \Omega_i$ and $\hat{\Omega}_i$, which have the same domain and codomain, we see that, on one hand, $\tilde{\Omega}_i\circ \Omega_i$ corresponds to a $\Sg$-path of the form $\begin{tikzcd}
	\bullet & \bullet & \bullet & \bullet
	\arrow["d", squiggly, from=1-1, to=1-2]
	\arrow["u", squiggly, from=1-2, to=1-3]
	\arrow["{d_1}", squiggly, from=1-3, to=1-4]
\end{tikzcd}$, as in the top line of Diagram~\eqref{eq:diag-d} below, and, on the other hand, $\hat{\Omega}_i$ is obtained by a $\Sg$-path of the form \begin{tikzcd}
	\bullet & \bullet & \bullet
	\arrow["d", squiggly, from=1-1, to=1-2]
	\arrow["u", squiggly, from=1-2, to=1-3]
\end{tikzcd}, as in the bottom of the same diagram. Moreover, both $\Sg$-paths start  with the same $\Sg$-step \begin{tikzcd}
	\bullet & \bullet
	\arrow["d", squiggly, from=1-1, to=1-2]
\end{tikzcd}.
\begin{equation}\label{eq:diag-d}
\begin{tikzcd}
	\bullet & \bullet & \bullet & \bullet
	\arrow["d", squiggly, from=1-1, to=1-2]
	\arrow["u", squiggly, from=1-2, to=1-3]
	\arrow[""{name=0, anchor=center, inner sep=0}, "u"', curve={height=18pt}, squiggly, from=1-2, to=1-4]
	\arrow["{{{d_1}}}", squiggly, from=1-3, to=1-4]
	\arrow["\equiv"{description, pos=0.3}, draw=none, from=1-3, to=0]
\end{tikzcd}
\end{equation}
By Lemma \ref{lem:Sigma-steps} (see also \cref{rem:cycle}), the two $\Sg$-paths are  equivalent --- that is, $\tilde{\Omega}_i\circ \Omega_i=\hat{\Omega}_i$. Now consider
\begin{equation}\label{eq:D5}
	\begin{tikzcd}
	& A &&&&&&&& C \\
	{} & A &&&& V && Y && C \\
	{} & A &&&& Q && {\hat{Y}} && C \\
	&&&&& Q &&&& C \\
	{} & A &&&& Q && {\hat{Y}} && C \\
	{} & A &&&& V && Y && C \\
	& A &&&&&&&& C
	\arrow[""{name=0, anchor=center, inner sep=0}, "{{{{{{{(g_1,s_1)\circ (f_1,r_1)}}}}}}}", from=1-2, to=1-10]
	\arrow[equals, from=1-2, to=2-2]
	\arrow[equals, from=1-10, to=2-10]
	\arrow["{{{\tilde{\Omega}_1:}}}", draw=none, from=2-1, to=3-1]
	\arrow["{{{{{{{y'_1x_1f_1}}}}}}}", from=2-2, to=2-6]
	\arrow[equals, from=2-2, to=3-2]
	\arrow[""{name=1, anchor=center, inner sep=0}, "{{{{{{d'}}}}}}"', from=2-6, to=3-6]
	\arrow["v"', from=2-8, to=2-6]
	\arrow[""{name=2, anchor=center, inner sep=0}, "d"', from=2-8, to=3-8]
	\arrow["{{{{y_3}}}}"', from=2-10, to=2-8]
	\arrow[""{name=3, anchor=center, inner sep=0}, equals, from=2-10, to=3-10]
	\arrow[""{name=4, anchor=center, inner sep=0}, "{{{{d'y'_1x_1f_1}}}}"{description}, from=3-2, to=3-6]
	\arrow[equals, from=3-2, to=5-2]
	\arrow[""{name=5, anchor=center, inner sep=0}, equals, from=3-6, to=4-6]
	\arrow["u"{description}, from=3-8, to=3-6]
	\arrow["{{{{{{\hat{y}_3}}}}}}"{description}, from=3-10, to=3-8]
	\arrow[""{name=6, anchor=center, inner sep=0}, Rightarrow, from=3-10, to=4-10]
	\arrow[""{name=7, anchor=center, inner sep=0}, equals, from=4-6, to=5-6]
	\arrow["{{{{u\hat{y}_3}}}}"{description}, from=4-10, to=4-6]
	\arrow[""{name=8, anchor=center, inner sep=0}, equals, from=4-10, to=5-10]
	\arrow[""{name=9, anchor=center, inner sep=0}, "{{{{d'y'_2x_2f_2}}}}"{description}, from=5-2, to=5-6]
	\arrow[equals, from=5-2, to=6-2]
	\arrow["u"{description}, from=5-8, to=5-6]
	\arrow["{{{{\hat{y}_3}}}}"{description}, from=5-10, to=5-8]
	\arrow[""{name=10, anchor=center, inner sep=0}, equals, from=5-10, to=6-10]
	\arrow["{{{\tilde{\Omega}^{-1}_2:}}}"', draw=none, from=6-1, to=5-1]
	\arrow["{{{{y'_2x_2f_2}}}}"', from=6-2, to=6-6]
	\arrow[equals, from=6-2, to=7-2]
	\arrow[""{name=11, anchor=center, inner sep=0}, "{{{{d'}}}}", from=6-6, to=5-6]
	\arrow[""{name=12, anchor=center, inner sep=0}, "d", from=6-8, to=5-8]
	\arrow["v", from=6-8, to=6-6]
	\arrow["{{{{y_3}}}}", from=6-10, to=6-8]
	\arrow[equals, from=6-10, to=7-10]
	\arrow[""{name=13, anchor=center, inner sep=0}, "{{{{(g_2,s_2)\circ (f_2,r_2)}}}}"', from=7-2, to=7-10]
	\arrow["{{{{{{\Omega_1}}}}}}", between={0.2}{1}, Rightarrow, from=0, to=2-6]
	\arrow["{{{{\SIGMA^{\mu}}}}}"{description}, draw=none, from=1, to=2]
	\arrow["{{{{\SIGMA^{\lambda}}}}}"{description}, draw=none, from=2, to=3]
	\arrow["{{{{(d'\circ \beta')\circ \alpha}}}}", shift right, between={0.3}{0.7}, Rightarrow, from=4, to=9]
	\arrow["{{{{\SIGMA^{\id}}}}}"{description}, draw=none, from=5, to=6]
	\arrow["{{{{\SIGMA^{\id}}}}}"{description}, draw=none, from=7, to=8]
	\arrow["{{{{\SIGMA^{\mu}}}}}"{description}, draw=none, from=11, to=12]
	\arrow["{{{{\Omega^{-1}_2}}}}", between={0}{0.8}, Rightarrow, from=6-6, to=13]
	\arrow["{{{{\SIGMA^{\lambda}}}}}"{description}, draw=none, from=12, to=10]
\end{tikzcd}\; .
\end{equation}

We find
\[\begin{array}{ll}
	[(\beta,y_1,y_2)\circ (\alpha,x_1,x_2)]=\Omega^{-1}_2\cdot [\beta'\circ \alpha,1_V,1_V] \cdot \Omega_1 & \text{\ (see \cref{def:horizontal})}\\
=\Omega^{-1}_2\cdot [(d'\circ \beta')\circ \alpha,d',d'] \cdot  \Omega_1 & \text{}\\
	=\Omega^{-1}_2\cdot (\tilde{\Omega}^{-1}_2\cdot [(d'\circ \beta')\circ \alpha,1_Q,1_Q] \cdot \tilde{\Omega}_1)\cdot \Omega_1 & \text{}\\
	=(\Omega^{-1}_2\cdot \tilde{\Omega}^{-1}_2)\cdot [(d'\circ \beta')\circ \alpha,1_Q,1_Q] \cdot (\tilde{\Omega}_1\cdot \Omega_1) & \text{}\\
	=\hat{\Omega}^{-1}_2\cdot [(d'\circ \beta')\circ \alpha,1_Q,1_Q] \cdot \hat{\Omega}_1 & \text{}\\
	=[(\hat{\beta},\hat{y}_1,\hat{y}_2)\circ (\alpha,x_1,x_2)],
\end{array}\]
as required.
\end{proof}

The following remark will be useful in the proof of the next proposition.
\begin{lemma}\label{rem:two-2-cells}
Assume we are given the data
\[\begin{tikzcd}
	J && B & I
	\arrow[""{name=0, anchor=center, inner sep=0}, "{h_1}"', curve={height=30pt}, from=1-3, to=1-1]
	\arrow[""{name=1, anchor=center, inner sep=0}, "{h_3}", curve={height=-30pt}, from=1-3, to=1-1]
	\arrow[""{name=2, anchor=center, inner sep=0}, "{h_2}"{pos=0.7}, from=1-3, to=1-1]
	\arrow["r", from=1-3, to=1-4]
	\arrow["{\gamma_1}"', shorten <=4pt, shorten >=4pt, Rightarrow, from=0, to=2]
	\arrow["{\gamma_2}"', shorten <=4pt, shorten >=4pt, Rightarrow, from=2, to=1]
\end{tikzcd}\]
with
$r\in \Sigma$.
Then there are
\[\begin{tikzcd}
	B & I & {} &&& I && {} \\
	J & V & {} &&& V && {}
	\arrow["r", from=1-1, to=1-2]
	\arrow[""{name=0, anchor=center, inner sep=0}, "{h_i}"', from=1-1, to=2-1]
	\arrow[""{name=1, anchor=center, inner sep=0}, "{h^{\prime}_i}", from=1-2, to=2-2]
	\arrow[""{name=2, anchor=center, inner sep=0}, "{\ (i=1,2,3)}"{description}, draw=none, from=1-3, to=2-3]
	\arrow[""{name=3, anchor=center, inner sep=0}, "{h^{\prime}_i}"', curve={height=18pt}, from=1-6, to=2-6]
	\arrow[""{name=4, anchor=center, inner sep=0}, "{h^{\prime}_{i+1}}", curve={height=-18pt}, from=1-6, to=2-6]
	\arrow[""{name=5, anchor=center, inner sep=0}, draw=none, from=1-8, to=2-8]
	\arrow["v"', from=2-1, to=2-2]
	\arrow["{\SIGMA^{\xi_i}}"{description}, draw=none, from=0, to=1]
	\arrow["{\text{\normalsize and}}"{description}, draw=none, from=2, to=3]
	\arrow["{\gamma'_{i}}", shorten <=7pt, shorten >=7pt, Rightarrow, from=3, to=4]
	\arrow["{\ (i=1,2)}"{description, pos=0.3}, draw=none, from=4, to=5]
\end{tikzcd}\]
such that
\begin{equation}\label{eq:R1}
\begin{tikzcd}
	B & I && B & I && {} \\
	J & V && J & V && {}
	\arrow["r", from=1-1, to=1-2]
	\arrow[""{name=0, anchor=center, inner sep=0}, "{{{{{{{h_{i+1}}}}}}}}"{description}, curve={height=6pt}, from=1-1, to=2-1]
	\arrow[""{name=1, anchor=center, inner sep=0}, "{{{{{{{{h_i}}}}}}}}"', curve={height=40pt}, from=1-1, to=2-1]
	\arrow[""{name=2, anchor=center, inner sep=0}, "{{{{{{{{h^{\prime}_{i+1}}}}}}}}}", from=1-2, to=2-2]
	\arrow["r", from=1-4, to=1-5]
	\arrow[""{name=3, anchor=center, inner sep=0}, "{{{{{{{{h_i}}}}}}}}"', from=1-4, to=2-4]
	\arrow[""{name=4, anchor=center, inner sep=0}, "{{{{{{{{h^{\prime}_i}}}}}}}}", from=1-5, to=2-5]
	\arrow[""{name=5, anchor=center, inner sep=0}, "{{{{{{{{h^{\prime}_{i+1}}}}}}}}}", curve={height=-40pt}, from=1-5, to=2-5]
	\arrow[""{name=6, anchor=center, inner sep=0}, draw=none, from=1-7, to=2-7]
	\arrow["v"', from=2-1, to=2-2]
	\arrow["v"', from=2-4, to=2-5]
	\arrow["{{{{\gamma_i}}}}", between={0.2}{0.6}, Rightarrow, from=1, to=0]
	\arrow["{{{{\SIGMA^{\xi_{i+1}}}}}}"{description, pos=0.6}, draw=none, from=0, to=2]
	\arrow["{{{{{{{{\text{\normalsize =}}}}}}}}}"{description}, draw=none, from=2, to=3]
	\arrow["{{{{{{{{\SIGMA^{\xi_{i}}}}}}}}}}"{description}, draw=none, from=3, to=4]
	\arrow["{{{{{{{{\gamma^{\prime}_i}}}}}}}}", between={0.4}{0.8}, Rightarrow, from=4, to=5]
	\arrow["{{{{{{{{(i=1,2)}}}}}}}}"{description, pos=1}, draw=none, from=5, to=6]
\end{tikzcd}
\end{equation}
\end{lemma}
\begin{proof}
Rule 6 of Proposition \ref{pro:useful_rules} tells us that this is true for just one 2-cell $\gamma_1$. Thus, we have equalities of the form
\[\begin{tikzcd}
	B & I & B & I &&&& B & I & B & I \\
	J & {V_1} & J & {V_1} &&&& J & {V_2} & J & {V_2}
	\arrow["r", from=1-1, to=1-2]
	\arrow[""{name=0, anchor=center, inner sep=0}, "{{{{h_2}}}}"{description}, from=1-1, to=2-1]
	\arrow[""{name=1, anchor=center, inner sep=0}, "{{{{h_1}}}}"', curve={height=30pt}, from=1-1, to=2-1]
	\arrow[""{name=2, anchor=center, inner sep=0}, "{{{{\tilde{h}_2}}}}", from=1-2, to=2-2]
	\arrow["r", from=1-3, to=1-4]
	\arrow[""{name=3, anchor=center, inner sep=0}, "{{{{h_1}}}}"', from=1-3, to=2-3]
	\arrow[""{name=4, anchor=center, inner sep=0}, "{{{{\tilde{h}_1}}}}"{description}, from=1-4, to=2-4]
	\arrow[""{name=5, anchor=center, inner sep=0}, "{{{{\tilde{h}_2}}}}", curve={height=-30pt}, from=1-4, to=2-4]
	\arrow["r", from=1-8, to=1-9]
	\arrow[""{name=6, anchor=center, inner sep=0}, "{{{{h_3}}}}"{description}, from=1-8, to=2-8]
	\arrow[""{name=7, anchor=center, inner sep=0}, "{{{{h_2}}}}"', curve={height=30pt}, from=1-8, to=2-8]
	\arrow[""{name=8, anchor=center, inner sep=0}, "{{{{\hat{h}_3}}}}", from=1-9, to=2-9]
	\arrow["r", from=1-10, to=1-11]
	\arrow[""{name=9, anchor=center, inner sep=0}, "{{{{h_2}}}}"', from=1-10, to=2-10]
	\arrow[""{name=10, anchor=center, inner sep=0}, "{{{{\hat{h}_3}}}}", curve={height=-30pt}, from=1-11, to=2-11]
	\arrow[""{name=11, anchor=center, inner sep=0}, "{{{{\hat{h}_2}}}}"{description}, from=1-11, to=2-11]
	\arrow["{{{{v_1}}}}"', from=2-1, to=2-2]
	\arrow["{{{{v_1}}}}"', from=2-3, to=2-4]
	\arrow["{{{{v_2}}}}"', from=2-8, to=2-9]
	\arrow["{{{{v_2}}}}"', from=2-10, to=2-11]
	\arrow["{{{{\gamma_1}}}}", between={0.3}{0.7}, Rightarrow, from=1, to=0]
	\arrow["{{{{\SIGMA^{\sigma_2}}}}}"{description}, draw=none, from=0, to=2]
	\arrow["{{{{\text{\normalsize =}}}}}"{description}, draw=none, from=2, to=3]
	\arrow["{{{{\SIGMA^{\sigma_1}}}}}"{description}, draw=none, from=3, to=4]
	\arrow["{{{{\tilde{\gamma}_1}}}}", between={0.3}{0.7}, Rightarrow, from=4, to=5]
	\arrow["{{{{\text{\normalsize and}}}}}"{description}, draw=none, from=5, to=7]
	\arrow["{{{{\gamma_2}}}}", between={0.3}{0.7}, Rightarrow, from=7, to=6]
	\arrow["{{{{\SIGMA^{\mu_3}}}}}"{description}, draw=none, from=6, to=8]
	\arrow["{{{{\text{\normalsize =}}}}}"{description}, draw=none, from=8, to=9]
	\arrow["{{{{\SIGMA^{\mu_2}}}}}"{description}, draw=none, from=9, to=11]
	\arrow["{{{{\hat{\gamma}_2}}}}", between={0.3}{0.7}, Rightarrow, from=11, to=10]
\end{tikzcd}\, .\]
Using Rule 4', we obtain the equality
\[\begin{tikzcd}
	B & I && B & I \\
	J & {V_1} & {V_2} & J & {V_2} \\
	J & V && J & V
	\arrow["r", from=1-1, to=1-2]
	\arrow[""{name=0, anchor=center, inner sep=0}, "{h_2}"', from=1-1, to=2-1]
	\arrow[""{name=1, anchor=center, inner sep=0}, "{\tilde{h}_2}", from=1-2, to=2-2]
	\arrow["{\hat{h}_2}", curve={height=-12pt}, from=1-2, to=2-3]
	\arrow["r", from=1-4, to=1-5]
	\arrow[""{name=2, anchor=center, inner sep=0}, "{h_2}"', from=1-4, to=2-4]
	\arrow[""{name=3, anchor=center, inner sep=0}, "{\hat{h}_2}", from=1-5, to=2-5]
	\arrow["{v_1}"', from=2-1, to=2-2]
	\arrow[""{name=4, anchor=center, inner sep=0}, Rightarrow, no head, from=2-1, to=3-1]
	\arrow["\rho", Rightarrow, from=2-2, to=2-3]
	\arrow[""{name=5, anchor=center, inner sep=0}, "{w_1}", from=2-2, to=3-2]
	\arrow["{\text{\Large =}}"{description}, draw=none, from=2-3, to=2-4]
	\arrow["{w_2}", curve={height=-12pt}, from=2-3, to=3-2]
	\arrow["{v_2}"', from=2-4, to=2-5]
	\arrow[""{name=6, anchor=center, inner sep=0}, Rightarrow, no head, from=2-4, to=3-4]
	\arrow[""{name=7, anchor=center, inner sep=0}, "{w_2}", from=2-5, to=3-5]
	\arrow["v"', from=3-1, to=3-2]
	\arrow["v"', from=3-4, to=3-5]
	\arrow["{\SIGMA^{\sigma_2}}"{description}, draw=none, from=0, to=1]
	\arrow["{\SIGMA^{\mu_2}}"{description}, draw=none, from=2, to=3]
	\arrow["{\SIGMA^{\omega_1}}"{description}, draw=none, from=4, to=5]
	\arrow["{\SIGMA^{\omega_2}}"{description}, draw=none, from=6, to=7]
\end{tikzcd}\]
where $\rho$ is invertible.
Then put:

$\begin{array}{l}
h^{\prime}_1 = w_1\tilde{h}_1;\; h^{\prime}_2 = w_1\tilde{h}_2;\; h^{\prime}_3 = w_2\hat{h}_3;\\
 \gamma^{\prime}_1=w_1\tilde{\gamma}_1;\; \gamma^{\prime}_2= (w_2\hat{\gamma}_2)\cdot \rho; \\
 \Sigma^{\xi_1}=\Sigma^{\omega_1\odot \sigma_1};\; \Sigma^{\xi_2}=\Sigma^{\omega_1\odot \sigma_2};\; \Sigma^{\xi_3}=\Sigma^{\omega_2\odot \mu_3}.\;
\end{array}$

It is easy to see that these satisfy the desired equalities.
\end{proof}

\begin{proposition}\label{pro:comp}
For each triple of objects $A,B,C$, horizontal composition gives a functor
from $\catx[\Sigma_\ast](A,B)\times \catx[\Sigma_\ast](B,C)$ to $\catx[\Sigma_\ast](A,C)$.
\end{proposition}

\begin{proof} The equality $\id_g\circ \id_f=\id_{g\circ f}$ is clear: observe that  in the application of \eqref{eq:(A12)}, for $\beta=\id_g$, we may put $\beta'=\id_{\dot{g}}$ using the canonical $\Sigma$-square of $r$ and $g$.

	Now, given 1-cells $\bar{f}_i=(f_i,r_i)$ and $\bar{g}_i=(g_i,s_i)$, and 2-cells
\[\begin{tikzcd}
	A && B && C
	\arrow[""{name=0, anchor=center, inner sep=0}, "{{\bar{f}_1}}", curve={height=-24pt}, from=1-1, to=1-3]
	\arrow[""{name=1, anchor=center, inner sep=0}, "{\bar{f}_3}"', curve={height=24pt}, from=1-1, to=1-3]
	\arrow[""{name=2, anchor=center, inner sep=0}, "{\bar{f}_2}"{description}, from=1-1, to=1-3]
	\arrow[""{name=3, anchor=center, inner sep=0}, "{\bar{g}_1}", curve={height=-24pt}, from=1-3, to=1-5]
	\arrow[""{name=4, anchor=center, inner sep=0}, "{\bar{g}_3}"{description}, curve={height=24pt}, from=1-3, to=1-5]
	\arrow[""{name=5, anchor=center, inner sep=0}, "{\bar{g}_2}"{description}, from=1-3, to=1-5]
	\arrow["{{\bar{\alpha}_1}}", shorten <=3pt, shorten >=3pt, Rightarrow, from=0, to=2]
	\arrow["{{\bar{\alpha}_2}}", shorten <=3pt, shorten >=3pt, Rightarrow, from=2, to=1]
	\arrow["{{\bar{\beta}_1}}", shorten <=3pt, shorten >=3pt, Rightarrow, from=3, to=5]
	\arrow["{{\bar{\beta}_2}}", shorten <=3pt, shorten >=3pt, Rightarrow, from=5, to=4]
\end{tikzcd}\]
	we want to show that
	\begin{equation}\label{eq:interchange}(\bar{\beta}_2\cdot \bar{\beta}_1)\circ (\bar{\alpha}_2\cdot \bar{\alpha}_1)=(\bar{\beta}_2\circ \bar{\alpha}_2) \cdot (\bar{\beta}_1\circ \bar{\alpha}_1)\,.
 \end{equation}

First we show that, for $A\xrightarrow{\bar{f}}B\xrightarrow{\bar{g}}C$,  we always have
\begin{equation}\label{eq:W1}(\bar{g}\circ \bar{\alpha}_2)\cdot (\bar{g}\circ \bar{\alpha}_1)=\bar{g}\circ (\bar{\alpha}_2\cdot \bar{\alpha}_1) \; \text{ and }\; (\bar{\beta}_2\circ \bar{f})\cdot (\bar{\beta}_1\circ \bar{f})=(\bar{\beta}_2\cdot \bar{\beta}_1)\circ \bar{f}\, .\end{equation}
After that, to obtain \eqref{eq:interchange}, it is enough to prove the `whiskering law' that given two 2-cells $\bar{\alpha}$ and $\bar{\beta}$ of the form
\[\begin{tikzcd}
	A && B && C
	\arrow[""{name=0, anchor=center, inner sep=0}, "{\bar{f}_1}", curve={height=-12pt}, from=1-1, to=1-3]
	\arrow[""{name=1, anchor=center, inner sep=0}, "{\bar{f}_2}"', curve={height=12pt}, from=1-1, to=1-3]
	\arrow[""{name=2, anchor=center, inner sep=0}, "{\bar{g}_1}", curve={height=-12pt}, from=1-3, to=1-5]
	\arrow[""{name=3, anchor=center, inner sep=0}, "{\bar{g}_2}"', curve={height=12pt}, from=1-3, to=1-5]
	\arrow["{\bar{\alpha}}", shorten <=3pt, shorten >=3pt, Rightarrow, from=0, to=1]
	\arrow["{\bar{\beta}}", shorten <=3pt, shorten >=3pt, Rightarrow, from=2, to=3]
\end{tikzcd}\]
 the equalities
 \begin{equation}\label{eq:W2}(\bar{\beta}\circ \bar{f}_2)\cdot (\bar{g}_1 \circ \bar{\alpha})=\bar{\beta}\circ \bar{\alpha}=(\bar{g}_2\circ \bar{\alpha})\cdot (\bar{\beta} \circ \bar{f}_1)
 \end{equation}
 hold.

\noindent (1) In order to show the first equality of \eqref{eq:W1}, suppose we have $\bar{f}_i=(f_i,I_i,r_i)\colon A\to B$, $i=1,2,3$, $\bar{\alpha}_i=[\alpha_i,x_{i1}, x_{i2}, x_{i3},\delta_{i1},\delta_{i2}]\colon \bar{f}_i\Rightarrow \bar{f}_{i+1}$, and $\bar{g}=(g,J,s)$.

 Consider the following data.
 \begin{equation}\label{eq:new}
 \begin{tikzcd}
	B & {I_2} && B & {I_2} \\
	B & {X_1} & {X_2} & B & {X_2} \\
	B & D && B & D
	\arrow["{r_2}", from=1-1, to=1-2]
	\arrow[""{name=0, anchor=center, inner sep=0}, Rightarrow, no head, from=1-1, to=2-1]
	\arrow[""{name=1, anchor=center, inner sep=0}, "{x_{12}}", from=1-2, to=2-2]
	\arrow["{x_{21}}", curve={height=-6pt}, from=1-2, to=2-3]
	\arrow["{r_2}", from=1-4, to=1-5]
	\arrow[""{name=2, anchor=center, inner sep=0}, Rightarrow, no head, from=1-4, to=2-4]
	\arrow[""{name=3, anchor=center, inner sep=0}, "{x_{21}}", from=1-5, to=2-5]
	\arrow["{x_{13}}", from=2-1, to=2-2]
	\arrow[""{name=4, anchor=center, inner sep=0}, Rightarrow, no head, from=2-1, to=3-1]
	\arrow["\theta", Rightarrow, from=2-2, to=2-3]
	\arrow[""{name=5, anchor=center, inner sep=0}, "{d_1}", from=2-2, to=3-2]
	\arrow["{\text{\Large =}}"{description}, draw=none, from=2-3, to=2-4]
	\arrow["{d_2}", curve={height=-6pt}, from=2-3, to=3-2]
	\arrow["{x_{23}}", from=2-4, to=2-5]
	\arrow[""{name=6, anchor=center, inner sep=0}, Rightarrow, no head, from=2-4, to=3-4]
	\arrow[""{name=7, anchor=center, inner sep=0}, "{d_2}", from=2-5, to=3-5]
	\arrow["d", from=3-1, to=3-2]
	\arrow["d", from=3-4, to=3-5]
	\arrow["{\SIGMA^{\delta_{12}}}"{description}, draw=none, from=0, to=1]
	\arrow["{\SIGMA^{\delta_{21}}}"{description}, draw=none, from=2, to=3]
	\arrow["{\SIGMA^{\delta_{1}}}"{description}, draw=none, from=4, to=5]
	\arrow["{\SIGMA^{\delta_{2}}}"{description}, draw=none, from=6, to=7]
\end{tikzcd}\quad ; \qquad \;
\begin{tikzcd}
	B & D \\
	J & {D'}
	\arrow["d", from=1-1, to=1-2]
	\arrow[""{name=0, anchor=center, inner sep=0}, "g"', from=1-1, to=2-1]
	\arrow[""{name=1, anchor=center, inner sep=0}, "{g'}", from=1-2, to=2-2]
	\arrow["{d'}"', from=2-1, to=2-2]
	\arrow["{\SIGMA^{\sigma}}"{description}, draw=none, from=0, to=1]
\end{tikzcd}\; .
\end{equation}
Then
$\bar{\alpha}_2\cdot \bar{\alpha}_1=[(d_2\circ \alpha_2)\cdot(\theta\circ f_2)\cdot (d_1\circ \alpha_1),\, d_1x_{11},\, d_2x_{22},\, d,\, \delta_1\odot\delta_{11},\, \delta_2\odot\delta_{22}]\, ,$
and $\bar{g}\circ (\bar{\alpha}_2\cdot \bar{\alpha}_1)$ is represented by the diagram
\begin{equation}\label{eq:F1}
\begin{tikzcd}
	A &&&&&&& C \\
	A && {I_1} & {X_1} & D & {D'} & J & C \\
	A && {I_2} & {X_1} & D \\
	A &&&& D & {D'} && C \\
	A && {I_2} & {X_2} & D \\
	A && {I_3} & {X_2} & D & {D'} & J & C \\
	A &&&&&&& C \\
	\\
	&&&& {}
	\arrow[""{name=0, anchor=center, inner sep=0}, "{{\bar{g}\circ \bar{f}_1}}"{description}, from=1-1, to=1-8]
	\arrow[Rightarrow, no head, from=1-1, to=2-1]
	\arrow[Rightarrow, no head, from=1-8, to=2-8]
	\arrow["{{{{f_1}}}}", from=2-1, to=2-3]
	\arrow[Rightarrow, no head, from=2-1, to=3-1]
	\arrow["{{{x_{11}}}}", from=2-3, to=2-4]
	\arrow[""{name=1, anchor=center, inner sep=0}, "{{{d_1}}}"', from=2-4, to=2-5]
	\arrow[Rightarrow, no head, from=2-4, to=3-4]
	\arrow["{{{g'}}}", from=2-5, to=2-6]
	\arrow[Rightarrow, no head, from=2-5, to=3-5]
	\arrow[""{name=2, anchor=center, inner sep=0}, Rightarrow, no head, from=2-6, to=4-6]
	\arrow["{{{d'}}}"', from=2-7, to=2-6]
	\arrow["s"', from=2-8, to=2-7]
	\arrow[""{name=3, anchor=center, inner sep=0}, Rightarrow, no head, from=2-8, to=4-8]
	\arrow[""{name=4, anchor=center, inner sep=0}, "{f_2}"', from=3-1, to=3-3]
	\arrow[Rightarrow, no head, from=3-1, to=4-1]
	\arrow["{x_{12}}"', from=3-3, to=3-4]
	\arrow[Rightarrow, no head, from=3-3, to=5-3]
	\arrow[""{name=5, anchor=center, inner sep=0}, "{d_1}"', from=3-4, to=3-5]
	\arrow[Rightarrow, no head, from=3-5, to=4-5]
	\arrow[Rightarrow, no head, from=4-1, to=5-1]
	\arrow[Rightarrow, no head, from=4-5, to=5-5]
	\arrow[""{name=6, anchor=center, inner sep=0}, Rightarrow, no head, from=4-6, to=6-6]
	\arrow["{d's}"', from=4-8, to=4-6]
	\arrow[""{name=7, anchor=center, inner sep=0}, Rightarrow, no head, from=4-8, to=6-8]
	\arrow["{f_2}", from=5-1, to=5-3]
	\arrow[Rightarrow, no head, from=5-1, to=6-1]
	\arrow[""{name=8, anchor=center, inner sep=0}, "{x_{21}}", from=5-3, to=5-4]
	\arrow["{d_2}", from=5-4, to=5-5]
	\arrow[Rightarrow, no head, from=5-4, to=6-4]
	\arrow[Rightarrow, no head, from=5-5, to=6-5]
	\arrow[""{name=9, anchor=center, inner sep=0}, "{{{{f_3}}}}"', from=6-1, to=6-3]
	\arrow[Rightarrow, no head, from=6-1, to=7-1]
	\arrow["{{{x_{22}}}}"', from=6-3, to=6-4]
	\arrow[""{name=10, anchor=center, inner sep=0}, "{{{d_2}}}", from=6-4, to=6-5]
	\arrow["{{{g'}}}"', from=6-5, to=6-6]
	\arrow["{{{d'}}}", from=6-7, to=6-6]
	\arrow["s", from=6-8, to=6-7]
	\arrow[Rightarrow, no head, from=6-8, to=7-8]
	\arrow[""{name=11, anchor=center, inner sep=0}, "{{\bar{g}\circ \bar{f}_3}}"', from=7-1, to=7-8]
	\arrow["{{\Omega_1}}", shorten <=4pt, shorten >=4pt, Rightarrow, from=0, to=1]
	\arrow["{\alpha_1}"', shorten <=11pt, shorten >=11pt, Rightarrow, from=2-4, to=4]
	\arrow["{\SIGMA^{\id}}"{description}, draw=none, from=2, to=3]
	\arrow["\theta"', shorten <=20pt, shorten >=15pt, Rightarrow, from=5, to=8]
	\arrow["{\SIGMA^{\id}}"{description}, draw=none, from=6, to=7]
	\arrow["{\alpha_2}"', shorten <=10pt, shorten >=10pt, Rightarrow, from=8, to=9]
	\arrow["{{\bar{\Omega}_2^{-1}}}", shorten <=4pt, shorten >=4pt, Rightarrow, from=10, to=11]
\end{tikzcd}
\end{equation}
 where $\Omega_i\colon \bar{g}\circ \bar{f}_i\Rightarrow (g'd_1x_{1i}f_i,d's)$ is the $\Omega$ 2-cell corresponding just to the $\Sg$-path
 \[\begin{tikzcd}
	& {} & {} \\
	{} & {} & {}
	\arrow["{r_i}", from=1-2, to=1-3]
	\arrow[""{name=0, anchor=center, inner sep=0}, "g"', from=1-2, to=2-2]
	\arrow[""{name=1, anchor=center, inner sep=0}, "{\dot{g}_i}", from=1-3, to=2-3]
	\arrow["s", from=2-1, to=2-2]
	\arrow["{\dot{r}_i}"', from=2-2, to=2-3]
	\arrow["{\SIGMAc}"{description}, draw=none, from=0, to=1]
\end{tikzcd}
 \xymatrix{\ar@{~>}[r]^{u}&}
 \begin{tikzcd}
	& {} & {} \\
	{} & {} & {}
	\arrow["{{r_i}}", from=1-2, to=1-3]
	\arrow[""{name=0, anchor=center, inner sep=0}, "g"', from=1-2, to=2-2]
	\arrow[""{name=1, anchor=center, inner sep=0}, "{g'd_1x_{1i}}", from=1-3, to=2-3]
	\arrow["s", from=2-1, to=2-2]
	\arrow["{d'}"', from=2-2, to=2-3]
	\arrow["\SIGMA"{description}, draw=none, from=0, to=1]
\end{tikzcd}\]
 and, similarly, $\bar{\Omega}_i\colon \bar{g}\circ \bar{f}_{i+1}\Rightarrow (g'd_2x_{2i}f_{i+1},d's)$ corresponds to
 \[\begin{tikzcd}
	& {} & {} \\
	{} & {} & {}
	\arrow["{r_{i+1}}", from=1-2, to=1-3]
	\arrow[""{name=0, anchor=center, inner sep=0}, "g"', from=1-2, to=2-2]
	\arrow[""{name=1, anchor=center, inner sep=0}, "{\dot{g}_{i+1}}", from=1-3, to=2-3]
	\arrow["s", from=2-1, to=2-2]
	\arrow["{\dot{r}_{i+1}}"', from=2-2, to=2-3]
	\arrow["{{\SIGMAc}}"{description}, draw=none, from=0, to=1]
\end{tikzcd}
 \xymatrix{\ar@{~>}[r]^{u}&}
 \begin{tikzcd}
	& {} & {} \\
	{} & {} & {}
	\arrow["{r_{i+1}}", from=1-2, to=1-3]
	\arrow[""{name=0, anchor=center, inner sep=0}, "g"', from=1-2, to=2-2]
	\arrow[""{name=1, anchor=center, inner sep=0}, "{g'd_2x_{2i}}", from=1-3, to=2-3]
	\arrow["s", from=2-1, to=2-2]
	\arrow["{{d'}}"', from=2-2, to=2-3]
	\arrow["\SIGMA"{description}, draw=none, from=0, to=1]
\end{tikzcd}\, .\]
Using \eqref{eq:new} again, the horizontal composition $\bar{g}\circ \bar{\alpha}_1$ is represented by

\begin{equation}\label{eq:F2}
\begin{tikzcd}
	A &&&&&&& C \\
	A && {I_1} & {X_1} & D & {D'} & J & C \\
	A && {I_2} & {X_1} & D & {D'} && C \\
	A &&&&&&& C
	\arrow[""{name=0, anchor=center, inner sep=0}, "{{{\bar{g}\circ \bar{f}_1}}}"{description}, from=1-1, to=1-8]
	\arrow[Rightarrow, no head, from=1-1, to=2-1]
	\arrow[Rightarrow, no head, from=1-8, to=2-8]
	\arrow["{{{{{{{f_1}}}}}}}", from=2-1, to=2-3]
	\arrow[Rightarrow, no head, from=2-1, to=3-1]
	\arrow["{{{{{{x_{11}}}}}}}", from=2-3, to=2-4]
	\arrow[""{name=1, anchor=center, inner sep=0}, "{{{{{{d_1}}}}}}"', from=2-4, to=2-5]
	\arrow[Rightarrow, no head, from=2-4, to=3-4]
	\arrow[""{name=2, anchor=center, inner sep=0}, "{{{{{{g'}}}}}}", from=2-5, to=2-6]
	\arrow[""{name=3, anchor=center, inner sep=0}, Rightarrow, from=2-6, to=3-6]
	\arrow["{{{{{{d'}}}}}}"', from=2-7, to=2-6]
	\arrow["s"', from=2-8, to=2-7]
	\arrow[""{name=4, anchor=center, inner sep=0}, Rightarrow, no head, from=2-8, to=3-8]
	\arrow[""{name=5, anchor=center, inner sep=0}, "{{{{f_2}}}}"', from=3-1, to=3-3]
	\arrow[Rightarrow, no head, from=3-1, to=4-1]
	\arrow["{{{{x_{12}}}}}"', from=3-3, to=3-4]
	\arrow[""{name=6, anchor=center, inner sep=0}, "{{{{d_1}}}}", from=3-4, to=3-5]
	\arrow["{{{g'}}}"', from=3-5, to=3-6]
	\arrow["{{{d's}}}", from=3-8, to=3-6]
	\arrow[Rightarrow, no head, from=3-8, to=4-8]
	\arrow[""{name=7, anchor=center, inner sep=0}, "{{{\bar{g}\circ \bar{f}_2}}}"{description}, from=4-1, to=4-8]
	\arrow["{{{{{\Omega_1}}}}}", shorten <=4pt, shorten >=4pt, Rightarrow, from=0, to=1]
	\arrow["{{{{\alpha_1}}}}"', shorten <=11pt, shorten >=11pt, Rightarrow, from=2-4, to=5]
	\arrow["{{{\SIGMA^{\id}}}}"{description}, draw=none, from=3, to=4]
	\arrow["{{{\Omega^{-1}_2}}}", shorten <=4pt, shorten >=4pt, Rightarrow, from=6, to=7]
\end{tikzcd}
\end{equation}
An analogous diagram represents the composition $\bar{g}\circ \bar{\alpha}_2$. Moreover the 2-cell $[g'\circ \theta\circ f_2, 1_{D'},1_{D'}]$ is just the $\Omega$ 2-cell corresponding to the $\Sg$-path
 \[\begin{tikzcd}
	& {} & {} \\
	{} & {} & {}
	\arrow["{r_2}", from=1-2, to=1-3]
	\arrow[""{name=0, anchor=center, inner sep=0}, "g"', from=1-2, to=2-2]
	\arrow[""{name=1, anchor=center, inner sep=0}, "{g'd_1x_{12}}", from=1-3, to=2-3]
	\arrow["s", from=2-1, to=2-2]
	\arrow["{d'}"', from=2-2, to=2-3]
	\arrow["\SIGMA"{description}, draw=none, from=0, to=1]
\end{tikzcd}
 \xymatrix{\ar@{~>}[r]^{u}&}
 \begin{tikzcd}
	& {} & {} \\
	{} & {} & {}
	\arrow["{r_2}", from=1-2, to=1-3]
	\arrow[""{name=0, anchor=center, inner sep=0}, "g"', from=1-2, to=2-2]
	\arrow[""{name=1, anchor=center, inner sep=0}, "{g'd_2x_{21}}", from=1-3, to=2-3]
	\arrow["s", from=2-1, to=2-2]
	\arrow["{{{d'}}}"', from=2-2, to=2-3]
	\arrow["\SIGMA"{description}, draw=none, from=0, to=1]
\end{tikzcd}\]
Then, by Lemma \ref{lem:Sigma-steps}, $[g'\circ \theta\circ f_2,\, 1_{D'},1_{D'}]\circ \Omega_2=\bar{\Omega}_1$.

Thus, we have
$$\begin{array}{rl}
\bar{g}\circ (\bar{\alpha}_2 \cdot \bar{\alpha}_1)\hspace{-2mm}
&=\bar{\Omega}^{-1}_2\cdot [g'\circ \big((d_2\circ \alpha_2)\cdot(\theta\circ f_2)\cdot (d_1\circ \alpha_1)\big),1_{D'},1_{D'}]\cdot \Omega_1 \\
&=\bar{\Omega}^{-1}_2\cdot [g'd_2\circ  \alpha_2,1_{D'},1_{D'}]\cdot [g'\circ \theta\circ f_2,1_{D'},1_{D'}]\cdot [g'd_1\circ \alpha_1,1_{D'},1_{D'}]\cdot \Omega_1 \\
&=\bar{\Omega}^{-1}_2\cdot [g'd_2\circ  \alpha_2,1_{D'},1_{D'}]\cdot \big([g'\circ \theta\circ f_2,1_{D'},1_{D'}]\cdot \Omega_2\big) \cdot\big(\Omega^{-1}_2\cdot [g'd_1\circ \alpha_1,1_{D'},1_{D'}]\cdot \Omega_1\big) \\
&=\big(\bar{\Omega}^{-1}_2\cdot [g'd_2\circ  \alpha_2,1_{D'},1_{D'}]\cdot \bar{\Omega}_1\big) \cdot\big(\Omega^{-1}_2\cdot [g'd_1\circ \alpha_1,1_{D'},1_{D'}]\cdot \Omega_1\big) \\
&=\big(\bar{g}\circ \bar{\alpha}_2\big)\cdot \big(\bar{g}\circ \bar{\alpha}_1\big)
\end{array}$$

\noindent
(2) In order to show the second equality of \eqref{eq:W1}, suppose that we have $\bar{f}=(f,I,r)\colon A\to B$, $\bar{g}_i=(g_i,J_i,s_i)\colon B\to C$ and $\bar{\beta}_i=[\beta_i, y_{i1}, y_{i2}, y_{i3}, \epsilon_{i1},\epsilon_{i2}]\colon \bar{g}_i\Rightarrow \bar{g}_{i+1}$.
Use Rule 4' to form the equality
\[\begin{tikzcd}
	C & {J_2} && C & {J_2} \\
	C & {Y_1} & {Y_2} & C & {Y_2} \\
	C & D && C & D
	\arrow["{s_2}", from=1-1, to=1-2]
	\arrow[""{name=0, anchor=center, inner sep=0}, Rightarrow, no head, from=1-1, to=2-1]
	\arrow[""{name=1, anchor=center, inner sep=0}, "{y_{12}}", from=1-2, to=2-2]
	\arrow["{y_{21}}", curve={height=-6pt}, from=1-2, to=2-3]
	\arrow["{s_2}", from=1-4, to=1-5]
	\arrow[""{name=2, anchor=center, inner sep=0}, Rightarrow, no head, from=1-4, to=2-4]
	\arrow[""{name=3, anchor=center, inner sep=0}, "{y_{21}}", from=1-5, to=2-5]
	\arrow["{y_{13}}", from=2-1, to=2-2]
	\arrow[""{name=4, anchor=center, inner sep=0}, Rightarrow, no head, from=2-1, to=3-1]
	\arrow["\theta", Rightarrow, from=2-2, to=2-3]
	\arrow[""{name=5, anchor=center, inner sep=0}, "{d_1}", from=2-2, to=3-2]
	\arrow["{\text{\normalsize =}}"{description}, draw=none, from=2-3, to=2-4]
	\arrow["{d_2}", curve={height=-6pt}, from=2-3, to=3-2]
	\arrow["{y_{23}}", from=2-4, to=2-5]
	\arrow[""{name=6, anchor=center, inner sep=0}, Rightarrow, no head, from=2-4, to=3-4]
	\arrow[""{name=7, anchor=center, inner sep=0}, "{d_2}", from=2-5, to=3-5]
	\arrow["d"', from=3-1, to=3-2]
	\arrow["d"', from=3-4, to=3-5]
	\arrow["{\SIGMA^{\epsilon_{12}}}"{description}, draw=none, from=0, to=1]
	\arrow["{\SIGMA^{\epsilon_{21}}}"{description}, draw=none, from=2, to=3]
	\arrow["{\SIGMA^{\delta_{1}}}"{description}, draw=none, from=4, to=5]
	\arrow["{\SIGMA^{\delta_{2}}}"{description}, draw=none, from=6, to=7]
\end{tikzcd}\, .\]
We can use this data to obtain $\Sigma$-extensions of $(\beta_1, y_{11}, y_{12})$ and $(\beta_2, y_{21}, y_{22})$; more precisely, we have that
$$\bar{\beta}_1=[(\theta\circ g_2)\cdot (d_1\circ \beta_1),d_1y_{11},d_2y_{21},d, \delta_1\odot \epsilon_{11},\delta_2\odot \epsilon_{21} ]\; \text{ and } \; \bar{\beta}_2=[d_2\circ\beta_2, d_2y_{21}, d_2y_{22},d, \delta_2\odot \epsilon_{21}, \delta_2\odot \epsilon_{22}]\,.$$
We form $\bar{\beta}_i\circ \bar{f}$, for $i=1,2$, using these last $\approx$-representatives. Now apply the result of \cref{rem:two-2-cells}, with $\gamma_1=(\theta\circ g_2)\cdot (d_1\circ \beta_1)$ and $\gamma_2=d_2\circ \beta_2$,
\[\begin{tikzcd}
	& B \\
	{J_1} & {J_2} & {J_3} \\
	& D
	\arrow["{g_1}"', from=1-2, to=2-1]
	\arrow["{g_2}", from=1-2, to=2-2]
	\arrow["{g_3}", from=1-2, to=2-3]
	\arrow["{\gamma_1}"{pos=0.6}, Rightarrow, from=2-1, to=2-2]
	\arrow["{d_1 y_{11}}"', from=2-1, to=3-2]
	\arrow["{\gamma_2}", Rightarrow, from=2-2, to=2-3]
	\arrow["{d_2 y_{21}}"'{pos=0.2}, from=2-2, to=3-2]
	\arrow["{d_2 y_{22}}", from=2-3, to=3-2]
\end{tikzcd}\]
to obtain a situation as in \cref{eq:R1}, with $h_1=d_1y_{11}g_1$, $h_2=d_2y_{21}g_2$ and $h_3=d_2y_{22}g_3$.

Then, we may represent the 2-cells $(\bar{\beta}_i\circ \bar{f})$ of $\catx[\Sigma_{\ast}]$ by the following diagram, where  $\Omega_i$ and $\Omega_{i+1}$ are the due 2-cells of type $\Omega$, according to \cref{def:horizontal}.
\[\begin{tikzcd}
	A &&&& C \\
	A & I & V & D & C \\
	A & I & V & D & C \\
	A &&&& C
	\arrow[""{name=0, anchor=center, inner sep=0}, "{\bar{g}_i\circ \bar{f}}", from=1-1, to=1-5]
	\arrow[Rightarrow, no head, from=1-1, to=2-1]
	\arrow[Rightarrow, no head, from=1-5, to=2-5]
	\arrow[""{name=1, anchor=center, inner sep=0}, "f", from=2-1, to=2-2]
	\arrow[Rightarrow, no head, from=2-1, to=3-1]
	\arrow[""{name=2, anchor=center, inner sep=0}, "{h'_i}", from=2-2, to=2-3]
	\arrow[Rightarrow, no head, from=2-2, to=3-2]
	\arrow[Rightarrow, no head, from=2-3, to=3-3]
	\arrow["v"', from=2-4, to=2-3]
	\arrow["d"', from=2-5, to=2-4]
	\arrow[Rightarrow, no head, from=2-5, to=3-5]
	\arrow[""{name=3, anchor=center, inner sep=0}, "f"', from=3-1, to=3-2]
	\arrow[Rightarrow, no head, from=3-1, to=4-1]
	\arrow[""{name=4, anchor=center, inner sep=0}, "{h'_{i+1}}"', from=3-2, to=3-3]
	\arrow["v", from=3-4, to=3-3]
	\arrow["d", from=3-5, to=3-4]
	\arrow[Rightarrow, no head, from=3-5, to=4-5]
	\arrow[""{name=5, anchor=center, inner sep=0}, "{\bar{g}_{i+1}\circ \bar{f}}"', from=4-1, to=4-5]
	\arrow["{\Omega_i}", shorten <=3pt, Rightarrow, from=0, to=2-3]
	\arrow["{\gamma'_i}", shorten <=4pt, shorten >=4pt, Rightarrow, from=2, to=4]
	\arrow["{\Omega_{i+1}^{-1}}", shorten >=3pt, Rightarrow, from=3-3, to=5]
\end{tikzcd}\]
Observing that $\bar{\beta}_2\cdot \bar{\beta}_1=[\gamma_2\cdot\gamma_1, d_1y_{11}, d_2y_{22}]$, and that we have
\[\begin{tikzcd}
	B && I & B && I \\
	D && V & D && V
	\arrow["r", from=1-1, to=1-3]
	\arrow[""{name=0, anchor=center, inner sep=0}, from=1-1, to=2-1]
	\arrow[""{name=1, anchor=center, inner sep=0}, curve={height=-30pt}, from=1-1, to=2-1]
	\arrow[""{name=2, anchor=center, inner sep=0}, curve={height=30pt}, from=1-1, to=2-1]
	\arrow[""{name=3, anchor=center, inner sep=0}, "{{{h'_3}}}", from=1-3, to=2-3]
	\arrow["r", from=1-4, to=1-6]
	\arrow[""{name=4, anchor=center, inner sep=0}, from=1-4, to=2-4]
	\arrow[""{name=5, anchor=center, inner sep=0}, from=1-6, to=2-6]
	\arrow[""{name=6, anchor=center, inner sep=0}, curve={height=-30pt}, from=1-6, to=2-6]
	\arrow[""{name=7, anchor=center, inner sep=0}, curve={height=30pt}, from=1-6, to=2-6]
	\arrow["v"', from=2-1, to=2-3]
	\arrow["v"', from=2-4, to=2-6]
	\arrow["{{{\gamma_1}}}", between={0.2}{0.8}, Rightarrow, from=2, to=0]
	\arrow["{{{\gamma_2}}}", between={0.2}{0.8}, Rightarrow, from=0, to=1]
	\arrow["{{{\SIGMA^{\xi_3}}}}"{description, pos=0.4}, draw=none, from=1, to=3]
	\arrow["{{{\text{\normalsize =}}}}"{description, pos=0.6}, draw=none, from=3, to=4]
	\arrow["{{{\SIGMA^{\xi_1}}}}"{description, pos=0.6}, draw=none, from=4, to=7]
	\arrow["{{{\gamma'_1}}}", shift right, between={0.2}{0.8}, Rightarrow, from=7, to=5]
	\arrow["{{{\gamma'_2}}}", shift right, between={0.2}{0.8}, Rightarrow, from=5, to=6]
\end{tikzcd}\]
and so we obtain:
$$\begin{array}{rl}(\bar{\beta}_2\circ \bar{f})\cdot (\bar{\beta}_1\circ \bar{f})&=\big(\Omega_3^{-1}\cdot[\gamma'_2\circ f,1_V,1_V]\cdot\Omega_2\big)\cdot\big(\Omega_2^{-1}\cdot[\gamma'_1\circ f,1_V,1_V]\cdot \Omega_1\big)\\
&=\Omega_3^{-1}\cdot[\gamma'_2\circ f,1_V,1_V]\cdot[\gamma'_1\circ f,1_V,1_V]\cdot \Omega_1\\
&=\Omega_3^{-1}\cdot[(\gamma'_2\cdot \gamma'_1)\circ f,1_V,1_V]\cdot \Omega_1\\
&=(\bar{\beta}_2\cdot\bar{\beta}_1)\circ \bar{f}.
\end{array}$$

\noindent (3)  Concerning \eqref{eq:W2},
let us consider two 2-cells $\overline{\alpha}=[\alpha, x_1, x_2]$ and $\overline{\beta}=[\beta, y_1, y_2]$  represented in the following diagram.
\begin{equation}\label{eq:circ1}
\begin{tikzcd}
	A & {I_1} & B & {J_1} & C \\
	& X & B & Y & C \\
	A & {I_2} & B & {J_2} & C
	\arrow["{{f_1}}", from=1-1, to=1-2]
	\arrow[equals, from=1-1, to=3-1]
	\arrow[""{name=0, anchor=center, inner sep=0}, "{{x_1}}", from=1-2, to=2-2]
	\arrow["\alpha"', shorten <=11pt, shorten >=11pt, Rightarrow, from=1-2, to=3-1]
	\arrow["{{r_1}}"', from=1-3, to=1-2]
	\arrow["{{g_1}}", from=1-3, to=1-4]
	\arrow[""{name=1, anchor=center, inner sep=0}, equals, from=1-3, to=2-3]
	\arrow[""{name=2, anchor=center, inner sep=0}, "{{y_1}}"', from=1-4, to=2-4]
	\arrow["\beta"', shorten <=11pt, shorten >=11pt, Rightarrow, from=1-4, to=3-3]
	\arrow["{{s_1}}"', from=1-5, to=1-4]
	\arrow[""{name=3, anchor=center, inner sep=0}, equals, from=1-5, to=2-5]
	\arrow["{x_3}", from=2-3, to=2-2]
	\arrow[""{name=4, anchor=center, inner sep=0}, equals, from=2-3, to=3-3]
	\arrow["{y_3}", from=2-5, to=2-4]
	\arrow[""{name=5, anchor=center, inner sep=0}, equals, from=2-5, to=3-5]
	\arrow["{{f_2}}"', from=3-1, to=3-2]
	\arrow[""{name=6, anchor=center, inner sep=0}, "{{x_2}}"', from=3-2, to=2-2]
	\arrow["{{r_2}}", from=3-3, to=3-2]
	\arrow["{{g_2}}"', from=3-3, to=3-4]
	\arrow[""{name=7, anchor=center, inner sep=0}, "{{y_2}}", from=3-4, to=2-4]
	\arrow["{{s_2}}", from=3-5, to=3-4]
	\arrow["{\SIGMA^{\delta_1}}"{description}, draw=none, from=1, to=0]
	\arrow["{\SIGMA^{\epsilon_1}}"{description}, draw=none, from=3, to=2]
	\arrow["{\SIGMA^{\delta_2}}"{description}, draw=none, from=4, to=6]
	\arrow["{\SIGMA^{\epsilon_2}}"{description}, draw=none, from=5, to=7]
\end{tikzcd}
\end{equation}
We want to show that
\begin{equation}\label{eq:circ0}(\bar{\beta}\circ \bar{f}_2)\cdot(\bar{g}_1\circ \bar{\alpha})= \bar{\beta}\circ \bar{\alpha}\qquad \text{and}\qquad (\bar{g}_2\circ \bar{\alpha})\cdot(\bar{\beta}\circ \bar{f}_1)= \bar{\beta}\circ \bar{\alpha}.
\end{equation}

Consider the following data, used in the composition $\bar{\beta}\circ \bar{\alpha}$ (see \eqref{eq:(A12)}).
\begin{equation}\label{eq:circ2}
\begin{tikzcd}
	B & X &&& B & X \\
	{J_1} &&& {J_1} & {J_2} \\
	Y & V &&& Y & V
	\arrow["{{{x_3}}}", from=1-1, to=1-2]
	\arrow["{{{g_1}}}"', from=1-1, to=2-1]
	\arrow[""{name=0, anchor=center, inner sep=0}, "{{{y'_1}}}", from=1-2, to=3-2]
	\arrow[""{name=1, anchor=center, inner sep=0}, "{{{y'_2}}}"{pos=0.6}, curve={height=-30pt}, from=1-2, to=3-2]
	\arrow["{{{x_3}}}", from=1-5, to=1-6]
	\arrow["{{{g_1}}}"', curve={height=6pt}, from=1-5, to=2-4]
	\arrow["{{{g_2}}}", from=1-5, to=2-5]
	\arrow[""{name=2, anchor=center, inner sep=0}, "{{{y'_2}}}", from=1-6, to=3-6]
	\arrow["{{{y_1}}}"', from=2-1, to=3-1]
	\arrow["\beta", Rightarrow, from=2-4, to=2-5]
	\arrow["{{{y_1}}}"', curve={height=6pt}, from=2-4, to=3-5]
	\arrow["{{{y_2}}}", from=2-5, to=3-5]
	\arrow["v", from=3-1, to=3-2]
	\arrow["v", from=3-5, to=3-6]
	\arrow["{{{\text{\normalsize =}}}}"{description, pos=0.6}, draw=none, from=1, to=2-4]
	\arrow["{{{\beta'}}}"{pos=0.6}, shorten <=9pt, shorten >=3pt, Rightarrow, from=0, to=1]
	\arrow["{{{\SIGMA^{\xi_1}}}}"{description}, draw=none, from=2-1, to=0]
	\arrow["{{{\SIGMA^{\xi_2}}}}"{description}, draw=none, from=2-5, to=2]
\end{tikzcd}
\end{equation}

We prove the first equality of \eqref{eq:circ0}. First observe that $\bar{\beta} \circ \bar{f}_2=\Omega_2^{-1}\cdot[\beta'\circ x_2f_2, 1,1]\cdot \Omega'$, where
$\Omega_i\colon \bar{g}_i\circ \bar{f}_i \Rightarrow (y'_ix_if_i,vy_3)$ and $\Omega'\colon \bar{g}_1\circ \bar{f}_2\Rightarrow (y'_1x_2f_2,vy_3)$ are the due 2-cells of type $\Omega$, see Diagram~\eqref{eq:circ3} below.

\begin{equation}\label{eq:circ3}
\begin{tikzcd}
	A &&&&& C \\
	A & {I_2} & X & V & Y & C \\
	A & {I_2} & {V_0} & V & Y & C \\
	A &&&&& C
	\arrow[""{name=0, anchor=center, inner sep=0}, "{{{\bar{g}_1\circ \bar{f}_2}}}"{description}, from=1-1, to=1-6]
	\arrow[equals, from=1-1, to=2-1]
	\arrow[equals, from=1-6, to=2-6]
	\arrow["{{f_2}}", from=2-1, to=2-2]
	\arrow[equals, from=2-1, to=3-1]
	\arrow["{{x_2}}", from=2-2, to=2-3]
	\arrow[equals, from=2-2, to=3-2]
	\arrow[""{name=1, anchor=center, inner sep=0}, "{{y'_1}}", from=2-3, to=2-4]
	\arrow[equals, from=2-3, to=3-3]
	\arrow["{{{{\beta'}}}}", between={0.2}{0.8}, Rightarrow, from=2-4, to=3-3]
	\arrow[""{name=2, anchor=center, inner sep=0}, equals, from=2-4, to=3-4]
	\arrow["v"', from=2-5, to=2-4]
	\arrow[""{name=3, anchor=center, inner sep=0}, equals, from=2-5, to=3-5]
	\arrow["{{y_3}}"', from=2-6, to=2-5]
	\arrow["{{f_2}}"', from=3-1, to=3-2]
	\arrow[equals, from=3-1, to=4-1]
	\arrow["{{x_2}}"', from=3-2, to=3-3]
	\arrow[""{name=4, anchor=center, inner sep=0}, "{{{{y'_2}}}}"', from=3-3, to=3-4]
	\arrow["v", from=3-5, to=3-4]
	\arrow[""{name=5, anchor=center, inner sep=0}, equals, from=3-6, to=2-6]
	\arrow["{{y_3}}", from=3-6, to=3-5]
	\arrow[equals, from=3-6, to=4-6]
	\arrow[""{name=6, anchor=center, inner sep=0}, "{{{\bar{g}_2\circ \bar{f}_2}}}"{description}, from=4-1, to=4-6]
	\arrow["{{\Omega'}}"{pos=0.4}, between={0.2}{0.6}, Rightarrow, from=0, to=1]
	\arrow["{{{{\SIGMA^{\id}}}}}"{description}, draw=none, from=2, to=3]
	\arrow["{{{{\SIGMA^{\id}}}}}"{description}, draw=none, from=3, to=5]
	\arrow["{{{\Omega_2^{-1}}}}"{pos=0.6}, shift left, between={0.4}{0.8}, Rightarrow, from=4, to=6]
\end{tikzcd}
\end{equation}
	Moreover, taking into account that $(\id_{g_1}, 1_{J_1},1_{J_1},s_1,\id,\id)\approx (\id_{y_1g_1},y_1,y_1,y_3,\epsilon_1,\epsilon_1)$,  the composition $\bar{g}_1\circ \bar{\alpha}$ may be represented by
\begin{equation}\label{eq:circ4}
\begin{tikzcd}
	A &&&&& C \\
	A & {I_1} & X & V & Y & C \\
	A & {I_2} & X & V & Y & C \\
	A &&&&& C
	\arrow[""{name=0, anchor=center, inner sep=0}, "{{{{\bar{g}_1\circ \bar{f}_1}}}}"{description}, from=1-1, to=1-6]
	\arrow[equals, from=1-1, to=2-1]
	\arrow[equals, from=1-6, to=2-6]
	\arrow["{{{{f_1}}}}", from=2-1, to=2-2]
	\arrow[equals, from=2-1, to=3-1]
	\arrow["{{{{x_1}}}}", from=2-2, to=2-3]
	\arrow[""{name=1, anchor=center, inner sep=0}, "{{{{{y'_1}}}}}", from=2-3, to=2-4]
	\arrow["\alpha"', between={0.3}{0.7}, Rightarrow, from=2-3, to=3-1]
	\arrow[equals, from=2-3, to=3-3]
	\arrow[""{name=2, anchor=center, inner sep=0}, equals, from=2-4, to=3-4]
	\arrow["v"', from=2-5, to=2-4]
	\arrow[""{name=3, anchor=center, inner sep=0}, equals, from=2-5, to=3-5]
	\arrow["{{{{y_3}}}}"', from=2-6, to=2-5]
	\arrow[""{name=4, anchor=center, inner sep=0}, equals, from=2-6, to=3-6]
	\arrow["{{{{f_2}}}}"', from=3-1, to=3-2]
	\arrow[equals, from=3-1, to=4-1]
	\arrow["{{{{x_2}}}}"', from=3-2, to=3-3]
	\arrow[""{name=5, anchor=center, inner sep=0}, "{{{{{y'_1}}}}}"', from=3-3, to=3-4]
	\arrow["v", from=3-5, to=3-4]
	\arrow["{{{{y_3}}}}", from=3-6, to=3-5]
	\arrow[equals, from=3-6, to=4-6]
	\arrow[""{name=6, anchor=center, inner sep=0}, "{{{{\bar{g}_1\circ \bar{f}_2}}}}"{description}, from=4-1, to=4-6]
	\arrow["{{{\Omega_1}}}"{pos=0.4}, between={0.2}{0.6}, Rightarrow, from=0, to=1]
	\arrow["{{{{\SIGMA^{\id}}}}}"{description}, draw=none, from=3, to=2]
	\arrow["{{{{\SIGMA^{\id}}}}}"{description}, draw=none, from=4, to=3]
	\arrow["{{{\Omega'^{-1}}}}"{pos=0.6}, between={0.5}{0.8}, Rightarrow, from=5, to=6]
\end{tikzcd}
\; .
\end{equation}
Composing vertically \eqref{eq:circ4} with \eqref{eq:circ3}, we immediately see that the resulting 2-morphism is indeed a representative  of the composition $\bar{\beta}\circ \bar{\alpha}$, that is,
$(\bar{\beta}\circ \bar{f}_2)\cdot (\bar{g}_1\circ \bar{\alpha})\approx\bar{\beta}\circ \bar{\alpha}$.

The proof of the second  equality of \eqref{eq:circ0} is analogous.
\end{proof}

\begin{proposition}\label{pro:natural-iso} The isomorphisms of Definition \ref{def:assoc} indeed form a natural transformation from $(-\circ -)\circ -$   to $-\circ(-\circ -)$.
\end{proposition}

\begin{proof}We want to show that, for 2-cells
\[\begin{tikzcd}
	A & B & C & D
	\arrow[""{name=0, anchor=center, inner sep=0}, "{\bar{f}_1}", curve={height=-18pt}, from=1-1, to=1-2]
	\arrow[""{name=1, anchor=center, inner sep=0}, "{\bar{f}_2}"', curve={height=18pt}, from=1-1, to=1-2]
	\arrow[""{name=2, anchor=center, inner sep=0}, "{\bar{g}_1}", curve={height=-18pt}, from=1-2, to=1-3]
	\arrow[""{name=3, anchor=center, inner sep=0}, "{\bar{g}_2}"', curve={height=18pt}, from=1-2, to=1-3]
	\arrow[""{name=4, anchor=center, inner sep=0}, "{\bar{h}_2}"', curve={height=18pt}, from=1-3, to=1-4]
	\arrow[""{name=5, anchor=center, inner sep=0}, "{\bar{h}_1}", curve={height=-18pt}, from=1-3, to=1-4]
	\arrow["{\bar{\alpha}}", shorten <=5pt, shorten >=5pt, Rightarrow, from=0, to=1]
	\arrow["{\bar{\beta}}", shorten <=5pt, shorten >=5pt, Rightarrow, from=2, to=3]
	\arrow["{\bar{\gamma}}", shorten <=5pt, shorten >=5pt, Rightarrow, from=5, to=4]
\end{tikzcd}\]
in $\catx[\Sigma_{\ast}]$ the following diagram
\begin{equation}\label{eq:nat-trans}
\begin{tikzcd}
	{(\bar{h}_1\circ \bar{g}_1)\circ \bar{f}_1} &&& {\bar{h}_1\circ (\bar{g}_1\circ \bar{f}_1)} \\
	{(\bar{h}_2\circ \bar{g}_2)\circ \bar{f}_2} &&& {\bar{h}_2\circ (\bar{g}_2\circ \bar{f}_2)}
	\arrow["{{{\Assoc_{\bar{h}_1,\bar{g}_1,\bar{f}_1}}}}", Rightarrow, from=1-1, to=1-4]
	\arrow["{{(\bar{\gamma}\circ \bar{\beta})\circ \bar{\alpha}}}"', Rightarrow, from=1-1, to=2-1]
	\arrow["{{\bar{\gamma}\circ(\bar{\beta}\circ \bar{\alpha})}}", Rightarrow, from=1-4, to=2-4]
	\arrow["{{{\Assoc_{\bar{h}_2,\bar{g}_2,\bar{f}_2}}}}", Rightarrow, from=2-1, to=2-4]
\end{tikzcd}
\end{equation}
is commutative in $\catx[\Sigma_*](A,D)$.

We are given
\[\begin{tikzcd}
	A & {I_1} & B & {J_1} & C & {K_1} & D \\
	& X & B & Y & C & Z & D \\
	A & {I_2} & B & {J_2} & C & {K_2} & D \rlap{\,.}
	\arrow["{{f_1}}", from=1-1, to=1-2]
	\arrow[Rightarrow, no head, from=1-1, to=3-1]
	\arrow[""{name=0, anchor=center, inner sep=0}, "{{x_1}}", from=1-2, to=2-2]
	\arrow["\alpha"', shorten <=11pt, shorten >=11pt, Rightarrow, from=1-2, to=3-1]
	\arrow["{{r_1}}"', from=1-3, to=1-2]
	\arrow["{{g_1}}", from=1-3, to=1-4]
	\arrow[""{name=1, anchor=center, inner sep=0}, Rightarrow, no head, from=1-3, to=2-3]
	\arrow[""{name=2, anchor=center, inner sep=0}, "{{y_1}}", from=1-4, to=2-4]
	\arrow["\beta"', shorten <=11pt, shorten >=11pt, Rightarrow, from=1-4, to=3-3]
	\arrow["{{s_1}}"', from=1-5, to=1-4]
	\arrow["{h_1}", from=1-5, to=1-6]
	\arrow[""{name=3, anchor=center, inner sep=0}, Rightarrow, no head, from=1-5, to=2-5]
	\arrow[""{name=4, anchor=center, inner sep=0}, "{z_1}"', from=1-6, to=2-6]
	\arrow["\gamma"', shorten <=11pt, shorten >=11pt, Rightarrow, from=1-6, to=3-5]
	\arrow["{t_1}"', from=1-7, to=1-6]
	\arrow[""{name=5, anchor=center, inner sep=0}, Rightarrow, no head, from=1-7, to=2-7]
	\arrow["{{x_3}}", from=2-3, to=2-2]
	\arrow[""{name=6, anchor=center, inner sep=0}, Rightarrow, no head, from=2-3, to=3-3]
	\arrow["{{y_3}}"', from=2-5, to=2-4]
	\arrow[""{name=7, anchor=center, inner sep=0}, Rightarrow, no head, from=2-5, to=3-5]
	\arrow["{z_3}", from=2-7, to=2-6]
	\arrow[""{name=8, anchor=center, inner sep=0}, Rightarrow, no head, from=2-7, to=3-7]
	\arrow["{{f_2}}"', from=3-1, to=3-2]
	\arrow[""{name=9, anchor=center, inner sep=0}, "{{x_2}}"', from=3-2, to=2-2]
	\arrow["{{r_2}}", from=3-3, to=3-2]
	\arrow["{{g_2}}"', from=3-3, to=3-4]
	\arrow[""{name=10, anchor=center, inner sep=0}, "{{y_2}}"', from=3-4, to=2-4]
	\arrow["{{s_2}}", from=3-5, to=3-4]
	\arrow["{h_2}"', from=3-5, to=3-6]
	\arrow[""{name=11, anchor=center, inner sep=0}, "{z_2}", from=3-6, to=2-6]
	\arrow["{t_2}", from=3-7, to=3-6]
	\arrow["{{\SIGMA^{\delta_1}}}"{description}, draw=none, from=0, to=1]
	\arrow["{{\SIGMA^{\epsilon_1}}}"{description}, draw=none, from=2, to=3]
	\arrow["{\SIGMA^{\zeta_1}}"{description}, draw=none, from=4, to=5]
	\arrow["{{\SIGMA^{\delta_2}}}"{description}, draw=none, from=9, to=6]
	\arrow["{{\SIGMA^{\epsilon_2}}}"{description}, draw=none, from=10, to=7]
	\arrow["{\SIGMA^{\zeta_2}}"{description}, draw=none, from=11, to=8]
\end{tikzcd}\]

(a) \emph{Formation of $\bar{\beta}\circ \bar{\alpha}$:} \\ We apply Rule 6 to obtain
\begin{equation}\label{eq:Ass1}\begin{tikzcd}
	B & X &&& B & X \\
	{J_1} &&& {J_1} & {J_2} \\
	Y & V &&& Y & V
	\arrow["{{x_3}}", from=1-1, to=1-2]
	\arrow["{{g_1}}"', from=1-1, to=2-1]
	\arrow[""{name=0, anchor=center, inner sep=0}, "{{y'_1}}", from=1-2, to=3-2]
	\arrow[""{name=1, anchor=center, inner sep=0}, "{{y'_2}}"{pos=0.6}, curve={height=-40pt}, from=1-2, to=3-2]
	\arrow["{{x_3}}", from=1-5, to=1-6]
	\arrow["{{g_1}}"', curve={height=6pt}, from=1-5, to=2-4]
	\arrow["{{g_2}}", from=1-5, to=2-5]
	\arrow[""{name=2, anchor=center, inner sep=0}, "{{y'_2}}", from=1-6, to=3-6]
	\arrow["{{y_1}}"', from=2-1, to=3-1]
	\arrow["\beta", Rightarrow, from=2-4, to=2-5]
	\arrow["{{y_1}}"', curve={height=6pt}, from=2-4, to=3-5]
	\arrow["{{y_2}}", from=2-5, to=3-5]
	\arrow["v", from=3-1, to=3-2]
	\arrow["v", from=3-5, to=3-6]
	\arrow["{{\text{\normalsize =}}}"{description, pos=0.6}, draw=none, from=1, to=2-4]
	\arrow["{{\beta'}}"{pos=0.6}, shorten <=12pt, shorten >=6pt, Rightarrow, from=0, to=1]
	\arrow["{{\SIGMA^{\xi_1}}}"{description}, draw=none, from=2-1, to=0]
	\arrow["{{\SIGMA^{\xi_2}}}"{description}, draw=none, from=2-5, to=2]
\end{tikzcd}
\end{equation}
and get the following 2-morphism
\begin{tikzcd}
	A & {I_1} & X & V & Y & C \\
	&&& V & Y & C \\
	A & {I_2} & X & V & Y & C
	\arrow["{f_1}", from=1-1, to=1-2]
	\arrow[Rightarrow, no head, from=1-1, to=3-1]
	\arrow["{{x_1}}", from=1-2, to=1-3]
	\arrow["{{y_1'}}", from=1-3, to=1-4]
	\arrow["\alpha"', shorten <=11pt, shorten >=11pt, Rightarrow, from=1-3, to=3-1]
	\arrow[Rightarrow, no head, from=1-3, to=3-3]
	\arrow["v", tail reversed, no head, from=1-4, to=1-5]
	\arrow[""{name=0, anchor=center, inner sep=0}, Rightarrow, no head, from=1-4, to=2-4]
	\arrow["{{\beta'}}"', shorten <=11pt, shorten >=8pt, Rightarrow, from=1-4, to=3-3]
	\arrow["{y_3}", tail reversed, no head, from=1-5, to=1-6]
	\arrow[""{name=1, anchor=center, inner sep=0}, Rightarrow, no head, from=1-6, to=2-6]
	\arrow["v", from=2-5, to=2-4]
	\arrow["{y_3}", from=2-6, to=2-5]
	\arrow[""{name=2, anchor=center, inner sep=0}, Rightarrow, no head, from=2-6, to=3-6]
	\arrow["{f_2}"', from=3-1, to=3-2]
	\arrow["{{x_2}}"', from=3-2, to=3-3]
	\arrow["{{y'_2}}"', from=3-3, to=3-4]
	\arrow[""{name=3, anchor=center, inner sep=0}, Rightarrow, no head, from=3-4, to=2-4]
	\arrow["v", from=3-5, to=3-4]
	\arrow["{y_3}", from=3-6, to=3-5]
	\arrow["{\SIGMA^{\id}}"{description}, draw=none, from=0, to=1]
	\arrow["{\SIGMA^{\id}}"{description}, draw=none, from=3, to=2]
\end{tikzcd}\, . \,
Now, consider the $\Sigma$-path
\begin{equation}\label{eq:S-sch1}
(\Omega_i) \quad
\adjustbox{scale=0.90}{\begin{tikzcd}
	& B & {I_i} \\
	C & {J_i} & {\dot{B}_i} \\
	C & {J_i} & {\dot{B}_i}
	\arrow["{{r_i}}", from=1-2, to=1-3]
	\arrow[""{name=0, anchor=center, inner sep=0}, "{{g_i}}"', from=1-2, to=2-2]
	\arrow[""{name=1, anchor=center, inner sep=0}, "{{\dot{g}_i}}", from=1-3, to=2-3]
	\arrow["{{s_i}}", from=2-1, to=2-2]
	\arrow[""{name=2, anchor=center, inner sep=0}, equals, from=2-1, to=3-1]
	\arrow["{{\dot{r}_i}}", from=2-2, to=2-3]
	\arrow[""{name=3, anchor=center, inner sep=0}, equals, from=2-2, to=3-2]
	\arrow[""{name=4, anchor=center, inner sep=0}, equals, from=2-3, to=3-3]
	\arrow["{s_i}", from=3-1, to=3-2]
	\arrow["{{\dot{r}_i}}", from=3-2, to=3-3]
	\arrow["{{\SIGMA^{\dot{\alpha}_i}}}"{description}, draw=none, from=0, to=1]
	\arrow["{\SIGMA^{\id}}"{description}, draw=none, from=2, to=3]
	\arrow["{\SIGMA^{\id}}"{description}, draw=none, from=3, to=4]
\end{tikzcd}}
\begin{tikzcd}
	{} & {}
	\arrow["d", squiggly, from=1-1, to=1-2]
\end{tikzcd}
\adjustbox{scale=0.90}{\begin{tikzcd}
	& B & {I_i} \\
	C & {J_i} & {\dot{B}_i} \\
	C & Y & {}
	\arrow["{{{r_i}}}", from=1-2, to=1-3]
	\arrow[""{name=0, anchor=center, inner sep=0}, "{{{g_i}}}"', from=1-2, to=2-2]
	\arrow[""{name=1, anchor=center, inner sep=0}, "{{{\dot{g}_i}}}", from=1-3, to=2-3]
	\arrow["{{{s_i}}}", from=2-1, to=2-2]
	\arrow[""{name=2, anchor=center, inner sep=0}, equals, from=2-1, to=3-1]
	\arrow["{{{\dot{r}_i}}}"', from=2-2, to=2-3]
	\arrow[""{name=3, anchor=center, inner sep=0}, "{{y_i}}"', from=2-2, to=3-2]
	\arrow[""{name=4, anchor=center, inner sep=0}, from=2-3, to=3-3]
	\arrow[from=3-1, to=3-2]
	\arrow[from=3-2, to=3-3]
	\arrow["{{{\SIGMA^{\dot{\alpha}_i}}}}"{description}, draw=none, from=0, to=1]
	\arrow["{{\SIGMA^{\epsilon_i}}}"{description}, draw=none, from=2, to=3]
	\arrow["{{\SIGMAc}}"{description}, draw=none, from=3, to=4]
\end{tikzcd}}
\begin{tikzcd}
	{} & {}
	\arrow["u", squiggly, from=1-1, to=1-2]
\end{tikzcd}
\adjustbox{scale=0.90}{\begin{tikzcd}
	& B & {I_i} \\
	C & {J_i} \\
	C & Y & V
	\arrow["{{r_i}}", from=1-2, to=1-3]
	\arrow["{{g_i}}", from=1-2, to=2-2]
	\arrow[""{name=0, anchor=center, inner sep=0}, "{{y'_ix_i}}", from=1-3, to=3-3]
	\arrow["{{s_i}}", from=2-1, to=2-2]
	\arrow[""{name=1, anchor=center, inner sep=0}, equals, from=2-1, to=3-1]
	\arrow[""{name=2, anchor=center, inner sep=0}, "{{y_i}}", from=2-2, to=3-2]
	\arrow["{{y_3}}"', from=3-1, to=3-2]
	\arrow["v"', from=3-2, to=3-3]
	\arrow["{{\SIGMA^{\epsilon_i}}}"{description}, draw=none, from=1, to=2]
	\arrow["{\SIGMA^{\xi_i\odot \delta_i}}"{description}, draw=none, from=2-2, to=0]
\end{tikzcd}}
\end{equation}
and let
$$\Omega_i\colon (\dot{g}_if_i,\dot{r}_is_i)\Rightarrow (y'_ix_if_i,vy_3)\qquad (i=1,2)$$
be the corresponding   $\Omega$ 2-cell.
Thus, according to \cref{def:horizontal}, $\bar{\beta}\circ \bar{\alpha}$ is given by
$$\bar{g}_1\circ \bar{f}_1=(\dot{g}_1f_1,\dot{r}_1s_1)\xRightarrow{\Omega_1}(y'_1x_1f_1,vy_3)\xRightarrow{[\beta'\circ \alpha, 1_V,1_V]}(y'_2x_2f_2,vy_3)\xRightarrow{\Omega_2^{-1}}(\dot{g}_2f_2,\dot{r}_2s_2)=\bar{g}_2\circ \bar{f}_2\; .$$

(b) \emph{Formation of $\bar{\gamma}\circ(\bar{\beta}\circ \bar{\alpha})$  and $(\bar{\gamma}\circ \bar{\beta})\circ \bar{\alpha}$:}

First use Rule 6 to find
\begin{equation}\label{eq:Ass2}\begin{tikzcd}
	C && Y & C && Y \\
	Z && {V_1} & Z && {V_1}
	\arrow["{{{y_3}}}", from=1-1, to=1-3]
	\arrow[""{name=0, anchor=center, inner sep=0}, "{{z_1h_1}}"', curve={height=30pt}, from=1-1, to=2-1]
	\arrow[""{name=1, anchor=center, inner sep=0}, "{{z_2h_2}}"{description}, curve={height=-30pt}, from=1-1, to=2-1]
	\arrow[""{name=2, anchor=center, inner sep=0}, "{{z'_2}}", from=1-3, to=2-3]
	\arrow["{{{y_3}}}", from=1-4, to=1-6]
	\arrow[""{name=3, anchor=center, inner sep=0}, "{{z_1h_1}}"{description}, from=1-4, to=2-4]
	\arrow[""{name=4, anchor=center, inner sep=0}, "{{{z'_1}}}", curve={height=30pt}, from=1-6, to=2-6]
	\arrow[""{name=5, anchor=center, inner sep=0}, "{{{z'_2}}}", curve={height=-30pt}, from=1-6, to=2-6]
	\arrow["{{{v_1}}}"', from=2-1, to=2-3]
	\arrow["{{{v_1}}}"{description}, from=2-4, to=2-6]
	\arrow["\gamma"{description}, between={0.2}{0.8}, Rightarrow, from=0, to=1]
	\arrow["{{{\SIGMA^{\mu_2}}}}"{description, pos=0.5}, draw=none, from=1, to=2]
	\arrow["{{{\text{\normalsize =}}}}"{description}, draw=none, from=2, to=3]
	\arrow["{{{\SIGMA^{\mu_1}}}}"{description, pos=0.6}, draw=none, from=3, to=4]
	\arrow["{{{\gamma'}}}"{description}, between={0.2}{0.8}, Rightarrow, from=4, to=5]
\end{tikzcd}
\end{equation}
and, subsequently, also
\begin{equation}\label{eq:Ass3}\begin{tikzcd}
	Y && V & Y && V \\
	{V_1} && {V_2} & {V_1} && {V_2}
	\arrow["v", from=1-1, to=1-3]
	\arrow[""{name=0, anchor=center, inner sep=0}, "{{{z'_1}}}"', curve={height=30pt}, from=1-1, to=2-1]
	\arrow[""{name=1, anchor=center, inner sep=0}, "{{{z'_2}}}", curve={height=-30pt}, from=1-1, to=2-1]
	\arrow[""{name=2, anchor=center, inner sep=0}, "{{{{z''_2}}}}", from=1-3, to=2-3]
	\arrow["v", from=1-4, to=1-6]
	\arrow[""{name=3, anchor=center, inner sep=0}, "{{{z'_1}}}", from=1-4, to=2-4]
	\arrow[""{name=4, anchor=center, inner sep=0}, "{{{{z''_1}}}}", curve={height=30pt}, from=1-6, to=2-6]
	\arrow[""{name=5, anchor=center, inner sep=0}, "{{{{z''_2}}}}", curve={height=-30pt}, from=1-6, to=2-6]
	\arrow["{{{{v_2}}}}"', from=2-1, to=2-3]
	\arrow["{{{{v_2}}}}"', from=2-4, to=2-6]
	\arrow["{{\gamma'}}"{description}, shorten <=12pt, shorten >=12pt, Rightarrow, from=0, to=1]
	\arrow["{\SIGMA^{\mu'_2}}"{description}, draw=none, from=1, to=2]
	\arrow["{{{{\text{\normalsize =}}}}}"{description}, draw=none, from=2, to=3]
	\arrow["{\SIGMA^{\mu'_1}}"{description}, draw=none, from=3, to=4]
	\arrow["{{{{\gamma''}}}}"{description}, shorten <=12pt, shorten >=10pt, Rightarrow, from=4, to=5]
\end{tikzcd}\; .
\end{equation}
These data will be used in both compositions.

(b1) Observe that, using the interchange law, already proven to be valid, we have that
$$\bar{\gamma}\circ (\bar{\beta}\circ \bar{\alpha})=(\id_{\bar{h}_2}\circ \Omega_2^{-1})\cdot (\bar{\gamma}\circ [\beta'\circ \alpha,1,1])\cdot (\id_{\bar{h}_1}\circ \Omega_1)\, .$$
Concerning $\bar{\gamma}\circ [\beta'\circ \alpha,1,1]$, from \eqref{eq:Ass2} and \eqref{eq:Ass3}, we have that
$$\bar{\gamma}\circ [\beta'\circ \alpha,1,1]=\Omega_{12}^{-1} \cdot [\gamma''\circ \beta'\circ \alpha, 1_{V_2}, 1_{V_2}]\cdot \Omega_{11}$$
where
\begin{equation}\label{eq:Omega1i}\Omega_{1i}\colon \bar{h}_i\circ (y'_ix_if_i,vy_3)=(\dot{h}_{1i}y'_ix_if_i,\dot{v}t_1)\Rightarrow (z''_iy'_1x_1f_1,v_2v_1z_3)\end{equation}
are the obvious $\Omega$ 2-cells, in accordance with Definition \ref{def:horizontal}.
Hence,
\begin{equation}
\label{eq:c(ba)}\bar{\gamma}\circ (\bar{\beta}\circ \bar{\alpha})=(\id_{\bar{h}_2}\circ \Omega_2^{-1})\cdot \Omega_{12}^{-1} \cdot [\gamma''\circ \beta'\circ \alpha, 1_{V_2}, 1_{V_2}]\cdot \Omega_{11}\cdot (\id_{\bar{h}_1}\circ \Omega_1)\,.
\end{equation}

Note that, combining the formation of $(y'_ix_i,vy_3)$ with the definition of composition of $\Sigma$-cospans, we see that the 2-cell $\Omega_{1i}$ is determined by the following $\Sigma$-path between $\Sigma$-schemes of level 3 (where $\dot{\Sigma}$ indicates a canonical $\Sigma$-square) and we use $\Sigma$-squares from \cref{eq:Ass2} and \cref{eq:Ass3}.
\begin{equation}\label{eq:S-sch2}
(\Omega_{1i}) \quad
\adjustbox{scale=0.70}{\begin{tikzcd}
	&& {B_i} & {I_i} \\
	& C & J \\
	& C & Y & V \\
	D & {K_i} && {} \\
	D & Z && {}
	\arrow["{{{r_i}}}", from=1-3, to=1-4]
	\arrow["{{{g_i}}}"', from=1-3, to=2-3]
	\arrow[""{name=0, anchor=center, inner sep=0}, "{{{y'_ix_i}}}", from=1-4, to=3-4]
	\arrow["{{{s_i}}}", from=2-2, to=2-3]
	\arrow[""{name=1, anchor=center, inner sep=0}, equals, from=2-2, to=3-2]
	\arrow[""{name=2, anchor=center, inner sep=0}, "{{{y_i}}}", from=2-3, to=3-3]
	\arrow["{{{y_3}}}", from=3-2, to=3-3]
	\arrow[""{name=3, anchor=center, inner sep=0}, "{{{h_i}}}"', from=3-2, to=4-2]
	\arrow["v", from=3-3, to=3-4]
	\arrow[""{name=4, anchor=center, inner sep=0}, "{{{\dot{h}_{1i}}}}", from=3-4, to=4-4]
	\arrow["{{{t_i}}}", from=4-1, to=4-2]
	\arrow[""{name=5, anchor=center, inner sep=0}, equals, from=4-1, to=5-1]
	\arrow["{{{\dot{v}_i}}}", from=4-2, to=4-4]
	\arrow[""{name=6, anchor=center, inner sep=0}, equals, from=4-2, to=5-2]
	\arrow[""{name=7, anchor=center, inner sep=0}, equals, from=4-4, to=5-4]
	\arrow["{{t_i}}", from=5-1, to=5-2]
	\arrow["{{{\dot{v}_i}}}", from=5-2, to=5-4]
	\arrow["\SIGMA"{description}, draw=none, from=1, to=2]
	\arrow["\SIGMA"{description}, draw=none, from=2-3, to=0]
	\arrow["{\SIGMAc}"{description}, shift left, draw=none, from=3, to=4]
	\arrow["{\SIGMA^{\id}}"{description}, draw=none, from=5, to=6]
	\arrow["{\SIGMA^{\id}}"{description}, shift left, draw=none, from=6, to=7]
\end{tikzcd}}
\hspace*{-1mm}\begin{tikzcd}
	{} & {}
	\arrow["\mathbf{d}", squiggly, from=1-1, to=1-2]
\end{tikzcd} \hspace*{-5mm}
\adjustbox{scale=0.70}{\begin{tikzcd}
	&& {B_i} & {I_i} \\
	& C & J \\
	& C & Y & V \\
	D & {K_i} && {} \\
	D & Z && {}
	\arrow["{r_i}", from=1-3, to=1-4]
	\arrow["{g_i}"', from=1-3, to=2-3]
	\arrow[""{name=0, anchor=center, inner sep=0}, "{y'_ix_i}", from=1-4, to=3-4]
	\arrow["{s_i}", from=2-2, to=2-3]
	\arrow[""{name=1, anchor=center, inner sep=0}, equals, from=2-2, to=3-2]
	\arrow[""{name=2, anchor=center, inner sep=0}, "{y_i}", from=2-3, to=3-3]
	\arrow["{y_3}", from=3-2, to=3-3]
	\arrow[""{name=3, anchor=center, inner sep=0}, "{h_i}"', from=3-2, to=4-2]
	\arrow["v", from=3-3, to=3-4]
	\arrow[""{name=4, anchor=center, inner sep=0}, "{\dot{h}_{1i}}", from=3-4, to=4-4]
	\arrow["{t_i}", from=4-1, to=4-2]
	\arrow[""{name=5, anchor=center, inner sep=0}, equals, from=4-1, to=5-1]
	\arrow["{\dot{v}_i}", from=4-2, to=4-4]
	\arrow[""{name=6, anchor=center, inner sep=0}, "{z_i}", from=4-2, to=5-2]
	\arrow[""{name=7, anchor=center, inner sep=0}, from=4-4, to=5-4]
	\arrow["{z_3}"', from=5-1, to=5-2]
	\arrow[from=5-2, to=5-4]
	\arrow["\SIGMA"{description}, draw=none, from=1, to=2]
	\arrow["\SIGMA"{description}, draw=none, from=2-3, to=0]
	\arrow["{\SIGMAc}"{description}, shift left, draw=none, from=3, to=4]
	\arrow["\SIGMA"{description}, draw=none, from=5, to=6]
	\arrow["\SIGMAc"{description}, draw=none, from=6, to=7]
\end{tikzcd}}
\hspace*{-1mm}\begin{tikzcd}
	{} & {}
	\arrow["\mathbf{s}_1", squiggly, from=1-1, to=1-2]
\end{tikzcd}\hspace*{-5mm}
\adjustbox{scale=0.70}{\begin{tikzcd}
	&& {B_i} & {I_i} \\
	& C & J \\
	& C & Y & V \\
	D & {K_i} \\
	D & Z & {V_1} & {V_2}
	\arrow["{{r_i}}", from=1-3, to=1-4]
	\arrow["{{g_i}}"', from=1-3, to=2-3]
	\arrow[""{name=0, anchor=center, inner sep=0}, "{{y'_ix_i}}", from=1-4, to=3-4]
	\arrow["{{s_i}}", from=2-2, to=2-3]
	\arrow[""{name=1, anchor=center, inner sep=0}, equals, from=2-2, to=3-2]
	\arrow[""{name=2, anchor=center, inner sep=0}, "{{y_i}}", from=2-3, to=3-3]
	\arrow["{{y_3}}", from=3-2, to=3-3]
	\arrow["{{h_i}}"', from=3-2, to=4-2]
	\arrow["v", from=3-3, to=3-4]
	\arrow[""{name=3, anchor=center, inner sep=0}, "{z^{\prime}_i}", from=3-3, to=5-3]
	\arrow[""{name=4, anchor=center, inner sep=0}, "{z^{\prime\prime}_i}", from=3-4, to=5-4]
	\arrow["{{t_i}}", from=4-1, to=4-2]
	\arrow[""{name=5, anchor=center, inner sep=0}, equals, from=4-1, to=5-1]
	\arrow[""{name=6, anchor=center, inner sep=0}, "{{z_i}}", from=4-2, to=5-2]
	\arrow["{{z_3}}"', from=5-1, to=5-2]
	\arrow["{v_1}"', from=5-2, to=5-3]
	\arrow["{v_2}"', from=5-3, to=5-4]
	\arrow["\SIGMA"{description}, draw=none, from=1, to=2]
	\arrow["\SIGMA"{description}, draw=none, from=2-3, to=0]
	\arrow["{\SIGMA^{\mu'_i}}"{description}, draw=none, from=3, to=4]
	\arrow["\SIGMA"{description}, draw=none, from=5, to=6]
	\arrow["{\SIGMA^{\mu_i}}"{description}, draw=none, from=4-2, to=3]
\end{tikzcd}}\, .
\end{equation}

(b2) Concerning $\bar{\gamma}\circ \bar{\beta}$, let
$$\Omega_{2i}\colon \bar{h}_i\circ \bar{g}_i=(\dot{h}_ig_i,\dot{s}_it_i)\Rightarrow (z'_iy_ig_i, v_1z_3)$$
be the 2-cell of $\Omega$ type obtained via the following  $\Sigma$-path between $\Sigma$-schemes of level 2:
\begin{equation}\label{eq:S-sch3}
(\Omega_{2i}) \quad
\adjustbox{scale=0.90}{\begin{tikzcd}
	& {C} & {J_i} \\
	D & {K_i} & {\dot{C}_i} \\
	D & {K_i} & {\dot{C}_i}
	\arrow["{{s_i}}", from=1-2, to=1-3]
	\arrow[""{name=0, anchor=center, inner sep=0}, "{{h_i}}"', from=1-2, to=2-2]
	\arrow[""{name=1, anchor=center, inner sep=0}, "{\dot{h}_i}", from=1-3, to=2-3]
	\arrow["{{t_i}}", from=2-1, to=2-2]
	\arrow[""{name=2, anchor=center, inner sep=0}, equals, from=2-1, to=3-1]
	\arrow["{\dot{s}_i}", from=2-2, to=2-3]
	\arrow[""{name=3, anchor=center, inner sep=0}, equals, from=2-2, to=3-2]
	\arrow[""{name=4, anchor=center, inner sep=0}, equals, from=2-3, to=3-3]
	\arrow["{t_i}", from=3-1, to=3-2]
	\arrow["{\dot{s}_i}", from=3-2, to=3-3]
	\arrow["{\SIGMAc}"{description}, draw=none, from=0, to=1]
	\arrow["{\SIGMA^{\id}}"{description}, draw=none, from=2, to=3]
	\arrow["{\SIGMA^{\id}}"{description}, draw=none, from=3, to=4]
\end{tikzcd}}
\hspace*{-1mm}\begin{tikzcd}
	{} & {}
	\arrow["d", squiggly, from=1-1, to=1-2]
\end{tikzcd}\hspace*{-2mm}
\adjustbox{scale=0.90}{\begin{tikzcd}
	& {C} & {J_i} \\
	D & {K_i} & {\dot{C}_i} \\
	D & {K_i} & {}
	\arrow["{{{s_i}}}", from=1-2, to=1-3]
	\arrow[""{name=0, anchor=center, inner sep=0}, "{{{h_i}}}"', from=1-2, to=2-2]
	\arrow[""{name=1, anchor=center, inner sep=0}, "{{\dot{h}_i}}", from=1-3, to=2-3]
	\arrow["{{{t_i}}}", from=2-1, to=2-2]
	\arrow[""{name=2, anchor=center, inner sep=0}, equals, from=2-1, to=3-1]
	\arrow["{{\dot{s}_i}}", from=2-2, to=2-3]
	\arrow[""{name=3, anchor=center, inner sep=0}, "{z_i}", from=2-2, to=3-2]
	\arrow[""{name=4, anchor=center, inner sep=0}, from=2-3, to=3-3]
	\arrow["{z_3}", from=3-1, to=3-2]
	\arrow[from=3-2, to=3-3]
	\arrow["{{\SIGMAc}}"{description}, draw=none, from=0, to=1]
	\arrow["\SIGMA"{description}, draw=none, from=2, to=3]
	\arrow["\SIGMA"{description}, draw=none, from=3, to=4]
\end{tikzcd}}
\hspace*{-1mm}\begin{tikzcd}
	{} & {}
	\arrow["u", squiggly, from=1-1, to=1-2]
\end{tikzcd}\hspace*{-5mm}
\adjustbox{scale=0.90}{\begin{tikzcd}
	& C & {J_i} \\
	& C & Y \\
	D & {K_i} \\
	D & Z & {V_1}
	\arrow["{{{s_i}}}", from=1-2, to=1-3]
	\arrow[""{name=0, anchor=center, inner sep=0}, equals, from=1-2, to=2-2]
	\arrow[""{name=1, anchor=center, inner sep=0}, "{{{y_i}}}", from=1-3, to=2-3]
	\arrow["{{{y_3}}}", from=2-2, to=2-3]
	\arrow["{{{h_i}}}"', from=2-2, to=3-2]
	\arrow[""{name=2, anchor=center, inner sep=0}, "{{z'_i}}", from=2-3, to=4-3]
	\arrow["{{{t_i}}}", from=3-1, to=3-2]
	\arrow[""{name=3, anchor=center, inner sep=0}, equals, from=3-1, to=4-1]
	\arrow[""{name=4, anchor=center, inner sep=0}, "{{{z_i}}}", from=3-2, to=4-2]
	\arrow["{{{z_3}}}"', from=4-1, to=4-2]
	\arrow["{{{v_1}}}"', from=4-2, to=4-3]
	\arrow["{{{\SIGMA^{\epsilon_i}}}}"{description}, draw=none, from=0, to=1]
	\arrow["{{{\SIGMA^{\zeta_i}}}}"{description}, draw=none, from=3, to=4]
	\arrow["{{{\SIGMA^{\mu_i}}}}"{description}, draw=none, from=3-2, to=2]
\end{tikzcd}}\, .
\end{equation}
Then
$\bar{\gamma}\circ \bar{\beta}=\Omega_{22}^{-1}\cdot [\gamma'\circ \beta, 1_{V_1},1_{V_1}]\cdot \Omega_{21}$
where $[\gamma'\circ \beta, 1_{V_1},1_{V_1}]$ is represented by
\[\begin{tikzcd}
	B & {J_1} & Y & {V_1} & Z & D \\
	B & {J_2} & Y & {V_1} & Z & D
	\arrow["{{{g_1}}}", from=1-1, to=1-2]
	\arrow[equals, from=1-1, to=2-1]
	\arrow["{{{y_1}}}", from=1-2, to=1-3]
	\arrow["\beta", between={0}{0.8}, Rightarrow, from=1-2, to=2-2]
	\arrow[""{name=0, anchor=center, inner sep=0}, "{{{z'_1}}}", from=1-3, to=1-4]
	\arrow[equals, from=1-3, to=2-3]
	\arrow[""{name=1, anchor=center, inner sep=0}, equals, from=1-4, to=2-4]
	\arrow["{{{v_1}}}"', from=1-5, to=1-4]
	\arrow[""{name=2, anchor=center, inner sep=0}, equals, from=1-5, to=2-5]
	\arrow["{{{z_3}}}"', from=1-6, to=1-5]
	\arrow[""{name=3, anchor=center, inner sep=0}, equals, from=1-6, to=2-6]
	\arrow["{{{g_2}}}", from=2-1, to=2-2]
	\arrow["{{{y_2}}}", from=2-2, to=2-3]
	\arrow[""{name=4, anchor=center, inner sep=0}, "{{{z'_2}}}", from=2-3, to=2-4]
	\arrow["{{{v_1}}}"', from=2-5, to=2-4]
	\arrow["{{{z_3}}}"', from=2-6, to=2-5]
	\arrow["{{{\gamma'}}}", between={0.2}{0.6}, Rightarrow, from=0, to=4]
	\arrow["\SIGMA"{description}, draw=none, from=1, to=2]
	\arrow["\SIGMA"{description}, draw=none, from=2, to=3]
\end{tikzcd}\, .\]
Again using the interchange law, we have that
\begin{equation}\label{eq:eqx}
(\bar{\gamma}\circ \bar{\beta})\circ \bar{\alpha}=(\Omega_{22}^{-1}\circ \id_{\bar{f}_2})\cdot ([\gamma'\circ \beta,1_{V_1},1_{V_1}]\circ \bar{\alpha})\cdot (\Omega_{21}\circ \id_{\bar{f}_1}).
\end{equation}
In order to obtain $[\gamma'\circ \beta,1_{V_1},1_{V_1}]\circ \bar{\alpha}$, we will use \cref{eq:Ass1} and \cref{eq:Ass2}, or more precisely, the fact that we have
\[\begin{tikzcd}
	B & X &&& B & X \\
	{J_1} &&& {J_1} & {J_2} \\
	Y & V &&& Y & V \\
	{V_1} & {V_2} &&& {V_1} & {V_2}
	\arrow["{{{{{{x_3}}}}}}", from=1-1, to=1-2]
	\arrow["{{{{{{g_1}}}}}}"', from=1-1, to=2-1]
	\arrow[""{name=0, anchor=center, inner sep=0}, "{{{{{{y'_1}}}}}}", from=1-2, to=3-2]
	\arrow[""{name=1, anchor=center, inner sep=0}, "{{{{{{y'_2}}}}}}"{pos=0.6}, curve={height=-40pt}, from=1-2, to=3-2]
	\arrow["{{{{{{x_3}}}}}}", from=1-5, to=1-6]
	\arrow["{{{{{{g_1}}}}}}"', curve={height=6pt}, from=1-5, to=2-4]
	\arrow["{{{{{{g_2}}}}}}", from=1-5, to=2-5]
	\arrow[""{name=2, anchor=center, inner sep=0}, "{{{{{{y'_2}}}}}}", from=1-6, to=3-6]
	\arrow["{{{{{{y_1}}}}}}"', from=2-1, to=3-1]
	\arrow["\beta", Rightarrow, from=2-4, to=2-5]
	\arrow["{{{{{{y_1}}}}}}"', curve={height=6pt}, from=2-4, to=3-5]
	\arrow["{{{{{{y_2}}}}}}", from=2-5, to=3-5]
	\arrow["v", from=3-1, to=3-2]
	\arrow[""{name=3, anchor=center, inner sep=0}, "{{{{z'_1}}}}"', from=3-1, to=4-1]
	\arrow[""{name=4, anchor=center, inner sep=0}, "{{{{z''_1}}}}", from=3-2, to=4-2]
	\arrow[""{name=5, anchor=center, inner sep=0}, "{{z''_2}}", curve={height=-40pt}, from=3-2, to=4-2]
	\arrow["v", from=3-5, to=3-6]
	\arrow[""{name=6, anchor=center, inner sep=0}, "{{{{z'_2}}}}", from=3-5, to=4-5]
	\arrow[""{name=7, anchor=center, inner sep=0}, "{{{{z'_1}}}}"', curve={height=30pt}, from=3-5, to=4-5]
	\arrow[""{name=8, anchor=center, inner sep=0}, "{{{{z''_2}}}}", from=3-6, to=4-6]
	\arrow["{{{{v_2}}}}", from=4-1, to=4-2]
	\arrow["{{{{v_2}}}}", from=4-5, to=4-6]
	\arrow["{{{{{{\beta'}}}}}}"{pos=0.6}, between={0.4}{0.9}, Rightarrow, from=0, to=1]
	\arrow["{{{{{{\SIGMA^{\xi_1}}}}}}}"{description}, draw=none, from=2-1, to=0]
	\arrow["{{{{{{\SIGMA^{\xi_2}}}}}}}"{description}, draw=none, from=2-5, to=2]
	\arrow["{{{{\SIGMA^{\mu'_1}}}}}"{description}, draw=none, from=3, to=4]
	\arrow["{{{\text{\normalsize =}}}}"{description, pos=0.6}, draw=none, from=5, to=2-4]
	\arrow["{{{{\gamma''}}}}", shift right, between={0.3}{0.8}, Rightarrow, from=4, to=5]
	\arrow["{{{{\gamma'}}}}", between={0.2}{0.8}, Rightarrow, from=7, to=6]
	\arrow["{{{{\SIGMA^{\mu'_2}}}}}"{description}, draw=none, from=6, to=8]
\end{tikzcd}\, .\]
Indeed, we again obtain the 2-morphism $(\gamma''\circ \beta'\circ \alpha, 1_{V_2}, 1_{V_2})$, already appearing in \cref{eq:c(ba)}. More precisely, we have that
$$[\gamma'\circ \beta,1_{V_1},1_{V_1}]\circ \bar{\alpha}=\Omega_{32}^{-1}\cdot [\gamma''\circ \beta'\circ \alpha, 1_{V_2}, 1_{V_2}]\cdot \Omega_{31}\, ,$$
where
$$\Omega_{3i}\colon (z'_iy_ig_i,v_1z_3)\circ \bar{f}_i=(\dot{z}_if_i,\dot{r}_{3i}v_1z_3) \Rightarrow (z''_iy'_ix_if_i,v_2v_1z_3)$$
is the basic $\Omega$ 2-cell determined by the $\Sigma$-path
\begin{equation*}
 \begin{tikzcd}
	& B & {I_i} \\
	D & {V_1} & {\dot{B}_{3i}} \\
	D & {V_1} & {\dot{B}_{3i}}
	\arrow["{{{{{r_i}}}}}", from=1-2, to=1-3]
	\arrow[""{name=0, anchor=center, inner sep=0}, "{{z'_iy_ig_i}}"', from=1-2, to=2-2]
	\arrow[""{name=1, anchor=center, inner sep=0}, "{{\dot{z}_i}}", from=1-3, to=2-3]
	\arrow["{{v_1z_3}}", from=2-1, to=2-2]
	\arrow[""{name=2, anchor=center, inner sep=0}, equals, from=2-1, to=3-1]
	\arrow["{\dot{r}_{3i}}", from=2-2, to=2-3]
	\arrow[""{name=3, anchor=center, inner sep=0}, equals, from=2-2, to=3-2]
	\arrow[""{name=4, anchor=center, inner sep=0}, equals, from=2-3, to=3-3]
	\arrow["{v_1z_3}", from=3-1, to=3-2]
	\arrow["{\dot{r}_{3i}}", from=3-2, to=3-3]
	\arrow["\SIGMAc"{description}, draw=none, from=0, to=1]
	\arrow["\SIGMAc"{description}, draw=none, from=2, to=3]
	\arrow["\SIGMAc"{description}, draw=none, from=3, to=4]
\end{tikzcd}
\hspace*{-1mm}
\begin{tikzcd}
	{} & {}
	\arrow["d(=\id)", squiggly, from=1-1, to=1-2]
\end{tikzcd}\hspace*{-2mm}
\begin{tikzcd}
	& B & {I_i} \\
	D & {V_1} & {\dot{B}_{3i}} \\
	D & {V_1} & {\dot{B}_{3i}}
	\arrow["{{{{{r_i}}}}}", from=1-2, to=1-3]
	\arrow[""{name=0, anchor=center, inner sep=0}, "{{z'_iy_ig_i}}"', from=1-2, to=2-2]
	\arrow[""{name=1, anchor=center, inner sep=0}, "{{\dot{z}_i}}", from=1-3, to=2-3]
	\arrow["{{v_1z_3}}", from=2-1, to=2-2]
	\arrow[""{name=2, anchor=center, inner sep=0}, equals, from=2-1, to=3-1]
	\arrow["{\dot{r}_{3i}}", from=2-2, to=2-3]
	\arrow[""{name=3, anchor=center, inner sep=0}, equals, from=2-2, to=3-2]
	\arrow[""{name=4, anchor=center, inner sep=0}, equals, from=2-3, to=3-3]
	\arrow["{v_1z_3}", from=3-1, to=3-2]
	\arrow["{\dot{r}_{3i}}", from=3-2, to=3-3]
	\arrow["\SIGMAc"{description}, draw=none, from=0, to=1]
	\arrow["\SIGMAc"{description}, draw=none, from=2, to=3]
	\arrow["\SIGMAc"{description}, draw=none, from=3, to=4]
\end{tikzcd}
\hspace*{-1mm}
\begin{tikzcd}
	{} & {}
	\arrow["u", squiggly, from=1-1, to=1-2]
\end{tikzcd}\hspace*{-2mm}
\begin{tikzcd}
	& B & {I_i} \\
	D & {V_1} \\
	D & {V_1} & {V_2}
	\arrow["{{{{{r_i}}}}}", from=1-2, to=1-3]
	\arrow["{{z'_iy_ig_i}}"', from=1-2, to=2-2]
	\arrow[""{name=0, anchor=center, inner sep=0}, "{{z''_iy'_ix_i}}", from=1-3, to=3-3]
	\arrow["{{v_1z_3}}", from=2-1, to=2-2]
	\arrow[""{name=1, anchor=center, inner sep=0}, equals, from=2-1, to=3-1]
	\arrow[""{name=2, anchor=center, inner sep=0}, equals, from=2-2, to=3-2]
	\arrow["{{v_1z_3}}"', from=3-1, to=3-2]
	\arrow["{v_2}"', from=3-2, to=3-3]
	\arrow["\SIGMAc"{description}, draw=none, from=1, to=2]
	\arrow["\SIGMA"{description, pos=0.3}, draw=none, from=2-2, to=0]
\end{tikzcd}\, ,
\end{equation*}
and thus, also determined just by the $\Sigma$-step below between $\Sg$-schemes of level 3:
\begin{equation}\label{eq:S-sch4}
  (\Omega_{3i}) \quad
 \begin{tikzcd}
	&& B & {I_i} \\
	& C & {J_i} \\
	& C & Y \\
	D & Z \\
	D & Z & {V_1} & {\dot{B}_{3i}}
	\arrow["{{{r_i}}}", from=1-3, to=1-4]
	\arrow["{{{{g_i}}}}"', from=1-3, to=2-3]
	\arrow[""{name=0, anchor=center, inner sep=0}, "{\dot{z}_i}", from=1-4, to=5-4]
	\arrow["{{s_i}}", from=2-2, to=2-3]
	\arrow[""{name=1, anchor=center, inner sep=0}, equals, from=2-2, to=3-2]
	\arrow[""{name=2, anchor=center, inner sep=0}, "{{{{y_i}}}}"', from=2-3, to=3-3]
	\arrow["{{y_3}}", from=3-2, to=3-3]
	\arrow["{h_i}"', from=3-2, to=4-2]
	\arrow[""{name=3, anchor=center, inner sep=0}, "{z'_i}", from=3-3, to=5-3]
	\arrow["{{t_i}}", from=4-1, to=4-2]
	\arrow[""{name=4, anchor=center, inner sep=0}, equals, from=4-1, to=5-1]
	\arrow[""{name=5, anchor=center, inner sep=0}, "{z_i}"', from=4-2, to=5-2]
	\arrow["{z_3}", from=5-1, to=5-2]
	\arrow["{v_1}", from=5-2, to=5-3]
	\arrow["{\dot{r}_{3i}}", from=5-3, to=5-4]
	\arrow["\SIGMA"{description}, draw=none, from=1, to=2]
	\arrow["{\SIGMAc}"{description}, draw=none, from=3-3, to=0]
	\arrow["\SIGMA"{description}, draw=none, from=4, to=5]
	\arrow["\SIGMA"{description}, draw=none, from=4-2, to=3]
\end{tikzcd}
\begin{tikzcd}
	{} & {}
	\arrow["\mathbf{u}", squiggly, from=1-1, to=1-2]
\end{tikzcd}\hspace*{-4mm}
\begin{tikzcd}
	&& B & {I_i} \\
	& C & {J_i} \\
	& C & Y & V \\
	D & Z \\
	D & Z & {V_1} & {V_2} %
	\arrow["{{{{r_i}}}}", from=1-3, to=1-4]
	\arrow["{{{{{g_i}}}}}"', from=1-3, to=2-3]
	\arrow[""{name=0, anchor=center, inner sep=0}, "{y'_ix_i}", from=1-4, to=3-4]
	\arrow["{{{s_i}}}", from=2-2, to=2-3]
	\arrow[""{name=1, anchor=center, inner sep=0}, equals, from=2-2, to=3-2]
	\arrow[""{name=2, anchor=center, inner sep=0}, "{{{{{y_i}}}}}"', from=2-3, to=3-3]
	\arrow["{{{y_3}}}", from=3-2, to=3-3]
	\arrow["{{h_i}}"', from=3-2, to=4-2]
	\arrow["v", from=3-3, to=3-4]
	\arrow[""{name=3, anchor=center, inner sep=0}, "{{z'_i}}", from=3-3, to=5-3]
	\arrow[""{name=4, anchor=center, inner sep=0}, "{z^{\prime\prime}_i}", from=3-4, to=5-4]
	\arrow["{t_i}", from=4-1, to=4-2]
	\arrow[""{name=5, anchor=center, inner sep=0}, equals, from=4-1, to=5-1]
	\arrow[""{name=6, anchor=center, inner sep=0}, "{{z_i}}"', from=4-2, to=5-2]
	\arrow["{{z_3}}", from=5-1, to=5-2]
	\arrow["{{v_1}}", from=5-2, to=5-3]
	\arrow["{v_2}", from=5-3, to=5-4]
	\arrow["\SIGMA"{description}, draw=none, from=1, to=2]
	\arrow["\SIGMA"{description}, draw=none, from=2-3, to=0]
	\arrow["\SIGMA"{description}, draw=none, from=3, to=4]
	\arrow["\SIGMA"{description}, draw=none, from=5, to=6]
	\arrow["\SIGMA"{description}, draw=none, from=4-2, to=3]
\end{tikzcd}\, .
\end{equation}
Hence, using \eqref{eq:eqx}, we have
\begin{equation}\label{eq:(cb)a}(\bar{\gamma}\circ \bar{\beta})\circ \bar{\alpha}=
(\Omega_{22}^{-1}\circ \id_{\bar{f}_2})\cdot \Omega_{32}^{-1}\cdot [\gamma''\circ \beta'\circ \alpha, 1_{V_2}, 1_{V_2}]\cdot \Omega_{31}\cdot (\Omega_{21}\circ \id_{\bar{f}_1})
\,.\end{equation}
Put
$$\hat{\Omega}_i=\Omega_{3i}\cdot (\Omega_{2i}\circ \id_{\bar{f}_i})$$
and
$$\tilde{\Omega}_i=\Omega_{1i}\cdot (\id_{\bar{h}_i}\circ \Omega_i)\, .$$
Using \cref{eq:c(ba)} and \cref{eq:(cb)a}, we see that, in order to prove \eqref{eq:nat-trans}, we just need to show that $\Assoc_{\bar{h}_2,\bar{g}_2,\bar{f}_2}\cdot \hat{\Omega}_2^{-1}=\tilde{\Omega}_2^{-1}$ and $\hat{\Omega}_1=\tilde{\Omega}_1\cdot \Assoc_{\bar{h}_1,\bar{g}_1,\bar{f}_1}$. That is, we need to prove that
\begin{equation}\label{eq:equality}\Assoc_{\bar{h}_i,\bar{g}_i,\bar{f}_i}=\tilde{\Omega}_i^{-1}\cdot \hat{\Omega}_i \; \; \; \; \; (i=1,2).\end{equation}
We know that the first member of this equality is an $\Omega$ 2-cell corresponding to a $\Sg$-path between $\Sg$-schemes of level 2, and clearly it can be seen as corresponding to a $\Sg$-path of interest between $\Sg$-schemes of level 3 with left border $(r_i,g_i,s_i,h_i,t_i,1)$.  The 2-cells $\Omega_{1i}$ and $\Omega_{3i}$ correspond to $\Sg$-paths of interest, namely \eqref{eq:S-sch2} and \eqref{eq:S-sch4}. The 2-cells $\Omega_i$ and $\Omega_{2i}$ correspond to $\Sg$-steps between $\Sg$-schemes of level 2, each one made of $\Sg$-steps of type $d$ and $u$, see \eqref{eq:S-sch1} and \eqref{eq:S-sch3}. Then by \cref{pro:alpha.f} and \cref{cor:Omega-h},
$\Omega_{2i}\circ \id_{\bar{f}}$ and $\id_{\bar{h}}\circ \Omega_i$ are $\Omega$ 2-cells corresponding to $\Sg$-paths of interest. Since both members of \eqref{eq:equality} correspond to $\Sg$-paths of interest between $\Sg$-schemes of left border $(r_i,g_i,s_i,h_i,t_i,1)$, they coincide.
\end{proof}

\begin{proposition}The isomorphisms of Definition \ref{def:assoc} fulfil the Pentagon Axiom.
\end{proposition}

\begin{proof}
Let
$\xymatrix{A   \ar[r]^{\bar{f}}    &   B\ar[r]^{\bar{g}}&C\ar[r]^{\bar{h}}&D\ar[r]^{\bar{k}}&E}$
be composable 1-cells in $\catx[\Sigma_*]$, where $\bar{f}=(f,r)$, $\bar{g}=(g,s)$, $\bar{h}=(h,t)$ and $\bar{k}=(k,u)$. We want to prove the commutativity of the diagram
\[\begin{tikzcd}
	&& {(\bar{k}\bar{h})(\bar{g}\bar{f})} \\
	{((\bar{k}\bar{h})\bar{g})\bar{f}} &&&& {\bar{k}(\bar{h}(\bar{g}\bar{f}))} \\
	& {(\bar{k}(\bar{h}\bar{g}))\bar{f}} && {\bar{k}((\bar{h}\bar{g})\bar{f})}
	\arrow["{\Assoc_{\bar{k},\bar{h},\bar{g}\bar{f}}}"{pos=0.6}, from=1-3, to=2-5]
	\arrow["{\Assoc_{\bar{k}\bar{h},\bar{g},\bar{f}}}", from=2-1, to=1-3]
	\arrow["{\Assoc_{\bar{k},\bar{h},\bar{g}}\circ 1_{\bar{f}}}"'{pos=0.1}, from=2-1, to=3-2]
	\arrow["{{\Assoc_{\bar{k}, \bar{h}\bar{g}, \bar{f}}}}"', from=3-2, to=3-4]
	\arrow["{{1_{\bar{k}}\circ\Assoc_{\bar{h},\bar{g},\bar{f}}}}"'{pos=1}, from=3-4, to=2-5]
\end{tikzcd}\]
in $\catx[\Sg_*](A,E)$.

Concerning the top of the pentagon diagram,  using the definitions of composition of $\Sg$-cospans and of associator, we see that
$\Assoc_{\bar{k}\bar{h},\bar{g},\bar{f}}$  is an  $\Omega$  2-cell corresponding to a $\Sg$-path of interest of the form
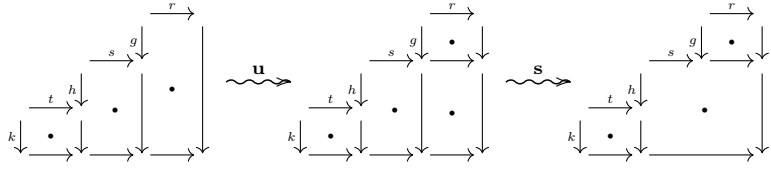
\[
\adjustbox{scale=0.70}{\begin{tikzcd}
	&& {} & {} \\
	& {} & {} \\
	{} & {} \\
	{} & {} & {} & {}
	\arrow["r", from=1-3, to=1-4]
	\arrow["g"', from=1-3, to=2-3]
	\arrow[""{name=0, anchor=center, inner sep=0}, draw=none, from=1-3, to=4-3]
	\arrow[""{name=1, anchor=center, inner sep=0}, from=1-4, to=4-4]
	\arrow["s", from=2-2, to=2-3]
	\arrow["h"', from=2-2, to=3-2]
	\arrow[""{name=2, anchor=center, inner sep=0}, from=2-3, to=4-3]
	\arrow["t", from=3-1, to=3-2]
	\arrow[""{name=3, anchor=center, inner sep=0}, "k"', from=3-1, to=4-1]
	\arrow[""{name=4, anchor=center, inner sep=0}, from=3-2, to=4-2]
	\arrow[from=4-1, to=4-2]
	\arrow[from=4-2, to=4-3]
	\arrow[from=4-3, to=4-4]
	\arrow["\bullet"{description}, draw=none, from=0, to=1]
	\arrow["\bullet"{description}, draw=none, from=3, to=4]
	\arrow["\bullet"{description}, draw=none, from=3-2, to=2]
\end{tikzcd}}
\xymatrix{\ar@{~>}[r]^{\mathbf{u}}&}
\hspace*{-3mm}
\adjustbox{scale=0.70}{\begin{tikzcd}
	&& {} & {} \\
	& {} & {} & {} \\
	{} & {} \\
	{} & {} & {} & {}
	\arrow["r", from=1-3, to=1-4]
	\arrow[""{name=0, anchor=center, inner sep=0}, "g"', from=1-3, to=2-3]
	\arrow[""{name=1, anchor=center, inner sep=0}, from=1-4, to=2-4]
	\arrow["s", from=2-2, to=2-3]
	\arrow["h"', from=2-2, to=3-2]
	\arrow[from=2-3, to=2-4]
	\arrow[""{name=2, anchor=center, inner sep=0}, from=2-3, to=4-3]
	\arrow[""{name=3, anchor=center, inner sep=0}, from=2-4, to=4-4]
	\arrow["t", from=3-1, to=3-2]
	\arrow[""{name=4, anchor=center, inner sep=0}, "k"', from=3-1, to=4-1]
	\arrow[""{name=5, anchor=center, inner sep=0}, from=3-2, to=4-2]
	\arrow[from=4-1, to=4-2]
	\arrow[from=4-2, to=4-3]
	\arrow[from=4-3, to=4-4]
	\arrow["\bullet"{description}, draw=none, from=0, to=1]
	\arrow["\bullet"{description}, draw=none, from=2, to=3]
	\arrow["\bullet"{description}, draw=none, from=4, to=5]
	\arrow["\bullet"{description}, draw=none, from=3-2, to=2]
\end{tikzcd}}
\xymatrix{\ar@{~>}[r]^{\mathbf{s}}&}
\hspace*{-3mm}
\adjustbox{scale=0.70}{\begin{tikzcd}
	&& {} & {} \\
	& {} & {} & {} \\
	{} & {} \\
	{} & {} && {}
	\arrow["r", from=1-3, to=1-4]
	\arrow[""{name=0, anchor=center, inner sep=0}, "g"', from=1-3, to=2-3]
	\arrow[""{name=1, anchor=center, inner sep=0}, from=1-4, to=2-4]
	\arrow["s", from=2-2, to=2-3]
	\arrow["h"', from=2-2, to=3-2]
	\arrow[from=2-3, to=2-4]
	\arrow[""{name=2, anchor=center, inner sep=0}, from=2-4, to=4-4]
	\arrow["t", from=3-1, to=3-2]
	\arrow[""{name=3, anchor=center, inner sep=0}, "k"', from=3-1, to=4-1]
	\arrow[""{name=4, anchor=center, inner sep=0}, from=3-2, to=4-2]
	\arrow[from=4-1, to=4-2]
	\arrow[from=4-2, to=4-4]
	\arrow["\bullet"{description}, draw=none, from=0, to=1]
	\arrow["\bullet"{description}, draw=none, from=3, to=4]
	\arrow["\bullet"{description}, draw=none, from=3-2, to=2]
\end{tikzcd}}
\]
where the bullet inside a square means that it is the canonical $\Sigma$-square.
Analogously, $\Assoc_{\bar{k},\bar{h},\bar{g}\bar{f}}$ is an $\Omega$ 2-cell corresponding  to a $\Sg$-path of interest of the form
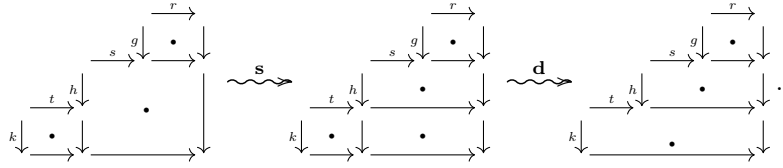
\[
\adjustbox{scale=0.70}{\begin{tikzcd}
	&& {} & {} \\
	& {} & {} & {} \\
	{} & {} \\
	{} & {} && {}
	\arrow["r", from=1-3, to=1-4]
	\arrow[""{name=0, anchor=center, inner sep=0}, "g"', from=1-3, to=2-3]
	\arrow[""{name=1, anchor=center, inner sep=0}, from=1-4, to=2-4]
	\arrow["s", from=2-2, to=2-3]
	\arrow["h"', from=2-2, to=3-2]
	\arrow[from=2-3, to=2-4]
	\arrow[""{name=2, anchor=center, inner sep=0}, from=2-4, to=4-4]
	\arrow["t", from=3-1, to=3-2]
	\arrow[""{name=3, anchor=center, inner sep=0}, "k"', from=3-1, to=4-1]
	\arrow[""{name=4, anchor=center, inner sep=0}, from=3-2, to=4-2]
	\arrow[from=4-1, to=4-2]
	\arrow[from=4-2, to=4-4]
	\arrow["\bullet"{description}, draw=none, from=0, to=1]
	\arrow["\bullet"{description}, draw=none, from=3, to=4]
	\arrow["\bullet"{description}, draw=none, from=3-2, to=2]
\end{tikzcd}}
\xymatrix{\ar@{~>}[r]^{\mathbf{s}}&}
\hspace*{-3mm}
\adjustbox{scale=0.70}{\begin{tikzcd}
	&& {} & {} \\
	& {} & {} & {} \\
	{} & {} && {} \\
	{} & {} && {}
	\arrow["r", from=1-3, to=1-4]
	\arrow[""{name=0, anchor=center, inner sep=0}, "g"', from=1-3, to=2-3]
	\arrow[""{name=1, anchor=center, inner sep=0}, from=1-4, to=2-4]
	\arrow["s", from=2-2, to=2-3]
	\arrow[""{name=2, anchor=center, inner sep=0}, "h"', from=2-2, to=3-2]
	\arrow[from=2-3, to=2-4]
	\arrow[""{name=3, anchor=center, inner sep=0}, from=2-4, to=3-4]
	\arrow["t", from=3-1, to=3-2]
	\arrow[""{name=4, anchor=center, inner sep=0}, "k"', from=3-1, to=4-1]
	\arrow[from=3-2, to=3-4]
	\arrow[""{name=5, anchor=center, inner sep=0}, from=3-2, to=4-2]
	\arrow[""{name=6, anchor=center, inner sep=0}, from=3-4, to=4-4]
	\arrow[from=4-1, to=4-2]
	\arrow[from=4-2, to=4-4]
	\arrow["\bullet"{description}, draw=none, from=0, to=1]
	\arrow["\bullet"{description}, draw=none, from=2, to=3]
	\arrow["\bullet"{description}, draw=none, from=4, to=5]
	\arrow["\bullet"{description}, draw=none, from=5, to=6]
\end{tikzcd}}
\xymatrix{\ar@{~>}[r]^{\mathbf{d}}&}
\hspace*{-3mm}
\adjustbox{scale=0.70}{\begin{tikzcd}
	&& {} & {} \\
	& {} & {} & {} \\
	{} & {} && {} \\
	{} &&& {}
	\arrow["r", from=1-3, to=1-4]
	\arrow[""{name=0, anchor=center, inner sep=0}, "g"', from=1-3, to=2-3]
	\arrow[""{name=1, anchor=center, inner sep=0}, from=1-4, to=2-4]
	\arrow["s", from=2-2, to=2-3]
	\arrow[""{name=2, anchor=center, inner sep=0}, "h"', from=2-2, to=3-2]
	\arrow[from=2-3, to=2-4]
	\arrow[""{name=3, anchor=center, inner sep=0}, from=2-4, to=3-4]
	\arrow["t", from=3-1, to=3-2]
	\arrow[""{name=4, anchor=center, inner sep=0}, "k"', from=3-1, to=4-1]
	\arrow[from=3-2, to=3-4]
	\arrow[""{name=5, anchor=center, inner sep=0}, from=3-4, to=4-4]
	\arrow[from=4-1, to=4-4]
	\arrow["\bullet"{description}, draw=none, from=0, to=1]
	\arrow["\bullet"{description}, draw=none, from=2, to=3]
	\arrow["\bullet"', draw=none, from=4, to=5]
\end{tikzcd}}
\, .\]
The middle map in the bottom line of the pentagon, $\Assoc_{\bar{k}, \bar{h}\bar{g}, \bar{f}}$, is the $\Omega$ 2-cell obtained via the $\Sg$-path of interest
\begin{equation}\label{eq:bot-pent-2}
\adjustbox{scale=0.70}{\begin{tikzcd}
	&& {} & {} \\
	& {} & {} \\
	{} & {} & {} \\
	{} && {} & {} \\
	{} &&& {}
	\arrow["r", from=1-3, to=1-4]
	\arrow["g"', from=1-3, to=2-3]
	\arrow[""{name=0, anchor=center, inner sep=0}, from=1-4, to=4-4]
	\arrow["s", from=2-2, to=2-3]
	\arrow[""{name=1, anchor=center, inner sep=0}, "h"', from=2-2, to=3-2]
	\arrow[""{name=2, anchor=center, inner sep=0}, from=2-3, to=3-3]
	\arrow["t", from=3-1, to=3-2]
	\arrow[""{name=3, anchor=center, inner sep=0}, "k"', from=3-1, to=4-1]
	\arrow[from=3-2, to=3-3]
	\arrow[""{name=4, anchor=center, inner sep=0}, from=3-3, to=4-3]
	\arrow[from=4-1, to=4-3]
	\arrow[from=4-3, to=4-4]
	\arrow["{S_6}", draw=none, from=5-1, to=5-4]
	\arrow["\bullet"{description}, draw=none, from=1, to=2]
	\arrow["\bullet"{description}, draw=none, from=2, to=0]
	\arrow["\bullet"{description}, draw=none, from=3, to=4]
\end{tikzcd}}
\xymatrix{\ar@{~>}[r]^{\mathbf{u}}&}
\hspace*{-3mm}
\adjustbox{scale=0.70}{\begin{tikzcd}
	&& {} & {} \\
	& {} & {} \\
	{} & {} & {} & {} \\
	{} && {} & {} \\
	{} &&& {}
	\arrow["r", from=1-3, to=1-4]
	\arrow["g"', from=1-3, to=2-3]
	\arrow[""{name=0, anchor=center, inner sep=0}, from=1-4, to=3-4]
	\arrow["s", from=2-2, to=2-3]
	\arrow[""{name=1, anchor=center, inner sep=0}, "h"', from=2-2, to=3-2]
	\arrow[""{name=2, anchor=center, inner sep=0}, from=2-3, to=3-3]
	\arrow["t", from=3-1, to=3-2]
	\arrow[""{name=3, anchor=center, inner sep=0}, "k"', from=3-1, to=4-1]
	\arrow[from=3-2, to=3-3]
	\arrow[from=3-3, to=3-4]
	\arrow[""{name=4, anchor=center, inner sep=0}, from=3-3, to=4-3]
	\arrow[""{name=5, anchor=center, inner sep=0}, from=3-4, to=4-4]
	\arrow[from=4-1, to=4-3]
	\arrow[from=4-3, to=4-4]
	\arrow["{S_7}", draw=none, from=5-1, to=5-4]
	\arrow["\bullet"{description}, draw=none, from=1, to=2]
	\arrow["\bullet"{description}, draw=none, from=2-3, to=0]
	\arrow["\bullet"{description}, draw=none, from=3, to=4]
	\arrow["\bullet"{description}, draw=none, from=4, to=5]
\end{tikzcd}}
\xymatrix{\ar@{~>}[r]^{\mathbf{d}}&}
\hspace*{-3mm}
\adjustbox{scale=0.70}{\begin{tikzcd}
	&& {} & {} \\
	& {} & {} \\
	{} & {} & {} & {} \\
	{} &&& {} \\
	{} &&& {}
	\arrow["r", from=1-3, to=1-4]
	\arrow["g"', from=1-3, to=2-3]
	\arrow[""{name=0, anchor=center, inner sep=0}, from=1-4, to=3-4]
	\arrow["s", from=2-2, to=2-3]
	\arrow[""{name=1, anchor=center, inner sep=0}, "h"', from=2-2, to=3-2]
	\arrow[""{name=2, anchor=center, inner sep=0}, from=2-3, to=3-3]
	\arrow["t", from=3-1, to=3-2]
	\arrow[""{name=3, anchor=center, inner sep=0}, "k"', from=3-1, to=4-1]
	\arrow[from=3-2, to=3-3]
	\arrow[from=3-3, to=3-4]
	\arrow[""{name=4, anchor=center, inner sep=0}, from=3-4, to=4-4]
	\arrow[from=4-1, to=4-4]
	\arrow["{S_8}", draw=none, from=5-1, to=5-4]
	\arrow["\bullet"{description}, draw=none, from=1, to=2]
	\arrow["\bullet"{description}, draw=none, from=2-3, to=0]
	\arrow["\bullet"{description}, draw=none, from=3, to=4]
\end{tikzcd}}\, .
\end{equation}

The 2-cells $\Assoc_{\bar{k},\bar{h},\bar{g}}$ and $\Assoc_{\bar{h},\bar{g},\bar{f}}$ correspond, respectively, to the $\Sg$-paths of level 2
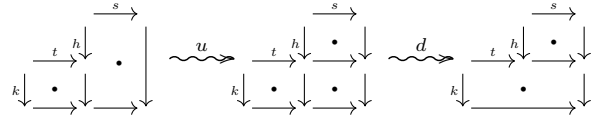
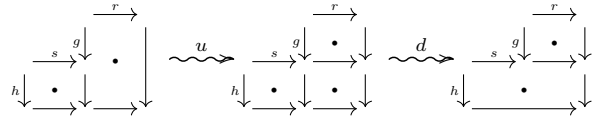
\[
\adjustbox{scale=0.70}{\begin{tikzcd}
	& {} & {} \\
	{} & {} \\
	{} & {} & {}
	\arrow["s", from=1-2, to=1-3]
	\arrow["h"', from=1-2, to=2-2]
	\arrow[""{name=0, anchor=center, inner sep=0}, from=1-3, to=3-3]
	\arrow["t", from=2-1, to=2-2]
	\arrow[""{name=1, anchor=center, inner sep=0}, "k"', from=2-1, to=3-1]
	\arrow[""{name=2, anchor=center, inner sep=0}, from=2-2, to=3-2]
	\arrow[from=3-1, to=3-2]
	\arrow[from=3-2, to=3-3]
	\arrow["\bullet"{description}, draw=none, from=1, to=2]
	\arrow["\bullet"{description}, draw=none, from=2-2, to=0]
\end{tikzcd}}
\xymatrix{\ar@{~>}[r]^{u}&}
\hspace*{-3mm}
\adjustbox{scale=0.70}{\begin{tikzcd}
	& {} & {} \\
	{} & {} & {} \\
	{} & {} & {}
	\arrow["s", from=1-2, to=1-3]
	\arrow[""{name=0, anchor=center, inner sep=0}, "h"', from=1-2, to=2-2]
	\arrow[""{name=1, anchor=center, inner sep=0}, from=1-3, to=2-3]
	\arrow["t", from=2-1, to=2-2]
	\arrow[""{name=2, anchor=center, inner sep=0}, "k"', from=2-1, to=3-1]
	\arrow[from=2-2, to=2-3]
	\arrow[""{name=3, anchor=center, inner sep=0}, from=2-2, to=3-2]
	\arrow[""{name=4, anchor=center, inner sep=0}, from=2-3, to=3-3]
	\arrow[from=3-1, to=3-2]
	\arrow[from=3-2, to=3-3]
	\arrow["\bullet"{description}, draw=none, from=0, to=1]
	\arrow["\bullet"{description}, draw=none, from=2, to=3]
	\arrow["\bullet"{description}, draw=none, from=3, to=4]
\end{tikzcd}}
\xymatrix{\ar@{~>}[r]^{d}&}
\hspace*{-3mm}
\adjustbox{scale=0.70}{\begin{tikzcd}
	& {} & {} \\
	{} & {} & {} \\
	{} && {}
	\arrow["s", from=1-2, to=1-3]
	\arrow[""{name=0, anchor=center, inner sep=0}, "h"', from=1-2, to=2-2]
	\arrow[""{name=1, anchor=center, inner sep=0}, from=1-3, to=2-3]
	\arrow["t", from=2-1, to=2-2]
	\arrow[""{name=2, anchor=center, inner sep=0}, "k"', from=2-1, to=3-1]
	\arrow[from=2-2, to=2-3]
	\arrow[""{name=3, anchor=center, inner sep=0}, from=2-3, to=3-3]
	\arrow[from=3-1, to=3-3]
	\arrow["\bullet"{description}, draw=none, from=0, to=1]
	\arrow["\bullet"{description}, draw=none, from=2, to=3]
\end{tikzcd}}
\]
and
\[
\adjustbox{scale=0.70}{\begin{tikzcd}
	& {} & {} \\
	{} & {} & {} \\
	{} & {} & {}
	\arrow["r", from=1-2, to=1-3]
	\arrow["g"', from=1-2, to=2-2]
	\arrow[from=1-3, to=3-3]
	\arrow["s", from=2-1, to=2-2]
	\arrow[""{name=0, anchor=center, inner sep=0}, "h"', from=2-1, to=3-1]
	\arrow["\bullet"{description}, draw=none, from=2-2, to=2-3]
	\arrow[""{name=1, anchor=center, inner sep=0}, from=2-2, to=3-2]
	\arrow[from=3-1, to=3-2]
	\arrow[from=3-2, to=3-3]
	\arrow["\bullet"{description}, draw=none, from=0, to=1]
\end{tikzcd}}
\xymatrix{\ar@{~>}[r]^{u}&}
\hspace*{-3mm}
\adjustbox{scale=0.70}{\begin{tikzcd}
	& {} & {} \\
	{} & {} & {} \\
	{} & {} & {}
	\arrow["r", from=1-2, to=1-3]
	\arrow[""{name=0, anchor=center, inner sep=0}, "g"', from=1-2, to=2-2]
	\arrow[""{name=1, anchor=center, inner sep=0}, from=1-3, to=2-3]
	\arrow["s", from=2-1, to=2-2]
	\arrow[""{name=2, anchor=center, inner sep=0}, "h"', from=2-1, to=3-1]
	\arrow[from=2-2, to=2-3]
	\arrow[""{name=3, anchor=center, inner sep=0}, from=2-2, to=3-2]
	\arrow[""{name=4, anchor=center, inner sep=0}, from=2-3, to=3-3]
	\arrow[from=3-1, to=3-2]
	\arrow[from=3-2, to=3-3]
	\arrow["\bullet"{description}, draw=none, from=0, to=1]
	\arrow["\bullet"{description}, draw=none, from=2, to=3]
	\arrow["\bullet"{description}, draw=none, from=3, to=4]
\end{tikzcd}}
\xymatrix{\ar@{~>}[r]^{d}&}
\hspace*{-3mm}
\adjustbox{scale=0.70}{\begin{tikzcd}
	& {} & {} \\
	{} & {} & {} \\
	{} && {}
	\arrow["r", from=1-2, to=1-3]
	\arrow[""{name=0, anchor=center, inner sep=0}, "g"', from=1-2, to=2-2]
	\arrow[""{name=1, anchor=center, inner sep=0}, from=1-3, to=2-3]
	\arrow["s", from=2-1, to=2-2]
	\arrow[""{name=2, anchor=center, inner sep=0}, "h"', from=2-1, to=3-1]
	\arrow[from=2-2, to=2-3]
	\arrow[""{name=3, anchor=center, inner sep=0}, from=2-3, to=3-3]
	\arrow[from=3-1, to=3-3]
	\arrow["\bullet"{description}, draw=none, from=0, to=1]
	\arrow["\bullet"{description}, draw=none, from=2, to=3]
\end{tikzcd}}
\]

Hence, by Proposition \ref{pro:alpha.f}, $\Assoc_{\bar{k},\bar{h},\bar{g}}\circ 1_{\bar{f}}$ and $1_{\bar{k}}\circ\Assoc_{\bar{h},\bar{g},\bar{f}}$  are $\Omega$ 2-cells corresponding to $\Sg$-paths of interest between $\Sg$-schemes of left border $(r,g,s,t,k)$.

Since the top and the bottom of the pentagon correspond to two $\Sg$-paths of interest with the same starting and ending, those $\Sg$-paths are equivalent --- that is, the pentagon is commutative in the category $\catx[\Sigma_{\ast}](A,E)$.
\end{proof}

Finally, it is easy to see that for composable 1-cells $A\xrightarrow{\bar{f}}B\xrightarrow{\bar{g}}C$, the associator component $\Assoc_{\bar{g}, 1_B, \bar{f}}$ is an identity 2-cell. Since the unitors are also just identities, the coherence for the unitors trivially holds. This completes the proof of the main result of the present section:

\begin{theorem}\label{thm:main3} As defined in this section, $\catx[\Sigma_{\ast}]$ is indeed a bicategory.
\end{theorem}

\section{The universal property}\label{sec:universal}

We want to show that $\catx[\Sigma_*]$ is the universal bicategory which turns $\Sigma$-morphisms into laris and $\Sigma$-squares to Beck--Chevalley squares (abbreviate as {\em BC squares}).
Let us start by showing that $\catx[\Sigma_*]$ at least does do this.

\begin{definition}\label{def:Psigma}
 We define a pseudofunctor $P_\Sigma\colon \catx \to \catx[\Sigma_*]$ such that:
 \begin{itemize}
  \item on objects, $P_\Sigma(X) = X$,
  \item on morphisms, $P_\Sigma(f) = (f,1)$,
  \item on 2-morphisms, $P_\Sigma(\alpha) = [\alpha, 1, 1, 1, \id, \id]$,
  \item the unitors $\iota^P_X\colon 1_{P_\Sigma(X)} \to P_\Sigma(1_X)$ are identities,
  \item the compositors $\gamma^P_{g,f} \colon P_\Sigma(g) \circ P_\Sigma(f) \to P_\Sigma(gf)$ are also identities.
 \end{itemize}
\end{definition}

\begin{lemma}
 $P_\Sigma$ as defined above is indeed a pseudofunctor.
\end{lemma}
\begin{proof}
 It is clear that $P_\Sigma$ preserves identity 2-cells. For vertical composition of 2-cells, note that to compute the composite $P_\Sigma(\beta) \cdot P_\Sigma(\alpha)$ as in \cref{data-2} we may take the $\Sigma$-squares $\phi_x, \phi_y$ to be vertical identity squares and the isomorphism $\gamma$ to be an identity 2-cell. Then it is easy to see that $P_\Sigma(\beta) \cdot P_\Sigma(\alpha) = P_\Sigma(\beta\cdot \alpha)$ and so $P_\Sigma$ preserves vertical composition.

 It is also clear that $P_\Sigma$ preserves identity 1-cells and so the unitors are well-defined.
 For morphisms $g\colon A \to B$ and $f\colon B \to C$ in $\catx$, the composite $P_\Sigma(f) \circ P_\Sigma(g)$ is found using the canonical $\Sigma$-square
 \[\begin{tikzcd} %
    A & B & B \\
    & C & C & C,
    \arrow["f", from=1-1, to=1-2]
    \arrow["1_B"', from=1-3, to=1-2]
    \arrow[""{name=0, anchor=center, inner sep=0}, "g", from=1-3, to=2-3]
    \arrow["1_C", from=2-3, to=2-2]
    \arrow[""{name=1, anchor=center, inner sep=0}, "g"', from=1-2, to=2-2]
    \arrow["1_C", from=2-4, to=2-3]
    \arrow["{\SIGMA^{\id}}"{description}, draw=none, from=1, to=0]
 \end{tikzcd}\]
 and so give $P_\Sigma(gf)$ ensuring the compositors are well-defined.

 The naturality condition for the compositors requires that $P_\Sigma(\beta) \circ P_\Sigma(\alpha) = P_\Sigma(\beta \circ \alpha)$.
 Computing the left-hand side as in \cref{def:horizontal} we may take the necessary $\Sigma$-squares in \cref{eq:(A12)} to be $\Sigma^\id$ and $\beta' = \beta$. The desired equality follows.

 Finally, we must show the coherence conditions. Coherence for the compositor (the `associativity' condition) reduces to requiring that the $\Omega$ 2-cell
 $$\mathrm{\Assoc}_{P_\Sigma(f),P_\Sigma(g),P_\Sigma(h)}\colon (P_\Sigma(h)\circ P_\Sigma(g))\circ P_\Sigma(f) \Rightarrow P_\Sigma(h)\circ (P_\Sigma(g)\circ P_\Sigma(f))$$
  is the identity. This is immediate, since, by replacing the morphisms $r$ and $s$ with identities in the diagrams (1), (2) and (3) of \cref{def:assoc}, the domain and codomain of the $\Omega$ 2-cell which corresponds to the $\Sigma$-path $\begin{tikzcd}
 	{(1)} & {(2)} & {(3)}
 	\arrow["\bu", squiggly, from=1-1, to=1-2]
 	\arrow["\bd", squiggly, from=1-2, to=1-3]
 \end{tikzcd}$ conincide and thus this $\Omega$ 2-cell is clearly the identity.
 The left and right unit conditions are automatic since all the maps in the diagram are identities. %
\end{proof}

\begin{proposition}\label{prop:P_gives_laris}
 The pseudofunctor $P_\Sigma$ sends 1-cells in $\Sigma$ to laris and $\Sigma$-squares to Beck--Chevalley squares.\footnote{Here we refer to the standard action of a pseudofunctor on a pasting diagram, which involves inserting compositors (and unitors) as appropriate.}
\end{proposition}
\begin{proof}
 Consider $s\colon A \to B$ in $\Sigma$. We first show $P_\Sigma(s) = (s,1_B)$ is a lari. We claim its right adjoint is $(1_B,s)$.
 Note that $(1_B,s) \circ (s,1_B) = (s,s)$. Put
  $$\bar{\eta}=[\id_s, s, 1_B, s, \id, \id].$$
  This is indeed the basic $\Omega$ 2-cell given by the $\Sigma$-path
  \(\adjustbox{scale=0.80}{\begin{tikzcd}
	{} & {} \\
	{} & {}
	\arrow["1", from=1-1, to=1-2]
	\arrow[""{name=0, anchor=center, inner sep=0}, "1"', from=1-1, to=2-1]
	\arrow[""{name=1, anchor=center, inner sep=0}, "1", from=1-2, to=2-2]
	\arrow["1"', from=2-1, to=2-2]
	\arrow["\SIGMAc"{description}, shift left=2, draw=none, from=0, to=1]
\end{tikzcd}}\)
$\begin{tikzcd}
	{} & {}
	\arrow[squiggly, from=1-1, to=1-2]
\end{tikzcd}$
\(\adjustbox{scale=0.80}{\begin{tikzcd}
	{} & {} \\
	{} & {}
	\arrow["1", from=1-1, to=1-2]
	\arrow[""{name=0, anchor=center, inner sep=0}, "1"', from=1-1, to=2-1]
	\arrow[""{name=1, anchor=center, inner sep=0}, "s", from=1-2, to=2-2]
	\arrow["s"', from=2-1, to=2-2]
	\arrow["\SIGMA"{description}, shift left=2, draw=none, from=0, to=1]
\end{tikzcd}}\),
 where the last $\Sg$-square is derived from Identity.
 In particular, note that $\bar{\eta}$ is invertible.

 On the other hand, we find $(s,1_B)\circ (1_B,s)$ using
 \[\begin{tikzcd}
    B & B & A \\
    & \dot{A} & B & B,
    \arrow["1_B", from=1-1, to=1-2]
    \arrow["s"', from=1-3, to=1-2]
    \arrow[""{name=0, anchor=center, inner sep=0}, "s", from=1-3, to=2-3]
    \arrow["\dot{s}_2", from=2-3, to=2-2]
    \arrow[""{name=1, anchor=center, inner sep=0}, "\dot{s}_1"', from=1-2, to=2-2]
    \arrow["1_B", from=2-4, to=2-3]
    \arrow["{\SIGMA^{\dot{\alpha}}}"{description}, draw=none, from=1, to=0]
 \end{tikzcd}\]
 to give $(\dot{s}_1,\dot{s}_2)$. Applying Equi-insertion to the above $\Sigma$-square and the 2-cell $\dot{\alpha}^{-1}\colon \dot{s}_1 s \Rightarrow \dot{s}_2 s$ we obtain a morphism $q\colon \dot{A} \to Q$ and a 2-cell $\epsilon\colon q\dot{s}_1 \Rightarrow q\dot{s}_2$ such that
 \begin{equation}\label{eq:epsilon1}q\dot{\alpha}^{-1} = \epsilon s\qquad \text{ and } \qquad
 \begin{tikzcd}
    B & \dot{A} \\
    B & Q
    \arrow["\dot{s}_2", from=1-1, to=1-2]
    \arrow[""{name=0, anchor=center, inner sep=0}, "q", from=1-2, to=2-2]
    \arrow["q\dot{s}_2"', from=2-1, to=2-2]
    \arrow[""{name=1, anchor=center, inner sep=0}, Rightarrow, no head, from=1-1, to=2-1]
    \arrow["{\SIGMA^{\id}}"{description}, draw=none, from=1, to=0]
 \end{tikzcd}\,.
 \end{equation}
 We can now form a 2-cell $\overline{\epsilon}\colon (\dot{s_1},\dot{s}_2) \Rightarrow (\id_B, \id_B)$ represented by the diagram
 \begin{equation}\label{eq:epsilon2}
 \begin{tikzcd}
	B & \dot{A} & B \\
	& Q & B \\
	B & B & B
	\arrow[Rightarrow, no head, from=1-1, to=3-1]
	\arrow[""{name=0, anchor=center, inner sep=0}, Rightarrow, no head, from=1-3, to=2-3]
	\arrow[""{name=1, anchor=center, inner sep=0}, Rightarrow, no head, from=2-3, to=3-3]
	\arrow[""{name=2, anchor=center, inner sep=0}, "q"', from=1-2, to=2-2]
	\arrow["q\dot{s_2}"', from=2-3, to=2-2]
	\arrow["{1_B}", from=3-3, to=3-2]
	\arrow["\dot{s_2}"', from=1-3, to=1-2]
	\arrow["\dot{s_1}", from=1-1, to=1-2]
	\arrow["{1_B}"', from=3-1, to=3-2]
	\arrow["\epsilon"', shorten <=11pt, shorten >=11pt, Rightarrow, from=1-2, to=3-1]
	\arrow[""{name=3, anchor=center, inner sep=0}, "q\dot{s_2}", from=3-2, to=2-2]
	\arrow["\SIGMA^{\id}"{description}, draw=none, from=2, to=0]
	\arrow["\SIGMA^{\id}"{description}, draw=none, from=3, to=1]
 \end{tikzcd}
 \end{equation}
 where the lower $\Sg$-square is given by Identity.

 Using the notations $\bar{s}=(s,1)$ and $s_*=(1,s)$, we now want to show the triangle identites
 \begin{equation}\label{eq:triangles}\id_{s_*} = (s_*\circ \bar{\epsilon}) \cdot (\bar{\eta}\circ s_*) \qquad \text{ and } \qquad \id_{\bar{s}} = (\bar{\epsilon}\circ \bar{s}) \cdot (\bar{s}\circ \bar{\eta})\, .
 \end{equation}
 Since $\bar{\eta}$ is an $\Omega$ 2-cell, so are $\bar{\eta}\circ s_*$ and $\bar{s}\circ \bar{\eta}$ (see \cref{pro:Omega-f}).

  Moreover, we easily see by \cref{def:horizontal} that $\bar{\eta}\circ s_*$ reduces to the $\Omega$ 2-cell $\Omega_1=[\id_{\dot{s}_1},\dot{s}_1, 1_{\dot{A}},\dot{s}_2s, \dot{\alpha},\id]$ determined by the $\Sg$-path
 \[\begin{tikzcd}
	& {} & {} \\
	{} & {} & {} \\
	{} & {} & {}
	\arrow["s", from=1-2, to=1-3]
	\arrow[""{name=0, anchor=center, inner sep=0}, equals, from=1-2, to=2-2]
	\arrow[""{name=1, anchor=center, inner sep=0}, equals, from=1-3, to=2-3]
	\arrow[equals, from=2-1, to=2-2]
	\arrow[""{name=2, anchor=center, inner sep=0}, equals, from=2-1, to=3-1]
	\arrow["s", from=2-2, to=2-3]
	\arrow[""{name=3, anchor=center, inner sep=0}, equals, from=2-2, to=3-2]
	\arrow[""{name=4, anchor=center, inner sep=0}, equals, from=2-3, to=3-3]
	\arrow[equals, from=3-1, to=3-2]
	\arrow["s", from=3-2, to=3-3]
	\arrow["{{\dot{\Sigma}}}"{description}, shift left, draw=none, from=0, to=1]
	\arrow["{{\dot{\Sigma}}}"{description}, shift left, draw=none, from=2, to=3]
	\arrow["{{\dot{\Sigma}}}"{description}, shift left, draw=none, from=3, to=4]
\end{tikzcd}
\begin{tikzcd}
	{} & {}
	\arrow[squiggly, from=1-1, to=1-2]
\end{tikzcd}
\begin{tikzcd}
	& {} & {} \\
	{} & {} & {} \\
	{} & {} & {}
	\arrow["s", from=1-2, to=1-3]
	\arrow[""{name=0, anchor=center, inner sep=0}, equals, from=1-2, to=2-2]
	\arrow[""{name=1, anchor=center, inner sep=0}, equals, from=1-3, to=2-3]
	\arrow[equals, from=2-1, to=2-2]
	\arrow[""{name=2, anchor=center, inner sep=0}, equals, from=2-1, to=3-1]
	\arrow["s"', from=2-2, to=2-3]
	\arrow[""{name=3, anchor=center, inner sep=0}, "s", from=2-2, to=3-2]
	\arrow[""{name=4, anchor=center, inner sep=0}, "{{{{\dot{s}_1}}}}", from=2-3, to=3-3]
	\arrow["s"', from=3-1, to=3-2]
	\arrow["{{{{\dot{s}_2}}}}"', from=3-2, to=3-3]
	\arrow["{{\dot{\Sigma}}}"{description}, shift left, draw=none, from=0, to=1]
	\arrow["{{\Sigma^{\id}}}"{description}, shift left, draw=none, from=2, to=3]
	\arrow["{{\dot{\Sigma}}}"{description}, shift left, draw=none, from=3, to=4]
\end{tikzcd}
\]
  or equivalently, by the $\Sg$-path
  \[\begin{tikzcd}
	{} & {} \\
	{} & {}
	\arrow["s", from=1-1, to=1-2]
	\arrow[""{name=0, anchor=center, inner sep=0}, equals, from=1-1, to=2-1]
	\arrow[""{name=1, anchor=center, inner sep=0}, equals, from=1-2, to=2-2]
	\arrow["s"', from=2-1, to=2-2]
	\arrow["{\Sigma^{\dot{\alpha}}}"{description}, shift left, draw=none, from=0, to=1]
\end{tikzcd}
 \begin{tikzcd}
	{} & {}
	\arrow[squiggly, from=1-1, to=1-2]
\end{tikzcd}
\begin{tikzcd}
	{} & {} & {} \\
	{} & {} & {} \\
	{} & {} & {}
	\arrow[equals, from=1-1, to=1-2]
	\arrow[""{name=0, anchor=center, inner sep=0}, equals, from=1-1, to=2-1]
	\arrow["s", from=1-2, to=1-3]
	\arrow[""{name=1, anchor=center, inner sep=0}, equals, from=1-2, to=2-2]
	\arrow[""{name=2, anchor=center, inner sep=0}, equals, from=1-3, to=2-3]
	\arrow[equals, from=2-1, to=2-2]
	\arrow[""{name=3, anchor=center, inner sep=0}, equals, from=2-1, to=3-1]
	\arrow["s", from=2-2, to=2-3]
	\arrow[""{name=4, anchor=center, inner sep=0}, "s"{description}, from=2-2, to=3-2]
	\arrow[""{name=5, anchor=center, inner sep=0}, "{{\dot{s}_1}}", from=2-3, to=3-3]
	\arrow["s"', from=3-1, to=3-2]
	\arrow["{{\dot{s}_2}}"', from=3-2, to=3-3]
	\arrow["{\dot{\Sigma}}"{description}, shift left, draw=none, from=0, to=1]
	\arrow["{\dot{\Sigma}}"{description}, shift left, draw=none, from=1, to=2]
	\arrow["{\Sigma^{\id}}"{description}, shift left, draw=none, from=3, to=4]
	\arrow["{\Sigma^{\dot{\alpha}}}"{description}, shift left, draw=none, from=4, to=5]
\end{tikzcd}\, .\]
 For obtaining $(1,s) \circ \bar{\epsilon}$ we use again \cref{def:horizontal} and the equality
  \begin{equation*}
 \begin{tikzcd}
	B & Q \\
	B \\
	B & Q
	\arrow["{q\dot{s}_2}", from=1-1, to=1-2]
	\arrow[equals, from=1-1, to=2-1]
	\arrow[""{name=0, anchor=center, inner sep=0}, "{1_B}", curve={height=-24pt}, from=1-2, to=3-2]
	\arrow[""{name=1, anchor=center, inner sep=0}, equals, from=1-2, to=3-2]
	\arrow[equals, from=2-1, to=3-1]
	\arrow["{q\dot{s}_2}"', from=3-1, to=3-2]
	\arrow["\id"{pos=0.5}, shorten <=5pt, shorten >=2pt, Rightarrow, from=1, to=0]
	\arrow["{\SIGMA^{\id}}"{description}, draw=none, from=2-1, to=1]
 \end{tikzcd}
 =
 \begin{tikzcd}
	B & Q \\
	B \\
	B & Q\rlap{.}
	\arrow["{q\dot{s}_2}", from=1-1, to=1-2]
	\arrow[equals, from=1-1, to=2-1]
	\arrow[""{name=0, anchor=center, inner sep=0}, "{1_B}"', curve={height=24pt}, from=1-1, to=3-1]
	\arrow[""{name=1, anchor=center, inner sep=0}, equals, from=1-2, to=3-2]
	\arrow[equals, from=2-1, to=3-1]
	\arrow["{q\dot{s}_2}"', from=3-1, to=3-2]
	\arrow["\id"{pos=0.2}, shorten <=2pt, Rightarrow, from=0, to=2-1]
	\arrow["{\SIGMA^{\id}}"{description}, draw=none, from=2-1, to=1]
 \end{tikzcd}
 \end{equation*}
 Inspecting the involved $\Omega$ 2-cells lead us to the following representative of the composition.
 \begin{equation}\label{eq:(1,s)epsilon}
 \begin{tikzcd}
	B & {\dot{A}} & {\dot{A}} && B & A \\
	B & {\dot{A}} & Q && B & A \\
	&& Q &&& A \\
	B & {\dot{A}} & Q & {\dot{A}} & B & A \\
	B && B && B & A
	\arrow["{\dot{s}_1}", from=1-1, to=1-2]
	\arrow[equals, from=1-1, to=2-1]
	\arrow[equals, from=1-2, to=1-3]
	\arrow[equals, from=1-2, to=2-2]
	\arrow[""{name=0, anchor=center, inner sep=0}, "q", from=1-3, to=2-3]
	\arrow["{\dot{s}_2}"', from=1-5, to=1-3]
	\arrow[""{name=1, anchor=center, inner sep=0}, equals, from=1-5, to=2-5]
	\arrow["s"', from=1-6, to=1-5]
	\arrow[""{name=2, anchor=center, inner sep=0}, equals, from=1-6, to=2-6]
	\arrow["{\dot{s}_1}", from=2-1, to=2-2]
	\arrow[equals, from=2-1, to=4-1]
	\arrow["q", from=2-2, to=2-3]
	\arrow[""{name=3, anchor=center, inner sep=0}, equals, from=2-3, to=3-3]
	\arrow["\epsilon"', between={0.3}{0.7}, Rightarrow, from=2-3, to=4-1]
	\arrow["{q\dot{s}_2}"', from=2-5, to=2-3]
	\arrow["s"', from=2-6, to=2-5]
	\arrow[""{name=4, anchor=center, inner sep=0}, equals, from=2-6, to=3-6]
	\arrow["{{q\dot{s}_2 s}}"', from=3-6, to=3-3]
	\arrow[""{name=5, anchor=center, inner sep=0}, equals, from=3-6, to=4-6]
	\arrow["{{\dot{s}_2}}"', from=4-1, to=4-2]
	\arrow[equals, from=4-1, to=5-1]
	\arrow["q"', from=4-2, to=4-3]
	\arrow[""{name=6, anchor=center, inner sep=0}, equals, from=4-3, to=3-3]
	\arrow["{{q}}", from=4-4, to=4-3]
	\arrow["{{\dot{s}_2}}", from=4-5, to=4-4]
	\arrow[""{name=7, anchor=center, inner sep=0}, equals, from=4-5, to=5-5]
	\arrow["s"', from=4-6, to=4-5]
	\arrow[""{name=8, anchor=center, inner sep=0}, equals, from=4-6, to=5-6]
	\arrow[equals, from=5-1, to=5-3]
	\arrow[""{name=9, anchor=center, inner sep=0}, "{q\dot{s}_2}"', from=5-3, to=4-3]
	\arrow[equals, from=5-3, to=5-5]
	\arrow["s"', from=5-6, to=5-5]
	\arrow["{\SIGMA^{\id}}"{marking, allow upside down}, draw=none, from=0, to=1]
	\arrow["{\SIGMA^{\id}}"{marking, allow upside down}, draw=none, from=1, to=2]
	\arrow["{{\SIGMA^{\id}}}"{description}, draw=none, from=3, to=4]
	\arrow["{{\SIGMA^{\id}}}"{description}, draw=none, from=6, to=5]
	\arrow["{\SIGMA^{\id}}"{description}, draw=none, from=7, to=8]
	\arrow["{\SIGMA^{\id}}"{description}, draw=none, from=9, to=7]
\end{tikzcd}
\end{equation}
This is clearly the $\Omega$ 2-cell given by the  $\Sg$-path
\(\;\begin{tikzcd}
	{} & {} & {} \\
	{} & {} & {}
	\arrow[equals, from=1-1, to=1-2]
	\arrow[""{name=0, anchor=center, inner sep=0}, equals, from=1-1, to=2-1]
	\arrow["s", from=1-2, to=1-3]
	\arrow[""{name=1, anchor=center, inner sep=0}, "s", from=1-2, to=2-2]
	\arrow[""{name=2, anchor=center, inner sep=0}, "{{\dot{s}_1}}", from=1-3, to=2-3]
	\arrow["s"', from=2-1, to=2-2]
	\arrow["{{\dot{s}_2}}"', from=2-2, to=2-3]
	\arrow["{{\SIGMA^{\id}}}"{description}, shift left, draw=none, from=0, to=1]
	\arrow["{{\SIGMA^{\dot{\alpha}}}}"{description}, shift right, draw=none, from=2, to=1]
\end{tikzcd}
\begin{tikzcd}
	{} & {}
	\arrow[squiggly, from=1-1, to=1-2]
\end{tikzcd}
\begin{tikzcd}
	{} & {} \\
	{} & {}
	\arrow["s", from=1-1, to=1-2]
	\arrow[""{name=0, anchor=center, inner sep=0}, equals, from=1-1, to=2-1]
	\arrow[""{name=1, anchor=center, inner sep=0}, equals, from=1-2, to=2-2]
	\arrow["s"', from=2-1, to=2-2]
	\arrow["{{\SIGMA^{\id}}}"{description}, shift right, draw=none, from=1, to=0]
\end{tikzcd}\, ,\)
as shown by the following equality which follows from \eqref{eq:epsilon1}:
\[\begin{tikzcd}
	{} & {} & {} \\
	{} & {} & {} & {} \\
	{} & {} & {}
	\arrow[equals, from=1-1, to=1-2]
	\arrow[""{name=0, anchor=center, inner sep=0}, equals, from=1-1, to=2-1]
	\arrow["s", from=1-2, to=1-3]
	\arrow[""{name=1, anchor=center, inner sep=0}, "s"', from=1-2, to=2-2]
	\arrow[""{name=2, anchor=center, inner sep=0}, "{{{\dot{s}_1}}}", from=1-3, to=2-3]
	\arrow[curve={height=-12pt}, equals, from=1-3, to=2-4]
	\arrow["s"{description}, from=2-1, to=2-2]
	\arrow[""{name=3, anchor=center, inner sep=0}, equals, from=2-1, to=3-1]
	\arrow["{\dot{s}_2}"{description}, from=2-2, to=2-3]
	\arrow[""{name=4, anchor=center, inner sep=0}, equals, from=2-2, to=3-2]
	\arrow["\epsilon", Rightarrow, from=2-3, to=2-4]
	\arrow[""{name=5, anchor=center, inner sep=0}, "q", from=2-3, to=3-3]
	\arrow["{{{q\dot{s}_2}}}", curve={height=-12pt}, from=2-4, to=3-3]
	\arrow["s"', from=3-1, to=3-2]
	\arrow["{q\dot{s}_2}"', from=3-2, to=3-3]
	\arrow["{\SIGMA^{\id}}"{description}, shift left, between={0.2}{0.8}, Rightarrow, from=0, to=1]
	\arrow["{{{\SIGMA^{\dot{\alpha}}}}}"{description}, shift left, draw=none, from=1, to=2]
	\arrow["{\SIGMA^{\id}}"{description}, shift left, between={0.2}{0.8}, Rightarrow, from=3, to=4]
	\arrow["{\SIGMA^{\id}}"{description}, shift left, draw=none, from=4, to=5]
\end{tikzcd}
\qquad =\qquad
\begin{tikzcd}
	{} & {} \\
	{} & {} \\
	{} & {} \\
	{} & {}
	\arrow["s", from=1-1, to=1-2]
	\arrow[""{name=0, anchor=center, inner sep=0}, equals, from=1-1, to=2-1]
	\arrow[""{name=1, anchor=center, inner sep=0}, equals, from=1-2, to=2-2]
	\arrow["s"{description}, from=2-1, to=2-2]
	\arrow[""{name=2, anchor=center, inner sep=0}, equals, from=2-1, to=3-1]
	\arrow[""{name=3, anchor=center, inner sep=0}, "{\dot{s}_2}", from=2-2, to=3-2]
	\arrow["{\dot{s}_2s}"{description}, from=3-1, to=3-2]
	\arrow[""{name=4, anchor=center, inner sep=0}, equals, from=3-1, to=4-1]
	\arrow[""{name=5, anchor=center, inner sep=0}, "q", from=3-2, to=4-2]
	\arrow["{q\dot{s}_2s}"', from=4-1, to=4-2]
	\arrow["{{{\SIGMA^{\id}}}}"{description}, shift left, draw=none, from=0, to=1]
	\arrow["{\SIGMA^{\id}}"{description}, shift left, draw=none, from=2, to=3]
	\arrow["{\SIGMA^{\id}}"{description}, shift left, draw=none, from=4, to=5]
\end{tikzcd}\, .\]

Thus, the composition $(s_* \circ \bar{\epsilon})\cdot (\bar{\eta}\circ s_*)$, being a vertical composition of $\Omega$ 2-cells corresponding to $\Sg$-paths between $\Sg$-schemes of the same left border $(s, 1_A)$, is an $\Omega$ 2-cell, and so it is the identity.

Concerning the other triangle equality, first observe that, analogously to the case $\bar{\eta}\circ s_*$, it is easy to see that $\bar{s}\circ \bar{\eta}$ reduces to the $\Omega$ 2-cell corresponding to the $\Sg$-path
\[\begin{tikzcd}
	{} & {} \\
	{} & {}
	\arrow[equals, from=1-1, to=1-2]
	\arrow[""{name=0, anchor=center, inner sep=0}, "s"', from=1-1, to=2-1]
	\arrow[""{name=1, anchor=center, inner sep=0}, "s", from=1-2, to=2-2]
	\arrow[equals, from=2-1, to=2-2]
	\arrow["{{\Sigma^{\id}}}"{description}, shift left, draw=none, from=0, to=1]
\end{tikzcd}
\begin{tikzcd}
	{} & {}
	\arrow[squiggly, from=1-1, to=1-2]
\end{tikzcd}
\begin{tikzcd}
	{} & {} \\
	{} & {}
	\arrow[equals, from=1-1, to=1-2]
	\arrow[""{name=0, anchor=center, inner sep=0}, "s"', from=1-1, to=2-1]
	\arrow[""{name=1, anchor=center, inner sep=0}, "{{\dot{s}_1s}}", from=1-2, to=2-2]
	\arrow["{{\dot{s}_2}}"', from=2-1, to=2-2]
	\arrow["{{\Sigma^{\dot{\alpha}}}}"{description}, shift left, draw=none, from=0, to=1]
\end{tikzcd}\, .\]
Calculating  $\bar{\epsilon} \circ (s,1)$, according to \cref{def:horizontal}, it is straightforward to verify that it has the following representative.

\[\begin{tikzcd}
	A & B & {\dot{A}} && B \\
	A & B & Q & {\dot{A}} & B \\
	&& Q && B \\
	A & B & Q & {\dot{A}} & B \\
	A & B & B && B
	\arrow["s", from=1-1, to=1-2]
	\arrow[equals, from=1-1, to=2-1]
	\arrow["{\dot{s}_1}", from=1-2, to=1-3]
	\arrow[equals, from=1-2, to=2-2]
	\arrow[draw=none, from=1-3, to=2-1]
	\arrow[""{name=0, anchor=center, inner sep=0}, "q", from=1-3, to=2-3]
	\arrow["{\dot{s}_2}"', from=1-5, to=1-3]
	\arrow[""{name=1, anchor=center, inner sep=0}, equals, from=1-5, to=2-5]
	\arrow["s", from=2-1, to=2-2]
	\arrow[equals, from=2-1, to=4-1]
	\arrow["{{q {\dot{s}_1}}}", from=2-2, to=2-3]
	\arrow[equals, from=2-2, to=4-2]
	\arrow[""{name=2, anchor=center, inner sep=0}, equals, from=2-3, to=3-3]
	\arrow["\epsilon"', between={0.3}{0.7}, Rightarrow, from=2-3, to=4-2]
	\arrow["q"', from=2-4, to=2-3]
	\arrow["{{{\dot{s}_2}}}"', from=2-5, to=2-4]
	\arrow[""{name=3, anchor=center, inner sep=0}, equals, from=2-5, to=3-5]
	\arrow["{{q\dot{s}_2}}"', from=3-5, to=3-3]
	\arrow[""{name=4, anchor=center, inner sep=0}, equals, from=3-5, to=4-5]
	\arrow["s"', from=4-1, to=4-2]
	\arrow[equals, from=4-1, to=5-1]
	\arrow["{{q {\dot{s}_2}}}"', from=4-2, to=4-3]
	\arrow[equals, from=4-2, to=5-2]
	\arrow[""{name=5, anchor=center, inner sep=0}, equals, from=4-3, to=3-3]
	\arrow["{{{q}}}", from=4-4, to=4-3]
	\arrow["{{{\dot{s}_2}}}", from=4-5, to=4-4]
	\arrow[""{name=6, anchor=center, inner sep=0}, equals, from=4-5, to=5-5]
	\arrow["s"', from=5-1, to=5-2]
	\arrow[equals, from=5-2, to=5-3]
	\arrow[""{name=7, anchor=center, inner sep=0}, "{q {\dot{s}_2}}"{description}, from=5-3, to=4-3]
	\arrow[equals, from=5-5, to=5-3]
	\arrow["{\SIGMA^{\id}}"{description}, draw=none, from=0, to=1]
	\arrow["{{{\SIGMA^{\id}}}}"{description}, draw=none, from=2, to=3]
	\arrow["{\SIGMA^{\id}}"{description}, draw=none, from=5, to=4]
	\arrow["{\SIGMA^{\id}}"{description}, draw=none, from=7, to=6]
\end{tikzcd}\]
This gives a $\Omega$ 2-cell corresponding to a $\Sg$-path of left border $(1_A,s)$, as shown  by the equality
 \begin{equation*}
 \begin{tikzcd}
	A & A \\
	A & A \\
	B & Q & Q \\
	B & Q
	\arrow[equals, from=1-1, to=1-2]
	\arrow[""{name=0, anchor=center, inner sep=0}, equals, from=1-1, to=2-1]
	\arrow[""{name=1, anchor=center, inner sep=0}, "s", from=1-2, to=2-2]
	\arrow["{{q\dot{s}_2s}}", curve={height=-12pt}, from=1-2, to=3-3]
	\arrow["s", from=2-1, to=2-2]
	\arrow[""{name=2, anchor=center, inner sep=0}, "s"', from=2-1, to=3-1]
	\arrow[""{name=3, anchor=center, inner sep=0}, "{{\dot{s}_1}}", from=2-2, to=3-2]
	\arrow["{\dot{s}_2}", from=3-1, to=3-2]
	\arrow[""{name=4, anchor=center, inner sep=0}, equals, from=3-1, to=4-1]
	\arrow["{{\epsilon s}}", shift left=3, Rightarrow, from=3-2, to=3-3]
	\arrow[""{name=5, anchor=center, inner sep=0}, "q", from=3-2, to=4-2]
	\arrow[curve={height=-12pt}, equals, from=3-3, to=4-2]
	\arrow["{{{q \dot{s}_2}}}"', from=4-1, to=4-2]
	\arrow["{{\SIGMA^{\id}}}"{description}, draw=none, from=0, to=1]
	\arrow["{{\SIGMA^{\dot{\alpha}}}}"{marking, allow upside down}, shift right, draw=none, from=2, to=3]
	\arrow["{{\SIGMA^{\id}}}"{marking, allow upside down}, shift right, draw=none, from=4, to=5]
\end{tikzcd}
 =\;
 \begin{tikzcd}
	A & A \\
	B & B \\
	B & Q
	\arrow[equals, from=1-1, to=1-2]
	\arrow[""{name=0, anchor=center, inner sep=0}, "s"', from=1-1, to=2-1]
	\arrow[""{name=1, anchor=center, inner sep=0}, "s", from=1-2, to=2-2]
	\arrow[equals, from=2-1, to=2-2]
	\arrow[""{name=2, anchor=center, inner sep=0}, equals, from=2-1, to=3-1]
	\arrow[""{name=3, anchor=center, inner sep=0}, "{{q\dot{s}_2}}", from=2-2, to=3-2]
	\arrow["{{{q \dot{s}_2}}}"', from=3-1, to=3-2]
	\arrow["{{{\SIGMA^\id}}}"{marking, allow upside down}, shift right, draw=none, from=0, to=1]
	\arrow["{{\SIGMA^{\id}}}"{marking, allow upside down}, shift right, draw=none, from=2, to=3]
\end{tikzcd}\, ,
 \end{equation*}
 which follows from the first equality of  \eqref{eq:epsilon1}. Consequently, the composition $(\bar{\epsilon} \circ \bar{s})\cdot (\bar{s}\circ \bar{\eta})$ is an $\Omega$ 2-cell corresponding to a $\Sigma$-path of $\Sigma$-schemes with left border $(1_A,s)$ starting and ending in the same $\Sigma$-scheme. Thus, it is an identity.

 In conclusion, the triangle identities hold and $P_\Sigma$ indeed sends 1-cells in $\Sigma$ to laris.

 We now show $P_\Sigma$ sends $\Sigma$-squares to Beck--Chevalley squares.
 Suppose we have a $\Sigma$-square
 \[\begin{tikzcd}
	A & B \\
	C & D\rlap{.}
	\arrow["s", from=1-1, to=1-2]
	\arrow[""{name=0, anchor=center, inner sep=0}, "f"', from=1-1, to=2-1]
	\arrow[""{name=1, anchor=center, inner sep=0}, "g", from=1-2, to=2-2]
	\arrow["t"', from=2-1, to=2-2]
	\arrow["{\SIGMA^\delta}"{marking, allow upside down}, shift right, draw=none, from=0, to=1]
 \end{tikzcd}\]
 Then applying $P_\Sigma$ we have $P_\Sigma(\delta)\colon (t,1) \circ (f,1) \to (g,1) \circ (s,1)$. %
 The mate of this is a 2-cell $(f,1) \circ (1,s) \to (1,t) \circ (1,g)$ given by the composite:
 \begin{align*}
  & ((g,t) \circ \epsilon^s) \cdot \Assoc_{(g,t),(s,1),(1,s)} \cdot (\Assoc_{(1,t),(g,1),(s,1)}^{-1} \circ (1,s)) \\
  & \ {} \cdot ( ((1,t) \circ P_\Sigma(\delta)) \circ (1,s) ) \cdot (\Assoc_{(1,t),(t,1),(f,1)} \circ (1,s)) \cdot ((\eta^t \circ (f,1)) \circ (1,s)).
 \end{align*}
 The Beck--Chevalley condition says that this composite is an isomorphism. All of the factors except $(g,t) \circ \epsilon^s$ are always isomorphisms. We now show that $(g,t) \circ \epsilon^s$ is also an isomorphism.

 From $\Sigma^\delta$ and the canonical $\Sigma$-square
 \[\begin{tikzcd}
	A & B \\
	C & \ddot{A}
	\arrow["s", from=1-1, to=1-2]
	\arrow[""{name=0, anchor=center, inner sep=0}, "f"', from=1-1, to=2-1]
	\arrow[""{name=1, anchor=center, inner sep=0}, "{\dot{f}}", from=1-2, to=2-2]
	\arrow["{\dot{s}}"', from=2-1, to=2-2]
	\arrow["{\SIGMA^{\ddot{\alpha}}}"{marking, allow upside down}, shift right, draw=none, from=0, to=1]
 \end{tikzcd}\]
 we have an $\Omega$ 2-cell between $(g,t)$ and $(\dot{f},\dot{s})$. Note that $(\dot{f},\dot{s})$ is the composite $(f,1) \circ (1,s)$. Hence $(g,t) \cong (f,1) \circ (1,s)$ and it suffices to show $(f,1) \circ ((1,s) \circ \overline{\epsilon})$ is an isomorphism. But $(1,s) \circ \overline{\epsilon}$ is an isomorphism by one of the triangle identities and so we are done.
\end{proof}

Before we state the formal universal property it will useful to know the following fact about pseudonatural transformations.
\begin{lemma}\label{lem:transformation_at_adjoint_statement}
 Let $H_1, H_2\colon \cata \to \catc$ be pseudofunctors and let $\upsilon\colon H_1 \to H_2$ be a pseudonatural transformation. Suppose $r\colon B \to I$ is a morphism in $\cata$ that has a right adjoint $r_*\colon I \to B$ (with unit $\eta$ and counit $\epsilon$). Then the 2-morphism $\upsilon_{r_*}\colon H_2(r_*) \circ \upsilon_I \to \upsilon_B \circ H_1(r_*)$ in the pseudonaturality square for $r_*$ is the inverse of the mate of 2-morphism $\upsilon_r$ in the pseudonaturality square for $r$.
\end{lemma}
\begin{proof}
 See \cref{lem:transformation_at_adjoint}.
\end{proof}

\begin{assumption}
 For simplicity, in the following we consider only bicategories and pseudofunctors that are strictly unitary --- that is, whose unitors are identities. The bicategory $\catx[\Sigma_*]$ and the pseudofunctor $P_\Sigma$ are strictly unitary. There is no loss of generality since every bicategory is biequivalent to a strictly unitary one and every pseudofunctor is pseudonaturally equivalent to a strictly unitary one.
\end{assumption}

\begin{theorem}\label{thm:universal_prop}
 Let $\catx$ be a 2-category admitting a  calculus of left lax fractions for $\Sigma$. The pseudofunctor $P_\Sigma\colon \catx \to \catx[\Sigma_*]$ is the universal (strictly unital) pseudofunctor %
 that satisfies the conditions of \cref{prop:P_gives_laris}.
 More precisely, we have
 \begin{enumerate}[label=\normalfont(\alph*)]
  \item If $F\colon \catx \to \catc$ is a pseudofunctor sending $\Sigma$-morphisms to laris and $\Sigma$-squares to BC squares, then there is a pseudofunctor $H\colon \catx[\Sigma_*] \to \catc$ such that $F \simeq H \circ P_\Sigma$.
  \item If $H,H'\colon \catx[\Sigma_*] \to \catc$ are pseudofunctors and $\xi\colon H \circ P_\Sigma \to H' \circ P_\Sigma$ is a pseudonatural transformation for which the pseudonaturality squares for $r\colon B \to I$ in $\catx$ are BC squares whenever $r \in \Sigma$, then there is a pseudonatural transformation $\upsilon\colon H \to H'$ such that $\xi \cong \upsilon \circ P_\Sigma$.
  \item If $H,H'\colon \catx[\Sigma_*] \to \catc$ are pseudofunctors, $\upsilon, \upsilon' \colon H \to H'$ are pseudonatural transformations, and $\aleph \colon \upsilon \circ P_\Sigma \to \upsilon' \circ P_\Sigma$ is a modification, then there is a unique modification $\beth\colon \upsilon \to \upsilon'$ such that $\aleph = \beth \circ P_\Sigma$.
 \end{enumerate}
\end{theorem}
\begin{remark}
 We also remark that if $H\colon \catx[\Sigma_*] \to \catc$ is a pseudofunctor, then $H \circ P_\Sigma$ is a pseudofunctor which sends $\Sigma$-morphisms to laris and $\Sigma$-squares to BC squares (since being a lari or a BC square is preserved by pseudofunctors).
 Moroever, if $H,H'\colon \catx[\Sigma_*] \to \catc$ are pseudofunctors and $\upsilon\colon H \to H'$ is a pseudonatural transformation, then the pseudonaturality squares at $r \in \Sigma$ for $\upsilon \circ P_\Sigma$ are BC squares by \cref{lem:transformation_at_adjoint_statement}.
\end{remark}

We give the full proof of \cref{thm:universal_prop} in Appendix~\ref{sec:appendixB} using string diagrams. Here we provide a proof sketch for those who want to avoid string diagrams or simply want the general idea.

\begin{proof}[Proof of \cref{thm:universal_prop}]
 (a) Let $\catc$ be a bicategory and let $F\colon \catx \to \catc$ be a pseudofunctor that sends 1-cells in $\Sigma$ to laris in $\catc$ and $\Sigma$-squares to BC squares.
 We may assume the unitors of $\catc$ are identities and that $F$ is strictly unitary.
 For each lari $f$ of $\catc$ we choose a right adjoint $f_*$ (and associated unit/counits $\eta^f$ and $\epsilon^f$) such that $(1_X)_* = 1_X$ for every identity 1-cell (and $\eta^f = \epsilon^f = \id$ in this case).
 There is a pseudofunctor $H\colon \catx[\Sigma_*] \to \catc$ defined as follows.
 \begin{itemize}
  \item On objects, $H(X) = F(X)$,
  \item On morphisms, $H( (f,r) ) = (F r)_* (F f)$,
  \item On 2-morphisms, $H\left(\begin{tikzcd}
	A & I & B \\
	& X & B \\
	A & J & B
	\arrow["f", from=1-1, to=1-2]
	\arrow[equals, from=1-1, to=3-1]
	\arrow[""{name=0, anchor=center, inner sep=0}, "{x_1}"', from=1-2, to=2-2]
	\arrow["\alpha"', shorten <=15pt, shorten >=15pt, Rightarrow, from=1-2, to=3-1]
	\arrow["r"', from=1-3, to=1-2]
	\arrow[""{name=1, anchor=center, inner sep=0}, equals, from=1-3, to=2-3]
	\arrow["{x_3}"', from=2-3, to=2-2]
	\arrow["g"', from=3-1, to=3-2]
	\arrow[""{name=2, anchor=center, inner sep=0}, "{x_2}", from=3-2, to=2-2]
	\arrow[""{name=3, anchor=center, inner sep=0}, equals, from=3-3, to=2-3]
	\arrow["s", from=3-3, to=3-2]
	\arrow["{\SIGMA^{\delta_1}}"{marking, allow upside down}, draw=none, from=0, to=1]
	\arrow["{\SIGMA^{\delta_2}}"{marking, allow upside down}, draw=none, from=2, to=3]
  \end{tikzcd}\right) = \!\!\!
  \begin{tikzcd}
	{F A} & {F I} & {F B} \\
	& {F X} & {F B} \\
	{F A} & {F J} & {F B}
	\arrow[""{name=0, anchor=center, inner sep=0}, "Ff", from=1-1, to=1-2]
	\arrow[""{name=1, anchor=center, inner sep=0}, "\text{\tiny $F(x_1 f)$}"'{yshift=2pt,xshift=2pt}, from=1-1, to=2-2, gray]
	\arrow[equals, from=1-1, to=3-1]
	\arrow[""{name=2, anchor=center, inner sep=0}, "{F x_1}"{pos=0.3,xshift=-1pt}, from=1-2, to=2-2]
	\arrow["(Fr)_*", to=1-3, from=1-2]
	\arrow[""{name=3, anchor=center, inner sep=0}, equals, from=1-3, to=2-3]
	\arrow["F(\delta_1)_*"{xshift=-4pt}, shorten <=6pt, shorten >=6pt, Rightarrow, from=1-3, to=2-2]
	\arrow["F(x_3)_*"'{pos=0.4, yshift=1pt}, to=2-3, from=2-2]
	\arrow["F(\delta_2)_*^{-1}"{xshift=-5pt}, shorten <=6pt, shorten >=6pt, Rightarrow, from=2-2, to=3-3, yshift=-3pt,xshift=-1pt]
	\arrow[""{name=4, anchor=center, inner sep=0}, "\text{\tiny $F(x_2 g)$}"{yshift=-2pt,xshift=2pt}, from=3-1, to=2-2, gray]
	\arrow[""{name=5, anchor=center, inner sep=0}, "{F g}"', from=3-1, to=3-2]
	\arrow[""{name=6, anchor=center, inner sep=0}, "{F x_2}"'{pos=0.3,xshift=-1pt}, from=3-2, to=2-2]
	\arrow[""{name=7, anchor=center, inner sep=0}, equals, from=3-3, to=2-3]
	\arrow["(Fs)_*"', to=3-3, from=3-2]
	\arrow["{F(\alpha)}"', shorten <=22pt, shorten >=22pt, Rightarrow, from=0, to=5]
	\arrow["\text{\tiny $\gamma^F_{x_1,f}$}"'{xshift=2pt}, shorten >=2pt, Rightarrow, from=1-2, to=1, xshift=4pt, yshift=-2pt, gray]
	\arrow["\text{\tiny $\gamma^F_{x_2,g}{}^{\!\!\!\!-1}$}"'{yshift=-2pt,xshift=3pt}, shorten <=2pt, Rightarrow, from=4, to=3-2, xshift=4pt,yshift=3pt, gray]
  \end{tikzcd}$, \\
  where $F(\delta_1)_*$ and $F(\delta_2)_*$ denote the mate of $F(\delta_1)$ and $F(\delta_2)$, respectively.
  \item The unitors $\iota^H_X\colon 1_{H(X)} \to H(1_X)$ are identites,
  \item The compositors $\gamma^H_{(g,s),(f,r)} \colon H( (g,s) ) \circ H( (f,r) ) \to H( (g,s) \circ (f,r) )$ are given by the composite
  {\tiny\begin{align*}
   & {( F(s)_* F(g) ) \circ ( F(r)_* F(f) )}
    \xrightarrow{\alpha^{\catc\; -1}_{F(s)_*F(g),F(r)_*,F(f)}}
    {((F(s)_* F(g)) \circ F(r)_*) \circ F(f)}
    \xrightarrow{\alpha^\catc_{F(s)_*,F(g),F(r)_*} \circ F(f)} \\
   & {(F(s)_* \circ (F(g) F(r)_*)) \circ F(f)}
    \xrightarrow{(F(s)_* \circ F(\dot{\beta})_*) \circ F(f)}
    {(F(s)_* \circ (F(\dot{r})_* F(\dot{g}))) \circ F(f)}
    \xrightarrow{\alpha^{\catc\; -1}_{F(s)_*,F(\dot{r}),F(\dot{g})}\circ F(f)} \\
   & {((F(s)_* F(\dot{r})_*) \circ F(\dot{g}))) \circ F(f)}
    \xrightarrow{\alpha^{\catc}_{F(s)_*F(\dot{r}),F(\dot{g}),F(f)}}
    {(F(s)_* F(\dot{r})_*) \circ (F(\dot{g}) F(f))}
    \xrightarrow{(F(s)_* F(\dot{r})_*) \circ \gamma^F_{\dot{g},f}} \\
   & {(F(s)_* F(\dot{r})_*) \circ F(\dot{g} f)}
    \xrightarrow{\sigma \circ F(\dot{g} f)}
    {(F(\dot{r}) F(s))_* \circ F(\dot{g} f)}
    \xrightarrow{(\gamma^{F\; -1}_{\dot{r},s})_* \circ F(\dot{g} f)}
    {F(\dot{r} s)_* F(\dot{g} f)}
  \end{align*}}
  where $\dot{g}$, $\dot{r}$ and $\dot{\beta}$ are given by the composite $(g,s) \circ (f,r) = (\dot{g}f, \dot{r}s)$ from

  \begin{equation}\label{eq:compositor_beta_dot}
   \begin{tikzcd}
	A & I & B \\
	& {\dot{B}} & J & C\rlap{,}
	\arrow["f", from=1-1, to=1-2]
	\arrow[""{name=0, anchor=center, inner sep=0}, "{\dot{g}}"', from=1-2, to=2-2]
	\arrow["r"', from=1-3, to=1-2]
	\arrow[""{name=1, anchor=center, inner sep=0}, "g", from=1-3, to=2-3]
	\arrow["{\dot{r}}", from=2-3, to=2-2]
	\arrow["s", from=2-4, to=2-3]
	\arrow["{\SIGMA^{\dot{\beta}}}"{marking, allow upside down}, draw=none, from=0, to=1]
   \end{tikzcd}
  \end{equation}
  and where $\alpha^\catc_{\bullet,\bullet,\bullet}$ denotes the associator for $\catc$, $\gamma^F_{\bullet,\bullet}$ denotes the compositor for $F$, $F(\dot{\beta})_*$ is the mate of $F(\Sigma^{\dot{\beta}})$, $\sigma\colon F(s)_* F(\dot{r})_* \to (F(\dot{r}) F(s))_*$ is the canonical isomorphism given by composition of adjoints and $(\gamma^{F\; -1}_{\dot{r},s})_*\colon (F(\dot{r}) F(s))_* \to F(\dot{r} s)_*$ is the mate of the inverse of the compositor.
 \end{itemize}

 We show that $H$ is indeed a pseudofunctor.

 (a1) First we must show that $H$ is well-defined on 2-morphisms.
 Let us consider a $\Sigma$-extension of the 2-morphism $(\alpha, x_1, x_2, x_3, \delta_1 , \delta_2)$:
 \[\begin{tikzcd}
	A && I && B \\
	& X & {} & X & B \\
	&& D && B \\
	& X & {} & X & B \\
	A && J && B\rlap{.}
	\arrow[""{name=0, anchor=center, inner sep=0}, "f", from=1-1, to=1-3]
	\arrow[Rightarrow, no head, from=1-1, to=5-1]
	\arrow["{{{{{x_1}}}}}"{description}, curve={height=6pt}, from=1-3, to=2-2]
	\arrow[""{name=1, anchor=center, inner sep=0}, "{{{{x_1}}}}"{description}, curve={height=-6pt}, from=1-3, to=2-4]
	\arrow["{{{z_1}}}"{description}, from=1-3, to=3-3]
	\arrow["{{{r}}}"', from=1-5, to=1-3]
	\arrow[""{name=2, anchor=center, inner sep=0}, Rightarrow, no head, from=1-5, to=2-5]
	\arrow["{d}"{description}, curve={height=6pt}, from=2-2, to=3-3]
	\arrow["{{{{{\theta_1^{-1}}}}}}"', Rightarrow, from=2-3, to=2-2]
	\arrow["{{{{{\theta_1}}}}}"', Rightarrow, from=2-4, to=2-3]
	\arrow[""{name=3, anchor=center, inner sep=0}, "{d}"{description}, curve={height=-6pt}, from=2-4, to=3-3]
	\arrow["{{{{{x_3}}}}}"', from=2-5, to=2-4]
	\arrow[""{name=4, anchor=center, inner sep=0}, Rightarrow, no head, from=2-5, to=3-5]
	\arrow["z_3"', from=3-5, to=3-3] %
	\arrow[""{name=5, anchor=center, inner sep=0}, Rightarrow, no head, from=3-5, to=4-5]
	\arrow["{d}"{description}, curve={height=-6pt}, from=4-2, to=3-3]
	\arrow["{{{{{\theta_2}}}}}", Rightarrow, from=4-2, to=4-3]
	\arrow[""{name=6, anchor=center, inner sep=0}, "{d}"{description}, curve={height=6pt}, from=4-4, to=3-3]
	\arrow["{{{{{\theta_2}}}}}"', Rightarrow, from=4-4, to=4-3]
	\arrow["{{{{{x_3}}}}}"', from=4-5, to=4-4]
	\arrow[""{name=7, anchor=center, inner sep=0}, Rightarrow, no head, from=4-5, to=5-5]
	\arrow[""{name=8, anchor=center, inner sep=0}, "g"', from=5-1, to=5-3]
	\arrow["{{{z_2}}}"{description}, from=5-3, to=3-3]
	\arrow["{{{{{x_2}}}}}"{description}, curve={height=-6pt}, from=5-3, to=4-2]
	\arrow[""{name=9, anchor=center, inner sep=0}, "{{{{x_2}}}}"{description}, curve={height=6pt}, from=5-3, to=4-4]
	\arrow["{{{s}}}", from=5-5, to=5-3]
	\arrow["{d\circ\alpha}"', shift right=5, shorten <=26pt, shorten >=26pt, Rightarrow, from=0, to=8]
	\arrow["{{{{{\SIGMA^{\delta_1}}}}}}"{description}, draw=none, from=1, to=2]
	\arrow["{{{{{\SIGMA^{\psi}}}}}}"{description}, draw=none, from=3, to=4]
	\arrow["{{{{\SIGMA^{\psi}}}}}"{description}, draw=none, from=6, to=5]
	\arrow["{{{{{\SIGMA^{\delta_2}}}}}}"{description}, draw=none, from=9, to=7]
 \end{tikzcd}\]
 Applying $H$ to $(\alpha, x_1, x_2, x_3, \delta_1 , \delta_2)$ gives
 \[\begin{tikzcd}
	{F A} & {F I} & {F B} \\
	& {F X} & {F B} \\
	{F A} & {F J} & {F B}\rlap{,}
	\arrow[""{name=0, anchor=center, inner sep=0}, "Ff", from=1-1, to=1-2]
	\arrow[""{name=1, anchor=center, inner sep=0}, "\text{\tiny $F(x_1 f)$}"'{yshift=2pt,xshift=2pt}, from=1-1, to=2-2, gray]
	\arrow[equals, from=1-1, to=3-1]
	\arrow[""{name=2, anchor=center, inner sep=0}, "{F x_1}"{pos=0.3,xshift=-1pt}, from=1-2, to=2-2]
	\arrow["(Fr)_*", to=1-3, from=1-2]
	\arrow[""{name=3, anchor=center, inner sep=0}, equals, from=1-3, to=2-3]
	\arrow["F(\delta_1)_*"{xshift=-4pt}, shorten <=6pt, shorten >=6pt, Rightarrow, from=1-3, to=2-2]
	\arrow["F(x_3)_*"'{pos=0.4, yshift=1pt}, to=2-3, from=2-2]
	\arrow["F(\delta_2)_*^{-1}"{xshift=-5pt}, shorten <=6pt, shorten >=6pt, Rightarrow, from=2-2, to=3-3, yshift=-3pt,xshift=-1pt]
	\arrow[""{name=4, anchor=center, inner sep=0}, "\text{\tiny $F(x_2 g)$}"{yshift=-2pt,xshift=2pt}, from=3-1, to=2-2, gray]
	\arrow[""{name=5, anchor=center, inner sep=0}, "{F g}"', from=3-1, to=3-2]
	\arrow[""{name=6, anchor=center, inner sep=0}, "{F x_2}"'{pos=0.3,xshift=-1pt}, from=3-2, to=2-2]
	\arrow[""{name=7, anchor=center, inner sep=0}, equals, from=3-3, to=2-3]
	\arrow["(Fs)_*"', to=3-3, from=3-2]
	\arrow["{F(\alpha)}"', shorten <=22pt, shorten >=22pt, Rightarrow, from=0, to=5]
	\arrow["\text{\tiny $\gamma^F_{x_1,f}$}"'{xshift=2pt}, shorten >=2pt, Rightarrow, from=1-2, to=1, xshift=4pt, yshift=-2pt, gray]
	\arrow["\text{\tiny $\gamma^F_{x_2,g}{}^{\!\!\!\!-1}$}"'{yshift=-2pt,xshift=3pt}, shorten <=2pt, Rightarrow, from=4, to=3-2, xshift=4pt,yshift=3pt, gray]
  \end{tikzcd}\]
  while applying $H$ to the above $\Sigma$-extension gives
  \[\begin{tikzcd}[sep=large]
	{F A} & {F I} & {F B} \\
	& {F X} & {F B} \\
	{F A} & {F J} & {F B}\rlap{.}
	\arrow[""{name=0, anchor=center, inner sep=0}, "Ff", from=1-1, to=1-2]
	\arrow[""{name=1, anchor=center, inner sep=0}, "\text{\tiny $F(z_1 f)$}"'{yshift=2pt,xshift=2pt}, from=1-1, to=2-2, gray]
	\arrow[equals, from=1-1, to=3-1]
	\arrow[""{name=2, anchor=center, inner sep=0}, "{F z_1}"{pos=0.2,xshift=-1pt}, from=1-2, to=2-2]
	\arrow["(Fr)_*", to=1-3, from=1-2]
	\arrow[""{name=3, anchor=center, inner sep=0}, equals, from=1-3, to=2-3]
	\arrow["\text{\tiny $\substack{F(\theta_1r \cdot d\delta_1 \\ {} \cdot \psi)_*}$}"{pos=0.5,xshift=-2pt}, shorten <=10pt, shorten >=10pt, Rightarrow, from=1-3, to=2-2, yshift=5pt,xshift=-5pt]
	\arrow["F(z_3)_*"'{pos=0.4, yshift=1pt}, to=2-3, from=2-2]
	\arrow["\text{\tiny $\substack{F(\theta_2r \cdot d\delta_2 \\ {} \cdot \psi)_*^{-1}}$}"{pos=0.6, xshift=-5pt}, shorten <=10pt, shorten >=10pt, Rightarrow, from=2-2, to=3-3, yshift=-5pt,xshift=-5pt]
	\arrow[""{name=4, anchor=center, inner sep=0}, "\text{\tiny $F(z_2 g)$}"{yshift=-2pt,xshift=2pt}, from=3-1, to=2-2, gray]
	\arrow[""{name=5, anchor=center, inner sep=0}, "{F g}"', from=3-1, to=3-2]
	\arrow[""{name=6, anchor=center, inner sep=0}, "{F z_2}"'{pos=0.2,xshift=-1pt}, from=3-2, to=2-2]
	\arrow[""{name=7, anchor=center, inner sep=0}, equals, from=3-3, to=2-3]
	\arrow["(Fs)_*"', to=3-3, from=3-2]
	\arrow["\text{\tiny $\substack{F(\theta_2g \cdot d\alpha \\ {} \cdot \theta_1^{-1}f)}$}"', shorten <=30pt, shorten >=30pt, Rightarrow, from=0, to=5, xshift=5pt]
	\arrow["\text{\tiny $\gamma^F_{z_1,f}$}"'{xshift=2pt}, shorten >=2pt, Rightarrow, from=1-2, to=1, xshift=4pt, yshift=-2pt, gray]
	\arrow["\text{\tiny $\gamma^F_{z_2,g}{}^{\!\!\!\!-1}$}"'{yshift=-2pt,xshift=3pt}, shorten <=2pt, Rightarrow, from=4, to=3-2, xshift=4pt,yshift=3pt, gray]
  \end{tikzcd}\]
  We must show these are equal. In \cref{sec:appendixB} we prove this using string diagrams. Here we provide an alternative proof.
Consider the $\Sigma$-square $S_1$  and the square $F(S_1)$ in $\catc$  given by
\[S_1:=\begin{tikzcd}
	B && I \\
	B & X \\
	B && D
	\arrow["r", from=1-1, to=1-3]
	\arrow[""{name=0, anchor=center, inner sep=0}, equals, from=1-1, to=2-1]
	\arrow[""{name=1, anchor=center, inner sep=0}, "{x_1}"', from=1-3, to=2-2]
	\arrow[""{name=2, anchor=center, inner sep=0}, "{z_1}", from=1-3, to=3-3]
	\arrow["{x_3}"', from=2-1, to=2-2]
	\arrow[""{name=3, anchor=center, inner sep=0}, equals, from=2-1, to=3-1]
	\arrow[""{name=4, anchor=center, inner sep=0}, "d"', from=2-2, to=3-3]
	\arrow["{z_3}"', from=3-1, to=3-3]
	\arrow["{\SIGMA^{\delta_1}}"{description}, draw=none, from=0, to=1]
	\arrow["{\SIGMA^{\psi}}"{description}, draw=none, from=3, to=4]
	\arrow["{\theta_1}", between={0.1}{0.8}, Rightarrow, from=2-2, to=2]
\end{tikzcd}
\text{ \hspace{2mm} and\hspace{2mm}  }
F(S_1):=
\begin{tikzcd}
	FB && FI \\
	\\
	FB && FD
	\arrow["Fr", from=1-1, to=1-3]
	\arrow[equals, from=1-1, to=3-1]
	\arrow[""{name=0, anchor=center, inner sep=0}, "{{F(z_1r)}}"{description}, from=1-1, to=3-3]
	\arrow["{{Fz_1}}", from=1-3, to=3-3]
	\arrow["{{Fz_3}}"', from=3-1, to=3-3]
	\arrow["{{\gamma^{-1}}}", between={0.3}{0.8}, Rightarrow, from=0, to=1-3]
	\arrow["{{F\epsilon_1}}", between={0.2}{0.7}, Rightarrow, from=3-1, to=0]
\end{tikzcd}\]
where $\gamma$ refers to $\gamma^F_{z_1,r}$ and $\epsilon_1=(\psi\odot \delta_1)\oplus \theta_1=(z_3\xRightarrow{\psi}dx_3\xRightarrow{d\circ \delta_1} dx_1r \xRightarrow{\theta_1r}z_1r)$.
Now consider the pasting diagram
\[\bar{S}_1:=
\begin{tikzcd}
	FB && FI \\
	FB && FX \\
	FB && FD
	\arrow["Fr", from=1-1, to=1-3]
	\arrow[equals, from=1-1, to=2-1]
	\arrow[""{name=0, anchor=center, inner sep=0}, from=1-1, to=2-3]
	\arrow["{{{{{{{Fx_1}}}}}}}", from=1-3, to=2-3]
	\arrow[""{name=1, anchor=center, inner sep=0}, shift left=2, curve={height=-18pt}, from=1-3, to=3-3]
	\arrow[""{name=2, anchor=center, inner sep=0}, "{{{{{{{Fz_1}}}}}}}", shift left=5, curve={height=-30pt}, from=1-3, to=3-3]
	\arrow["{{{{{{{Fx_3}}}}}}}"', from=2-1, to=2-3]
	\arrow[equals, from=2-1, to=3-1]
	\arrow[""{name=3, anchor=center, inner sep=0}, from=2-1, to=3-3]
	\arrow["Fd", from=2-3, to=3-3]
	\arrow["{{{{{{{{Fz_3}}}}}}}}"', from=3-1, to=3-3]
	\arrow["{{{{{{{\gamma^{-1}}}}}}}}"'{pos=0.7}, between={0.2}{0.8}, Rightarrow, from=0, to=1-3]
	\arrow["{F\theta_1}", between={0.2}{0.8}, Rightarrow, from=1, to=2]
	\arrow["{{{{{{{\gamma^{-1}}}}}}}}"'{pos=0.7}, between={0.2}{0.8}, Rightarrow, from=3, to=2-3]
	\arrow["{{{{{{{F\delta_1}}}}}}}", between={0.2}{0.8}, Rightarrow, from=2-1, to=0]
	\arrow["\gamma"{pos=0.3}, between={0}{0.8}, Rightarrow, from=2-3, to=1]
	\arrow["{{{{{{{F\psi}}}}}}}", between={0.2}{0.8}, Rightarrow, from=3-1, to=3]
\end{tikzcd}\]
where $\gamma$ denotes in each occurrence the obvious component of the compositor of $F$, and the passage from $Fd\circ (Fx_1\circ Fr)$ to $(Fd\circ Fx_1)\circ Fr)$ is made by the inverse of the associator $\alpha^{\catc}_{Fd,Fx_1, Fr}$.
 We claim that indeed $FS_1=\bar{S}_1$. This is shown by the following diagram, where \textcircled{\scriptsize 1} and \textcircled{\scriptsize 2} are given by the naturality of $\gamma=\gamma^F$ and the `associativity' coherence condition of $F$, respectively.
\[\adjustbox{scale=0.90}{\begin{tikzcd}
	& {Fd\circ Fx_3} & {Fd\circ F(x_1r)} && {Fd\circ(Fx_1\circ Fr)} & {(Fd\circ Fx_1)\circ Fr} \\
	{Fz_3} & {F(dx_3)} & {F(d(x_1r))} & {F((dx_1)r)} & {F(dx_1)\circ Fr} & {F(dx_1)\circ Fr} \\
	&& {F(z_1\circ r)} &&& {Fz_1\circ Fr}
	\arrow["{{{{Fd\circ F\delta_1}}}}"{pos=0.4}, from=1-2, to=1-3]
	\arrow["{{{{{{{Fd\circ \gamma^{-1}}}}}}}}", from=1-3, to=1-5]
	\arrow["{{{{(\alpha^{\catc})^{-1}}}}}", from=1-5, to=1-6]
	\arrow["{{{{\gamma\circ Fr}}}}", from=1-6, to=2-6]
	\arrow["{{{F\psi}}}"', from=2-1, to=2-2]
	\arrow["{{{F\epsilon_1}}}"', curve={height=18pt}, from=2-1, to=3-3]
	\arrow[""{name=0, anchor=center, inner sep=0}, "{{{\gamma^{-1}}}}", from=2-2, to=1-2]
	\arrow["{{{F(d\circ \delta)}}}"', from=2-2, to=2-3]
	\arrow[""{name=1, anchor=center, inner sep=0}, "{{{{{{{\gamma^{-1}}}}}}}}"', from=2-3, to=1-3]
	\arrow[equals, from=2-3, to=2-4]
	\arrow[""{name=2, anchor=center, inner sep=0}, "{{{{{{{F(\theta_1r)}}}}}}}", from=2-3, to=3-3]
	\arrow["{{{{{{{\gamma^{-1}}}}}}}}", from=2-4, to=2-5]
	\arrow[""{name=3, anchor=center, inner sep=0}, "{{{{\gamma^{-1}\circ Fr}}}}"{pos=0.2}, from=2-5, to=1-6]
	\arrow[equals, from=2-5, to=2-6]
	\arrow[""{name=4, anchor=center, inner sep=0}, "{{{{F\theta_1\circ Fr}}}}", from=2-6, to=3-6]
	\arrow["{{{{{{{\gamma^{-1}_{z_1,r}}}}}}}}"', from=3-3, to=3-6]
	\arrow["{\text{\textcircled{\scriptsize 1}}}"{description}, draw=none, from=0, to=1]
	\arrow["{{\text{\textcircled{ \scriptsize 2}}}}"{description}, draw=none, from=1, to=3]
	\arrow["{\text{\textcircled{\scriptsize 1}}}"{description}, draw=none, from=2, to=4]
\end{tikzcd}}\]
Put
\[(F\delta_1)_*= \text{mate of } \begin{tikzcd}
	FB & FI \\
	FB & FD
	\arrow["Fr", from=1-1, to=1-2]
	\arrow[equals, from=1-1, to=2-1]
	\arrow[""{name=0, anchor=center, inner sep=0}, from=1-1, to=2-2]
	\arrow["{{Fx_1}}", from=1-2, to=2-2]
	\arrow["{{Fx_3}}"', from=2-1, to=2-2]
	\arrow["{{\gamma^{-1}}}"{pos=0.6}, shift right, between={0.2}{1}, Rightarrow, from=0, to=1-2]
	\arrow["{{F\delta_1}}"{pos=0.6}, shift right=2, between={0}{0.8}, Rightarrow, from=2-1, to=0]
\end{tikzcd}
\text{\hspace{5mm}and \hspace{5mm}}
(F\psi)_*= \text{mate of } \begin{tikzcd}
	FB & FX \\
	FB & FD
	\arrow["{{Fx_3}}", from=1-1, to=1-2]
	\arrow[equals, from=1-1, to=2-1]
	\arrow[""{name=0, anchor=center, inner sep=0}, from=1-1, to=2-2]
	\arrow["Fd", from=1-2, to=2-2]
	\arrow["{{Fz_3}}"', from=2-1, to=2-2]
	\arrow["{{{\gamma^{-1}}}}"{pos=0.6}, shift right, between={0.2}{1}, Rightarrow, from=0, to=1-2]
	\arrow["{{F\psi}}"{pos=0.6}, shift right=2, between={0}{0.8}, Rightarrow, from=2-1, to=0]
\end{tikzcd}\, .\]
The mate of the vertical composition of two Beck-Chevalley squares is clearly equal to the composition of the respective mates (see \cref{rem:lari}); and analogously for the horizontal composition of BC squares. Thus,
the mate of $\bar{S}_1$ is given by
\[(\bar{S}_1)_*=
\begin{tikzcd}
	FI & FI & FI & FB \\
	&& FX & FB \\
	FD & FD & FD & FB
	\arrow[equals, from=1-1, to=1-2]
	\arrow[""{name=0, anchor=center, inner sep=0}, "{Fz_1}"', from=1-1, to=3-1]
	\arrow[equals, from=1-2, to=1-3]
	\arrow[""{name=1, anchor=center, inner sep=0}, from=1-2, to=3-2]
	\arrow["{{{(Fr)_*}}}", from=1-3, to=1-4]
	\arrow["{{{Fx_1}}}"', from=1-3, to=2-3]
	\arrow["{{{(F\delta_1)_*}}}"', shift left=2, Rightarrow, from=1-4, to=2-3]
	\arrow[equals, from=1-4, to=2-4]
	\arrow["{{{(Fx_3)_*}}}"', from=2-3, to=2-4]
	\arrow["Fd"', from=2-3, to=3-3]
	\arrow["{{{(F\psi)_*}}}"{pos=0.6}, shift right, between={0.3}{1}, Rightarrow, from=2-4, to=3-3]
	\arrow[between={0}{0.9}, equals, from=2-4, to=3-4]
	\arrow[equals, from=3-1, to=3-2]
	\arrow[equals, from=3-2, to=3-3]
	\arrow["{{{(Fz_3)_*}}}"', from=3-3, to=3-4]
	\arrow["{F\theta_1}"', between={0.2}{0.8}, Rightarrow, from=1, to=0]
	\arrow["\gamma"', between={0}{0.8}, Rightarrow, from=2-3, to=1]
\end{tikzcd}
\]

What we did for $S_1$ and $\bar{S}_1$ can be done analogously for $S_2$ and $\bar{S}_2$, obtaining the equality $F(S_2)=\bar{S}_2$.

Thus, $H([\mu, z_1, z_2, z_3, \epsilon_1, \epsilon_2])$, where $\mu=(\theta_2\circ g)\cdot (d\circ \alpha)\cdot (\theta_1^{-1}\circ f)$ and $\epsilon_i=(\psi\odot \delta_i)\oplus \theta_i$, is given by the following pasting diagram:
\[\begin{tikzcd}
	FA &&& FI & FI & FI && FB \\
	&&&&& FX && FB \\
	&&& FD & FD & FD && FB \\
	&&&&& FX && FB \\
	FA &&& FJ & FJ & FJ && FB
	\arrow["Ff", from=1-1, to=1-4]
	\arrow[""{name=0, anchor=center, inner sep=0}, "{{F(z_1f)}}"{description}, from=1-1, to=3-4]
	\arrow[equals, from=1-1, to=5-1]
	\arrow[equals, from=1-4, to=1-5]
	\arrow[""{name=1, anchor=center, inner sep=0}, "{{Fz_1}}"{description}, from=1-4, to=3-4]
	\arrow[equals, from=1-5, to=1-6]
	\arrow[""{name=2, anchor=center, inner sep=0}, from=1-5, to=3-5]
	\arrow["{{{(Fr)_*}}}", from=1-6, to=1-8]
	\arrow["{{{Fx_1}}}", from=1-6, to=2-6]
	\arrow["{{{(F\delta_1)_*}}}"{description}, between={0.1}{1}, Rightarrow, from=1-8, to=2-6]
	\arrow[equals, from=1-8, to=2-8]
	\arrow["{{{(Fx_3)_*}}}", from=2-6, to=2-8]
	\arrow["Fd"', from=2-6, to=3-6]
	\arrow["{{{(F\psi)_*}}}"{description}, between={0.1}{0.9}, Rightarrow, from=2-8, to=3-6]
	\arrow[equals, from=2-8, to=3-8]
	\arrow[equals, from=3-4, to=3-5]
	\arrow[equals, from=3-5, to=3-6]
	\arrow["{{{(Fz_3)_*}}}", from=3-6, to=3-8]
	\arrow["{{{(F\psi)_*^{-1}}}}"{description}, between={0.1}{0.9}, Rightarrow, from=3-6, to=4-8]
	\arrow[equals, from=3-8, to=4-8]
	\arrow["Fd", from=4-6, to=3-6]
	\arrow["{{{(Fx_3)_*}}}"', from=4-6, to=4-8]
	\arrow["{{{(F\delta_2)_*^{-1}}}}"{description}, between={0.1}{0.9}, Rightarrow, from=4-6, to=5-8]
	\arrow[equals, from=4-8, to=5-8]
	\arrow[""{name=3, anchor=center, inner sep=0}, "{{F(z_2g)}}"{description}, from=5-1, to=3-4]
	\arrow["Fg"', from=5-1, to=5-4]
	\arrow[""{name=4, anchor=center, inner sep=0}, "{{Fz_2}}"{description}, from=5-4, to=3-4]
	\arrow[equals, from=5-4, to=5-5]
	\arrow[""{name=5, anchor=center, inner sep=0}, from=5-5, to=3-5]
	\arrow[equals, from=5-5, to=5-6]
	\arrow["{{{Fx_2}}}"', from=5-6, to=4-6]
	\arrow["{{{(Fs)_*}}}"', from=5-6, to=5-8]
	\arrow["{{F\mu}}"', shift right=5, between={0.2}{0.8}, Rightarrow, from=0, to=3]
	\arrow["\gamma"', between={0.2}{0.7}, Rightarrow, from=1-4, to=0]
	\arrow["{{F\theta_1}}"', between={0.2}{0.8}, Rightarrow, from=2, to=1]
	\arrow["\gamma"', between={0}{0.8}, Rightarrow, from=2-6, to=2]
	\arrow["{{\gamma^{-1}}}"', between={0.3}{0.8}, Rightarrow, from=3, to=5-4]
	\arrow["{{F\theta_2^{-1}}}", between={0.2}{0.8}, Rightarrow, from=4, to=5]
	\arrow["{{\gamma^{-1}}}", between={0.2}{1}, Rightarrow, from=5, to=4-6]
\end{tikzcd}\]
The pasting diagram representing $H([\alpha, x_1,x_2,x_3,\delta_1,\delta_2])$, by inserting in it the identity 2-cell
$$(Fx_3)_*\xRightarrow{(F\psi)_*}(Fz_3)_*Fd\xRightarrow{(F\psi)_*^{-1}}(Fx_3)_*\, ,$$
is equal to the pasting diagram
\[\begin{tikzcd}
	FA && FI && FB \\
	&& FX && FB \\
	&& FD && FB \\
	&& FX && FB \\
	FA && FJ && FB
	\arrow["Ff", from=1-1, to=1-3]
	\arrow[equals, from=1-1, to=5-1]
	\arrow["{{(Fr)_*}}", from=1-3, to=1-5]
	\arrow["{{Fx_1}}", from=1-3, to=2-3]
	\arrow["{Fd\circ F\alpha}"', between={0.3}{0.7}, Rightarrow, from=1-3, to=5-1]
	\arrow["{{(F\delta_1)_*}}"{description}, between={0.1}{1}, Rightarrow, from=1-5, to=2-3]
	\arrow[equals, from=1-5, to=2-5]
	\arrow["{{(Fx_3)_*}}", from=2-3, to=2-5]
	\arrow["Fd"', from=2-3, to=3-3]
	\arrow["{{(F\psi)_*}}"{description}, between={0.1}{0.9}, Rightarrow, from=2-5, to=3-3]
	\arrow[equals, from=2-5, to=3-5]
	\arrow["{{(Fz_3)_*}}", from=3-3, to=3-5]
	\arrow["{{(F\psi)_*^{-1}}}"{description}, between={0.1}{0.9}, Rightarrow, from=3-3, to=4-5]
	\arrow[equals, from=3-5, to=4-5]
	\arrow["Fd", from=4-3, to=3-3]
	\arrow["{{(Fx_3)_*}}"', from=4-3, to=4-5]
	\arrow["{{(F\delta_2)_*^{-1}}}"{description}, between={0.1}{0.9}, Rightarrow, from=4-3, to=5-5]
	\arrow[equals, from=4-5, to=5-5]
	\arrow["Fg", from=5-1, to=5-3]
	\arrow["{{Fx_2}}"', from=5-3, to=4-3]
	\arrow["{{(Fs)_*}}"', from=5-3, to=5-5]
\end{tikzcd}\; .\]
We see that the right part of the above two diagrams coincide. Concerning the left part of those diagrams, using the naturality of $\gamma^F$ and the associativity coherence condition of $F$,
we see that, after cancelling $F\theta_i$ with $F\theta_i^{-1}$, they are also equal, thus yielding the desired result. This is shown by the following diagram, where $\alpha^{\catc}$ refers to the associator of $\catc$.
\[\begin{tikzcd}
	{(Fd\circ Fx_1)\circ Ff} && {F(dx_1)\circ Ff} && {Fz_1\circ Ff} \\
	{Fd\circ(Fx_1\circ Ff)} && {F((dx_1)f)} && {F(z_1f)} \\
	{Fd\circ F(x_1f)} && {F(d(x_1f))} & {F((dx_1)f)} \\
	{Fd\circ F(x_2g)} && {F(d(x_2g))} & {F((dx_2)g)} \\
	{Fd\circ (Fx_2\circ Fg)} && {F((dx_2)g)} && {F(z_2f)} \\
	{(Fd\circ Fx_2)\circ Fg} && {F(dx_2)\circ Fg} && {Fz_2\circ Fg}
	\arrow["{{{{\gamma_{d,x_1}\circ Ff}}}}", from=1-1, to=1-3]
	\arrow["{\alpha^{\catc}_{Fd,Fx_1,Ff}}"', from=1-1, to=2-1]
	\arrow["{{F\theta_1\circ Ff}}", from=1-3, to=1-5]
	\arrow["{{\gamma_{dx_1,f}}}", from=1-3, to=2-3]
	\arrow["{{\gamma_{z_1,f}}}", from=1-5, to=2-5]
	\arrow["{{{{Fd\circ\gamma_{x_1,f}}}}}"', from=2-1, to=3-1]
	\arrow["{{F(\theta_1\circ f)}}"', from=2-3, to=2-5]
	\arrow[equals, from=2-3, to=3-3]
	\arrow["{{F(\theta_1^{-1}f)}}", from=2-5, to=3-4]
	\arrow["{{F\mu}}", from=2-5, to=5-5]
	\arrow["{{\gamma_{d,x_1f}}}", from=3-1, to=3-3]
	\arrow["{{{{Fd\circ F\alpha}}}}"', from=3-1, to=4-1]
	\arrow[equals, from=3-3, to=3-4]
	\arrow["{{F(d\circ\alpha)}}", from=3-3, to=4-3]
	\arrow["{{\gamma_{d,x_2g}}}"', from=4-1, to=4-3]
	\arrow["{{Fd\circ \gamma^{-1}}}"', from=4-1, to=5-1]
	\arrow[equals, from=4-3, to=4-4]
	\arrow[equals, from=4-3, to=5-3]
	\arrow["{{F(\theta_2 g)}}", from=4-4, to=5-5]
	\arrow["{(\alpha^{\catc})^{-1}_{Fd,Fx_2,Fg}}"', from=5-1, to=6-1]
	\arrow["{{\gamma^{-1}_{dx_2,g}}}", from=5-3, to=6-3]
	\arrow["{{F(\theta_2^{-1}g)}}", from=5-5, to=5-3]
	\arrow["{{\gamma^{-1}_{z_2,f}}}", from=5-5, to=6-5]
	\arrow["{{\gamma_{d,x_2}^{-1}\circ Fg}}", from=6-3, to=6-1]
	\arrow["{{F\theta_2^{-1}\circ Fg}}", from=6-5, to=6-3]
\end{tikzcd}\]

  (a2) It is easy to see $H$ sends identity 2-cells to identity 2-cells by direct computation. %
  We now show that $H$ preserves vertical composition of 2-cells. Consider 2-morphisms $\bar{\alpha} = (\alpha, x_1, x_2, x_3, \delta_1, \delta_2)\colon (f,r) \Rightarrow (g,s)$ and $\bar{\beta} = (\beta, y_1, y_2, y_3, \epsilon_1, \epsilon_2)\colon (g,s) \Rightarrow (h,t)$. By applying Rule 4' to $\Sigma^{\delta_2}$ and $\Sigma^{\epsilon_1}$ we obtain
  \begin{equation*}
   \begin{tikzcd}
    B & J \\
    B & X & Y & {} \\
    B & D
    \arrow["s", from=1-1, to=1-2]
    \arrow[""{name=0, anchor=center, inner sep=0}, Rightarrow, no head, from=1-1, to=2-1]
    \arrow[""{name=1, anchor=center, inner sep=0}, "{{x_2}}", from=1-2, to=2-2]
    \arrow["{y_1}", curve={height=-6pt}, from=1-2, to=2-3]
    \arrow["{x_3}", from=2-1, to=2-2]
    \arrow[""{name=2, anchor=center, inner sep=0}, Rightarrow, no head, from=2-1, to=3-1]
    \arrow["\gamma", Rightarrow, from=2-2, to=2-3]
    \arrow["\cong"', draw=none, from=2-2, to=2-3]
    \arrow[""{name=3, anchor=center, inner sep=0}, "{{d_x}}", from=2-2, to=3-2]
    \arrow["{d_y}", curve={height=-6pt}, from=2-3, to=3-2]
    \arrow["u"', from=3-1, to=3-2]
    \arrow["{\SIGMA^{\delta_2}}"{description}, draw=none, from=0, to=1]
    \arrow["{\SIGMA^{\phi_x}}"{description}, draw=none, from=2, to=3]
   \end{tikzcd}
   =\qquad
   \begin{tikzcd}
    B & J \\
    B & Y \\
    B & D \rlap{\,.}
    \arrow["s", from=1-1, to=1-2]
    \arrow[""{name=0, anchor=center, inner sep=0}, Rightarrow, no head, from=1-1, to=2-1]
    \arrow[""{name=1, anchor=center, inner sep=0}, "{{y_1}}", from=1-2, to=2-2]
    \arrow["{{y_3}}", dashed, from=2-1, to=2-2]
    \arrow[""{name=2, anchor=center, inner sep=0}, Rightarrow, no head, from=2-1, to=3-1]
    \arrow[""{name=3, anchor=center, inner sep=0}, "{{d_y}}", from=2-2, to=3-2]
    \arrow["u"', from=3-1, to=3-2]
    \arrow["{\SIGMA^{\epsilon_1}}"{description}, draw=none, from=0, to=1]
    \arrow["{\SIGMA^{\phi_y}}"{description}, draw=none, from=2, to=3]
   \end{tikzcd}
  \end{equation*}
  We may use $\Sigma^{\phi_x}$ and $\gamma$ to obtain a $\Sigma$-extension of $\bar{\alpha}$ and $\Sigma^{\phi_y}$ to obtain a $\Sigma$-extension of $\bar{\beta}$, which agree on their abutting $\Sigma$-squares. By well-definedness of $H$ we may replace $\bar{\alpha}$ and $\bar{\beta}$ by these. Thus, we may assume without loss of generality that $\Sigma^{\delta_2} = \Sigma^{\epsilon_1}$. Such 2-morphisms compose in a particularly simple way (since the necessary $\Sigma$-squares can be chosen to be identities). In particular, we may simply compose the $\alpha$ and $\beta$ and remove the matching $\Sigma$-squares. It is now easy to see that the repeated parts also cancel after applying $H$, since the mate and the inverse of the mate of the same morphism end up adjacent to each other and cancel. Thus, $H$ preserves vertical composition.

(a3) Now we show naturality of the compositors $\gamma^H$.
   Given horizontally composable 2-morphisms $\bar{\alpha}=(\alpha, x_1,x_2,x_3, \delta_1, \delta_2)\colon (f_1,r_1) \to (f_2,r_2)$ and $\bar{\beta}=(\beta, y_1,y_2,y_3, \epsilon_1, \epsilon_2)\colon (g_1,s_1) \to (g_2,s_2)$, we must prove
  \begin{equation}\label{eq:naturality_of_gamma}
   \gamma^H_{(g_2,s_2),(f_2,r_2)} \cdot (H(\bar{\beta})\circ H(\bar{\alpha})) = H(\bar{\beta} \circ \bar{\alpha}) \cdot \gamma^H_{(g_1,s_1),(f_1,r_1)}.
  \end{equation}
  Verifying this equality requires some technical calculations which are provided in \cref{sec:appendixB}.

  (a4) It is easy to see that the unit coherence condition for the compositor holds. We claim that also the associativity condition is satisfied.
  Suppose we have the triple composite
  $\bar{h}\circ \bar{g}\circ \bar{f}\colon W\to Z$ with $\bar{h}=(h,t)$, $\bar{g}=(g,s)$ and $\bar{f}=(f,r)$.  Our aim is to show the commutativity of the following diagram in $\catc(W,Z)$:
  \[\begin{tikzcd}
	{(H\bar{h}\circ H\bar{g})\circ H\bar{f}} && {H\bar{h}\circ (H\bar{g}\circ H\bar{f})} && {H\bar{h}\circ H(\bar{g}\circ \bar{f})} \\
	{ H(\bar{h}\circ \bar{g})\circ H\bar{f}} && {H((\bar{h}\circ \bar{g})\circ \bar{f})} && {H(\bar{h}\circ (\bar{g}\circ \bar{f})}
	\arrow["{\alpha^{\catc}_{H\bar{h},H\bar{g},H\bar{f}}}", from=1-1, to=1-3]
	\arrow["{\gamma^H_{\bar{h},\bar{g}}\circ 1_{H\bar{f}}}"', from=1-1, to=2-1]
	\arrow["{1_{H\bar{h}}\circ \gamma^H_{\bar{g},\bar{f}}}", from=1-3, to=1-5]
	\arrow["{\gamma^H_{\bar{h},\bar{g}\bar{f}}}", from=1-5, to=2-5]
	\arrow["{\gamma^H_{\bar{h}\bar{g},\bar{f}}}"', from=2-1, to=2-3]
	\arrow["{H\Assoc_{\bar{h},\bar{g},\bar{f}}}"', from=2-3, to=2-5]
\end{tikzcd}\]
The detailed calculations are again quite technical and can be found in \cref{sec:appendixB}.

With (a1)-(a4) we have shown that $H$ is a pseudofunctor.
  It is then easy to see that $F = H \circ P_\Sigma$. %
  The hardest part is checking that the compositor of $H \circ P_\Sigma$ agrees with that of $F$. This holds since the 2-cell $\dot{\beta}$ in the canonical $\Sigma$-square in \cref{eq:compositor_beta_dot} for $\gamma^H_{(g,1),(f,1)}$ %
  is chosen to be the identity and due to strict unitarity of the bicategories involved.

 (b) We now show the `2-dimensional' universality condition. Assume $H, H'\colon \catx[\Sigma_*] \to \catc$ are strictly unitary pseudofunctors and $\xi\colon H \circ P_\Sigma \to H' \circ P_\Sigma$ is a pseudonatural transformation for which the pseudonaturality squares involving $\xi_r$ are BC squares whenever $r \in \Sigma$. We define $\upsilon\colon H \to H'$ by
  \begin{itemize}
   \item $\upsilon_X = \xi_X$ for objects $X \in \catx$,
   \item $\upsilon_{(f,r)}$ for $(f,r) \in \catx[\Sigma_*]$ is given by the composite
   \begin{equation}\label{eq:upsilon}
   \begin{tikzcd}
	{H(A)} & {H(I)} & {H(B)} \\
	{H'(A)} & {H'(I)} & {H'(B)}
	\arrow["{H(f)}", from=1-1, to=1-2]
	\arrow["{\xi_A}"', from=1-1, to=2-1]
	\arrow["{H(r)_*}", from=1-2, to=1-3]
	\arrow["{\xi_I}", from=1-2, to=2-2]
	\arrow["{\xi_B}", from=1-3, to=2-3]
	\arrow["{\xi_f}"', shorten <=4pt, shorten >=4pt, Rightarrow, from=2-1, to=1-2]
	\arrow["{H'(f)}"', from=2-1, to=2-2]
	\arrow["{(\xi_r)_*^{-1}}"', shift left, shorten <=4pt, shorten >=4pt, Rightarrow, from=2-2, to=1-3]
	\arrow["{H'(r)_*}"', from=2-2, to=2-3]
   \end{tikzcd}
   \end{equation}
   where we write $H(f)$ for $H(f,1) = H(P_\Sigma(f))$ and
   where $(\xi_r)_*^{-1}$ denotes the inverse of the mate of $\xi_r$. Note that this inverse exists by the BC assumption.
   We also note that, by \cref{lem:transformation_at_adjoint} of \cref{sec:appendixB}, the right part of \cref{eq:upsilon} is essentially the mate of the square given by $\xi_{r}$. %
  \end{itemize}

  Assuming that $\upsilon$ is a pseudonatural transformation, by expanding the definition it is  easy to see that $\xi = \upsilon \circ P_\Sigma$.

  We prove the pseudonaturality of $\upsilon$ in \cref{sec:appendixB}.

(c) Finally, we show the `3-dimensional' universality condition. For data as in the statement of the theorem we simply set $\beth_X = \aleph_X$. It is clear that, being a modification, it is  the only modification $\beth$ such that $\aleph = \beth \circ P_\Sigma$.

 Let us show that $\beth$ is indeed a modification. The necessary naturality condition on $\beth$ holds for 1-morphisms of the form $(f,1)$ by same condition on $\aleph$. It then suffices to show the condition for 1-morphisms of the form $(1,r)$, since every morphism of $\catx[\Sigma_*]$ is a composite of these two kinds. This is easy using string diagrams as shown in \cref{sec:appendixB}.
\end{proof}

As a consequence of the previous theorem we recover the result that if $\catx$ has a class of morphisms $\Sigma$ admitting a full calculus of left lax fractions, as described in \cref{def:Full-Sigma}, then $\catx[\Sigma_*]$ is the universal bicategory which turns morphisms of $\Sigma$ into equivalences.

\begin{corollary}\label{cor:universal_prop}
 Let $\Sigma$ be a class of morphisms of the 2-category $\catx$ admitting a full calculus of left lax fractions.  The pseudofunctor $P_\Sigma\colon \catx \to \catx[\Sigma_*]$ is the universal (strictly unital) pseudofunctor %
 between those taking morphisms of $\Sigma$ to equivalences.
 More precisely, we have
 \begin{enumerate}[label=\normalfont\alph*)]
  \item If $F\colon \catx \to \catc$ is a pseudofunctor sending morphisms of $\Sigma$ to equivalences then there is a pseudofunctor $H\colon \catx[\Sigma_*] \to \catc$ such that $F \simeq H \circ P_\Sigma$.
  \item If $H,H'\colon \catx[\Sigma_*] \to \catc$ are pseudofunctors and $\xi\colon H \circ P_\Sigma \to H' \circ P_\Sigma$ is a pseudonatural transformation, then there is a pseudonatural transformation $\upsilon\colon H \to H'$ such that $\xi \cong \upsilon \circ P_\Sigma$.
  \item If $H,H'\colon \catx[\Sigma_*] \to \catc$ are pseudofunctors, $\upsilon, \upsilon' \colon H \to H'$ are pseudonatural transformations, and $\aleph \colon \upsilon \circ P_\Sigma \to \upsilon' \circ P_\Sigma$ is a modification, then there is a unique modification $\beth\colon \upsilon \to \upsilon'$ such that $\aleph = \beth \circ P_\Sigma$.
 \end{enumerate}
 \end{corollary}

 \begin{proof}
 First observe that in a bicategory every pseudocommutative square whose horizontal morphisms are equivalences is a Beck-Chevalley square.

 Let $\Sigma$ be a class of morphisms of the 2-category $\catx$ which, seen as a full subcategory of $\catx^{\to\cong}$, admits a calculus of left lax fractions,   and let $F\colon \catx \to \catc$ be a (strictly unital) pseudofunctor. The corollary immediately follows from the equivalence of the following two conditions:
 \begin{enumerate}[label=(\roman*)]
  \item $F$ takes morphisms in $\Sigma$ into equivalences;
  \item $F$ takes morphisms in $\Sigma$ into laris and pseudocommutative squares whose horizontal morphisms are in $\Sigma$ into Beck-Chevalley squares.
 \end{enumerate}

 That (i) implies (ii) is clear. For the converse, assume (ii) and take $r\in \Sigma$. By hypothesis, the commutative square
 \[\begin{tikzcd}
	A & B \\
	B & B
	\arrow["r", from=1-1, to=1-2]
	\arrow["r"', from=1-1, to=2-1]
	\arrow["{1_B}", from=1-2, to=2-2]
	\arrow["{1_B}"', from=2-1, to=2-2]
\end{tikzcd}\]
 corresponds to a morphism of the subcategory $\Sigma$ of $\catx^{\to\cong}$. Consequently, the commutative square in the middle of the diagram

 \[\begin{tikzcd}
	& FB \\
	FA & FB \\
	FB & FB \\
	FB
	\arrow["{(Fr)_*}"', from=1-2, to=2-1]
	\arrow[""{name=0, anchor=center, inner sep=0}, equals, from=1-2, to=2-2]
	\arrow["Fr"', from=2-1, to=2-2]
	\arrow[""{name=1, anchor=center, inner sep=0}, "Fr"', from=2-1, to=3-1]
	\arrow[""{name=2, anchor=center, inner sep=0}, "{F1_B}", from=2-2, to=3-2]
	\arrow["{1_{FB}}"', from=3-1, to=3-2]
	\arrow[""{name=3, anchor=center, inner sep=0}, between={0.2}{1}, equals, from=3-1, to=4-1]
	\arrow[""{name=4, anchor=center, inner sep=0}, "{1_{FB}}", from=3-2, to=4-1]
	\arrow["{\epsilon^{Fr}}"{description, pos=0.4}, between={0}{0.8}, Rightarrow, from=2-1, to=0]
	\arrow[between={0.4}{0.6}, equals, from=1, to=2]
	\arrow[between={0.3}{0.7}, equals, from=3, to=4]
\end{tikzcd}\]
 is a Beck-Chevalley square in $\catc$. But this just means that the counit $\epsilon^{Fr}$ is invertible. Since $Fr$ is a lari, the corresponding unit is also invertible, hence $Fr$ is an equivalence.
 \end{proof}

\appendix

\section{Appendix: On \texorpdfstring{$\Sigma$}{Σ}-paths and \texorpdfstring{$\Omega$}{Ω} 2-cells}\label{sec:appendix}

This appendix completes Subsection~\ref{sec:sigma_schemes} by providing the proofs of \cref{pro:du=ud} (here \cref{cor:length2}), \cref{pro:Omega-f} (here included in \cref{pro:alpha.f}), and \cref{pro:of-interest} (here \cref{thm:of-interest5}).

\Cref{lem:use} and \Cref{rem:use} will have a role in the proof of \cref{pro:adend}, which states a fundamental property needed for the rest of this section.

\begin{lemma}\label{lem:use} Given $\begin{tikzcd}
	{} & {} & {} \\
	{} & {} & {}
	\arrow["u", from=1-1, to=1-2]
	\arrow[""{name=0, anchor=center, inner sep=0}, "f"', from=1-1, to=2-1]
	\arrow["v", from=1-2, to=1-3]
	\arrow[""{name=1, anchor=center, inner sep=0}, "{{{{f_i}}}}", from=1-2, to=2-2]
	\arrow[""{name=2, anchor=center, inner sep=0}, "{{{{f'_i}}}}", from=1-3, to=2-3]
	\arrow["{u_i}"', from=2-1, to=2-2]
	\arrow["{v_i}"', from=2-2, to=2-3]
	\arrow["\SIGMA"{description}, draw=none, from=0, to=1]
	\arrow["\SIGMA"{description}, draw=none, from=1, to=2]
\end{tikzcd}$, for $i=1,2$, there are $\Sigma$-squares and invertible 2-cells $\theta$ and $\theta'$ such that
\[
(1)\quad  \begin{tikzcd}
	{} & {} && {} & {} \\
	{} & {} & {} & {} & {} \\
	{} & {} && {} & {}
	\arrow["u", from=1-1, to=1-2]
	\arrow[""{name=0, anchor=center, inner sep=0}, "f"', from=1-1, to=2-1]
	\arrow[""{name=1, anchor=center, inner sep=0}, "{{{{{f_1}}}}}", from=1-2, to=2-2]
	\arrow["{{f_2}}", curve={height=-12pt}, from=1-2, to=2-3]
	\arrow["u", from=1-4, to=1-5]
	\arrow[""{name=2, anchor=center, inner sep=0}, "f"', from=1-4, to=2-4]
	\arrow[""{name=3, anchor=center, inner sep=0}, "{{{{f_2}}}}", from=1-5, to=2-5]
	\arrow["{{{{{u_1}}}}}"', from=2-1, to=2-2]
	\arrow[""{name=4, anchor=center, inner sep=0}, equals, from=2-1, to=3-1]
	\arrow["\theta"{description}, Rightarrow, from=2-2, to=2-3]
	\arrow[""{name=5, anchor=center, inner sep=0}, "{{{{d_1}}}}", from=2-2, to=3-2]
	\arrow["\text{\normalsize $=$}"{description}, draw=none, from=2-3, to=2-4]
	\arrow["{{d_2}}", curve={height=-12pt}, from=2-3, to=3-2]
	\arrow["{{{{u_2}}}}"', from=2-4, to=2-5]
	\arrow[""{name=6, anchor=center, inner sep=0}, equals, from=2-4, to=3-4]
	\arrow[""{name=7, anchor=center, inner sep=0}, "{{d_2}}", from=2-5, to=3-5]
	\arrow["d"', from=3-1, to=3-2]
	\arrow["d"', from=3-4, to=3-5]
	\arrow["\SIGMA"{description}, draw=none, from=0, to=1]
	\arrow["\SIGMA"{description}, draw=none, from=2, to=3]
	\arrow["\SIGMA"{description}, draw=none, from=4, to=5]
	\arrow["\SIGMA"{description}, draw=none, from=6, to=7]
\end{tikzcd}
\; , \quad \quad
(2) \quad
\begin{tikzcd}
	{} & {} &&& {} & {} \\
	{} & {} & {} & {} & {} & {} \\
	{} & {} &&& {} & {}
	\arrow["v", from=1-1, to=1-2]
	\arrow[""{name=0, anchor=center, inner sep=0}, "{{{f_1}}}"', from=1-1, to=2-1]
	\arrow[""{name=1, anchor=center, inner sep=0}, "{{{f'_1}}}", from=1-2, to=2-2]
	\arrow["{{{f'_2}}}", curve={height=-12pt}, from=1-2, to=2-3]
	\arrow["v", from=1-5, to=1-6]
	\arrow["{{{f_1}}}"', curve={height=12pt}, from=1-5, to=2-4]
	\arrow[""{name=2, anchor=center, inner sep=0}, "{{{{{f_2}}}}}"', from=1-5, to=2-5]
	\arrow[""{name=3, anchor=center, inner sep=0}, "{{{f'_2}}}", from=1-6, to=2-6]
	\arrow["{{{v_1}}}"', from=2-1, to=2-2]
	\arrow[""{name=4, anchor=center, inner sep=0}, "{{{d_1}}}"', from=2-1, to=3-1]
	\arrow["{{{\theta'}}}"{description}, Rightarrow, from=2-2, to=2-3]
	\arrow[""{name=5, anchor=center, inner sep=0}, "{{{e_1}}}", from=2-2, to=3-2]
	\arrow["{\text{\normalsize $=$}}"{description}, draw=none, from=2-3, to=2-4]
	\arrow["{{{e_2}}}", curve={height=-12pt}, from=2-3, to=3-2]
	\arrow["\theta", Rightarrow, from=2-4, to=2-5]
	\arrow["{{{d_1}}}"', curve={height=12pt}, from=2-4, to=3-5]
	\arrow["{{{v_2}}}"', from=2-5, to=2-6]
	\arrow[""{name=6, anchor=center, inner sep=0}, "{{{{{d_2}}}}}"', from=2-5, to=3-5]
	\arrow[""{name=7, anchor=center, inner sep=0}, "{{{e_2}}}", from=2-6, to=3-6]
	\arrow["e"', from=3-1, to=3-2]
	\arrow["e"', from=3-5, to=3-6]
	\arrow["\SIGMA"{description}, draw=none, from=0, to=1]
	\arrow["\SIGMA"{description}, draw=none, from=2, to=3]
	\arrow["\SIGMA"{description}, draw=none, from=4, to=5]
	\arrow["\SIGMA"{description}, draw=none, from=6, to=7]
\end{tikzcd}
\]
and, consequently, also
\[
(3) \quad
\begin{tikzcd}
	{} & {} & {} \\
	{} & {} & {} & {} \\
	{} & {} & {}
	\arrow["u", from=1-1, to=1-2]
	\arrow[""{name=0, anchor=center, inner sep=0}, "f"', from=1-1, to=2-1]
	\arrow["v", from=1-2, to=1-3]
	\arrow[""{name=1, anchor=center, inner sep=0}, "{{f_1}}", from=1-2, to=2-2]
	\arrow[""{name=2, anchor=center, inner sep=0}, "{{f'_1}}", from=1-3, to=2-3]
	\arrow["{{f'_2}}", curve={height=-12pt}, from=1-3, to=2-4]
	\arrow["{u_1}"', from=2-1, to=2-2]
	\arrow[""{name=3, anchor=center, inner sep=0}, equals, from=2-1, to=3-1]
	\arrow["{{v_1}}"', from=2-2, to=2-3]
	\arrow[""{name=4, anchor=center, inner sep=0}, "{{d_1}}", from=2-2, to=3-2]
	\arrow["{{\theta'}}"{description}, Rightarrow, from=2-3, to=2-4]
	\arrow[""{name=5, anchor=center, inner sep=0}, "{{e_1}}", from=2-3, to=3-3]
	\arrow["{{e_2}}", curve={height=-12pt}, from=2-4, to=3-3]
	\arrow["d"', from=3-1, to=3-2]
	\arrow["e"', from=3-2, to=3-3]
	\arrow["\SIGMA"{description}, draw=none, from=0, to=1]
	\arrow["\SIGMA"{description}, draw=none, from=1, to=2]
	\arrow["\SIGMA"{description}, draw=none, from=3, to=4]
	\arrow["\SIGMA"{description}, draw=none, from=4, to=5]
\end{tikzcd}
\quad \text{\normalsize $=$} \quad
\begin{tikzcd}
	{} & {} & {} \\
	{} & {} & {} \\
	{} & {} & {}
	\arrow["u", from=1-1, to=1-2]
	\arrow[""{name=0, anchor=center, inner sep=0}, "f"', from=1-1, to=2-1]
	\arrow["v", from=1-2, to=1-3]
	\arrow[""{name=1, anchor=center, inner sep=0}, "{f_2}", from=1-2, to=2-2]
	\arrow[""{name=2, anchor=center, inner sep=0}, "{f'_2}", from=1-3, to=2-3]
	\arrow["{{u_2}}"', from=2-1, to=2-2]
	\arrow[""{name=3, anchor=center, inner sep=0}, equals, from=2-1, to=3-1]
	\arrow["{{{v_2}}}"', from=2-2, to=2-3]
	\arrow[""{name=4, anchor=center, inner sep=0}, "{d_2}", from=2-2, to=3-2]
	\arrow[""{name=5, anchor=center, inner sep=0}, "{e_2}", from=2-3, to=3-3]
	\arrow["d"', from=3-1, to=3-2]
	\arrow["e"', from=3-2, to=3-3]
	\arrow["\SIGMA"{description}, draw=none, from=0, to=1]
	\arrow["\SIGMA"{description}, draw=none, from=1, to=2]
	\arrow["\SIGMA"{description}, draw=none, from=3, to=4]
	\arrow["\SIGMA"{description}, draw=none, from=4, to=5]
\end{tikzcd}
\, . \]
\end{lemma}

\begin{proof} (1) This is just Rule 4' of Proposition \ref{pro:useful_rules}.

(2) By Rule 6 of Proposition \ref{pro:useful_rules}, we obtain
\[
\begin{tikzcd}
	{} & {} \\
	{} \\
	{} & {}
	\arrow["v", from=1-1, to=1-2]
	\arrow["{f_i}"', from=1-1, to=2-1]
	\arrow[""{name=0, anchor=center, inner sep=0}, "{q_i}", from=1-2, to=3-2]
	\arrow["{d_i}"', from=2-1, to=3-1]
	\arrow["q"', from=3-1, to=3-2]
	\arrow["\SIGMA"{description}, draw=none, from=2-1, to=0]
\end{tikzcd}
\quad (i=1,2) \quad \text{and} \quad
\psi\colon q_1\Rightarrow q_2
\]
such that
\[
\begin{tikzcd}
	{} & {} \\
	{} \\
	D & Q
	\arrow["v", from=1-1, to=1-2]
	\arrow["{f_1}"', from=1-1, to=2-1]
	\arrow[""{name=0, anchor=center, inner sep=0}, "{q_1}", from=1-2, to=3-2]
	\arrow[""{name=1, anchor=center, inner sep=0}, "{q_2}", shift left=3, curve={height=-30pt}, from=1-2, to=3-2]
	\arrow["{d_1}"', from=2-1, to=3-1]
	\arrow["q"', from=3-1, to=3-2]
	\arrow["\psi"{pos=0.6}, shorten <=11pt, shorten >=7pt, Rightarrow, from=0, to=1]
	\arrow["\SIGMA"{description}, draw=none, from=2-1, to=0]
\end{tikzcd}
\quad = \quad
\begin{tikzcd}
	& {} & {} \\
	{} & {} \\
	& D & Q
	\arrow["v", from=1-2, to=1-3]
	\arrow["{f_1}"', curve={height=12pt}, from=1-2, to=2-1]
	\arrow["{f_2}"', from=1-2, to=2-2]
	\arrow[""{name=0, anchor=center, inner sep=0}, "{q_2}", from=1-3, to=3-3]
	\arrow["\theta", Rightarrow, from=2-1, to=2-2]
	\arrow["{d_1}"', curve={height=12pt}, from=2-1, to=3-2]
	\arrow["{d_2}"', from=2-2, to=3-2]
	\arrow["q"', from=3-2, to=3-3]
	\arrow["\SIGMA"{description}, draw=none, from=2-2, to=0]
\end{tikzcd}
\; .\]
On the other hand, for each $i=1,2$, by applying Square followed by  Rule 4', we get
\[
\begin{tikzcd}
	{} & {} \\
	{} & {} \\
	{} & {D_i} & Q \\
	D & {R_i}
	\arrow["v", from=1-1, to=1-2]
	\arrow[""{name=0, anchor=center, inner sep=0}, "{f_i}"', from=1-1, to=2-1]
	\arrow[""{name=1, anchor=center, inner sep=0}, "{{{f'_i}}}", from=1-2, to=2-2]
	\arrow["{q_i}", curve={height=-12pt}, from=1-2, to=3-3]
	\arrow["{{{v_i}}}"', from=2-1, to=2-2]
	\arrow[""{name=2, anchor=center, inner sep=0}, "{d_i}"', from=2-1, to=3-1]
	\arrow[""{name=3, anchor=center, inner sep=0}, "{d'_i}", from=2-2, to=3-2]
	\arrow["{w_i}"', from=3-1, to=3-2]
	\arrow[""{name=4, anchor=center, inner sep=0}, equals, from=3-1, to=4-1]
	\arrow[""{name=5, anchor=center, inner sep=0}, "{r_{i1}}", from=3-2, to=4-2]
	\arrow["{r_{i2}}", curve={height=-6pt}, from=3-3, to=4-2]
	\arrow["{r_i}"', from=4-1, to=4-2]
	\arrow["\SIGMA"{description}, draw=none, from=0, to=1]
	\arrow["\SIGMA"{description}, draw=none, from=2, to=3]
	\arrow["{\phi_i}"{pos=0.6}, shorten <=12pt, Rightarrow, from=3, to=3-3]
	\arrow["\SIGMA"{description}, draw=none, from=4, to=5]
\end{tikzcd}
\quad = \quad
\begin{tikzcd}
	{} & {} \\
	{} \\
	{} & Q \\
	D & {R_i}
	\arrow["v", from=1-1, to=1-2]
	\arrow["{{f_i}}"', from=1-1, to=2-1]
	\arrow[""{name=0, anchor=center, inner sep=0}, "{q_i}", from=1-2, to=3-2]
	\arrow["{{d_i}}"', from=2-1, to=3-1]
	\arrow["{{q}}"', from=3-1, to=3-2]
	\arrow[""{name=1, anchor=center, inner sep=0}, equals, from=3-1, to=4-1]
	\arrow[""{name=2, anchor=center, inner sep=0}, "{r_{i2}}", from=3-2, to=4-2]
	\arrow["{{r_i}}"', from=4-1, to=4-2]
	\arrow["\SIGMA"{description}, draw=none, from=2-1, to=0]
	\arrow["\SIGMA"{description}, draw=none, from=1, to=2]
\end{tikzcd}
\; .\]
Using again Rule 4',
 we obtain
 \[
 \begin{tikzcd}
	D & Q \\
	D & {R_1} & {R_2} \\
	D & E
	\arrow["q", from=1-1, to=1-2]
	\arrow[""{name=0, anchor=center, inner sep=0}, equals, from=1-1, to=2-1]
	\arrow[""{name=1, anchor=center, inner sep=0}, "{{r_{12}}}", from=1-2, to=2-2]
	\arrow["{r_{22}}", curve={height=-6pt}, from=1-2, to=2-3]
	\arrow["{{{r_1}}}"', from=2-1, to=2-2]
	\arrow[""{name=2, anchor=center, inner sep=0}, equals, from=2-1, to=3-1]
	\arrow["\rho", Rightarrow, from=2-2, to=2-3]
	\arrow[""{name=3, anchor=center, inner sep=0}, "{s_1}", from=2-2, to=3-2]
	\arrow["{s_2}", curve={height=-6pt}, from=2-3, to=3-2]
	\arrow["e"', from=3-1, to=3-2]
	\arrow["\SIGMA"{description}, draw=none, from=0, to=1]
	\arrow["\SIGMA"{description}, draw=none, from=2, to=3]
\end{tikzcd}
\quad = \quad
\begin{tikzcd}
	D & Q \\
	D & {R_2} \\
	D & E
	\arrow["q", from=1-1, to=1-2]
	\arrow[""{name=0, anchor=center, inner sep=0}, equals, from=1-1, to=2-1]
	\arrow[""{name=1, anchor=center, inner sep=0}, "{r_{22}}", from=1-2, to=2-2]
	\arrow["{r_2}"', from=2-1, to=2-2]
	\arrow[""{name=2, anchor=center, inner sep=0}, equals, from=2-1, to=3-1]
	\arrow[""{name=3, anchor=center, inner sep=0}, "{s_2}", from=2-2, to=3-2]
	\arrow["e"', from=3-1, to=3-2]
	\arrow["\SIGMA"{description}, draw=none, from=0, to=1]
	\arrow["\SIGMA"{description}, draw=none, from=2, to=3]
\end{tikzcd}
\]
with $\rho$ invertible. Thus we can form the following pasting diagram:
\[\begin{tikzcd}
	{} & {} \\
	{} & {} &&& {} \\
	{} & {D_1} && Q & {D_2} \\
	D & {R_1} && {R_2} \\
	D & E
	\arrow["v", from=1-1, to=1-2]
	\arrow[""{name=0, anchor=center, inner sep=0}, "{{{{f_1}}}}"', from=1-1, to=2-1]
	\arrow[""{name=1, anchor=center, inner sep=0}, "{{{{f'_1}}}}"', from=1-2, to=2-2]
	\arrow["{{{f'_2}}}", from=1-2, to=2-5]
	\arrow[""{name=2, anchor=center, inner sep=0}, "{{{q_1}}}"{description}, curve={height=12pt}, from=1-2, to=3-4]
	\arrow[""{name=3, anchor=center, inner sep=0}, "{{{q_2}}}"{description}, curve={height=-12pt}, from=1-2, to=3-4]
	\arrow["{{{{v_1}}}}"', from=2-1, to=2-2]
	\arrow[""{name=4, anchor=center, inner sep=0}, "{{{{d_1}}}}"', from=2-1, to=3-1]
	\arrow[""{name=5, anchor=center, inner sep=0}, "{{{{d'_1}}}}", from=2-2, to=3-2]
	\arrow["{{{d'_2}}}", from=2-5, to=3-5]
	\arrow["{{{{w_1}}}}"', from=3-1, to=3-2]
	\arrow[""{name=6, anchor=center, inner sep=0}, equals, from=3-1, to=4-1]
	\arrow[""{name=7, anchor=center, inner sep=0}, "{{{{r_{11}}}}}", from=3-2, to=4-2]
	\arrow["{{{\phi_2^{-1}}}}", between={0.1}{0.8}, Rightarrow, from=3-4, to=2-5]
	\arrow["{{{r_{12}}}}"', from=3-4, to=4-2]
	\arrow["{{{r_{22}}}}", from=3-4, to=4-4]
	\arrow["{{{r_{21}}}}", from=3-5, to=4-4]
	\arrow["{{{{r_1}}}}"', from=4-1, to=4-2]
	\arrow[""{name=8, anchor=center, inner sep=0}, equals, from=4-1, to=5-1]
	\arrow["\rho"', between={0.2}{0.7}, Rightarrow, from=4-2, to=4-4]
	\arrow[""{name=9, anchor=center, inner sep=0}, "{{{{s_1}}}}", from=4-2, to=5-2]
	\arrow["{{{{s_2}}}}", from=4-4, to=5-2]
	\arrow["e"', from=5-1, to=5-2]
	\arrow["\SIGMA"{description}, draw=none, from=0, to=1]
	\arrow["\psi"', between={0.2}{0.8}, Rightarrow, from=2, to=3]
	\arrow["\SIGMA"{description}, draw=none, from=4, to=5]
	\arrow["\SIGMA"{description}, draw=none, from=6, to=7]
	\arrow["{{{\phi_1}}}"', shift right=3, between={0}{0.8}, Rightarrow, from=3-2, to=2]
	\arrow["\SIGMA"{description}, draw=none, from=8, to=9]
\end{tikzcd}\]
Put $e_i=s_ir_{i1}d'_i$ and $\theta'=(s_2\circ \phi_2^{-1})\cdot (\rho \circ \psi)\cdot (s_1\circ \phi_1)$.
It is easy to see that in this way we obtain the desired situation (2).
\end{proof}

\begin{remark}\label{rem:use} Given two 2-morphisms
\begin{equation}\label{eq:eq-ab}
\begin{tikzcd}
	{} & {} & {} \\
	& {} & {} \\
	{} & {} & {}
	\arrow["f", from=1-1, to=1-2]
	\arrow[equals, from=1-1, to=3-1]
	\arrow[""{name=0, anchor=center, inner sep=0}, "{{{x_1}}}"', from=1-2, to=2-2]
	\arrow["\alpha"', between={0.3}{0.7}, Rightarrow, from=1-2, to=3-1]
	\arrow["r"', from=1-3, to=1-2]
	\arrow[""{name=1, anchor=center, inner sep=0}, equals, from=1-3, to=2-3]
	\arrow["{{{x_3}}}", from=2-3, to=2-2]
	\arrow[""{name=2, anchor=center, inner sep=0}, equals, from=2-3, to=3-3]
	\arrow["g"', from=3-1, to=3-2]
	\arrow[""{name=3, anchor=center, inner sep=0}, "{{{x_2}}}", from=3-2, to=2-2]
	\arrow["s", from=3-3, to=3-2]
	\arrow["\SIGMA"{description}, draw=none, from=0, to=1]
	\arrow["\SIGMA"{description}, draw=none, from=3, to=2]
\end{tikzcd}
\qquad \text{ and } \qquad
\begin{tikzcd}
	{} & {} & {} \\
	& {} & {} \\
	{} & {} & {}
	\arrow["f", from=1-1, to=1-2]
	\arrow[equals, from=1-1, to=3-1]
	\arrow[""{name=0, anchor=center, inner sep=0}, "{{{y_1}}}"', from=1-2, to=2-2]
	\arrow["\beta"', between={0.3}{0.7}, Rightarrow, from=1-2, to=3-1]
	\arrow["r"', from=1-3, to=1-2]
	\arrow[""{name=1, anchor=center, inner sep=0}, equals, from=1-3, to=2-3]
	\arrow["{{{y_3}}}", from=2-3, to=2-2]
	\arrow[""{name=2, anchor=center, inner sep=0}, equals, from=2-3, to=3-3]
	\arrow["g"', from=3-1, to=3-2]
	\arrow[""{name=3, anchor=center, inner sep=0}, "{{{y_2}}}", from=3-2, to=2-2]
	\arrow["s", from=3-3, to=3-2]
	\arrow["\SIGMA"{description}, draw=none, from=0, to=1]
	\arrow["\SIGMA"{description}, draw=none, from=3, to=2]
\end{tikzcd}
\end{equation}
apply Rule 4 to obtain $\mu_1$ and $\mu_2$ invertible such that
\[
\begin{tikzcd}
	{} & {} \\
	{} & {} & {} \\
	{} & {} \\
	{} & {} & {} \\
	{} & {}
	\arrow["r", from=1-1, to=1-2]
	\arrow[""{name=0, anchor=center, inner sep=0}, equals, from=1-1, to=2-1]
	\arrow[""{name=1, anchor=center, inner sep=0}, "{{{{{{x_1}}}}}}", from=1-2, to=2-2]
	\arrow["{{{{{{y_1}}}}}}", curve={height=-6pt}, from=1-2, to=2-3]
	\arrow["{{{{{{x_3}}}}}}"', from=2-1, to=2-2]
	\arrow[""{name=2, anchor=center, inner sep=0}, equals, from=2-1, to=3-1]
	\arrow["{{{{\mu_1}}}}", between={0}{0.9}, Rightarrow, from=2-2, to=2-3]
	\arrow[""{name=3, anchor=center, inner sep=0}, "{{{t_1}}}"{pos=0.4}, from=2-2, to=3-2]
	\arrow["{{{{{{t_2}}}}}}"{pos=0.4}, curve={height=-6pt}, from=2-3, to=3-2]
	\arrow["t"', from=3-1, to=3-2]
	\arrow[""{name=4, anchor=center, inner sep=0}, equals, from=3-1, to=4-1]
	\arrow["{{{{{{x_3}}}}}}"', from=4-1, to=4-2]
	\arrow[""{name=5, anchor=center, inner sep=0}, equals, from=4-1, to=5-1]
	\arrow[""{name=6, anchor=center, inner sep=0}, "{{{t_1}}}"', from=4-2, to=3-2]
	\arrow["{{{{\mu_2}}}}", between={0}{0.9}, Rightarrow, from=4-2, to=4-3]
	\arrow["{{{{{{t_2}}}}}}"'{pos=0.4}, curve={height=6pt}, from=4-3, to=3-2]
	\arrow["s"', from=5-1, to=5-2]
	\arrow[""{name=7, anchor=center, inner sep=0}, "{{{{{{x_2}}}}}}"', from=5-2, to=4-2]
	\arrow["{{{{{{y_2}}}}}}"', curve={height=6pt}, from=5-2, to=4-3]
	\arrow["\SIGMA"{description}, draw=none, from=0, to=1]
	\arrow["\SIGMA"{description}, draw=none, from=2, to=3]
	\arrow["\SIGMA"{description}, draw=none, from=4, to=6]
	\arrow["\SIGMA"{description}, draw=none, from=5, to=7]
\end{tikzcd}
\hspace{5mm} = \hspace{5mm}
\begin{tikzcd}
	{} & {} \\
	{} & {} \\
	{} & {} \\
	{} & {} \\
	{} & {}
	\arrow["r", from=1-1, to=1-2]
	\arrow[""{name=0, anchor=center, inner sep=0}, equals, from=1-1, to=2-1]
	\arrow[""{name=1, anchor=center, inner sep=0}, "{{y_1}}", from=1-2, to=2-2]
	\arrow["{{y_3}}"', from=2-1, to=2-2]
	\arrow[""{name=2, anchor=center, inner sep=0}, equals, from=2-1, to=3-1]
	\arrow[""{name=3, anchor=center, inner sep=0}, "{{t_2}}", from=2-2, to=3-2]
	\arrow["t"', from=3-1, to=3-2]
	\arrow[""{name=4, anchor=center, inner sep=0}, equals, from=3-1, to=4-1]
	\arrow["{{y_3}}"', from=4-1, to=4-2]
	\arrow[""{name=5, anchor=center, inner sep=0}, equals, from=4-1, to=5-1]
	\arrow[""{name=6, anchor=center, inner sep=0}, "{{t_2}}"', from=4-2, to=3-2]
	\arrow["s"', from=5-1, to=5-2]
	\arrow[""{name=7, anchor=center, inner sep=0}, "{{y_2}}"', from=5-2, to=4-2]
	\arrow["\SIGMA"{description}, draw=none, from=0, to=1]
	\arrow["\SIGMA"{description}, draw=none, from=2, to=3]
	\arrow["\SIGMA"{description}, draw=none, from=4, to=6]
	\arrow["\SIGMA"{description}, draw=none, from=5, to=7]
\end{tikzcd}\, .
\]
Assume that there is a $\Sg$-square
\begin{equation}\label{eq:sp-Square}\begin{tikzcd}
	{} & {} \\
	{} & {}
	\arrow["u", from=1-1, to=1-2]
	\arrow[""{name=0, anchor=center, inner sep=0}, "h"', from=1-1, to=2-1]
	\arrow[""{name=1, anchor=center, inner sep=0}, "f", from=1-2, to=2-2]
	\arrow["r"', from=2-1, to=2-2]
	\arrow["\SIGMA"{description}, shift left, draw=none, from=0, to=1]
\end{tikzcd}
\end{equation}
such that
\begin{equation*}
\begin{tikzcd}
	{} & {} & {} \\
	{} & {} \\
	{} & {} & {} & {} \\
	{} & {}
	\arrow["u", from=1-1, to=1-2]
	\arrow[""{name=0, anchor=center, inner sep=0}, "h"', from=1-1, to=2-1]
	\arrow["g", from=1-2, to=1-3]
	\arrow[""{name=1, anchor=center, inner sep=0}, "f", from=1-2, to=2-2]
	\arrow[""{name=2, anchor=center, inner sep=0}, "{{{{{{x_2}}}}}}"{description}, from=1-3, to=3-3]
	\arrow["{{{{y_2}}}}"{description}, from=1-3, to=3-4]
	\arrow["r"{description}, from=2-1, to=2-2]
	\arrow[""{name=3, anchor=center, inner sep=0}, equals, from=2-1, to=3-1]
	\arrow[""{name=4, anchor=center, inner sep=0}, "{{{{{{y_1}}}}}}", from=2-2, to=3-2]
	\arrow["{{{{x_1}}}}", between={0}{0.9}, from=2-2, to=3-3]
	\arrow["{{{y_3}}}"{description}, from=3-1, to=3-2]
	\arrow[""{name=5, anchor=center, inner sep=0}, equals, from=3-1, to=4-1]
	\arrow["{{\mu_1^{-1}}}", shift right, between={0.1}{0.8}, Rightarrow, from=3-2, to=3-3]
	\arrow[""{name=6, anchor=center, inner sep=0}, "{{{{{{t_2}}}}}}"{description}, from=3-2, to=4-2]
	\arrow["{{{{\mu_2}}}}", between={0.1}{0.9}, Rightarrow, from=3-3, to=3-4]
	\arrow["{{{t_1}}}"{description, pos=0.45}, between={0}{0.8}, from=3-3, to=4-2]
	\arrow["{{{t_2}}}"{description}, from=3-4, to=4-2]
	\arrow["t"', from=4-1, to=4-2]
	\arrow["\SIGMA"{description}, draw=none, from=0, to=1]
	\arrow["\alpha", between={0.3}{0.8}, Rightarrow, from=1, to=2]
	\arrow["\SIGMA"{description}, draw=none, from=3, to=4]
	\arrow["\SIGMA"{description}, draw=none, from=5, to=6]
\end{tikzcd}
\hspace{5mm} =  \hspace{5mm}
\begin{tikzcd}
	{} & {} \\
	{} & {} & {} \\
	{} & {} \\
	{} & {       }
	\arrow["u", from=1-1, to=1-2]
	\arrow[""{name=0, anchor=center, inner sep=0}, "h"', from=1-1, to=2-1]
	\arrow[""{name=1, anchor=center, inner sep=0}, "f", from=1-2, to=2-2]
	\arrow["g", from=1-2, to=2-3]
	\arrow["r"{description}, from=2-1, to=2-2]
	\arrow[""{name=2, anchor=center, inner sep=0}, equals, from=2-1, to=3-1]
	\arrow["\beta", between={0.1}{0.9}, Rightarrow, from=2-2, to=2-3]
	\arrow[""{name=3, anchor=center, inner sep=0}, "{{{{y_1}}}}", from=2-2, to=3-2]
	\arrow["{{{{{y_2}}}}}", from=2-3, to=3-2]
	\arrow["{{{y_3}}}"{description}, from=3-1, to=3-2]
	\arrow[""{name=4, anchor=center, inner sep=0}, equals, from=3-1, to=4-1]
	\arrow[""{name=5, anchor=center, inner sep=0}, "{{{{{{t_2}}}}}}", from=3-2, to=4-2]
	\arrow["t"', from=4-1, to=4-2]
	\arrow["\SIGMA"{description}, draw=none, from=0, to=1]
	\arrow["\SIGMA"{description}, draw=none, from=2, to=3]
	\arrow["\SIGMA"{description}, draw=none, from=4, to=5]
\end{tikzcd}\,.
\end{equation*}
Then the 2-morphisms of \eqref{eq:eq-ab} are $\approx$-equivalent.

Indeed, applying Equification to the $\Sg$-square $\begin{tikzcd}
	{} & {} \\
	{} & {}
	\arrow["u", from=1-1, to=1-2]
	\arrow[""{name=0, anchor=center, inner sep=0}, "h"', from=1-1, to=2-1]
	\arrow[""{name=1, anchor=center, inner sep=0}, from=1-2, to=2-2]
	\arrow["t"', from=2-1, to=2-2]
	\arrow["\SIGMA"{description}, shift left, draw=none, from=0, to=1]
\end{tikzcd}$ together with the 2-cells $(\mu_2\circ g)\cdot (t_1\circ \alpha) \cdot (\mu_1^{-1}\circ f)$ and $t_2\circ \beta$, we obtain a 2-morphism which is a $\Sg$-extension of both the two 2-morphisms of \eqref{eq:eq-ab}.
\end{remark}

\begin{proposition}\label{pro:adend} For $\Sg$-schemes $S_1$, $S_2$, $S_3$ and $S_4$ of level 3 with common left border, every two $\Sg$-paths of the form
$S_1\begin{tikzcd}
	{} & {}
	\arrow["i", squiggly, from=1-1, to=1-2]
\end{tikzcd}S_3\begin{tikzcd}
	{} & {}
	\arrow["j", squiggly, from=1-1, to=1-2]
\end{tikzcd}S_2$
  and
  $S_1\begin{tikzcd}
	{} & {}
	\arrow["j", squiggly, from=1-1, to=1-2]
\end{tikzcd}S_4\begin{tikzcd}
	{} & {}
	\arrow["i", squiggly, from=1-1, to=1-2]
\end{tikzcd}S_2$,
where $i,j \in \{ \mathbf{d},\mathbf{u},\mathbf{s},\mathbf{d}_1,\mathbf{s}_1\}$, are equivalent.
\end{proposition}

\begin{proof}
The cases $i=j$, $\{i,j\}=\{\bd, {\bd}_1\}$, and  $\{i,j\}=\{\bs, {\bs}_1\}$ follow from Lemma~\ref{lem:Sigma-steps}.

\noindent
\underline{Case $i=\mathbf{d}$ and $j=\mathbf{u}$}. We have
\(\begin{tikzcd}
	{S_1} & {S_3} & {S_4}
	\arrow["\mathbf{d}", squiggly, from=1-1, to=1-2]
	\arrow["\mathbf{u}", squiggly, from=1-2, to=1-3]
\end{tikzcd}\) and \(\begin{tikzcd}
	{S_1} & {S_4} & {S_2}
	\arrow["\mathbf{u}", squiggly, from=1-1, to=1-2]
	\arrow["\mathbf{d}", squiggly, from=1-2, to=1-3]
\end{tikzcd}\).
The combination of the $\Sigma$-steps of type $\mathbf{d}$ with the ones of type $\mathbf{u}$ oblige the four $\Sigma$-schemes to have the configuration
\[\adjustbox{scale=0.60}{\begin{tikzcd}
	&& {} & {} \\
	& {} & {} & {} \\
	{} & {} & {} & {} \\
	{} & {} & {} & {}
	\arrow[from=1-3, to=1-4]
	\arrow[from=1-3, to=2-3]
	\arrow[from=1-4, to=3-4]
	\arrow[from=2-2, to=2-3]
	\arrow[from=2-2, to=3-2]
	\arrow[from=2-3, to=3-3]
	\arrow[from=3-1, to=3-2]
	\arrow[from=3-1, to=4-1]
	\arrow[from=3-2, to=3-3]
	\arrow[from=3-3, to=3-4]
	\arrow[from=3-3, to=4-3]
	\arrow[from=3-4, to=4-4]
	\arrow[from=4-1, to=4-3]
	\arrow[from=4-3, to=4-4]
\end{tikzcd}}\, .\]
Hence, our $\Sg$-paths look as follows.
\[
\text{\normalsize (P1)} \quad
\adjustbox{scale=0.70}{\begin{tikzcd}
	&& {} & {} \\
	& {} & {} \\
	{} & {} & {} & {} \\
	{} && {} & {}
	\arrow["r", from=1-3, to=1-4]
	\arrow["g"', from=1-3, to=2-3]
	\arrow[""{name=0, anchor=center, inner sep=0}, "{{h_2}}", from=1-4, to=3-4]
	\arrow["s", from=2-2, to=2-3]
	\arrow[""{name=1, anchor=center, inner sep=0}, "h"', from=2-2, to=3-2]
	\arrow[""{name=2, anchor=center, inner sep=0}, "{{h_1}}", from=2-3, to=3-3]
	\arrow["t", from=3-1, to=3-2]
	\arrow[""{name=3, anchor=center, inner sep=0}, "k"', from=3-1, to=4-1]
	\arrow["{{s_1}}"', from=3-2, to=3-3]
	\arrow["{{r_1}}", from=3-3, to=3-4]
	\arrow[""{name=4, anchor=center, inner sep=0}, "{{k_1}}"', from=3-3, to=4-3]
	\arrow[""{name=5, anchor=center, inner sep=0}, "{{k_2}}", from=3-4, to=4-4]
	\arrow["{{s_2}}"', from=4-1, to=4-3]
	\arrow["{{r_2}}"', from=4-3, to=4-4]
	\arrow["\SIGMA"{description}, draw=none, from=1, to=2]
	\arrow["\SIGMA"{description}, draw=none, from=2-3, to=0]
	\arrow["\SIGMA"{description}, draw=none, from=3, to=4]
	\arrow["\SIGMA"{description}, draw=none, from=4, to=5]
\end{tikzcd}}
\begin{tikzcd}
	{} & {}
	\arrow["\bd", squiggly, from=1-1, to=1-2]
\end{tikzcd}
\hspace*{-3mm}
\adjustbox{scale=0.70}{\begin{tikzcd}
	&& {} & {} \\
	& {} & {} \\
	{} & {} & {} & {} \\
	{} && {} & {}
	\arrow["r", from=1-3, to=1-4]
	\arrow["g"', from=1-3, to=2-3]
	\arrow[""{name=0, anchor=center, inner sep=0}, "{{h_2}}", from=1-4, to=3-4]
	\arrow["s", from=2-2, to=2-3]
	\arrow[""{name=1, anchor=center, inner sep=0}, "h"', from=2-2, to=3-2]
	\arrow[""{name=2, anchor=center, inner sep=0}, "{{h_1}}", from=2-3, to=3-3]
	\arrow["t", from=3-1, to=3-2]
	\arrow[""{name=3, anchor=center, inner sep=0}, "k"', from=3-1, to=4-1]
	\arrow["{{s_1}}"', from=3-2, to=3-3]
	\arrow["{{r_1}}", from=3-3, to=3-4]
	\arrow[""{name=4, anchor=center, inner sep=0}, "{k'_1}"', from=3-3, to=4-3]
	\arrow[""{name=5, anchor=center, inner sep=0}, "{k'_2}", from=3-4, to=4-4]
	\arrow["{{s'_2}}"', from=4-1, to=4-3]
	\arrow["{{r'_2}}"', from=4-3, to=4-4]
	\arrow["\SIGMA"{description}, draw=none, from=1, to=2]
	\arrow["\SIGMA"{description}, draw=none, from=2-3, to=0]
	\arrow["\SIGMA"{description}, draw=none, from=3, to=4]
	\arrow["\SIGMA"{description}, draw=none, from=4, to=5]
\end{tikzcd}}
\begin{tikzcd}
	{} & {}
	\arrow["\bu", squiggly, from=1-1, to=1-2]
\end{tikzcd}
\hspace*{-3mm}
\adjustbox{scale=0.70}{\begin{tikzcd}
	&& {} & {} \\
	& {} & {} \\
	{} & {} & {} & {} \\
	{} && {} & {}
	\arrow["r", from=1-3, to=1-4]
	\arrow["g"', from=1-3, to=2-3]
	\arrow[""{name=0, anchor=center, inner sep=0}, "{h_3}", from=1-4, to=3-4]
	\arrow["s", from=2-2, to=2-3]
	\arrow[""{name=1, anchor=center, inner sep=0}, "h"', from=2-2, to=3-2]
	\arrow[""{name=2, anchor=center, inner sep=0}, "{{{h_1}}}", from=2-3, to=3-3]
	\arrow["t", from=3-1, to=3-2]
	\arrow[""{name=3, anchor=center, inner sep=0}, "k"', from=3-1, to=4-1]
	\arrow["{s_1}"', from=3-2, to=3-3]
	\arrow["{r_3}", from=3-3, to=3-4]
	\arrow[""{name=4, anchor=center, inner sep=0}, "{{k'_1}}"', from=3-3, to=4-3]
	\arrow[""{name=5, anchor=center, inner sep=0}, "{k'_3}", from=3-4, to=4-4]
	\arrow["{{{s'_2}}}"', from=4-1, to=4-3]
	\arrow["{r'_3}"', from=4-3, to=4-4]
	\arrow["\SIGMA"{description}, draw=none, from=1, to=2]
	\arrow["\SIGMA"{description}, draw=none, from=2-3, to=0]
	\arrow["\SIGMA"{description}, draw=none, from=3, to=4]
	\arrow["\SIGMA"{description}, draw=none, from=4, to=5]
\end{tikzcd}}
\]

\[
\text{\normalsize (P2)} \quad
\adjustbox{scale=0.70}{\begin{tikzcd}
	&& {} & {} \\
	& {} & {} \\
	{} & {} & {} & {} \\
	{} && {} & {}
	\arrow["r", from=1-3, to=1-4]
	\arrow["g"', from=1-3, to=2-3]
	\arrow[""{name=0, anchor=center, inner sep=0}, "{h_2}", from=1-4, to=3-4]
	\arrow["s", from=2-2, to=2-3]
	\arrow[""{name=1, anchor=center, inner sep=0}, "h"', from=2-2, to=3-2]
	\arrow[""{name=2, anchor=center, inner sep=0}, "{h_1}", from=2-3, to=3-3]
	\arrow["t", from=3-1, to=3-2]
	\arrow[""{name=3, anchor=center, inner sep=0}, "k"', from=3-1, to=4-1]
	\arrow["{s_1}"', from=3-2, to=3-3]
	\arrow["{r_1}", from=3-3, to=3-4]
	\arrow[""{name=4, anchor=center, inner sep=0}, "{k_1}"', from=3-3, to=4-3]
	\arrow[""{name=5, anchor=center, inner sep=0}, "{k_2}", from=3-4, to=4-4]
	\arrow["{s_2}"', from=4-1, to=4-3]
	\arrow["{r_2}"', from=4-3, to=4-4]
	\arrow["\SIGMA"{description}, draw=none, from=1, to=2]
	\arrow["\SIGMA"{description}, draw=none, from=2-3, to=0]
	\arrow["\SIGMA"{description}, draw=none, from=3, to=4]
	\arrow["\SIGMA"{description}, draw=none, from=4, to=5]
\end{tikzcd}}
\begin{tikzcd}
	{} & {}
	\arrow["\bu", squiggly, from=1-1, to=1-2]
\end{tikzcd}
\hspace*{-3mm}
\adjustbox{scale=0.70}{\begin{tikzcd}
	&& {} & {} \\
	& {} & {} \\
	{} & {} & {} & {} \\
	{} && {} & {}
	\arrow["r", from=1-3, to=1-4]
	\arrow["g"', from=1-3, to=2-3]
	\arrow[""{name=0, anchor=center, inner sep=0}, "{h_3}", from=1-4, to=3-4]
	\arrow["s", from=2-2, to=2-3]
	\arrow[""{name=1, anchor=center, inner sep=0}, "h"', from=2-2, to=3-2]
	\arrow[""{name=2, anchor=center, inner sep=0}, "{{{h_1}}}", from=2-3, to=3-3]
	\arrow["t", from=3-1, to=3-2]
	\arrow[""{name=3, anchor=center, inner sep=0}, "k"', from=3-1, to=4-1]
	\arrow["{s_1}"', from=3-2, to=3-3]
	\arrow["{r_3}", from=3-3, to=3-4]
	\arrow[""{name=4, anchor=center, inner sep=0}, "{{k_1}}"', from=3-3, to=4-3]
	\arrow[""{name=5, anchor=center, inner sep=0}, "{k_3}", from=3-4, to=4-4]
	\arrow["{s_2}"', from=4-1, to=4-3]
	\arrow["{r_4}"', from=4-3, to=4-4]
	\arrow["\SIGMA"{description}, draw=none, from=1, to=2]
	\arrow["\SIGMA"{description}, draw=none, from=2-3, to=0]
	\arrow["\SIGMA"{description}, draw=none, from=3, to=4]
	\arrow["\SIGMA"{description}, draw=none, from=4, to=5]
\end{tikzcd}}
\begin{tikzcd}
	{} & {}
	\arrow["\bd", squiggly, from=1-1, to=1-2]
\end{tikzcd}
\hspace*{-3mm}
\adjustbox{scale=0.70}{\begin{tikzcd}
	&& {} & {} \\
	& {} & {} \\
	{} & {} & {} & {} \\
	{} && {} & {}
	\arrow["r", from=1-3, to=1-4]
	\arrow["g"', from=1-3, to=2-3]
	\arrow[""{name=0, anchor=center, inner sep=0}, "{h_3}", from=1-4, to=3-4]
	\arrow["s", from=2-2, to=2-3]
	\arrow[""{name=1, anchor=center, inner sep=0}, "h"', from=2-2, to=3-2]
	\arrow[""{name=2, anchor=center, inner sep=0}, "{{{h_1}}}", from=2-3, to=3-3]
	\arrow["t", from=3-1, to=3-2]
	\arrow[""{name=3, anchor=center, inner sep=0}, "k"', from=3-1, to=4-1]
	\arrow["{s_1}"', from=3-2, to=3-3]
	\arrow["{r_3}", from=3-3, to=3-4]
	\arrow[""{name=4, anchor=center, inner sep=0}, "{{k'_1}}"', from=3-3, to=4-3]
	\arrow[""{name=5, anchor=center, inner sep=0}, "{k'_3}", from=3-4, to=4-4]
	\arrow["{{{s'_2}}}"', from=4-1, to=4-3]
	\arrow["{r'_3}"', from=4-3, to=4-4]
	\arrow["\SIGMA"{description}, draw=none, from=1, to=2]
	\arrow["\SIGMA"{description}, draw=none, from=2-3, to=0]
	\arrow["\SIGMA"{description}, draw=none, from=3, to=4]
	\arrow["\SIGMA"{description}, draw=none, from=4, to=5]
\end{tikzcd}}
\]
We want to show that they give rise to the same $\Omega$ 2-cell.

\noindent \textbf{A.} We describe a 2-morphism representing the $\Omega$ 2-cell corresponding to (P1).

Observe that, using Lemma \ref{lem:use}, we obtain $\Sg$-squares and invertible 2-cells $\theta$ and $\theta'$ such that
\begin{equation}\label{eq:eq-A1}
 \hspace{-10mm} \begin{tikzcd}
	{} & {} & {} \\
	{} && {} & {} \\
	{} && {}
	\arrow["t", from=1-1, to=1-2]
	\arrow[""{name=0, anchor=center, inner sep=0}, "k"', from=1-1, to=2-1]
	\arrow["{{{{{{s_1}}}}}}", from=1-2, to=1-3]
	\arrow[""{name=1, anchor=center, inner sep=0}, "{{{k_1}}}"', from=1-3, to=2-3]
	\arrow["{{k'_1}}", curve={height=-6pt}, from=1-3, to=2-4]
	\arrow["{{{{s_2}}}}"', from=2-1, to=2-3]
	\arrow[""{name=2, anchor=center, inner sep=0}, equals, from=2-1, to=3-1]
	\arrow["\theta", between={0.1}{0.9}, Rightarrow, from=2-3, to=2-4]
	\arrow[""{name=3, anchor=center, inner sep=0}, "{{{a_1}}}"', from=2-3, to=3-3]
	\arrow["{{a_2}}", curve={height=-6pt}, from=2-4, to=3-3]
	\arrow["a"', from=3-1, to=3-3]
	\arrow["\SIGMA"{description}, draw=none, from=0, to=1]
	\arrow["\SIGMA"{description}, draw=none, from=2, to=3]
\end{tikzcd}
 =
\begin{tikzcd}
	{} & {} & {} \\
	{} && {} \\
	{} && {}
	\arrow["t", from=1-1, to=1-2]
	\arrow[""{name=0, anchor=center, inner sep=0}, "k"', from=1-1, to=2-1]
	\arrow["{{{{{{s_1}}}}}}", from=1-2, to=1-3]
	\arrow[""{name=1, anchor=center, inner sep=0}, "{{k'_1}}", from=1-3, to=2-3]
	\arrow["{s'_2}"{description}, from=2-1, to=2-3]
	\arrow[""{name=2, anchor=center, inner sep=0}, equals, from=2-1, to=3-1]
	\arrow[""{name=3, anchor=center, inner sep=0}, "{{a_2}}", from=2-3, to=3-3]
	\arrow["a"', from=3-1, to=3-3]
	\arrow["\SIGMA"{description}, draw=none, from=0, to=1]
	\arrow["\SIGMA"{description}, draw=none, from=2, to=3]
\end{tikzcd}
\  \text{and} \
\begin{tikzcd}
	{} & {} \\
	{} & {} & {} \\
	A & B
	\arrow["{r_1}", from=1-1, to=1-2]
	\arrow[""{name=0, anchor=center, inner sep=0}, "{k_1}"', from=1-1, to=2-1]
	\arrow[""{name=1, anchor=center, inner sep=0}, "{k_2}", from=1-2, to=2-2]
	\arrow["{k'_2}", curve={height=-6pt}, from=1-2, to=2-3]
	\arrow["{r_2}"', from=2-1, to=2-2]
	\arrow[""{name=2, anchor=center, inner sep=0}, "{a_1}"', from=2-1, to=3-1]
	\arrow["{\theta'}", between={0.1}{0.9}, Rightarrow, from=2-2, to=2-3]
	\arrow[""{name=3, anchor=center, inner sep=0}, "{b_1}", from=2-2, to=3-2]
	\arrow["{b_2}", curve={height=-6pt}, from=2-3, to=3-2]
	\arrow["b"', from=3-1, to=3-2]
	\arrow["\SIGMA"{description}, draw=none, from=0, to=1]
	\arrow["\SIGMA"{description}, draw=none, from=2, to=3]
\end{tikzcd}
 =
\begin{tikzcd}
	& {} & {} \\
	{} & {} & {} \\
	& A & B
	\arrow["{{r_1}}", from=1-2, to=1-3]
	\arrow["{{k_1}}"', curve={height=6pt}, from=1-2, to=2-1]
	\arrow[""{name=0, anchor=center, inner sep=0}, "{{k'_1}}"', from=1-2, to=2-2]
	\arrow[""{name=1, anchor=center, inner sep=0}, "{{k'_2}}", from=1-3, to=2-3]
	\arrow["\theta", Rightarrow, from=2-1, to=2-2]
	\arrow["{{a_1}}"', curve={height=6pt}, from=2-1, to=3-2]
	\arrow["{{r'_2}}"{description}, from=2-2, to=2-3]
	\arrow[""{name=2, anchor=center, inner sep=0}, "{{a_2}}"', from=2-2, to=3-2]
	\arrow[""{name=3, anchor=center, inner sep=0}, "{{b_2}}", from=2-3, to=3-3]
	\arrow["b"', from=3-2, to=3-3]
	\arrow["\SIGMA"{description}, draw=none, from=0, to=1]
	\arrow["\SIGMA"{description}, draw=none, from=2, to=3]
\end{tikzcd}
\, .\end{equation}
Thus, as in Lemma \ref{lem:use}, we have
\begin{equation}\label{eq:eq-A1a}
\begin{tikzcd}
	{} & {} & {} & {} \\
	{} && {} & {} & {} \\
	{} && {} & {}
	\arrow["t", from=1-1, to=1-2]
	\arrow[""{name=0, anchor=center, inner sep=0}, "k"', from=1-1, to=2-1]
	\arrow["{{{{s_1}}}}", from=1-2, to=1-3]
	\arrow["{{{{r_1}}}}", from=1-3, to=1-4]
	\arrow[""{name=1, anchor=center, inner sep=0}, "{{{{k_1}}}}"', from=1-3, to=2-3]
	\arrow[""{name=2, anchor=center, inner sep=0}, "{{{{k_2}}}}", from=1-4, to=2-4]
	\arrow["{{k'_2}}", curve={height=-6pt}, from=1-4, to=2-5]
	\arrow["{{{{s_2}}}}"', from=2-1, to=2-3]
	\arrow[""{name=3, anchor=center, inner sep=0}, equals, from=2-1, to=3-1]
	\arrow["{{{{r_2}}}}"', from=2-3, to=2-4]
	\arrow[""{name=4, anchor=center, inner sep=0}, "{a_1}"', from=2-3, to=3-3]
	\arrow["{{\theta'}}", between={0.1}{0.9}, Rightarrow, from=2-4, to=2-5]
	\arrow[""{name=5, anchor=center, inner sep=0}, "{{b_1}}", from=2-4, to=3-4]
	\arrow["{b_2}", curve={height=-6pt}, from=2-5, to=3-4]
	\arrow["a"', from=3-1, to=3-3]
	\arrow["b"', from=3-3, to=3-4]
	\arrow["\SIGMA"{description}, draw=none, from=0, to=1]
	\arrow["\SIGMA"{description}, draw=none, from=1, to=2]
	\arrow["\SIGMA"{description}, draw=none, from=3, to=4]
	\arrow["\SIGMA"{description}, draw=none, from=4, to=5]
\end{tikzcd}
\quad = \quad
\begin{tikzcd}
	{} & {} & {} & {} \\
	{} && {} & {} \\
	{} && {} & {}
	\arrow["t", from=1-1, to=1-2]
	\arrow[""{name=0, anchor=center, inner sep=0}, "k"', from=1-1, to=2-1]
	\arrow["{{{{s_1}}}}", from=1-2, to=1-3]
	\arrow["{{{{r_1}}}}", from=1-3, to=1-4]
	\arrow[""{name=1, anchor=center, inner sep=0}, "{k'_1}"', from=1-3, to=2-3]
	\arrow[""{name=2, anchor=center, inner sep=0}, "{{k'_2}}", from=1-4, to=2-4]
	\arrow["{{s'_2}}"{description}, from=2-1, to=2-3]
	\arrow[""{name=3, anchor=center, inner sep=0}, equals, from=2-1, to=3-1]
	\arrow["{{r'_2}}"{description}, from=2-3, to=2-4]
	\arrow[""{name=4, anchor=center, inner sep=0}, "{a_2}", from=2-3, to=3-3]
	\arrow[""{name=5, anchor=center, inner sep=0}, "{{b_2}}", from=2-4, to=3-4]
	\arrow["a"', from=3-1, to=3-3]
	\arrow["b"', from=3-3, to=3-4]
	\arrow["\SIGMA"{description}, draw=none, from=0, to=1]
	\arrow["\SIGMA"{description}, draw=none, from=1, to=2]
	\arrow["\SIGMA"{description}, draw=none, from=3, to=4]
	\arrow["\SIGMA"{description}, draw=none, from=4, to=5]
\end{tikzcd}
\, .\end{equation}
Moreover, by Rule 4', we have $\Sigma$-squares and an invertible 2-cell $\phi$ such that
\begin{equation}\label{eq:eq-A2}
\begin{tikzcd}
	{} & {} \\
	{} && {} \\
	{} & {} \\
	{} & {} & {} \\
	{} & C
	\arrow["r", from=1-1, to=1-2]
	\arrow["g"', from=1-1, to=2-1]
	\arrow["{{{h_3}}}", from=1-2, to=2-3]
	\arrow[""{name=0, anchor=center, inner sep=0}, "{{{h_2}}}", from=1-2, to=3-2]
	\arrow["{{h_1}}"', from=2-1, to=3-1]
	\arrow[""{name=1, anchor=center, inner sep=0}, "{{{k'_3}}}", from=2-3, to=4-3]
	\arrow["{{{r_1}}}", from=3-1, to=3-2]
	\arrow[""{name=2, anchor=center, inner sep=0}, "{{{k'_1}}}"', from=3-1, to=4-1]
	\arrow[""{name=3, anchor=center, inner sep=0}, "{{{k'_2}}}", from=3-2, to=4-2]
	\arrow["{{{r'_2}}}"{description}, from=4-1, to=4-2]
	\arrow[""{name=4, anchor=center, inner sep=0}, equals, from=4-1, to=5-1]
	\arrow[""{name=5, anchor=center, inner sep=0}, "{{{c_1}}}", from=4-2, to=5-2]
	\arrow["{{{c_2}}}", from=4-3, to=5-2]
	\arrow["c"', from=5-1, to=5-2]
	\arrow["\SIGMA"{description}, draw=none, from=2-1, to=0]
	\arrow["\SIGMA"{description}, draw=none, from=2, to=3]
	\arrow["\phi"{pos=0.35}, between={0.1}{0.8}, Rightarrow, from=3-2, to=1]
	\arrow["\SIGMA"{description}, draw=none, from=4, to=5]
\end{tikzcd}
\quad = \quad
\begin{tikzcd}
	{} & {} \\
	{} \\
	{} & {} \\
	{} & {} \\
	{} & C
	\arrow["r", from=1-1, to=1-2]
	\arrow["g"', from=1-1, to=2-1]
	\arrow[""{name=0, anchor=center, inner sep=0}, "{{h_3}}", from=1-2, to=3-2]
	\arrow["{h_1}"', from=2-1, to=3-1]
	\arrow["{r_3}", from=3-1, to=3-2]
	\arrow[""{name=1, anchor=center, inner sep=0}, "{{{k'_1}}}"', from=3-1, to=4-1]
	\arrow[""{name=2, anchor=center, inner sep=0}, "{{k'_3}}", from=3-2, to=4-2]
	\arrow["{{r'_3}}"', from=4-1, to=4-2]
	\arrow[""{name=3, anchor=center, inner sep=0}, equals, from=4-1, to=5-1]
	\arrow[""{name=4, anchor=center, inner sep=0}, "{{c_2}}", from=4-2, to=5-2]
	\arrow["c"', from=5-1, to=5-2]
	\arrow["\SIGMA"{description}, draw=none, from=2-1, to=0]
	\arrow["\SIGMA"{description}, draw=none, from=1, to=2]
	\arrow["\SIGMA"{description}, draw=none, from=3, to=4]
\end{tikzcd}
\, .\end{equation}
Hence, \eqref{eq:eq-A1a} and \eqref{eq:eq-A2} give rise to the following vertical juxtaposition of the two basic $\Omega$ 2-morphisms corresponding to (P1):
\begin{equation}\label{eq:eq[1]}
 \begin{tikzcd}
	{} & {} & {} & {} & {} \\
	&& B & A & {} \\
	{} & {} & {} & {} & {} \\
	&& C & {} & {} \\
	{} & {} & {} & {} & {}
	\arrow["{{{h_2}}}", from=1-1, to=1-2]
	\arrow[""{name=0, anchor=center, inner sep=0}, equals, from=1-1, to=3-1]
	\arrow["{{{k_2}}}", from=1-2, to=1-3]
	\arrow[""{name=1, anchor=center, inner sep=0}, equals, from=1-2, to=3-2]
	\arrow[""{name=2, anchor=center, inner sep=0}, "{{{b_1}}}"', from=1-3, to=2-3]
	\arrow["{{{\theta'}}}"', between={0.3}{0.7}, Rightarrow, from=1-3, to=3-2]
	\arrow["{{{r_2}}}"', from=1-4, to=1-3]
	\arrow[""{name=3, anchor=center, inner sep=0}, "{{{a_1}}}", from=1-4, to=2-4]
	\arrow["{{{s_2}}}"', from=1-5, to=1-4]
	\arrow[""{name=4, anchor=center, inner sep=0}, equals, from=1-5, to=2-5]
	\arrow["b", from=2-4, to=2-3]
	\arrow["a", from=2-5, to=2-4]
	\arrow[""{name=5, anchor=center, inner sep=0}, equals, from=2-5, to=3-5]
	\arrow["{{{h_2}}}"', from=3-1, to=3-2]
	\arrow[equals, from=3-1, to=5-1]
	\arrow[""{name=6, anchor=center, inner sep=0}, "{{{k'_2}}}"', from=3-2, to=3-3]
	\arrow[""{name=7, anchor=center, inner sep=0}, "{{{b_2}}}", from=3-3, to=2-3]
	\arrow[""{name=8, anchor=center, inner sep=0}, "{{{c_1}}}"', from=3-3, to=4-3]
	\arrow[""{name=9, anchor=center, inner sep=0}, "{{{a_2}}}"', from=3-4, to=2-4]
	\arrow["{{r'_2}}"{description}, from=3-4, to=3-3]
	\arrow[""{name=10, anchor=center, inner sep=0}, equals, from=3-4, to=4-4]
	\arrow["{{{s'_2}}}", from=3-5, to=3-4]
	\arrow[""{name=11, anchor=center, inner sep=0}, equals, from=3-5, to=4-5]
	\arrow["c", from=4-4, to=4-3]
	\arrow[""{name=12, anchor=center, inner sep=0}, equals, from=4-4, to=5-4]
	\arrow["{{s'_2}}"{description}, from=4-5, to=4-4]
	\arrow[""{name=13, anchor=center, inner sep=0}, equals, from=4-5, to=5-5]
	\arrow[""{name=14, anchor=center, inner sep=0}, "{{{h_3}}}"', from=5-1, to=5-2]
	\arrow["{{{k'_3}}}"', from=5-2, to=5-3]
	\arrow[""{name=15, anchor=center, inner sep=0}, "{{{c_2}}}", from=5-3, to=4-3]
	\arrow["{{{r'_3}}}", from=5-4, to=5-3]
	\arrow["{{{s'_2}}}", from=5-5, to=5-4]
	\arrow[between={0.4}{0.6}, equals, from=0, to=1]
	\arrow["\SIGMA"{description}, draw=none, from=2, to=3]
	\arrow["\SIGMA"{description}, draw=none, from=3, to=4]
	\arrow["\phi"', between={0.3}{0.7}, Rightarrow, from=6, to=14]
	\arrow["\SIGMA"{description}, draw=none, from=7, to=9]
	\arrow["\SIGMA"{description}, draw=none, from=8, to=10]
	\arrow["\SIGMA"{description}, draw=none, from=9, to=5]
	\arrow["\SIGMA"{description}, draw=none, from=10, to=11]
	\arrow["\SIGMA"{description}, draw=none, from=12, to=13]
	\arrow["\SIGMA"{description}, draw=none, from=15, to=12]
\end{tikzcd}
\; .\end{equation}
We want to obtain the vertical composition of these two 2-morphisms. The approach we now take will also be useful later. First, we use Square on $\begin{tikzcd}
	A & \bullet & C
	\arrow["{a_2}"', from=1-2, to=1-1]
	\arrow["c", from=1-2, to=1-3]
\end{tikzcd}$ and compose the obtained $\Sigma$-square with the top $\Sigma$-square of the right-hand of equality \eqref{eq:eq-A3} below. Then we apply Rule 4' to the top $\Sigma$-square of the left-hand side of the equality and this composition. This leads to $\Sg$-squares and an invertible 2-cell $\psi$ such that
\begin{equation}\label{eq:eq-A3}
 \begin{tikzcd}
	{} & {} \\
	A & B & {} \\
	A & {\hat{C}}
	\arrow["{{r'_2}}", from=1-1, to=1-2]
	\arrow[""{name=0, anchor=center, inner sep=0}, "{{a_2}}"', from=1-1, to=2-1]
	\arrow[""{name=1, anchor=center, inner sep=0}, "{{b_2}}", from=1-2, to=2-2]
	\arrow["{{c_1}}", curve={height=-12pt}, from=1-2, to=2-3]
	\arrow["b"', from=2-1, to=2-2]
	\arrow[""{name=2, anchor=center, inner sep=0}, equals, from=2-1, to=3-1]
	\arrow["\psi", Rightarrow, from=2-2, to=2-3]
	\arrow[""{name=3, anchor=center, inner sep=0}, "{{\hat{a}_1}}", from=2-2, to=3-2]
	\arrow["{\hat{a}_2}", curve={height=-12pt}, from=2-3, to=3-2]
	\arrow["{{\hat{c}}}"', from=3-1, to=3-2]
	\arrow["\SIGMA"{description}, draw=none, from=0, to=1]
	\arrow["\SIGMA"{description}, draw=none, from=2, to=3]
\end{tikzcd}
\quad = \quad
\adjustbox{scale=0.8}{
\begin{tikzcd}
	{} & {} \\
	{} & C \\
	A & {\hat{C}} \\
	A & {\hat{C}}
	\arrow["{{r'_2}}", from=1-1, to=1-2]
	\arrow[""{name=0, anchor=center, inner sep=0}, equals, from=1-1, to=2-1]
	\arrow[""{name=1, anchor=center, inner sep=0}, "{{c_1}}", from=1-2, to=2-2]
	\arrow["c"', from=2-1, to=2-2]
	\arrow[""{name=2, anchor=center, inner sep=0}, "{{a_2}}"', from=2-1, to=3-1]
	\arrow[""{name=3, anchor=center, inner sep=0}, from=2-2, to=3-2]
	\arrow[""{name=4, anchor=center, inner sep=0}, "{{\hat{a}_2}}", curve={height=-18pt}, from=2-2, to=4-2]
	\arrow[from=3-1, to=3-2]
	\arrow[""{name=5, anchor=center, inner sep=0}, equals, from=3-1, to=4-1]
	\arrow[""{name=6, anchor=center, inner sep=0}, from=3-2, to=4-2]
	\arrow["{\hat{c}}"', from=4-1, to=4-2]
	\arrow["\SIGMA"{description}, draw=none, from=0, to=1]
	\arrow["\SIGMAc"{description}, draw=none, from=2, to=3]
	\arrow["\SIGMA"{description}, draw=none, from=5, to=6]
	\arrow["{=}"{description, pos=0.3}, draw=none, from=3-2, to=4]
\end{tikzcd}}
\, .
\end{equation}
By composing both sides of this equation on the left with
\[\begin{tikzcd}
	{} & {} \\
	{} & {} \\
	{} & {}
	\arrow["{{s'_2}}", from=1-1, to=1-2]
	\arrow[""{name=0, anchor=center, inner sep=0}, equals, from=1-1, to=2-1]
	\arrow[""{name=1, anchor=center, inner sep=0}, "{{a_2}}", from=1-2, to=2-2]
	\arrow["a"{description}, from=2-1, to=2-2]
	\arrow[""{name=2, anchor=center, inner sep=0}, equals, from=2-1, to=3-1]
	\arrow[""{name=3, anchor=center, inner sep=0}, equals, from=2-2, to=3-2]
	\arrow["a"', from=3-1, to=3-2]
	\arrow["\SIGMA"{description}, shift left, draw=none, from=0, to=1]
	\arrow["\SIGMA"{description}, draw=none, from=2, to=3]
\end{tikzcd}\qquad \text{and}\qquad \,
\begin{tikzcd}
	{} & {} \\
	{} & {} \\
	{} & {}
	\arrow["{{{{{s'_2}}}}}", from=1-1, to=1-2]
	\arrow[""{name=0, anchor=center, inner sep=0}, equals, from=1-1, to=2-1]
	\arrow[""{name=1, anchor=center, inner sep=0}, equals, from=1-2, to=2-2]
	\arrow["{{{s'_2}}}"{description}, from=2-1, to=2-2]
	\arrow[""{name=2, anchor=center, inner sep=0}, equals, from=2-1, to=3-1]
	\arrow[""{name=3, anchor=center, inner sep=0}, "{{{a_2}}}", from=2-2, to=3-2]
	\arrow["a"', from=3-1, to=3-2]
	\arrow["\SIGMA"{description}, shift left, draw=none, from=0, to=1]
	\arrow["\SIGMA"{description}, draw=none, from=2, to=3]
\end{tikzcd}\, ,
\]
respectively, we obtain an equality which allows us to represent the vertical composition of the 2-morphisms of \eqref{eq:eq[1]}  by the 2-morphism
\begin{equation}\label{eq:eq[[1]]}
 \begin{tikzcd}
	{} & {} && {} & {} & {} \\
	&&& B & A & {} \\
	{} & {} & {} & {\hat{C}} & {} & {} \\
	&&& C & {} & {} \\
	{} & {} && {} & {} & {}
	\arrow["{{{h_2}}}", from=1-1, to=1-2]
	\arrow[equals, from=1-1, to=3-1]
	\arrow[""{name=0, anchor=center, inner sep=0}, "{{{k_2}}}", from=1-2, to=1-4]
	\arrow[equals, from=1-2, to=3-2]
	\arrow[""{name=1, anchor=center, inner sep=0}, "{{{{b_1}}}}"', from=1-4, to=2-4]
	\arrow["{{{{r_2}}}}"', from=1-5, to=1-4]
	\arrow[""{name=2, anchor=center, inner sep=0}, "{{{{a_1}}}}", from=1-5, to=2-5]
	\arrow["{{{{s_2}}}}"', from=1-6, to=1-5]
	\arrow[""{name=3, anchor=center, inner sep=0}, equals, from=1-6, to=2-6]
	\arrow[""{name=4, anchor=center, inner sep=0}, "{{{\hat{a}_1}}}"', from=2-4, to=3-4]
	\arrow["b", from=2-5, to=2-4]
	\arrow["a", from=2-6, to=2-5]
	\arrow[""{name=5, anchor=center, inner sep=0}, equals, from=2-6, to=3-6]
	\arrow["{{{h_2}}}", from=3-1, to=3-2]
	\arrow[equals, from=3-1, to=5-1]
	\arrow["{{{k'_2}}}", from=3-2, to=3-3]
	\arrow[""{name=6, anchor=center, inner sep=0}, "{{{b_2}}}", from=3-3, to=2-4]
	\arrow[""{name=7, anchor=center, inner sep=0}, "{{{c_1}}}"', from=3-3, to=4-4]
	\arrow[""{name=8, anchor=center, inner sep=0}, equals, from=3-5, to=2-5]
	\arrow["{{{\hat{c}}}}", from=3-5, to=3-4]
	\arrow["a", from=3-6, to=3-5]
	\arrow[""{name=9, anchor=center, inner sep=0}, equals, from=3-6, to=4-6]
	\arrow[""{name=10, anchor=center, inner sep=0}, "{{{\hat{a}_2}}}", from=4-4, to=3-4]
	\arrow[""{name=11, anchor=center, inner sep=0}, "{{{a_2}}}"', from=4-5, to=3-5]
	\arrow["c", from=4-5, to=4-4]
	\arrow[""{name=12, anchor=center, inner sep=0}, equals, from=4-5, to=5-5]
	\arrow["{s'_2}"{description}, from=4-6, to=4-5]
	\arrow[""{name=13, anchor=center, inner sep=0}, equals, from=4-6, to=5-6]
	\arrow["{{{h_3}}}"', from=5-1, to=5-2]
	\arrow[""{name=14, anchor=center, inner sep=0}, "{{{k'_3}}}"', from=5-2, to=5-4]
	\arrow[""{name=15, anchor=center, inner sep=0}, "{{{{c_2}}}}", from=5-4, to=4-4]
	\arrow["{{{{r'_3}}}}", from=5-5, to=5-4]
	\arrow["{{{{s'_2}}}}", from=5-6, to=5-5]
	\arrow["{{{\theta'}}}"', between={0.2}{0.7}, Rightarrow, from=0, to=3-3]
	\arrow["\SIGMA"{description}, draw=none, from=1, to=2]
	\arrow["\SIGMA"{description}, draw=none, from=2, to=3]
	\arrow["\SIGMA"{description}, draw=none, from=4, to=8]
	\arrow["\phi"', between={0.1}{0.8}, Rightarrow, from=3-2, to=14]
	\arrow["\psi", between={0.2}{0.8}, Rightarrow, from=6, to=7]
	\arrow["\SIGMA"{description}, draw=none, from=8, to=5]
	\arrow["\SIGMA"{description}, draw=none, from=10, to=11]
	\arrow["\SIGMA"{description}, draw=none, from=11, to=9]
	\arrow["\SIGMA"{description}, draw=none, from=12, to=13]
	\arrow["\SIGMA"{description}, draw=none, from=15, to=12]
\end{tikzcd}
\end{equation}

\noindent \textbf{B.} In order to form the vertical composition corresponding to (P2), observe that, for the first $\Sg$-step, we have a basic $\Omega$ 2-cell determined by the following data:
\begin{equation}\label{eq:BB}
\begin{tikzcd}
	{} & {} \\
	{} & {} \\
	{} & {} & {} \\
	{} & {} & {} \\
	{} & D
	\arrow["r", from=1-1, to=1-2]
	\arrow["g"', from=1-1, to=2-1]
	\arrow["{{h_2}}", from=1-2, to=3-2]
	\arrow["{{h_3}}", from=1-2, to=3-3]
	\arrow["\SIGMA"{description}, draw=none, from=2-1, to=2-2]
	\arrow["{{h_1}}"', from=2-1, to=3-1]
	\arrow["{{r_1}}", from=3-1, to=3-2]
	\arrow[""{name=0, anchor=center, inner sep=0}, "{{k_1}}"', from=3-1, to=4-1]
	\arrow["\alpha", between={0.1}{0.9}, Rightarrow, from=3-2, to=3-3]
	\arrow[""{name=1, anchor=center, inner sep=0}, "{{k_2}}", from=3-2, to=4-2]
	\arrow["{{k_3}}", from=3-3, to=4-3]
	\arrow["{{r_2}}"', from=4-1, to=4-2]
	\arrow[""{name=2, anchor=center, inner sep=0}, equals, from=4-1, to=5-1]
	\arrow[""{name=3, anchor=center, inner sep=0}, "{{d_1}}", from=4-2, to=5-2]
	\arrow["{{d_2}}", from=4-3, to=5-2]
	\arrow["d"', from=5-1, to=5-2]
	\arrow["\SIGMA"{description}, draw=none, from=0, to=1]
	\arrow["\SIGMA"{description}, draw=none, from=2, to=3]
\end{tikzcd}
\quad = \quad
\begin{tikzcd}
	{} & {} \\
	{} & {} \\
	{} & {} \\
	{} & {} \\
	{} & D
	\arrow["r", from=1-1, to=1-2]
	\arrow["g"', from=1-1, to=2-1]
	\arrow["{h_3}", from=1-2, to=3-2]
	\arrow["\SIGMA"{description}, draw=none, from=2-1, to=2-2]
	\arrow["{{h_1}}"', from=2-1, to=3-1]
	\arrow["{r_3}", from=3-1, to=3-2]
	\arrow[""{name=0, anchor=center, inner sep=0}, "{{k_1}}"', from=3-1, to=4-1]
	\arrow[""{name=1, anchor=center, inner sep=0}, "{k_3}", from=3-2, to=4-2]
	\arrow["{r_4}"', from=4-1, to=4-2]
	\arrow[""{name=2, anchor=center, inner sep=0}, equals, from=4-1, to=5-1]
	\arrow[""{name=3, anchor=center, inner sep=0}, "{d_2}", from=4-2, to=5-2]
	\arrow["d"', from=5-1, to=5-2]
	\arrow["\SIGMA"{description}, draw=none, from=0, to=1]
	\arrow["\SIGMA"{description}, draw=none, from=2, to=3]
\end{tikzcd}
\, .
\end{equation}
For the second $\Sg$-step, use Lemma \ref{lem:use} to first obtain $\theta$ exactly as in the first equality in \eqref{eq:eq-A1}, and then the equality

\begin{equation}\label{eq:eq-B2}
\begin{tikzcd}
	& {} && {} \\
	{} && {} & {} \\
	& A && B
	\arrow["{{{r_3}}}", from=1-2, to=1-4]
	\arrow["{{{k_1}}}"', curve={height=6pt}, from=1-2, to=2-1]
	\arrow[""{name=0, anchor=center, inner sep=0}, "{{{k'_1}}}"{description}, curve={height=-6pt}, from=1-2, to=2-3]
	\arrow[""{name=1, anchor=center, inner sep=0}, "{{{k'_3}}}", from=1-4, to=2-4]
	\arrow["\theta", between={0.2}{0.8}, Rightarrow, from=2-1, to=2-3]
	\arrow["{{{a_1}}}"', curve={height=6pt}, from=2-1, to=3-2]
	\arrow["{{{r'_3}}}"', from=2-3, to=2-4]
	\arrow[""{name=2, anchor=center, inner sep=0}, "{a_2}"{description}, curve={height=-6pt}, from=2-3, to=3-2]
	\arrow[""{name=3, anchor=center, inner sep=0}, "{{{b_4}}}", from=2-4, to=3-4]
	\arrow["b"', from=3-2, to=3-4]
	\arrow["\SIGMA"{description}, draw=none, from=0, to=1]
	\arrow["\SIGMA"{description}, draw=none, from=2, to=3]
\end{tikzcd}
\quad = \quad
\begin{tikzcd}
	{} && {} \\
	{} & {} && {} \\
	A && B
	\arrow["{{{r_3}}}", from=1-1, to=1-3]
	\arrow[""{name=0, anchor=center, inner sep=0}, "{{k_1}}"', from=1-1, to=2-1]
	\arrow[""{name=1, anchor=center, inner sep=0}, "{{k_3}}", curve={height=6pt}, from=1-3, to=2-2]
	\arrow["{{k'_3}}", curve={height=-6pt}, from=1-3, to=2-4]
	\arrow["{{r_4}}"', from=2-1, to=2-2]
	\arrow[""{name=2, anchor=center, inner sep=0}, "{{a_1}}"', from=2-1, to=3-1]
	\arrow["{{\theta''}}", shorten <=14pt, shorten >=14pt, Rightarrow, from=2-2, to=2-4]
	\arrow[""{name=3, anchor=center, inner sep=0}, "{{b_3}}", curve={height=6pt}, from=2-2, to=3-3]
	\arrow["{{b_4}}", curve={height=-6pt}, from=2-4, to=3-3]
	\arrow["b"', from=3-1, to=3-3]
	\arrow["\SIGMA"{description}, draw=none, from=0, to=1]
	\arrow["\SIGMA"{description}, draw=none, from=2, to=3]
\end{tikzcd}
\, .
\end{equation}
Observe that, without loss of generality, we may indeed use the same morphism $A\xrightarrow{b}B$ in \eqref{eq:eq-B2} as in \eqref{eq:eq-A1a}: if in \eqref{eq:eq-B2} we have $A\xrightarrow{b'}B'$, instead of $A\xrightarrow{b}B$, just apply Square of $b$ and $b'$, giving
$\adjustbox{scale=0.70}{\begin{tikzcd}
	{} & {} \\
	{} & {}
	\arrow["b", from=1-1, to=1-2]
	\arrow[""{name=0, anchor=center, inner sep=0}, "{b'}"', from=1-1, to=2-1]
	\arrow[""{name=1, anchor=center, inner sep=0}, "{n'}", from=1-2, to=2-2]
	\arrow["n"', from=2-1, to=2-2]
	\arrow["\SIGMA"{description}, draw=none, from=0, to=1]
\end{tikzcd}}$,
then apply Rule 2b of Proposition \ref{pro:useful_rules} and consider the new $b$ to be $nb'$.

Thus, as described in Lemma \ref{lem:use}, we have that
\[
\begin{tikzcd}
	{} & {} & {} & {} \\
	{} && {} & {} & {} \\
	{} && A & B
	\arrow["t", from=1-1, to=1-2]
	\arrow[""{name=0, anchor=center, inner sep=0}, "k"', from=1-1, to=2-1]
	\arrow["{{s_1}}", from=1-2, to=1-3]
	\arrow["{{r_3}}", from=1-3, to=1-4]
	\arrow[""{name=1, anchor=center, inner sep=0}, "{{k_1}}"', from=1-3, to=2-3]
	\arrow[""{name=2, anchor=center, inner sep=0}, "{{k_3}}", from=1-4, to=2-4]
	\arrow["{{k'_3}}", curve={height=-6pt}, from=1-4, to=2-5]
	\arrow["{{s_2}}"', from=2-1, to=2-3]
	\arrow[""{name=3, anchor=center, inner sep=0}, equals, from=2-1, to=3-1]
	\arrow["{{r_4}}"', from=2-3, to=2-4]
	\arrow[""{name=4, anchor=center, inner sep=0}, "{{a_1}}"', from=2-3, to=3-3]
	\arrow["{{\theta''}}", Rightarrow, from=2-4, to=2-5]
	\arrow[""{name=5, anchor=center, inner sep=0}, "{{b_3}}", from=2-4, to=3-4]
	\arrow["{{b_4}}", curve={height=-6pt}, from=2-5, to=3-4]
	\arrow["a"', from=3-1, to=3-3]
	\arrow["b"', from=3-3, to=3-4]
	\arrow["\SIGMA"{description}, draw=none, from=0, to=1]
	\arrow["\SIGMA"{description}, draw=none, from=1, to=2]
	\arrow["\SIGMA"{description}, draw=none, from=3, to=4]
	\arrow["\SIGMA"{description}, draw=none, from=4, to=5]
\end{tikzcd}
\quad = \quad
\begin{tikzcd}
	{} & {} & {} & {} \\
	{} && {} & {} \\
	{} && A & B
	\arrow["t", from=1-1, to=1-2]
	\arrow[""{name=0, anchor=center, inner sep=0}, "k"', from=1-1, to=2-1]
	\arrow["{{{{s_1}}}}", from=1-2, to=1-3]
	\arrow["{{{{r_3}}}}", from=1-3, to=1-4]
	\arrow[""{name=1, anchor=center, inner sep=0}, "{{k'_1}}"', from=1-3, to=2-3]
	\arrow[""{name=2, anchor=center, inner sep=0}, "{{k'_3}}", from=1-4, to=2-4]
	\arrow["{{s'_2}}"', from=2-1, to=2-3]
	\arrow[""{name=3, anchor=center, inner sep=0}, equals, from=2-1, to=3-1]
	\arrow["{{r'_3}}"{description}, from=2-3, to=2-4]
	\arrow[""{name=4, anchor=center, inner sep=0}, "{{a_2}}"', from=2-3, to=3-3]
	\arrow[""{name=5, anchor=center, inner sep=0}, "{{b_4}}", from=2-4, to=3-4]
	\arrow["a"', from=3-1, to=3-3]
	\arrow["b"', from=3-3, to=3-4]
	\arrow["\SIGMA"{description}, draw=none, from=0, to=1]
	\arrow["\SIGMA"{description}, draw=none, from=1, to=2]
	\arrow["\SIGMA"{description}, draw=none, from=3, to=4]
	\arrow["\SIGMA"{description}, draw=none, from=4, to=5]
\end{tikzcd}
\, . \]
Combining the two $\Sg$-steps of (P2), we obtain:
\begin{equation}\label{eq:eq[2]}
  \begin{tikzcd}
	{} & {} & {} & {} & {} \\
	&& D & {} & {} \\
	{} & {} & {} & {} & {} \\
	&& B & A & {} \\
	{} & {} & {} & {} & {}
	\arrow["{h_2}", from=1-1, to=1-2]
	\arrow[equals, from=1-1, to=3-1]
	\arrow["{k_2}", from=1-2, to=1-3]
	\arrow[""{name=0, anchor=center, inner sep=0}, "{d_1}"', from=1-3, to=2-3]
	\arrow["\alpha"', shorten <=17pt, shorten >=17pt, Rightarrow, from=1-3, to=3-1]
	\arrow["{r_2}"', from=1-4, to=1-3]
	\arrow[""{name=1, anchor=center, inner sep=0}, equals, from=1-4, to=2-4]
	\arrow["{s_2}"', from=1-5, to=1-4]
	\arrow[""{name=2, anchor=center, inner sep=0}, equals, from=1-5, to=2-5]
	\arrow["d", from=2-4, to=2-3]
	\arrow[""{name=3, anchor=center, inner sep=0}, equals, from=2-4, to=3-4]
	\arrow["{s_2}", from=2-5, to=2-4]
	\arrow[""{name=4, anchor=center, inner sep=0}, equals, from=2-5, to=3-5]
	\arrow["{h_3}", from=3-1, to=3-2]
	\arrow[""{name=5, anchor=center, inner sep=0}, equals, from=3-1, to=5-1]
	\arrow["{k_3}", from=3-2, to=3-3]
	\arrow[""{name=6, anchor=center, inner sep=0}, equals, from=3-2, to=5-2]
	\arrow[""{name=7, anchor=center, inner sep=0}, "{d_2}", from=3-3, to=2-3]
	\arrow[""{name=8, anchor=center, inner sep=0}, "{b_3}"', from=3-3, to=4-3]
	\arrow["{\theta''}"', shorten <=11pt, shorten >=8pt, Rightarrow, from=3-3, to=5-2]
	\arrow["{r_4}", from=3-4, to=3-3]
	\arrow[""{name=9, anchor=center, inner sep=0}, "{a_1}", from=3-4, to=4-4]
	\arrow["{s_2}", from=3-5, to=3-4]
	\arrow[""{name=10, anchor=center, inner sep=0}, equals, from=3-5, to=4-5]
	\arrow["b", from=4-4, to=4-3]
	\arrow["a", from=4-5, to=4-4]
	\arrow[""{name=11, anchor=center, inner sep=0}, equals, from=4-5, to=5-5]
	\arrow["{h_3}"', from=5-1, to=5-2]
	\arrow["{k'_3}"', from=5-2, to=5-3]
	\arrow[""{name=12, anchor=center, inner sep=0}, "{b_4}", from=5-3, to=4-3]
	\arrow[""{name=13, anchor=center, inner sep=0}, "{a_2}"', from=5-4, to=4-4]
	\arrow["{r'_3}", from=5-4, to=5-3]
	\arrow["{s'_2}", from=5-5, to=5-4]
	\arrow["\SIGMA"{description}, draw=none, from=0, to=1]
	\arrow["\SIGMA"{description}, draw=none, from=1, to=2]
	\arrow["\SIGMA"{description}, draw=none, from=3, to=4]
	\arrow[shorten <=16pt, shorten >=10pt, equals, from=5, to=6]
	\arrow["\SIGMA"{description}, draw=none, from=7, to=3]
	\arrow["\SIGMA"{description}, draw=none, from=8, to=9]
	\arrow["\SIGMA"{description}, draw=none, from=9, to=10]
	\arrow["\SIGMA"{description}, draw=none, from=12, to=13]
	\arrow["\SIGMA"{description}, draw=none, from=13, to=11]
\end{tikzcd}
\end{equation}
Now, we want to obtain the vertical composition of the two 2-morphisms of \eqref{eq:eq[2]}.
For that, first apply Square to $d$ and $a_1$, followed by  Rule 4'. We obtain the equality of the part drawn in solid lines in Diagram \eqref{eq:eq-das-sol} below.
\begin{equation}\label{eq:eq-das-sol}
\begin{tikzcd}
	{} & {} & {} \\
	{} & {} & D & {} \\
	{} & A & {} \\
	{} & A & {\tilde{B}}
	\arrow["{{{{s_2}}}}", dotted, from=1-1, to=1-2]
	\arrow[""{name=0, anchor=center, inner sep=0}, equals, dotted, from=1-1, to=2-1]
	\arrow["{{{{{r_4}}}}}", from=1-2, to=1-3]
	\arrow[""{name=1, anchor=center, inner sep=0}, equals, from=1-2, to=2-2]
	\arrow[""{name=2, anchor=center, inner sep=0}, "{{{{{d_2}}}}}", from=1-3, to=2-3]
	\arrow["{{{{{b_3}}}}}", curve={height=-12pt}, from=1-3, to=2-4]
	\arrow["{{{{s_2}}}}"', dotted, from=2-1, to=2-2]
	\arrow[""{name=3, anchor=center, inner sep=0}, equals, dotted, from=2-1, to=3-1]
	\arrow["d"', from=2-2, to=2-3]
	\arrow[""{name=4, anchor=center, inner sep=0}, "{{{{{a_1}}}}}"', from=2-2, to=3-2]
	\arrow["\delta", shift right, Rightarrow, from=2-3, to=2-4]
	\arrow[""{name=5, anchor=center, inner sep=0}, from=2-3, to=3-3]
	\arrow[""{name=6, anchor=center, inner sep=0}, "{{{{{\tilde{b}_1}}}}}", curve={height=-18pt}, from=2-3, to=4-3]
	\arrow["{{{{{\tilde{b}_2}}}}}", curve={height=-18pt}, from=2-4, to=4-3]
	\arrow["a"', dotted, from=3-1, to=3-2]
	\arrow[""{name=7, anchor=center, inner sep=0}, equals, dotted, from=3-1, to=4-1]
	\arrow[from=3-2, to=3-3]
	\arrow[""{name=8, anchor=center, inner sep=0}, equals, from=3-2, to=4-2]
	\arrow[""{name=9, anchor=center, inner sep=0}, from=3-3, to=4-3]
	\arrow["a"', dotted, from=4-1, to=4-2]
	\arrow["{{{{{\tilde{b}}}}}}"', from=4-2, to=4-3]
	\arrow["\SIGMA"{description}, draw=none, from=0, to=1]
	\arrow["\SIGMA"{description}, draw=none, from=1, to=2]
	\arrow["\SIGMA"{description}, draw=none, from=3, to=4]
	\arrow["\SIGMA"{description}, draw=none, from=4, to=5]
	\arrow["\SIGMA"{description}, draw=none, from=7, to=8]
	\arrow["\SIGMA"{description}, draw=none, from=8, to=9]
	\arrow["{=}"{description, pos=0.2}, draw=none, from=3-3, to=6]
\end{tikzcd}
\quad = \quad
\begin{tikzcd}
	{} & {} & {} \\
	{} & {} & {} \\
	{} & {} & {\tilde{B}}
	\arrow["{{{s_2}}}", dotted, from=1-1, to=1-2]
	\arrow[""{name=0, anchor=center, inner sep=0}, equals, dotted, from=1-1, to=2-1]
	\arrow["{{{{r_4}}}}", from=1-2, to=1-3]
	\arrow[""{name=1, anchor=center, inner sep=0}, "{a_1}"{description}, from=1-2, to=2-2]
	\arrow[""{name=2, anchor=center, inner sep=0}, "{{{{b_3}}}}", from=1-3, to=2-3]
	\arrow["a"', dotted, from=2-1, to=2-2]
	\arrow[""{name=3, anchor=center, inner sep=0}, equals, dotted, from=2-1, to=3-1]
	\arrow["b"', from=2-2, to=2-3]
	\arrow[""{name=4, anchor=center, inner sep=0}, equals, from=2-2, to=3-2]
	\arrow[""{name=5, anchor=center, inner sep=0}, "{{{{\tilde{b}_2}}}}", from=2-3, to=3-3]
	\arrow["a"', dotted, from=3-1, to=3-2]
	\arrow["{{{{{\tilde{b}}}}}}"', from=3-2, to=3-3]
	\arrow["\SIGMA"{description}, draw=none, from=0, to=1]
	\arrow["\SIGMA"{description}, draw=none, from=1, to=2]
	\arrow["\SIGMA"{description}, draw=none, from=3, to=4]
	\arrow["\SIGMA"{description}, draw=none, from=4, to=5]
\end{tikzcd}
\, .
\end{equation}
Composing with the dotted squares, the resulting equality may be used to obtain  the vertical composition of the two 2-morphisms of \eqref{eq:eq[2]}:
\begin{equation}\label{eq:eq[[2]]}
 \begin{tikzcd}
	{} && {} & {} & {} & {} \\
	&&& D & {} & {} \\
	{} & {} & {} & {\tilde{B}} & A & {} \\
	&&& B & A & {} \\
	{} & {} && {} & {} & {}
	\arrow["{{{{h_2}}}}", from=1-1, to=1-3]
	\arrow[equals, from=1-1, to=3-1]
	\arrow["{{{{{k_2}}}}}", from=1-3, to=1-4]
	\arrow["\alpha"', between={0.2}{0.7}, Rightarrow, from=1-3, to=3-2]
	\arrow[""{name=0, anchor=center, inner sep=0}, "{{{{{d_1}}}}}"', from=1-4, to=2-4]
	\arrow["{{{{{r_2}}}}}"', from=1-5, to=1-4]
	\arrow[""{name=1, anchor=center, inner sep=0}, equals, from=1-5, to=2-5]
	\arrow["{{{{{s_2}}}}}"', from=1-6, to=1-5]
	\arrow[""{name=2, anchor=center, inner sep=0}, equals, from=1-6, to=2-6]
	\arrow[""{name=3, anchor=center, inner sep=0}, "{{{\tilde{b}_1}}}"', from=2-4, to=3-4]
	\arrow["d", from=2-5, to=2-4]
	\arrow[""{name=4, anchor=center, inner sep=0}, "{{{{a_1}}}}", from=2-5, to=3-5]
	\arrow["{{{{{s_2}}}}}", from=2-6, to=2-5]
	\arrow[""{name=5, anchor=center, inner sep=0}, equals, from=2-6, to=3-6]
	\arrow["{{{{h_3}}}}", from=3-1, to=3-2]
	\arrow[""{name=6, anchor=center, inner sep=0}, equals, from=3-1, to=5-1]
	\arrow["{{{{k_3}}}}", from=3-2, to=3-3]
	\arrow[""{name=7, anchor=center, inner sep=0}, equals, from=3-2, to=5-2]
	\arrow["{{{d_2}}}"{description}, curve={height=-6pt}, from=3-3, to=2-4]
	\arrow[""{name=8, anchor=center, inner sep=0}, "{{{b_3}}}"{description}, curve={height=6pt}, from=3-3, to=4-4]
	\arrow["{\tilde{b}}"{description}, from=3-5, to=3-4]
	\arrow[""{name=9, anchor=center, inner sep=0}, equals, from=3-5, to=4-5]
	\arrow["a", from=3-6, to=3-5]
	\arrow[""{name=10, anchor=center, inner sep=0}, equals, from=3-6, to=4-6]
	\arrow[""{name=11, anchor=center, inner sep=0}, "{{{\tilde{b}_2}}}", from=4-4, to=3-4]
	\arrow["b", from=4-5, to=4-4]
	\arrow["a", from=4-6, to=4-5]
	\arrow[""{name=12, anchor=center, inner sep=0}, equals, from=4-6, to=5-6]
	\arrow["{{{{h_3}}}}"', from=5-1, to=5-2]
	\arrow["{{{k'_3}}}"', from=5-2, to=5-4]
	\arrow[""{name=13, anchor=center, inner sep=0}, "{{{{{b_4}}}}}", from=5-4, to=4-4]
	\arrow[""{name=14, anchor=center, inner sep=0}, "{{{{{a_2}}}}}"', from=5-5, to=4-5]
	\arrow["{{{{{r'_3}}}}}", from=5-5, to=5-4]
	\arrow["{{{{{s'_2}}}}}", from=5-6, to=5-5]
	\arrow["\SIGMA"{description}, draw=none, from=0, to=1]
	\arrow["\SIGMA"{description}, draw=none, from=1, to=2]
	\arrow["\SIGMA"{description}, draw=none, from=3, to=4]
	\arrow["\delta"', between={0.4}{0.7}, Rightarrow, from=2-4, to=8]
	\arrow["\SIGMA"{description}, draw=none, from=4, to=5]
	\arrow["{{\theta''}}", between={0.2}{0.8}, Rightarrow, from=8, to=5-2]
	\arrow["\SIGMA"{description}, draw=none, from=9, to=10]
	\arrow["\SIGMA"{description}, draw=none, from=11, to=9]
	\arrow["\SIGMA"{description}, draw=none, from=13, to=14]
	\arrow["\SIGMA"{description}, draw=none, from=14, to=12]
\end{tikzcd}
\end{equation}
\noindent \textbf{C.} Now, in order to compare the 2-morphism \eqref{eq:eq[[1]]}
with the  2-morphism \eqref{eq:eq[[2]]}, observe first that the column of $\Sg$-squares on the right-hand side is equal in both of them. Thus, in the following we ignore that column and consider only  the remaining parts of the two 2-morphisms. For concluding the $\approx$-equivalence between them, we use the property stated in Remark \ref{rem:use}. Accordingly, apply Rule 4 to obtain
\begin{equation}\label{eq:CC}
\begin{tikzcd}
	{} & {} \\
	A & B \\
	A & {\hat{C}} & Q \\
	A & T \\
	A & {\hat{C}} & {\tilde{B}} \\
	{} & C \\
	{} & {}
	\arrow["{{{r_2}}}", from=1-1, to=1-2]
	\arrow[""{name=0, anchor=center, inner sep=0}, "{{{a_1}}}"', from=1-1, to=2-1]
	\arrow[""{name=1, anchor=center, inner sep=0}, "{{{b_1}}}", from=1-2, to=2-2]
	\arrow["{{{\tilde{b}_1d_1}}}", curve={height=-12pt}, from=1-2, to=3-3]
	\arrow["b"', from=2-1, to=2-2]
	\arrow[""{name=2, anchor=center, inner sep=0}, equals, from=2-1, to=3-1]
	\arrow[""{name=3, anchor=center, inner sep=0}, "{{{\hat{a}_1}}}", from=2-2, to=3-2]
	\arrow["{{{\hat{c}}}}"', from=3-1, to=3-2]
	\arrow[""{name=4, anchor=center, inner sep=0}, equals, from=3-1, to=4-1]
	\arrow["{{{\mu_1}}}"{pos=0.6}, shift left=3, between={0.3}{1}, Rightarrow, from=3-2, to=3-3]
	\arrow[""{name=5, anchor=center, inner sep=0}, "{{{t_1}}}", from=3-2, to=4-2]
	\arrow["{{{t_2}}}", curve={height=-6pt}, from=3-3, to=4-2]
	\arrow["t"', from=4-1, to=4-2]
	\arrow[""{name=6, anchor=center, inner sep=0}, equals, from=4-1, to=5-1]
	\arrow["{{{\hat{c}}}}"', from=5-1, to=5-2]
	\arrow[""{name=7, anchor=center, inner sep=0}, "{{{t_1}}}"', from=5-2, to=4-2]
	\arrow["{{{\mu_2}}}"{pos=0.6}, shift right=4, between={0.3}{1}, Rightarrow, from=5-2, to=5-3]
	\arrow["{{{t_2}}}"', curve={height=6pt}, from=5-3, to=4-2]
	\arrow[""{name=8, anchor=center, inner sep=0}, "{{{a_2}}}", from=6-1, to=5-1]
	\arrow["c"', from=6-1, to=6-2]
	\arrow[""{name=9, anchor=center, inner sep=0}, "{{{\hat{a}_2}}}"', from=6-2, to=5-2]
	\arrow[""{name=10, anchor=center, inner sep=0}, equals, from=7-1, to=6-1]
	\arrow["{{{r'_3}}}"', from=7-1, to=7-2]
	\arrow["{{{\tilde{b}_2b_4}}}"', curve={height=12pt}, from=7-2, to=5-3]
	\arrow[""{name=11, anchor=center, inner sep=0}, "{{{c_2}}}"', from=7-2, to=6-2]
	\arrow["\SIGMA"{description}, draw=none, from=0, to=1]
	\arrow["\SIGMA"{description}, draw=none, from=2, to=3]
	\arrow["\SIGMA"{description}, draw=none, from=4, to=5]
	\arrow["\SIGMA"{description}, draw=none, from=6, to=7]
	\arrow["\SIGMA"{description}, draw=none, from=8, to=9]
	\arrow["\SIGMA"{description}, draw=none, from=10, to=11]
\end{tikzcd}
\quad = \quad
\begin{tikzcd}
	{} & {} \\
	{} & B \\
	A & {\tilde{B}} \\
	A & T \\
	A & {\tilde{B}} \\
	A & B \\
	{} & {}
	\arrow["{{{r_2}}}", from=1-1, to=1-2]
	\arrow[""{name=0, anchor=center, inner sep=0}, equals, from=1-1, to=2-1]
	\arrow[""{name=1, anchor=center, inner sep=0}, "{{d_1}}", from=1-2, to=2-2]
	\arrow["d"', from=2-1, to=2-2]
	\arrow[""{name=2, anchor=center, inner sep=0}, "{{a_1}}"', from=2-1, to=3-1]
	\arrow[""{name=3, anchor=center, inner sep=0}, "{{\tilde{b}_1}}", from=2-2, to=3-2]
	\arrow["{{\tilde{b}}}"{description}, from=3-1, to=3-2]
	\arrow[""{name=4, anchor=center, inner sep=0}, equals, from=3-1, to=4-1]
	\arrow[""{name=5, anchor=center, inner sep=0}, "{{t_2}}", from=3-2, to=4-2]
	\arrow["t"', from=4-1, to=4-2]
	\arrow[""{name=6, anchor=center, inner sep=0}, equals, from=4-1, to=5-1]
	\arrow["{{\tilde{b}}}"{description}, from=5-1, to=5-2]
	\arrow[""{name=7, anchor=center, inner sep=0}, "{{t_2}}"', from=5-2, to=4-2]
	\arrow[""{name=8, anchor=center, inner sep=0}, equals, from=6-1, to=5-1]
	\arrow["b"', from=6-1, to=6-2]
	\arrow[""{name=9, anchor=center, inner sep=0}, "{{\tilde{b}_2}}"', from=6-2, to=5-2]
	\arrow[""{name=10, anchor=center, inner sep=0}, "{{a_2}}", from=7-1, to=6-1]
	\arrow["{{{r'_3}}}"', from=7-1, to=7-2]
	\arrow[""{name=11, anchor=center, inner sep=0}, "{{b_4}}"', from=7-2, to=6-2]
	\arrow["\SIGMA"{description}, draw=none, from=0, to=1]
	\arrow["\SIGMA"{description}, draw=none, from=2, to=3]
	\arrow["\SIGMA"{description}, draw=none, from=4, to=5]
	\arrow["\SIGMA"{description}, draw=none, from=6, to=7]
	\arrow["\SIGMA"{description}, draw=none, from=8, to=9]
	\arrow["\SIGMA"{description}, draw=none, from=10, to=11]
\end{tikzcd}
\, .
\end{equation}
To name the left-hand parts of \eqref{eq:eq[[1]]} and \eqref{eq:eq[[2]]}, put
$$\begin{array}{l}
\Lambda_1=(\hat{a}_2\circ \phi)\cdot (\psi\circ k'_2h_2)\cdot(\hat{a}_1\circ \theta'\circ h_2)\;\colon\; \hat{a}_1b_1k_2h_2\Rightarrow \hat{a}_2c_2k'_3h_3 \,;\\
\Lambda_2=(\tilde{b}_2\circ \theta''\circ h_3)\cdot (\delta\circ k_3h_3)\cdot(\tilde{b}_1\circ \alpha)\;\colon\; \tilde{b}_1d_1k_2h_2\Rightarrow \tilde{b}_2b_4k'_3h_3.
\end{array}$$
Now using \eqref{eq:CC}, \eqref{eq:eq-A1}, \eqref{eq:eq-A3}, \eqref{eq:eq-A2}, and \eqref{eq:CC} again, in sequence, we obtain the first equality below, and using \eqref{eq:BB}, \eqref{eq:eq-das-sol} and \eqref{eq:eq-B2} we obtain the second equality.
\[
\begin{tikzcd}
	{} & {} \\
	{} &&& {} \\
	{} & {} && {} \\
	{} & {} && {} & {} \\
	{} & D & B \\
	{} & {\tilde{B}} & {\hat{C}} && {} \\
	{} & T
	\arrow["r", from=1-1, to=1-2]
	\arrow["g"', from=1-1, to=2-1]
	\arrow["{{h_3}}", curve={height=-12pt}, from=1-2, to=2-4]
	\arrow[""{name=0, anchor=center, inner sep=0}, "{{h_2}}", from=1-2, to=3-2]
	\arrow["{{h_1}}"', from=2-1, to=3-1]
	\arrow["{{k'_3}}", from=2-4, to=3-4]
	\arrow["{{r_1}}"', from=3-1, to=3-2]
	\arrow[""{name=1, anchor=center, inner sep=0}, "{{k_1}}"', from=3-1, to=4-1]
	\arrow["{{\Lambda_1}}", between={0.2}{0.8}, Rightarrow, from=3-2, to=3-4]
	\arrow[""{name=2, anchor=center, inner sep=0}, "{{k_2}}", from=3-2, to=4-2]
	\arrow["{{c_2}}", from=3-4, to=4-4]
	\arrow["{{b_4}}", curve={height=-12pt}, from=3-4, to=4-5]
	\arrow["{{r_2}}"', from=4-1, to=4-2]
	\arrow[""{name=3, anchor=center, inner sep=0}, equals, from=4-1, to=5-1]
	\arrow[""{name=4, anchor=center, inner sep=0}, "{{d_1}}", from=4-2, to=5-2]
	\arrow["{b_1}", curve={height=-12pt}, from=4-2, to=5-3]
	\arrow[""{name=5, anchor=center, inner sep=0}, "{{\hat{a}_2}}"'{pos=0.4}, curve={height=-12pt}, from=4-4, to=6-3]
	\arrow[""{name=6, anchor=center, inner sep=0}, "{{\tilde{b}_2}}", from=4-5, to=6-5]
	\arrow["d"', from=5-1, to=5-2]
	\arrow[""{name=7, anchor=center, inner sep=0}, "{{a_1}}"', from=5-1, to=6-1]
	\arrow[""{name=8, anchor=center, inner sep=0}, "{{\tilde{b}_1}}", from=5-2, to=6-2]
	\arrow[""{name=9, anchor=center, inner sep=0}, "{{\hat{a}_1}}", from=5-3, to=6-3]
	\arrow["{{\tilde{b}}}"', from=6-1, to=6-2]
	\arrow[""{name=10, anchor=center, inner sep=0}, "{{t_2}}", from=6-2, to=7-2]
	\arrow["{{t_1}}"', curve={height=-12pt}, from=6-3, to=7-2]
	\arrow["{{t_2}}", curve={height=-18pt}, from=6-5, to=7-2]
	\arrow[""{name=11, anchor=center, inner sep=0}, equals, from=7-1, to=6-1]
	\arrow["t"', from=7-1, to=7-2]
	\arrow["\SIGMA"{description}, draw=none, from=2-1, to=0]
	\arrow["\SIGMA"{description}, draw=none, from=1, to=2]
	\arrow["\SIGMA"{description}, draw=none, from=3, to=4]
	\arrow["{{\mu_2}}"{pos=0.6}, between={0.3}{0.8}, Rightarrow, from=5, to=6]
	\arrow["\SIGMA"{description}, draw=none, from=7, to=8]
	\arrow["{{\mu_1^{-1}}}"{pos=0.6}, between={0.3}{0.8}, Rightarrow, from=8, to=9]
	\arrow["\SIGMA"{description}, draw=none, from=11, to=10]
\end{tikzcd}
\quad = \quad
\begin{tikzcd}
	& {} & {} \\
	& {} \\
	& {} & {} \\
	{} & {} & {} \\
	& {} & D \\
	& {} & {\tilde{B}} \\
	& {} & T
	\arrow["r", from=1-2, to=1-3]
	\arrow["g"', from=1-2, to=2-2]
	\arrow[""{name=0, anchor=center, inner sep=0}, "{{h_3}}", from=1-3, to=3-3]
	\arrow["{{{{{h_1}}}}}"', from=2-2, to=3-2]
	\arrow["{{{{r_2}}}}"', from=3-2, to=3-3]
	\arrow["{{{k_1}}}"', curve={height=6pt}, from=3-2, to=4-1]
	\arrow[""{name=1, anchor=center, inner sep=0}, "{{{k'_1}}}"', from=3-2, to=4-2]
	\arrow[""{name=2, anchor=center, inner sep=0}, "{{{k'_3}}}", from=3-3, to=4-3]
	\arrow["\theta", between={0.1}{0.9}, Rightarrow, from=4-1, to=4-2]
	\arrow["{{{a_1}}}"', curve={height=6pt}, from=4-1, to=5-2]
	\arrow["{{{{r'_3}}}}"', from=4-2, to=4-3]
	\arrow[""{name=3, anchor=center, inner sep=0}, "{{{a_2}}}"', from=4-2, to=5-2]
	\arrow[""{name=4, anchor=center, inner sep=0}, "{{{{b_4}}}}", from=4-3, to=5-3]
	\arrow["b"', from=5-2, to=5-3]
	\arrow[""{name=5, anchor=center, inner sep=0}, equals, from=5-2, to=6-2]
	\arrow[""{name=6, anchor=center, inner sep=0}, "{{{{\tilde{b}_2}}}}", from=5-3, to=6-3]
	\arrow["{{{{{\tilde{b}}}}}}"', from=6-2, to=6-3]
	\arrow[""{name=7, anchor=center, inner sep=0}, "{{{{{t_2}}}}}", from=6-3, to=7-3]
	\arrow[""{name=8, anchor=center, inner sep=0}, equals, from=7-2, to=6-2]
	\arrow["t"', from=7-2, to=7-3]
	\arrow["\SIGMA"{description}, draw=none, from=2-2, to=0]
	\arrow["\SIGMA"{description}, draw=none, from=1, to=2]
	\arrow["\SIGMA"{description}, draw=none, from=3, to=4]
	\arrow["\SIGMA"{description}, draw=none, from=5, to=6]
	\arrow["\SIGMA"{description}, draw=none, from=8, to=7]
\end{tikzcd}
\quad = \quad
\begin{tikzcd}
	{} & {} \\
	{} \\
	{} & {} & {} \\
	{} & {} & {} \\
	{} & D & {} \\
	{} & {\tilde{B}} \\
	{} & T
	\arrow["r", from=1-1, to=1-2]
	\arrow["g"', from=1-1, to=2-1]
	\arrow[""{name=0, anchor=center, inner sep=0}, "{h_2}", from=1-2, to=3-2]
	\arrow["{h_3}", curve={height=-12pt}, from=1-2, to=3-3]
	\arrow["{{{h_1}}}"', from=2-1, to=3-1]
	\arrow["{r_1}"', from=3-1, to=3-2]
	\arrow[""{name=1, anchor=center, inner sep=0}, "{{{k_1}}}"', from=3-1, to=4-1]
	\arrow[""{name=2, anchor=center, inner sep=0}, "{k_2}", from=3-2, to=4-2]
	\arrow[""{name=3, anchor=center, inner sep=0}, "{k'_3}", from=3-3, to=4-3]
	\arrow["{r_2}"', from=4-1, to=4-2]
	\arrow[""{name=4, anchor=center, inner sep=0}, equals, from=4-1, to=5-1]
	\arrow[""{name=5, anchor=center, inner sep=0}, "{d_1}", from=4-2, to=5-2]
	\arrow["{b_4}", from=4-3, to=5-3]
	\arrow["d"', from=5-1, to=5-2]
	\arrow[""{name=6, anchor=center, inner sep=0}, "{a_1}"', from=5-1, to=6-1]
	\arrow[""{name=7, anchor=center, inner sep=0}, "{\tilde{b}_1}", from=5-2, to=6-2]
	\arrow["{\tilde{b}_2}", curve={height=-6pt}, from=5-3, to=6-2]
	\arrow["{{{\tilde{b}}}}"', from=6-1, to=6-2]
	\arrow[""{name=8, anchor=center, inner sep=0}, "{{{t_2}}}", from=6-2, to=7-2]
	\arrow[""{name=9, anchor=center, inner sep=0}, equals, from=7-1, to=6-1]
	\arrow["t"', from=7-1, to=7-2]
	\arrow["\SIGMA"{description}, draw=none, from=2-1, to=0]
	\arrow["\SIGMA"{description}, draw=none, from=1, to=2]
	\arrow["{\Lambda_2}"{pos=0.6}, shorten <=13pt, shorten >=6pt, Rightarrow, from=2, to=3]
	\arrow["\SIGMA"{description}, draw=none, from=4, to=5]
	\arrow["\SIGMA"{description}, draw=none, from=6, to=7]
	\arrow["\SIGMA"{description}, draw=none, from=9, to=8]
\end{tikzcd}
\, .\]
Consequently, by Remark \ref{rem:use},
 taking the $\Sg$-square
 \(\begin{tikzcd}
	{} & {} \\
	{} & {}
	\arrow["r", from=1-1, to=1-2]
	\arrow[""{name=0, anchor=center, inner sep=0}, "{k_1h_1g}"', from=1-1, to=2-1]
	\arrow[""{name=1, anchor=center, inner sep=0}, "{k_2h_2}", from=1-2, to=2-2]
	\arrow["{r_2}"', from=2-1, to=2-2]
	\arrow["\SIGMA"{description}, draw=none, from=0, to=1]
\end{tikzcd}\)
 to play the role of \eqref{eq:sp-Square}, we conclude that the 2-morphisms represented in \eqref{eq:eq[[1]]}  and \eqref{eq:eq[[2]]} are $\approx$-equivalent.

\noindent \underline{Case $i=\mathbf{d}$ and $j=\mathbf{s}_1$ (or $j=\mathbf{s}$)}. We consider $j=\mathbf{s}_1$, as for $j=\mathbf{s}$ the procedure is the same.

We have $\Sg$-paths of the form
\(\begin{tikzcd}
	{S_1} & {S_3} & {S_2}
	\arrow["\mathbf{d}", squiggly, from=1-1, to=1-2]
	\arrow["\mathbf{s}_1", squiggly, from=1-2, to=1-3]
\end{tikzcd}\)
 and
 \(\begin{tikzcd}
	{S_1} & {S_4} & {S_2}
	\arrow["\mathbf{s}_1", squiggly, from=1-1, to=1-2]
	\arrow["\mathbf{d}", squiggly, from=1-2, to=1-3]
\end{tikzcd}\).
Note that the $\Sg$-steps
\(\begin{tikzcd}
	{S_1} & {S_3}
	\arrow["\mathbf{d}", squiggly, from=1-1, to=1-2]
\end{tikzcd}\)
and
 \(\begin{tikzcd}
	{S_4} & {S_2}
	\arrow["\mathbf{d}", squiggly, from=1-1, to=1-2]
\end{tikzcd}\)
determine that the four $\Sg$-schemes are of the form
\(
\adjustbox{scale=0.60}{\begin{tikzcd}
	&& {} & {} \\
	& {} & {} \\
	{} & {} && {} \\
	{} &&& {}
	\arrow[from=1-3, to=1-4]
	\arrow[from=1-3, to=2-3]
	\arrow[from=1-4, to=3-4]
	\arrow[from=2-2, to=2-3]
	\arrow[from=2-2, to=3-2]
	\arrow[from=3-1, to=3-2]
	\arrow[from=3-1, to=4-1]
	\arrow[from=3-2, to=3-4]
	\arrow[from=3-4, to=4-4]
	\arrow[from=4-1, to=4-4]
\end{tikzcd}}
\, .\)
The combination of this  with the existence of the $\Sg$-steps
\(\begin{tikzcd}
	{S_3} & {S_2}
	\arrow["\mathbf{s}_1", squiggly, from=1-1, to=1-2]
\end{tikzcd}\)
and
 \(\begin{tikzcd}
	{S_1} & {S_4}
	\arrow["\mathbf{s}_1", squiggly, from=1-1, to=1-2]
\end{tikzcd}\)
implies that all $\Sg$-schemes are of the form
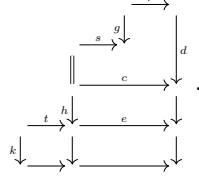
\[
\adjustbox{scale=0.60}{\begin{tikzcd}
	&& {} & {} \\
	& {} & {} \\
	& {} && {} \\
	{} & {} && {} \\
	{} & {} && {}
	\arrow["r", from=1-3, to=1-4]
	\arrow["g"', from=1-3, to=2-3]
	\arrow["d", from=1-4, to=3-4]
	\arrow["s", from=2-2, to=2-3]
	\arrow[equals, from=2-2, to=3-2]
	\arrow["c", from=3-2, to=3-4]
	\arrow["h"', from=3-2, to=4-2]
	\arrow[from=3-4, to=4-4]
	\arrow["t", from=4-1, to=4-2]
	\arrow["k"', from=4-1, to=5-1]
	\arrow["e", from=4-2, to=4-4]
	\arrow[from=4-2, to=5-2]
	\arrow[from=4-4, to=5-4]
	\arrow[from=5-1, to=5-2]
	\arrow[from=5-2, to=5-4]
\end{tikzcd}}
\, . \]
Observe that along the two $\Sg$-paths,
\(
\adjustbox{scale=0.60}{\begin{tikzcd}
	& {} & {} \\
	{} & {} \\
	{} && {}
	\arrow["r", from=1-2, to=1-3]
	\arrow["g"', from=1-2, to=2-2]
	\arrow["d", from=1-3, to=3-3]
	\arrow["s", from=2-1, to=2-2]
	\arrow[equals, from=2-1, to=3-1]
	\arrow["c", from=3-1, to=3-3]
\end{tikzcd}}
\)
remains unchanged.
Thus, to obtain the basic $\Omega$ 2-cells, and the subsequent two $\Omega$ 2-cells corresponding to the two $\Sg$-paths, we just need to work with the part
\(
\adjustbox{scale=0.60}
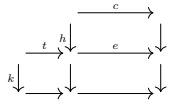
{\begin{tikzcd}
	& {} && {} \\
	{} & {} && {} \\
	{} & {} && {}
	\arrow["c", from=1-2, to=1-4]
	\arrow["h"', from=1-2, to=2-2]
	\arrow[from=1-4, to=2-4]
	\arrow["t", from=2-1, to=2-2]
	\arrow["k"', from=2-1, to=3-1]
	\arrow["e", from=2-2, to=2-4]
	\arrow[from=2-2, to=3-2]
	\arrow[from=2-4, to=3-4]
	\arrow[from=3-1, to=3-2]
	\arrow[from=3-2, to=3-4]
\end{tikzcd}}
\) .
 Now the procedure is analogous to the previous case: we make use of Lemma \ref{lem:use}  applied to the parts of the form
 \(
 \adjustbox{scale=0.60}{\begin{tikzcd}
	{} & {} && {} \\
	{} & {} && {}
	\arrow["t", from=1-1, to=1-2]
	\arrow["k"', from=1-1, to=2-1]
	\arrow["e", from=1-2, to=1-4]
	\arrow[from=1-2, to=2-2]
	\arrow[from=1-4, to=2-4]
	\arrow[from=2-1, to=2-2]
	\arrow[from=2-2, to=2-4]
\end{tikzcd}}
 \)
  in the passages of type $\mathbf{d}$ and, at the end, we use Remark \ref{lem:use} to conclude that the two $\Sg$-paths give rise to $\approx$-equivalent 2-morphisms.

\noindent \underline{Case $i=\mathbf{u}$ and $j=\mathbf{s}_1$ (or $j=\mathbf{s}$)}. As before, we consider $j=s_1$.

We have $\Sg$-paths of the form
\(\begin{tikzcd}
	{S_1} & {S_3} & {S_2}
	\arrow["\mathbf{u}", squiggly, from=1-1, to=1-2]
	\arrow["\mathbf{s}_1", squiggly, from=1-2, to=1-3]
\end{tikzcd}\)
 and
 \(\begin{tikzcd}
	{S_1} & {S_4} & {S_2}
	\arrow["\mathbf{s}_1", squiggly, from=1-1, to=1-2]
	\arrow["\mathbf{u}", squiggly, from=1-2, to=1-3]
\end{tikzcd}\).
Thus, the four $\Sg$-schemes are of the form
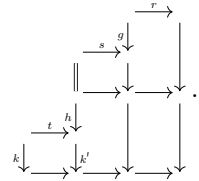
\[\adjustbox{scale=0.60}{\begin{tikzcd}
	&& {} & {} \\
	& {} & {} \\
	& {} & {} & {} \\
	{} & {} \\
	{} & {} & {} & {}
	\arrow["r", from=1-3, to=1-4]
	\arrow["g"', from=1-3, to=2-3]
	\arrow[from=1-4, to=3-4]
	\arrow["s", from=2-2, to=2-3]
	\arrow[equals, from=2-2, to=3-2]
	\arrow[from=2-3, to=3-3]
	\arrow[from=3-2, to=3-3]
	\arrow["h"', from=3-2, to=4-2]
	\arrow[from=3-3, to=3-4]
	\arrow[from=3-3, to=5-3]
	\arrow[from=3-4, to=5-4]
	\arrow["t", from=4-1, to=4-2]
	\arrow["k"', from=4-1, to=5-1]
	\arrow["{k'}", from=4-2, to=5-2]
	\arrow[from=5-1, to=5-2]
	\arrow[from=5-2, to=5-3]
	\arrow[from=5-3, to=5-4]
\end{tikzcd}}\, .\]
In each $\Sg$-step above, the $\Sg$-squares
\(
\begin{tikzcd}
	{} & {} \\
	{} & {}
	\arrow["s", from=1-1, to=1-2]
	\arrow[""{name=0, anchor=center, inner sep=0}, equals, from=1-1, to=2-1]
	\arrow[""{name=1, anchor=center, inner sep=0}, from=1-2, to=2-2]
	\arrow[from=2-1, to=2-2]
	\arrow["\SIGMA"{description}, draw=none, from=0, to=1]
\end{tikzcd}
\)
and
\(
\begin{tikzcd}
	{} & {} \\
	{} & {}
	\arrow["t", from=1-1, to=1-2]
	\arrow[""{name=0, anchor=center, inner sep=0}, "k"', from=1-1, to=2-1]
	\arrow[""{name=1, anchor=center, inner sep=0}, "{k'}", from=1-2, to=2-2]
	\arrow["{t'}", from=2-1, to=2-2]
	\arrow["\SIGMA"{description}, draw=none, from=0, to=1]
\end{tikzcd}\,
\)
remain unchanged, so they do not interfere with the $\Sg$-steps. To obtain the desired result we act as for $i=\bd$ and $j=\bu$: Apply Lemma \ref{lem:use} to the parts of the $\Sg$-schemes given by
\(
\adjustbox{scale=0.60}{\begin{tikzcd}
	{} & {} & {} \\
	{} \\
	{} & {} & {}
	\arrow[from=1-1, to=1-2]
	\arrow["h"', from=1-1, to=2-1]
	\arrow[from=1-2, to=1-3]
	\arrow[from=1-2, to=3-2]
	\arrow[from=1-3, to=3-3]
	\arrow["{{k'}}", from=2-1, to=3-1]
	\arrow[from=3-1, to=3-2]
	\arrow[from=3-2, to=3-3]
\end{tikzcd}}
\)
in the passages of type $\mathbf{s}_1$, and, at the end, use Remark \ref{rem:use}.

\noindent \underline{Case $i=\mathbf{u}$ and $j=\mathbf{d}_1$}. This is similar to $i=\mathbf{u}$ and $j=\mathbf{d}$.

\noindent \underline{Case $i=\mathbf{s}$ and $j=\mathbf{d}_1$}. This is similar to $i=\mathbf{d}$ and $j=\mathbf{s}$.

\noindent \underline{Case $i=\mathbf{d_1}$ and $j=\mathbf{s}_1$}. We easily see that all the $\Sg$-schemes have the configuration
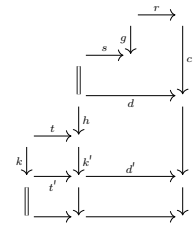
\[
\adjustbox{scale=0.60}{\begin{tikzcd}
	&& {} & {} \\
	& {} & {} \\
	& {} && {} \\
	{} & {} \\
	{} & {} && {} \\
	{} & {} && {}
	\arrow["r", from=1-3, to=1-4]
	\arrow["g"', from=1-3, to=2-3]
	\arrow["c", from=1-4, to=3-4]
	\arrow["s", from=2-2, to=2-3]
	\arrow[equals, from=2-2, to=3-2]
	\arrow["d"', from=3-2, to=3-4]
	\arrow["h", from=3-2, to=4-2]
	\arrow[from=3-4, to=5-4]
	\arrow["t", from=4-1, to=4-2]
	\arrow["k"', from=4-1, to=5-1]
	\arrow["{{k'}}", from=4-2, to=5-2]
	\arrow["{{t'}}"', from=5-1, to=5-2]
	\arrow[equals, from=5-1, to=6-1]
	\arrow["{d'}", from=5-2, to=5-4]
	\arrow[from=5-2, to=6-2]
	\arrow[from=5-4, to=6-4]
	\arrow[from=6-1, to=6-2]
	\arrow[from=6-2, to=6-4]
\end{tikzcd}}
\]
where the parts
\(
\begin{tikzcd}
& {} & {} \\
	{} & {} \\
	{} && {}
	\arrow["r", from=1-2, to=1-3]
	\arrow["g"', from=1-2, to=2-2]
	\arrow["c", from=1-3, to=3-3]
	\arrow["s", from=2-1, to=2-2]
	\arrow[equals, from=2-1, to=3-1]
	\arrow["d"', from=3-1, to=3-3]
\end{tikzcd}
\)
and
\(
\begin{tikzcd}
{} & {} \\
	{} & {}
	\arrow["t", from=1-1, to=1-2]
	\arrow["k"', from=1-1, to=2-1]
	\arrow["{{k'}}", from=1-2, to=2-2]
	\arrow["{{t'}}"', from=2-1, to=2-2]
\end{tikzcd}
\)
are unchanged. We proceed as for the first case, by applying Lemma \ref{lem:use} to the part
\(
\begin{tikzcd}
	{} & {} && {} \\
	{} & {} && {}
	\arrow["{{{t'}}}"', from=1-1, to=1-2]
	\arrow[equals, from=1-1, to=2-1]
	\arrow["{d'}"', from=1-2, to=1-4]
	\arrow[from=1-2, to=2-2]
	\arrow[from=1-4, to=2-4]
	\arrow[from=2-1, to=2-2]
	\arrow[from=2-2, to=2-4]
\end{tikzcd}
\)
 when obtaining the basic $\Omega$ 2-cells corresponding to the passages ${\bd}_1$ and using Remark \ref{rem:use} at the end.
\end{proof}

\begin{corollary}\label{cor:length2} Every two $\Sigma$-paths of length 2 or less (as in \cref{assumption-path}) between $\Sg$-schemes of level 3 are equivalent. Equivalently, every cycle of 4 or less $\Sg$-steps between $\Sg$-schemes of level 3 is equivalent to an identity $\Sg$-step.
\end{corollary}

\begin{proof}  This was already  seen in Lemma \ref{lem:Sigma-steps} for cycles of length 1 or 2.
Concerning cycles of length 3, it suffices to consider the cases where we have three consecutive $\Sg$-steps all of different types. Indeed, in the case we have two repeated types, then since a $\Sg$-path of two $\Sg$-steps of the same type are equivalent to a $\Sg$-path of just one $\Sg$-step of that type, we again fall under the case of length 1 or 2.

It remains to consider all cycles of length 3 or 4 not yet encompassed by Proposition \ref{pro:adend}. Indeed, as we are going to see, all of them are obtained via  Lemma \ref{lem:Sigma-steps} and that proposition.

 \underline{Cycles of length 3}. If the cycle  contains the $\Sg$-steps $\bs$ and ${\bs}_1$, or the $\Sg$-steps $\bd$ and ${\bd}_1$, then the result  follows immediately from Lemma \ref{lem:Sigma-steps}. Since every $\Sg$-step may be seen as an undirected edge (because basic $\Omega$ 2-cells are invertible and the inverse corresponds to a $\Sg$-step of the same type, just reversed), the only cases to study are $\mathbf{dus}$, $\mathbf{dus}_1$, $\mathbf{d}_1\mathbf{us}$ and $\mathbf{d}_1\mathbf{us}_1$.  We analyse  $\mathbf{dus}$, the remaining cases are similar.

  \underline{Case $\mathbf{dus}$}. Let us consider
  \[\begin{tikzcd}
	{S_1} & {S_2} & {S_3} & {S_1}
	\arrow["\bd", squiggly, from=1-1, to=1-2]
	\arrow["\bu", squiggly, from=1-2, to=1-3]
	\arrow["\bs", squiggly, from=1-3, to=1-4]
\end{tikzcd}\, .\]
Observe that, taking into account the type of the $\Sigma$-steps, we conclude successively as follows, where the necessary equal $\Sigma$-squares will be indicated with the same number:

$\bullet$ $S_1$ is simultaneously of the forms $\bd$ and $\bs$, hence
\(
S_1=\adjustbox{scale=0.60}{\begin{tikzcd}
	&& {} & {} \\
	& {} & {} & {} \\
	{} & {} & {} & {} \\
	{} & {} & {} & {}
	\arrow[from=1-3, to=1-4]
	\arrow[from=1-3, to=2-3]
	\arrow[from=1-4, to=2-4]
	\arrow[from=2-2, to=2-3]
	\arrow[from=2-2, to=3-2]
	\arrow[from=2-3, to=2-4]
	\arrow[from=2-4, to=3-4]
	\arrow[from=3-1, to=3-2]
	\arrow[from=3-1, to=4-1]
	\arrow[from=3-2, to=3-4]
	\arrow[from=3-2, to=4-2]
	\arrow[from=3-4, to=4-4]
	\arrow[from=4-1, to=4-2]
	\arrow[from=4-2, to=4-4]
\end{tikzcd}}
\);

$\bullet$ $S_2$ is simultaneously of the forms $\bd$ and $\bu$, hence
\(S_2=\adjustbox{scale=0.60}{
\begin{tikzcd}
	&& {} & {} \\
	& {} & {} & {} \\
	{} & {} & {} & {} \\
	{} && {} & {}
	\arrow[from=1-3, to=1-4]
	\arrow[from=1-3, to=2-3]
	\arrow[from=1-4, to=3-4]
	\arrow[from=2-2, to=2-3]
	\arrow[from=2-2, to=3-2]
	\arrow[from=2-3, to=3-3]
	\arrow[from=3-1, to=3-2]
	\arrow[from=3-1, to=4-1]
	\arrow[from=3-2, to=3-3]
	\arrow[from=3-3, to=3-4]
	\arrow[from=3-3, to=4-3]
	\arrow[from=3-4, to=4-4]
	\arrow[from=4-1, to=4-3]
	\arrow[from=4-3, to=4-4]
\end{tikzcd}
}
\);

$\bullet$ $S_3$ is simultaneously of the forms $\bu$ and $\bs$, hence
\(S_3=\adjustbox{scale=0.60}{
\begin{tikzcd}
	&& {} & {} \\
	& {} & {} & {} \\
	{} & {} \\
	{} & {} & {} & {}
	\arrow[from=1-3, to=1-4]
	\arrow[from=1-3, to=2-3]
	\arrow[from=1-4, to=2-4]
	\arrow[from=2-2, to=2-3]
	\arrow[from=2-2, to=3-2]
	\arrow[from=2-3, to=2-4]
	\arrow[from=2-3, to=4-3]
	\arrow[from=2-4, to=4-4]
	\arrow[from=3-1, to=3-2]
	\arrow[from=3-1, to=4-1]
	\arrow[from=3-2, to=4-2]
	\arrow[from=4-1, to=4-2]
	\arrow[from=4-2, to=4-3]
	\arrow[from=4-3, to=4-4]
\end{tikzcd}
}
\);

$\bullet$ $\bd\colon  S_1\rightsquigarrow S_2$ obliges us to have
\(S_1=
\adjustbox{scale=0.60}{\begin{tikzcd}
	&& {} & {} \\
	& {} & {} & {} \\
	{} & {} & {} & {} \\
	{} & {} & {} & {}
	\arrow[from=1-3, to=1-4]
	\arrow[""{name=0, anchor=center, inner sep=0}, from=1-3, to=2-3]
	\arrow[""{name=1, anchor=center, inner sep=0}, from=1-4, to=2-4]
	\arrow[from=2-2, to=2-3]
	\arrow[""{name=2, anchor=center, inner sep=0}, from=2-2, to=3-2]
	\arrow[from=2-3, to=2-4]
	\arrow[""{name=3, anchor=center, inner sep=0}, from=2-3, to=3-3]
	\arrow[""{name=4, anchor=center, inner sep=0}, from=2-4, to=3-4]
	\arrow[from=3-1, to=3-2]
	\arrow[""{name=5, anchor=center, inner sep=0}, from=3-1, to=4-1]
	\arrow[from=3-2, to=3-3]
	\arrow[""{name=6, anchor=center, inner sep=0}, from=3-2, to=4-2]
	\arrow[from=3-3, to=3-4]
	\arrow[""{name=7, anchor=center, inner sep=0}, from=3-4, to=4-4]
	\arrow[from=4-1, to=4-2]
	\arrow[from=4-2, to=4-4]
	\arrow["1"{description}, draw=none, from=0, to=1]
	\arrow["2"{description}, draw=none, from=2, to=3]
	\arrow["3"{description}, draw=none, from=3, to=4]
	\arrow["4"{description}, draw=none, from=5, to=6]
	\arrow["5"{description}, draw=none, from=6, to=7]
\end{tikzcd}}\quad
\)
and
\( \quad
S_2=
\adjustbox{scale=0.60}{\begin{tikzcd}
	&& {} & {} \\
	& {} & {} & {} \\
	{} & {} & {} & {} \\
	{} && {} & {}
	\arrow[from=1-3, to=1-4]
	\arrow[""{name=0, anchor=center, inner sep=0}, from=1-3, to=2-3]
	\arrow[""{name=1, anchor=center, inner sep=0}, from=1-4, to=2-4]
	\arrow[from=2-2, to=2-3]
	\arrow[""{name=2, anchor=center, inner sep=0}, from=2-2, to=3-2]
	\arrow[from=2-3, to=2-4]
	\arrow[""{name=3, anchor=center, inner sep=0}, from=2-3, to=3-3]
	\arrow[""{name=4, anchor=center, inner sep=0}, from=2-4, to=3-4]
	\arrow[from=3-1, to=3-2]
	\arrow[""{name=5, anchor=center, inner sep=0}, from=3-1, to=4-1]
	\arrow[from=3-2, to=3-3]
	\arrow[from=3-3, to=3-4]
	\arrow[""{name=6, anchor=center, inner sep=0}, from=3-3, to=4-3]
	\arrow[""{name=7, anchor=center, inner sep=0}, from=3-4, to=4-4]
	\arrow[from=4-1, to=4-3]
	\arrow[from=4-3, to=4-4]
	\arrow["1"{description}, draw=none, from=0, to=1]
	\arrow["2"{description}, draw=none, from=2, to=3]
	\arrow["3"{description}, draw=none, from=3, to=4]
	\arrow["6"{description}, draw=none, from=5, to=6]
	\arrow["8"{description}, draw=none, from=6, to=7]
\end{tikzcd}}
\);

$\bullet$ $\bu\colon S_2\rightsquigarrow S_3$ obliges us to have
\(S_3=
\adjustbox{scale=0.60}{\begin{tikzcd}
	&& {} & {} \\
	& {} & {} & {} \\
	{} & {} & {} & {} \\
	{} & {} & {} & {}
	\arrow[from=1-3, to=1-4]
	\arrow[from=1-3, to=2-3]
	\arrow[from=1-4, to=2-4]
	\arrow[from=2-2, to=2-3]
	\arrow[""{name=0, anchor=center, inner sep=0}, from=2-2, to=3-2]
	\arrow[from=2-3, to=2-4]
	\arrow[""{name=1, anchor=center, inner sep=0}, from=2-3, to=3-3]
	\arrow[""{name=2, anchor=center, inner sep=0}, from=2-4, to=3-4]
	\arrow[from=3-1, to=3-2]
	\arrow[""{name=3, anchor=center, inner sep=0}, from=3-1, to=4-1]
	\arrow[from=3-2, to=3-3]
	\arrow[""{name=4, anchor=center, inner sep=0}, from=3-2, to=4-2]
	\arrow[from=3-3, to=3-4]
	\arrow[""{name=5, anchor=center, inner sep=0}, from=3-3, to=4-3]
	\arrow[""{name=6, anchor=center, inner sep=0}, from=3-4, to=4-4]
	\arrow[from=4-1, to=4-2]
	\arrow[from=4-2, to=4-3]
	\arrow[from=4-3, to=4-4]
	\arrow["2"{description}, draw=none, from=0, to=1]
	\arrow["9"{description}, draw=none, from=1, to=2]
	\arrow["4"{description}, draw=none, from=3, to=4]
	\arrow["7"{description}, draw=none, from=4, to=5]
	\arrow["10"{description}, draw=none, from=5, to=6]
\end{tikzcd}}\)
and
\( \quad
S_2=
\adjustbox{scale=0.60}{\begin{tikzcd}
	&& {} & {} \\
	& {} & {} & {} \\
	{} & {} & {} & {} \\
	{} & {} & {} & {}
	\arrow[from=1-3, to=1-4]
	\arrow[""{name=0, anchor=center, inner sep=0}, from=1-3, to=2-3]
	\arrow[""{name=1, anchor=center, inner sep=0}, from=1-4, to=2-4]
	\arrow[from=2-2, to=2-3]
	\arrow[""{name=2, anchor=center, inner sep=0}, from=2-2, to=3-2]
	\arrow[from=2-3, to=2-4]
	\arrow[""{name=3, anchor=center, inner sep=0}, from=2-3, to=3-3]
	\arrow[""{name=4, anchor=center, inner sep=0}, from=2-4, to=3-4]
	\arrow[from=3-1, to=3-2]
	\arrow[""{name=5, anchor=center, inner sep=0}, from=3-1, to=4-1]
	\arrow[from=3-2, to=3-3]
	\arrow[""{name=6, anchor=center, inner sep=0}, from=3-2, to=4-2]
	\arrow[from=3-3, to=3-4]
	\arrow[""{name=7, anchor=center, inner sep=0}, from=3-3, to=4-3]
	\arrow[""{name=8, anchor=center, inner sep=0}, from=3-4, to=4-4]
	\arrow[from=4-1, to=4-2]
	\arrow[from=4-2, to=4-3]
	\arrow[from=4-3, to=4-4]
	\arrow["1"{description}, draw=none, from=0, to=1]
	\arrow["2"{description}, draw=none, from=2, to=3]
	\arrow["3"{description}, draw=none, from=3, to=4]
	\arrow["4"{description}, draw=none, from=5, to=6]
	\arrow["7"{description}, draw=none, from=6, to=7]
	\arrow["8"{description}, draw=none, from=7, to=8]
\end{tikzcd}}
\);

$\bullet$ $\bs\colon S_3\rightsquigarrow S_1$ requires
\(S_3=
\adjustbox{scale=0.60}{\begin{tikzcd}
	&& {} & {} \\
	& {} & {} & {} \\
	{} & {} & {} & {} \\
	{} & {} & {} & {}
	\arrow[from=1-3, to=1-4]
	\arrow[""{name=0, anchor=center, inner sep=0}, from=1-3, to=2-3]
	\arrow[""{name=1, anchor=center, inner sep=0}, from=1-4, to=2-4]
	\arrow[from=2-2, to=2-3]
	\arrow[""{name=2, anchor=center, inner sep=0}, from=2-2, to=3-2]
	\arrow[from=2-3, to=2-4]
	\arrow[""{name=3, anchor=center, inner sep=0}, from=2-3, to=3-3]
	\arrow[""{name=4, anchor=center, inner sep=0}, from=2-4, to=3-4]
	\arrow[from=3-1, to=3-2]
	\arrow[""{name=5, anchor=center, inner sep=0}, from=3-1, to=4-1]
	\arrow[from=3-2, to=3-3]
	\arrow[""{name=6, anchor=center, inner sep=0}, from=3-2, to=4-2]
	\arrow[from=3-3, to=3-4]
	\arrow[""{name=7, anchor=center, inner sep=0}, from=3-3, to=4-3]
	\arrow[""{name=8, anchor=center, inner sep=0}, from=3-4, to=4-4]
	\arrow[from=4-1, to=4-2]
	\arrow[from=4-2, to=4-3]
	\arrow[from=4-3, to=4-4]
	\arrow["1"{description}, draw=none, from=0, to=1]
	\arrow["2"{description}, draw=none, from=2, to=3]
	\arrow["9"{description}, draw=none, from=3, to=4]
	\arrow["4"{description}, draw=none, from=5, to=6]
	\arrow["7"{description}, draw=none, from=6, to=7]
	\arrow["10"{description}, draw=none, from=7, to=8]
\end{tikzcd}} \quad
\, .\)
 \newline
 Consequently, we have:
\[
\adjustbox{scale=0.60}{\begin{tikzcd}
	&& {} & {} \\
	& {} & {} & {} \\
	{} & {} & {} & {} \\
	{} & {} & {} & {}
	\arrow[from=1-3, to=1-4]
	\arrow[""{name=0, anchor=center, inner sep=0}, from=1-3, to=2-3]
	\arrow[""{name=1, anchor=center, inner sep=0}, from=1-4, to=2-4]
	\arrow[from=2-2, to=2-3]
	\arrow[""{name=2, anchor=center, inner sep=0}, from=2-2, to=3-2]
	\arrow[from=2-3, to=2-4]
	\arrow[""{name=3, anchor=center, inner sep=0}, from=2-3, to=3-3]
	\arrow[""{name=4, anchor=center, inner sep=0}, from=2-4, to=3-4]
	\arrow[from=3-1, to=3-2]
	\arrow[""{name=5, anchor=center, inner sep=0}, from=3-1, to=4-1]
	\arrow[from=3-2, to=3-3]
	\arrow[""{name=6, anchor=center, inner sep=0}, from=3-2, to=4-2]
	\arrow[from=3-3, to=3-4]
	\arrow[""{name=7, anchor=center, inner sep=0}, from=3-4, to=4-4]
	\arrow[from=4-1, to=4-2]
	\arrow[from=4-2, to=4-4]
	\arrow["1"{description}, draw=none, from=0, to=1]
	\arrow["2"{description}, draw=none, from=2, to=3]
	\arrow["3"{description}, draw=none, from=3, to=4]
	\arrow["4"{description}, draw=none, from=5, to=6]
	\arrow["5"{description}, draw=none, from=6, to=7]
\end{tikzcd}}
\begin{tikzcd}
	{} & {}
	\arrow["\bd", squiggly, from=1-1, to=1-2]
\end{tikzcd}
\adjustbox{scale=0.60}{\begin{tikzcd}
	&& {} & {} \\
	& {} & {} & {} \\
	{} & {} & {} & {} \\
	{} & {} & {} & {}
	\arrow[from=1-3, to=1-4]
	\arrow[""{name=0, anchor=center, inner sep=0}, from=1-3, to=2-3]
	\arrow[""{name=1, anchor=center, inner sep=0}, from=1-4, to=2-4]
	\arrow[from=2-2, to=2-3]
	\arrow[""{name=2, anchor=center, inner sep=0}, from=2-2, to=3-2]
	\arrow[from=2-3, to=2-4]
	\arrow[""{name=3, anchor=center, inner sep=0}, from=2-3, to=3-3]
	\arrow[""{name=4, anchor=center, inner sep=0}, from=2-4, to=3-4]
	\arrow[from=3-1, to=3-2]
	\arrow[""{name=5, anchor=center, inner sep=0}, from=3-1, to=4-1]
	\arrow[from=3-2, to=3-3]
	\arrow[""{name=6, anchor=center, inner sep=0}, from=3-2, to=4-2]
	\arrow[from=3-3, to=3-4]
	\arrow[""{name=7, anchor=center, inner sep=0}, from=3-3, to=4-3]
	\arrow[""{name=8, anchor=center, inner sep=0}, from=3-4, to=4-4]
	\arrow[from=4-1, to=4-2]
	\arrow[from=4-2, to=4-3]
	\arrow[from=4-3, to=4-4]
	\arrow["1"{description}, draw=none, from=0, to=1]
	\arrow["2"{description}, draw=none, from=2, to=3]
	\arrow["3"{description}, draw=none, from=3, to=4]
	\arrow["4"{description}, draw=none, from=5, to=6]
	\arrow["7"{description}, draw=none, from=6, to=7]
	\arrow["8"{description}, draw=none, from=7, to=8]
\end{tikzcd}}
\begin{tikzcd}
	{} & {}
	\arrow["\bu", squiggly, from=1-1, to=1-2]
\end{tikzcd}
\adjustbox{scale=0.60}{\begin{tikzcd}
	&& {} & {} \\
	& {} & {} & {} \\
	{} & {} & {} & {} \\
	{} & {} & {} & {}
	\arrow[from=1-3, to=1-4]
	\arrow[""{name=0, anchor=center, inner sep=0}, from=1-3, to=2-3]
	\arrow[""{name=1, anchor=center, inner sep=0}, from=1-4, to=2-4]
	\arrow[from=2-2, to=2-3]
	\arrow[""{name=2, anchor=center, inner sep=0}, from=2-2, to=3-2]
	\arrow[from=2-3, to=2-4]
	\arrow[""{name=3, anchor=center, inner sep=0}, from=2-3, to=3-3]
	\arrow[""{name=4, anchor=center, inner sep=0}, from=2-4, to=3-4]
	\arrow[from=3-1, to=3-2]
	\arrow[""{name=5, anchor=center, inner sep=0}, from=3-1, to=4-1]
	\arrow[from=3-2, to=3-3]
	\arrow[""{name=6, anchor=center, inner sep=0}, from=3-2, to=4-2]
	\arrow[from=3-3, to=3-4]
	\arrow[""{name=7, anchor=center, inner sep=0}, from=3-3, to=4-3]
	\arrow[""{name=8, anchor=center, inner sep=0}, from=3-4, to=4-4]
	\arrow[from=4-1, to=4-2]
	\arrow[from=4-2, to=4-3]
	\arrow[from=4-3, to=4-4]
	\arrow["1"{description}, draw=none, from=0, to=1]
	\arrow["2"{description}, draw=none, from=2, to=3]
	\arrow["9"{description}, draw=none, from=3, to=4]
	\arrow["4"{description}, draw=none, from=5, to=6]
	\arrow["7"{description}, draw=none, from=6, to=7]
	\arrow["10"{description}, draw=none, from=7, to=8]
\end{tikzcd}}
\begin{tikzcd}
	{} & {}
	\arrow["\bs", squiggly, from=1-1, to=1-2]
\end{tikzcd}
\adjustbox{scale=0.60}{\begin{tikzcd}
	&& {} & {} \\
	& {} & {} & {} \\
	{} & {} & {} & {} \\
	{} & {} & {} & {}
	\arrow[from=1-3, to=1-4]
	\arrow[""{name=0, anchor=center, inner sep=0}, from=1-3, to=2-3]
	\arrow[""{name=1, anchor=center, inner sep=0}, from=1-4, to=2-4]
	\arrow[from=2-2, to=2-3]
	\arrow[""{name=2, anchor=center, inner sep=0}, from=2-2, to=3-2]
	\arrow[from=2-3, to=2-4]
	\arrow[""{name=3, anchor=center, inner sep=0}, from=2-3, to=3-3]
	\arrow[""{name=4, anchor=center, inner sep=0}, from=2-4, to=3-4]
	\arrow[from=3-1, to=3-2]
	\arrow[""{name=5, anchor=center, inner sep=0}, from=3-1, to=4-1]
	\arrow[from=3-2, to=3-3]
	\arrow[""{name=6, anchor=center, inner sep=0}, from=3-2, to=4-2]
	\arrow[from=3-3, to=3-4]
	\arrow[""{name=7, anchor=center, inner sep=0}, from=3-4, to=4-4]
	\arrow[from=4-1, to=4-2]
	\arrow[from=4-2, to=4-4]
	\arrow["1"{description}, draw=none, from=0, to=1]
	\arrow["2"{description}, draw=none, from=2, to=3]
	\arrow["3"{description}, draw=none, from=3, to=4]
	\arrow["4"{description}, draw=none, from=5, to=6]
	\arrow["5"{description}, draw=none, from=6, to=7]
\end{tikzcd}}
\]
Observe that the basic $\Omega$ 2-cell corresponding to $\mathbf{d}$ may be obtained by applying Rule 4' just to the $\Sigma$-squares 5 and $7\oplus 8$ and composing with the $\Sigma$-square 4. Then, it coincides with the basic $\Omega$ 2-cell obtained by the $\Sigma$-step of type $\mathbf{s}$. Analogously, the $\Omega$ 2-cell corresponding to $\mathbf{u}$ may be obtained by applying Rule 4' to the $\Sigma$-squares $8\odot 3$  and $10\odot 9$   and composing with the $\Sigma$-square 1. Thus, this $\Omega$ 2-cell also coincides with the one obtained by a $\Sigma$-step of type $\mathbf{s}$. Therefore, we conclude that the $\Sigma$-path is equivalent to a $\Sigma$-path consisting of a single $\Sigma$-step of type $\mathbf{s}$, and hence it is just the identity.

 \underline{Cycles of length 4}. We analyse the cases which are not reducible to a cycle of length $\leq 3$ in an obvious way and which are not encompassed by  \cref{pro:adend}. Indeed, between all possible $\Sigma$-cycles $\; \begin{tikzcd}
	{S_1} & {S_2} & {S_3} & {S_4} & {S_1}
	\arrow[squiggly, from=1-1, to=1-2]
	\arrow[squiggly, from=1-2, to=1-3]
	\arrow[squiggly, from=1-3, to=1-4]
	\arrow[squiggly, from=1-4, to=1-5]
\end{tikzcd}$, taking into account \cref{lem:Sigma-steps}, which reduces some of them to a cycle of length $\leq 3$, \cref{pro:adend} and also \cref{rem:cycle}, we conclude that there are just 23 cases to be analysed\footnote{Namely:  $\mathbf{dsds_1}$, $\mathbf{dsdu}$, $\mathbf{dsd_1s}$, $\mathbf{dsd_1s_1}$, $\mathbf{dsd_1u}$, $\mathbf{dsus}$, $\mathbf{dsus_1}$, $\mathbf{ds_1du}$, $\mathbf{ds_1d_1s_1}$, $\mathbf{ds_1d_1u}$, $\mathbf{ds_1us_1}$, $\mathbf{dud_1u}$, $\mathbf{dusu}$, $\mathbf{dus_1u}$, $\mathbf{ud_1sd_1}$, $\mathbf{ud_1s_1d_1}$,
$\mathbf{ud_1us}$,
$\mathbf{ud_1us_1}$,
$\mathbf{usd_1s}$,
$\mathbf{usd_1s_1}$,
 $\mathbf{usus_1}$,
  $\mathbf{us_1d_1s_1}$ and
    $\mathbf{sd_1s_1d_1}$.}. We present the proof for some of these cases. The remaining proofs use analogous arguments.

 \underline{Case $\mathbf{dsds_1}$}.  Arguing in a way similar to case $\mathbf{dus}$, we see that the $\Sigma$-schemes of a $\Sigma$-path  of $\Sigma$-steps $\; \begin{tikzcd}
	{S_1} & {S_2} & {S_3} & {S_4} & {S_1}
	\arrow["\mathbf{d}", squiggly, from=1-1, to=1-2]
	\arrow["\mathbf{s}", squiggly, from=1-2, to=1-3]
	\arrow["\mathbf{d}", squiggly, from=1-3, to=1-4]
	\arrow["{\mathbf{s}_1}", squiggly, from=1-4, to=1-5]
\end{tikzcd}\;$ must be of the form
\[
S_1=\hspace{-1.5mm}\adjustbox{scale=0.60}{\begin{tikzcd}
	&& {} & {} \\
	& {} & {} & {} \\
	& {} && {} \\
	{} & {} && {} \\
	{} & {} && {}
	\arrow[from=1-3, to=1-4]
	\arrow[""{name=0, anchor=center, inner sep=0}, from=1-3, to=2-3]
	\arrow[""{name=1, anchor=center, inner sep=0}, from=1-4, to=2-4]
	\arrow[from=2-2, to=2-3]
	\arrow[""{name=2, anchor=center, inner sep=0}, equals, from=2-2, to=3-2]
	\arrow[from=2-3, to=2-4]
	\arrow[""{name=3, anchor=center, inner sep=0}, from=2-4, to=3-4]
	\arrow[from=3-2, to=3-4]
	\arrow[""{name=4, anchor=center, inner sep=0}, from=3-2, to=4-2]
	\arrow[""{name=5, anchor=center, inner sep=0}, from=3-4, to=4-4]
	\arrow[from=4-1, to=4-2]
	\arrow[""{name=6, anchor=center, inner sep=0}, from=4-1, to=5-1]
	\arrow[from=4-2, to=4-4]
	\arrow[""{name=7, anchor=center, inner sep=0}, from=4-2, to=5-2]
	\arrow[""{name=8, anchor=center, inner sep=0}, from=4-4, to=5-4]
	\arrow[from=5-1, to=5-2]
	\arrow[from=5-2, to=5-4]
	\arrow["1"{description}, draw=none, from=0, to=1]
	\arrow["2"{description}, draw=none, from=2, to=3]
	\arrow["3"{marking, allow upside down}, draw=none, from=4, to=5]
	\arrow["4"{description}, draw=none, from=6, to=7]
	\arrow["5"{description}, draw=none, from=7, to=8]
\end{tikzcd}},
\quad%
S_2=\hspace{-1.5mm}\adjustbox{scale=0.60}{\begin{tikzcd}
	&& {} & {} \\
	& {} & {} & {} \\
	& {} && {} \\
	{} & {} && {} \\
	{} & {} && {}
	\arrow[from=1-3, to=1-4]
	\arrow[""{name=0, anchor=center, inner sep=0}, from=1-3, to=2-3]
	\arrow[""{name=1, anchor=center, inner sep=0}, from=1-4, to=2-4]
	\arrow[from=2-2, to=2-3]
	\arrow[""{name=2, anchor=center, inner sep=0}, equals, from=2-2, to=3-2]
	\arrow[from=2-3, to=2-4]
	\arrow[""{name=3, anchor=center, inner sep=0}, from=2-4, to=3-4]
	\arrow[from=3-2, to=3-4]
	\arrow[""{name=4, anchor=center, inner sep=0}, from=3-2, to=4-2]
	\arrow[""{name=5, anchor=center, inner sep=0}, from=3-4, to=4-4]
	\arrow[from=4-1, to=4-2]
	\arrow[""{name=6, anchor=center, inner sep=0}, from=4-1, to=5-1]
	\arrow[from=4-2, to=4-4]
	\arrow[""{name=7, anchor=center, inner sep=0}, from=4-2, to=5-2]
	\arrow[""{name=8, anchor=center, inner sep=0}, from=4-4, to=5-4]
	\arrow[from=5-1, to=5-2]
	\arrow[from=5-2, to=5-4]
	\arrow["1"{marking, allow upside down}, draw=none, from=0, to=1]
	\arrow["2"{marking, allow upside down}, draw=none, from=2, to=3]
	\arrow["3"{marking, allow upside down}, draw=none, from=4, to=5]
	\arrow["6"{description}, draw=none, from=6, to=7]
	\arrow["7"{marking, allow upside down}, draw=none, from=7, to=8]
\end{tikzcd}},
\quad
S_3=\hspace{-1.5mm}
\adjustbox{scale=0.60}{\begin{tikzcd}
	&& {} & {} \\
	& {} & {} & {} \\
	& {} && {} \\
	{} & {} && {} \\
	{} & {} && {}
	\arrow[from=1-3, to=1-4]
	\arrow[""{name=0, anchor=center, inner sep=0}, from=1-3, to=2-3]
	\arrow[""{name=1, anchor=center, inner sep=0}, from=1-4, to=2-4]
	\arrow[from=2-2, to=2-3]
	\arrow[""{name=2, anchor=center, inner sep=0}, equals, from=2-2, to=3-2]
	\arrow[from=2-3, to=2-4]
	\arrow[""{name=3, anchor=center, inner sep=0}, from=2-4, to=3-4]
	\arrow[from=3-2, to=3-4]
	\arrow[""{name=4, anchor=center, inner sep=0}, from=3-2, to=4-2]
	\arrow[""{name=5, anchor=center, inner sep=0}, from=3-4, to=4-4]
	\arrow[from=4-1, to=4-2]
	\arrow[""{name=6, anchor=center, inner sep=0}, from=4-1, to=5-1]
	\arrow[from=4-2, to=4-4]
	\arrow[""{name=7, anchor=center, inner sep=0}, from=4-2, to=5-2]
	\arrow[""{name=8, anchor=center, inner sep=0}, from=4-4, to=5-4]
	\arrow[from=5-1, to=5-2]
	\arrow[from=5-2, to=5-4]
	\arrow["1"{marking, allow upside down}, draw=none, from=0, to=1]
	\arrow["2"{marking, allow upside down}, draw=none, from=2, to=3]
	\arrow["8"{marking, allow upside down}, draw=none, from=4, to=5]
	\arrow["6"{description}, draw=none, from=6, to=7]
	\arrow["9"{marking, allow upside down}, draw=none, from=7, to=8]
\end{tikzcd}}
\quad%
\text{and}
\quad%
S_4=\hspace{-1.5mm}
\adjustbox{scale=0.60}{\begin{tikzcd}
	&& {} & {} \\
	& {} & {} & {} \\
	& {} && {} \\
	{} & {} && {} \\
	{} & {} && {}
	\arrow[from=1-3, to=1-4]
	\arrow[""{name=0, anchor=center, inner sep=0}, from=1-3, to=2-3]
	\arrow[""{name=1, anchor=center, inner sep=0}, from=1-4, to=2-4]
	\arrow[from=2-2, to=2-3]
	\arrow[""{name=2, anchor=center, inner sep=0}, equals, from=2-2, to=3-2]
	\arrow[from=2-3, to=2-4]
	\arrow[""{name=3, anchor=center, inner sep=0}, from=2-4, to=3-4]
	\arrow[from=3-2, to=3-4]
	\arrow[""{name=4, anchor=center, inner sep=0}, from=3-2, to=4-2]
	\arrow[""{name=5, anchor=center, inner sep=0}, from=3-4, to=4-4]
	\arrow[from=4-1, to=4-2]
	\arrow[""{name=6, anchor=center, inner sep=0}, from=4-1, to=5-1]
	\arrow[from=4-2, to=4-4]
	\arrow[""{name=7, anchor=center, inner sep=0}, from=4-2, to=5-2]
	\arrow[""{name=8, anchor=center, inner sep=0}, from=4-4, to=5-4]
	\arrow[from=5-1, to=5-2]
	\arrow[from=5-2, to=5-4]
	\arrow["1"{marking, allow upside down}, draw=none, from=0, to=1]
	\arrow["2"{marking, allow upside down}, draw=none, from=2, to=3]
	\arrow["8"{marking, allow upside down}, draw=none, from=4, to=5]
	\arrow["4"{description}, draw=none, from=6, to=7]
	\arrow["10"{marking, allow upside down}, draw=none, from=7, to=8]
\end{tikzcd}}.
\]
The $\Sigma$-step $\; \begin{tikzcd}
	{S_4} & {S_1}
	\arrow["{\mathbf{s}_1}", squiggly, from=1-1, to=1-2]
\end{tikzcd}\;$ gives rise to a basic $\Omega$ 2-cell by applying Rule 4' to the $\Sigma$-squares $10\odot 8$ and $5\odot 3$. But this clearly coincides with the basic $\Omega$ 2-cell obtained by applying Rule 4' to the $\Sigma$-squares $10\odot 8\odot2$ and $5\odot 3\odot2$. Thus, this $\Sigma$-step of type $\mathbf{s}_1$ is  also a $\Sigma$-step of type $\mathbf{s}$. Therefore, the whole $\Sigma$-path is equivalent to a $\Sigma$-path given by a sequence $\mathbf{dsds}$, hence it is equivalent to the identity $\Sigma$-path, by \cref{pro:adend}.

  \underline{Case $\mathbf{duds}$}. Arguing in a way similar to case $\mathbf{dus}$, we see that the $\Sigma$-schemes of a cycle of $\Sigma$-steps
  $\begin{tikzcd}
	{S_1} & {S_2} & {S_3} & {S_4} & {S_1}
	\arrow["\bd", squiggly, from=1-1, to=1-2]
	\arrow["\bu", squiggly, from=1-2, to=1-3]
	\arrow["\bd", squiggly, from=1-3, to=1-4]
	\arrow["\bs", squiggly, from=1-4, to=1-5]
\end{tikzcd}$
 must be of the form
 represented by the outside $\Sigma$-path of the diagram below, where $S_1$ is the $\Sg$-scheme placed in the top right.
 The $\Sigma$-squares 11, 12 and 13 are obtained by Square.
 The equivalences of $\Sigma$-paths and $\Sigma$-steps indicated by the symbol $\equiv$ are easily seen --- namely, we use the fact that the juxtaposition of $\Sigma$-steps of the same type is equivalent to a $\Sigma$-step of that type, and  the case $i=\bd$ and $j=\bu$ studied in Proposition \ref{pro:adend}. Consequently, we see that the $\Omega$ 2-cell corresponding to the $\Sigma$-cycle is the identity.

  \[
 \adjustbox{scale=0.70}
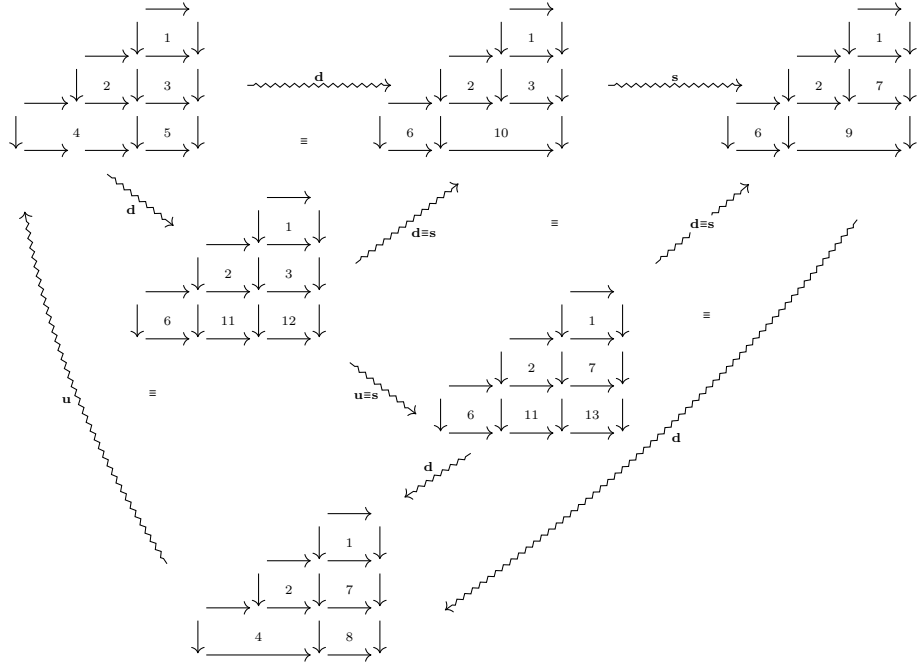
{%
  \begin{tikzcd}
	&& {} & {} &&&&& {} & {} &&&&& {} & {} \\
	& {} & {} & {} &&&& {} & {} & {} &&&& {} & {} & {} \\
	{} & {} & {} & {} &&& {} & {} & {} & {} &&& {} & {} & {} & {} \\
	{} & {} & {} & {} &&& {} & {} & {} & {} &&& {} & {} & {} & {} \\
	&&&& {} & {} \\
	&&& {} & {} & {} \\
	&& {} & {} & {} & {} &&&& {} & {} \\
	&& {} & {} & {} & {} &&& {} & {} & {} &&& {} \\
	&&&&&&& {} & {} & {} & {} \\
	&&&&&&& {} & {} & {} & {} \\
	\\
	&&&&& {} & {} \\
	&&&& {} & {} & {} \\
	&&& {} & {} & {} & {} \\
	&&& {} & {} & {} & {}
	\arrow[from=1-3, to=1-4]
	\arrow[""{name=0, anchor=center, inner sep=0}, from=1-3, to=2-3]
	\arrow[""{name=1, anchor=center, inner sep=0}, from=1-4, to=2-4]
	\arrow[from=1-9, to=1-10]
	\arrow[""{name=2, anchor=center, inner sep=0}, from=1-9, to=2-9]
	\arrow[""{name=3, anchor=center, inner sep=0}, from=1-10, to=2-10]
	\arrow[from=1-15, to=1-16]
	\arrow[""{name=4, anchor=center, inner sep=0}, from=1-15, to=2-15]
	\arrow[""{name=5, anchor=center, inner sep=0}, from=1-16, to=2-16]
	\arrow[from=2-2, to=2-3]
	\arrow[""{name=6, anchor=center, inner sep=0}, from=2-2, to=3-2]
	\arrow[from=2-3, to=2-4]
	\arrow[""{name=7, anchor=center, inner sep=0}, from=2-3, to=3-3]
	\arrow[""{name=8, anchor=center, inner sep=0}, from=2-4, to=3-4]
	\arrow[from=2-8, to=2-9]
	\arrow[""{name=9, anchor=center, inner sep=0}, from=2-8, to=3-8]
	\arrow[from=2-9, to=2-10]
	\arrow[""{name=10, anchor=center, inner sep=0}, from=2-9, to=3-9]
	\arrow[""{name=11, anchor=center, inner sep=0}, from=2-10, to=3-10]
	\arrow[from=2-14, to=2-15]
	\arrow[""{name=12, anchor=center, inner sep=0}, from=2-14, to=3-14]
	\arrow[from=2-15, to=2-16]
	\arrow[""{name=13, anchor=center, inner sep=0}, from=2-15, to=3-15]
	\arrow[""{name=14, anchor=center, inner sep=0}, from=2-16, to=3-16]
	\arrow[from=3-1, to=3-2]
	\arrow[""{name=15, anchor=center, inner sep=0}, from=3-1, to=4-1]
	\arrow[from=3-2, to=3-3]
	\arrow[from=3-3, to=3-4]
	\arrow[""{name=16, anchor=center, inner sep=0}, from=3-3, to=4-3]
	\arrow[""{name=17, anchor=center, inner sep=0}, from=3-4, to=4-4]
	\arrow[from=3-7, to=3-8]
	\arrow[""{name=18, anchor=center, inner sep=0}, from=3-7, to=4-7]
	\arrow[from=3-8, to=3-9]
	\arrow[""{name=19, anchor=center, inner sep=0}, from=3-8, to=4-8]
	\arrow[from=3-9, to=3-10]
	\arrow[""{name=20, anchor=center, inner sep=0}, from=3-10, to=4-10]
	\arrow[from=3-13, to=3-14]
	\arrow[""{name=21, anchor=center, inner sep=0}, from=3-13, to=4-13]
	\arrow[from=3-14, to=3-15]
	\arrow[""{name=22, anchor=center, inner sep=0}, from=3-14, to=4-14]
	\arrow[from=3-15, to=3-16]
	\arrow[""{name=23, anchor=center, inner sep=0}, from=3-16, to=4-16]
	\arrow[from=4-1, to=4-2]
	\arrow[from=4-2, to=4-3]
	\arrow["\bd"', between={0.2}{0.8}, squiggly, from=4-2, to=6-4]
	\arrow[from=4-3, to=4-4]
	\arrow[from=4-7, to=4-8]
	\arrow[from=4-8, to=4-10]
	\arrow[from=4-13, to=4-14]
	\arrow[from=4-14, to=4-16]
	\arrow["\bd", curve={height=-24pt}, between={0.1}{0.9}, squiggly, from=4-16, to=15-7]
	\arrow[""{name=24, anchor=center, inner sep=0}, from=5-5, to=5-6]
	\arrow[""{name=25, anchor=center, inner sep=0}, from=5-5, to=6-5]
	\arrow[""{name=26, anchor=center, inner sep=0}, from=5-6, to=6-6]
	\arrow[from=6-4, to=6-5]
	\arrow[""{name=27, anchor=center, inner sep=0}, from=6-4, to=7-4]
	\arrow[from=6-5, to=6-6]
	\arrow[""{name=28, anchor=center, inner sep=0}, from=6-5, to=7-5]
	\arrow[""{name=29, anchor=center, inner sep=0}, from=6-6, to=7-6]
	\arrow[from=7-3, to=7-4]
	\arrow[""{name=30, anchor=center, inner sep=0}, from=7-3, to=8-3]
	\arrow[from=7-4, to=7-5]
	\arrow[""{name=31, anchor=center, inner sep=0}, from=7-4, to=8-4]
	\arrow[from=7-5, to=7-6]
	\arrow[""{name=32, anchor=center, inner sep=0}, from=7-5, to=8-5]
	\arrow[""{name=33, anchor=center, inner sep=0}, "{{\bd\equiv \bs}}"', between={0.2}{0.8}, squiggly, from=7-6, to=4-9]
	\arrow[""{name=34, anchor=center, inner sep=0}, from=7-6, to=8-6]
	\arrow[from=7-10, to=7-11]
	\arrow[""{name=35, anchor=center, inner sep=0}, from=7-10, to=8-10]
	\arrow[""{name=36, anchor=center, inner sep=0}, "{{\bd\equiv \bs}}"{description}, between={0.2}{0.8}, squiggly, from=7-11, to=4-14]
	\arrow[""{name=37, anchor=center, inner sep=0}, from=7-11, to=8-11]
	\arrow["{{{{{{\equiv}}}}}}"{description}, draw=none, from=7-11, to=8-14]
	\arrow[from=8-3, to=8-4]
	\arrow[from=8-4, to=8-5]
	\arrow[from=8-5, to=8-6]
	\arrow["{{\bu\equiv \bs}}"', between={0.2}{0.8}, squiggly, from=8-6, to=10-8]
	\arrow[from=8-9, to=8-10]
	\arrow[""{name=38, anchor=center, inner sep=0}, from=8-9, to=9-9]
	\arrow[from=8-10, to=8-11]
	\arrow[""{name=39, anchor=center, inner sep=0}, from=8-10, to=9-10]
	\arrow[""{name=40, anchor=center, inner sep=0}, from=8-11, to=9-11]
	\arrow[from=9-8, to=9-9]
	\arrow[""{name=41, anchor=center, inner sep=0}, from=9-8, to=10-8]
	\arrow[from=9-9, to=9-10]
	\arrow[""{name=42, anchor=center, inner sep=0}, from=9-9, to=10-9]
	\arrow[from=9-10, to=9-11]
	\arrow[""{name=43, anchor=center, inner sep=0}, from=9-10, to=10-10]
	\arrow[""{name=44, anchor=center, inner sep=0}, from=9-11, to=10-11]
	\arrow[from=10-8, to=10-9]
	\arrow[from=10-9, to=10-10]
	\arrow["\bd"', between={0.2}{0.8}, squiggly, from=10-9, to=12-7]
	\arrow[from=10-10, to=10-11]
	\arrow[from=12-6, to=12-7]
	\arrow[""{name=45, anchor=center, inner sep=0}, from=12-6, to=13-6]
	\arrow[""{name=46, anchor=center, inner sep=0}, from=12-7, to=13-7]
	\arrow[from=13-5, to=13-6]
	\arrow[""{name=47, anchor=center, inner sep=0}, from=13-5, to=14-5]
	\arrow[from=13-6, to=13-7]
	\arrow[""{name=48, anchor=center, inner sep=0}, from=13-6, to=14-6]
	\arrow[""{name=49, anchor=center, inner sep=0}, from=13-7, to=14-7]
	\arrow[""{name=50, anchor=center, inner sep=0}, "\bu", curve={height=-24pt}, between={0.1}{0.9}, squiggly, from=14-4, to=4-1]
	\arrow[from=14-4, to=14-5]
	\arrow[""{name=51, anchor=center, inner sep=0}, from=14-4, to=15-4]
	\arrow[from=14-5, to=14-6]
	\arrow[from=14-6, to=14-7]
	\arrow[""{name=52, anchor=center, inner sep=0}, from=14-6, to=15-6]
	\arrow[""{name=53, anchor=center, inner sep=0}, from=14-7, to=15-7]
	\arrow[from=15-4, to=15-6]
	\arrow[from=15-6, to=15-7]
	\arrow["1"{description}, draw=none, from=0, to=1]
	\arrow["1"{description}, draw=none, from=2, to=3]
	\arrow["1"{description}, draw=none, from=4, to=5]
	\arrow["2"{description}, draw=none, from=6, to=7]
	\arrow["3"{description}, draw=none, from=7, to=8]
	\arrow[""{name=54, anchor=center, inner sep=0}, "\bd", between={0.2}{0.8}, squiggly, from=8, to=9]
	\arrow["2"{description}, draw=none, from=9, to=10]
	\arrow["3"{description}, draw=none, from=10, to=11]
	\arrow["\bs", between={0.2}{0.8}, squiggly, from=11, to=12]
	\arrow["2"{description}, draw=none, from=12, to=13]
	\arrow["7"{description}, draw=none, from=13, to=14]
	\arrow["4"{description}, draw=none, from=15, to=16]
	\arrow["5"{description}, draw=none, from=16, to=17]
	\arrow["6"{description}, draw=none, from=18, to=19]
	\arrow["10"{description}, draw=none, from=19, to=20]
	\arrow["6"{description}, draw=none, from=21, to=22]
	\arrow["9"{description}, draw=none, from=22, to=23]
	\arrow["1"{description}, draw=none, from=25, to=26]
	\arrow["2"{description}, draw=none, from=27, to=28]
	\arrow["3"{description}, draw=none, from=28, to=29]
	\arrow["6"{description}, draw=none, from=30, to=31]
	\arrow["11"{description}, draw=none, from=31, to=32]
	\arrow["12"{description}, draw=none, from=32, to=34]
	\arrow["\equiv"{description}, draw=none, from=33, to=36]
	\arrow["1"{description}, draw=none, from=35, to=37]
	\arrow["2"{description}, draw=none, from=38, to=39]
	\arrow["7"{description}, draw=none, from=39, to=40]
	\arrow["6"{description}, draw=none, from=41, to=42]
	\arrow["11"{description}, draw=none, from=42, to=43]
	\arrow["13"{description}, draw=none, from=43, to=44]
	\arrow["1"{description}, draw=none, from=45, to=46]
	\arrow["2"{description}, draw=none, from=47, to=48]
	\arrow["7"{description}, draw=none, from=48, to=49]
	\arrow["{{{{{{\equiv}}}}}}"{description, pos=0.4}, shift right=5, draw=none, from=50, to=8-5]
	\arrow["4"{description}, draw=none, from=51, to=52]
	\arrow["8"{description}, draw=none, from=52, to=53]
	\arrow["{{{{{{\equiv}}}}}}"{description}, draw=none, from=54, to=24]
\end{tikzcd}
}%
\]

 \underline{Case $\mathbf{dud}_1\mathbf{u}$}. Arguing in a similar way as above we see that our cycle is of the form
 \begin{equation}\label{eq:dud1u}
 \xymatrix{S_1\ar@{~>}[r]^{\mathbf{d}}&S_2\ar@{~>}[r]^{\mathbf{u}}&S_3\ar@{~>}[r]^{\mathbf{d}_1}&S_4\ar@{~>}[r]^{\mathbf{u}}&S_1}
 \end{equation}
 with
\[
S_1=
\adjustbox{scale=0.60}{\begin{tikzcd}
	&& {} & {} \\
	& {} & {} \\
	{} & {} & {} & {} \\
	{} && {} \\
	{} && {} & {}
	\arrow[from=1-3, to=1-4]
	\arrow[from=1-3, to=2-3]
	\arrow[""{name=0, anchor=center, inner sep=0}, from=1-4, to=3-4]
	\arrow[from=2-2, to=2-3]
	\arrow[""{name=1, anchor=center, inner sep=0}, from=2-2, to=3-2]
	\arrow[""{name=2, anchor=center, inner sep=0}, from=2-3, to=3-3]
	\arrow[from=3-1, to=3-2]
	\arrow[""{name=3, anchor=center, inner sep=0}, from=3-1, to=4-1]
	\arrow[from=3-2, to=3-3]
	\arrow[from=3-3, to=3-4]
	\arrow[""{name=4, anchor=center, inner sep=0}, from=3-3, to=4-3]
	\arrow[""{name=5, anchor=center, inner sep=0}, from=3-4, to=5-4]
	\arrow[from=4-1, to=4-3]
	\arrow[""{name=6, anchor=center, inner sep=0}, equals, from=4-1, to=5-1]
	\arrow[""{name=7, anchor=center, inner sep=0}, from=4-3, to=5-3]
	\arrow[from=5-1, to=5-3]
	\arrow[from=5-3, to=5-4]
	\arrow["1"{description}, draw=none, from=1, to=2]
	\arrow["2"{description}, draw=none, from=2-3, to=0]
	\arrow["3"{description}, draw=none, from=3, to=4]
	\arrow["4"{description}, draw=none, from=6, to=7]
	\arrow["5"{description}, draw=none, from=4-3, to=5]
\end{tikzcd}},
\quad
S_2=
\adjustbox{scale=0.60}{\begin{tikzcd}
	&& {} & {} \\
	& {} & {} \\
	{} & {} & {} & {} \\
	{} && {} & {} \\
	{} && {} & {}
	\arrow[from=1-3, to=1-4]
	\arrow[from=1-3, to=2-3]
	\arrow[""{name=0, anchor=center, inner sep=0}, from=1-4, to=3-4]
	\arrow[from=2-2, to=2-3]
	\arrow[""{name=1, anchor=center, inner sep=0}, from=2-2, to=3-2]
	\arrow[""{name=2, anchor=center, inner sep=0}, from=2-3, to=3-3]
	\arrow[from=3-1, to=3-2]
	\arrow[""{name=3, anchor=center, inner sep=0}, from=3-1, to=4-1]
	\arrow[from=3-2, to=3-3]
	\arrow[from=3-3, to=3-4]
	\arrow[""{name=4, anchor=center, inner sep=0}, from=3-3, to=4-3]
	\arrow[from=3-4, to=5-4]
	\arrow[from=4-1, to=4-3]
	\arrow[""{name=5, anchor=center, inner sep=0}, equals, from=4-1, to=5-1]
	\arrow["7"{description}, draw=none, from=4-3, to=4-4]
	\arrow[""{name=6, anchor=center, inner sep=0}, from=4-3, to=5-3]
	\arrow[from=5-1, to=5-3]
	\arrow[from=5-3, to=5-4]
	\arrow["1"{description}, draw=none, from=1, to=2]
	\arrow["2"{description}, draw=none, from=2-3, to=0]
	\arrow["3"{description}, draw=none, from=3, to=4]
	\arrow["8"{description}, draw=none, from=5, to=6]
\end{tikzcd}},
\quad
S_3=
\adjustbox{scale=0.60}{\begin{tikzcd}
	&& {} & {} \\
	& {} & {} \\
	{} & {} & {} \\
	{} && {} & {} \\
	{} && {} & {}
	\arrow[from=1-3, to=1-4]
	\arrow[from=1-3, to=2-3]
	\arrow[""{name=0, anchor=center, inner sep=0}, from=1-4, to=4-4]
	\arrow[from=2-2, to=2-3]
	\arrow[""{name=1, anchor=center, inner sep=0}, from=2-2, to=3-2]
	\arrow[""{name=2, anchor=center, inner sep=0}, from=2-3, to=3-3]
	\arrow[from=3-1, to=3-2]
	\arrow[""{name=3, anchor=center, inner sep=0}, from=3-1, to=4-1]
	\arrow[from=3-2, to=3-3]
	\arrow[""{name=4, anchor=center, inner sep=0}, from=3-3, to=4-3]
	\arrow[from=4-1, to=4-3]
	\arrow[""{name=5, anchor=center, inner sep=0}, equals, from=4-1, to=5-1]
	\arrow[from=4-3, to=4-4]
	\arrow[""{name=6, anchor=center, inner sep=0}, from=4-3, to=5-3]
	\arrow[""{name=7, anchor=center, inner sep=0}, from=4-4, to=5-4]
	\arrow[from=5-1, to=5-3]
	\arrow[from=5-3, to=5-4]
	\arrow["1"{description}, draw=none, from=1, to=2]
	\arrow["6"{description}, draw=none, from=2, to=0]
	\arrow["3"{description}, draw=none, from=3, to=4]
	\arrow["8"{description}, draw=none, from=5, to=6]
	\arrow["9"{description}, draw=none, from=6, to=7]
\end{tikzcd}}
\quad
\text{and}
\quad
S_4=
\adjustbox{scale=0.60}{
\begin{tikzcd}
	&& {} & {} \\
	& {} & {} \\
	{} & {} & {} \\
	{} && {} & {} \\
	{} && {} & {}
	\arrow[from=1-3, to=1-4]
	\arrow[from=1-3, to=2-3]
	\arrow[""{name=0, anchor=center, inner sep=0}, from=1-4, to=4-4]
	\arrow[from=2-2, to=2-3]
	\arrow[""{name=1, anchor=center, inner sep=0}, from=2-2, to=3-2]
	\arrow[""{name=2, anchor=center, inner sep=0}, from=2-3, to=3-3]
	\arrow[from=3-1, to=3-2]
	\arrow[""{name=3, anchor=center, inner sep=0}, from=3-1, to=4-1]
	\arrow[from=3-2, to=3-3]
	\arrow[""{name=4, anchor=center, inner sep=0}, from=3-3, to=4-3]
	\arrow[from=4-1, to=4-3]
	\arrow[""{name=5, anchor=center, inner sep=0}, equals, from=4-1, to=5-1]
	\arrow[from=4-3, to=4-4]
	\arrow[""{name=6, anchor=center, inner sep=0}, from=4-3, to=5-3]
	\arrow[""{name=7, anchor=center, inner sep=0}, from=4-4, to=5-4]
	\arrow[from=5-1, to=5-3]
	\arrow[from=5-3, to=5-4]
	\arrow["1"{description}, draw=none, from=1, to=2]
	\arrow["6"{description}, draw=none, from=2, to=0]
	\arrow["3"{description}, draw=none, from=3, to=4]
	\arrow["4"{description}, draw=none, from=5, to=6]
	\arrow["10"{description}, draw=none, from=6, to=7]
\end{tikzcd}}.
\]
Let
\[
R_1=
\adjustbox{scale=0.60}{\begin{tikzcd}
	&& {} & {} \\
	& {} & {} \\
	{} & {} & {} & {} \\
	{} && {} & {} \\
	{} && {} & {}
	\arrow[from=1-3, to=1-4]
	\arrow[from=1-3, to=2-3]
	\arrow[""{name=0, anchor=center, inner sep=0}, from=1-4, to=3-4]
	\arrow[from=2-2, to=2-3]
	\arrow[""{name=1, anchor=center, inner sep=0}, from=2-2, to=3-2]
	\arrow[""{name=2, anchor=center, inner sep=0}, from=2-3, to=3-3]
	\arrow[from=3-1, to=3-2]
	\arrow[""{name=3, anchor=center, inner sep=0}, from=3-1, to=4-1]
	\arrow[from=3-2, to=3-3]
	\arrow[from=3-3, to=3-4]
	\arrow[""{name=4, anchor=center, inner sep=0}, from=3-3, to=4-3]
	\arrow[""{name=5, anchor=center, inner sep=0}, from=3-4, to=4-4]
	\arrow[from=4-1, to=4-3]
	\arrow[""{name=6, anchor=center, inner sep=0}, equals, from=4-1, to=5-1]
	\arrow[from=4-3, to=4-4]
	\arrow[""{name=7, anchor=center, inner sep=0}, from=4-3, to=5-3]
	\arrow[""{name=8, anchor=center, inner sep=0}, from=4-4, to=5-4]
	\arrow[from=5-1, to=5-3]
	\arrow[from=5-3, to=5-4]
	\arrow["1"{description}, draw=none, from=1, to=2]
	\arrow["2"{description}, draw=none, from=2-3, to=0]
	\arrow["3"{description}, draw=none, from=3, to=4]
	\arrow["11"{description}, draw=none, from=4, to=5]
	\arrow["8"{description}, draw=none, from=6, to=7]
	\arrow["12"{description}, draw=none, from=7, to=8]
\end{tikzcd}} \quad
\text{and}
\quad
R_2=
\adjustbox{scale=0.60}{\begin{tikzcd}
	&& {} & {} \\
	& {} & {} \\
	{} & {} & {} & {} \\
	{} && {} & {} \\
	{} && {} & {}
	\arrow[from=1-3, to=1-4]
	\arrow[from=1-3, to=2-3]
	\arrow[""{name=0, anchor=center, inner sep=0}, from=1-4, to=3-4]
	\arrow[from=2-2, to=2-3]
	\arrow[""{name=1, anchor=center, inner sep=0}, from=2-2, to=3-2]
	\arrow[""{name=2, anchor=center, inner sep=0}, from=2-3, to=3-3]
	\arrow[from=3-1, to=3-2]
	\arrow[""{name=3, anchor=center, inner sep=0}, from=3-1, to=4-1]
	\arrow[from=3-2, to=3-3]
	\arrow[from=3-3, to=3-4]
	\arrow[""{name=4, anchor=center, inner sep=0}, from=3-3, to=4-3]
	\arrow[""{name=5, anchor=center, inner sep=0}, from=3-4, to=4-4]
	\arrow[from=4-1, to=4-3]
	\arrow[""{name=6, anchor=center, inner sep=0}, equals, from=4-1, to=5-1]
	\arrow[from=4-3, to=4-4]
	\arrow[""{name=7, anchor=center, inner sep=0}, from=4-3, to=5-3]
	\arrow[""{name=8, anchor=center, inner sep=0}, from=4-4, to=5-4]
	\arrow[from=5-1, to=5-3]
	\arrow[from=5-3, to=5-4]
	\arrow["1"{description}, draw=none, from=1, to=2]
	\arrow["2"{description}, draw=none, from=2-3, to=0]
	\arrow["3"{description}, draw=none, from=3, to=4]
	\arrow["11"{description}, draw=none, from=4, to=5]
	\arrow["4"{description}, draw=none, from=6, to=7]
	\arrow["13"{description}, draw=none, from=7, to=8]
\end{tikzcd}}\, .
\]
Then, using  \cref{lem:Sigma-steps} and \cref{pro:adend}, we see the indicated equivalences of $\Sg$-paths:
\[\begin{tikzcd}
	&&& {S_1} \\
	{S_2} && {R_1} && {R_2} && {S_4} \\
	&&& {S_3}
	\arrow["\bu", between={0}{0.9}, squiggly, from=1-4, to=2-7]
	\arrow["\bd", squiggly, from=2-1, to=1-4]
	\arrow["\bu"{description}, squiggly, from=2-1, to=2-3]
	\arrow[""{name=0, anchor=center, inner sep=0}, "\bu"', squiggly, from=2-1, to=3-4]
	\arrow[""{name=1, anchor=center, inner sep=0}, "{{{\bd\equiv {\bd}_1}}}", squiggly, from=2-3, to=2-5]
	\arrow["\bu"{description}, between={0}{0.9}, squiggly, from=2-5, to=2-7]
	\arrow[""{name=2, anchor=center, inner sep=0}, "{{{{\bd}_1}}}"', between={0}{0.9}, squiggly, from=3-4, to=2-7]
	\arrow["\equiv"{description, pos=0.3}, draw=none, from=1-4, to=1]
	\arrow["\equiv"{description}, draw=none, from=0, to=2]
\end{tikzcd}\]
showing that \eqref{eq:dud1u} corresponds to an identity 2-cell.

  \underline{Case $\mathbf{udus}$}. Analogously, we show this case by noting that the cycle $\mathbf{udus}$ must be as described by the outside part of the diagram below and observing that the indicated equivalences between $\Sigma$-paths hold.
  \[
  \adjustbox{scale=0.70}{\begin{tikzcd}
	&& {} & {} &&&&& {} & {} &&&&& {} & {} \\
	& {} & {} &&&&& {} & {} & {} &&&& {} & {} & {} \\
	{} & {} & {} & {} &&& {} & {} & {} &&&& {} & {} & {} \\
	{} & {} & {} & {} &&& {} & {} & {} & {} &&& {} & {} & {} & {} \\
	&&&& {} & {} \\
	&&& {} & {} & {} \\
	&& {} & {} & {} & {} &&&& {} & {} \\
	&& {} & {} & {} & {} &&& {} & {} & {} &&& {} \\
	&&&&&&& {} & {} & {} & {} \\
	&&&&&&& {} & {} & {} & {} \\
	\\
	&&&&& {} & {} \\
	&&&& {} & {} \\
	&&& {} & {} & {} & {} \\
	&&& {} & {} & {} & {}
	\arrow[from=1-3, to=1-4]
	\arrow[from=1-3, to=2-3]
	\arrow[""{name=0, anchor=center, inner sep=0}, from=1-4, to=3-4]
	\arrow[from=1-9, to=1-10]
	\arrow[""{name=1, anchor=center, inner sep=0}, from=1-9, to=2-9]
	\arrow[""{name=2, anchor=center, inner sep=0}, from=1-10, to=2-10]
	\arrow[from=1-15, to=1-16]
	\arrow[""{name=3, anchor=center, inner sep=0}, from=1-15, to=2-15]
	\arrow[""{name=4, anchor=center, inner sep=0}, from=1-16, to=2-16]
	\arrow[from=2-2, to=2-3]
	\arrow[""{name=5, anchor=center, inner sep=0}, from=2-2, to=3-2]
	\arrow[""{name=6, anchor=center, inner sep=0}, from=2-3, to=3-3]
	\arrow[from=2-8, to=2-9]
	\arrow[""{name=7, anchor=center, inner sep=0}, from=2-8, to=3-8]
	\arrow[from=2-9, to=2-10]
	\arrow[""{name=8, anchor=center, inner sep=0}, from=2-9, to=3-9]
	\arrow[""{name=9, anchor=center, inner sep=0}, from=2-10, to=4-10]
	\arrow[from=2-14, to=2-15]
	\arrow[""{name=10, anchor=center, inner sep=0}, from=2-14, to=3-14]
	\arrow[from=2-15, to=2-16]
	\arrow[""{name=11, anchor=center, inner sep=0}, from=2-15, to=3-15]
	\arrow[""{name=12, anchor=center, inner sep=0}, from=2-16, to=4-16]
	\arrow[from=3-1, to=3-2]
	\arrow[""{name=13, anchor=center, inner sep=0}, from=3-1, to=4-1]
	\arrow[from=3-2, to=3-3]
	\arrow[""{name=14, anchor=center, inner sep=0}, from=3-2, to=4-2]
	\arrow[from=3-3, to=3-4]
	\arrow[""{name=15, anchor=center, inner sep=0}, from=3-3, to=4-3]
	\arrow[""{name=16, anchor=center, inner sep=0}, from=3-4, to=4-4]
	\arrow[from=3-7, to=3-8]
	\arrow[""{name=17, anchor=center, inner sep=0}, from=3-7, to=4-7]
	\arrow[from=3-8, to=3-9]
	\arrow[""{name=18, anchor=center, inner sep=0}, from=3-8, to=4-8]
	\arrow[""{name=19, anchor=center, inner sep=0}, from=3-9, to=4-9]
	\arrow[from=3-13, to=3-14]
	\arrow[""{name=20, anchor=center, inner sep=0}, from=3-13, to=4-13]
	\arrow[from=3-14, to=3-15]
	\arrow[""{name=21, anchor=center, inner sep=0}, from=3-14, to=4-14]
	\arrow[""{name=22, anchor=center, inner sep=0}, from=3-15, to=4-15]
	\arrow[from=4-1, to=4-2]
	\arrow[from=4-2, to=4-3]
	\arrow["\bu"', between={0.2}{0.8}, squiggly, from=4-2, to=6-4]
	\arrow[from=4-3, to=4-4]
	\arrow[from=4-7, to=4-8]
	\arrow[from=4-8, to=4-9]
	\arrow[from=4-9, to=4-10]
	\arrow[from=4-13, to=4-14]
	\arrow[from=4-14, to=4-15]
	\arrow[from=4-15, to=4-16]
	\arrow["\bu", curve={height=-30pt}, between={0.1}{0.9}, squiggly, from=4-16, to=15-7]
	\arrow[""{name=23, anchor=center, inner sep=0}, from=5-5, to=5-6]
	\arrow[""{name=24, anchor=center, inner sep=0}, from=5-5, to=6-5]
	\arrow[""{name=25, anchor=center, inner sep=0}, from=5-6, to=6-6]
	\arrow[from=6-4, to=6-5]
	\arrow[""{name=26, anchor=center, inner sep=0}, from=6-4, to=7-4]
	\arrow[from=6-5, to=6-6]
	\arrow[""{name=27, anchor=center, inner sep=0}, from=6-5, to=7-5]
	\arrow[""{name=28, anchor=center, inner sep=0}, from=6-6, to=7-6]
	\arrow[from=7-3, to=7-4]
	\arrow[""{name=29, anchor=center, inner sep=0}, from=7-3, to=8-3]
	\arrow[from=7-4, to=7-5]
	\arrow[""{name=30, anchor=center, inner sep=0}, from=7-4, to=8-4]
	\arrow[from=7-5, to=7-6]
	\arrow[""{name=31, anchor=center, inner sep=0}, from=7-5, to=8-5]
	\arrow[""{name=32, anchor=center, inner sep=0}, "{{{\bu\equiv \bs}}}"', between={0.2}{0.8}, squiggly, from=7-6, to=4-9]
	\arrow[""{name=33, anchor=center, inner sep=0}, from=7-6, to=8-6]
	\arrow[from=7-10, to=7-11]
	\arrow[""{name=34, anchor=center, inner sep=0}, from=7-10, to=8-10]
	\arrow[""{name=35, anchor=center, inner sep=0}, "{{{\bu\equiv \bs}}}"{description}, between={0.1}{0.8}, squiggly, from=7-11, to=4-14]
	\arrow[""{name=36, anchor=center, inner sep=0}, from=7-11, to=8-11]
	\arrow["{{{\equiv}}}"', draw=none, from=7-11, to=8-14]
	\arrow[from=8-3, to=8-4]
	\arrow[from=8-4, to=8-5]
	\arrow[from=8-5, to=8-6]
	\arrow["{{{\bd\equiv \bs}}}"', between={0.2}{0.8}, squiggly, from=8-6, to=10-8]
	\arrow[from=8-9, to=8-10]
	\arrow[""{name=37, anchor=center, inner sep=0}, from=8-9, to=9-9]
	\arrow[from=8-10, to=8-11]
	\arrow[""{name=38, anchor=center, inner sep=0}, from=8-10, to=9-10]
	\arrow[""{name=39, anchor=center, inner sep=0}, from=8-11, to=9-11]
	\arrow[from=9-8, to=9-9]
	\arrow[""{name=40, anchor=center, inner sep=0}, from=9-8, to=10-8]
	\arrow[from=9-9, to=9-10]
	\arrow[""{name=41, anchor=center, inner sep=0}, from=9-9, to=10-9]
	\arrow[from=9-10, to=9-11]
	\arrow[""{name=42, anchor=center, inner sep=0}, from=9-10, to=10-10]
	\arrow[""{name=43, anchor=center, inner sep=0}, from=9-11, to=10-11]
	\arrow[from=10-8, to=10-9]
	\arrow[from=10-9, to=10-10]
	\arrow["\bu"', between={0.2}{0.8}, squiggly, from=10-9, to=12-7]
	\arrow[from=10-10, to=10-11]
	\arrow[from=12-6, to=12-7]
	\arrow[from=12-6, to=13-6]
	\arrow[""{name=44, anchor=center, inner sep=0}, from=12-7, to=14-7]
	\arrow[from=13-5, to=13-6]
	\arrow[""{name=45, anchor=center, inner sep=0}, from=13-5, to=14-5]
	\arrow[""{name=46, anchor=center, inner sep=0}, from=13-6, to=14-6]
	\arrow[""{name=47, anchor=center, inner sep=0}, "\bd", curve={height=-24pt}, between={0.1}{0.9}, squiggly, from=14-4, to=4-1]
	\arrow[from=14-4, to=14-5]
	\arrow[""{name=48, anchor=center, inner sep=0}, from=14-4, to=15-4]
	\arrow[from=14-5, to=14-6]
	\arrow[""{name=49, anchor=center, inner sep=0}, from=14-5, to=15-5]
	\arrow[from=14-6, to=14-7]
	\arrow[""{name=50, anchor=center, inner sep=0}, from=14-6, to=15-6]
	\arrow[""{name=51, anchor=center, inner sep=0}, from=14-7, to=15-7]
	\arrow[from=15-4, to=15-5]
	\arrow[from=15-5, to=15-6]
	\arrow[from=15-6, to=15-7]
	\arrow[""{name=52, anchor=center, inner sep=0}, "\bu", between={0.2}{0.8}, squiggly, from=0, to=2-8]
	\arrow["10"{description}, draw=none, from=1, to=2]
	\arrow["10"{description}, draw=none, from=3, to=4]
	\arrow["1"{description}, draw=none, from=5, to=6]
	\arrow["2"{description}, draw=none, from=2-3, to=0]
	\arrow["1"{description}, draw=none, from=7, to=8]
	\arrow["\bs", between={0.2}{0.8}, squiggly, from=9, to=3-13]
	\arrow["1"{description}, draw=none, from=10, to=11]
	\arrow["5"{description}, draw=none, from=13, to=14]
	\arrow["9"{description}, draw=none, from=14, to=15]
	\arrow["4"{description}, draw=none, from=15, to=16]
	\arrow["5"{description}, draw=none, from=17, to=18]
	\arrow["9"{description}, draw=none, from=18, to=19]
	\arrow["11"{description}, draw=none, from=3-9, to=9]
	\arrow["5"{description}, draw=none, from=20, to=21]
	\arrow["6"{description}, draw=none, from=21, to=22]
	\arrow["12"{description}, draw=none, from=3-15, to=12]
	\arrow["10"{description}, draw=none, from=24, to=25]
	\arrow["1"{description}, draw=none, from=26, to=27]
	\arrow["13"{description}, draw=none, from=27, to=28]
	\arrow["5"{description}, draw=none, from=29, to=30]
	\arrow["9"{description}, draw=none, from=30, to=31]
	\arrow["14"{description}, draw=none, from=31, to=33]
	\arrow["\equiv"', draw=none, from=32, to=35]
	\arrow["10"{description}, draw=none, from=34, to=36]
	\arrow["1"{description}, draw=none, from=37, to=38]
	\arrow["13"{description}, draw=none, from=38, to=39]
	\arrow["5"{description}, draw=none, from=40, to=41]
	\arrow["6"{description}, draw=none, from=41, to=42]
	\arrow["16"{description}, draw=none, from=42, to=43]
	\arrow["1"{description}, draw=none, from=45, to=46]
	\arrow["2"{description}, draw=none, from=13-6, to=44]
	\arrow["\equiv"'{pos=0.6}, shift right=5, draw=none, from=47, to=8-5]
	\arrow["5"{description}, draw=none, from=48, to=49]
	\arrow["6"{description}, draw=none, from=49, to=50]
	\arrow["7"{description}, draw=none, from=50, to=51]
	\arrow["\equiv"{description}, draw=none, from=52, to=23]
\end{tikzcd}}
\]

\underline{Case $\bu{\bs}_1\bu\bs$}. In order to have the sequence of $\Sg$-steps $\xymatrix{S_1\ar@{~>}[r]^{\bu}&S_2\ar@{~>}[r]^{{\bs}_1}&S_3\ar@{~>}[r]^{\bu}&S_4\ar@{~>}[r]^{\bs}&S_1}$, the $\Sg$-schemes must be of the form
\[
S_1=\hspace{-1.5mm}\adjustbox{scale=0.60}{\begin{tikzcd}
	&& {} & {} \\
	& {} & {} & {} \\
	& {} & {} \\
	{} & {} \\
	{} & {} & {} & {}
	\arrow[from=1-3, to=1-4]
	\arrow[""{name=0, anchor=center, inner sep=0}, from=1-3, to=2-3]
	\arrow[""{name=1, anchor=center, inner sep=0}, from=1-4, to=2-4]
	\arrow[from=2-2, to=2-3]
	\arrow[""{name=2, anchor=center, inner sep=0}, equals, from=2-2, to=3-2]
	\arrow[from=2-3, to=2-4]
	\arrow[""{name=3, anchor=center, inner sep=0}, from=2-3, to=3-3]
	\arrow[""{name=4, anchor=center, inner sep=0}, draw=none, from=2-3, to=5-3]
	\arrow[""{name=5, anchor=center, inner sep=0}, from=2-4, to=5-4]
	\arrow[from=3-2, to=3-3]
	\arrow[from=3-2, to=4-2]
	\arrow[""{name=6, anchor=center, inner sep=0}, from=3-3, to=5-3]
	\arrow[from=4-1, to=4-2]
	\arrow[""{name=7, anchor=center, inner sep=0}, from=4-1, to=5-1]
	\arrow[""{name=8, anchor=center, inner sep=0}, from=4-2, to=5-2]
	\arrow[from=5-1, to=5-2]
	\arrow[from=5-2, to=5-3]
	\arrow[from=5-3, to=5-4]
	\arrow["1"{description}, draw=none, from=0, to=1]
	\arrow["2"{description}, draw=none, from=2, to=3]
	\arrow["3"{description}, draw=none, from=4, to=5]
	\arrow["4"{description}, draw=none, from=7, to=8]
	\arrow["5"{description}, draw=none, from=4-2, to=6]
\end{tikzcd}},
\quad%
S_2=\hspace{-1.5mm}\adjustbox{scale=0.60}{\begin{tikzcd}
	&& {} & {} \\
	& {} & {} \\
	& {} & {} & {} \\
	{} & {} \\
	{} & {} & {} & {}
	\arrow[from=1-3, to=1-4]
	\arrow[from=1-3, to=2-3]
	\arrow[""{name=0, anchor=center, inner sep=0}, from=1-4, to=3-4]
	\arrow[from=2-2, to=2-3]
	\arrow[""{name=1, anchor=center, inner sep=0}, equals, from=2-2, to=3-2]
	\arrow[""{name=2, anchor=center, inner sep=0}, from=2-3, to=3-3]
	\arrow[from=3-2, to=3-3]
	\arrow[from=3-2, to=4-2]
	\arrow[from=3-3, to=3-4]
	\arrow[""{name=3, anchor=center, inner sep=0}, from=3-3, to=5-3]
	\arrow[""{name=4, anchor=center, inner sep=0}, from=3-4, to=5-4]
	\arrow[from=4-1, to=4-2]
	\arrow[""{name=5, anchor=center, inner sep=0}, from=4-1, to=5-1]
	\arrow[""{name=6, anchor=center, inner sep=0}, from=4-2, to=5-2]
	\arrow[from=5-1, to=5-2]
	\arrow[from=5-2, to=5-3]
	\arrow[from=5-3, to=5-4]
	\arrow["2"{description}, draw=none, from=1, to=2]
	\arrow["6"{description}, draw=none, from=2-3, to=0]
	\arrow["7"{description}, draw=none, from=3, to=4]
	\arrow["4"{description}, draw=none, from=5, to=6]
	\arrow["5"{description}, draw=none, from=4-2, to=3]
\end{tikzcd}},
\quad
S_3=\hspace{-1.5mm}
\adjustbox{scale=0.60}{\begin{tikzcd}
	&& {} & {} \\
	& {} & {} \\
	& {} & {} & {} \\
	{} & {} \\
	{} & {} & {} & {}
	\arrow[from=1-3, to=1-4]
	\arrow[from=1-3, to=2-3]
	\arrow[""{name=0, anchor=center, inner sep=0}, from=1-4, to=3-4]
	\arrow[from=2-2, to=2-3]
	\arrow[""{name=1, anchor=center, inner sep=0}, equals, from=2-2, to=3-2]
	\arrow[""{name=2, anchor=center, inner sep=0}, from=2-3, to=3-3]
	\arrow[from=3-2, to=3-3]
	\arrow[from=3-2, to=4-2]
	\arrow[from=3-3, to=3-4]
	\arrow[""{name=3, anchor=center, inner sep=0}, from=3-3, to=5-3]
	\arrow[""{name=4, anchor=center, inner sep=0}, from=3-4, to=5-4]
	\arrow[from=4-1, to=4-2]
	\arrow[""{name=5, anchor=center, inner sep=0}, from=4-1, to=5-1]
	\arrow[""{name=6, anchor=center, inner sep=0}, from=4-2, to=5-2]
	\arrow[from=5-1, to=5-2]
	\arrow[from=5-2, to=5-3]
	\arrow[from=5-3, to=5-4]
	\arrow["2"{description}, draw=none, from=1, to=2]
	\arrow["6"{description}, draw=none, from=2-3, to=0]
	\arrow["9"{description}, draw=none, from=3, to=4]
	\arrow["4"{description}, draw=none, from=5, to=6]
	\arrow["8"{description}, draw=none, from=4-2, to=3]
\end{tikzcd}}
\quad%
\text{and}
\quad%
S_4=\hspace{-1.5mm}
\adjustbox{scale=0.60}{\begin{tikzcd}
	&& {} & {} \\
	& {} & {} & {} \\
	& {} & {} \\
	{} & {} \\
	{} & {} & {} & {}
	\arrow[from=1-3, to=1-4]
	\arrow[""{name=0, anchor=center, inner sep=0}, from=1-3, to=2-3]
	\arrow[""{name=1, anchor=center, inner sep=0}, from=1-4, to=2-4]
	\arrow[from=2-2, to=2-3]
	\arrow[""{name=2, anchor=center, inner sep=0}, equals, from=2-2, to=3-2]
	\arrow[from=2-3, to=2-4]
	\arrow[""{name=3, anchor=center, inner sep=0}, from=2-3, to=3-3]
	\arrow[""{name=4, anchor=center, inner sep=0}, from=2-4, to=5-4]
	\arrow[from=3-2, to=3-3]
	\arrow[from=3-2, to=4-2]
	\arrow[""{name=5, anchor=center, inner sep=0}, from=3-3, to=5-3]
	\arrow[from=4-1, to=4-2]
	\arrow[""{name=6, anchor=center, inner sep=0}, from=4-1, to=5-1]
	\arrow[""{name=7, anchor=center, inner sep=0}, from=4-2, to=5-2]
	\arrow[from=5-1, to=5-2]
	\arrow[from=5-2, to=5-3]
	\arrow[from=5-3, to=5-4]
	\arrow["1"{description}, draw=none, from=0, to=1]
	\arrow["2"{description}, draw=none, from=2, to=3]
	\arrow["10"{description}, draw=none, from=3-3, to=4]
	\arrow["4"{description}, draw=none, from=6, to=7]
	\arrow["8"{description}, draw=none, from=4-2, to=5]
\end{tikzcd}}.
\]
Let
\[
R_1=\hspace{-1.5mm}
\adjustbox{scale=0.60}{\begin{tikzcd}
	&& {} & {} \\
	& {} & {} & {} \\
	& {} & {} & {} \\
	{} & {} \\
	{} & {} & {} & {}
	\arrow[from=1-3, to=1-4]
	\arrow[""{name=0, anchor=center, inner sep=0}, from=1-3, to=2-3]
	\arrow[""{name=1, anchor=center, inner sep=0}, from=1-4, to=2-4]
	\arrow[from=2-2, to=2-3]
	\arrow[""{name=2, anchor=center, inner sep=0}, equals, from=2-2, to=3-2]
	\arrow[from=2-3, to=2-4]
	\arrow[""{name=3, anchor=center, inner sep=0}, from=2-3, to=3-3]
	\arrow[""{name=4, anchor=center, inner sep=0}, from=2-4, to=3-4]
	\arrow[from=3-2, to=3-3]
	\arrow[from=3-2, to=4-2]
	\arrow[from=3-3, to=3-4]
	\arrow[""{name=5, anchor=center, inner sep=0}, from=3-3, to=5-3]
	\arrow[""{name=6, anchor=center, inner sep=0}, from=3-4, to=5-4]
	\arrow[from=4-1, to=4-2]
	\arrow[""{name=7, anchor=center, inner sep=0}, from=4-1, to=5-1]
	\arrow[""{name=8, anchor=center, inner sep=0}, from=4-2, to=5-2]
	\arrow[from=5-1, to=5-2]
	\arrow[from=5-2, to=5-3]
	\arrow[from=5-3, to=5-4]
	\arrow["1"{description}, draw=none, from=0, to=1]
	\arrow["2"{description}, draw=none, from=2, to=3]
	\arrow["11"{description}, draw=none, from=3, to=4]
	\arrow["12"{description}, draw=none, from=5, to=6]
	\arrow["4"{description}, draw=none, from=7, to=8]
	\arrow["5"{description}, draw=none, from=4-2, to=5]
\end{tikzcd}}
\quad \text{ and }\quad
R_2=
\hspace{-1.5mm}\adjustbox{scale=0.60}{\begin{tikzcd}
	&& {} & {} \\
	& {} & {} & {} \\
	& {} & {} & {} \\
	{} & {} \\
	{} & {} & {} & {}
	\arrow[from=1-3, to=1-4]
	\arrow[""{name=0, anchor=center, inner sep=0}, from=1-3, to=2-3]
	\arrow[""{name=1, anchor=center, inner sep=0}, from=1-4, to=2-4]
	\arrow[from=2-2, to=2-3]
	\arrow[""{name=2, anchor=center, inner sep=0}, equals, from=2-2, to=3-2]
	\arrow[from=2-3, to=2-4]
	\arrow[""{name=3, anchor=center, inner sep=0}, from=2-3, to=3-3]
	\arrow[""{name=4, anchor=center, inner sep=0}, from=2-4, to=3-4]
	\arrow[from=3-2, to=3-3]
	\arrow[from=3-2, to=4-2]
	\arrow[from=3-3, to=3-4]
	\arrow[""{name=5, anchor=center, inner sep=0}, from=3-3, to=5-3]
	\arrow[""{name=6, anchor=center, inner sep=0}, from=3-4, to=5-4]
	\arrow[from=4-1, to=4-2]
	\arrow[""{name=7, anchor=center, inner sep=0}, from=4-1, to=5-1]
	\arrow[""{name=8, anchor=center, inner sep=0}, from=4-2, to=5-2]
	\arrow[from=5-1, to=5-2]
	\arrow[from=5-2, to=5-3]
	\arrow[from=5-3, to=5-4]
	\arrow["1"{description}, draw=none, from=0, to=1]
	\arrow["2"{description}, draw=none, from=2, to=3]
	\arrow["11"{description}, draw=none, from=3, to=4]
	\arrow["13"{description}, draw=none, from=5, to=6]
	\arrow["4"{description}, draw=none, from=7, to=8]
	\arrow["8"{description}, draw=none, from=4-2, to=5]
\end{tikzcd}}\, .
\]
The following diagram shows that the given cycle is indeed equivalent to the identity $\Sg$-path.
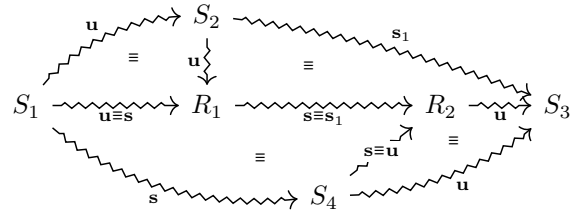
\[\begin{tikzcd}
	&& {S_2} \\
	{S_1} && {R_1} && {R_2} & {S_3} \\
	&&& {S_4}
	\arrow[""{name=0, anchor=center, inner sep=0}, "\bu"', squiggly, from=1-3, to=2-3]
	\arrow[""{name=1, anchor=center, inner sep=0}, "{{{{{{\bs}_1}}}}}", curve={height=-6pt}, squiggly, from=1-3, to=2-6]
	\arrow[""{name=2, anchor=center, inner sep=0}, "\bu", curve={height=-12pt}, squiggly, from=2-1, to=1-3]
	\arrow["{{{{\bu\equiv \bs}}}}"', squiggly, from=2-1, to=2-3]
	\arrow[""{name=3, anchor=center, inner sep=0}, "\bs"', curve={height=18pt}, squiggly, from=2-1, to=3-4]
	\arrow["{{{{{\bs\equiv {\bs}_1}}}}}"', squiggly, from=2-3, to=2-5]
	\arrow["\bu"', squiggly, from=2-5, to=2-6]
	\arrow[""{name=4, anchor=center, inner sep=0}, "{{{{\bs\equiv \bu}}}}"{description}, squiggly, from=3-4, to=2-5]
	\arrow[""{name=5, anchor=center, inner sep=0}, "\bu"', curve={height=12pt}, squiggly, from=3-4, to=2-6]
	\arrow["\equiv"', draw=none, from=0, to=1]
	\arrow["\equiv"'{pos=0.4}, draw=none, from=2, to=0]
	\arrow["\equiv", draw=none, from=3, to=4]
	\arrow["\equiv"{pos=0.8}, shift left=3, draw=none, from=4, to=5]
\end{tikzcd}\]

  The remaining cases are similar.
\end{proof}

Recall the notion of \textbf{\em configuration of interest} and \textbf{\em $\Sg$-path of interest} given in  \cref{def:interest}.
Recall also  the definition of \textbf{\em canonical $\Sg$-scheme} of level 3 (Definition \ref{assum1}). Using just a bullet instead of $\dot{\Sg}$ to represent canonical $\Sg$-squares, the canonical $\Sg$-scheme will be represented by
$$\Can=\adjustbox{scale=0.60}{\begin{tikzcd}
	&& {} & {} \\
	& {} & {} & {} \\
	{} & {} & {} & {} \\
	{} & {} & {} & {}
	\arrow[from=1-3, to=1-4]
	\arrow[""{name=0, anchor=center, inner sep=0}, from=1-3, to=2-3]
	\arrow[""{name=1, anchor=center, inner sep=0}, from=1-4, to=2-4]
	\arrow[from=2-2, to=2-3]
	\arrow[""{name=2, anchor=center, inner sep=0}, from=2-2, to=3-2]
	\arrow[from=2-3, to=2-4]
	\arrow[""{name=3, anchor=center, inner sep=0}, from=2-3, to=3-3]
	\arrow[""{name=4, anchor=center, inner sep=0}, from=2-4, to=3-4]
	\arrow[from=3-1, to=3-2]
	\arrow[""{name=5, anchor=center, inner sep=0}, from=3-1, to=4-1]
	\arrow[from=3-2, to=3-3]
	\arrow[""{name=6, anchor=center, inner sep=0}, from=3-2, to=4-2]
	\arrow[from=3-3, to=3-4]
	\arrow[""{name=7, anchor=center, inner sep=0}, from=3-3, to=4-3]
	\arrow[""{name=8, anchor=center, inner sep=0}, from=3-4, to=4-4]
	\arrow[from=4-1, to=4-2]
	\arrow[from=4-2, to=4-3]
	\arrow[from=4-3, to=4-4]
	\arrow["\bullet"{description}, draw=none, from=0, to=1]
	\arrow["\bullet"{description}, draw=none, from=2, to=3]
	\arrow["\bullet"{description}, draw=none, from=3, to=4]
	\arrow["\bullet"{description}, draw=none, from=5, to=6]
	\arrow["\bullet"{description}, draw=none, from=6, to=7]
	\arrow["\bullet"{description}, draw=none, from=7, to=8]
\end{tikzcd}}\, .$$

The next proposition, which is  just Proposition \ref{pro:Omega-f} given in more detail, states that the horizontal composition of an identity  $\Omega$ 2-cell with an $\Omega$ 2-cell is also an $\Omega$ 2-cell.

\begin{proposition}\label{pro:alpha.f}  Let $A\xrightarrow{\bar{f}}B\xrightarrow{\bar{g}}C\xrightarrow{\bar{h}}D\xrightarrow{\bar{k}}E$ be $\Sg$-cospans, where $\bar{f}=(f,r)$, $\bar{g}=(g,s)$, $\bar{h}=(h,t)$ and $\bar{k}=(k,u)$.

(1) Let $\Omega\colon (l_1g,m_1u) \rightarrow (l_2g, m_2u)$ be a basic $\Omega$ 2-cell determined by a $\Sg$-step of level 2 of the type $d$ or $u$ between two $\Sg$-schemes with left border $(s,h,t,k)$ and right border $(l_i,m_i)$, respectively. Then the 2-cell $\Omega\circ 1_{\bar{f}}$ is an $\Omega$ 2-cell corresponding to a $\Sg$-path of interest of $\Sg$-schemes of level 3 and left border $(r,g,s,h,t,k)$.

(2) Let $\Omega\colon (l_1f,m_1t) \rightarrow (l_2f,m_2t)$ be a basic $\Omega$ 2-cell determined by a $\Sg$-step of level 2 of the type $d$ or $u$ between two $\Sg$-schemes with left border $(r,g,s,h)$ and right border $(l_i,m_i)$, respectively. Then the 2-cell $1_{\bar{k}}\circ \Omega$ is an $\Omega$ 2-cell corresponding to a $\Sg$-path of interest of $\Sg$-schemes of level 3 and left border $(r,g,s,h,t,k)$.

Moreover, $\Omega\circ 1_{\bar{f}}$ and $1_{\bar{k}}\circ \Omega$ correspond to the $\Sigma$-paths given by the following tables, where $\Sigma$-squares are indicated by numbers, or also sometimes just by a bullet in the case they are canonical $\Sigma$-squares.

\noindent \hspace*{-1mm}\begin{tabular}{|c|c|}\hline &\\
 $\Sigma$-step for $\Omega$&$\Sigma$-path for the composition $\Omega\circ \bar{f}$\\ \hline \hline
\(\adjustbox{scale=0.60}{\begin{tikzcd}
	& {} & {} &&& {} & {} \\
	{} & {} & {} && {} & {} & {} \\
	{} & {} & {} && {} & {} & {}
	\arrow["s", from=1-2, to=1-3]
	\arrow[""{name=0, anchor=center, inner sep=0}, "h"', from=1-2, to=2-2]
	\arrow[""{name=1, anchor=center, inner sep=0}, "{{{h'}}}", from=1-3, to=2-3]
	\arrow["s", from=1-6, to=1-7]
	\arrow[""{name=2, anchor=center, inner sep=0}, "h"', from=1-6, to=2-6]
	\arrow[""{name=3, anchor=center, inner sep=0}, "{{{h'}}}", from=1-7, to=2-7]
	\arrow["t", from=2-1, to=2-2]
	\arrow[""{name=4, anchor=center, inner sep=0}, "k"', from=2-1, to=3-1]
	\arrow["{s'}"', from=2-2, to=2-3]
	\arrow["d", between={0.3}{0.7}, squiggly, from=2-3, to=2-5]
	\arrow[""{name=5, anchor=center, inner sep=0}, "{{{k_1}}}", from=2-3, to=3-3]
	\arrow["t", from=2-5, to=2-6]
	\arrow[""{name=6, anchor=center, inner sep=0}, "k"', from=2-5, to=3-5]
	\arrow["{s'}"', from=2-6, to=2-7]
	\arrow[""{name=7, anchor=center, inner sep=0}, "{{{k_2}}}", from=2-7, to=3-7]
	\arrow["{{{t_1}}}"', from=3-1, to=3-3]
	\arrow["{{{t_2}}}"', from=3-5, to=3-7]
	\arrow["1"{description}, draw=none, from=0, to=1]
	\arrow["1"{description}, draw=none, from=2, to=3]
	\arrow["2"{description}, draw=none, from=4, to=5]
	\arrow["3"{description}, draw=none, from=6, to=7]
\end{tikzcd}}\)
&
 \(
 \adjustbox{scale=0.60}{\begin{tikzcd}
	&& {} & {} &&& {} & {} &&& {} & {} &&& {} & {} \\
	& {} & {} &&& {} & {} & {} && {} & {} & {} && {} & {} & {} \\
	{} & {} & {} && {} & {} & {} & {} & {} & {} & {} & {} & {} & {} & {} & {} \\
	{} & {} & {} & {} & {} & {} & {} & {} & {} & {} & {} & {} & {} & {} & {} & {}
	\arrow["r", from=1-3, to=1-4]
	\arrow["g"', from=1-3, to=2-3]
	\arrow[""{name=0, anchor=center, inner sep=0}, from=1-4, to=4-4]
	\arrow["r", from=1-7, to=1-8]
	\arrow[""{name=1, anchor=center, inner sep=0}, "g"', from=1-7, to=2-7]
	\arrow[""{name=2, anchor=center, inner sep=0}, from=1-8, to=2-8]
	\arrow["r", from=1-11, to=1-12]
	\arrow[""{name=3, anchor=center, inner sep=0}, "g"', from=1-11, to=2-11]
	\arrow[""{name=4, anchor=center, inner sep=0}, from=1-12, to=2-12]
	\arrow["r", from=1-15, to=1-16]
	\arrow["g"', from=1-15, to=2-15]
	\arrow[""{name=5, anchor=center, inner sep=0}, from=1-16, to=4-16]
	\arrow[from=2-2, to=2-3]
	\arrow[""{name=6, anchor=center, inner sep=0}, from=2-2, to=3-2]
	\arrow[""{name=7, anchor=center, inner sep=0}, from=2-3, to=3-3]
	\arrow[from=2-6, to=2-7]
	\arrow[""{name=8, anchor=center, inner sep=0}, from=2-6, to=3-6]
	\arrow[from=2-7, to=2-8]
	\arrow[""{name=9, anchor=center, inner sep=0}, from=2-7, to=3-7]
	\arrow[""{name=10, anchor=center, inner sep=0}, from=2-8, to=3-8]
	\arrow[from=2-10, to=2-11]
	\arrow[""{name=11, anchor=center, inner sep=0}, from=2-10, to=3-10]
	\arrow[from=2-11, to=2-12]
	\arrow[""{name=12, anchor=center, inner sep=0}, from=2-11, to=3-11]
	\arrow[""{name=13, anchor=center, inner sep=0}, from=2-12, to=3-12]
	\arrow[from=2-14, to=2-15]
	\arrow[""{name=14, anchor=center, inner sep=0}, from=2-14, to=3-14]
	\arrow[""{name=15, anchor=center, inner sep=0}, from=2-15, to=3-15]
	\arrow[from=3-1, to=3-2]
	\arrow[""{name=16, anchor=center, inner sep=0}, from=3-1, to=4-1]
	\arrow[from=3-2, to=3-3]
	\arrow[""{name=17, anchor=center, inner sep=0}, from=3-3, to=4-3]
	\arrow[from=3-5, to=3-6]
	\arrow[""{name=18, anchor=center, inner sep=0}, from=3-5, to=4-5]
	\arrow[from=3-6, to=3-7]
	\arrow[from=3-7, to=3-8]
	\arrow[""{name=19, anchor=center, inner sep=0}, from=3-7, to=4-7]
	\arrow[""{name=20, anchor=center, inner sep=0}, from=3-8, to=4-8]
	\arrow[from=3-9, to=3-10]
	\arrow[""{name=21, anchor=center, inner sep=0}, from=3-9, to=4-9]
	\arrow[from=3-10, to=3-11]
	\arrow[from=3-11, to=3-12]
	\arrow[""{name=22, anchor=center, inner sep=0}, from=3-11, to=4-11]
	\arrow[""{name=23, anchor=center, inner sep=0}, from=3-12, to=4-12]
	\arrow[from=3-13, to=3-14]
	\arrow[""{name=24, anchor=center, inner sep=0}, from=3-13, to=4-13]
	\arrow[from=3-14, to=3-15]
	\arrow[""{name=25, anchor=center, inner sep=0}, from=3-15, to=4-15]
	\arrow[from=4-1, to=4-3]
	\arrow[from=4-3, to=4-4]
	\arrow[from=4-5, to=4-7]
	\arrow[from=4-7, to=4-8]
	\arrow[from=4-9, to=4-11]
	\arrow[from=4-11, to=4-12]
	\arrow[from=4-13, to=4-15]
	\arrow[from=4-15, to=4-16]
	\arrow["\bu", between={0.3}{0.7}, squiggly, from=0, to=8]
	\arrow["\bullet"{description}, draw=none, from=1, to=2]
	\arrow["\bullet"{description}, draw=none, from=3, to=4]
	\arrow["1"{description}, draw=none, from=6, to=7]
	\arrow["\bullet"{description}, draw=none, from=7, to=0]
	\arrow["1"{description}, draw=none, from=8, to=9]
	\arrow["\bullet"{description}, draw=none, from=9, to=10]
	\arrow["\bd", between={0.3}{0.7}, squiggly, from=10, to=11]
	\arrow["1"{description}, draw=none, from=11, to=12]
	\arrow["\bullet"{description}, draw=none, from=12, to=13]
	\arrow["\bu", between={0.3}{0.7}, squiggly, from=13, to=14]
	\arrow["1"{description}, draw=none, from=14, to=15]
	\arrow["\bullet"{description}, draw=none, from=15, to=5]
	\arrow["2"{description}, draw=none, from=16, to=17]
	\arrow["2"{description}, draw=none, from=18, to=19]
	\arrow["\bullet"{description}, draw=none, from=19, to=20]
	\arrow["3"{description}, draw=none, from=21, to=22]
	\arrow["\bullet"{description}, draw=none, from=22, to=23]
	\arrow["3"{description}, draw=none, from=24, to=25]
\end{tikzcd}}
\)
\\ \hline
\(\adjustbox{scale=0.60}{\begin{tikzcd}
	& {} & {} &&& {} & {} \\
	{} & {} & {} && {} & {} & {} \\
	{} & {} & {} && {} & {} & {}
	\arrow["s", from=1-2, to=1-3]
	\arrow["h"', from=1-2, to=2-2]
	\arrow["{{h_1}}", from=1-3, to=3-3]
	\arrow["s", from=1-6, to=1-7]
	\arrow["h"', from=1-6, to=2-6]
	\arrow["{{h_2}}", from=1-7, to=3-7]
	\arrow["t", from=2-1, to=2-2]
	\arrow[""{name=0, anchor=center, inner sep=0}, "k"', from=2-1, to=3-1]
	\arrow["2"{description}, draw=none, from=2-2, to=2-3]
	\arrow[""{name=1, anchor=center, inner sep=0}, "{{k_1}}", from=2-2, to=3-2]
	\arrow["u", between={0.3}{0.7}, squiggly, from=2-3, to=2-5]
	\arrow["t", from=2-5, to=2-6]
	\arrow[""{name=2, anchor=center, inner sep=0}, "k"', from=2-5, to=3-5]
	\arrow["3"{description}, draw=none, from=2-6, to=2-7]
	\arrow[""{name=3, anchor=center, inner sep=0}, "{{k_1}}", from=2-6, to=3-6]
	\arrow["{{t_1}}"', from=3-1, to=3-2]
	\arrow["{{s_1}}"', from=3-2, to=3-3]
	\arrow["{{t_1}}"', from=3-5, to=3-6]
	\arrow["{{s_2}}"', from=3-6, to=3-7]
	\arrow["1"{description}, draw=none, from=0, to=1]
	\arrow["1"{description}, draw=none, from=2, to=3]
\end{tikzcd}}\)
&
 \(
 \adjustbox{scale=0.60}{\begin{tikzcd}
	&& {} & {} &&& {} & {} &&& {} & {} &&& {} & {} \\
	& {} & {} & {} && {} & {} & {} && {} & {} & {} && {} & {} & {} \\
	{} & {} & {} & {} & {} & {} & {} & {} & {} & {} & {} & {} & {} & {} & {} & {} \\
	{} & {} & {} & {} & {} & {} & {} & {} & {} & {} & {} & {} & {} & {} & {} & {}
	\arrow["r", from=1-3, to=1-4]
	\arrow["g"', from=1-3, to=2-3]
	\arrow[""{name=0, anchor=center, inner sep=0}, from=1-4, to=4-4]
	\arrow["r", from=1-7, to=1-8]
	\arrow[""{name=1, anchor=center, inner sep=0}, "g"', from=1-7, to=2-7]
	\arrow[""{name=2, anchor=center, inner sep=0}, from=1-8, to=2-8]
	\arrow["r", from=1-11, to=1-12]
	\arrow[""{name=3, anchor=center, inner sep=0}, "g"', from=1-11, to=2-11]
	\arrow[""{name=4, anchor=center, inner sep=0}, from=1-12, to=2-12]
	\arrow["r", from=1-15, to=1-16]
	\arrow["g"', from=1-15, to=2-15]
	\arrow[from=1-16, to=4-16]
	\arrow[from=2-2, to=2-3]
	\arrow[from=2-2, to=3-2]
	\arrow[""{name=5, anchor=center, inner sep=0}, draw=none, from=2-3, to=3-3]
	\arrow[from=2-3, to=4-3]
	\arrow[""{name=6, anchor=center, inner sep=0}, draw=none, from=2-4, to=3-4]
	\arrow[from=2-6, to=2-7]
	\arrow[""{name=7, anchor=center, inner sep=0}, from=2-6, to=3-6]
	\arrow[from=2-7, to=2-8]
	\arrow[from=2-7, to=4-7]
	\arrow[""{name=8, anchor=center, inner sep=0}, draw=none, from=2-8, to=3-8]
	\arrow[from=2-8, to=4-8]
	\arrow[from=2-10, to=2-11]
	\arrow[""{name=9, anchor=center, inner sep=0}, from=2-10, to=3-10]
	\arrow[from=2-11, to=2-12]
	\arrow[from=2-11, to=4-11]
	\arrow[""{name=10, anchor=center, inner sep=0}, draw=none, from=2-12, to=3-12]
	\arrow[from=2-12, to=4-12]
	\arrow[from=2-14, to=2-15]
	\arrow[""{name=11, anchor=center, inner sep=0}, from=2-14, to=3-14]
	\arrow[""{name=12, anchor=center, inner sep=0}, draw=none, from=2-15, to=3-15]
	\arrow[from=2-15, to=4-15]
	\arrow[""{name=13, anchor=center, inner sep=0}, draw=none, from=2-16, to=3-16]
	\arrow[from=3-1, to=3-2]
	\arrow[""{name=14, anchor=center, inner sep=0}, from=3-1, to=4-1]
	\arrow["2"{description}, draw=none, from=3-2, to=3-3]
	\arrow[""{name=15, anchor=center, inner sep=0}, from=3-2, to=4-2]
	\arrow[from=3-5, to=3-6]
	\arrow[""{name=16, anchor=center, inner sep=0}, from=3-5, to=4-5]
	\arrow["2"{description}, draw=none, from=3-6, to=3-7]
	\arrow[""{name=17, anchor=center, inner sep=0}, from=3-6, to=4-6]
	\arrow["5"{description}, draw=none, from=3-7, to=3-8]
	\arrow[from=3-9, to=3-10]
	\arrow[""{name=18, anchor=center, inner sep=0}, from=3-9, to=4-9]
	\arrow["3"{description}, draw=none, from=3-10, to=3-11]
	\arrow[""{name=19, anchor=center, inner sep=0}, from=3-10, to=4-10]
	\arrow["6"{description}, draw=none, from=3-11, to=3-12]
	\arrow[from=3-13, to=3-14]
	\arrow[""{name=20, anchor=center, inner sep=0}, from=3-13, to=4-13]
	\arrow["3"{description}, draw=none, from=3-14, to=3-15]
	\arrow[""{name=21, anchor=center, inner sep=0}, from=3-14, to=4-14]
	\arrow[from=4-1, to=4-2]
	\arrow[from=4-2, to=4-3]
	\arrow[from=4-3, to=4-4]
	\arrow[from=4-5, to=4-6]
	\arrow[from=4-6, to=4-7]
	\arrow[from=4-7, to=4-8]
	\arrow[from=4-9, to=4-10]
	\arrow[from=4-10, to=4-11]
	\arrow[from=4-11, to=4-12]
	\arrow[from=4-13, to=4-14]
	\arrow[from=4-14, to=4-15]
	\arrow[from=4-15, to=4-16]
	\arrow["\bu", between={0.3}{0.7}, squiggly, from=0, to=7]
	\arrow["\bullet"{description}, draw=none, from=1, to=2]
	\arrow["\bullet"{description}, draw=none, from=3, to=4]
	\arrow["\bullet"{description}, draw=none, from=5, to=6]
	\arrow["\bs", between={0.3}{0.7}, squiggly, from=8, to=9]
	\arrow["\bu", between={0.3}{0.7}, squiggly, from=10, to=11]
	\arrow["\bullet"{description}, draw=none, from=12, to=13]
	\arrow["1"{description}, draw=none, from=14, to=15]
	\arrow["1"{description}, draw=none, from=16, to=17]
	\arrow["1"{description}, draw=none, from=18, to=19]
	\arrow["1"{description}, draw=none, from=20, to=21]
\end{tikzcd}}
\)
\\\hline
\end{tabular}
\noindent
\hspace*{-1mm}\begin{tabular}{|c|c|}\hline &\\
$\Sigma$-step for $\Omega$&$\Sigma$-path for the composition $\bar{k}\circ\Omega$\\ \hline \hline
\(\adjustbox{scale=0.60}{\begin{tikzcd}
	& {} & {} &&& {} & {} \\
	{} & {} & {} && {} & {} & {} \\
	{} & {} & {} && {} & {} & {}
	\arrow["r", from=1-2, to=1-3]
	\arrow[""{name=0, anchor=center, inner sep=0}, "g"', from=1-2, to=2-2]
	\arrow[""{name=1, anchor=center, inner sep=0}, "{g'}", from=1-3, to=2-3]
	\arrow["r", from=1-6, to=1-7]
	\arrow[""{name=2, anchor=center, inner sep=0}, "g"', from=1-6, to=2-6]
	\arrow[""{name=3, anchor=center, inner sep=0}, "{g'}", from=1-7, to=2-7]
	\arrow["s", from=2-1, to=2-2]
	\arrow[""{name=4, anchor=center, inner sep=0}, "h"', from=2-1, to=3-1]
	\arrow["{r'}"', from=2-2, to=2-3]
	\arrow["d", between={0.3}{0.7}, squiggly, from=2-3, to=2-5]
	\arrow[""{name=5, anchor=center, inner sep=0}, "{h_1}", from=2-3, to=3-3]
	\arrow["s", from=2-5, to=2-6]
	\arrow[""{name=6, anchor=center, inner sep=0}, "h"', from=2-5, to=3-5]
	\arrow["{r'}"', from=2-6, to=2-7]
	\arrow[""{name=7, anchor=center, inner sep=0}, "{h_2}", from=2-7, to=3-7]
	\arrow["{s_1}"', from=3-1, to=3-3]
	\arrow["{s_2}"', from=3-5, to=3-7]
	\arrow["1"{description}, draw=none, from=0, to=1]
	\arrow["1"{description}, draw=none, from=2, to=3]
	\arrow["2"{description}, draw=none, from=4, to=5]
	\arrow["3"{description}, draw=none, from=6, to=7]
\end{tikzcd}}\)
&
 \(
 \adjustbox{scale=0.60}{\begin{tikzcd}
	&& {} & {} &&& {} & {} &&& {} & {} &&& {} & {} \\
	& {} & {} & {} && {} & {} & {} && {} & {} & {} && {} & {} & {} \\
	{} & {} & {} & {} & {} & {} & {} & {} & {} & {} & {} & {} & {} & {} & {} & {} \\
	{} & {} & {} & {} & {} & {} & {} & {} & {} & {} & {} & {} & {} & {} & {} & {} \\
	&&&& {} & {} && {} & {} & {} && {}
	\arrow["r", from=1-3, to=1-4]
	\arrow[""{name=0, anchor=center, inner sep=0}, "g"', from=1-3, to=2-3]
	\arrow[""{name=1, anchor=center, inner sep=0}, from=1-4, to=2-4]
	\arrow[from=1-7, to=1-8]
	\arrow[""{name=2, anchor=center, inner sep=0}, from=1-7, to=2-7]
	\arrow[""{name=3, anchor=center, inner sep=0}, from=1-8, to=2-8]
	\arrow[from=1-11, to=1-12]
	\arrow[""{name=4, anchor=center, inner sep=0}, from=1-11, to=2-11]
	\arrow[""{name=5, anchor=center, inner sep=0}, from=1-12, to=2-12]
	\arrow["r", from=1-15, to=1-16]
	\arrow[""{name=6, anchor=center, inner sep=0}, "g"', from=1-15, to=2-15]
	\arrow[""{name=7, anchor=center, inner sep=0}, from=1-16, to=2-16]
	\arrow["s", from=2-2, to=2-3]
	\arrow[""{name=8, anchor=center, inner sep=0}, "h", from=2-2, to=3-2]
	\arrow[from=2-3, to=2-4]
	\arrow[""{name=9, anchor=center, inner sep=0}, from=2-4, to=3-4]
	\arrow[from=2-6, to=2-7]
	\arrow[""{name=10, anchor=center, inner sep=0}, from=2-6, to=3-6]
	\arrow[from=2-7, to=2-8]
	\arrow[""{name=11, anchor=center, inner sep=0}, from=2-8, to=3-8]
	\arrow[from=2-10, to=2-11]
	\arrow[""{name=12, anchor=center, inner sep=0}, from=2-10, to=3-10]
	\arrow[from=2-11, to=2-12]
	\arrow[""{name=13, anchor=center, inner sep=0}, from=2-12, to=3-12]
	\arrow[from=2-14, to=2-15]
	\arrow[""{name=14, anchor=center, inner sep=0}, from=2-14, to=3-14]
	\arrow[from=2-15, to=2-16]
	\arrow[""{name=15, anchor=center, inner sep=0}, from=2-16, to=3-16]
	\arrow["t", from=3-1, to=3-2]
	\arrow[""{name=16, anchor=center, inner sep=0}, "k", from=3-1, to=4-1]
	\arrow[from=3-2, to=3-4]
	\arrow[""{name=17, anchor=center, inner sep=0}, from=3-4, to=4-4]
	\arrow["t", from=3-5, to=3-6]
	\arrow[""{name=18, anchor=center, inner sep=0}, equals, from=3-5, to=4-5]
	\arrow[from=3-6, to=3-8]
	\arrow[""{name=19, anchor=center, inner sep=0}, equals, from=3-6, to=4-6]
	\arrow[""{name=20, anchor=center, inner sep=0}, from=3-8, to=4-8]
	\arrow[from=3-9, to=3-10]
	\arrow[""{name=21, anchor=center, inner sep=0}, equals, from=3-9, to=4-9]
	\arrow[from=3-10, to=3-12]
	\arrow[""{name=22, anchor=center, inner sep=0}, equals, from=3-10, to=4-10]
	\arrow[""{name=23, anchor=center, inner sep=0}, from=3-12, to=4-12]
	\arrow[from=3-13, to=3-14]
	\arrow[""{name=24, anchor=center, inner sep=0}, from=3-13, to=4-13]
	\arrow[from=3-14, to=3-16]
	\arrow[""{name=25, anchor=center, inner sep=0}, from=3-16, to=4-16]
	\arrow[from=4-1, to=4-4]
	\arrow[from=4-5, to=4-6]
	\arrow[""{name=26, anchor=center, inner sep=0}, "k", from=4-5, to=5-5]
	\arrow[from=4-6, to=4-8]
	\arrow[""{name=27, anchor=center, inner sep=0}, from=4-6, to=5-6]
	\arrow[""{name=28, anchor=center, inner sep=0}, from=4-8, to=5-8]
	\arrow[from=4-9, to=4-10]
	\arrow[""{name=29, anchor=center, inner sep=0}, from=4-9, to=5-9]
	\arrow[from=4-10, to=4-12]
	\arrow[""{name=30, anchor=center, inner sep=0}, from=4-10, to=5-10]
	\arrow[""{name=31, anchor=center, inner sep=0}, from=4-12, to=5-12]
	\arrow[from=4-13, to=4-16]
	\arrow[from=5-5, to=5-6]
	\arrow[from=5-6, to=5-8]
	\arrow[from=5-9, to=5-10]
	\arrow[from=5-10, to=5-12]
	\arrow["1"{description}, draw=none, from=0, to=1]
	\arrow["1"{description}, draw=none, from=2, to=3]
	\arrow["1"{description}, draw=none, from=4, to=5]
	\arrow["1"{description}, draw=none, from=6, to=7]
	\arrow["2"{description}, draw=none, from=8, to=9]
	\arrow["\bd", between={0.3}{0.7}, squiggly, from=9, to=10]
	\arrow["2"{description}, draw=none, from=10, to=11]
	\arrow["\bs", between={0.3}{0.7}, squiggly, from=11, to=12]
	\arrow["3"{description}, draw=none, from=12, to=13]
	\arrow["\bd", between={0.3}{0.7}, squiggly, from=13, to=14]
	\arrow["3"{description}, draw=none, from=14, to=15]
	\arrow["\bullet"{description}, draw=none, from=16, to=17]
	\arrow["\bullet"{description}, draw=none, from=18, to=19]
	\arrow["4"{description}, draw=none, from=19, to=20]
	\arrow["\bullet"{description}, draw=none, from=21, to=22]
	\arrow["5"{description}, draw=none, from=22, to=23]
	\arrow["\bullet"{description}, draw=none, from=24, to=25]
	\arrow["6"{description}, draw=none, from=26, to=27]
	\arrow["7"{description}, draw=none, from=27, to=28]
	\arrow["6"{description}, draw=none, from=29, to=30]
	\arrow["8"{description}, draw=none, from=30, to=31]
\end{tikzcd}}
\)
\\ \hline
\(\adjustbox{scale=0.60}{\begin{tikzcd}
	& {} & {} &&& {} & {} \\
	{} & {} & {} && {} & {} & {} \\
	{} & {} & {} && {} & {} & {}
	\arrow["r", from=1-2, to=1-3]
	\arrow["g"', from=1-2, to=2-2]
	\arrow["{{h_1}}", from=1-3, to=3-3]
	\arrow["r", from=1-6, to=1-7]
	\arrow["g"', from=1-6, to=2-6]
	\arrow["{{h_2}}", from=1-7, to=3-7]
	\arrow["s", from=2-1, to=2-2]
	\arrow[""{name=0, anchor=center, inner sep=0}, "h"', from=2-1, to=3-1]
	\arrow["2"{description}, draw=none, from=2-2, to=2-3]
	\arrow[""{name=1, anchor=center, inner sep=0}, "{h'}", from=2-2, to=3-2]
	\arrow["u", between={0.3}{0.7}, squiggly, from=2-3, to=2-5]
	\arrow["s", from=2-5, to=2-6]
	\arrow[""{name=2, anchor=center, inner sep=0}, "h"', from=2-5, to=3-5]
	\arrow["3"{description}, draw=none, from=2-6, to=2-7]
	\arrow[""{name=3, anchor=center, inner sep=0}, "{h'}", from=2-6, to=3-6]
	\arrow["{s'}"', from=3-1, to=3-2]
	\arrow["{{s_1}}"', from=3-2, to=3-3]
	\arrow["{s'}"', from=3-5, to=3-6]
	\arrow["{{s_2}}"', from=3-6, to=3-7]
	\arrow["1"{description}, draw=none, from=0, to=1]
	\arrow["1"{description}, draw=none, from=2, to=3]
\end{tikzcd}}\)
&
 \(
 \adjustbox{scale=0.60}{\begin{tikzcd}
	&& {} & {} &&& {} & {} &&& {} & {} &&& {} & {} \\
	& {} & {} & {} && {} & {} & {} && {} & {} & {} && {} & {} & {} \\
	{} & {} & {} & {} & {} & {} & {} & {} & {} & {} & {} & {} & {} & {} & {} & {} \\
	{} & {} & {} & {} & {} & {} & {} & {} & {} & {} & {} & {} & {} & {} & {} & {} \\
	&&&& {} && {} & {} & {} && {} & {}
	\arrow["r", from=1-3, to=1-4]
	\arrow["g"', from=1-3, to=2-3]
	\arrow[from=1-4, to=3-4]
	\arrow["r", from=1-7, to=1-8]
	\arrow["g"', from=1-7, to=2-7]
	\arrow[from=1-8, to=3-8]
	\arrow["r", from=1-11, to=1-12]
	\arrow["g"', from=1-11, to=2-11]
	\arrow[from=1-12, to=3-12]
	\arrow["r", from=1-15, to=1-16]
	\arrow["g"', from=1-15, to=2-15]
	\arrow[from=1-16, to=3-16]
	\arrow["s", from=2-2, to=2-3]
	\arrow[""{name=0, anchor=center, inner sep=0}, "h"', from=2-2, to=3-2]
	\arrow["2"{description}, draw=none, from=2-3, to=2-4]
	\arrow[""{name=1, anchor=center, inner sep=0}, from=2-3, to=3-3]
	\arrow[""{name=2, anchor=center, inner sep=0}, draw=none, from=2-4, to=3-4]
	\arrow[from=2-6, to=2-7]
	\arrow[""{name=3, anchor=center, inner sep=0}, from=2-6, to=3-6]
	\arrow["2"{description}, draw=none, from=2-7, to=2-8]
	\arrow[""{name=4, anchor=center, inner sep=0}, from=2-7, to=3-7]
	\arrow["\bu", between={0.3}{0.7}, squiggly, from=2-8, to=2-10]
	\arrow[from=2-10, to=2-11]
	\arrow[""{name=5, anchor=center, inner sep=0}, from=2-10, to=3-10]
	\arrow["3"{description}, draw=none, from=2-11, to=2-12]
	\arrow[""{name=6, anchor=center, inner sep=0}, from=2-11, to=3-11]
	\arrow[""{name=7, anchor=center, inner sep=0}, draw=none, from=2-12, to=3-12]
	\arrow[from=2-14, to=2-15]
	\arrow[""{name=8, anchor=center, inner sep=0}, from=2-14, to=3-14]
	\arrow["3"{description}, draw=none, from=2-15, to=2-16]
	\arrow[""{name=9, anchor=center, inner sep=0}, from=2-15, to=3-15]
	\arrow[draw=none, from=2-16, to=3-16]
	\arrow["t", from=3-1, to=3-2]
	\arrow[""{name=10, anchor=center, inner sep=0}, "k"', from=3-1, to=4-1]
	\arrow[from=3-2, to=3-3]
	\arrow[from=3-3, to=3-4]
	\arrow[""{name=11, anchor=center, inner sep=0}, from=3-4, to=4-4]
	\arrow["t", from=3-5, to=3-6]
	\arrow[""{name=12, anchor=center, inner sep=0}, equals, from=3-5, to=4-5]
	\arrow[from=3-6, to=3-7]
	\arrow[""{name=13, anchor=center, inner sep=0}, equals, from=3-6, to=4-6]
	\arrow[from=3-7, to=3-8]
	\arrow[""{name=14, anchor=center, inner sep=0}, equals, from=3-7, to=4-7]
	\arrow[""{name=15, anchor=center, inner sep=0}, "{{{d_1}}}", from=3-8, to=4-8]
	\arrow[from=3-9, to=3-10]
	\arrow[""{name=16, anchor=center, inner sep=0}, equals, from=3-9, to=4-9]
	\arrow[from=3-10, to=3-11]
	\arrow[""{name=17, anchor=center, inner sep=0}, equals, from=3-10, to=4-10]
	\arrow[from=3-11, to=3-12]
	\arrow[""{name=18, anchor=center, inner sep=0}, equals, from=3-11, to=4-11]
	\arrow[""{name=19, anchor=center, inner sep=0}, "{{{d_2}}}", from=3-12, to=4-12]
	\arrow[from=3-13, to=3-14]
	\arrow[""{name=20, anchor=center, inner sep=0}, from=3-13, to=4-13]
	\arrow[from=3-14, to=3-15]
	\arrow[from=3-15, to=3-16]
	\arrow[""{name=21, anchor=center, inner sep=0}, from=3-16, to=4-16]
	\arrow[from=4-1, to=4-4]
	\arrow[from=4-5, to=4-6]
	\arrow[""{name=22, anchor=center, inner sep=0}, "k"', from=4-5, to=5-5]
	\arrow[from=4-6, to=4-7]
	\arrow["d"', from=4-7, to=4-8]
	\arrow[""{name=23, anchor=center, inner sep=0}, from=4-7, to=5-7]
	\arrow[""{name=24, anchor=center, inner sep=0}, from=4-8, to=5-8]
	\arrow[from=4-9, to=4-10]
	\arrow[""{name=25, anchor=center, inner sep=0}, "", from=4-9, to=5-9]
	\arrow[from=4-10, to=4-11]
	\arrow["d"', from=4-11, to=4-12]
	\arrow[""{name=26, anchor=center, inner sep=0}, from=4-11, to=5-11]
	\arrow[""{name=27, anchor=center, inner sep=0}, from=4-12, to=5-12]
	\arrow[from=4-13, to=4-16]
	\arrow[from=5-5, to=5-7]
	\arrow[from=5-7, to=5-8]
	\arrow[from=5-9, to=5-11]
	\arrow[from=5-11, to=5-12]
	\arrow["1"{description}, draw=none, from=0, to=1]
	\arrow["\bd", between={0.3}{0.7}, squiggly, from=2, to=3]
	\arrow["1"{description}, draw=none, from=3, to=4]
	\arrow["1"{description}, draw=none, from=5, to=6]
	\arrow["\bd", between={0.3}{0.7},  squiggly, from=7, to=8]
	\arrow["1"{description}, draw=none, from=8, to=9]
	\arrow["\bullet"{description}, draw=none, from=10, to=11]
	\arrow["\bullet"{description}, draw=none, from=12, to=13]
	\arrow["\bullet"{description}, draw=none, from=13, to=14]
	\arrow["4"{description}, draw=none, from=14, to=15]
	\arrow["\bullet"{description}, draw=none, from=16, to=17]
	\arrow["\bullet"{description}, draw=none, from=17, to=18]
	\arrow["7"{description}, draw=none, from=18, to=19]
	\arrow["\bullet"{description}, draw=none, from=20, to=21]
	\arrow["5"{description}, draw=none, from=22, to=23]
	\arrow["6"{description}, draw=none, from=23, to=24]
	\arrow["5"{description}, draw=none, from=25, to=26]
	\arrow["6"{description}, draw=none, from=26, to=27]
\end{tikzcd}}
\)
\\\hline
\end{tabular}

\end{proposition}

\begin{proof} (1) Let $\Omega$ be the $\Omega$ 2-cell corresponding to the $\Sg$-step
\[\adjustbox{scale=0.80}{\begin{tikzcd}
	& {} & {} &&& {} & {} \\
	{} & {} & {} && {} & {} & {} \\
	{} & {} & {} && {} & {} & {}
	\arrow["s", from=1-2, to=1-3]
	\arrow[""{name=0, anchor=center, inner sep=0}, "h"', from=1-2, to=2-2]
	\arrow[""{name=1, anchor=center, inner sep=0}, "{h'}", from=1-3, to=2-3]
	\arrow["s", from=1-6, to=1-7]
	\arrow[""{name=2, anchor=center, inner sep=0}, "h"', from=1-6, to=2-6]
	\arrow[""{name=3, anchor=center, inner sep=0}, "{h'}", from=1-7, to=2-7]
	\arrow["t", from=2-1, to=2-2]
	\arrow[""{name=4, anchor=center, inner sep=0}, "k"', from=2-1, to=3-1]
	\arrow["{{s'}}"', from=2-2, to=2-3]
	\arrow["d", between={0.3}{0.7}, squiggly, from=2-3, to=2-5]
	\arrow[""{name=5, anchor=center, inner sep=0}, "{{k_1}}", from=2-3, to=3-3]
	\arrow["t", from=2-5, to=2-6]
	\arrow[""{name=6, anchor=center, inner sep=0}, "k"', from=2-5, to=3-5]
	\arrow["{{s'}}"', from=2-6, to=2-7]
	\arrow[""{name=7, anchor=center, inner sep=0}, "{{k_2}}", from=2-7, to=3-7]
	\arrow["{{t_1}}"', from=3-1, to=3-3]
	\arrow["{{t_2}}"', from=3-5, to=3-7]
	\arrow["1"{description}, draw=none, from=0, to=1]
	\arrow["1"{description}, draw=none, from=2, to=3]
	\arrow["2"{description}, draw=none, from=4, to=5]
	\arrow["3"{description}, draw=none, from=6, to=7]
\end{tikzcd}}\, .\]
Here $(l_i,m_i)=(k_ih',t_i)$.
Let $\Omega$ be represented by  the 2-morphism
\[\begin{tikzcd}
	{} & {} & {} & {} & {} & {} \\
	&&& {} & {} & {} \\
	{} & {} & {} & {} & {} & {}
	\arrow["g", from=1-1, to=1-2]
	\arrow[equals, from=1-1, to=3-1]
	\arrow["{h'}", from=1-2, to=1-3]
	\arrow["{{k_1}}", from=1-3, to=1-4]
	\arrow[equals, from=1-3, to=3-3]
	\arrow[""{name=0, anchor=center, inner sep=0}, "{{d_1}}"', from=1-4, to=2-4]
	\arrow["\theta"', shorten <=15pt, shorten >=11pt, Rightarrow, from=1-4, to=3-3]
	\arrow["{{t_1}}"', from=1-5, to=1-4]
	\arrow[""{name=1, anchor=center, inner sep=0}, equals, from=1-5, to=2-5]
	\arrow["u"', from=1-6, to=1-5]
	\arrow[""{name=2, anchor=center, inner sep=0}, equals, from=1-6, to=2-6]
	\arrow["d", from=2-5, to=2-4]
	\arrow[""{name=3, anchor=center, inner sep=0}, equals, from=2-5, to=3-5]
	\arrow["u", from=2-6, to=2-5]
	\arrow[""{name=4, anchor=center, inner sep=0}, equals, from=2-6, to=3-6]
	\arrow["{{=}}"{description}, draw=none, from=3-1, to=1-3]
	\arrow["g"', from=3-1, to=3-2]
	\arrow["{h'}"', from=3-2, to=3-3]
	\arrow["{{k_2}}"', from=3-3, to=3-4]
	\arrow[""{name=5, anchor=center, inner sep=0}, "{{d_2}}", from=3-4, to=2-4]
	\arrow["{{t_2}}", from=3-5, to=3-4]
	\arrow["u", from=3-6, to=3-5]
	\arrow["\SIGMA"{description}, draw=none, from=0, to=1]
	\arrow["\SIGMA"{description}, draw=none, from=1, to=2]
	\arrow["\SIGMA"{description}, draw=none, from=3, to=4]
	\arrow["\SIGMA"{description}, draw=none, from=5, to=3]
\end{tikzcd}\]
with $\theta$ invertible. Now, in order to compose this 2-morphism with $\bar{f}=(f,r)$, use vertical composition of canonical $\Sigma$-squares and Rule 6 to obtain:
\[
\begin{tikzcd}
	{} & {} \\
	{} & {} \\
	{} & {}
	\arrow["r", from=1-1, to=1-2]
	\arrow[""{name=0, anchor=center, inner sep=0}, "g"', from=1-1, to=2-1]
	\arrow[""{name=1, anchor=center, inner sep=0}, "{{\dot{g}}}", from=1-2, to=2-2]
	\arrow[from=2-1, to=2-2]
	\arrow[""{name=2, anchor=center, inner sep=0}, "{{h'}}"', from=2-1, to=3-1]
	\arrow[""{name=3, anchor=center, inner sep=0}, "{{\dot{h}}}", from=2-2, to=3-2]
	\arrow["{r'}"', from=3-1, to=3-2]
	\arrow["{{\SIGMAc}}"{description}, draw=none, from=0, to=1]
	\arrow["{{\SIGMAc}}"{description}, draw=none, from=2, to=3]
\end{tikzcd}
\qquad \qquad \qquad
\begin{tikzcd}
	{} & {} &&& {} & {} \\
	{} & {} && {} & {} & {} \\
	{} & {} &&& {} & {}
	\arrow["{{r'}}", from=1-1, to=1-2]
	\arrow["{{{{k_1}}}}"', from=1-1, to=2-1]
	\arrow[""{name=0, anchor=center, inner sep=0}, "{{{{k'_1}}}}", from=1-2, to=3-2]
	\arrow[""{name=1, anchor=center, inner sep=0}, "{{{{k'_2}}}}", shift left=3, curve={height=-24pt}, from=1-2, to=3-2]
	\arrow["{{r'}}", from=1-5, to=1-6]
	\arrow["{{{{k_1}}}}"', curve={height=6pt}, from=1-5, to=2-4]
	\arrow["{{{{k_2}}}}", from=1-5, to=2-5]
	\arrow["{{{{k'_2}}}}", from=1-6, to=3-6]
	\arrow["\SIGMA"{description}, draw=none, from=2-1, to=2-2]
	\arrow["{{{{d_1}}}}"', from=2-1, to=3-1]
	\arrow["\theta", Rightarrow, from=2-4, to=2-5]
	\arrow["{{{{d_1}}}}"', curve={height=6pt}, from=2-4, to=3-5]
	\arrow["\SIGMA"{description}, draw=none, from=2-5, to=2-6]
	\arrow["{{{{d_2}}}}", from=2-5, to=3-5]
	\arrow["e"', from=3-1, to=3-2]
	\arrow["e", from=3-5, to=3-6]
	\arrow["{\text{\normalsize $=$}}"{description, pos=0.6}, draw=none, from=1, to=2-4]
	\arrow["{{{{\theta'}}}}", between={0.3}{0.8}, Rightarrow, from=0, to=1]
\end{tikzcd}
\]
Thus, we can use $\theta'\circ \dot{h}\circ \dot{g}$ to obtain the middle 2-morphism of the desired horizontal composition (see Definition \ref{def:horizontal}):
\[\begin{tikzcd}
	{} & {} & {} & {} & {} & {} & {} & {} \\
	{} & {} & {} & {} & {} & {} & {} & {}
	\arrow["f", from=1-1, to=1-2]
	\arrow[""{name=0, anchor=center, inner sep=0}, equals, from=1-1, to=2-1]
	\arrow["{{{\dot{g}}}}", from=1-2, to=1-3]
	\arrow[""{name=1, anchor=center, inner sep=0}, equals, from=1-2, to=2-2]
	\arrow["{\dot{h}}", from=1-3, to=1-4]
	\arrow[""{name=2, anchor=center, inner sep=0}, "{{{k'_1}}}", from=1-4, to=1-5]
	\arrow[""{name=3, anchor=center, inner sep=0}, equals, from=1-4, to=2-4]
	\arrow[""{name=4, anchor=center, inner sep=0}, equals, from=1-5, to=2-5]
	\arrow["e"', from=1-6, to=1-5]
	\arrow["d"', from=1-7, to=1-6]
	\arrow[""{name=5, anchor=center, inner sep=0}, equals, from=1-7, to=2-7]
	\arrow["u"', from=1-8, to=1-7]
	\arrow[""{name=6, anchor=center, inner sep=0}, equals, from=1-8, to=2-8]
	\arrow["f"', from=2-1, to=2-2]
	\arrow["{{{\dot{g}}}}"', from=2-2, to=2-3]
	\arrow["{\dot{h}}"', from=2-3, to=2-4]
	\arrow[""{name=7, anchor=center, inner sep=0}, "{{{k'_2}}}"', from=2-4, to=2-5]
	\arrow["e", from=2-6, to=2-5]
	\arrow["d", from=2-7, to=2-6]
	\arrow["u", from=2-8, to=2-7]
	\arrow["{=}"{description}, draw=none, from=0, to=1]
	\arrow["{=}"{description}, draw=none, from=1, to=3]
	\arrow["{{{\theta'}}}", shorten <=4pt, shorten >=4pt, Rightarrow, from=2, to=7]
	\arrow["{\SIGMA^{\id}}"{description}, draw=none, from=4, to=5]
	\arrow["\SIGMA^{\id}"{description}, draw=none, from=5, to=6]
\end{tikzcd}\]
This 2-morphism is an $\Omega$ 2-morphism. Indeed, from the preceding equalities, we get
\[\begin{tikzcd}
	{} & {} & {} & {} \\
	{} && {} \\
	{} && {} & {}
	\arrow["t", from=1-1, to=1-2]
	\arrow[""{name=0, anchor=center, inner sep=0}, "k"', from=1-1, to=2-1]
	\arrow["{{s'}}", from=1-2, to=1-3]
	\arrow["{{r'}}", from=1-3, to=1-4]
	\arrow[""{name=1, anchor=center, inner sep=0}, "{{k_1}}"', from=1-3, to=2-3]
	\arrow[""{name=2, anchor=center, inner sep=0}, "{{k'_1}}", from=1-4, to=3-4]
	\arrow[""{name=3, anchor=center, inner sep=0}, "{{k'_2}}", shift left=3, curve={height=-30pt}, from=1-4, to=3-4]
	\arrow["{{t_1}}"', from=2-1, to=2-3]
	\arrow[""{name=4, anchor=center, inner sep=0}, equals, from=2-1, to=3-1]
	\arrow[""{name=5, anchor=center, inner sep=0}, "{{d_1}}"', from=2-3, to=3-3]
	\arrow["d"', from=3-1, to=3-3]
	\arrow["e"', from=3-3, to=3-4]
	\arrow["2"{description}, draw=none, from=0, to=1]
	\arrow["{{\theta'}}", between={0.3}{0.8}, Rightarrow, from=2, to=3]
	\arrow["\SIGMA"{description}, draw=none, from=4, to=5]
	\arrow["\SIGMA"{description}, draw=none, from=2-3, to=2]
\end{tikzcd}
\; =\;
\begin{tikzcd}
	{} & {} & {} & {} \\
	{} && {} \\
	{} && {} & {}
	\arrow["t", from=1-1, to=1-2]
	\arrow[""{name=0, anchor=center, inner sep=0}, "k"', from=1-1, to=2-1]
	\arrow["{{{s'}}}", from=1-2, to=1-3]
	\arrow["{{{r'}}}", from=1-3, to=1-4]
	\arrow[""{name=1, anchor=center, inner sep=0}, "{{k_2}}"', from=1-3, to=2-3]
	\arrow[""{name=2, anchor=center, inner sep=0}, "{k'_2}", from=1-4, to=3-4]
	\arrow["{{t_2}}"', from=2-1, to=2-3]
	\arrow[""{name=3, anchor=center, inner sep=0}, equals, from=2-1, to=3-1]
	\arrow[""{name=4, anchor=center, inner sep=0}, "{{d_2}}"', from=2-3, to=3-3]
	\arrow["d"', from=3-1, to=3-3]
	\arrow["e"', from=3-3, to=3-4]
	\arrow["3"{description}, draw=none, from=0, to=1]
	\arrow["\SIGMA"{description}, draw=none, from=3, to=4]
	\arrow["\SIGMA"{description}, draw=none, from=2-3, to=2]
\end{tikzcd}
\]
showing that it represents an $\Omega$ 2-cell corresponding to the $\Sigma$-step
$\,\begin{tikzcd}
	{S_3} & {S_4}
	\arrow["\bd", squiggly, from=1-1, to=1-2]
\end{tikzcd}\,$
 as in \eqref{eq:below0} below.
Combining this with the definition of horizontal composition as described in Definition \ref{def:horizontal},  first consider the diagram
\[\adjustbox{scale=0.60}{\begin{tikzcd}
	&& {} & {} \\
	& {} & {} \\
	{} & {} & {} \\
	{} && {} & {}
	\arrow["r", from=1-3, to=1-4]
	\arrow["g"', from=1-3, to=2-3]
	\arrow[""{name=0, anchor=center, inner sep=0}, from=1-4, to=4-4]
	\arrow["s", dashed, from=2-2, to=2-3]
	\arrow[""{name=1, anchor=center, inner sep=0}, "h"', dashed, from=2-2, to=3-2]
	\arrow[""{name=2, anchor=center, inner sep=0}, "{{h'}}", from=2-3, to=3-3]
	\arrow["t", dashed, from=3-1, to=3-2]
	\arrow[""{name=3, anchor=center, inner sep=0}, "k"', dashed, from=3-1, to=4-1]
	\arrow["{{s'}}"', dashed, from=3-2, to=3-3]
	\arrow[""{name=4, anchor=center, inner sep=0}, "{{k_1}}", from=3-3, to=4-3]
	\arrow["{{t_1}}"', from=4-1, to=4-3]
	\arrow[from=4-3, to=4-4]
	\arrow["1"{description}, draw=none, from=1, to=2]
	\arrow["\bullet"{description}, draw=none, from=2, to=0]
	\arrow["2"{description}, draw=none, from=3, to=4]
\end{tikzcd}}
\begin{tikzcd}
	{} & {}
	\arrow[squiggly, from=1-1, to=1-2]
\end{tikzcd}
\adjustbox{scale=0.60}{\begin{tikzcd}
	&& {} & {} \\
	& {} & {} \\
	{} & {} & {} \\
	{} && {} & {} \\
	{} && {} & {}
	\arrow["r", from=1-3, to=1-4]
	\arrow["g"', from=1-3, to=2-3]
	\arrow[""{name=0, anchor=center, inner sep=0}, from=1-4, to=4-4]
	\arrow["s", dashed, from=2-2, to=2-3]
	\arrow[""{name=1, anchor=center, inner sep=0}, "h"', dashed, from=2-2, to=3-2]
	\arrow[""{name=2, anchor=center, inner sep=0}, "{{{h'}}}", from=2-3, to=3-3]
	\arrow["t", dashed, from=3-1, to=3-2]
	\arrow[""{name=3, anchor=center, inner sep=0}, "k"', dashed, from=3-1, to=4-1]
	\arrow["{{{s'}}}"', dashed, from=3-2, to=3-3]
	\arrow[""{name=4, anchor=center, inner sep=0}, "{{{k_1}}}", from=3-3, to=4-3]
	\arrow["{{{t_1}}}"', from=4-1, to=4-3]
	\arrow[""{name=5, anchor=center, inner sep=0}, equals, from=4-1, to=5-1]
	\arrow[from=4-3, to=4-4]
	\arrow[""{name=6, anchor=center, inner sep=0}, "{{d_1}}"', from=4-3, to=5-3]
	\arrow[""{name=7, anchor=center, inner sep=0}, from=4-4, to=5-4]
	\arrow["d"', from=5-1, to=5-3]
	\arrow[from=5-3, to=5-4]
	\arrow["1"{description}, draw=none, from=1, to=2]
	\arrow["\bullet"{description}, draw=none, from=2, to=0]
	\arrow["2"{description}, draw=none, from=3, to=4]
	\arrow["\SIGMA"{description}, draw=none, from=5, to=6]
	\arrow["\bullet"{description}, draw=none, from=6, to=7]
\end{tikzcd}}
\begin{tikzcd}
	{} & {}
	\arrow[squiggly, from=1-1, to=1-2]
\end{tikzcd}
\adjustbox{scale=0.60}{\begin{tikzcd}
	&& {} & {} \\
	& {} & {} & {} \\
	{} & {} & {} & {} \\
	{} && {} \\
	{} && {} & {}
	\arrow["r", from=1-3, to=1-4]
	\arrow[""{name=0, anchor=center, inner sep=0}, "g"', from=1-3, to=2-3]
	\arrow[""{name=1, anchor=center, inner sep=0}, "{\dot{g}}", from=1-4, to=2-4]
	\arrow["s", dashed, from=2-2, to=2-3]
	\arrow[""{name=2, anchor=center, inner sep=0}, "h"', dashed, from=2-2, to=3-2]
	\arrow[from=2-3, to=2-4]
	\arrow[""{name=3, anchor=center, inner sep=0}, "{{{{{h'}}}}}"', from=2-3, to=3-3]
	\arrow[""{name=4, anchor=center, inner sep=0}, "{\dot{h}}", from=2-4, to=3-4]
	\arrow["t", dashed, from=3-1, to=3-2]
	\arrow[""{name=5, anchor=center, inner sep=0}, "k"', dashed, from=3-1, to=4-1]
	\arrow["{{{{{s'}}}}}"', dashed, from=3-2, to=3-3]
	\arrow[from=3-3, to=3-4]
	\arrow[""{name=6, anchor=center, inner sep=0}, "{{{{{k_1}}}}}"', from=3-3, to=4-3]
	\arrow[""{name=7, anchor=center, inner sep=0}, "{k'_1}", from=3-4, to=5-4]
	\arrow["{{{{{t_1}}}}}"', from=4-1, to=4-3]
	\arrow[""{name=8, anchor=center, inner sep=0}, equals, from=4-1, to=5-1]
	\arrow[""{name=9, anchor=center, inner sep=0}, "{{{{d_1}}}}"', from=4-3, to=5-3]
	\arrow["d"', from=5-1, to=5-3]
	\arrow["e"', from=5-3, to=5-4]
	\arrow["\bullet"{description}, draw=none, from=0, to=1]
	\arrow["1"{description}, draw=none, from=2, to=3]
	\arrow["\bullet"{description}, draw=none, from=3, to=4]
	\arrow["2"{description}, draw=none, from=5, to=6]
	\arrow["\SIGMA"{description}, draw=none, from=8, to=9]
	\arrow["\SIGMA"{description}, draw=none, from=4-3, to=7]
\end{tikzcd}}\, .
\]
Observe that the $\Omega$ 2-cell obtained from the $\Sigma$-path of level 2 given by the solid lines may also be obtained via the $\Sigma$-path of level 3 including the dashed lines, giving rise to
$\, \begin{tikzcd}
	{S_1} & {S_2} & {S_3}
	\arrow["{\bd_1}", squiggly, from=1-1, to=1-2]
	\arrow["\bu", squiggly, from=1-2, to=1-3]
\end{tikzcd}\,$
as in \eqref{eq:below0} below. In a similar way, we also get
$\, \begin{tikzcd}
	{S_4} & {S_5} & {S_6}
	\arrow["{\bu}", squiggly, from=1-1, to=1-2]
	\arrow["\bd_1", squiggly, from=1-2, to=1-3]
\end{tikzcd}\,$
 as in \eqref{eq:below0}. Thus, we see that the composition $\Omega\circ \bar{f}$ is given by the vertical composition of basic $\Omega$ 2-cells corresponding to a $\Sigma$-path
\begin{equation}\label{eq:below0}\begin{tikzcd}
	{S_1} & {S_2} & {S_3} & {S_4} & {S_5} & {S_6}
	\arrow["{{{\bd}_1}}", squiggly, from=1-1, to=1-2]
	\arrow["\bu", squiggly, from=1-2, to=1-3]
	\arrow["\bd", squiggly, from=1-3, to=1-4]
	\arrow["\bu", squiggly, from=1-4, to=1-5]
	\arrow["{{\bd}_1}", squiggly, from=1-5, to=1-6]
\end{tikzcd}\, ,
\end{equation}
where
\(\;S_1=\hspace{-2mm}\adjustbox{scale=0.60}{\begin{tikzcd}
	&& {} & {} \\
	& {} & {} & {} \\
	{} & {} & {} \\
	{} && {} & {}
	\arrow[from=1-3, to=1-4]
	\arrow[from=1-3, to=2-3]
	\arrow[from=1-4, to=4-4]
	\arrow[from=2-2, to=2-3]
	\arrow[""{name=0, anchor=center, inner sep=0}, from=2-2, to=3-2]
	\arrow["\bullet"', shift right=5, draw=none, from=2-3, to=2-4]
	\arrow[""{name=1, anchor=center, inner sep=0}, from=2-3, to=3-3]
	\arrow[from=3-1, to=3-2]
	\arrow[""{name=2, anchor=center, inner sep=0}, from=3-1, to=4-1]
	\arrow[from=3-2, to=3-3]
	\arrow[""{name=3, anchor=center, inner sep=0}, from=3-3, to=4-3]
	\arrow[from=4-1, to=4-3]
	\arrow[from=4-3, to=4-4]
	\arrow["1"{description}, draw=none, from=0, to=1]
	\arrow["2"{description}, draw=none, from=2, to=3]
\end{tikzcd}}\),
\(S_2=\adjustbox{scale=0.60}{\begin{tikzcd}
	&& {} & {} \\
	& {} & {} & {} \\
	{} & {} & {} \\
	{} && {} & {} \\
	{} && {} & {}
	\arrow[from=1-3, to=1-4]
	\arrow[from=1-3, to=2-3]
	\arrow[from=1-4, to=4-4]
	\arrow[from=2-2, to=2-3]
	\arrow[""{name=0, anchor=center, inner sep=0}, from=2-2, to=3-2]
	\arrow["\bullet"', shift right=5, draw=none, from=2-3, to=2-4]
	\arrow[""{name=1, anchor=center, inner sep=0}, from=2-3, to=3-3]
	\arrow[from=3-1, to=3-2]
	\arrow[""{name=2, anchor=center, inner sep=0}, from=3-1, to=4-1]
	\arrow[from=3-2, to=3-3]
	\arrow[""{name=3, anchor=center, inner sep=0}, from=3-3, to=4-3]
	\arrow[from=4-1, to=4-3]
	\arrow[""{name=4, anchor=center, inner sep=0}, equals, from=4-1, to=5-1]
	\arrow[from=4-3, to=4-4]
	\arrow[""{name=5, anchor=center, inner sep=0}, from=4-3, to=5-3]
	\arrow[""{name=6, anchor=center, inner sep=0}, from=4-4, to=5-4]
	\arrow[from=5-1, to=5-3]
	\arrow[from=5-3, to=5-4]
	\arrow["1"{description}, draw=none, from=0, to=1]
	\arrow["2"{description}, draw=none, from=2, to=3]
	\arrow["4"{description}, draw=none, from=4, to=5]
	\arrow["\bullet"{description}, draw=none, from=5, to=6]
\end{tikzcd}}\),
\(S_3=\adjustbox{scale=0.60}{\begin{tikzcd}
	&& {} & {} \\
	& {} & {} & {} \\
	{} & {} & {} & {} \\
	{} && {} \\
	{} && {} & {}
	\arrow["r", from=1-3, to=1-4]
	\arrow[""{name=0, anchor=center, inner sep=0}, "g"', from=1-3, to=2-3]
	\arrow[""{name=1, anchor=center, inner sep=0}, from=1-4, to=2-4]
	\arrow["s", from=2-2, to=2-3]
	\arrow[""{name=2, anchor=center, inner sep=0}, "h"', from=2-2, to=3-2]
	\arrow[from=2-3, to=2-4]
	\arrow[""{name=3, anchor=center, inner sep=0}, from=2-3, to=3-3]
	\arrow[""{name=4, anchor=center, inner sep=0}, from=2-4, to=3-4]
	\arrow["t", from=3-1, to=3-2]
	\arrow[""{name=5, anchor=center, inner sep=0}, "k"', from=3-1, to=4-1]
	\arrow[from=3-2, to=3-3]
	\arrow[from=3-3, to=3-4]
	\arrow[""{name=6, anchor=center, inner sep=0}, "{k_1}", from=3-3, to=4-3]
	\arrow[""{name=7, anchor=center, inner sep=0}, "{k'_1}", from=3-4, to=5-4]
	\arrow[from=4-1, to=4-3]
	\arrow[""{name=8, anchor=center, inner sep=0}, equals, from=4-1, to=5-1]
	\arrow[""{name=9, anchor=center, inner sep=0}, from=4-3, to=5-3]
	\arrow[from=5-1, to=5-3]
	\arrow[from=5-3, to=5-4]
	\arrow["\bullet"{description}, draw=none, from=0, to=1]
	\arrow["1"{description}, draw=none, from=2, to=3]
	\arrow["\bullet"{description}, draw=none, from=3, to=4]
	\arrow["2"{description}, draw=none, from=5, to=6]
	\arrow["4"{description}, draw=none, from=8, to=9]
	\arrow["6"{description}, draw=none, from=4-3, to=7]
\end{tikzcd}}\),
\(S_4=\adjustbox{scale=0.60}{\begin{tikzcd}
	&& {} & {} \\
	& {} & {} & {} \\
	{} & {} & {} & {} \\
	{} && {} \\
	{} && {} & {}
	\arrow["r", from=1-3, to=1-4]
	\arrow[""{name=0, anchor=center, inner sep=0}, "g"', from=1-3, to=2-3]
	\arrow[""{name=1, anchor=center, inner sep=0}, from=1-4, to=2-4]
	\arrow["s", from=2-2, to=2-3]
	\arrow[""{name=2, anchor=center, inner sep=0}, "h"', from=2-2, to=3-2]
	\arrow[from=2-3, to=2-4]
	\arrow[""{name=3, anchor=center, inner sep=0}, from=2-3, to=3-3]
	\arrow[""{name=4, anchor=center, inner sep=0}, from=2-4, to=3-4]
	\arrow["t", from=3-1, to=3-2]
	\arrow[""{name=5, anchor=center, inner sep=0}, "k"', from=3-1, to=4-1]
	\arrow[from=3-2, to=3-3]
	\arrow[from=3-3, to=3-4]
	\arrow[""{name=6, anchor=center, inner sep=0}, "{k_1}", from=3-3, to=4-3]
	\arrow[""{name=7, anchor=center, inner sep=0}, "{k'_1}", from=3-4, to=5-4]
	\arrow[from=4-1, to=4-3]
	\arrow[""{name=8, anchor=center, inner sep=0}, equals, from=4-1, to=5-1]
	\arrow[""{name=9, anchor=center, inner sep=0}, from=4-3, to=5-3]
	\arrow[from=5-1, to=5-3]
	\arrow[from=5-3, to=5-4]
	\arrow["\bullet"{description}, draw=none, from=0, to=1]
	\arrow["1"{description}, draw=none, from=2, to=3]
	\arrow["\bullet"{description}, draw=none, from=3, to=4]
	\arrow["3"{description}, draw=none, from=5, to=6]
	\arrow["5"{description}, draw=none, from=8, to=9]
	\arrow["7"{description}, draw=none, from=4-3, to=7]
\end{tikzcd}}\),
\\
\(S_5=\adjustbox{scale=0.60}{\begin{tikzcd}
	&& {} & {} \\
	& {} & {} & {} \\
	{} & {} & {} \\
	{} && {} & {} \\
	{} && {} & {}
	\arrow[from=1-3, to=1-4]
	\arrow[from=1-3, to=2-3]
	\arrow[from=1-4, to=4-4]
	\arrow[from=2-2, to=2-3]
	\arrow[""{name=0, anchor=center, inner sep=0}, from=2-2, to=3-2]
	\arrow["\bullet"', shift right=5, draw=none, from=2-3, to=2-4]
	\arrow[""{name=1, anchor=center, inner sep=0}, from=2-3, to=3-3]
	\arrow[from=3-1, to=3-2]
	\arrow[""{name=2, anchor=center, inner sep=0}, from=3-1, to=4-1]
	\arrow[from=3-2, to=3-3]
	\arrow[""{name=3, anchor=center, inner sep=0}, from=3-3, to=4-3]
	\arrow[from=4-1, to=4-3]
	\arrow[""{name=4, anchor=center, inner sep=0}, equals, from=4-1, to=5-1]
	\arrow[from=4-3, to=4-4]
	\arrow[""{name=5, anchor=center, inner sep=0}, from=4-3, to=5-3]
	\arrow[""{name=6, anchor=center, inner sep=0}, from=4-4, to=5-4]
	\arrow[from=5-1, to=5-3]
	\arrow[from=5-3, to=5-4]
	\arrow["1"{description}, draw=none, from=0, to=1]
	\arrow["3"{description}, draw=none, from=2, to=3]
	\arrow["5"{description}, draw=none, from=4, to=5]
	\arrow["\bullet"{description}, draw=none, from=5, to=6]
\end{tikzcd}}\)
and
\(S_6=\hspace{-2mm}\adjustbox{scale=0.60}{\begin{tikzcd}
	&& {} & {} \\
	& {} & {} & {} \\
	{} & {} & {} \\
	{} && {} & {}
	\arrow[from=1-3, to=1-4]
	\arrow[from=1-3, to=2-3]
	\arrow[from=1-4, to=4-4]
	\arrow[from=2-2, to=2-3]
	\arrow[""{name=0, anchor=center, inner sep=0}, from=2-2, to=3-2]
	\arrow["\bullet"', shift right=5, draw=none, from=2-3, to=2-4]
	\arrow[""{name=1, anchor=center, inner sep=0}, from=2-3, to=3-3]
	\arrow[from=3-1, to=3-2]
	\arrow[""{name=2, anchor=center, inner sep=0}, from=3-1, to=4-1]
	\arrow[from=3-2, to=3-3]
	\arrow[""{name=3, anchor=center, inner sep=0}, from=3-3, to=4-3]
	\arrow[from=4-1, to=4-3]
	\arrow[from=4-3, to=4-4]
	\arrow["1"{description}, draw=none, from=0, to=1]
	\arrow["3"{description}, draw=none, from=2, to=3]
\end{tikzcd}}
\).
 Let \(R_1=\hspace{-2mm}\adjustbox{scale=0.60}{\begin{tikzcd}
	&& {} & {} \\
	& {} & {} & {} \\
	{} & {} & {} & {} \\
	{} && {} & {}
	\arrow[from=1-3, to=1-4]
	\arrow[""{name=0, anchor=center, inner sep=0}, from=1-3, to=2-3]
	\arrow[""{name=1, anchor=center, inner sep=0}, from=1-4, to=2-4]
	\arrow[from=2-2, to=2-3]
	\arrow[""{name=2, anchor=center, inner sep=0}, from=2-2, to=3-2]
	\arrow[from=2-3, to=2-4]
	\arrow[""{name=3, anchor=center, inner sep=0}, from=2-3, to=3-3]
	\arrow[""{name=4, anchor=center, inner sep=0}, from=2-4, to=3-4]
	\arrow[from=3-1, to=3-2]
	\arrow[""{name=5, anchor=center, inner sep=0}, from=3-1, to=4-1]
	\arrow[from=3-2, to=3-3]
	\arrow[from=3-3, to=3-4]
	\arrow[""{name=6, anchor=center, inner sep=0}, from=3-3, to=4-3]
	\arrow[""{name=7, anchor=center, inner sep=0}, from=3-4, to=4-4]
	\arrow[from=4-1, to=4-3]
	\arrow[from=4-3, to=4-4]
	\arrow["\bullet"{description}, draw=none, from=0, to=1]
	\arrow["1"{description}, draw=none, from=2, to=3]
	\arrow["\bullet"{description}, draw=none, from=3, to=4]
	\arrow["2"{description}, draw=none, from=5, to=6]
	\arrow["\bullet"{description}, draw=none, from=6, to=7]
\end{tikzcd}}\, \)
 and
 \(R_2=\hspace{-2mm}\adjustbox{scale=0.60}{\begin{tikzcd}
	&& {} & {} \\
	& {} & {} & {} \\
	{} & {} & {} & {} \\
	{} && {} & {}
	\arrow[from=1-3, to=1-4]
	\arrow[""{name=0, anchor=center, inner sep=0}, from=1-3, to=2-3]
	\arrow[""{name=1, anchor=center, inner sep=0}, from=1-4, to=2-4]
	\arrow[from=2-2, to=2-3]
	\arrow[""{name=2, anchor=center, inner sep=0}, from=2-2, to=3-2]
	\arrow[from=2-3, to=2-4]
	\arrow[""{name=3, anchor=center, inner sep=0}, from=2-3, to=3-3]
	\arrow[""{name=4, anchor=center, inner sep=0}, from=2-4, to=3-4]
	\arrow[from=3-1, to=3-2]
	\arrow[""{name=5, anchor=center, inner sep=0}, from=3-1, to=4-1]
	\arrow[from=3-2, to=3-3]
	\arrow[from=3-3, to=3-4]
	\arrow[""{name=6, anchor=center, inner sep=0}, from=3-3, to=4-3]
	\arrow[""{name=7, anchor=center, inner sep=0}, from=3-4, to=4-4]
	\arrow[from=4-1, to=4-3]
	\arrow[from=4-3, to=4-4]
	\arrow["\bullet"{description}, draw=none, from=0, to=1]
	\arrow["1"{description}, draw=none, from=2, to=3]
	\arrow["\bullet"{description}, draw=none, from=3, to=4]
	\arrow["3"{description}, draw=none, from=5, to=6]
	\arrow["\bullet"{description}, draw=none, from=6, to=7]
\end{tikzcd}} \).
The equivalences of $\Sg$-paths indicated in the diagram
\[\begin{tikzcd}
	{S_1} & {S_2} & {S_3} & {S_4} & {S_5} & {S_6} \\
	&& {R_1} & {R_2}
	\arrow["{{{\bd}_1}}", squiggly, from=1-1, to=1-2]
	\arrow[""{name=0, anchor=center, inner sep=0}, "\bu"', squiggly, from=1-1, to=2-3]
	\arrow["\bu", squiggly, from=1-2, to=1-3]
	\arrow["\bd", squiggly, from=1-3, to=1-4]
	\arrow["\bu", squiggly, from=1-4, to=1-5]
	\arrow["{{{\bd}_1}}", squiggly, from=1-5, to=1-6]
	\arrow[""{name=1, anchor=center, inner sep=0}, "\bd", squiggly, from=2-3, to=1-3]
	\arrow["\bd"', squiggly, from=2-3, to=2-4]
	\arrow[""{name=2, anchor=center, inner sep=0}, "\bd"', squiggly, from=2-4, to=1-4]
	\arrow[""{name=3, anchor=center, inner sep=0}, "\bu"', squiggly, from=2-4, to=1-6]
	\arrow["\equiv"{description}, draw=none, from=0, to=1-3]
	\arrow["\equiv"{description}, draw=none, from=1-4, to=3]
	\arrow["\equiv"{description}, draw=none, from=1, to=2]
\end{tikzcd}\]
are clear:
 recall from Corollary~\ref{cor:length2} that two $\Sigma$-paths of the form \[\xymatrix{S_1\ar@{~>}[r]^{\mathbf{d}_1}&S_3\ar@{~>}[r]^{\mathbf{u}}&S_2} \text{ and } \xymatrix{S_1\ar@{~>}[r]^{\mathbf{u}}&S_4\ar@{~>}[r]^{\mathbf{d}}&S_2}\] are equivalent.
Therefore, we conclude that $\Omega\circ \bar{f}$ is given by a $\Sigma$-path of the form indicated in the table.

Now let $\Omega$ be the $\Omega$ 2-cell corresponding to the $\Sg$-step
\[\adjustbox{scale=0.80}{\begin{tikzcd}
	& {} & {} &&& {} & {} \\
	{} & {} & {} && {} & {} & {} \\
	{} & {} & {} && {} & {} & {}
	\arrow["s", from=1-2, to=1-3]
	\arrow["h"', from=1-2, to=2-2]
	\arrow["{h_1}", from=1-3, to=3-3]
	\arrow["s", from=1-6, to=1-7]
	\arrow["h"', from=1-6, to=2-6]
	\arrow["{h_2}", from=1-7, to=3-7]
	\arrow["t", from=2-1, to=2-2]
	\arrow[""{name=0, anchor=center, inner sep=0}, "k"', from=2-1, to=3-1]
	\arrow["2"{description}, draw=none, from=2-2, to=2-3]
	\arrow[""{name=1, anchor=center, inner sep=0}, "{k_1}", from=2-2, to=3-2]
	\arrow["u", between={0.3}{0.7}, squiggly, from=2-3, to=2-5]
	\arrow["t", from=2-5, to=2-6]
	\arrow[""{name=2, anchor=center, inner sep=0}, "k"', from=2-5, to=3-5]
	\arrow["3"{description}, draw=none, from=2-6, to=2-7]
	\arrow[""{name=3, anchor=center, inner sep=0}, "{k_1}", from=2-6, to=3-6]
	\arrow["{t_1}"', from=3-1, to=3-2]
	\arrow["{s_1}"', from=3-2, to=3-3]
	\arrow["{t_1}"', from=3-5, to=3-6]
	\arrow["{s_2}"', from=3-6, to=3-7]
	\arrow["1"{description}, draw=none, from=0, to=1]
	\arrow["1"{description}, draw=none, from=2, to=3]
\end{tikzcd}}\, .\]
Thus, here $(l_i,m_i)=(h_i,s_it_i)$.
Let $\Omega$ be represented by  the 2-morphism
\[\begin{tikzcd}
	{} & {} & {} & {} & {} & {} \\
	&& {} & {} & {} & {} \\
	{} & {} & {} & {} & {} & {}
	\arrow["g", from=1-1, to=1-2]
	\arrow[""{name=0, anchor=center, inner sep=0}, equals, from=1-1, to=3-1]
	\arrow["{{{h_1}}}", from=1-2, to=1-3]
	\arrow[""{name=1, anchor=center, inner sep=0}, equals, from=1-2, to=3-2]
	\arrow["{{s_1}}", from=1-3, to=1-4]
	\arrow[""{name=2, anchor=center, inner sep=0}, "{{d_1}}", from=1-3, to=2-3]
	\arrow["\theta"', shorten <=11pt, shorten >=11pt, Rightarrow, from=1-3, to=3-2]
	\arrow[""{name=3, anchor=center, inner sep=0}, equals, from=1-4, to=2-4]
	\arrow["{t_1}"', from=1-5, to=1-4]
	\arrow[""{name=4, anchor=center, inner sep=0}, equals, from=1-5, to=2-5]
	\arrow["u"', from=1-6, to=1-5]
	\arrow[""{name=5, anchor=center, inner sep=0}, equals, from=1-6, to=2-6]
	\arrow["d", from=2-4, to=2-3]
	\arrow[""{name=6, anchor=center, inner sep=0}, equals, from=2-4, to=3-4]
	\arrow["{t_1}", from=2-5, to=2-4]
	\arrow[""{name=7, anchor=center, inner sep=0}, equals, from=2-5, to=3-5]
	\arrow["u"', from=2-6, to=2-5]
	\arrow[""{name=8, anchor=center, inner sep=0}, equals, from=2-6, to=3-6]
	\arrow["g"', from=3-1, to=3-2]
	\arrow["{{h_2}}"', from=3-2, to=3-3]
	\arrow[""{name=9, anchor=center, inner sep=0}, "{{d_2}}"', from=3-3, to=2-3]
	\arrow["{{s_2}}"', from=3-3, to=3-4]
	\arrow["{t_1}", from=3-5, to=3-4]
	\arrow["u", from=3-6, to=3-5]
	\arrow[shorten <=13pt, shorten >=13pt, equals, from=0, to=1]
	\arrow["\SIGMA"{description}, draw=none, from=2, to=3]
	\arrow["\SIGMA"{description}, draw=none, from=3, to=4]
	\arrow["\SIGMA"{description}, draw=none, from=4, to=5]
	\arrow["\SIGMA"{description}, draw=none, from=7, to=6]
	\arrow["\SIGMA"{description}, draw=none, from=7, to=8]
	\arrow["\SIGMA"{description}, draw=none, from=9, to=6]
\end{tikzcd}\]
with $\theta$ invertible. Now, in order to compose it with $\bar{f}=(f,r)$, consider the canonical $\Sg$-square of $r$ along $g$  and an invertible 2-cell $\theta'$ as follows:
\[
\begin{tikzcd}
	{} & {} \\
	{} & {}
	\arrow["r", from=1-1, to=1-2]
	\arrow[""{name=0, anchor=center, inner sep=0}, "g"', from=1-1, to=2-1]
	\arrow[""{name=1, anchor=center, inner sep=0}, "{{{\dot{g}}}}", from=1-2, to=2-2]
	\arrow["{\dot{r}}"', from=2-1, to=2-2]
	\arrow["{{{{\dot{\SIGMA}}}}}"{description}, draw=none, from=0, to=1]
\end{tikzcd}
\qquad \qquad \qquad
\begin{tikzcd}
	{} & {} &&& {} & {} \\
	{} & {} && {} & {} & {} \\
	{} & {} &&& {} & {}
	\arrow["{\dot{r}}", from=1-1, to=1-2]
	\arrow["{h_1}"', from=1-1, to=2-1]
	\arrow[""{name=0, anchor=center, inner sep=0}, "{h'_1}", from=1-2, to=3-2]
	\arrow[""{name=1, anchor=center, inner sep=0}, "{h'_2}", shift left=3, curve={height=-24pt}, from=1-2, to=3-2]
	\arrow["{\dot{r}}", from=1-5, to=1-6]
	\arrow["{h_1}"', curve={height=6pt}, from=1-5, to=2-4]
	\arrow["{h_2}", from=1-5, to=2-5]
	\arrow["{h'_2}", from=1-6, to=3-6]
	\arrow["\SIGMA"{description}, draw=none, from=2-1, to=2-2]
	\arrow["{{{d_1}}}"', from=2-1, to=3-1]
	\arrow["\theta", Rightarrow, from=2-4, to=2-5]
	\arrow["{{{d_1}}}"', curve={height=6pt}, from=2-4, to=3-5]
	\arrow["\SIGMA"{description}, draw=none, from=2-5, to=2-6]
	\arrow["{{{d_2}}}", from=2-5, to=3-5]
	\arrow["e"', from=3-1, to=3-2]
	\arrow["e"', from=3-5, to=3-6]
	\arrow["{{\text{\Large $=$}}}"{description}, draw=none, from=1, to=2-4]
	\arrow["{{{\theta'}}}", shorten <=9pt, shorten >=6pt, Rightarrow, from=0, to=1]
\end{tikzcd}
\]
The 2-morphism of the middle of the representative of $\theta\circ {\bar{f}}$ is given by the diagram
\[\begin{tikzcd}
	{} & {} & {} & {} & {} & {} & {} \\
	{} & {} & {} & {} & {} & {} & {}
	\arrow["f", from=1-1, to=1-2]
	\arrow[""{name=0, anchor=center, inner sep=0}, equals, from=1-1, to=2-1]
	\arrow["{{{{\dot{g}}}}}", from=1-2, to=1-3]
	\arrow[""{name=1, anchor=center, inner sep=0}, "{{{h'_1}}}", from=1-3, to=1-4]
	\arrow[""{name=2, anchor=center, inner sep=0}, equals, from=1-3, to=2-3]
	\arrow[""{name=3, anchor=center, inner sep=0}, equals, from=1-4, to=2-4]
	\arrow["e"', from=1-5, to=1-4]
	\arrow["d"', from=1-6, to=1-5]
	\arrow[""{name=4, anchor=center, inner sep=0}, equals, from=1-6, to=2-6]
	\arrow["{{t_1u}}"', from=1-7, to=1-6]
	\arrow[""{name=5, anchor=center, inner sep=0}, equals, from=1-7, to=2-7]
	\arrow["f"', from=2-1, to=2-2]
	\arrow["{{{{\dot{g}}}}}"', from=2-2, to=2-3]
	\arrow[""{name=6, anchor=center, inner sep=0}, "{{{h'_2}}}"', from=2-3, to=2-4]
	\arrow["e", from=2-5, to=2-4]
	\arrow["d", from=2-6, to=2-5]
	\arrow["{{t_1u}}", from=2-7, to=2-6]
	\arrow["{\text{\normalsize $=$}}"{description}, draw=none, from=0, to=2]
	\arrow["{{{\theta'}}}", between={0.2}{0.8}, Rightarrow, from=1, to=6]
	\arrow["{{{\SIGMA^{\id}}}}"{description}, shift left, draw=none, from=3, to=4]
	\arrow["{\SIGMA^{\id}}"{description}, shift left, draw=none, from=4, to=5]
\end{tikzcd}\, .\]
This 2-morphism is an $\Omega$ 2-morphism. Indeed, it corresponds to a $\Sigma$-step between $\Sigma$-schemes of level 3 of the form $\begin{tikzcd}
	{S_3} & {S_4}
	\arrow["\mathbf{s}", squiggly, from=1-1, to=1-2]
\end{tikzcd}$ as in Diagram \eqref{eq:digcc} with $S_3$ and $S_4$ as indicated below that diagram.
Moreover, taking into account Definition \ref{def:horizontal}, and the original $\Sg$-step of type $u$ between $\Sg$-schemes of level 2, we see that  $\Omega\circ {\bar{f}}$ corresponds to the $\Sg$-path
 formed by the top line of the diagram
 \begin{equation}\label{eq:digcc}
   \begin{tikzcd}
	{S_1} & {S_2} & {S_3} & {S_4} & {S_5} & {S_6} \\
	& {R_1} & {R_2} & {R_3} & {R_4}
	\arrow["{{{{\bd}_1}}}", squiggly, from=1-1, to=1-2]
	\arrow[""{name=0, anchor=center, inner sep=0}, "\bu"', squiggly, from=1-1, to=2-2]
	\arrow["\bu", squiggly, from=1-2, to=1-3]
	\arrow["\bs", squiggly, from=1-3, to=1-4]
	\arrow[""{name=1, anchor=center, inner sep=0}, "{{{\bu\equiv \bs}}}"', squiggly, from=1-3, to=2-3]
	\arrow["\bu", squiggly, from=1-4, to=1-5]
	\arrow[""{name=2, anchor=center, inner sep=0}, "{{{\bu\equiv \bs}}}", squiggly, from=1-4, to=2-4]
	\arrow["{{{{\bd}_1}}}", squiggly, from=1-5, to=1-6]
	\arrow["{{{{\bd}_1}}}", squiggly, from=2-2, to=2-3]
	\arrow["\bs"', curve={height=18pt}, squiggly, from=2-2, to=2-5]
	\arrow["\equiv"', shift right=2, draw=none, from=2-2, to=2-5]
	\arrow["\bs", squiggly, from=2-3, to=2-4]
	\arrow[""{name=3, anchor=center, inner sep=0}, "\bu"', squiggly, from=2-5, to=1-6]
	\arrow["{{{{\bd}_1}}}"', squiggly, from=2-5, to=2-4]
	\arrow["\equiv"{description, pos=0.3}, shift left=3, draw=none, from=0, to=1]
	\arrow["\equiv"{description}, draw=none, from=1, to=2]
	\arrow["\equiv"{description, pos=0.7}, shift left=3, draw=none, from=2, to=3]
\end{tikzcd}\end{equation}
where
\newline
\(S_1=\hspace{-2mm}\adjustbox{scale=0.60}{\begin{tikzcd}
	&& {} & {} \\
	& {} & {} & {} \\
	{} & {} \\
	{} & {} & {} & {}
	\arrow[from=1-3, to=1-4]
	\arrow[from=1-3, to=2-3]
	\arrow[from=1-4, to=4-4]
	\arrow[from=2-2, to=2-3]
	\arrow[from=2-2, to=3-2]
	\arrow["\bullet"', shift right=5, draw=none, from=2-3, to=2-4]
	\arrow[""{name=0, anchor=center, inner sep=0}, from=2-3, to=4-3]
	\arrow[from=3-1, to=3-2]
	\arrow[""{name=1, anchor=center, inner sep=0}, from=3-1, to=4-1]
	\arrow[""{name=2, anchor=center, inner sep=0}, from=3-2, to=4-2]
	\arrow[from=4-1, to=4-2]
	\arrow[from=4-2, to=4-3]
	\arrow[from=4-3, to=4-4]
	\arrow["1"{description}, draw=none, from=1, to=2]
	\arrow["2"{description, pos=0.3}, draw=none, from=3-2, to=0]
\end{tikzcd}}\),
\(S_2=\adjustbox{scale=0.60}{\begin{tikzcd}
	&& {} & {} \\
	& {} & {} & {} \\
	{} & {} \\
	{} & {} & {} & {} \\
	{} & {} & {} & {}
	\arrow["r", from=1-3, to=1-4]
	\arrow["g"', from=1-3, to=2-3]
	\arrow[from=1-4, to=4-4]
	\arrow["s", from=2-2, to=2-3]
	\arrow["h"', from=2-2, to=3-2]
	\arrow["\bullet"', shift right=5, draw=none, from=2-3, to=2-4]
	\arrow[""{name=0, anchor=center, inner sep=0}, from=2-3, to=4-3]
	\arrow["t", from=3-1, to=3-2]
	\arrow[""{name=1, anchor=center, inner sep=0}, "k"', from=3-1, to=4-1]
	\arrow[""{name=2, anchor=center, inner sep=0}, "{k_1}", from=3-2, to=4-2]
	\arrow["{t_1}", from=4-1, to=4-2]
	\arrow[""{name=3, anchor=center, inner sep=0}, equals, from=4-1, to=5-1]
	\arrow["{s_1}", from=4-2, to=4-3]
	\arrow[""{name=4, anchor=center, inner sep=0}, equals, from=4-2, to=5-2]
	\arrow[from=4-3, to=4-4]
	\arrow[""{name=5, anchor=center, inner sep=0}, "{d_1}", from=4-3, to=5-3]
	\arrow[from=4-4, to=5-4]
	\arrow["{t_1}", from=5-1, to=5-2]
	\arrow["d", from=5-2, to=5-3]
	\arrow[from=5-3, to=5-4]
	\arrow["1"{description}, draw=none, from=1, to=2]
	\arrow["2"{description, pos=0.3}, draw=none, from=3-2, to=0]
	\arrow["\bullet"{description}, draw=none, from=3, to=4]
	\arrow["4"{description}, draw=none, from=4, to=5]
\end{tikzcd}}\),
\(S_3=\adjustbox{scale=0.60}{\begin{tikzcd}
	&& {} & {} \\
	& {} & {} & {} \\
	{} & {} \\
	{} & {} & {} & {} \\
	{} & {} & {} & {}
	\arrow["r", from=1-3, to=1-4]
	\arrow[""{name=0, anchor=center, inner sep=0}, "g"', from=1-3, to=2-3]
	\arrow[""{name=1, anchor=center, inner sep=0}, from=1-4, to=2-4]
	\arrow["s", from=2-2, to=2-3]
	\arrow["h"', from=2-2, to=3-2]
	\arrow[from=2-3, to=2-4]
	\arrow[""{name=2, anchor=center, inner sep=0}, from=2-3, to=4-3]
	\arrow[from=2-4, to=5-4]
	\arrow["t", from=3-1, to=3-2]
	\arrow[""{name=3, anchor=center, inner sep=0}, "k"', from=3-1, to=4-1]
	\arrow[""{name=4, anchor=center, inner sep=0}, from=3-2, to=4-2]
	\arrow[from=4-1, to=4-2]
	\arrow[""{name=5, anchor=center, inner sep=0}, equals, from=4-1, to=5-1]
	\arrow[from=4-2, to=4-3]
	\arrow[""{name=6, anchor=center, inner sep=0}, equals, from=4-2, to=5-2]
	\arrow["5"{description}, shift left=5, draw=none, from=4-3, to=4-4]
	\arrow[""{name=7, anchor=center, inner sep=0}, from=4-3, to=5-3]
	\arrow[from=5-1, to=5-2]
	\arrow[from=5-2, to=5-3]
	\arrow[from=5-3, to=5-4]
	\arrow["\bullet"{description}, draw=none, from=0, to=1]
	\arrow["1"{description}, draw=none, from=3, to=4]
	\arrow["2"{description, pos=0.3}, draw=none, from=3-2, to=2]
	\arrow["\bullet"{description}, draw=none, from=5, to=6]
	\arrow["4"{description}, draw=none, from=6, to=7]
\end{tikzcd}}\),
\(S_4=\adjustbox{scale=0.60}{\begin{tikzcd}
	&& {} & {} \\
	& {} & {} & {} \\
	{} & {} \\
	{} & {} & {} & {} \\
	{} & {} & {} & {}
	\arrow["r", from=1-3, to=1-4]
	\arrow[""{name=0, anchor=center, inner sep=0}, "g"', from=1-3, to=2-3]
	\arrow[""{name=1, anchor=center, inner sep=0}, from=1-4, to=2-4]
	\arrow["s", from=2-2, to=2-3]
	\arrow["h"', from=2-2, to=3-2]
	\arrow[from=2-3, to=2-4]
	\arrow[""{name=2, anchor=center, inner sep=0}, from=2-3, to=4-3]
	\arrow[from=2-4, to=5-4]
	\arrow["t", from=3-1, to=3-2]
	\arrow[""{name=3, anchor=center, inner sep=0}, "k"', from=3-1, to=4-1]
	\arrow[""{name=4, anchor=center, inner sep=0}, from=3-2, to=4-2]
	\arrow[from=4-1, to=4-2]
	\arrow[""{name=5, anchor=center, inner sep=0}, equals, from=4-1, to=5-1]
	\arrow[from=4-2, to=4-3]
	\arrow[""{name=6, anchor=center, inner sep=0}, equals, from=4-2, to=5-2]
	\arrow["7"{description}, shift left=5, draw=none, from=4-3, to=4-4]
	\arrow[""{name=7, anchor=center, inner sep=0}, from=4-3, to=5-3]
	\arrow[from=5-1, to=5-2]
	\arrow[from=5-2, to=5-3]
	\arrow[from=5-3, to=5-4]
	\arrow["\bullet"{description}, draw=none, from=0, to=1]
	\arrow["1"{description}, draw=none, from=3, to=4]
	\arrow["3"{description, pos=0.3}, draw=none, from=3-2, to=2]
	\arrow["\bullet"{description}, draw=none, from=5, to=6]
	\arrow["6"{description}, draw=none, from=6, to=7]
\end{tikzcd}}\),
\newline
\(S_5=\adjustbox{scale=0.60}{\begin{tikzcd}
	&& {} & {} \\
	& {} & {} & {} \\
	{} & {} & {} \\
	{} & {} & {} & {} \\
	{} & {} & {} & {}
	\arrow[from=1-3, to=1-4]
	\arrow[from=1-3, to=2-3]
	\arrow[from=1-4, to=4-4]
	\arrow[from=2-2, to=2-3]
	\arrow[from=2-2, to=3-2]
	\arrow["\bullet"', shift right=5, draw=none, from=2-3, to=2-4]
	\arrow[""{name=0, anchor=center, inner sep=0}, from=2-3, to=4-3]
	\arrow[from=3-1, to=3-2]
	\arrow[""{name=1, anchor=center, inner sep=0}, from=3-1, to=4-1]
	\arrow[""{name=2, anchor=center, inner sep=0}, from=3-2, to=4-2]
	\arrow[from=4-1, to=4-2]
	\arrow[""{name=3, anchor=center, inner sep=0}, equals, from=4-1, to=5-1]
	\arrow[from=4-2, to=4-3]
	\arrow[""{name=4, anchor=center, inner sep=0}, equals, from=4-2, to=5-2]
	\arrow[from=4-3, to=3-3]
	\arrow[from=4-3, to=4-4]
	\arrow[""{name=5, anchor=center, inner sep=0}, from=4-3, to=5-3]
	\arrow[from=4-4, to=5-4]
	\arrow[from=5-1, to=5-2]
	\arrow[from=5-2, to=5-3]
	\arrow[from=5-3, to=5-4]
	\arrow["1"{description}, draw=none, from=1, to=2]
	\arrow["3"{description, pos=0.3}, draw=none, from=3-2, to=0]
	\arrow["\bullet"{description}, draw=none, from=3, to=4]
	\arrow["6"{description}, draw=none, from=4, to=5]
\end{tikzcd}}\),
\(S_6=\hspace{-2mm}\adjustbox{scale=0.60}{\begin{tikzcd}
	&& {} & {} \\
	& {} & {} & {} \\
	{} & {} \\
	{} & {} & {} & {}
	\arrow[from=1-3, to=1-4]
	\arrow[from=1-3, to=2-3]
	\arrow[from=1-4, to=4-4]
	\arrow[from=2-2, to=2-3]
	\arrow[from=2-2, to=3-2]
	\arrow["\bullet"', shift right=5, draw=none, from=2-3, to=2-4]
	\arrow[""{name=0, anchor=center, inner sep=0}, from=2-3, to=4-3]
	\arrow[from=3-1, to=3-2]
	\arrow[""{name=1, anchor=center, inner sep=0}, from=3-1, to=4-1]
	\arrow[""{name=2, anchor=center, inner sep=0}, from=3-2, to=4-2]
	\arrow[from=4-1, to=4-2]
	\arrow[from=4-2, to=4-3]
	\arrow[from=4-3, to=4-4]
	\arrow["1"{description}, draw=none, from=1, to=2]
	\arrow["3"{description, pos=0.3}, draw=none, from=3-2, to=0]
\end{tikzcd}}
\),
\(R_1=\hspace{-2mm}\adjustbox{scale=0.60}{\begin{tikzcd}
	&& {} & {} \\
	& {} & {} & {} \\
	{} & {} \\
	{} & {} & {} & {}
	\arrow[from=1-3, to=1-4]
	\arrow[""{name=0, anchor=center, inner sep=0}, from=1-3, to=2-3]
	\arrow[""{name=1, anchor=center, inner sep=0}, from=1-4, to=2-4]
	\arrow[from=2-2, to=2-3]
	\arrow[from=2-2, to=3-2]
	\arrow[from=2-3, to=2-4]
	\arrow[""{name=2, anchor=center, inner sep=0}, from=2-3, to=4-3]
	\arrow[""{name=3, anchor=center, inner sep=0}, from=2-4, to=4-4]
	\arrow[from=3-1, to=3-2]
	\arrow[""{name=4, anchor=center, inner sep=0}, from=3-1, to=4-1]
	\arrow[""{name=5, anchor=center, inner sep=0}, from=3-2, to=4-2]
	\arrow[from=4-1, to=4-2]
	\arrow[from=4-2, to=4-3]
	\arrow[from=4-3, to=4-4]
	\arrow["\bullet"{description}, draw=none, from=0, to=1]
	\arrow["\bullet"{description}, draw=none, from=2, to=3]
	\arrow["1"{description}, draw=none, from=4, to=5]
	\arrow["2"{description, pos=0.3}, draw=none, from=3-2, to=2]
\end{tikzcd}
}\),
 \(R_2=\hspace{-2mm}\adjustbox{scale=0.60}{\begin{tikzcd}
	&& {} & {} \\
	& {} & {} & {} \\
	{} & {} \\
	{} & {} & {} & {} \\
	{} & {} & {} & {}
	\arrow[from=1-3, to=1-4]
	\arrow[""{name=0, anchor=center, inner sep=0}, from=1-3, to=2-3]
	\arrow[""{name=1, anchor=center, inner sep=0}, from=1-4, to=2-4]
	\arrow[from=2-2, to=2-3]
	\arrow[from=2-2, to=3-2]
	\arrow[from=2-3, to=2-4]
	\arrow[""{name=2, anchor=center, inner sep=0}, from=2-3, to=4-3]
	\arrow[""{name=3, anchor=center, inner sep=0}, from=2-4, to=4-4]
	\arrow[from=3-1, to=3-2]
	\arrow[""{name=4, anchor=center, inner sep=0}, from=3-1, to=4-1]
	\arrow[""{name=5, anchor=center, inner sep=0}, from=3-2, to=4-2]
	\arrow[from=4-1, to=4-2]
	\arrow[""{name=6, anchor=center, inner sep=0}, equals, from=4-1, to=5-1]
	\arrow[from=4-2, to=4-3]
	\arrow[""{name=7, anchor=center, inner sep=0}, equals, from=4-2, to=5-2]
	\arrow[from=4-3, to=4-4]
	\arrow[""{name=8, anchor=center, inner sep=0}, from=4-3, to=5-3]
	\arrow[""{name=9, anchor=center, inner sep=0}, from=4-4, to=5-4]
	\arrow[from=5-1, to=5-2]
	\arrow[from=5-2, to=5-3]
	\arrow[from=5-3, to=5-4]
	\arrow["\bullet"{description}, draw=none, from=0, to=1]
	\arrow["\bullet"{description}, draw=none, from=2, to=3]
	\arrow["1"{description}, draw=none, from=4, to=5]
	\arrow["2"{description, pos=0.3}, draw=none, from=3-2, to=2]
	\arrow["\bullet"{description}, draw=none, from=6, to=7]
	\arrow["4"{description}, draw=none, from=7, to=8]
	\arrow["\bullet"{description}, draw=none, from=8, to=9]
\end{tikzcd}}
 \),
\newline
 \(R_3=\adjustbox{scale=0.60}{\begin{tikzcd}
	&& {} & {} \\
	& {} & {} & {} \\
	{} & {} \\
	{} & {} & {} & {} \\
	{} & {} & {} & {}
	\arrow[from=1-3, to=1-4]
	\arrow[""{name=0, anchor=center, inner sep=0}, from=1-3, to=2-3]
	\arrow[""{name=1, anchor=center, inner sep=0}, from=1-4, to=2-4]
	\arrow[from=2-2, to=2-3]
	\arrow[from=2-2, to=3-2]
	\arrow[from=2-3, to=2-4]
	\arrow[""{name=2, anchor=center, inner sep=0}, from=2-3, to=4-3]
	\arrow[""{name=3, anchor=center, inner sep=0}, from=2-4, to=4-4]
	\arrow[from=3-1, to=3-2]
	\arrow[""{name=4, anchor=center, inner sep=0}, from=3-1, to=4-1]
	\arrow[""{name=5, anchor=center, inner sep=0}, from=3-2, to=4-2]
	\arrow[from=4-1, to=4-2]
	\arrow[""{name=6, anchor=center, inner sep=0}, equals, from=4-1, to=5-1]
	\arrow[from=4-2, to=4-3]
	\arrow[""{name=7, anchor=center, inner sep=0}, equals, from=4-2, to=5-2]
	\arrow[from=4-3, to=4-4]
	\arrow[""{name=8, anchor=center, inner sep=0}, from=4-3, to=5-3]
	\arrow[""{name=9, anchor=center, inner sep=0}, from=4-4, to=5-4]
	\arrow[from=5-1, to=5-2]
	\arrow[from=5-2, to=5-3]
	\arrow[from=5-3, to=5-4]
	\arrow["\bullet"{description}, draw=none, from=0, to=1]
	\arrow["\bullet"{description}, draw=none, from=2, to=3]
	\arrow["1"{description}, draw=none, from=4, to=5]
	\arrow["3"{description, pos=0.3}, draw=none, from=3-2, to=2]
	\arrow["\bullet"{description}, draw=none, from=6, to=7]
	\arrow["6"{description}, draw=none, from=7, to=8]
	\arrow["\bullet"{description}, draw=none, from=8, to=9]
\end{tikzcd}}
\)
and
\(R_4=\hspace{-2mm}\adjustbox{scale=0.60}{\begin{tikzcd}
	&& {} & {} \\
	& {} & {} & {} \\
	{} & {} \\
	{} & {} & {} & {}
	\arrow[from=1-3, to=1-4]
	\arrow[""{name=0, anchor=center, inner sep=0}, from=1-3, to=2-3]
	\arrow[""{name=1, anchor=center, inner sep=0}, from=1-4, to=2-4]
	\arrow[from=2-2, to=2-3]
	\arrow[from=2-2, to=3-2]
	\arrow[from=2-3, to=2-4]
	\arrow[""{name=2, anchor=center, inner sep=0}, from=2-3, to=4-3]
	\arrow[""{name=3, anchor=center, inner sep=0}, from=2-4, to=4-4]
	\arrow[from=3-1, to=3-2]
	\arrow[""{name=4, anchor=center, inner sep=0}, from=3-1, to=4-1]
	\arrow[""{name=5, anchor=center, inner sep=0}, from=3-2, to=4-2]
	\arrow[from=4-1, to=4-2]
	\arrow[from=4-2, to=4-3]
	\arrow[from=4-3, to=4-4]
	\arrow["\bullet"{description}, draw=none, from=0, to=1]
	\arrow["\bullet"{description}, draw=none, from=2, to=3]
	\arrow["1"{description}, draw=none, from=4, to=5]
	\arrow["3"{description, pos=0.3}, draw=none, from=3-2, to=2]
\end{tikzcd}
}\).

The equivalences of $\Sigma$-paths indicated in  diagram \eqref{eq:digcc} are clear from Lemma \ref{lem:Sigma-steps} and Proposition \ref{pro:adend}.
Therefore, we conclude that $\Omega\circ \overline{f}$ is given by a $\Sigma$-path of the form indicated in the table.

(2) Let $\Omega$ be the $\Omega$ 2-cell corresponding to the $\Sg$-step
\[\adjustbox{scale=0.80}{\begin{tikzcd}
	& {} & {} &&& {} & {} \\
	{} & {} & {} && {} & {} & {} \\
	{} & {} & {} && {} & {} & {}
	\arrow["r", from=1-2, to=1-3]
	\arrow[""{name=0, anchor=center, inner sep=0}, "g"', from=1-2, to=2-2]
	\arrow[""{name=1, anchor=center, inner sep=0}, "{{g'}}", from=1-3, to=2-3]
	\arrow["r", from=1-6, to=1-7]
	\arrow[""{name=2, anchor=center, inner sep=0}, "g"', from=1-6, to=2-6]
	\arrow[""{name=3, anchor=center, inner sep=0}, "{{g'}}", from=1-7, to=2-7]
	\arrow["s", from=2-1, to=2-2]
	\arrow[""{name=4, anchor=center, inner sep=0}, "h"', from=2-1, to=3-1]
	\arrow["{r'}"', from=2-2, to=2-3]
	\arrow["d", between={0.3}{0.7}, squiggly, from=2-3, to=2-5]
	\arrow[""{name=5, anchor=center, inner sep=0}, "{{h_1}}", from=2-3, to=3-3]
	\arrow["s", from=2-5, to=2-6]
	\arrow[""{name=6, anchor=center, inner sep=0}, "h"', from=2-5, to=3-5]
	\arrow["{r'}"', from=2-6, to=2-7]
	\arrow[""{name=7, anchor=center, inner sep=0}, "{{h_2}}", from=2-7, to=3-7]
	\arrow["{{s_1}}"', from=3-1, to=3-3]
	\arrow["{{s_2}}"', from=3-5, to=3-7]
	\arrow["1"{description}, draw=none, from=0, to=1]
	\arrow["1"{description}, draw=none, from=2, to=3]
	\arrow["2"{description}, draw=none, from=4, to=5]
	\arrow["3"{description}, draw=none, from=6, to=7]
\end{tikzcd}}\, .\]
Thus, here $(l_i,m_i)=(h_ig',s_i)$.
Let it be represented by the 2-morphism
\[\begin{tikzcd}
	{} & {} & {} & {} & {} & {} \\
	&&& {} & {} & {} \\
	{} & {} & {} & {} & {} & {}
	\arrow["f", from=1-1, to=1-2]
	\arrow[""{name=0, anchor=center, inner sep=0}, equals, from=1-1, to=3-1]
	\arrow["{g'}", from=1-2, to=1-3]
	\arrow["{h_1}", from=1-3, to=1-4]
	\arrow[""{name=1, anchor=center, inner sep=0}, equals, from=1-3, to=3-3]
	\arrow[""{name=2, anchor=center, inner sep=0}, "{{{d_1}}}"', from=1-4, to=2-4]
	\arrow["\theta"', shorten <=15pt, shorten >=11pt, Rightarrow, from=1-4, to=3-3]
	\arrow["{s_1}"', from=1-5, to=1-4]
	\arrow[""{name=3, anchor=center, inner sep=0}, equals, from=1-5, to=2-5]
	\arrow["t"', from=1-6, to=1-5]
	\arrow[""{name=4, anchor=center, inner sep=0}, equals, from=1-6, to=2-6]
	\arrow["d", from=2-5, to=2-4]
	\arrow[""{name=5, anchor=center, inner sep=0}, equals, from=2-5, to=3-5]
	\arrow["t", from=2-6, to=2-5]
	\arrow[""{name=6, anchor=center, inner sep=0}, equals, from=2-6, to=3-6]
	\arrow["f"', from=3-1, to=3-2]
	\arrow["{g'}"', from=3-2, to=3-3]
	\arrow["{h_2}"', from=3-3, to=3-4]
	\arrow[""{name=7, anchor=center, inner sep=0}, "{{{d_2}}}", from=3-4, to=2-4]
	\arrow["{s_2}", from=3-5, to=3-4]
	\arrow["t", from=3-6, to=3-5]
	\arrow[shorten <=26pt, shorten >=26pt, equals, from=0, to=1]
	\arrow["\SIGMA"{description}, draw=none, from=2, to=3]
	\arrow["\SIGMA"{description}, draw=none, from=3, to=4]
	\arrow["\SIGMA"{description}, draw=none, from=5, to=6]
	\arrow["\SIGMA"{description}, draw=none, from=7, to=5]
\end{tikzcd}\]
with $\theta$ invertible. Now, in order to compose with $\bar{k}=(k,u)$, just consider the $\Sg$-squares
\[\begin{tikzcd}
	{} & {} \\
	{} & {}
	\arrow["t", from=1-1, to=1-2]
	\arrow[""{name=0, anchor=center, inner sep=0}, "k"', from=1-1, to=2-1]
	\arrow[""{name=1, anchor=center, inner sep=0}, "{k_0}", from=1-2, to=2-2]
	\arrow["{t'}"', from=2-1, to=2-2]
	\arrow["\SIGMA"{description}, draw=none, from=0, to=1]
\end{tikzcd}
\quad \text{and} \quad
\begin{tikzcd}
	{} & {} \\
	{} & {}
	\arrow["d", from=1-1, to=1-2]
	\arrow[""{name=0, anchor=center, inner sep=0}, "{k_0}"', from=1-1, to=2-1]
	\arrow[""{name=1, anchor=center, inner sep=0}, "{k'}", from=1-2, to=2-2]
	\arrow["{d'}"', from=2-1, to=2-2]
	\arrow["\SIGMA"{description}, draw=none, from=0, to=1]
\end{tikzcd}
\]
and their horizontal composition.
Following Definition \ref{def:horizontal}, we obtain the middle 2-morphism of the desired composition $\bar{k}\circ\Omega$:
\[\begin{tikzcd}
	{} & {} & {} & {} & {} & {} & {} & {} \\
	{} & {} & {} & {} & {} & {} & {} & {}
	\arrow["f", from=1-1, to=1-2]
	\arrow[""{name=0, anchor=center, inner sep=0}, equals, from=1-1, to=2-1]
	\arrow["{{{g'}}}", from=1-2, to=1-3]
	\arrow["{{{h_1}}}", from=1-3, to=1-4]
	\arrow[""{name=1, anchor=center, inner sep=0}, equals, from=1-3, to=2-3]
	\arrow["{{{d_1}}}", from=1-4, to=1-5]
	\arrow["\theta", between={0}{0.9}, Rightarrow, from=1-4, to=2-4]
	\arrow["{{{k'}}}", from=1-5, to=1-6]
	\arrow[""{name=2, anchor=center, inner sep=0}, equals, from=1-5, to=2-5]
	\arrow[""{name=3, anchor=center, inner sep=0}, equals, from=1-6, to=2-6]
	\arrow["{{{d't'}}}"', from=1-7, to=1-6]
	\arrow[""{name=4, anchor=center, inner sep=0}, equals, from=1-7, to=2-7]
	\arrow["u"', from=1-8, to=1-7]
	\arrow[""{name=5, anchor=center, inner sep=0}, equals, from=1-8, to=2-8]
	\arrow["f"', from=2-1, to=2-2]
	\arrow["{{{g'}}}"', from=2-2, to=2-3]
	\arrow["{{{h_2}}}"', from=2-3, to=2-4]
	\arrow["{{{d_2}}}"', from=2-4, to=2-5]
	\arrow["{{{k'}}}", from=2-6, to=2-5]
	\arrow["{{{d't'}}}", from=2-7, to=2-6]
	\arrow["u", from=2-8, to=2-7]
	\arrow["{\text{\normalsize $=$}}"{description}, draw=none, from=0, to=1]
	\arrow["{\SIGMA^\id}"{description}, draw=none, from=2, to=3]
	\arrow["{\SIGMA^\id}"{description}, draw=none, from=3, to=4]
	\arrow["{\SIGMA^\id}"{description}, draw=none, from=4, to=5]
\end{tikzcd}\]
This 2-morphism is an $\Omega$ 2-morphism. Indeed, it corresponds to the following  $\Sigma$-step of type $\mathbf{s}$
\[
\adjustbox{scale=0.70}{\begin{tikzcd}
	&& {} & {} \\
	& {} & {} & {} \\
	& {} & {} & {} \\
	{} & {} && {} \\
	{} & {} && {}
	\arrow["r", from=1-3, to=1-4]
	\arrow[""{name=0, anchor=center, inner sep=0}, "g"', from=1-3, to=2-3]
	\arrow[""{name=1, anchor=center, inner sep=0}, "{{g'}}", from=1-4, to=2-4]
	\arrow["s", from=2-2, to=2-3]
	\arrow[""{name=2, anchor=center, inner sep=0}, "h"', from=2-2, to=3-2]
	\arrow["{{r'}}"', from=2-3, to=2-4]
	\arrow[""{name=3, anchor=center, inner sep=0}, "{{h_1}}", from=2-4, to=3-4]
	\arrow["{{s_1}}"', from=3-2, to=3-4]
	\arrow[""{name=4, anchor=center, inner sep=0}, equals, from=3-2, to=4-2]
	\arrow[""{name=5, anchor=center, inner sep=0}, "{{d_1}}", from=3-4, to=4-4]
	\arrow["t", from=4-1, to=4-2]
	\arrow[""{name=6, anchor=center, inner sep=0}, "k"', from=4-1, to=5-1]
	\arrow["d"', from=4-2, to=4-4]
	\arrow[""{name=7, anchor=center, inner sep=0}, "{{k_0}}", from=4-2, to=5-2]
	\arrow[""{name=8, anchor=center, inner sep=0}, "{{k'}}", from=4-4, to=5-4]
	\arrow["{{t'}}"', from=5-1, to=5-2]
	\arrow["{{d'}}"', from=5-2, to=5-4]
	\arrow["\SIGMA"{description}, draw=none, from=0, to=1]
	\arrow["\SIGMA"{description}, draw=none, from=2, to=3]
	\arrow["\SIGMA"{description}, draw=none, from=4, to=5]
	\arrow["\SIGMA"{description}, draw=none, from=6, to=7]
	\arrow["\SIGMA"{description}, draw=none, from=7, to=8]
\end{tikzcd}}
\begin{tikzcd}
	{} & {}
	\arrow["\bs", between={0.1}{1}, squiggly, from=1-1, to=1-2]
\end{tikzcd}
\adjustbox{scale=0.70}{\begin{tikzcd}
	&& {} & {} \\
	& {} & {} & {} \\
	& {} & {} & {} \\
	{} & {} && {} \\
	{} & {} && {}
	\arrow["r", from=1-3, to=1-4]
	\arrow[""{name=0, anchor=center, inner sep=0}, "g"', from=1-3, to=2-3]
	\arrow[""{name=1, anchor=center, inner sep=0}, "{{g'}}", from=1-4, to=2-4]
	\arrow["s", from=2-2, to=2-3]
	\arrow[""{name=2, anchor=center, inner sep=0}, "h"', from=2-2, to=3-2]
	\arrow["{{r'}}"', from=2-3, to=2-4]
	\arrow[""{name=3, anchor=center, inner sep=0}, "{h_2}", from=2-4, to=3-4]
	\arrow["{{s_2}}"', from=3-2, to=3-4]
	\arrow[""{name=4, anchor=center, inner sep=0}, equals, from=3-2, to=4-2]
	\arrow[""{name=5, anchor=center, inner sep=0}, "{d_2}", from=3-4, to=4-4]
	\arrow["t", from=4-1, to=4-2]
	\arrow[""{name=6, anchor=center, inner sep=0}, "k"', from=4-1, to=5-1]
	\arrow["d"', from=4-2, to=4-4]
	\arrow[""{name=7, anchor=center, inner sep=0}, "{{k_0}}", from=4-2, to=5-2]
	\arrow[""{name=8, anchor=center, inner sep=0}, "{{k'}}", from=4-4, to=5-4]
	\arrow["{{t'}}"', from=5-1, to=5-2]
	\arrow["{{d'}}"', from=5-2, to=5-4]
	\arrow["\SIGMA"{description}, draw=none, from=0, to=1]
	\arrow["\SIGMA"{description}, draw=none, from=2, to=3]
	\arrow["\SIGMA"{description}, draw=none, from=4, to=5]
	\arrow["\SIGMA"{description}, draw=none, from=6, to=7]
	\arrow["\SIGMA"{description}, draw=none, from=7, to=8]
\end{tikzcd}}\, .
\]
Observe that, in this case, the $\Omega$ 2-cells $\Omega_i$, $i=1,2$, as in Definition \ref{def:horizontal}, just reduce to the basic $\Omega$ 2-cell determined by the $\Sg$-path of level 1
\[
\adjustbox{scale=0.80}{\begin{tikzcd}
	{} & {} \\
	{} & {}
	\arrow["{s_it}", from=1-1, to=1-2]
	\arrow[""{name=0, anchor=center, inner sep=0}, "k"', from=1-1, to=2-1]
	\arrow[""{name=1, anchor=center, inner sep=0}, from=1-2, to=2-2]
	\arrow[from=2-1, to=2-2]
	\arrow["\dot{\SIGMA}"{description}, draw=none, from=0, to=1]
\end{tikzcd}}
\begin{tikzcd}
	{} & {}
	\arrow[between={0.2}{0.8}, squiggly, from=1-1, to=1-2]
\end{tikzcd}
\adjustbox{scale=0.70}{\begin{tikzcd}
	{} & {} & {} \\
	{} & {} & {} \\
	{} & {} & {}
	\arrow["t", from=1-1, to=1-2]
	\arrow[""{name=0, anchor=center, inner sep=0}, equals, from=1-1, to=2-1]
	\arrow["{{{s_i}}}", from=1-2, to=1-3]
	\arrow[""{name=1, anchor=center, inner sep=0}, equals, from=1-2, to=2-2]
	\arrow[""{name=2, anchor=center, inner sep=0}, "{{{d_i}}}", from=1-3, to=2-3]
	\arrow["t"', from=2-1, to=2-2]
	\arrow[""{name=3, anchor=center, inner sep=0}, "k"', from=2-1, to=3-1]
	\arrow["d"', from=2-2, to=2-3]
	\arrow[""{name=4, anchor=center, inner sep=0}, "{k_0}"{description}, from=2-2, to=3-2]
	\arrow[""{name=5, anchor=center, inner sep=0}, "{{{k'}}}", from=2-3, to=3-3]
	\arrow["{{{t'}}}"', from=3-1, to=3-2]
	\arrow["{{{d'}}}"', from=3-2, to=3-3]
	\arrow["\SIGMA"{description}, draw=none, from=0, to=1]
	\arrow["\SIGMA"{description}, draw=none, from=1, to=2]
	\arrow["\SIGMA"{description}, draw=none, from=3, to=4]
	\arrow["\SIGMA"{description}, draw=none, from=4, to=5]
\end{tikzcd}}\, ,
\]
or, equivalently, determined by the following $\Sg$-step of level 3:
\[
\adjustbox{scale=0.70}{\begin{tikzcd}
	&& {} & {} \\
	& {} & {} & {} \\
	{} & {} && {} \\
	{} &&& {}
	\arrow["r", from=1-3, to=1-4]
	\arrow[""{name=0, anchor=center, inner sep=0}, "g"', from=1-3, to=2-3]
	\arrow[""{name=1, anchor=center, inner sep=0}, "{{g'}}", from=1-4, to=2-4]
	\arrow["s", from=2-2, to=2-3]
	\arrow[""{name=2, anchor=center, inner sep=0}, "h"', from=2-2, to=3-2]
	\arrow["{{r'}}", from=2-3, to=2-4]
	\arrow[""{name=3, anchor=center, inner sep=0}, "{{h_i}}", from=2-4, to=3-4]
	\arrow["t", from=3-1, to=3-2]
	\arrow[""{name=4, anchor=center, inner sep=0}, "k"', from=3-1, to=4-1]
	\arrow["{{s_i}}"', from=3-2, to=3-4]
	\arrow[""{name=5, anchor=center, inner sep=0}, from=3-4, to=4-4]
	\arrow[from=4-1, to=4-4]
	\arrow["\SIGMA"{description}, draw=none, from=0, to=1]
	\arrow["\SIGMA"{description}, draw=none, from=2, to=3]
	\arrow["{\dot{\SIGMA}}"{description}, draw=none, from=4, to=5]
\end{tikzcd}}
\begin{tikzcd}
	{} & {}
	\arrow["\bd", between={0.1}{1}, squiggly, from=1-1, to=1-2]
\end{tikzcd}
\adjustbox{scale=0.70}{\begin{tikzcd}
	&& {} & {} \\
	& {} & {} & {} \\
	{} & {} && {} \\
	{} & {} && {} \\
	{} & {} && {}
	\arrow["r", from=1-3, to=1-4]
	\arrow[""{name=0, anchor=center, inner sep=0}, "g"', from=1-3, to=2-3]
	\arrow[""{name=1, anchor=center, inner sep=0}, "{g'}", from=1-4, to=2-4]
	\arrow["s", from=2-2, to=2-3]
	\arrow[""{name=2, anchor=center, inner sep=0}, "h"', from=2-2, to=3-2]
	\arrow["{r'}", from=2-3, to=2-4]
	\arrow[""{name=3, anchor=center, inner sep=0}, "{h_i}", from=2-4, to=3-4]
	\arrow["t", from=3-1, to=3-2]
	\arrow[""{name=4, anchor=center, inner sep=0}, equals, from=3-1, to=4-1]
	\arrow["{s_i}"', from=3-2, to=3-4]
	\arrow[""{name=5, anchor=center, inner sep=0}, equals, from=3-2, to=4-2]
	\arrow[""{name=6, anchor=center, inner sep=0}, "{d_i}", from=3-4, to=4-4]
	\arrow["t"', from=4-1, to=4-2]
	\arrow[""{name=7, anchor=center, inner sep=0}, "k"', from=4-1, to=5-1]
	\arrow["d"', from=4-2, to=4-4]
	\arrow[""{name=8, anchor=center, inner sep=0}, "{{k_0}}", from=4-2, to=5-2]
	\arrow[""{name=9, anchor=center, inner sep=0}, "{k'}", from=4-4, to=5-4]
	\arrow["{{t'}}"', from=5-1, to=5-2]
	\arrow["{d'}"', from=5-2, to=5-4]
	\arrow["\SIGMA"{description}, draw=none, from=0, to=1]
	\arrow["\SIGMA"{description}, draw=none, from=2, to=3]
	\arrow["{\dot{\SIGMA}}"{description}, draw=none, from=4, to=5]
	\arrow["\SIGMA"{description}, draw=none, from=5, to=6]
	\arrow["\SIGMA"{description}, draw=none, from=7, to=8]
	\arrow["\SIGMA"{description}, draw=none, from=8, to=9]
\end{tikzcd}}
\]
In conclusion, we obtain the $\Sg$-path as in the first row of the second table.

 The case of the second row of the second table, where $\Omega$ corresponds to the $\Sg$-step of type $u$, works similarly.

Finally, it is clear that every of the four $\Sg$-paths given in the above tables leading to $\Omega \circ {\bar{f}}$ and ${\bar{k}}\circ \Omega$ are $\Sg$-paths of interest.
\end{proof}
The following corollary is not quite an immediate special case of part (2) of \cref{pro:alpha.f}, since the desired left border of the $\Sigma$-path of interest is a little different. This is useful for proving \cref{pro:natural-iso}.

\begin{corollary}\label{cor:Omega-h}
 Let $\Omega\colon (l_1f,m_1)\to (l_2f,m_2)$ be the $\Omega$ 2-cell (with the indicated domain and codomain) corresponding to the $\Sg$-step between $\Sigma$-schemes of level 2 with left border $(r,g,s,1)$ of one of the following types:
 \begin{enumerate}
 \item[(1)]
 \(\adjustbox{scale=0.60}{\begin{tikzcd}
	& {} & {} &&& {} & {} \\
	{} & {} & {} && {} & {} & {} \\
	{} & {} & {} && {} & {} & {}
	\arrow["r", from=1-2, to=1-3]
	\arrow[""{name=0, anchor=center, inner sep=0}, "g"', from=1-2, to=2-2]
	\arrow[""{name=1, anchor=center, inner sep=0}, "{{g'}}", from=1-3, to=2-3]
	\arrow["r", from=1-6, to=1-7]
	\arrow[""{name=2, anchor=center, inner sep=0}, "g"', from=1-6, to=2-6]
	\arrow[""{name=3, anchor=center, inner sep=0}, "{{g'}}", from=1-7, to=2-7]
	\arrow["s", from=2-1, to=2-2]
	\arrow[""{name=4, anchor=center, inner sep=0}, equals, from=2-1, to=3-1]
	\arrow["{{r'}}"', from=2-2, to=2-3]
	\arrow["d", between={0.3}{0.7}, squiggly, from=2-3, to=2-5]
	\arrow[""{name=5, anchor=center, inner sep=0}, "{{h_1}}", from=2-3, to=3-3]
	\arrow["s", from=2-5, to=2-6]
	\arrow[""{name=6, anchor=center, inner sep=0}, equals, from=2-5, to=3-5]
	\arrow["{{r'}}", from=2-6, to=2-7]
	\arrow[""{name=7, anchor=center, inner sep=0}, "{{h_2}}", from=2-7, to=3-7]
	\arrow["{{s_1}}"', from=3-1, to=3-3]
	\arrow["{{s_2}}"', from=3-5, to=3-7]
	\arrow["1"{description}, draw=none, from=0, to=1]
	\arrow["1"{description}, draw=none, from=2, to=3]
	\arrow["2"{description}, draw=none, from=4, to=5]
	\arrow["3"{description}, draw=none, from=6, to=7]
\end{tikzcd}}\)
\hspace{3mm} with  $l_i=h_ig'$ and $m_i=s_i$;
\item[(2)]
\(\adjustbox{scale=0.60}{\begin{tikzcd}
	& {} & {} &&& {} & {} \\
	{} & {} & {} && {} & {} & {} \\
	{} & {} & {} && {} & {} & {}
	\arrow["r", from=1-2, to=1-3]
	\arrow["g"', from=1-2, to=2-2]
	\arrow["{{{h_1}}}", from=1-3, to=3-3]
	\arrow["r", from=1-6, to=1-7]
	\arrow["g"', from=1-6, to=2-6]
	\arrow["{{{h_2}}}", from=1-7, to=3-7]
	\arrow["s", from=2-1, to=2-2]
	\arrow[""{name=0, anchor=center, inner sep=0}, equals, from=2-1, to=3-1]
	\arrow["2"{description}, draw=none, from=2-2, to=2-3]
	\arrow[""{name=1, anchor=center, inner sep=0}, "{{h'}}", from=2-2, to=3-2]
	\arrow["u", between={0.3}{0.7}, squiggly, from=2-3, to=2-5]
	\arrow["s", from=2-5, to=2-6]
	\arrow[""{name=2, anchor=center, inner sep=0}, equals, from=2-5, to=3-5]
	\arrow["3"{description}, draw=none, from=2-6, to=2-7]
	\arrow[""{name=3, anchor=center, inner sep=0}, "{{h'}}", from=2-6, to=3-6]
	\arrow["{{s'}}"', from=3-1, to=3-2]
	\arrow["{{{s_1}}}"', from=3-2, to=3-3]
	\arrow["{{s'}}"', from=3-5, to=3-6]
	\arrow["{{{s_2}}}"', from=3-6, to=3-7]
	\arrow["1"{description}, draw=none, from=0, to=1]
	\arrow["1"{description}, draw=none, from=2, to=3]
\end{tikzcd}}\)
\hspace{3mm} with $l_i=h_i$ and $m_i=s_is'$.
\end{enumerate}
and let $\tilde{h}=(h,t)$ be horizontally composable with $\Omega$. Then the $\Omega$ 2-cell $\tilde{h} \circ \Omega$ corresponds to a $\Sg$-path of interest with left border $(r,g,s,h,t,1)$.
\end{corollary}

\begin{proof}
(1) From \cref{pro:alpha.f} we know that the 2-cell $\tilde{h} \circ \Omega$ is an $\Omega$ 2-cell  corresponding to the following $\Sg$-path of $\Sigma$-schemes of left border $(r,g,s,1,1,h)$ --- just consider $\bar{f}$ and $\bar{g}$ as in \cref{pro:alpha.f}, $\bar{h}=(1_C,1_C)$ and $\bar{k}=\tilde{h}$.

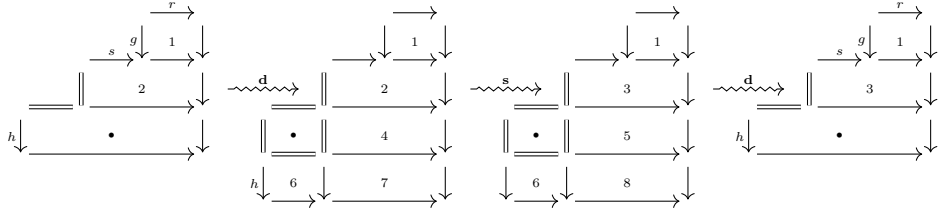
 \[
 \adjustbox{scale=0.70}{\begin{tikzcd}
	&& {} & {} &&& {} & {} &&& {} & {} &&& {} & {} \\
	& {} & {} & {} && {} & {} & {} && {} & {} & {} && {} & {} & {} \\
	{} & {} & {} & {} & {} & {} & {} & {} & {} & {} & {} & {} & {} & {} & {} & {} \\
	{} & {} & {} & {} & {} & {} & {} & {} & {} & {} & {} & {} & {} & {} & {} & {} \\
	&&&& {} & {} && {} & {} & {} && {}
	\arrow["r", from=1-3, to=1-4]
	\arrow[""{name=0, anchor=center, inner sep=0}, "g"', from=1-3, to=2-3]
	\arrow[""{name=1, anchor=center, inner sep=0}, from=1-4, to=2-4]
	\arrow[from=1-7, to=1-8]
	\arrow[""{name=2, anchor=center, inner sep=0}, from=1-7, to=2-7]
	\arrow[""{name=3, anchor=center, inner sep=0}, from=1-8, to=2-8]
	\arrow[from=1-11, to=1-12]
	\arrow[""{name=4, anchor=center, inner sep=0}, from=1-11, to=2-11]
	\arrow[""{name=5, anchor=center, inner sep=0}, from=1-12, to=2-12]
	\arrow["r", from=1-15, to=1-16]
	\arrow[""{name=6, anchor=center, inner sep=0}, "g"', from=1-15, to=2-15]
	\arrow[""{name=7, anchor=center, inner sep=0}, from=1-16, to=2-16]
	\arrow["s", from=2-2, to=2-3]
	\arrow[""{name=8, anchor=center, inner sep=0}, equals, from=2-2, to=3-2]
	\arrow[from=2-3, to=2-4]
	\arrow[""{name=9, anchor=center, inner sep=0}, from=2-4, to=3-4]
	\arrow[from=2-6, to=2-7]
	\arrow[""{name=10, anchor=center, inner sep=0}, equals, from=2-6, to=3-6]
	\arrow[from=2-7, to=2-8]
	\arrow[""{name=11, anchor=center, inner sep=0}, from=2-8, to=3-8]
	\arrow[from=2-10, to=2-11]
	\arrow[""{name=12, anchor=center, inner sep=0}, equals, from=2-10, to=3-10]
	\arrow[from=2-11, to=2-12]
	\arrow[""{name=13, anchor=center, inner sep=0}, from=2-12, to=3-12]
	\arrow["s", from=2-14, to=2-15]
	\arrow[""{name=14, anchor=center, inner sep=0}, equals, from=2-14, to=3-14]
	\arrow[from=2-15, to=2-16]
	\arrow[""{name=15, anchor=center, inner sep=0}, from=2-16, to=3-16]
	\arrow[equals, from=3-1, to=3-2]
	\arrow[""{name=16, anchor=center, inner sep=0}, "h"', from=3-1, to=4-1]
	\arrow[from=3-2, to=3-4]
	\arrow[""{name=17, anchor=center, inner sep=0}, from=3-4, to=4-4]
	\arrow[equals, from=3-5, to=3-6]
	\arrow[""{name=18, anchor=center, inner sep=0}, equals, from=3-5, to=4-5]
	\arrow[from=3-6, to=3-8]
	\arrow[""{name=19, anchor=center, inner sep=0}, equals, from=3-6, to=4-6]
	\arrow[""{name=20, anchor=center, inner sep=0}, from=3-8, to=4-8]
	\arrow[equals, from=3-9, to=3-10]
	\arrow[""{name=21, anchor=center, inner sep=0}, equals, from=3-9, to=4-9]
	\arrow[from=3-10, to=3-12]
	\arrow[""{name=22, anchor=center, inner sep=0}, equals, from=3-10, to=4-10]
	\arrow[""{name=23, anchor=center, inner sep=0}, from=3-12, to=4-12]
	\arrow[equals, from=3-13, to=3-14]
	\arrow[""{name=24, anchor=center, inner sep=0}, "h"', from=3-13, to=4-13]
	\arrow[from=3-14, to=3-16]
	\arrow[""{name=25, anchor=center, inner sep=0}, from=3-16, to=4-16]
	\arrow[from=4-1, to=4-4]
	\arrow[equals, from=4-5, to=4-6]
	\arrow[""{name=26, anchor=center, inner sep=0}, "h"', from=4-5, to=5-5]
	\arrow[from=4-6, to=4-8]
	\arrow[""{name=27, anchor=center, inner sep=0}, from=4-6, to=5-6]
	\arrow[""{name=28, anchor=center, inner sep=0}, from=4-8, to=5-8]
	\arrow[equals, from=4-9, to=4-10]
	\arrow[""{name=29, anchor=center, inner sep=0}, from=4-9, to=5-9]
	\arrow[from=4-10, to=4-12]
	\arrow[""{name=30, anchor=center, inner sep=0}, from=4-10, to=5-10]
	\arrow[""{name=31, anchor=center, inner sep=0}, from=4-12, to=5-12]
	\arrow[from=4-13, to=4-16]
	\arrow[from=5-5, to=5-6]
	\arrow[from=5-6, to=5-8]
	\arrow[from=5-9, to=5-10]
	\arrow[from=5-10, to=5-12]
	\arrow["1"{description}, draw=none, from=0, to=1]
	\arrow["1"{description}, draw=none, from=2, to=3]
	\arrow["1"{description}, draw=none, from=4, to=5]
	\arrow["1"{description}, draw=none, from=6, to=7]
	\arrow["2"{description}, draw=none, from=8, to=9]
	\arrow["\bd", between={0.2}{0.8}, squiggly, from=9, to=10]
	\arrow["2"{description}, draw=none, from=10, to=11]
	\arrow["\bs", between={0.2}{0.8}, squiggly, from=11, to=12]
	\arrow["3"{description}, draw=none, from=12, to=13]
	\arrow["\bd", between={0.2}{0.8}, squiggly, from=13, to=14]
	\arrow["3"{description}, draw=none, from=14, to=15]
	\arrow["\bullet"{description}, draw=none, from=16, to=17]
	\arrow["\bullet"{description}, draw=none, from=18, to=19]
	\arrow["4"{description}, draw=none, from=19, to=20]
	\arrow["\bullet"{description}, draw=none, from=21, to=22]
	\arrow["5"{description}, draw=none, from=22, to=23]
	\arrow["\bullet"{description}, draw=none, from=24, to=25]
	\arrow["6"{description}, draw=none, from=26, to=27]
	\arrow["7"{description}, draw=none, from=27, to=28]
	\arrow["6"{description}, draw=none, from=29, to=30]
	\arrow["8"{description}, draw=none, from=30, to=31]
\end{tikzcd}}
\]
Observe that, equivalently, $\bar{h}\circ \Omega$ corresponds to the $\Sg$-path of $\Sg$-schemes of level 2 given by
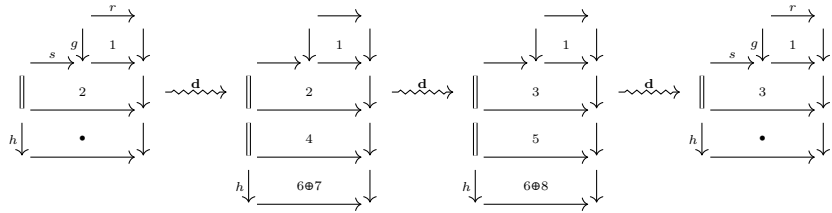
\[
\adjustbox{scale=0.70}{\begin{tikzcd}
	& {} & {} &&& {} & {} &&& {} & {} &&& {} & {} \\
	{} & {} & {} && {} & {} & {} && {} & {} & {} && {} & {} & {} \\
	{} & {} & {} && {} & {} & {} && {} & {} & {} && {} & {} & {} \\
	{} & {} & {} && {} & {} & {} && {} & {} & {} && {} & {} & {} \\
	&&&& {} && {} && {} && {}
	\arrow["r", from=1-2, to=1-3]
	\arrow[""{name=0, anchor=center, inner sep=0}, "g"', from=1-2, to=2-2]
	\arrow[""{name=1, anchor=center, inner sep=0}, from=1-3, to=2-3]
	\arrow[from=1-6, to=1-7]
	\arrow[""{name=2, anchor=center, inner sep=0}, from=1-6, to=2-6]
	\arrow[""{name=3, anchor=center, inner sep=0}, from=1-7, to=2-7]
	\arrow[from=1-10, to=1-11]
	\arrow[""{name=4, anchor=center, inner sep=0}, from=1-10, to=2-10]
	\arrow[""{name=5, anchor=center, inner sep=0}, from=1-11, to=2-11]
	\arrow["r", from=1-14, to=1-15]
	\arrow[""{name=6, anchor=center, inner sep=0}, "g"', from=1-14, to=2-14]
	\arrow[""{name=7, anchor=center, inner sep=0}, from=1-15, to=2-15]
	\arrow["s", from=2-1, to=2-2]
	\arrow[""{name=8, anchor=center, inner sep=0}, equals, from=2-1, to=3-1]
	\arrow[from=2-2, to=2-3]
	\arrow[""{name=9, anchor=center, inner sep=0}, from=2-3, to=3-3]
	\arrow[from=2-5, to=2-6]
	\arrow[""{name=10, anchor=center, inner sep=0}, equals, from=2-5, to=3-5]
	\arrow[from=2-6, to=2-7]
	\arrow[""{name=11, anchor=center, inner sep=0}, from=2-7, to=3-7]
	\arrow[from=2-9, to=2-10]
	\arrow[""{name=12, anchor=center, inner sep=0}, equals, from=2-9, to=3-9]
	\arrow[from=2-10, to=2-11]
	\arrow[""{name=13, anchor=center, inner sep=0}, from=2-11, to=3-11]
	\arrow["s", from=2-13, to=2-14]
	\arrow[""{name=14, anchor=center, inner sep=0}, equals, from=2-13, to=3-13]
	\arrow[from=2-14, to=2-15]
	\arrow[""{name=15, anchor=center, inner sep=0}, from=2-15, to=3-15]
	\arrow[from=3-1, to=3-3]
	\arrow[""{name=16, anchor=center, inner sep=0}, "h"', from=3-1, to=4-1]
	\arrow[""{name=17, anchor=center, inner sep=0}, from=3-3, to=4-3]
	\arrow[from=3-5, to=3-7]
	\arrow[""{name=18, anchor=center, inner sep=0}, equals, from=3-5, to=4-5]
	\arrow[""{name=19, anchor=center, inner sep=0}, from=3-7, to=4-7]
	\arrow[from=3-9, to=3-11]
	\arrow[""{name=20, anchor=center, inner sep=0}, equals, from=3-9, to=4-9]
	\arrow[""{name=21, anchor=center, inner sep=0}, from=3-11, to=4-11]
	\arrow[from=3-13, to=3-15]
	\arrow[""{name=22, anchor=center, inner sep=0}, "h"', from=3-13, to=4-13]
	\arrow[""{name=23, anchor=center, inner sep=0}, from=3-15, to=4-15]
	\arrow[from=4-1, to=4-3]
	\arrow[from=4-5, to=4-7]
	\arrow[""{name=24, anchor=center, inner sep=0}, "h"', from=4-5, to=5-5]
	\arrow[""{name=25, anchor=center, inner sep=0}, from=4-7, to=5-7]
	\arrow[from=4-9, to=4-11]
	\arrow[""{name=26, anchor=center, inner sep=0}, "h"', from=4-9, to=5-9]
	\arrow[""{name=27, anchor=center, inner sep=0}, from=4-11, to=5-11]
	\arrow[from=4-13, to=4-15]
	\arrow[from=5-5, to=5-7]
	\arrow[from=5-9, to=5-11]
	\arrow["1"{description}, draw=none, from=0, to=1]
	\arrow["1"{description}, draw=none, from=2, to=3]
	\arrow["1"{description}, draw=none, from=4, to=5]
	\arrow["1"{description}, draw=none, from=6, to=7]
	\arrow["2"{description}, draw=none, from=8, to=9]
	\arrow["\bd", between={0.2}{0.8}, squiggly, from=9, to=10]
	\arrow["2"{description}, draw=none, from=10, to=11]
	\arrow["\bd",  between={0.2}{0.8}, squiggly, from=11, to=12]
	\arrow["3"{description}, draw=none, from=12, to=13]
	\arrow["\bd",  between={0.2}{0.8}, squiggly, from=13, to=14]
	\arrow["3"{description}, draw=none, from=14, to=15]
	\arrow["\bullet"{description}, draw=none, from=16, to=17]
	\arrow["4"{description}, draw=none, from=18, to=19]
	\arrow["5"{description}, draw=none, from=20, to=21]
	\arrow["\bullet"{description}, draw=none, from=22, to=23]
	\arrow["{6\oplus 7}"{description}, draw=none, from=24, to=25]
	\arrow["{6\oplus 8}"{description}, draw=none, from=26, to=27]
\end{tikzcd}}
\]
Moreover, adding just an identity row below each $\Sg$-scheme, we obtain a $\Sg$-path between $\Sg$-schemes of level 3 and left border $(r,g,s,h,t,1)$:
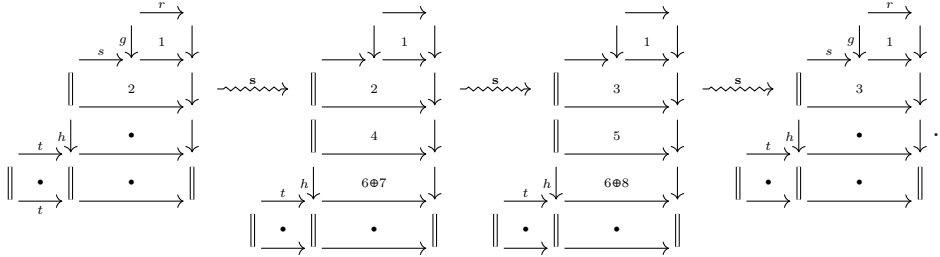
\[
\adjustbox{scale=0.70}{\begin{tikzcd}
	&& {} & {} &&& {} & {} &&& {} & {} &&& {} & {} \\
	& {} & {} & {} && {} & {} & {} && {} & {} & {} && {} & {} & {} \\
	& {} & {} & {} && {} & {} & {} && {} & {} & {} && {} & {} & {} \\
	{} & {} & {} & {} && {} & {} & {} && {} & {} & {} & {} & {} & {} & {} \\
	{} & {} && {} & {} & {} && {} & {} & {} && {} & {} & {} && {} \\
	&&&& {} & {} && {} & {} & {} && {}
	\arrow["r", from=1-3, to=1-4]
	\arrow[""{name=0, anchor=center, inner sep=0}, "g"', from=1-3, to=2-3]
	\arrow[""{name=1, anchor=center, inner sep=0}, from=1-4, to=2-4]
	\arrow[from=1-7, to=1-8]
	\arrow[""{name=2, anchor=center, inner sep=0}, from=1-7, to=2-7]
	\arrow[""{name=3, anchor=center, inner sep=0}, from=1-8, to=2-8]
	\arrow[from=1-11, to=1-12]
	\arrow[""{name=4, anchor=center, inner sep=0}, from=1-11, to=2-11]
	\arrow[""{name=5, anchor=center, inner sep=0}, from=1-12, to=2-12]
	\arrow["r", from=1-15, to=1-16]
	\arrow[""{name=6, anchor=center, inner sep=0}, "g"', from=1-15, to=2-15]
	\arrow[""{name=7, anchor=center, inner sep=0}, from=1-16, to=2-16]
	\arrow["s", from=2-2, to=2-3]
	\arrow[""{name=8, anchor=center, inner sep=0}, equals, from=2-2, to=3-2]
	\arrow[from=2-3, to=2-4]
	\arrow[""{name=9, anchor=center, inner sep=0}, from=2-4, to=3-4]
	\arrow[from=2-6, to=2-7]
	\arrow[""{name=10, anchor=center, inner sep=0}, equals, from=2-6, to=3-6]
	\arrow[from=2-7, to=2-8]
	\arrow[""{name=11, anchor=center, inner sep=0}, from=2-8, to=3-8]
	\arrow[from=2-10, to=2-11]
	\arrow[""{name=12, anchor=center, inner sep=0}, equals, from=2-10, to=3-10]
	\arrow[from=2-11, to=2-12]
	\arrow[""{name=13, anchor=center, inner sep=0}, from=2-12, to=3-12]
	\arrow["s", from=2-14, to=2-15]
	\arrow[""{name=14, anchor=center, inner sep=0}, equals, from=2-14, to=3-14]
	\arrow[from=2-15, to=2-16]
	\arrow[""{name=15, anchor=center, inner sep=0}, from=2-16, to=3-16]
	\arrow[from=3-2, to=3-4]
	\arrow[""{name=16, anchor=center, inner sep=0}, "h"', from=3-2, to=4-2]
	\arrow[""{name=17, anchor=center, inner sep=0}, from=3-4, to=4-4]
	\arrow[from=3-6, to=3-8]
	\arrow[""{name=18, anchor=center, inner sep=0}, equals, from=3-6, to=4-6]
	\arrow[""{name=19, anchor=center, inner sep=0}, from=3-8, to=4-8]
	\arrow[from=3-10, to=3-12]
	\arrow[""{name=20, anchor=center, inner sep=0}, equals, from=3-10, to=4-10]
	\arrow[""{name=21, anchor=center, inner sep=0}, from=3-12, to=4-12]
	\arrow[from=3-14, to=3-16]
	\arrow[""{name=22, anchor=center, inner sep=0}, "h"', from=3-14, to=4-14]
	\arrow[""{name=23, anchor=center, inner sep=0}, from=3-16, to=4-16]
	\arrow["t", from=4-1, to=4-2]
	\arrow[""{name=24, anchor=center, inner sep=0}, equals, from=4-1, to=5-1]
	\arrow[from=4-2, to=4-4]
	\arrow[""{name=25, anchor=center, inner sep=0}, equals, from=4-2, to=5-2]
	\arrow[""{name=26, anchor=center, inner sep=0}, equals, from=4-4, to=5-4]
	\arrow[from=4-6, to=4-8]
	\arrow[""{name=27, anchor=center, inner sep=0}, "h"', from=4-6, to=5-6]
	\arrow[""{name=28, anchor=center, inner sep=0}, from=4-8, to=5-8]
	\arrow[from=4-10, to=4-12]
	\arrow[""{name=29, anchor=center, inner sep=0}, "h"', from=4-10, to=5-10]
	\arrow[""{name=30, anchor=center, inner sep=0}, from=4-12, to=5-12]
	\arrow["t", from=4-13, to=4-14]
	\arrow[""{name=31, anchor=center, inner sep=0}, equals, from=4-13, to=5-13]
	\arrow[from=4-14, to=4-16]
	\arrow[""{name=32, anchor=center, inner sep=0}, equals, from=4-14, to=5-14]
	\arrow[""{name=33, anchor=center, inner sep=0}, equals, from=4-16, to=5-16]
	\arrow["t"', from=5-1, to=5-2]
	\arrow[from=5-2, to=5-4]
	\arrow["t", from=5-5, to=5-6]
	\arrow[""{name=34, anchor=center, inner sep=0}, equals, from=5-5, to=6-5]
	\arrow[from=5-6, to=5-8]
	\arrow[""{name=35, anchor=center, inner sep=0}, equals, from=5-6, to=6-6]
	\arrow[""{name=36, anchor=center, inner sep=0}, equals, from=5-8, to=6-8]
	\arrow["t", from=5-9, to=5-10]
	\arrow[""{name=37, anchor=center, inner sep=0}, equals, from=5-9, to=6-9]
	\arrow[from=5-10, to=5-12]
	\arrow[""{name=38, anchor=center, inner sep=0}, equals, from=5-10, to=6-10]
	\arrow[""{name=39, anchor=center, inner sep=0}, equals, from=5-12, to=6-12]
	\arrow[from=5-13, to=5-14]
	\arrow[from=5-14, to=5-16]
	\arrow[from=6-5, to=6-6]
	\arrow[from=6-6, to=6-8]
	\arrow[from=6-9, to=6-10]
	\arrow[from=6-10, to=6-12]
	\arrow["1"{description}, draw=none, from=0, to=1]
	\arrow["1"{description}, draw=none, from=2, to=3]
	\arrow["1"{description}, draw=none, from=4, to=5]
	\arrow["1"{description}, draw=none, from=6, to=7]
	\arrow["2"{description}, draw=none, from=8, to=9]
	\arrow["\bs", between={0.2}{0.8}, squiggly, from=9, to=10]
	\arrow["2"{description}, draw=none, from=10, to=11]
	\arrow["\bs", between={0.2}{0.8}, squiggly, from=11, to=12]
	\arrow["3"{description}, draw=none, from=12, to=13]
	\arrow["\bs", between={0.2}{0.8}, squiggly, from=13, to=14]
	\arrow["3"{description}, draw=none, from=14, to=15]
	\arrow["\bullet"{description}, draw=none, from=16, to=17]
	\arrow["4"{description}, draw=none, from=18, to=19]
	\arrow["5"{description}, draw=none, from=20, to=21]
	\arrow["\bullet"{description}, draw=none, from=22, to=23]
	\arrow["\bullet"{description}, draw=none, from=24, to=25]
	\arrow["\bullet"{description}, draw=none, from=25, to=26]
	\arrow["{{6\oplus 7}}"{description}, draw=none, from=27, to=28]
	\arrow["{{6\oplus 8}}"{description}, draw=none, from=29, to=30]
	\arrow["\bullet"{description}, draw=none, from=31, to=32]
	\arrow["\bullet"{description}, draw=none, from=32, to=33]
	\arrow["\bullet"{description}, draw=none, from=34, to=35]
	\arrow["\bullet"{description}, draw=none, from=35, to=36]
	\arrow["\bullet"{description}, draw=none, from=37, to=38]
	\arrow["\bullet"{description}, draw=none, from=38, to=39]
\end{tikzcd}}\, .\]
The 2-cell $\bar{h}\circ \Omega$ also corresponds to this $\Sg$-path, which indeed is equivalent to just a $\Sg$-step of type $\bs$. Hence, it is a $\Sg$-path of interest.

(2)
Acting in a similar way for the $\Sg$-step $u$, we obtain that the $\Omega$ 2-cell $\bar{h}\circ \Omega$ corresponds to the following $\Sg$-path %
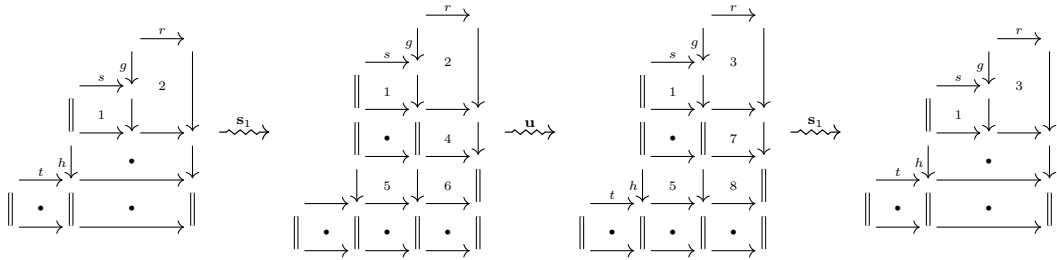
\[\adjustbox{scale=0.70}{\begin{tikzcd}
	&& {} & {} \\
	& {} & {} & {} \\
	& {} & {} & {} \\
	{} & {} && {} \\
	{} & {} && {}
	\arrow["r", from=1-3, to=1-4]
	\arrow["g"', from=1-3, to=2-3]
	\arrow[from=1-4, to=3-4]
	\arrow["s", from=2-2, to=2-3]
	\arrow[""{name=0, anchor=center, inner sep=0}, equals, from=2-2, to=3-2]
	\arrow["2"{description}, draw=none, from=2-3, to=2-4]
	\arrow[""{name=1, anchor=center, inner sep=0}, from=2-3, to=3-3]
	\arrow[from=3-2, to=3-3]
	\arrow[""{name=2, anchor=center, inner sep=0}, "h"', from=3-2, to=4-2]
	\arrow[from=3-3, to=3-4]
	\arrow[""{name=3, anchor=center, inner sep=0}, from=3-4, to=4-4]
	\arrow["t", from=4-1, to=4-2]
	\arrow[""{name=4, anchor=center, inner sep=0}, equals, from=4-1, to=5-1]
	\arrow[from=4-2, to=4-4]
	\arrow[""{name=5, anchor=center, inner sep=0}, equals, from=4-2, to=5-2]
	\arrow[""{name=6, anchor=center, inner sep=0}, equals, from=4-4, to=5-4]
	\arrow[from=5-1, to=5-2]
	\arrow[from=5-2, to=5-4]
	\arrow["1"{description}, draw=none, from=0, to=1]
	\arrow["\bullet"{description}, draw=none, from=2, to=3]
	\arrow["\bullet"{description}, draw=none, from=4, to=5]
	\arrow["\bullet"{description}, draw=none, from=5, to=6]
\end{tikzcd}}
\adjustbox{scale=0.80}{\begin{tikzcd}
	{} & {}
	\arrow["{{\bs}_1}", squiggly, from=1-1, to=1-2]
\end{tikzcd}}
\adjustbox{scale=0.70}{\begin{tikzcd}
	&& {} & {} \\
	& {} & {} & {} \\
	& {} & {} & {} \\
	& {} & {} & {} \\
	{} & {} & {} & {} \\
	{} & {} & {} & {}
	\arrow["r", from=1-3, to=1-4]
	\arrow["g"', from=1-3, to=2-3]
	\arrow[from=1-4, to=3-4]
	\arrow["s", from=2-2, to=2-3]
	\arrow[""{name=0, anchor=center, inner sep=0}, equals, from=2-2, to=3-2]
	\arrow["2"{description}, draw=none, from=2-3, to=2-4]
	\arrow[""{name=1, anchor=center, inner sep=0}, from=2-3, to=3-3]
	\arrow[from=3-2, to=3-3]
	\arrow[""{name=2, anchor=center, inner sep=0}, equals, from=3-2, to=4-2]
	\arrow[from=3-3, to=3-4]
	\arrow[""{name=3, anchor=center, inner sep=0}, equals, from=3-3, to=4-3]
	\arrow[""{name=4, anchor=center, inner sep=0}, from=3-4, to=4-4]
	\arrow[from=4-2, to=4-3]
	\arrow[""{name=5, anchor=center, inner sep=0}, from=4-2, to=5-2]
	\arrow[from=4-3, to=4-4]
	\arrow[""{name=6, anchor=center, inner sep=0}, from=4-3, to=5-3]
	\arrow[""{name=7, anchor=center, inner sep=0}, equals, from=4-4, to=5-4]
	\arrow[from=5-1, to=5-2]
	\arrow[""{name=8, anchor=center, inner sep=0}, equals, from=5-1, to=6-1]
	\arrow[from=5-2, to=5-3]
	\arrow[""{name=9, anchor=center, inner sep=0}, equals, from=5-2, to=6-2]
	\arrow[from=5-3, to=5-4]
	\arrow[""{name=10, anchor=center, inner sep=0}, equals, from=5-3, to=6-3]
	\arrow[""{name=11, anchor=center, inner sep=0}, equals, from=5-4, to=6-4]
	\arrow[from=6-1, to=6-2]
	\arrow[from=6-2, to=6-3]
	\arrow[from=6-3, to=6-4]
	\arrow["1"{description}, draw=none, from=0, to=1]
	\arrow["\bullet"{description}, draw=none, from=2, to=3]
	\arrow["4"{description}, draw=none, from=3, to=4]
	\arrow["5"{description}, draw=none, from=5, to=6]
	\arrow["6"{description}, draw=none, from=6, to=7]
	\arrow["\bullet"{description}, draw=none, from=8, to=9]
	\arrow["\bullet"{description}, draw=none, from=9, to=10]
	\arrow["\bullet"{description}, draw=none, from=10, to=11]
\end{tikzcd}}
\adjustbox{scale=0.80}{\begin{tikzcd}
	{} & {}
	\arrow["{\bu}",  squiggly, from=1-1, to=1-2]
\end{tikzcd}}
\adjustbox{scale=0.70}{\begin{tikzcd}
	&& {} & {} \\
	& {} & {} & {} \\
	& {} & {} & {} \\
	& {} & {} & {} \\
	{} & {} & {} & {} \\
	{} & {} & {} & {}
	\arrow["r", from=1-3, to=1-4]
	\arrow["g"', from=1-3, to=2-3]
	\arrow[from=1-4, to=3-4]
	\arrow["s", from=2-2, to=2-3]
	\arrow[""{name=0, anchor=center, inner sep=0}, equals, from=2-2, to=3-2]
	\arrow["3"{description}, draw=none, from=2-3, to=2-4]
	\arrow[""{name=1, anchor=center, inner sep=0}, from=2-3, to=3-3]
	\arrow[from=3-2, to=3-3]
	\arrow[""{name=2, anchor=center, inner sep=0}, equals, from=3-2, to=4-2]
	\arrow[from=3-3, to=3-4]
	\arrow[""{name=3, anchor=center, inner sep=0}, equals, from=3-3, to=4-3]
	\arrow[""{name=4, anchor=center, inner sep=0}, from=3-4, to=4-4]
	\arrow[from=4-2, to=4-3]
	\arrow[""{name=5, anchor=center, inner sep=0}, "h"', from=4-2, to=5-2]
	\arrow[from=4-3, to=4-4]
	\arrow[""{name=6, anchor=center, inner sep=0}, from=4-3, to=5-3]
	\arrow[""{name=7, anchor=center, inner sep=0}, equals, from=4-4, to=5-4]
	\arrow["t", from=5-1, to=5-2]
	\arrow[""{name=8, anchor=center, inner sep=0}, equals, from=5-1, to=6-1]
	\arrow[from=5-2, to=5-3]
	\arrow[""{name=9, anchor=center, inner sep=0}, equals, from=5-2, to=6-2]
	\arrow[from=5-3, to=5-4]
	\arrow[""{name=10, anchor=center, inner sep=0}, equals, from=5-3, to=6-3]
	\arrow[""{name=11, anchor=center, inner sep=0}, equals, from=5-4, to=6-4]
	\arrow[from=6-1, to=6-2]
	\arrow[from=6-2, to=6-3]
	\arrow[from=6-3, to=6-4]
	\arrow["1"{description}, draw=none, from=0, to=1]
	\arrow["\bullet"{description}, draw=none, from=2, to=3]
	\arrow["7"{description}, draw=none, from=3, to=4]
	\arrow["5"{description}, draw=none, from=5, to=6]
	\arrow["8"{description}, draw=none, from=6, to=7]
	\arrow["\bullet"{description}, draw=none, from=8, to=9]
	\arrow["\bullet"{description}, draw=none, from=9, to=10]
	\arrow["\bullet"{description}, draw=none, from=10, to=11]
\end{tikzcd}}
\adjustbox{scale=0.80}{\begin{tikzcd}
	{} & {}
	\arrow["{\bs_1}",  squiggly, from=1-1, to=1-2]
\end{tikzcd}}
\adjustbox{scale=0.70}{\begin{tikzcd}
	&& {} & {} \\
	& {} & {} & {} \\
	& {} & {} & {} \\
	{} & {} && {} \\
	{} & {} && {}
	\arrow["r", from=1-3, to=1-4]
	\arrow["g"', from=1-3, to=2-3]
	\arrow[from=1-4, to=3-4]
	\arrow["s", from=2-2, to=2-3]
	\arrow[""{name=0, anchor=center, inner sep=0}, equals, from=2-2, to=3-2]
	\arrow["3"{description}, draw=none, from=2-3, to=2-4]
	\arrow[""{name=1, anchor=center, inner sep=0}, from=2-3, to=3-3]
	\arrow[from=3-2, to=3-3]
	\arrow[""{name=2, anchor=center, inner sep=0}, "h"', from=3-2, to=4-2]
	\arrow[from=3-3, to=3-4]
	\arrow[""{name=3, anchor=center, inner sep=0}, from=3-4, to=4-4]
	\arrow["t", from=4-1, to=4-2]
	\arrow[""{name=4, anchor=center, inner sep=0}, equals, from=4-1, to=5-1]
	\arrow[from=4-2, to=4-4]
	\arrow[""{name=5, anchor=center, inner sep=0}, equals, from=4-2, to=5-2]
	\arrow[""{name=6, anchor=center, inner sep=0}, equals, from=4-4, to=5-4]
	\arrow[from=5-1, to=5-2]
	\arrow[from=5-2, to=5-4]
	\arrow["1"{description}, draw=none, from=0, to=1]
	\arrow["\bullet"{description}, draw=none, from=2, to=3]
	\arrow["\bullet"{description}, draw=none, from=4, to=5]
	\arrow["\bullet"{description}, draw=none, from=5, to=6]
\end{tikzcd}}
\]
which is of interest.
\end{proof}

The final theorem of this section states that every two $\Sg$-paths of interest, of any finite length, between two $\Sg$-schemes of level 3, are equivalent. As an auxiliary tool,  we  consider certain $\Sg$-paths into the canonical $\Sg$-scheme, $\Can$.

For a $\Sg$-scheme of a given configuration of interest $x$ (see Definition \ref{def:interest}) there may exist several $\Sg$-paths from it into the canonical $\Sg$-scheme. In the next table we choose one of these $\Sg$-paths to be the  \textbf{\em $x$-canonical $\Sg$-path}.

\begin{table}[h]
\centering
\caption{$x$-canonical $\Sigma$-paths}
\label{tab:paths}
\begin{tabular}{|c|c|}\hline
$x$&$x$-canonical $\Sg$-path\\ \hline \hline
$\mathbf{da}$
&
\(\adjustbox{scale=0.60}{\begin{tikzcd}
	&& {} & {} \\
	& {} & {} \\
	{} & {} & {} & {} \\
	{} &&& {}
	\arrow[from=1-3, to=1-4]
	\arrow[from=1-3, to=2-3]
	\arrow[""{name=0, anchor=center, inner sep=0}, from=1-4, to=3-4]
	\arrow[from=2-2, to=2-3]
	\arrow[""{name=1, anchor=center, inner sep=0}, from=2-2, to=3-2]
	\arrow[""{name=2, anchor=center, inner sep=0}, from=2-3, to=3-3]
	\arrow[from=3-1, to=3-2]
	\arrow[""{name=3, anchor=center, inner sep=0}, from=3-1, to=4-1]
	\arrow[from=3-2, to=3-3]
	\arrow[from=3-3, to=3-4]
	\arrow[""{name=4, anchor=center, inner sep=0}, from=3-4, to=4-4]
	\arrow[from=4-1, to=4-4]
	\arrow["1"{description}, draw=none, from=1, to=2]
	\arrow["2"{description}, draw=none, from=2-3, to=0]
	\arrow["3"{description}, draw=none, from=3, to=4]
\end{tikzcd}}
\begin{tikzcd}
	{} & {}
	\arrow["{\bd}",  squiggly, from=1-1, to=1-2]
\end{tikzcd}
\hspace*{-3mm}
\adjustbox{scale=0.60}{\begin{tikzcd}
	&& {} & {} \\
	& {} & {} \\
	{} & {} & {} & {} \\
	{} & {} & {} & {}
	\arrow[from=1-3, to=1-4]
	\arrow[from=1-3, to=2-3]
	\arrow[""{name=0, anchor=center, inner sep=0}, from=1-4, to=3-4]
	\arrow[from=2-2, to=2-3]
	\arrow[""{name=1, anchor=center, inner sep=0}, from=2-2, to=3-2]
	\arrow[""{name=2, anchor=center, inner sep=0}, from=2-3, to=3-3]
	\arrow[from=3-1, to=3-2]
	\arrow[""{name=3, anchor=center, inner sep=0}, from=3-1, to=4-1]
	\arrow[from=3-2, to=3-3]
	\arrow[""{name=4, anchor=center, inner sep=0}, from=3-2, to=4-2]
	\arrow[from=3-3, to=3-4]
	\arrow[""{name=5, anchor=center, inner sep=0}, from=3-3, to=4-3]
	\arrow[""{name=6, anchor=center, inner sep=0}, from=3-4, to=4-4]
	\arrow[from=4-1, to=4-2]
	\arrow[from=4-2, to=4-3]
	\arrow[from=4-3, to=4-4]
	\arrow["1"{description}, draw=none, from=1, to=2]
	\arrow["2"{description}, draw=none, from=2-3, to=0]
	\arrow["\bullet"{description}, draw=none, from=3, to=4]
	\arrow["\bullet"{description}, draw=none, from=4, to=5]
	\arrow["\bullet"{description}, draw=none, from=5, to=6]
\end{tikzcd}}
\begin{tikzcd}
	{} & {}
	\arrow["{\bu}",  squiggly, from=1-1, to=1-2]
\end{tikzcd}
\hspace*{-3mm}
\adjustbox{scale=0.60}{\begin{tikzcd}
	&& {} & {} \\
	& {} & {} & {} \\
	{} & {} & {} & {} \\
	{} & {} & {} & {}
	\arrow[from=1-3, to=1-4]
	\arrow[""{name=0, anchor=center, inner sep=0}, from=1-3, to=2-3]
	\arrow[""{name=1, anchor=center, inner sep=0}, from=1-4, to=2-4]
	\arrow[from=2-2, to=2-3]
	\arrow[""{name=2, anchor=center, inner sep=0}, from=2-2, to=3-2]
	\arrow[from=2-3, to=2-4]
	\arrow[""{name=3, anchor=center, inner sep=0}, from=2-3, to=3-3]
	\arrow[""{name=4, anchor=center, inner sep=0}, from=2-4, to=3-4]
	\arrow[from=3-1, to=3-2]
	\arrow[""{name=5, anchor=center, inner sep=0}, from=3-1, to=4-1]
	\arrow[from=3-2, to=3-3]
	\arrow[""{name=6, anchor=center, inner sep=0}, from=3-2, to=4-2]
	\arrow[from=3-3, to=3-4]
	\arrow[""{name=7, anchor=center, inner sep=0}, from=3-3, to=4-3]
	\arrow[""{name=8, anchor=center, inner sep=0}, from=3-4, to=4-4]
	\arrow[from=4-1, to=4-2]
	\arrow[from=4-2, to=4-3]
	\arrow[from=4-3, to=4-4]
	\arrow["\bullet"{description}, draw=none, from=0, to=1]
	\arrow["1"{description}, draw=none, from=2, to=3]
	\arrow["\bullet"{description}, draw=none, from=3, to=4]
	\arrow["\bullet"{description}, draw=none, from=5, to=6]
	\arrow["\bullet"{description}, draw=none, from=6, to=7]
	\arrow["\bullet"{description}, draw=none, from=7, to=8]
\end{tikzcd}}
\begin{tikzcd}
	{} & {}
	\arrow["{\bs}",  squiggly, from=1-1, to=1-2]
\end{tikzcd}
\Can
\)
\\ \hline
$\mathbf{db}$
&
\(\adjustbox{scale=0.60}{\begin{tikzcd}
	&& {} & {} \\
	& {} & {} & {} \\
	{} & {} && {} \\
	{} &&& {}
	\arrow[from=1-3, to=1-4]
	\arrow[""{name=0, anchor=center, inner sep=0}, from=1-3, to=2-3]
	\arrow[""{name=1, anchor=center, inner sep=0}, from=1-4, to=2-4]
	\arrow[from=2-2, to=2-3]
	\arrow[""{name=2, anchor=center, inner sep=0}, from=2-2, to=3-2]
	\arrow[from=2-3, to=2-4]
	\arrow[""{name=3, anchor=center, inner sep=0}, from=2-4, to=3-4]
	\arrow[from=3-1, to=3-2]
	\arrow[""{name=4, anchor=center, inner sep=0}, from=3-1, to=4-1]
	\arrow[from=3-2, to=3-4]
	\arrow[""{name=5, anchor=center, inner sep=0}, from=3-4, to=4-4]
	\arrow[from=4-1, to=4-4]
	\arrow["1"{description}, draw=none, from=0, to=1]
	\arrow["2"{description}, draw=none, from=2, to=3]
	\arrow["3"{description}, draw=none, from=4, to=5]
\end{tikzcd}}
\begin{tikzcd}
	{} & {}
	\arrow["{\bd}",  squiggly, from=1-1, to=1-2]
\end{tikzcd}
\hspace*{-3mm}
\adjustbox{scale=0.60}{\begin{tikzcd}
	&& {} & {} \\
	& {} & {} & {} \\
	{} & {} && {} \\
	{} & {} && {}
	\arrow[from=1-3, to=1-4]
	\arrow[""{name=0, anchor=center, inner sep=0}, from=1-3, to=2-3]
	\arrow[""{name=1, anchor=center, inner sep=0}, from=1-4, to=2-4]
	\arrow[from=2-2, to=2-3]
	\arrow[""{name=2, anchor=center, inner sep=0}, from=2-2, to=3-2]
	\arrow[from=2-3, to=2-4]
	\arrow[""{name=3, anchor=center, inner sep=0}, from=2-4, to=3-4]
	\arrow[from=3-1, to=3-2]
	\arrow[""{name=4, anchor=center, inner sep=0}, from=3-1, to=4-1]
	\arrow[from=3-2, to=3-4]
	\arrow[""{name=5, anchor=center, inner sep=0}, from=3-2, to=4-2]
	\arrow[""{name=6, anchor=center, inner sep=0}, from=3-4, to=4-4]
	\arrow[from=4-1, to=4-2]
	\arrow[from=4-2, to=4-4]
	\arrow["1"{description}, draw=none, from=0, to=1]
	\arrow["2"{description}, draw=none, from=2, to=3]
	\arrow["\bullet"{description}, draw=none, from=4, to=5]
	\arrow["\bullet"{description}, draw=none, from=5, to=6]
\end{tikzcd}}
\begin{tikzcd}
	{} & {}
	\arrow["{\bs}",  squiggly, from=1-1, to=1-2]
\end{tikzcd}
\hspace*{-3mm}
\adjustbox{scale=0.60}{\begin{tikzcd}
	&& {} & {} \\
	& {} & {} & {} \\
	{} & {} & {} & {} \\
	{} & {} & {} & {}
	\arrow[from=1-3, to=1-4]
	\arrow[""{name=0, anchor=center, inner sep=0}, from=1-3, to=2-3]
	\arrow[""{name=1, anchor=center, inner sep=0}, from=1-4, to=2-4]
	\arrow[from=2-2, to=2-3]
	\arrow[""{name=2, anchor=center, inner sep=0}, from=2-2, to=3-2]
	\arrow[from=2-3, to=2-4]
	\arrow[""{name=3, anchor=center, inner sep=0}, from=2-3, to=3-3]
	\arrow[""{name=4, anchor=center, inner sep=0}, from=2-4, to=3-4]
	\arrow[from=3-1, to=3-2]
	\arrow[""{name=5, anchor=center, inner sep=0}, from=3-1, to=4-1]
	\arrow[from=3-2, to=3-3]
	\arrow[""{name=6, anchor=center, inner sep=0}, from=3-2, to=4-2]
	\arrow[from=3-3, to=3-4]
	\arrow[""{name=7, anchor=center, inner sep=0}, from=3-3, to=4-3]
	\arrow[""{name=8, anchor=center, inner sep=0}, from=3-4, to=4-4]
	\arrow[from=4-1, to=4-2]
	\arrow[from=4-2, to=4-3]
	\arrow[from=4-3, to=4-4]
	\arrow["1"{description}, draw=none, from=0, to=1]
	\arrow["\bullet"{description}, draw=none, from=2, to=3]
	\arrow["\bullet"{description}, draw=none, from=3, to=4]
	\arrow["\bullet"{description}, draw=none, from=5, to=6]
	\arrow["\bullet"{description}, draw=none, from=6, to=7]
	\arrow["\bullet"{description}, draw=none, from=7, to=8]
\end{tikzcd}}
\begin{tikzcd}
	{} & {}
	\arrow["{\bu}",  squiggly, from=1-1, to=1-2]
\end{tikzcd}
\Can
\)
\\ \hline
$\mathbf{dc}$
&
\(\adjustbox{scale=0.60}{\begin{tikzcd}
	&& {} & {} \\
	& {} & {} \\
	& {} & {} & {} \\
	{} & {} && {} \\
	{} &&& {}
	\arrow[from=1-3, to=1-4]
	\arrow[from=1-3, to=2-3]
	\arrow[""{name=0, anchor=center, inner sep=0}, from=1-4, to=3-4]
	\arrow[from=2-2, to=2-3]
	\arrow[""{name=1, anchor=center, inner sep=0}, equals, from=2-2, to=3-2]
	\arrow[""{name=2, anchor=center, inner sep=0}, from=2-3, to=3-3]
	\arrow[from=3-2, to=3-3]
	\arrow[""{name=3, anchor=center, inner sep=0}, from=3-2, to=4-2]
	\arrow[from=3-3, to=3-4]
	\arrow[""{name=4, anchor=center, inner sep=0}, from=3-4, to=4-4]
	\arrow[from=4-1, to=4-2]
	\arrow[""{name=5, anchor=center, inner sep=0}, from=4-1, to=5-1]
	\arrow[from=4-2, to=4-4]
	\arrow[""{name=6, anchor=center, inner sep=0}, from=4-4, to=5-4]
	\arrow[from=5-1, to=5-4]
	\arrow["1"{description}, draw=none, from=1, to=2]
	\arrow["2"{description}, draw=none, from=2-3, to=0]
	\arrow["3"{description}, draw=none, from=3, to=4]
	\arrow["4"{description}, draw=none, from=5, to=6]
\end{tikzcd}}
\begin{tikzcd}
	{} & {}
	\arrow["{\bd}",  squiggly, from=1-1, to=1-2]
\end{tikzcd}
\hspace*{-3mm}
\adjustbox{scale=0.60}{\begin{tikzcd}
	&& {} & {} \\
	& {} & {} \\
	& {} & {} & {} \\
	{} & {} && {} \\
	{} & {} && {}
	\arrow[from=1-3, to=1-4]
	\arrow[from=1-3, to=2-3]
	\arrow[""{name=0, anchor=center, inner sep=0}, from=1-4, to=3-4]
	\arrow[from=2-2, to=2-3]
	\arrow[""{name=1, anchor=center, inner sep=0}, equals, from=2-2, to=3-2]
	\arrow[""{name=2, anchor=center, inner sep=0}, from=2-3, to=3-3]
	\arrow[from=3-2, to=3-3]
	\arrow[""{name=3, anchor=center, inner sep=0}, from=3-2, to=4-2]
	\arrow[from=3-3, to=3-4]
	\arrow[""{name=4, anchor=center, inner sep=0}, from=3-4, to=4-4]
	\arrow[from=4-1, to=4-2]
	\arrow[""{name=5, anchor=center, inner sep=0}, from=4-1, to=5-1]
	\arrow[from=4-2, to=4-4]
	\arrow[""{name=6, anchor=center, inner sep=0}, from=4-2, to=5-2]
	\arrow[""{name=7, anchor=center, inner sep=0}, from=4-4, to=5-4]
	\arrow[from=5-1, to=5-2]
	\arrow[from=5-2, to=5-4]
	\arrow["1"{description}, draw=none, from=1, to=2]
	\arrow["2"{description}, draw=none, from=2-3, to=0]
	\arrow["3"{description}, draw=none, from=3, to=4]
	\arrow["\bullet"{description}, draw=none, from=5, to=6]
	\arrow["\bullet"{description}, draw=none, from=6, to=7]
\end{tikzcd}}
\begin{tikzcd}
	{} & {}
	\arrow["{\bs_1}",  squiggly, from=1-1, to=1-2]
\end{tikzcd}
\hspace*{-3mm}
\adjustbox{scale=0.60}{\begin{tikzcd}
	&& {} & {} \\
	& {} & {} \\
	& {} & {} & {} \\
	{} & {} & {} & {} \\
	{} & {} & {} & {}
	\arrow[from=1-3, to=1-4]
	\arrow[from=1-3, to=2-3]
	\arrow[""{name=0, anchor=center, inner sep=0}, from=1-4, to=3-4]
	\arrow[from=2-2, to=2-3]
	\arrow[""{name=1, anchor=center, inner sep=0}, equals, from=2-2, to=3-2]
	\arrow[""{name=2, anchor=center, inner sep=0}, from=2-3, to=3-3]
	\arrow[from=3-2, to=3-3]
	\arrow[""{name=3, anchor=center, inner sep=0}, from=3-2, to=4-2]
	\arrow[from=3-3, to=3-4]
	\arrow[""{name=4, anchor=center, inner sep=0}, from=3-3, to=4-3]
	\arrow[""{name=5, anchor=center, inner sep=0}, from=3-4, to=4-4]
	\arrow[from=4-1, to=4-2]
	\arrow[""{name=6, anchor=center, inner sep=0}, from=4-1, to=5-1]
	\arrow[from=4-2, to=4-3]
	\arrow[""{name=7, anchor=center, inner sep=0}, from=4-2, to=5-2]
	\arrow[from=4-3, to=4-4]
	\arrow[""{name=8, anchor=center, inner sep=0}, from=4-3, to=5-3]
	\arrow[""{name=9, anchor=center, inner sep=0}, from=4-4, to=5-4]
	\arrow[from=5-1, to=5-2]
	\arrow[from=5-2, to=5-3]
	\arrow[from=5-3, to=5-4]
	\arrow["1"{description}, draw=none, from=1, to=2]
	\arrow["2"{description}, draw=none, from=2-3, to=0]
	\arrow["\bullet"{description}, draw=none, from=3, to=4]
	\arrow["\bullet"{description}, draw=none, from=4, to=5]
	\arrow["\bullet"{description}, draw=none, from=6, to=7]
	\arrow["\bullet"{description}, draw=none, from=7, to=8]
	\arrow["\bullet"{description}, draw=none, from=8, to=9]
\end{tikzcd}}
\begin{tikzcd}
	{} & {}
	\arrow["{\bu}",  squiggly, from=1-1, to=1-2]
\end{tikzcd}
\hspace*{-3mm}
\adjustbox{scale=0.60}{\begin{tikzcd}
	&& {} & {} \\
	& {} & {} & {} \\
	& {} & {} & {} \\
	{} & {} & {} & {} \\
	{} & {} & {} & {}
	\arrow[from=1-3, to=1-4]
	\arrow[""{name=0, anchor=center, inner sep=0}, from=1-3, to=2-3]
	\arrow[""{name=1, anchor=center, inner sep=0}, from=1-4, to=2-4]
	\arrow[from=2-2, to=2-3]
	\arrow[""{name=2, anchor=center, inner sep=0}, equals, from=2-2, to=3-2]
	\arrow[from=2-3, to=2-4]
	\arrow[""{name=3, anchor=center, inner sep=0}, from=2-3, to=3-3]
	\arrow[""{name=4, anchor=center, inner sep=0}, from=2-4, to=3-4]
	\arrow[from=3-2, to=3-3]
	\arrow[""{name=5, anchor=center, inner sep=0}, from=3-2, to=4-2]
	\arrow[from=3-3, to=3-4]
	\arrow[""{name=6, anchor=center, inner sep=0}, from=3-3, to=4-3]
	\arrow[""{name=7, anchor=center, inner sep=0}, from=3-4, to=4-4]
	\arrow[from=4-1, to=4-2]
	\arrow[""{name=8, anchor=center, inner sep=0}, from=4-1, to=5-1]
	\arrow[from=4-2, to=4-3]
	\arrow[""{name=9, anchor=center, inner sep=0}, from=4-2, to=5-2]
	\arrow[from=4-3, to=4-4]
	\arrow[""{name=10, anchor=center, inner sep=0}, from=4-3, to=5-3]
	\arrow[""{name=11, anchor=center, inner sep=0}, from=4-4, to=5-4]
	\arrow[from=5-1, to=5-2]
	\arrow[from=5-2, to=5-3]
	\arrow[from=5-3, to=5-4]
	\arrow["\bullet"{description}, draw=none, from=0, to=1]
	\arrow["1"{description}, draw=none, from=2, to=3]
	\arrow["\bullet"{description}, draw=none, from=3, to=4]
	\arrow["\bullet"{description}, draw=none, from=5, to=6]
	\arrow["\bullet"{description}, draw=none, from=6, to=7]
	\arrow["\bullet"{description}, draw=none, from=8, to=9]
	\arrow["\bullet"{description}, draw=none, from=9, to=10]
	\arrow["\bullet"{description}, draw=none, from=10, to=11]
\end{tikzcd}}
\begin{tikzcd}
	{} & {}
	\arrow["{\bs}",  squiggly, from=1-1, to=1-2]
\end{tikzcd}
\Can
\)
\\ \hline
$\mathbf{ua}$
&
\(\adjustbox{scale=0.60}{\begin{tikzcd}
	&& {} & {} \\
	& {} & {} & {} \\
	{} & {} & {} & {} \\
	{} & {} & {} & {}
	\arrow[from=1-3, to=1-4]
	\arrow[from=1-3, to=2-3]
	\arrow[from=1-4, to=4-4]
	\arrow[from=2-2, to=2-3]
	\arrow[from=2-2, to=3-2]
	\arrow[""{name=0, anchor=center, inner sep=0}, draw=none, from=2-3, to=3-3]
	\arrow[""{name=1, anchor=center, inner sep=0}, from=2-3, to=4-3]
	\arrow[""{name=2, anchor=center, inner sep=0}, draw=none, from=2-4, to=3-4]
	\arrow[from=3-1, to=3-2]
	\arrow[""{name=3, anchor=center, inner sep=0}, from=3-1, to=4-1]
	\arrow[""{name=4, anchor=center, inner sep=0}, from=3-2, to=4-2]
	\arrow[from=4-1, to=4-2]
	\arrow[from=4-2, to=4-3]
	\arrow[from=4-3, to=4-4]
	\arrow["3"{description}, draw=none, from=0, to=2]
	\arrow["1"{description}, draw=none, from=3, to=4]
	\arrow["2"{description}, draw=none, from=3-2, to=1]
\end{tikzcd}}
\begin{tikzcd}
	{} & {}
	\arrow["{\bu}",  squiggly, from=1-1, to=1-2]
\end{tikzcd}
\hspace*{-3mm}
\adjustbox{scale=0.60}{\begin{tikzcd}
	&& {} & {} \\
	& {} & {} & {} \\
	{} & {} & {} & {} \\
	{} & {} & {} & {}
	\arrow[from=1-3, to=1-4]
	\arrow[""{name=0, anchor=center, inner sep=0}, from=1-3, to=2-3]
	\arrow[""{name=1, anchor=center, inner sep=0}, from=1-4, to=2-4]
	\arrow[from=2-2, to=2-3]
	\arrow[from=2-2, to=3-2]
	\arrow[from=2-3, to=2-4]
	\arrow[draw=none, from=2-3, to=3-3]
	\arrow[""{name=2, anchor=center, inner sep=0}, from=2-3, to=4-3]
	\arrow[from=2-4, to=4-4]
	\arrow[from=3-1, to=3-2]
	\arrow[""{name=3, anchor=center, inner sep=0}, from=3-1, to=4-1]
	\arrow[""{name=4, anchor=center, inner sep=0}, from=3-2, to=4-2]
	\arrow["\bullet"{description}, draw=none, from=3-3, to=3-4]
	\arrow[from=4-1, to=4-2]
	\arrow[from=4-2, to=4-3]
	\arrow[from=4-3, to=4-4]
	\arrow["\bullet"{description}, draw=none, from=0, to=1]
	\arrow["1"{description}, draw=none, from=3, to=4]
	\arrow["2"{description}, draw=none, from=3-2, to=2]
\end{tikzcd}}
\begin{tikzcd}
	{} & {}
	\arrow["{\bs}",  squiggly, from=1-1, to=1-2]
\end{tikzcd}
\hspace*{-3mm}
\adjustbox{scale=0.60}{\begin{tikzcd}
	&& {} & {} \\
	& {} & {} & {} \\
	{} & {} & {} & {} \\
	{} & {} & {} & {}
	\arrow[from=1-3, to=1-4]
	\arrow[""{name=0, anchor=center, inner sep=0}, from=1-3, to=2-3]
	\arrow[""{name=1, anchor=center, inner sep=0}, from=1-4, to=2-4]
	\arrow[from=2-2, to=2-3]
	\arrow[""{name=2, anchor=center, inner sep=0}, from=2-2, to=3-2]
	\arrow[from=2-3, to=2-4]
	\arrow[""{name=3, anchor=center, inner sep=0}, from=2-3, to=3-3]
	\arrow[""{name=4, anchor=center, inner sep=0}, from=2-4, to=3-4]
	\arrow[from=3-1, to=3-2]
	\arrow[""{name=5, anchor=center, inner sep=0}, from=3-1, to=4-1]
	\arrow[from=3-2, to=3-3]
	\arrow[""{name=6, anchor=center, inner sep=0}, from=3-2, to=4-2]
	\arrow[from=3-3, to=3-4]
	\arrow[""{name=7, anchor=center, inner sep=0}, from=3-3, to=4-3]
	\arrow[""{name=8, anchor=center, inner sep=0}, from=3-4, to=4-4]
	\arrow[from=4-1, to=4-2]
	\arrow[from=4-2, to=4-3]
	\arrow[from=4-3, to=4-4]
	\arrow["\bullet"{description}, draw=none, from=0, to=1]
	\arrow["\bullet"{description}, draw=none, from=2, to=3]
	\arrow["\bullet"{description}, draw=none, from=3, to=4]
	\arrow["1"{description}, draw=none, from=5, to=6]
	\arrow["\bullet"{description}, draw=none, from=6, to=7]
	\arrow["\bullet"{description}, draw=none, from=7, to=8]
\end{tikzcd}}
\begin{tikzcd}
	{} & {}
	\arrow["{\bd}",  squiggly, from=1-1, to=1-2]
\end{tikzcd}
\Can
\)
\\ \hline
$\mathbf{ub}$
&
\(\adjustbox{scale=0.60}{\begin{tikzcd}
	&& {} & {} \\
	& {} & {} & {} \\
	{} & {} & {} & {} \\
	{} & {} & {} & {}
	\arrow[from=1-3, to=1-4]
	\arrow[from=1-3, to=2-3]
	\arrow[from=1-4, to=4-4]
	\arrow[from=2-2, to=2-3]
	\arrow[""{name=0, anchor=center, inner sep=0}, from=2-2, to=3-2]
	\arrow[""{name=1, anchor=center, inner sep=0}, from=2-3, to=3-3]
	\arrow[""{name=2, anchor=center, inner sep=0}, draw=none, from=2-4, to=3-4]
	\arrow[from=3-1, to=3-2]
	\arrow[""{name=3, anchor=center, inner sep=0}, from=3-1, to=4-1]
	\arrow[from=3-2, to=3-3]
	\arrow[""{name=4, anchor=center, inner sep=0}, from=3-3, to=4-3]
	\arrow[from=4-1, to=4-2]
	\arrow[from=4-2, to=4-3]
	\arrow[from=4-3, to=4-4]
	\arrow["1"{description}, draw=none, from=0, to=1]
	\arrow["3"{description}, draw=none, from=1, to=2]
	\arrow["2"{description}, draw=none, from=3, to=4]
\end{tikzcd}}
\begin{tikzcd}
	{} & {}
	\arrow["{\bu}",  squiggly, from=1-1, to=1-2]
\end{tikzcd}
\hspace*{-3mm}
\adjustbox{scale=0.60}{\begin{tikzcd}
	&& {} & {} \\
	& {} & {} & {} \\
	{} & {} & {} & {} \\
	{} && {} & {}
	\arrow[from=1-3, to=1-4]
	\arrow[""{name=0, anchor=center, inner sep=0}, from=1-3, to=2-3]
	\arrow[""{name=1, anchor=center, inner sep=0}, from=1-4, to=2-4]
	\arrow[from=2-2, to=2-3]
	\arrow[""{name=2, anchor=center, inner sep=0}, from=2-2, to=3-2]
	\arrow[from=2-3, to=2-4]
	\arrow[""{name=3, anchor=center, inner sep=0}, from=2-3, to=3-3]
	\arrow[""{name=4, anchor=center, inner sep=0}, from=2-4, to=3-4]
	\arrow[from=3-1, to=3-2]
	\arrow[""{name=5, anchor=center, inner sep=0}, from=3-1, to=4-1]
	\arrow[from=3-2, to=3-3]
	\arrow[from=3-3, to=3-4]
	\arrow[""{name=6, anchor=center, inner sep=0}, from=3-3, to=4-3]
	\arrow[""{name=7, anchor=center, inner sep=0}, from=3-4, to=4-4]
	\arrow[from=4-1, to=4-3]
	\arrow[from=4-3, to=4-4]
	\arrow["\bullet"{description}, draw=none, from=0, to=1]
	\arrow["1"{description}, draw=none, from=2, to=3]
	\arrow["\bullet"{description}, draw=none, from=3, to=4]
	\arrow["2"{description}, draw=none, from=5, to=6]
	\arrow["\bullet"{description}, draw=none, from=6, to=7]
\end{tikzcd}}
\begin{tikzcd}
	{} & {}
	\arrow["{\bd}",  squiggly, from=1-1, to=1-2]
\end{tikzcd}
\hspace*{-3mm}
\adjustbox{scale=0.60}{\begin{tikzcd}
	&& {} & {} \\
	& {} & {} & {} \\
	{} & {} & {} & {} \\
	{} & {} & {} & {}
	\arrow[from=1-3, to=1-4]
	\arrow[""{name=0, anchor=center, inner sep=0}, from=1-3, to=2-3]
	\arrow[""{name=1, anchor=center, inner sep=0}, from=1-4, to=2-4]
	\arrow[from=2-2, to=2-3]
	\arrow[""{name=2, anchor=center, inner sep=0}, from=2-2, to=3-2]
	\arrow[from=2-3, to=2-4]
	\arrow[""{name=3, anchor=center, inner sep=0}, from=2-3, to=3-3]
	\arrow[""{name=4, anchor=center, inner sep=0}, from=2-4, to=3-4]
	\arrow[from=3-1, to=3-2]
	\arrow[""{name=5, anchor=center, inner sep=0}, from=3-1, to=4-1]
	\arrow[from=3-2, to=3-3]
	\arrow[""{name=6, anchor=center, inner sep=0}, from=3-2, to=4-2]
	\arrow[from=3-3, to=3-4]
	\arrow[""{name=7, anchor=center, inner sep=0}, from=3-3, to=4-3]
	\arrow[""{name=8, anchor=center, inner sep=0}, from=3-4, to=4-4]
	\arrow[from=4-1, to=4-2]
	\arrow[from=4-2, to=4-3]
	\arrow[from=4-3, to=4-4]
	\arrow["\bullet"{description}, draw=none, from=0, to=1]
	\arrow["1"{description}, draw=none, from=2, to=3]
	\arrow["\bullet"{description}, draw=none, from=3, to=4]
	\arrow["\bullet"{description}, draw=none, from=5, to=6]
	\arrow["\bullet"{description}, draw=none, from=6, to=7]
	\arrow["\bullet"{description}, draw=none, from=7, to=8]
\end{tikzcd}}
\begin{tikzcd}
	{} & {}
	\arrow["{\bs}",  squiggly, from=1-1, to=1-2]
\end{tikzcd}
\Can
\)
\\ \hline
$\mathbf{s}$
&
\(\adjustbox{scale=0.60}{\begin{tikzcd}
	&& {} & {} \\
	& {} & {} & {} \\
	{} & {} \\
	{} & {} && {}
	\arrow[from=1-3, to=1-4]
	\arrow[""{name=0, anchor=center, inner sep=0}, from=1-3, to=2-3]
	\arrow[""{name=1, anchor=center, inner sep=0}, from=1-4, to=2-4]
	\arrow[from=2-2, to=2-3]
	\arrow[from=2-2, to=3-2]
	\arrow[from=2-3, to=2-4]
	\arrow[""{name=2, anchor=center, inner sep=0}, from=2-4, to=4-4]
	\arrow[from=3-1, to=3-2]
	\arrow[""{name=3, anchor=center, inner sep=0}, from=3-1, to=4-1]
	\arrow[""{name=4, anchor=center, inner sep=0}, from=3-2, to=4-2]
	\arrow[from=4-1, to=4-2]
	\arrow[from=4-2, to=4-4]
	\arrow["1"{description}, draw=none, from=0, to=1]
	\arrow["2"{description}, draw=none, from=3, to=4]
	\arrow["3"{description}, draw=none, from=3-2, to=2]
\end{tikzcd}}
\begin{tikzcd}
	{} & {}
	\arrow["{\bs}",  squiggly, from=1-1, to=1-2]
\end{tikzcd}
\hspace*{-3mm}
\adjustbox{scale=0.60}{\begin{tikzcd}
	&& {} & {} \\
	& {} & {} & {} \\
	{} & {} & {} & {} \\
	{} & {} & {} & {}
	\arrow[from=1-3, to=1-4]
	\arrow[""{name=0, anchor=center, inner sep=0}, from=1-3, to=2-3]
	\arrow[""{name=1, anchor=center, inner sep=0}, from=1-4, to=2-4]
	\arrow[from=2-2, to=2-3]
	\arrow[""{name=2, anchor=center, inner sep=0}, from=2-2, to=3-2]
	\arrow[from=2-3, to=2-4]
	\arrow[""{name=3, anchor=center, inner sep=0}, from=2-3, to=3-3]
	\arrow[""{name=4, anchor=center, inner sep=0}, from=2-4, to=3-4]
	\arrow[from=3-1, to=3-2]
	\arrow[""{name=5, anchor=center, inner sep=0}, from=3-1, to=4-1]
	\arrow[from=3-2, to=3-3]
	\arrow[""{name=6, anchor=center, inner sep=0}, from=3-2, to=4-2]
	\arrow[from=3-3, to=3-4]
	\arrow[""{name=7, anchor=center, inner sep=0}, from=3-3, to=4-3]
	\arrow[""{name=8, anchor=center, inner sep=0}, from=3-4, to=4-4]
	\arrow[from=4-1, to=4-2]
	\arrow[from=4-2, to=4-3]
	\arrow[from=4-3, to=4-4]
	\arrow["1"{description}, draw=none, from=0, to=1]
	\arrow["\bullet"{description}, draw=none, from=2, to=3]
	\arrow["\bullet"{description}, draw=none, from=3, to=4]
	\arrow["2"{description}, draw=none, from=5, to=6]
	\arrow["\bullet"{description}, draw=none, from=6, to=7]
	\arrow["\bullet"{description}, draw=none, from=7, to=8]
\end{tikzcd}}
\begin{tikzcd}
	{} & {}
	\arrow["{\bd}",  squiggly, from=1-1, to=1-2]
\end{tikzcd}
\hspace*{-3mm}
\adjustbox{scale=0.60}{\begin{tikzcd}
	&& {} & {} \\
	& {} & {} & {} \\
	{} & {} & {} & {} \\
	{} & {} & {} & {}
	\arrow[from=1-3, to=1-4]
	\arrow[""{name=0, anchor=center, inner sep=0}, from=1-3, to=2-3]
	\arrow[""{name=1, anchor=center, inner sep=0}, from=1-4, to=2-4]
	\arrow[from=2-2, to=2-3]
	\arrow[""{name=2, anchor=center, inner sep=0}, from=2-2, to=3-2]
	\arrow[from=2-3, to=2-4]
	\arrow[""{name=3, anchor=center, inner sep=0}, from=2-3, to=3-3]
	\arrow[""{name=4, anchor=center, inner sep=0}, from=2-4, to=3-4]
	\arrow[from=3-1, to=3-2]
	\arrow[""{name=5, anchor=center, inner sep=0}, from=3-1, to=4-1]
	\arrow[from=3-2, to=3-3]
	\arrow[""{name=6, anchor=center, inner sep=0}, from=3-2, to=4-2]
	\arrow[from=3-3, to=3-4]
	\arrow[""{name=7, anchor=center, inner sep=0}, from=3-3, to=4-3]
	\arrow[""{name=8, anchor=center, inner sep=0}, from=3-4, to=4-4]
	\arrow[from=4-1, to=4-2]
	\arrow[from=4-2, to=4-3]
	\arrow[from=4-3, to=4-4]
	\arrow["1"{description}, draw=none, from=0, to=1]
	\arrow["\bullet"{description}, draw=none, from=2, to=3]
	\arrow["\bullet"{description}, draw=none, from=3, to=4]
	\arrow["\bullet"{description}, draw=none, from=5, to=6]
	\arrow["\bullet"{description}, draw=none, from=6, to=7]
	\arrow["\bullet"{description}, draw=none, from=7, to=8]
\end{tikzcd}}
\begin{tikzcd}
	{} & {}
	\arrow["{\bu}",  squiggly, from=1-1, to=1-2]
\end{tikzcd}
\Can
\)
\\ \hline
$\mathbf{s}_1$
&
\(\adjustbox{scale=0.60}{\begin{tikzcd}
	&& {} & {} \\
	& {} & {} & {} \\
	& {} & {} & {} \\
	{} & {} && {} \\
	{} & {} && {}
	\arrow[from=1-3, to=1-4]
	\arrow[from=1-3, to=2-3]
	\arrow[from=1-4, to=2-4]
	\arrow[from=2-2, to=2-3]
	\arrow[""{name=0, anchor=center, inner sep=0}, equals, from=2-2, to=3-2]
	\arrow["2"{description}, draw=none, from=2-3, to=2-4]
	\arrow[""{name=1, anchor=center, inner sep=0}, from=2-3, to=3-3]
	\arrow[from=2-4, to=3-4]
	\arrow[from=3-2, to=3-3]
	\arrow[from=3-2, to=4-2]
	\arrow[from=3-3, to=3-4]
	\arrow[from=3-4, to=5-4]
	\arrow[from=4-1, to=4-2]
	\arrow[""{name=2, anchor=center, inner sep=0}, from=4-1, to=5-1]
	\arrow["4"{description}, draw=none, from=4-2, to=4-4]
	\arrow[""{name=3, anchor=center, inner sep=0}, from=4-2, to=5-2]
	\arrow[from=5-1, to=5-2]
	\arrow[from=5-2, to=5-4]
	\arrow["1"{description}, draw=none, from=0, to=1]
	\arrow["3"{description}, draw=none, from=2, to=3]
\end{tikzcd}}
\begin{tikzcd}
	{} & {}
	\arrow["{\bs_1}",  squiggly, from=1-1, to=1-2]
\end{tikzcd}
\hspace*{-3mm}
\adjustbox{scale=0.60}{\begin{tikzcd}
	&& {} & {} \\
	& {} & {} & {} \\
	& {} & {} & {} \\
	{} & {} & {} & {} \\
	{} & {} & {} & {}
	\arrow[from=1-3, to=1-4]
	\arrow[from=1-3, to=2-3]
	\arrow[from=1-4, to=2-4]
	\arrow[from=2-2, to=2-3]
	\arrow[""{name=0, anchor=center, inner sep=0}, equals, from=2-2, to=3-2]
	\arrow["2"{description}, draw=none, from=2-3, to=2-4]
	\arrow[""{name=1, anchor=center, inner sep=0}, from=2-3, to=3-3]
	\arrow[from=2-4, to=3-4]
	\arrow[from=3-2, to=3-3]
	\arrow[""{name=2, anchor=center, inner sep=0}, from=3-2, to=4-2]
	\arrow[from=3-3, to=3-4]
	\arrow[""{name=3, anchor=center, inner sep=0}, from=3-3, to=4-3]
	\arrow[""{name=4, anchor=center, inner sep=0}, from=3-4, to=4-4]
	\arrow[from=4-1, to=4-2]
	\arrow[""{name=5, anchor=center, inner sep=0}, from=4-1, to=5-1]
	\arrow[from=4-2, to=4-3]
	\arrow[""{name=6, anchor=center, inner sep=0}, from=4-2, to=5-2]
	\arrow[from=4-3, to=4-4]
	\arrow[""{name=7, anchor=center, inner sep=0}, from=4-3, to=5-3]
	\arrow[""{name=8, anchor=center, inner sep=0}, from=4-4, to=5-4]
	\arrow[from=5-1, to=5-2]
	\arrow[from=5-2, to=5-3]
	\arrow[from=5-3, to=5-4]
	\arrow["1"{description}, draw=none, from=0, to=1]
	\arrow["\bullet"{description}, draw=none, from=2, to=3]
	\arrow["\bullet"{description}, draw=none, from=3, to=4]
	\arrow["3"{description}, draw=none, from=5, to=6]
	\arrow["\bullet"{description}, draw=none, from=6, to=7]
	\arrow["\bullet"{description}, draw=none, from=7, to=8]
\end{tikzcd}}
\begin{tikzcd}
	{} & {}
	\arrow["{\bu}",  squiggly, from=1-1, to=1-2]
\end{tikzcd}
\hspace*{-3mm}
\adjustbox{scale=0.60}{\begin{tikzcd}
	&& {} & {} \\
	& {} & {} & {} \\
	& {} & {} & {} \\
	{} & {} & {} & {} \\
	{} & {} & {} & {}
	\arrow[from=1-3, to=1-4]
	\arrow[""{name=0, anchor=center, inner sep=0}, from=1-3, to=2-3]
	\arrow[""{name=1, anchor=center, inner sep=0}, from=1-4, to=2-4]
	\arrow[from=2-2, to=2-3]
	\arrow[""{name=2, anchor=center, inner sep=0}, equals, from=2-2, to=3-2]
	\arrow[from=2-3, to=2-4]
	\arrow[""{name=3, anchor=center, inner sep=0}, from=2-3, to=3-3]
	\arrow[""{name=4, anchor=center, inner sep=0}, from=2-4, to=3-4]
	\arrow[from=3-2, to=3-3]
	\arrow[""{name=5, anchor=center, inner sep=0}, from=3-2, to=4-2]
	\arrow[from=3-3, to=3-4]
	\arrow[""{name=6, anchor=center, inner sep=0}, from=3-3, to=4-3]
	\arrow[""{name=7, anchor=center, inner sep=0}, from=3-4, to=4-4]
	\arrow[from=4-1, to=4-2]
	\arrow[""{name=8, anchor=center, inner sep=0}, from=4-1, to=5-1]
	\arrow[from=4-2, to=4-3]
	\arrow[""{name=9, anchor=center, inner sep=0}, from=4-2, to=5-2]
	\arrow[from=4-3, to=4-4]
	\arrow[""{name=10, anchor=center, inner sep=0}, from=4-3, to=5-3]
	\arrow[""{name=11, anchor=center, inner sep=0}, from=4-4, to=5-4]
	\arrow[from=5-1, to=5-2]
	\arrow[from=5-2, to=5-3]
	\arrow[from=5-3, to=5-4]
	\arrow["\bullet"{description}, draw=none, from=0, to=1]
	\arrow["1"{description}, draw=none, from=2, to=3]
	\arrow["\bullet"{description}, draw=none, from=3, to=4]
	\arrow["\bullet"{description}, draw=none, from=5, to=6]
	\arrow["\bullet"{description}, draw=none, from=6, to=7]
	\arrow["3"{description}, draw=none, from=8, to=9]
	\arrow["\bullet"{description}, draw=none, from=9, to=10]
	\arrow["\bullet"{description}, draw=none, from=10, to=11]
\end{tikzcd}}
\begin{tikzcd}
	{} & {}
	\arrow["{\bs}",  squiggly, from=1-1, to=1-2]
\end{tikzcd}
\hspace*{-3mm}
\adjustbox{scale=0.60}{\begin{tikzcd}
	&& {} & {} \\
	& {} & {} & {} \\
	{} & {} & {} & {} \\
	{} & {} & {} & {}
	\arrow[from=1-3, to=1-4]
	\arrow[""{name=0, anchor=center, inner sep=0}, from=1-3, to=2-3]
	\arrow[""{name=1, anchor=center, inner sep=0}, from=1-4, to=2-4]
	\arrow[from=2-2, to=2-3]
	\arrow[""{name=2, anchor=center, inner sep=0}, from=2-2, to=3-2]
	\arrow[from=2-3, to=2-4]
	\arrow[""{name=3, anchor=center, inner sep=0}, from=2-3, to=3-3]
	\arrow[""{name=4, anchor=center, inner sep=0}, from=2-4, to=3-4]
	\arrow[from=3-1, to=3-2]
	\arrow[""{name=5, anchor=center, inner sep=0}, from=3-1, to=4-1]
	\arrow[from=3-2, to=3-3]
	\arrow[""{name=6, anchor=center, inner sep=0}, from=3-2, to=4-2]
	\arrow[from=3-3, to=3-4]
	\arrow[""{name=7, anchor=center, inner sep=0}, from=3-3, to=4-3]
	\arrow[""{name=8, anchor=center, inner sep=0}, from=3-4, to=4-4]
	\arrow[from=4-1, to=4-2]
	\arrow[from=4-2, to=4-3]
	\arrow[from=4-3, to=4-4]
	\arrow["\bullet"{description}, draw=none, from=0, to=1]
	\arrow["\bullet"{description}, draw=none, from=2, to=3]
	\arrow["\bullet"{description}, draw=none, from=3, to=4]
	\arrow["3"{description}, draw=none, from=5, to=6]
	\arrow["\bullet"{description}, draw=none, from=6, to=7]
	\arrow["\bullet"{description}, draw=none, from=7, to=8]
\end{tikzcd}}
\begin{tikzcd}
	{} & {}
	\arrow["{\bd}",  squiggly, from=1-1, to=1-2]
\end{tikzcd}
\Can
\)
\\ \hline
\end{tabular}
\end{table}

\begin{theorem}\label{thm:of-interest5} Every two $\Sg$-paths of interest starting and ending at the same $\Sg$-schemes are equivalent.
\end{theorem}

\begin{proof}

\textbf{A.} First, we show that for every  $\Sg$-scheme of interest exhibiting two different kinds of configurations of interest, the corresponding two canonical $\Sg$-paths are equivalent.

1) $\mathbf{da}$ and $\mathbf{db}$. A $\Sg$-scheme exhibiting the both configurations looks as follows:
$$S_1=\adjustbox{scale=0.60}{\begin{tikzcd}
	&& {} & {} \\
	& {} & {} & {} \\
	{} & {} & {} & {} \\
	{} &&& {}
	\arrow[from=1-3, to=1-4]
	\arrow[""{name=0, anchor=center, inner sep=0}, from=1-3, to=2-3]
	\arrow[""{name=1, anchor=center, inner sep=0}, from=1-4, to=2-4]
	\arrow[from=2-2, to=2-3]
	\arrow[""{name=2, anchor=center, inner sep=0}, from=2-2, to=3-2]
	\arrow[from=2-3, to=2-4]
	\arrow[""{name=3, anchor=center, inner sep=0}, from=2-3, to=3-3]
	\arrow[""{name=4, anchor=center, inner sep=0}, from=2-4, to=3-4]
	\arrow[from=3-1, to=3-2]
	\arrow[""{name=5, anchor=center, inner sep=0}, from=3-1, to=4-1]
	\arrow[from=3-2, to=3-3]
	\arrow[from=3-3, to=3-4]
	\arrow[""{name=6, anchor=center, inner sep=0}, from=3-4, to=4-4]
	\arrow[from=4-1, to=4-4]
	\arrow["1"{description}, draw=none, from=0, to=1]
	\arrow["2"{description}, draw=none, from=2, to=3]
	\arrow["3"{description}, draw=none, from=3, to=4]
	\arrow["4"{description}, draw=none, from=5, to=6]
\end{tikzcd}}\,.$$
Taking the $\mathbf{da}$-canonical $\Sg$-path and the $\mathbf{db}$-canonical $\Sg$-path  we obtain the top and the bottom lines of the diagram below. The indicated equivalences between $\Sg$-paths are given by  \cref{cor:length2}. Hence, the two canonical $\Sg$-paths are equivalent.
\[\begin{tikzcd}
	{S_1} & {S_2} & {S_3} & \Can \\
	& {S_4} & {S_5}
	\arrow["d", squiggly, from=1-1, to=1-2]
	\arrow[""{name=0, anchor=center, inner sep=0}, "\bd"', squiggly, from=1-1, to=2-2]
	\arrow["\bu", squiggly, from=1-2, to=1-3]
	\arrow[""{name=1, anchor=center, inner sep=0}, "\bs", squiggly, from=1-2, to=2-3]
	\arrow["\bs", squiggly, from=1-3, to=1-4]
	\arrow[""{name=2, anchor=center, inner sep=0}, "\bd"', squiggly, from=2-2, to=1-2]
	\arrow["\bs"', squiggly, from=2-2, to=2-3]
	\arrow[""{name=3, anchor=center, inner sep=0}, "\bu"', squiggly, from=2-3, to=1-4]
	\arrow["\equiv"{description}, shift left=2, draw=none, from=0, to=2]
	\arrow["\equiv"{description}, draw=none, from=1, to=3]
	\arrow["\equiv"{description, pos=0.6}, shift right=3, draw=none, from=2, to=1]
\end{tikzcd}\, .\]

2) $\mathbf{da}$ and $\mathbf{dc}$. A $\Sg$-scheme exhibiting the both configurations is of the form
$$S_1=\adjustbox{scale=0.60}{\begin{tikzcd}
	&& {} & {} \\
	& {} & {} \\
	& {} & {} & {} \\
	{} & {} & {} & {} \\
	{} &&& {}
	\arrow[from=1-3, to=1-4]
	\arrow[from=1-3, to=2-3]
	\arrow[""{name=0, anchor=center, inner sep=0}, from=1-4, to=3-4]
	\arrow[from=2-2, to=2-3]
	\arrow[""{name=1, anchor=center, inner sep=0}, equals, from=2-2, to=3-2]
	\arrow[""{name=2, anchor=center, inner sep=0}, from=2-3, to=3-3]
	\arrow[from=3-2, to=3-3]
	\arrow[""{name=3, anchor=center, inner sep=0}, from=3-2, to=4-2]
	\arrow[from=3-3, to=3-4]
	\arrow[""{name=4, anchor=center, inner sep=0}, from=3-3, to=4-3]
	\arrow[""{name=5, anchor=center, inner sep=0}, from=3-4, to=4-4]
	\arrow[from=4-1, to=4-2]
	\arrow[""{name=6, anchor=center, inner sep=0}, from=4-1, to=5-1]
	\arrow[from=4-2, to=4-3]
	\arrow[from=4-3, to=4-4]
	\arrow[""{name=7, anchor=center, inner sep=0}, from=4-4, to=5-4]
	\arrow[from=5-1, to=5-4]
	\arrow["1"{description}, draw=none, from=1, to=2]
	\arrow["2"{description}, draw=none, from=2-3, to=0]
	\arrow["3"{description}, draw=none, from=3, to=4]
	\arrow["4"{description}, draw=none, from=4, to=5]
	\arrow["5"{description}, draw=none, from=6, to=7]
\end{tikzcd}}\,.$$
 Putting the $\mathbf{da}$-canonical $\Sg$-path and the $\mathbf{dc}$-canonical $\Sg$-path on the top and on the bottom of the following diagram, respectively, we obtain equivalences of $\Sg$-paths showing that the two canonical $\Sg$-paths are equivalent.
\[\begin{tikzcd}
	{S_1} && {S_2} & {S_3} &&& \Can \\
	& {S_4} && {S_5} && {S_6}
	\arrow["\bd", squiggly, from=1-1, to=1-3]
	\arrow[""{name=0, anchor=center, inner sep=0}, "\bd"', squiggly, from=1-1, to=2-2]
	\arrow["\bu", squiggly, from=1-3, to=1-4]
	\arrow[""{name=1, anchor=center, inner sep=0}, "{{{{\bs}_1}}}"', squiggly, from=1-3, to=2-4]
	\arrow["\bs", squiggly, from=1-4, to=1-7]
	\arrow[""{name=2, anchor=center, inner sep=0}, "\bs"', squiggly, from=1-4, to=2-6]
	\arrow["{{{{\bs}_1}}}"', squiggly, from=2-2, to=2-4]
	\arrow["\bu"', squiggly, from=2-4, to=2-6]
	\arrow[""{name=3, anchor=center, inner sep=0}, "\bs"', squiggly, from=2-6, to=1-7]
	\arrow["\equiv"{description}, draw=none, from=0, to=1]
	\arrow["\equiv"{description}, draw=none, from=1, to=2]
	\arrow["\equiv"{description}, draw=none, from=2, to=3]
\end{tikzcd}\]

3) $\mathbf{da}$ and $\mathbf{ua}$. Consider the $\Sg$-schemes
\[S_1=\adjustbox{scale=0.60}{\begin{tikzcd}
	&& {} & {} \\
	& {} & {} \\
	{} & {} & {} & {} \\
	{} & {} & {} & {}
	\arrow[from=1-3, to=1-4]
	\arrow[from=1-3, to=2-3]
	\arrow[""{name=0, anchor=center, inner sep=0}, from=1-4, to=3-4]
	\arrow[from=2-2, to=2-3]
	\arrow[""{name=1, anchor=center, inner sep=0}, from=2-2, to=3-2]
	\arrow[""{name=2, anchor=center, inner sep=0}, from=2-3, to=3-3]
	\arrow[from=3-1, to=3-2]
	\arrow[""{name=3, anchor=center, inner sep=0}, from=3-1, to=4-1]
	\arrow[from=3-2, to=3-3]
	\arrow[""{name=4, anchor=center, inner sep=0}, from=3-2, to=4-2]
	\arrow[from=3-3, to=3-4]
	\arrow[""{name=5, anchor=center, inner sep=0}, from=3-3, to=4-3]
	\arrow[""{name=6, anchor=center, inner sep=0}, from=3-4, to=4-4]
	\arrow[from=4-1, to=4-2]
	\arrow[from=4-2, to=4-3]
	\arrow[from=4-3, to=4-4]
	\arrow["1"{description}, draw=none, from=1, to=2]
	\arrow["2"{description}, draw=none, from=2-3, to=0]
	\arrow["3"{description}, draw=none, from=3, to=4]
	\arrow["4"{description}, draw=none, from=4, to=5]
	\arrow["5"{description}, draw=none, from=5, to=6]
\end{tikzcd}}
\qquad\text{ and }\qquad%
R=\adjustbox{scale=0.60}{\begin{tikzcd}
	&& {} & {} \\
	& {} & {} & {} \\
	{} & {} & {} & {} \\
	{} & {} & {} & {}
	\arrow[from=1-3, to=1-4]
	\arrow[""{name=0, anchor=center, inner sep=0}, from=1-3, to=2-3]
	\arrow[""{name=1, anchor=center, inner sep=0}, from=1-4, to=2-4]
	\arrow[from=2-2, to=2-3]
	\arrow[""{name=2, anchor=center, inner sep=0}, from=2-2, to=3-2]
	\arrow[from=2-3, to=2-4]
	\arrow[""{name=3, anchor=center, inner sep=0}, from=2-3, to=3-3]
	\arrow[""{name=4, anchor=center, inner sep=0}, from=2-4, to=3-4]
	\arrow[from=3-1, to=3-2]
	\arrow[""{name=5, anchor=center, inner sep=0}, from=3-1, to=4-1]
	\arrow[from=3-2, to=3-3]
	\arrow[""{name=6, anchor=center, inner sep=0}, from=3-2, to=4-2]
	\arrow[from=3-3, to=3-4]
	\arrow[""{name=7, anchor=center, inner sep=0}, from=3-3, to=4-3]
	\arrow[""{name=8, anchor=center, inner sep=0}, from=3-4, to=4-4]
	\arrow[from=4-1, to=4-2]
	\arrow[from=4-2, to=4-3]
	\arrow[from=4-3, to=4-4]
	\arrow["\bullet"{description}, draw=none, from=0, to=1]
	\arrow["1"{description}, draw=none, from=2, to=3]
	\arrow["\bullet"{description}, draw=none, from=3, to=4]
	\arrow["3"{description}, draw=none, from=5, to=6]
	\arrow["4"{description}, draw=none, from=6, to=7]
	\arrow["\bullet"{description}, draw=none, from=7, to=8]
\end{tikzcd}}
\, .
\]
 A $\Sg$-scheme that exhibits both configurations, $\mathbf{da}$ and $\mathbf{ua}$, is of the form of the $\Sg$-scheme $S_1$. Let the $\mathbf{da}$-canonical and $\mathbf{ua}$-canonical $\Sg$-paths be, respectively, the top line and the bottom line of the following diagram.
\[\begin{tikzcd}
	{S_1} & {S_2} && {S_3} & \Can \\
	& {S_4} & R & {S_5}
	\arrow["\bd", squiggly, from=1-1, to=1-2]
	\arrow[""{name=0, anchor=center, inner sep=0}, "\bu"', squiggly, from=1-1, to=2-2]
	\arrow["\bu", squiggly, from=1-2, to=1-4]
	\arrow["\bs", squiggly, from=1-4, to=1-5]
	\arrow["{\bu\equiv\bs}", squiggly, from=2-2, to=2-3]
	\arrow["\bs"', curve={height=18pt}, squiggly, from=2-2, to=2-4]
	\arrow["\equiv"', shift right, draw=none, from=2-2, to=2-4]
	\arrow[""{name=1, anchor=center, inner sep=0}, "\bd"{description}, squiggly, from=2-3, to=1-4]
	\arrow["\bs", squiggly, from=2-3, to=2-4]
	\arrow[""{name=2, anchor=center, inner sep=0}, "\bd"', squiggly, from=2-4, to=1-5]
	\arrow["\equiv"{description}, shift left, draw=none, from=0, to=1]
	\arrow["\equiv"{description}, draw=none, from=1, to=2]
\end{tikzcd}\]
 Hence, by  \cref{cor:length2}, the two canonical $\Sg$-paths are equivalent.

 4) $\mathbf{da}$ and $\mathbf{ub}$. We have $S_1=\adjustbox{scale=0.60}{\begin{tikzcd}
	&& {} & {} \\
	& {} & {} \\
	{} & {} & {} & {} \\
	{} & {} & {} & {}
	\arrow[from=1-3, to=1-4]
	\arrow[from=1-3, to=2-3]
	\arrow[""{name=0, anchor=center, inner sep=0}, from=1-4, to=3-4]
	\arrow[from=2-2, to=2-3]
	\arrow[""{name=1, anchor=center, inner sep=0}, from=2-2, to=3-2]
	\arrow[""{name=2, anchor=center, inner sep=0}, from=2-3, to=3-3]
	\arrow[from=3-1, to=3-2]
	\arrow[""{name=3, anchor=center, inner sep=0}, from=3-1, to=4-1]
	\arrow[from=3-2, to=3-3]
	\arrow[from=3-3, to=3-4]
	\arrow[""{name=4, anchor=center, inner sep=0}, from=3-3, to=4-3]
	\arrow[""{name=5, anchor=center, inner sep=0}, from=3-4, to=4-4]
	\arrow[from=4-1, to=4-2]
	\arrow[from=4-2, to=4-3]
	\arrow[from=4-3, to=4-4]
	\arrow["1"{description}, draw=none, from=1, to=2]
	\arrow["2"{description}, draw=none, from=2-3, to=0]
	\arrow["3"{description}, draw=none, from=3, to=4]
	\arrow["4"{description}, draw=none, from=4, to=5]
\end{tikzcd}}\; $ and can show the equivalence using
\[\begin{tikzcd}
	{S_1} & {S_2} & {S_3} & \Can \\
	& {S_4}
	\arrow["\bd", squiggly, from=1-1, to=1-2]
	\arrow[""{name=0, anchor=center, inner sep=0}, "\bu"', squiggly, from=1-1, to=2-2]
	\arrow["\bu", squiggly, from=1-2, to=1-3]
	\arrow["\bs", squiggly, from=1-3, to=1-4]
	\arrow[""{name=1, anchor=center, inner sep=0}, "\bd"', squiggly, from=2-2, to=1-3]
	\arrow["\equiv"{description}, draw=none, from=0, to=1]
\end{tikzcd}\, .\]

5) $\mathbf{da}$ and $\mathbf{s}$. The combination of the two configurations give a $\Sg$-scheme of the form $$S_1=\adjustbox{scale=0.60}{\begin{tikzcd}
	&& {} & {} \\
	& {} & {} & {} \\
	{} & {} & {} & {} \\
	{} & {} && {}
	\arrow[from=1-3, to=1-4]
	\arrow[""{name=0, anchor=center, inner sep=0}, from=1-3, to=2-3]
	\arrow[""{name=1, anchor=center, inner sep=0}, from=1-4, to=2-4]
	\arrow[from=2-2, to=2-3]
	\arrow[""{name=2, anchor=center, inner sep=0}, from=2-2, to=3-2]
	\arrow[from=2-3, to=2-4]
	\arrow[""{name=3, anchor=center, inner sep=0}, from=2-3, to=3-3]
	\arrow[""{name=4, anchor=center, inner sep=0}, from=2-4, to=3-4]
	\arrow[from=3-1, to=3-2]
	\arrow[""{name=5, anchor=center, inner sep=0}, from=3-1, to=4-1]
	\arrow[from=3-2, to=3-3]
	\arrow[""{name=6, anchor=center, inner sep=0}, from=3-2, to=4-2]
	\arrow[from=3-3, to=3-4]
	\arrow[""{name=7, anchor=center, inner sep=0}, from=3-4, to=4-4]
	\arrow[from=4-1, to=4-2]
	\arrow[from=4-2, to=4-4]
	\arrow["1"{description}, draw=none, from=0, to=1]
	\arrow["2"{description}, draw=none, from=2, to=3]
	\arrow["3"{description}, draw=none, from=3, to=4]
	\arrow["4"{description}, draw=none, from=5, to=6]
	\arrow["5"{description}, draw=none, from=6, to=7]
\end{tikzcd}}\, ,$$
  and the equivalence of the two  canonical $\Sg$-paths is shown by the diagram
\[\begin{tikzcd}
	{S_1} & {S_2} & {S_3} & \Can \\
	& {S_4} & {S_5}
	\arrow["\bd", squiggly, from=1-1, to=1-2]
	\arrow[""{name=0, anchor=center, inner sep=0}, "\bs"', squiggly, from=1-1, to=2-2]
	\arrow["\bu", squiggly, from=1-2, to=1-3]
	\arrow[""{name=1, anchor=center, inner sep=0}, "\bs"{description}, squiggly, from=1-2, to=2-3]
	\arrow["\bs", squiggly, from=1-3, to=1-4]
	\arrow["\bd"', squiggly, from=2-2, to=2-3]
	\arrow[""{name=2, anchor=center, inner sep=0}, "\bu"', squiggly, from=2-3, to=1-4]
	\arrow["\equiv"{description}, draw=none, from=0, to=1]
	\arrow["\equiv"{description}, draw=none, from=1, to=2]
\end{tikzcd}\,.\]

6) $\mathbf{da}$ and $\mathbf{s}_1$. The equivalence is given by the following diagram
\[\begin{tikzcd}
	& {S_2} & {S_3} \\
	{S_1} & {R_1} & {R_2} && \Can \\
	& {S_4} & {S_5} & {S_6}
	\arrow["\bu", squiggly, from=1-2, to=1-3]
	\arrow[""{name=0, anchor=center, inner sep=0}, "{{{\bs}_1}}", squiggly, from=1-2, to=2-2]
	\arrow[""{name=1, anchor=center, inner sep=0}, "{{{\bs}}}", squiggly, from=1-3, to=2-3]
	\arrow[""{name=2, anchor=center, inner sep=0}, "\bs", squiggly, from=1-3, to=2-5]
	\arrow["\bd", squiggly, from=2-1, to=1-2]
	\arrow["\equiv"{description}, draw=none, from=2-1, to=2-2]
	\arrow["{{{\bs}_1}}"', squiggly, from=2-1, to=3-2]
	\arrow["\bu"', squiggly, from=2-2, to=2-3]
	\arrow[""{name=3, anchor=center, inner sep=0}, "\bd"', squiggly, from=2-2, to=3-2]
	\arrow["\bs"', squiggly, from=2-3, to=2-5]
	\arrow[""{name=4, anchor=center, inner sep=0}, "\bd", squiggly, from=2-3, to=3-3]
	\arrow["\bu"', squiggly, from=3-2, to=3-3]
	\arrow["\bs"', squiggly, from=3-3, to=3-4]
	\arrow[""{name=5, anchor=center, inner sep=0}, "\bd"', squiggly, from=3-4, to=2-5]
	\arrow["\equiv"{description}, draw=none, from=0, to=1]
	\arrow["\equiv"', draw=none, from=1, to=2]
	\arrow["\equiv"{description}, draw=none, from=3, to=4]
	\arrow["\equiv"{description}, draw=none, from=4, to=5]
\end{tikzcd}\]
where
\[
S_1=\hspace{-2mm}\adjustbox{scale=0.60}{\begin{tikzcd}
	&& {} & {} \\
	& {} & {} \\
	& {} & {} & {} \\
	{} & {} & {} & {} \\
	{} & {} && {}
	\arrow[from=1-3, to=1-4]
	\arrow[from=1-3, to=2-3]
	\arrow[""{name=0, anchor=center, inner sep=0}, from=1-4, to=3-4]
	\arrow[from=2-2, to=2-3]
	\arrow[""{name=1, anchor=center, inner sep=0}, equals, from=2-2, to=3-2]
	\arrow[""{name=2, anchor=center, inner sep=0}, from=2-3, to=3-3]
	\arrow[from=3-2, to=3-3]
	\arrow[""{name=3, anchor=center, inner sep=0}, from=3-2, to=4-2]
	\arrow[from=3-3, to=3-4]
	\arrow[""{name=4, anchor=center, inner sep=0}, from=3-3, to=4-3]
	\arrow[""{name=5, anchor=center, inner sep=0}, from=3-4, to=4-4]
	\arrow[from=4-1, to=4-2]
	\arrow[""{name=6, anchor=center, inner sep=0}, from=4-1, to=5-1]
	\arrow[from=4-2, to=4-3]
	\arrow[""{name=7, anchor=center, inner sep=0}, from=4-2, to=5-2]
	\arrow[from=4-3, to=4-4]
	\arrow[""{name=8, anchor=center, inner sep=0}, from=4-4, to=5-4]
	\arrow[from=5-1, to=5-2]
	\arrow[from=5-2, to=5-4]
	\arrow["1"{description}, draw=none, from=1, to=2]
	\arrow["2"{description}, draw=none, from=2-3, to=0]
	\arrow["3"{description}, draw=none, from=3, to=4]
	\arrow["4"{description}, draw=none, from=4, to=5]
	\arrow["5"{description}, draw=none, from=6, to=7]
	\arrow["6"{description}, draw=none, from=7, to=8]
\end{tikzcd}}\, ,\quad
S_2=\hspace{-2mm}
\adjustbox{scale=0.60}{\begin{tikzcd}
	&& {} & {} \\
	& {} & {} \\
	& {} & {} & {} \\
	{} & {} & {} & {} \\
	{} & {} & {} & {}
	\arrow[from=1-3, to=1-4]
	\arrow[from=1-3, to=2-3]
	\arrow[""{name=0, anchor=center, inner sep=0}, from=1-4, to=3-4]
	\arrow[from=2-2, to=2-3]
	\arrow[""{name=1, anchor=center, inner sep=0}, equals, from=2-2, to=3-2]
	\arrow[""{name=2, anchor=center, inner sep=0}, from=2-3, to=3-3]
	\arrow[from=3-2, to=3-3]
	\arrow[""{name=3, anchor=center, inner sep=0}, from=3-2, to=4-2]
	\arrow[from=3-3, to=3-4]
	\arrow[""{name=4, anchor=center, inner sep=0}, from=3-3, to=4-3]
	\arrow[""{name=5, anchor=center, inner sep=0}, from=3-4, to=4-4]
	\arrow[from=4-1, to=4-2]
	\arrow[""{name=6, anchor=center, inner sep=0}, from=4-1, to=5-1]
	\arrow[from=4-2, to=4-3]
	\arrow[""{name=7, anchor=center, inner sep=0}, from=4-2, to=5-2]
	\arrow[from=4-3, to=4-4]
	\arrow[""{name=8, anchor=center, inner sep=0}, from=4-3, to=5-3]
	\arrow[""{name=9, anchor=center, inner sep=0}, from=4-4, to=5-4]
	\arrow[from=5-1, to=5-2]
	\arrow[from=5-2, to=5-3]
	\arrow[from=5-3, to=5-4]
	\arrow["1"{description}, draw=none, from=1, to=2]
	\arrow["2"{description}, draw=none, from=2-3, to=0]
	\arrow["3"{description}, draw=none, from=3, to=4]
	\arrow["4"{description}, draw=none, from=4, to=5]
	\arrow["\bullet"{description}, draw=none, from=6, to=7]
	\arrow["\bullet"{description}, draw=none, from=7, to=8]
	\arrow["\bullet"{description}, draw=none, from=8, to=9]
\end{tikzcd}}\, , \quad
S_3=\hspace{-2mm}
\adjustbox{scale=0.60}{\begin{tikzcd}
	&& {} & {} \\
	& {} & {} & {} \\
	& {} & {} \\
	{} & {} & {} & {} \\
	{} & {} & {} & {}
	\arrow[from=1-3, to=1-4]
	\arrow[""{name=0, anchor=center, inner sep=0}, from=1-3, to=2-3]
	\arrow[""{name=1, anchor=center, inner sep=0}, from=1-4, to=2-4]
	\arrow[from=2-2, to=2-3]
	\arrow[""{name=2, anchor=center, inner sep=0}, equals, from=2-2, to=3-2]
	\arrow[from=2-3, to=2-4]
	\arrow[""{name=3, anchor=center, inner sep=0}, from=2-3, to=3-3]
	\arrow[""{name=4, anchor=center, inner sep=0}, from=2-4, to=4-4]
	\arrow[from=3-2, to=3-3]
	\arrow[""{name=5, anchor=center, inner sep=0}, from=3-2, to=4-2]
	\arrow[""{name=6, anchor=center, inner sep=0}, from=3-3, to=4-3]
	\arrow[from=4-1, to=4-2]
	\arrow[""{name=7, anchor=center, inner sep=0}, from=4-1, to=5-1]
	\arrow[from=4-2, to=4-3]
	\arrow[""{name=8, anchor=center, inner sep=0}, from=4-2, to=5-2]
	\arrow[from=4-3, to=4-4]
	\arrow[""{name=9, anchor=center, inner sep=0}, from=4-3, to=5-3]
	\arrow[""{name=10, anchor=center, inner sep=0}, from=4-4, to=5-4]
	\arrow[from=5-1, to=5-2]
	\arrow[from=5-2, to=5-3]
	\arrow[from=5-3, to=5-4]
	\arrow["\bullet"{description}, draw=none, from=0, to=1]
	\arrow["1"{description}, draw=none, from=2, to=3]
	\arrow["3"{description}, draw=none, from=5, to=6]
	\arrow["\bullet"{description}, draw=none, from=3-3, to=4]
	\arrow["\bullet"{description}, draw=none, from=7, to=8]
	\arrow["\bullet"{description}, draw=none, from=8, to=9]
	\arrow["\bullet"{description}, draw=none, from=9, to=10]
\end{tikzcd}}\, ,\quad
S_4=\hspace{-2mm}
\adjustbox{scale=0.60}{\begin{tikzcd}
	&& {} & {} \\
	& {} & {} \\
	& {} & {} & {} \\
	{} & {} & {} & {} \\
	{} & {} & {} & {}
	\arrow[from=1-3, to=1-4]
	\arrow[from=1-3, to=2-3]
	\arrow[""{name=0, anchor=center, inner sep=0}, from=1-4, to=3-4]
	\arrow[from=2-2, to=2-3]
	\arrow[""{name=1, anchor=center, inner sep=0}, equals, from=2-2, to=3-2]
	\arrow[""{name=2, anchor=center, inner sep=0}, from=2-3, to=3-3]
	\arrow[from=3-2, to=3-3]
	\arrow[""{name=3, anchor=center, inner sep=0}, from=3-2, to=4-2]
	\arrow[from=3-3, to=3-4]
	\arrow[""{name=4, anchor=center, inner sep=0}, from=3-3, to=4-3]
	\arrow[""{name=5, anchor=center, inner sep=0}, from=3-4, to=4-4]
	\arrow[from=4-1, to=4-2]
	\arrow[""{name=6, anchor=center, inner sep=0}, from=4-1, to=5-1]
	\arrow[from=4-2, to=4-3]
	\arrow[""{name=7, anchor=center, inner sep=0}, from=4-2, to=5-2]
	\arrow[from=4-3, to=4-4]
	\arrow[""{name=8, anchor=center, inner sep=0}, from=4-3, to=5-3]
	\arrow[""{name=9, anchor=center, inner sep=0}, from=4-4, to=5-4]
	\arrow[from=5-1, to=5-2]
	\arrow[from=5-2, to=5-3]
	\arrow[from=5-3, to=5-4]
	\arrow["1"{description}, draw=none, from=1, to=2]
	\arrow["2"{description}, draw=none, from=2-3, to=0]
	\arrow["\bullet"{description}, draw=none, from=3, to=4]
	\arrow["\bullet"{description}, draw=none, from=4, to=5]
	\arrow["5"{description}, draw=none, from=6, to=7]
	\arrow["\bullet"{description}, draw=none, from=7, to=8]
	\arrow["\bullet"{description}, draw=none, from=8, to=9]
\end{tikzcd}}\, ,
\]
\[
S_5=\hspace{-2mm}
\adjustbox{scale=0.60}{\begin{tikzcd}
	&& {} & {} \\
	& {} & {} & {} \\
	& {} & {} & {} \\
	{} & {} & {} & {} \\
	{} & {} & {} & {}
	\arrow[from=1-3, to=1-4]
	\arrow[""{name=0, anchor=center, inner sep=0}, from=1-3, to=2-3]
	\arrow[""{name=1, anchor=center, inner sep=0}, from=1-4, to=2-4]
	\arrow[from=2-2, to=2-3]
	\arrow[""{name=2, anchor=center, inner sep=0}, equals, from=2-2, to=3-2]
	\arrow[from=2-3, to=2-4]
	\arrow[""{name=3, anchor=center, inner sep=0}, from=2-3, to=3-3]
	\arrow[""{name=4, anchor=center, inner sep=0}, from=2-4, to=3-4]
	\arrow[from=3-2, to=3-3]
	\arrow[""{name=5, anchor=center, inner sep=0}, from=3-2, to=4-2]
	\arrow[from=3-3, to=3-4]
	\arrow[""{name=6, anchor=center, inner sep=0}, from=3-3, to=4-3]
	\arrow[""{name=7, anchor=center, inner sep=0}, from=3-4, to=4-4]
	\arrow[from=4-1, to=4-2]
	\arrow[""{name=8, anchor=center, inner sep=0}, from=4-1, to=5-1]
	\arrow[from=4-2, to=4-3]
	\arrow[""{name=9, anchor=center, inner sep=0}, from=4-2, to=5-2]
	\arrow[from=4-3, to=4-4]
	\arrow[""{name=10, anchor=center, inner sep=0}, from=4-3, to=5-3]
	\arrow[""{name=11, anchor=center, inner sep=0}, from=4-4, to=5-4]
	\arrow[from=5-1, to=5-2]
	\arrow[from=5-2, to=5-3]
	\arrow[from=5-3, to=5-4]
	\arrow["\bullet"{description}, draw=none, from=0, to=1]
	\arrow["1"{description}, draw=none, from=2, to=3]
	\arrow["\bullet"{description}, draw=none, from=3, to=4]
	\arrow["\bullet"{description}, draw=none, from=5, to=6]
	\arrow["\bullet"{description}, draw=none, from=6, to=7]
	\arrow["5"{description}, draw=none, from=8, to=9]
	\arrow["\bullet"{description}, draw=none, from=9, to=10]
	\arrow["\bullet"{description}, draw=none, from=10, to=11]
\end{tikzcd}}\, ,\quad
R_1=\hspace{-2mm}
\adjustbox{scale=0.60}{\begin{tikzcd}
	&& {} & {} \\
	& {} & {} \\
	& {} & {} & {} \\
	{} & {} & {} & {} \\
	{} & {} & {} & {}
	\arrow[from=1-3, to=1-4]
	\arrow[from=1-3, to=2-3]
	\arrow[""{name=0, anchor=center, inner sep=0}, from=1-4, to=3-4]
	\arrow[from=2-2, to=2-3]
	\arrow[""{name=1, anchor=center, inner sep=0}, equals, from=2-2, to=3-2]
	\arrow[""{name=2, anchor=center, inner sep=0}, from=2-3, to=3-3]
	\arrow[from=3-2, to=3-3]
	\arrow[""{name=3, anchor=center, inner sep=0}, from=3-2, to=4-2]
	\arrow[from=3-3, to=3-4]
	\arrow[""{name=4, anchor=center, inner sep=0}, from=3-3, to=4-3]
	\arrow[""{name=5, anchor=center, inner sep=0}, from=3-4, to=4-4]
	\arrow[from=4-1, to=4-2]
	\arrow[""{name=6, anchor=center, inner sep=0}, from=4-1, to=5-1]
	\arrow[from=4-2, to=4-3]
	\arrow[""{name=7, anchor=center, inner sep=0}, from=4-2, to=5-2]
	\arrow[from=4-3, to=4-4]
	\arrow[""{name=8, anchor=center, inner sep=0}, from=4-3, to=5-3]
	\arrow[""{name=9, anchor=center, inner sep=0}, from=4-4, to=5-4]
	\arrow[from=5-1, to=5-2]
	\arrow[from=5-2, to=5-3]
	\arrow[from=5-3, to=5-4]
	\arrow["1"{description}, draw=none, from=1, to=2]
	\arrow["2"{description}, draw=none, from=2-3, to=0]
	\arrow["\bullet"{description}, draw=none, from=3, to=4]
	\arrow["\bullet"{description}, draw=none, from=4, to=5]
	\arrow["\bullet"{description}, draw=none, from=6, to=7]
	\arrow["\bullet"{description}, draw=none, from=7, to=8]
	\arrow["\bullet"{description}, draw=none, from=8, to=9]
\end{tikzcd}} \quad \text{and} \quad
R_2=\hspace{-2mm}
\adjustbox{scale=0.60}{\begin{tikzcd}
	&& {} & {} \\
	& {} & {} & {} \\
	& {} & {} & {} \\
	{} & {} & {} & {} \\
	{} & {} & {} & {}
	\arrow[from=1-3, to=1-4]
	\arrow[""{name=0, anchor=center, inner sep=0}, from=1-3, to=2-3]
	\arrow[""{name=1, anchor=center, inner sep=0}, from=1-4, to=2-4]
	\arrow[from=2-2, to=2-3]
	\arrow[""{name=2, anchor=center, inner sep=0}, equals, from=2-2, to=3-2]
	\arrow[from=2-3, to=2-4]
	\arrow[""{name=3, anchor=center, inner sep=0}, from=2-3, to=3-3]
	\arrow[""{name=4, anchor=center, inner sep=0}, from=2-4, to=3-4]
	\arrow[from=3-2, to=3-3]
	\arrow[""{name=5, anchor=center, inner sep=0}, from=3-2, to=4-2]
	\arrow[from=3-3, to=3-4]
	\arrow[""{name=6, anchor=center, inner sep=0}, from=3-3, to=4-3]
	\arrow[""{name=7, anchor=center, inner sep=0}, from=3-4, to=4-4]
	\arrow[from=4-1, to=4-2]
	\arrow[""{name=8, anchor=center, inner sep=0}, from=4-1, to=5-1]
	\arrow[from=4-2, to=4-3]
	\arrow[""{name=9, anchor=center, inner sep=0}, from=4-2, to=5-2]
	\arrow[from=4-3, to=4-4]
	\arrow[""{name=10, anchor=center, inner sep=0}, from=4-3, to=5-3]
	\arrow[""{name=11, anchor=center, inner sep=0}, from=4-4, to=5-4]
	\arrow[from=5-1, to=5-2]
	\arrow[from=5-2, to=5-3]
	\arrow[from=5-3, to=5-4]
	\arrow["\bullet"{description}, draw=none, from=0, to=1]
	\arrow["1"{description}, draw=none, from=2, to=3]
	\arrow["\bullet"{description}, draw=none, from=3, to=4]
	\arrow["\bullet"{description}, draw=none, from=5, to=6]
	\arrow["\bullet"{description}, draw=none, from=6, to=7]
	\arrow["\bullet"{description}, draw=none, from=8, to=9]
	\arrow["\bullet"{description}, draw=none, from=9, to=10]
	\arrow["\bullet"{description}, draw=none, from=10, to=11]
\end{tikzcd}}\; .
\]

7) $\mathbf{db}$ and $\mathbf{dc}$. Put $S_1=\adjustbox{scale=0.60}{\begin{tikzcd}
	&& {} & {} \\
	& {} & {} & {} \\
	& {} & {} & {} \\
	{} & {} && {} \\
	{} &&& {}
	\arrow[from=1-3, to=1-4]
	\arrow[""{name=0, anchor=center, inner sep=0}, from=1-3, to=2-3]
	\arrow[""{name=1, anchor=center, inner sep=0}, from=1-4, to=2-4]
	\arrow[from=2-2, to=2-3]
	\arrow[""{name=2, anchor=center, inner sep=0}, equals, from=2-2, to=3-2]
	\arrow[from=2-3, to=2-4]
	\arrow[""{name=3, anchor=center, inner sep=0}, from=2-3, to=3-3]
	\arrow[""{name=4, anchor=center, inner sep=0}, from=2-4, to=3-4]
	\arrow[from=3-2, to=3-3]
	\arrow[""{name=5, anchor=center, inner sep=0}, from=3-2, to=4-2]
	\arrow[from=3-3, to=3-4]
	\arrow[""{name=6, anchor=center, inner sep=0}, from=3-4, to=4-4]
	\arrow[from=4-1, to=4-2]
	\arrow[""{name=7, anchor=center, inner sep=0}, from=4-1, to=5-1]
	\arrow[from=4-2, to=4-4]
	\arrow[""{name=8, anchor=center, inner sep=0}, from=4-4, to=5-4]
	\arrow[from=5-1, to=5-4]
	\arrow["1"{description}, draw=none, from=0, to=1]
	\arrow["2"{description}, draw=none, from=2, to=3]
	\arrow["3"{description}, draw=none, from=3, to=4]
	\arrow["4"{description}, draw=none, from=5, to=6]
	\arrow["5"{description}, draw=none, from=7, to=8]
\end{tikzcd}}$.
We can construct a diagram of the form
\[\begin{tikzcd}
	{S_1} & {S_2} & {S_3} && \Can \\
	&& {S_4} & {S_5}
	\arrow["\bd", squiggly, from=1-1, to=1-2]
	\arrow["\bs", squiggly, from=1-2, to=1-3]
	\arrow[""{name=0, anchor=center, inner sep=0}, "{{{\bs}_1}}"', squiggly, from=1-2, to=2-3]
	\arrow["\bu", squiggly, from=1-3, to=1-5]
	\arrow[""{name=1, anchor=center, inner sep=0}, "\bs"', squiggly, from=2-3, to=1-3]
	\arrow["\bu"', squiggly, from=2-3, to=2-4]
	\arrow[""{name=2, anchor=center, inner sep=0}, "\bs"', squiggly, from=2-4, to=1-5]
	\arrow["\equiv"{description}, draw=none, from=0, to=1-3]
	\arrow["\equiv"{description}, draw=none, from=1, to=2]
\end{tikzcd}\]
showing that the two canonical $\Sg$-paths are equivalent.

 8) $\mathbf{db}$ and $\mathbf{ua}$. Put $S_1=\adjustbox{scale=0.60}{\begin{tikzcd}
	&& {} & {} \\
	& {} & {} & {} \\
	{} & {} & {} & {} \\
	{} & {} & {} & {}
	\arrow[from=1-3, to=1-4]
	\arrow[""{name=0, anchor=center, inner sep=0}, from=1-3, to=2-3]
	\arrow[""{name=1, anchor=center, inner sep=0}, from=1-4, to=2-4]
	\arrow[from=2-2, to=2-3]
	\arrow[""{name=2, anchor=center, inner sep=0}, from=2-2, to=3-2]
	\arrow[from=2-3, to=2-4]
	\arrow[""{name=3, anchor=center, inner sep=0}, from=2-3, to=3-3]
	\arrow[""{name=4, anchor=center, inner sep=0}, from=2-4, to=3-4]
	\arrow[from=3-1, to=3-2]
	\arrow[""{name=5, anchor=center, inner sep=0}, from=3-1, to=4-1]
	\arrow[from=3-2, to=3-3]
	\arrow[""{name=6, anchor=center, inner sep=0}, from=3-2, to=4-2]
	\arrow[from=3-3, to=3-4]
	\arrow[""{name=7, anchor=center, inner sep=0}, from=3-3, to=4-3]
	\arrow[""{name=8, anchor=center, inner sep=0}, from=3-4, to=4-4]
	\arrow[from=4-1, to=4-2]
	\arrow[from=4-2, to=4-3]
	\arrow[from=4-3, to=4-4]
	\arrow["1"{description}, draw=none, from=0, to=1]
	\arrow["2"{description}, draw=none, from=2, to=3]
	\arrow["3"{description}, draw=none, from=3, to=4]
	\arrow["4"{description}, draw=none, from=5, to=6]
	\arrow["5"{description}, draw=none, from=6, to=7]
	\arrow["6"{description}, draw=none, from=7, to=8]
\end{tikzcd}}$,
\(R_1=\adjustbox{scale=0.60}{\begin{tikzcd}
	&& {} & {} \\
	& {} & {} & {} \\
	{} & {} & {} & {} \\
	{} & {} & {} & {}
	\arrow[from=1-3, to=1-4]
	\arrow[""{name=0, anchor=center, inner sep=0}, from=1-3, to=2-3]
	\arrow[""{name=1, anchor=center, inner sep=0}, from=1-4, to=2-4]
	\arrow[from=2-2, to=2-3]
	\arrow[""{name=2, anchor=center, inner sep=0}, from=2-2, to=3-2]
	\arrow[from=2-3, to=2-4]
	\arrow[""{name=3, anchor=center, inner sep=0}, from=2-3, to=3-3]
	\arrow[""{name=4, anchor=center, inner sep=0}, from=2-4, to=3-4]
	\arrow[from=3-1, to=3-2]
	\arrow[""{name=5, anchor=center, inner sep=0}, from=3-1, to=4-1]
	\arrow[from=3-2, to=3-3]
	\arrow[""{name=6, anchor=center, inner sep=0}, from=3-2, to=4-2]
	\arrow[from=3-3, to=3-4]
	\arrow[""{name=7, anchor=center, inner sep=0}, from=3-3, to=4-3]
	\arrow[""{name=8, anchor=center, inner sep=0}, from=3-4, to=4-4]
	\arrow[from=4-1, to=4-2]
	\arrow[from=4-2, to=4-3]
	\arrow[from=4-3, to=4-4]
	\arrow["1"{description}, draw=none, from=0, to=1]
	\arrow["2"{description}, draw=none, from=2, to=3]
	\arrow["3"{description}, draw=none, from=3, to=4]
	\arrow["4"{description}, draw=none, from=5, to=6]
	\arrow["5"{description}, draw=none, from=6, to=7]
	\arrow["\bullet"{description}, draw=none, from=7, to=8]
\end{tikzcd}}
$
and
\(R_2= \adjustbox{scale=0.60}{\begin{tikzcd}
	&& {} & {} \\
	& {} & {} & {} \\
	{} & {} & {} & {} \\
	{} & {} & {} & {}
	\arrow[from=1-3, to=1-4]
	\arrow[""{name=0, anchor=center, inner sep=0}, from=1-3, to=2-3]
	\arrow[""{name=1, anchor=center, inner sep=0}, from=1-4, to=2-4]
	\arrow[from=2-2, to=2-3]
	\arrow[""{name=2, anchor=center, inner sep=0}, from=2-2, to=3-2]
	\arrow[from=2-3, to=2-4]
	\arrow[""{name=3, anchor=center, inner sep=0}, from=2-3, to=3-3]
	\arrow[""{name=4, anchor=center, inner sep=0}, from=2-4, to=3-4]
	\arrow[from=3-1, to=3-2]
	\arrow[""{name=5, anchor=center, inner sep=0}, from=3-1, to=4-1]
	\arrow[from=3-2, to=3-3]
	\arrow[""{name=6, anchor=center, inner sep=0}, from=3-2, to=4-2]
	\arrow[from=3-3, to=3-4]
	\arrow[""{name=7, anchor=center, inner sep=0}, from=3-3, to=4-3]
	\arrow[""{name=8, anchor=center, inner sep=0}, from=3-4, to=4-4]
	\arrow[from=4-1, to=4-2]
	\arrow[from=4-2, to=4-3]
	\arrow[from=4-3, to=4-4]
	\arrow["1"{description}, draw=none, from=0, to=1]
	\arrow["\bullet"{description}, draw=none, from=2, to=3]
	\arrow["\bullet"{description}, draw=none, from=3, to=4]
	\arrow["4"{description}, draw=none, from=5, to=6]
	\arrow["\bullet"{description}, draw=none, from=6, to=7]
	\arrow["\bullet"{description}, draw=none, from=7, to=8]
\end{tikzcd}}\).\newline
 Then use the diagram
\[\begin{tikzcd}
	& {S_2} & {S_3} \\
	{S_1} & {R_1} & {R_2} & \Can \\
	& {S_4} & {S_5}
	\arrow["\bs", squiggly, from=1-2, to=1-3]
	\arrow[""{name=0, anchor=center, inner sep=0}, "\bd", squiggly, from=1-2, to=2-2]
	\arrow[""{name=1, anchor=center, inner sep=0}, "\bd", squiggly, from=1-3, to=2-3]
	\arrow["\bu", squiggly, from=1-3, to=2-4]
	\arrow["\bd", squiggly, from=2-1, to=1-2]
	\arrow["\equiv"{description}, draw=none, from=2-1, to=2-2]
	\arrow["\bu"', squiggly, from=2-1, to=3-2]
	\arrow["\bs"', squiggly, from=2-2, to=2-3]
	\arrow["\equiv"{description}, draw=none, from=2-3, to=2-4]
	\arrow[""{name=2, anchor=center, inner sep=0}, "\bu"', squiggly, from=3-2, to=2-2]
	\arrow["\bs"', squiggly, from=3-2, to=3-3]
	\arrow[""{name=3, anchor=center, inner sep=0}, "\bu"', squiggly, from=3-3, to=2-3]
	\arrow["\bd"', squiggly, from=3-3, to=2-4]
	\arrow["\equiv"{description}, draw=none, from=0, to=1]
	\arrow["\equiv"{description}, draw=none, from=2, to=3]
\end{tikzcd}\,.\]

  9) $\mathbf{db}$ and $\mathbf{ub}$. Put $S_1=\hspace{-1.5mm}\adjustbox{scale=0.60}{\begin{tikzcd}
	&& {} & {} \\
	& {} & {} & {} \\
	{} & {} & {} & {} \\
	{} & {} & {} & {}
	\arrow[from=1-3, to=1-4]
	\arrow[""{name=0, anchor=center, inner sep=0}, from=1-3, to=2-3]
	\arrow[""{name=1, anchor=center, inner sep=0}, from=1-4, to=2-4]
	\arrow[from=2-2, to=2-3]
	\arrow[""{name=2, anchor=center, inner sep=0}, from=2-2, to=3-2]
	\arrow[from=2-3, to=2-4]
	\arrow[""{name=3, anchor=center, inner sep=0}, from=2-3, to=3-3]
	\arrow[""{name=4, anchor=center, inner sep=0}, from=2-4, to=3-4]
	\arrow[from=3-1, to=3-2]
	\arrow[""{name=5, anchor=center, inner sep=0}, from=3-1, to=4-1]
	\arrow[from=3-2, to=3-3]
	\arrow[from=3-3, to=3-4]
	\arrow[""{name=6, anchor=center, inner sep=0}, from=3-3, to=4-3]
	\arrow[""{name=7, anchor=center, inner sep=0}, from=3-4, to=4-4]
	\arrow[from=4-1, to=4-2]
	\arrow[from=4-2, to=4-3]
	\arrow[from=4-3, to=4-4]
	\arrow["1"{description}, draw=none, from=0, to=1]
	\arrow["2"{description}, draw=none, from=2, to=3]
	\arrow["3"{description}, draw=none, from=3, to=4]
	\arrow["4"{description}, draw=none, from=5, to=6]
	\arrow["5"{description}, draw=none, from=6, to=7]
\end{tikzcd}}\; $
and use the diagram
\(
\begin{tikzcd}
	{S_1} & {S_2} & {S_3} & \Can \\
	{S_4} & R & {S_5}
	\arrow["\bd", squiggly, from=1-1, to=1-2]
	\arrow["\bu"', squiggly, from=1-1, to=2-1]
	\arrow[""{name=0, anchor=center, inner sep=0}, "\bd"', squiggly, from=1-1, to=2-2]
	\arrow["\bs", squiggly, from=1-2, to=1-3]
	\arrow["\bu", squiggly, from=1-3, to=1-4]
	\arrow["\equiv"{description}, draw=none, from=2-1, to=2-2]
	\arrow["\bd"', curve={height=18pt}, squiggly, from=2-1, to=2-3]
	\arrow[""{name=1, anchor=center, inner sep=0}, "\bd"{description}, squiggly, from=2-2, to=1-2]
	\arrow[""{name=2, anchor=center, inner sep=0}, "\bs"{description}, squiggly, from=2-2, to=1-3]
	\arrow["\bu"', squiggly, from=2-2, to=2-3]
	\arrow[""{name=3, anchor=center, inner sep=0}, "\bs", squiggly, from=2-3, to=1-4]
	\arrow["\equiv"{description}, shift left=3, draw=none, from=0, to=1]
	\arrow["\equiv"{description}, shift left=3, draw=none, from=1, to=2]
	\arrow["\equiv"{description}, draw=none, from=2, to=3]
\end{tikzcd}
\),
where $R$ is defined in the obvious way, to conclude that the two canonical $\Sg$-steps are equivalent.

 10) $\mathbf{db}$ and $\mathbf{s}$. Put $S_1=\adjustbox{scale=0.60}{\begin{tikzcd}
	&& {} & {} \\
	& {} & {} & {} \\
	{} & {} && {} \\
	{} & {} && {}
	\arrow[from=1-3, to=1-4]
	\arrow[""{name=0, anchor=center, inner sep=0}, from=1-3, to=2-3]
	\arrow[""{name=1, anchor=center, inner sep=0}, from=1-4, to=2-4]
	\arrow[from=2-2, to=2-3]
	\arrow[""{name=2, anchor=center, inner sep=0}, from=2-2, to=3-2]
	\arrow[from=2-3, to=2-4]
	\arrow[""{name=3, anchor=center, inner sep=0}, from=2-4, to=3-4]
	\arrow[from=3-1, to=3-2]
	\arrow[""{name=4, anchor=center, inner sep=0}, from=3-1, to=4-1]
	\arrow[from=3-2, to=3-4]
	\arrow[""{name=5, anchor=center, inner sep=0}, from=3-2, to=4-2]
	\arrow[""{name=6, anchor=center, inner sep=0}, from=3-4, to=4-4]
	\arrow[from=4-1, to=4-2]
	\arrow[from=4-2, to=4-4]
	\arrow["1"{description}, draw=none, from=0, to=1]
	\arrow["2"{description}, draw=none, from=2, to=3]
	\arrow["3"{description}, draw=none, from=4, to=5]
	\arrow["4"{description}, draw=none, from=5, to=6]
\end{tikzcd}}$
and use the diagram
\[\begin{tikzcd}
	{S_1} & {S_2} & {S_3} & \Can \\
	& {S_4}
	\arrow["\bd", squiggly, from=1-1, to=1-2]
	\arrow[""{name=0, anchor=center, inner sep=0}, "\bs"', squiggly, from=1-1, to=2-2]
	\arrow["\bs", squiggly, from=1-2, to=1-3]
	\arrow["\bu", squiggly, from=1-3, to=1-4]
	\arrow[""{name=1, anchor=center, inner sep=0}, "\bd"', squiggly, from=2-2, to=1-3]
	\arrow["\equiv"{description}, draw=none, from=0, to=1]
\end{tikzcd}\]
to conclude the equivalence of the two canonical $\Sg$-paths.

   11) $\mathbf{db}$ and $\mathbf{s}_1$. Put $S_1=\hspace{-1.5mm}\adjustbox{scale=0.60}{\begin{tikzcd}
	&& {} & {} \\
	& {} & {} & {} \\
	& {} & {} & {} \\
	{} & {} && {} \\
	{} & {} && {}
	\arrow[from=1-3, to=1-4]
	\arrow[""{name=0, anchor=center, inner sep=0}, from=1-3, to=2-3]
	\arrow[""{name=1, anchor=center, inner sep=0}, from=1-4, to=2-4]
	\arrow[from=2-2, to=2-3]
	\arrow[""{name=2, anchor=center, inner sep=0}, equals, from=2-2, to=3-2]
	\arrow[from=2-3, to=2-4]
	\arrow[""{name=3, anchor=center, inner sep=0}, from=2-3, to=3-3]
	\arrow[""{name=4, anchor=center, inner sep=0}, from=2-4, to=3-4]
	\arrow[from=3-2, to=3-3]
	\arrow[""{name=5, anchor=center, inner sep=0}, from=3-2, to=4-2]
	\arrow[from=3-3, to=3-4]
	\arrow[""{name=6, anchor=center, inner sep=0}, from=3-4, to=4-4]
	\arrow[from=4-1, to=4-2]
	\arrow[""{name=7, anchor=center, inner sep=0}, from=4-1, to=5-1]
	\arrow[from=4-2, to=4-4]
	\arrow[""{name=8, anchor=center, inner sep=0}, from=4-2, to=5-2]
	\arrow[""{name=9, anchor=center, inner sep=0}, from=4-4, to=5-4]
	\arrow[from=5-1, to=5-2]
	\arrow[from=5-2, to=5-4]
	\arrow["1"{description}, draw=none, from=0, to=1]
	\arrow["2"{description}, draw=none, from=2, to=3]
	\arrow["3"{description}, draw=none, from=3, to=4]
	\arrow["4"{description}, draw=none, from=5, to=6]
	\arrow["5"{description}, draw=none, from=7, to=8]
	\arrow["6"{description}, draw=none, from=8, to=9]
\end{tikzcd}}$
and consider
$R= \hspace{-1.5mm}\adjustbox{scale=0.60}{\begin{tikzcd}
	&& {} & {} \\
	& {} & {} & {} \\
	& {} & {} & {} \\
	{} & {} & {} & {} \\
	{} & {} & {} & {}
	\arrow[from=1-3, to=1-4]
	\arrow[""{name=0, anchor=center, inner sep=0}, from=1-3, to=2-3]
	\arrow[""{name=1, anchor=center, inner sep=0}, from=1-4, to=2-4]
	\arrow[from=2-2, to=2-3]
	\arrow[""{name=2, anchor=center, inner sep=0}, equals, from=2-2, to=3-2]
	\arrow[from=2-3, to=2-4]
	\arrow[""{name=3, anchor=center, inner sep=0}, from=2-3, to=3-3]
	\arrow[""{name=4, anchor=center, inner sep=0}, from=2-4, to=3-4]
	\arrow[from=3-2, to=3-3]
	\arrow[""{name=5, anchor=center, inner sep=0}, from=3-2, to=4-2]
	\arrow[from=3-3, to=3-4]
	\arrow[""{name=6, anchor=center, inner sep=0}, from=3-3, to=4-3]
	\arrow[""{name=7, anchor=center, inner sep=0}, from=3-4, to=4-4]
	\arrow[from=4-1, to=4-2]
	\arrow[""{name=8, anchor=center, inner sep=0}, from=4-1, to=5-1]
	\arrow[from=4-2, to=4-3]
	\arrow[""{name=9, anchor=center, inner sep=0}, from=4-2, to=5-2]
	\arrow[from=4-3, to=4-4]
	\arrow[""{name=10, anchor=center, inner sep=0}, from=4-3, to=5-3]
	\arrow[""{name=11, anchor=center, inner sep=0}, from=4-4, to=5-4]
	\arrow[from=5-1, to=5-2]
	\arrow[from=5-2, to=5-3]
	\arrow[from=5-3, to=5-4]
	\arrow["1"{description}, draw=none, from=0, to=1]
	\arrow["2"{description}, draw=none, from=2, to=3]
	\arrow["3"{description}, draw=none, from=3, to=4]
	\arrow["\bullet"{description}, draw=none, from=5, to=6]
	\arrow["\bullet"{description}, draw=none, from=6, to=7]
	\arrow["\bullet"{description}, draw=none, from=8, to=9]
	\arrow["\bullet"{description}, draw=none, from=9, to=10]
	\arrow["\bullet"{description}, draw=none, from=10, to=11]
\end{tikzcd}}$.
Use the following diagram
\[\begin{tikzcd}
	& {S_2} && {S_3} \\
	{S_1} & {R} &&& \Can\, , \\
	& {S_4} & {S_5} & {S_6}
	\arrow["\bs", squiggly, from=1-2, to=1-4]
	\arrow[""{name=0, anchor=center, inner sep=0}, "{{\bs}_1}"{pos=0.4}, squiggly, from=1-2, to=2-2]
	\arrow["\bu", squiggly, from=1-4, to=2-5]
	\arrow["\bd", squiggly, from=2-1, to=1-2]
	\arrow["\equiv"{description}, draw=none, from=2-1, to=2-2]
	\arrow["{{\bs}_1}"', squiggly, from=2-1, to=3-2]
	\arrow[""{name=1, anchor=center, inner sep=0}, "\bs"{description}, squiggly, from=2-2, to=1-4]
	\arrow["{(\ast)}"{description}, draw=none, from=2-2, to=2-5]
	\arrow["\bd"', squiggly, from=3-2, to=2-2]
	\arrow["\bu"', squiggly, from=3-2, to=3-3]
	\arrow["\bs"', squiggly, from=3-3, to=3-4]
	\arrow["\bd"', squiggly, from=3-4, to=2-5]
	\arrow["\equiv", draw=none, from=0, to=1]
\end{tikzcd}\]
where the equivalence in $(\ast)$ is obtained as for the case $\mathbf{db}$ and $\mathbf{ua}$.

  12) $\mathbf{dc}$ and $\mathbf{ua}$. Put $S_1=\hspace{-2mm}\adjustbox{scale=0.60}{\begin{tikzcd}
	&& {} & {} \\
	& {} & {} & {} \\
	& {} & {} & {} \\
	{} & {} & {} & {} \\
	{} & {} & {} & {}
	\arrow[from=1-3, to=1-4]
	\arrow[from=1-3, to=2-3]
	\arrow[from=1-4, to=3-4]
	\arrow[from=2-2, to=2-3]
	\arrow[""{name=0, anchor=center, inner sep=0}, equals, from=2-2, to=3-2]
	\arrow["2"{description}, draw=none, from=2-3, to=2-4]
	\arrow[""{name=1, anchor=center, inner sep=0}, from=2-3, to=3-3]
	\arrow[from=3-2, to=3-3]
	\arrow[""{name=2, anchor=center, inner sep=0}, from=3-2, to=4-2]
	\arrow[from=3-3, to=3-4]
	\arrow[""{name=3, anchor=center, inner sep=0}, from=3-3, to=4-3]
	\arrow[""{name=4, anchor=center, inner sep=0}, from=3-4, to=4-4]
	\arrow[from=4-1, to=4-2]
	\arrow[""{name=5, anchor=center, inner sep=0}, from=4-1, to=5-1]
	\arrow[from=4-2, to=4-3]
	\arrow[""{name=6, anchor=center, inner sep=0}, from=4-2, to=5-2]
	\arrow[from=4-3, to=4-4]
	\arrow[""{name=7, anchor=center, inner sep=0}, from=4-3, to=5-3]
	\arrow[""{name=8, anchor=center, inner sep=0}, from=4-4, to=5-4]
	\arrow[from=5-1, to=5-2]
	\arrow[from=5-2, to=5-3]
	\arrow[from=5-3, to=5-4]
	\arrow["1"{description}, draw=none, from=0, to=1]
	\arrow["3"{description}, draw=none, from=2, to=3]
	\arrow["4"{description}, draw=none, from=3, to=4]
	\arrow["5"{description}, draw=none, from=5, to=6]
	\arrow["6"{description}, draw=none, from=6, to=7]
	\arrow["7"{description}, draw=none, from=7, to=8]
\end{tikzcd}}$,
$R_1= \hspace{-2mm}\adjustbox{scale=0.60}{\begin{tikzcd}
	&& {} & {} \\
	& {} & {} & {} \\
	& {} & {} & {} \\
	{} & {} & {} & {} \\
	{} & {} & {} & {}
	\arrow[from=1-3, to=1-4]
	\arrow[from=1-3, to=2-3]
	\arrow[from=1-4, to=3-4]
	\arrow[from=2-2, to=2-3]
	\arrow[""{name=0, anchor=center, inner sep=0}, equals, from=2-2, to=3-2]
	\arrow["2"{description}, draw=none, from=2-3, to=2-4]
	\arrow[""{name=1, anchor=center, inner sep=0}, from=2-3, to=3-3]
	\arrow[from=3-2, to=3-3]
	\arrow[""{name=2, anchor=center, inner sep=0}, from=3-2, to=4-2]
	\arrow[from=3-3, to=3-4]
	\arrow[""{name=3, anchor=center, inner sep=0}, from=3-3, to=4-3]
	\arrow[""{name=4, anchor=center, inner sep=0}, from=3-4, to=4-4]
	\arrow[from=4-1, to=4-2]
	\arrow[""{name=5, anchor=center, inner sep=0}, from=4-1, to=5-1]
	\arrow[from=4-2, to=4-3]
	\arrow[""{name=6, anchor=center, inner sep=0}, from=4-2, to=5-2]
	\arrow[from=4-3, to=4-4]
	\arrow[""{name=7, anchor=center, inner sep=0}, from=4-3, to=5-3]
	\arrow[""{name=8, anchor=center, inner sep=0}, from=4-4, to=5-4]
	\arrow[from=5-1, to=5-2]
	\arrow[from=5-2, to=5-3]
	\arrow[from=5-3, to=5-4]
	\arrow["1"{description}, draw=none, from=0, to=1]
	\arrow["\bullet"{description}, draw=none, from=2, to=3]
	\arrow["\bullet"{description}, draw=none, from=3, to=4]
	\arrow["5"{description}, draw=none, from=5, to=6]
	\arrow["\bullet"{description}, draw=none, from=6, to=7]
	\arrow["\bullet"{description}, draw=none, from=7, to=8]
\end{tikzcd}}\; $
\hspace{1mm} and \hspace{1mm}
$R_2=\hspace{-2mm}\adjustbox{scale=0.60}{\begin{tikzcd}
	&& {} & {} \\
	& {} & {} & {} \\
	& {} & {} & {} \\
	{} & {} & {} & {} \\
	{} & {} & {} & {}
	\arrow[from=1-3, to=1-4]
	\arrow[""{name=0, anchor=center, inner sep=0}, from=1-3, to=2-3]
	\arrow[""{name=1, anchor=center, inner sep=0}, from=1-4, to=2-4]
	\arrow[from=2-2, to=2-3]
	\arrow[""{name=2, anchor=center, inner sep=0}, equals, from=2-2, to=3-2]
	\arrow[from=2-3, to=2-4]
	\arrow[""{name=3, anchor=center, inner sep=0}, from=2-3, to=3-3]
	\arrow[""{name=4, anchor=center, inner sep=0}, from=2-4, to=3-4]
	\arrow[from=3-2, to=3-3]
	\arrow[""{name=5, anchor=center, inner sep=0}, from=3-2, to=4-2]
	\arrow[from=3-3, to=3-4]
	\arrow[""{name=6, anchor=center, inner sep=0}, from=3-3, to=4-3]
	\arrow[""{name=7, anchor=center, inner sep=0}, from=3-4, to=4-4]
	\arrow[from=4-1, to=4-2]
	\arrow[""{name=8, anchor=center, inner sep=0}, from=4-1, to=5-1]
	\arrow[from=4-2, to=4-3]
	\arrow[""{name=9, anchor=center, inner sep=0}, from=4-2, to=5-2]
	\arrow[from=4-3, to=4-4]
	\arrow[""{name=10, anchor=center, inner sep=0}, from=4-3, to=5-3]
	\arrow[""{name=11, anchor=center, inner sep=0}, from=4-4, to=5-4]
	\arrow[from=5-1, to=5-2]
	\arrow[from=5-2, to=5-3]
	\arrow[from=5-3, to=5-4]
	\arrow["\bullet"{description}, draw=none, from=0, to=1]
	\arrow["1"{description}, draw=none, from=2, to=3]
	\arrow["\bullet"{description}, draw=none, from=3, to=4]
	\arrow["\bullet"{description}, draw=none, from=5, to=6]
	\arrow["\bullet"{description}, draw=none, from=6, to=7]
	\arrow["5"{description}, draw=none, from=8, to=9]
	\arrow["\bullet"{description}, draw=none, from=9, to=10]
	\arrow["\bullet"{description}, draw=none, from=10, to=11]
\end{tikzcd}}$.\newline
Use the  diagram
\[\begin{tikzcd}
	& {S_2} && {S_3} && {S_4} \\
	{S_1} && {R_1} && {R_2} && \Can \\
	&& {S_5} &&& {S_6}
	\arrow["{{\bs}_1}", squiggly, from=1-2, to=1-4]
	\arrow["\bu", squiggly, from=1-4, to=1-6]
	\arrow["\bs", squiggly, from=1-6, to=2-7]
	\arrow[""{name=0, anchor=center, inner sep=0}, "\bd", squiggly, from=2-1, to=1-2]
	\arrow["{{\bs}_1}"', squiggly, from=2-1, to=2-3]
	\arrow[""{name=1, anchor=center, inner sep=0}, "\bu"', squiggly, from=2-1, to=3-3]
	\arrow[""{name=2, anchor=center, inner sep=0}, "\bd"', squiggly, from=2-3, to=1-4]
	\arrow[""{name=3, anchor=center, inner sep=0}, "\bu"', squiggly, from=2-3, to=2-5]
	\arrow[""{name=4, anchor=center, inner sep=0}, "\bd"', squiggly, from=2-5, to=1-6]
	\arrow["\equiv"{description}, draw=none, from=2-5, to=2-7]
	\arrow[""{name=5, anchor=center, inner sep=0}, "\bs"', squiggly, from=2-5, to=3-6]
	\arrow[""{name=6, anchor=center, inner sep=0}, "\bs"', squiggly, from=3-3, to=2-5]
	\arrow["\bs"', squiggly, from=3-3, to=3-6]
	\arrow["\bd"', squiggly, from=3-6, to=2-7]
	\arrow["\equiv"{description}, draw=none, from=0, to=2]
	\arrow["\equiv"', draw=none, from=1, to=3]
	\arrow["\equiv"{description}, draw=none, from=2, to=4]
	\arrow["\equiv"', draw=none, from=6, to=5]
\end{tikzcd}\]
to conclude the equivalence of the two canonical $\Sg$-paths.

13)   $\mathbf{dc}$ and $\mathbf{ub}$. Put $S_1=\hspace{-2mm}\adjustbox{scale=0.60}{\begin{tikzcd}
	&& {} & {} \\
	& {} & {} & {} \\
	& {} & {} & {} \\
	{} & {} & {} & {} \\
	{} && {} & {}
	\arrow[from=1-3, to=1-4]
	\arrow[from=1-3, to=2-3]
	\arrow[from=1-4, to=3-4]
	\arrow[from=2-2, to=2-3]
	\arrow[""{name=0, anchor=center, inner sep=0}, equals, from=2-2, to=3-2]
	\arrow["2"{description}, draw=none, from=2-3, to=2-4]
	\arrow[""{name=1, anchor=center, inner sep=0}, from=2-3, to=3-3]
	\arrow[from=3-2, to=3-3]
	\arrow[""{name=2, anchor=center, inner sep=0}, from=3-2, to=4-2]
	\arrow[from=3-3, to=3-4]
	\arrow[""{name=3, anchor=center, inner sep=0}, from=3-3, to=4-3]
	\arrow[""{name=4, anchor=center, inner sep=0}, from=3-4, to=4-4]
	\arrow[from=4-1, to=4-2]
	\arrow[""{name=5, anchor=center, inner sep=0}, from=4-1, to=5-1]
	\arrow[from=4-2, to=4-3]
	\arrow[from=4-3, to=4-4]
	\arrow[""{name=6, anchor=center, inner sep=0}, from=4-3, to=5-3]
	\arrow[""{name=7, anchor=center, inner sep=0}, from=4-4, to=5-4]
	\arrow[from=5-1, to=5-3]
	\arrow[from=5-3, to=5-4]
	\arrow["1"{description}, draw=none, from=0, to=1]
	\arrow["3"{description}, draw=none, from=2, to=3]
	\arrow["4"{description}, draw=none, from=3, to=4]
	\arrow["5"{description}, draw=none, from=5, to=6]
	\arrow["6"{description}, draw=none, from=6, to=7]
\end{tikzcd}}.\;$
We find that the two canonical $\Sg$-paths  are equivalent  by the diagram
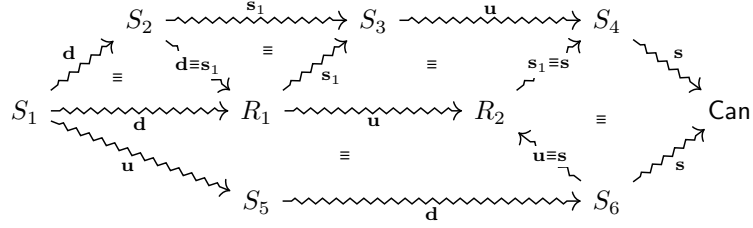
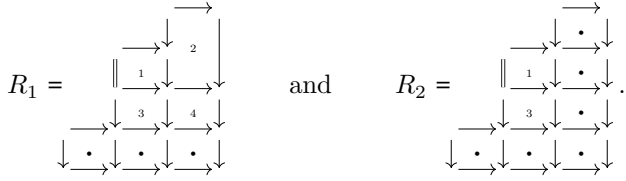
\[\begin{tikzcd}
	& {S_2} && {S_3} && {S_4} \\
	{S_1} && {R_1} && {R_2} && \Can \\
	&& {S_5} &&& {S_6}
	\arrow["{{{{{{\bs}_1}}}}}", squiggly, from=1-2, to=1-4]
	\arrow[""{name=0, anchor=center, inner sep=0}, "{{{\bd\equiv {\bs}_1}}}"{description}, squiggly, from=1-2, to=2-3]
	\arrow["\bu", squiggly, from=1-4, to=1-6]
	\arrow["\bs", squiggly, from=1-6, to=2-7]
	\arrow[""{name=1, anchor=center, inner sep=0}, "\bd", squiggly, from=2-1, to=1-2]
	\arrow["\bd"', squiggly, from=2-1, to=2-3]
	\arrow[""{name=2, anchor=center, inner sep=0}, "\bu"', squiggly, from=2-1, to=3-3]
	\arrow[""{name=3, anchor=center, inner sep=0}, "{{{{\bs}_1}}}"', squiggly, from=2-3, to=1-4]
	\arrow["\bu"', squiggly, from=2-3, to=2-5]
	\arrow[""{name=4, anchor=center, inner sep=0}, "{{{\bs}_1\equiv \bs}}"{description}, squiggly, from=2-5, to=1-6]
	\arrow["\bd"', squiggly, from=3-3, to=3-6]
	\arrow[""{name=5, anchor=center, inner sep=0}, "{{{\bu\equiv \bs}}}"{description}, squiggly, from=3-6, to=2-5]
	\arrow[""{name=6, anchor=center, inner sep=0}, "\bs"', squiggly, from=3-6, to=2-7]
	\arrow["\equiv"'{pos=0.6}, shift left=4, draw=none, from=0, to=3]
	\arrow["\equiv"{description, pos=0.3}, shift right=2, draw=none, from=1, to=0]
	\arrow["\equiv"{description}, draw=none, from=2, to=5]
	\arrow["\equiv"{description}, draw=none, from=3, to=4]
	\arrow["\equiv"{description}, shift right=2, draw=none, from=4, to=6]
\end{tikzcd}\]
where

\(R_1=\hspace{-2mm}\adjustbox{scale=0.60}{\begin{tikzcd}
	&& {} & {} \\
	& {} & {} & {} \\
	& {} & {} & {} \\
	{} & {} & {} & {} \\
	{} & {} & {} & {}
	\arrow[from=1-3, to=1-4]
	\arrow[from=1-3, to=2-3]
	\arrow[from=1-4, to=3-4]
	\arrow[from=2-2, to=2-3]
	\arrow[""{name=0, anchor=center, inner sep=0}, equals, from=2-2, to=3-2]
	\arrow["2"{description}, draw=none, from=2-3, to=2-4]
	\arrow[""{name=1, anchor=center, inner sep=0}, from=2-3, to=3-3]
	\arrow[from=3-2, to=3-3]
	\arrow[""{name=2, anchor=center, inner sep=0}, from=3-2, to=4-2]
	\arrow[from=3-3, to=3-4]
	\arrow[""{name=3, anchor=center, inner sep=0}, from=3-3, to=4-3]
	\arrow[""{name=4, anchor=center, inner sep=0}, from=3-4, to=4-4]
	\arrow[from=4-1, to=4-2]
	\arrow[""{name=5, anchor=center, inner sep=0}, from=4-1, to=5-1]
	\arrow[from=4-2, to=4-3]
	\arrow[""{name=6, anchor=center, inner sep=0}, from=4-2, to=5-2]
	\arrow[from=4-3, to=4-4]
	\arrow[""{name=7, anchor=center, inner sep=0}, from=4-3, to=5-3]
	\arrow[""{name=8, anchor=center, inner sep=0}, from=4-4, to=5-4]
	\arrow[from=5-1, to=5-2]
	\arrow[from=5-2, to=5-3]
	\arrow[from=5-3, to=5-4]
	\arrow["1"{description}, draw=none, from=0, to=1]
	\arrow["3"{description}, draw=none, from=2, to=3]
	\arrow["4"{description}, draw=none, from=3, to=4]
	\arrow["\bullet"{description}, draw=none, from=5, to=6]
	\arrow["\bullet"{description}, draw=none, from=6, to=7]
	\arrow["\bullet"{description}, draw=none, from=7, to=8]
\end{tikzcd}}\) \hspace{6mm} and \hspace{6mm}
 \(R_2=
 \hspace{-2mm}\adjustbox{scale=0.60}{\begin{tikzcd}
	&& {} & {} \\
	& {} & {} & {} \\
	& {} & {} & {} \\
	{} & {} & {} & {} \\
	{} & {} & {} & {}
	\arrow[from=1-3, to=1-4]
	\arrow[""{name=0, anchor=center, inner sep=0}, from=1-3, to=2-3]
	\arrow[""{name=1, anchor=center, inner sep=0}, from=1-4, to=2-4]
	\arrow[from=2-2, to=2-3]
	\arrow[""{name=2, anchor=center, inner sep=0}, equals, from=2-2, to=3-2]
	\arrow[from=2-3, to=2-4]
	\arrow[""{name=3, anchor=center, inner sep=0}, from=2-3, to=3-3]
	\arrow[""{name=4, anchor=center, inner sep=0}, from=2-4, to=3-4]
	\arrow[from=3-2, to=3-3]
	\arrow[""{name=5, anchor=center, inner sep=0}, from=3-2, to=4-2]
	\arrow[from=3-3, to=3-4]
	\arrow[""{name=6, anchor=center, inner sep=0}, from=3-3, to=4-3]
	\arrow[""{name=7, anchor=center, inner sep=0}, from=3-4, to=4-4]
	\arrow[from=4-1, to=4-2]
	\arrow[""{name=8, anchor=center, inner sep=0}, from=4-1, to=5-1]
	\arrow[from=4-2, to=4-3]
	\arrow[""{name=9, anchor=center, inner sep=0}, from=4-2, to=5-2]
	\arrow[from=4-3, to=4-4]
	\arrow[""{name=10, anchor=center, inner sep=0}, from=4-3, to=5-3]
	\arrow[""{name=11, anchor=center, inner sep=0}, from=4-4, to=5-4]
	\arrow[from=5-1, to=5-2]
	\arrow[from=5-2, to=5-3]
	\arrow[from=5-3, to=5-4]
	\arrow["\bullet"{description}, draw=none, from=0, to=1]
	\arrow["1"{description}, draw=none, from=2, to=3]
	\arrow["\bullet"{description}, draw=none, from=3, to=4]
	\arrow["3"{description}, draw=none, from=5, to=6]
	\arrow["\bullet"{description}, draw=none, from=6, to=7]
	\arrow["\bullet"{description}, draw=none, from=8, to=9]
	\arrow["\bullet"{description}, draw=none, from=9, to=10]
	\arrow["\bullet"{description}, draw=none, from=10, to=11]
\end{tikzcd}}\, \).

14)  $\mathbf{dc}$ and $\mathbf{s}$. Put $S_1=\hspace{-2mm}\adjustbox{scale=0.60}{\begin{tikzcd}
	&& {} & {} \\
	& {} & {} & {} \\
	& {} & {} & {} \\
	{} & {} && {} \\
	{} & {} && {}
	\arrow[from=1-3, to=1-4]
	\arrow[""{name=0, anchor=center, inner sep=0}, from=1-3, to=2-3]
	\arrow[""{name=1, anchor=center, inner sep=0}, from=1-4, to=2-4]
	\arrow[from=2-2, to=2-3]
	\arrow[""{name=2, anchor=center, inner sep=0}, equals, from=2-2, to=3-2]
	\arrow[from=2-3, to=2-4]
	\arrow[""{name=3, anchor=center, inner sep=0}, from=2-3, to=3-3]
	\arrow[""{name=4, anchor=center, inner sep=0}, from=2-4, to=3-4]
	\arrow[from=3-2, to=3-3]
	\arrow[""{name=5, anchor=center, inner sep=0}, from=3-2, to=4-2]
	\arrow[from=3-3, to=3-4]
	\arrow[""{name=6, anchor=center, inner sep=0}, from=3-4, to=4-4]
	\arrow[from=4-1, to=4-2]
	\arrow[""{name=7, anchor=center, inner sep=0}, from=4-1, to=5-1]
	\arrow[from=4-2, to=4-4]
	\arrow[""{name=8, anchor=center, inner sep=0}, from=4-2, to=5-2]
	\arrow[""{name=9, anchor=center, inner sep=0}, from=4-4, to=5-4]
	\arrow[from=5-1, to=5-2]
	\arrow[from=5-2, to=5-4]
	\arrow["1"{description}, draw=none, from=0, to=1]
	\arrow["2"{description}, draw=none, from=2, to=3]
	\arrow["3"{description}, draw=none, from=3, to=4]
	\arrow["4"{description}, draw=none, from=5, to=6]
	\arrow["5"{description}, draw=none, from=7, to=8]
	\arrow["6"{description}, draw=none, from=8, to=9]
\end{tikzcd}}$.
We use the following diagram to show the two canonical $\Sg$-paths are  equivalent.
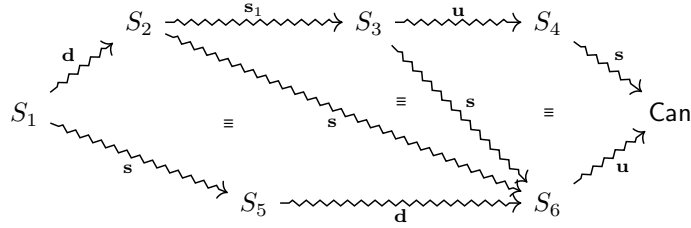
\[\begin{tikzcd}
	& {S_2} && {S_3} && {S_4} \\
	{S_1} &&&&&& \Can \\
	&& {S_5} &&& {S_6}
	\arrow["{{{{\bs}_1}}}", squiggly, from=1-2, to=1-4]
	\arrow[""{name=0, anchor=center, inner sep=0}, "\bs"', squiggly, from=1-2, to=3-6]
	\arrow["\bu", squiggly, from=1-4, to=1-6]
	\arrow[""{name=1, anchor=center, inner sep=0}, "\bs", squiggly, from=1-4, to=3-6]
	\arrow["\bs", squiggly, from=1-6, to=2-7]
	\arrow["\bd", squiggly, from=2-1, to=1-2]
	\arrow[""{name=2, anchor=center, inner sep=0}, "\bs"', squiggly, from=2-1, to=3-3]
	\arrow["\bd"', squiggly, from=3-3, to=3-6]
	\arrow["\bu"', squiggly, from=3-6, to=2-7]
	\arrow["\equiv", draw=none, from=0, to=1]
	\arrow["\equiv"{description}, draw=none, from=1, to=2-7]
	\arrow["\equiv", draw=none, from=2, to=0]
\end{tikzcd}\]

15) $\mathbf{dc}$ and $\mathbf{s}_1$. Put $S_1=\hspace{-2mm}\adjustbox{scale=0.60}{\begin{tikzcd}
	&& {} & {} \\
	& {} & {} & {} \\
	& {} & {} & {} \\
	{} & {} && {} \\
	{} & {} && {}
	\arrow[from=1-3, to=1-4]
	\arrow[from=1-3, to=2-3]
	\arrow[from=1-4, to=3-4]
	\arrow[from=2-2, to=2-3]
	\arrow[""{name=0, anchor=center, inner sep=0}, equals, from=2-2, to=3-2]
	\arrow["2"{description}, draw=none, from=2-3, to=2-4]
	\arrow[""{name=1, anchor=center, inner sep=0}, from=2-3, to=3-3]
	\arrow[from=3-2, to=3-3]
	\arrow[""{name=2, anchor=center, inner sep=0}, from=3-2, to=4-2]
	\arrow[from=3-3, to=3-4]
	\arrow[""{name=3, anchor=center, inner sep=0}, from=3-4, to=4-4]
	\arrow[from=4-1, to=4-2]
	\arrow[""{name=4, anchor=center, inner sep=0}, from=4-1, to=5-1]
	\arrow[from=4-2, to=4-4]
	\arrow[""{name=5, anchor=center, inner sep=0}, from=4-2, to=5-2]
	\arrow[""{name=6, anchor=center, inner sep=0}, from=4-4, to=5-4]
	\arrow[from=5-1, to=5-2]
	\arrow[from=5-2, to=5-4]
	\arrow["1"{description}, draw=none, from=0, to=1]
	\arrow["3"{description}, draw=none, from=2, to=3]
	\arrow["4"{description}, draw=none, from=4, to=5]
	\arrow["5"{description}, draw=none, from=5, to=6]
\end{tikzcd}}\, $.
The following diagram shows the two canonical $\Sg$-paths are  equivalent.
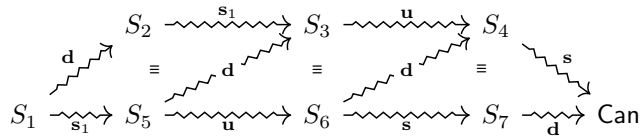
\[\begin{tikzcd}
	& {S_2} && {S_3} && {S_4} \\
	{S_1} & {S_5} && {S_6} && {S_7} & \Can
	\arrow["{{{{{\bs}_1}}}}", squiggly, from=1-2, to=1-4]
	\arrow["\bu", squiggly, from=1-4, to=1-6]
	\arrow[""{name=0, anchor=center, inner sep=0}, "\bs", squiggly, from=1-6, to=2-7]
	\arrow[""{name=1, anchor=center, inner sep=0}, "\bd", squiggly, from=2-1, to=1-2]
	\arrow["{{\bs}_1}"', squiggly, from=2-1, to=2-2]
	\arrow[""{name=2, anchor=center, inner sep=0}, "\bd"{description}, squiggly, from=2-2, to=1-4]
	\arrow["\bu"', squiggly, from=2-2, to=2-4]
	\arrow[""{name=3, anchor=center, inner sep=0}, "\bd"{description}, squiggly, from=2-4, to=1-6]
	\arrow["\bs"', squiggly, from=2-4, to=2-6]
	\arrow["\bd"', squiggly, from=2-6, to=2-7]
	\arrow["\equiv"{description}, draw=none, from=1, to=2]
	\arrow["\equiv"{description}, draw=none, from=2, to=3]
	\arrow["\equiv"{description}, draw=none, from=3, to=0]
\end{tikzcd}\]

The remaining six cases ($\mathbf{ua}$ \& $\mathbf{ub}$,  $\mathbf{ua}$ \& $\mathbf{s}$, $\mathbf{ua}$ \& $\mathbf{s}_1$, $\mathbf{ub}$ \& $\mathbf{s}$, $\mathbf{ub}$ \& $\mathbf{s}_1$ and $\mathbf{s}$ \& $\mathbf{s}_1$) are proved with similar easy arguments.

\noindent \textbf{B.} We can now prove the theorem.  Observe (in \cref{tab:paths}) that given a $\Sg$-step of interest of a certain type $k$, say
$$\xymatrix{S_1\ar@{~>}[r]^k&S_2},$$
then $S_1$ and $S_2$ have a common configuration $y$ where the name of $y$ starts with the character $k$. Moreover, the $y$-canonical $\Sg$-path from $S_1$ to $\Can$ factors through $\xymatrix{S_1\ar@{~>}[r]^k&S_2}$; more precisely,  we have
\begin{equation}\label{eq:factor}
\begin{tikzcd}
	{S_1} & R & \dots & \Can \\
	{S_2}
	\arrow["k", squiggly, from=1-1, to=1-2]
	\arrow["k"', squiggly, from=1-1, to=2-1]
	\arrow[squiggly, from=1-2, to=1-3]
	\arrow[squiggly, from=1-3, to=1-4]
	\arrow["k"', squiggly, from=2-1, to=1-2]
\end{tikzcd}
\end{equation}
where the top $\Sg$-path is the $y$-canonical $\Sg$-path from $S_1$ to $\Can$ and the bottom line is the $y$-canonical $\Sg$-path from $S_2$ to $\Can$.

Now suppose we have a $\Sg$-path of interest
\[\begin{tikzcd}
	{S_1} & {S_2} & \dots & {S_n}
	\arrow["{k_1}", squiggly, from=1-1, to=1-2]
	\arrow["{k_2}", squiggly, from=1-2, to=1-3]
	\arrow["{k_{n-1}}", squiggly, from=1-3, to=1-4]
\end{tikzcd}\, .\]
In the following diagram, let $p_i$ be the $y_i$-canonical $\Sg$-path from $S_i$ to $\Can$, and let $q_i$ be the reverse of the $y_i$-canonical $\Sg$-path from $S_{i+1}$ to $\Can$.
\[\begin{tikzcd}
	{S_1} & \Can & {S_2} & \Can & {S_3} & \Can & {S_{n-1}} & \Can & {S_n}
	\arrow["{{{p_1}}}", dashed, from=1-1, to=1-2]
	\arrow[""{name=0, anchor=center, inner sep=0}, "{{{k_1}}}", curve={height=-24pt}, squiggly, from=1-1, to=1-3]
	\arrow["{{{q_1}}}", dashed, from=1-2, to=1-3]
	\arrow[""{name=1, anchor=center, inner sep=0}, "{{{\text{identity $\Sg$-step}}}}"', curve={height=24pt}, squiggly, from=1-2, to=1-4]
	\arrow["{{{p_2}}}", dashed, from=1-3, to=1-4]
	\arrow[""{name=2, anchor=center, inner sep=0}, "{{{k_2}}}", curve={height=-24pt}, squiggly, from=1-3, to=1-5]
	\arrow["{{{q_2}}}", dashed, from=1-4, to=1-5]
	\arrow[""{name=3, anchor=center, inner sep=0}, "{{{\text{identity $\Sg$-step}}}}"', curve={height=24pt}, squiggly, from=1-4, to=1-6]
	\arrow["{{{p_3}}}", dashed, from=1-5, to=1-6]
	\arrow["{{\text{\large $\dots$}}}"{description}, draw=none, from=1-6, to=1-7]
	\arrow["{{p_{n-1}}}", dashed, from=1-7, to=1-8]
	\arrow[""{name=4, anchor=center, inner sep=0}, "{{k_{n-1}}}", curve={height=-24pt}, squiggly, from=1-7, to=1-9]
	\arrow["{{{q_{n-1}}}}", dashed, from=1-8, to=1-9]
	\arrow["\equiv"{description, pos=0.3}, draw=none, from=1-2, to=0]
	\arrow["\equiv"{description, pos=0.3}, draw=none, from=1-3, to=1]
	\arrow["\equiv"{description, pos=0.3}, draw=none, from=1-4, to=2]
	\arrow["\equiv"{description, pos=0.3}, draw=none, from=1-5, to=3]
	\arrow["\equiv"{description, pos=0.3}, draw=none, from=1-8, to=4]
\end{tikzcd}\]
By the above argument we have the equivalences on the top, and the $\Omega$ 2-cell corresponding to each
\(\begin{tikzcd}
	\Can & {S_{i+1}} & \Can
	\arrow["{q_i}", squiggly, from=1-1, to=1-2]
	\arrow["{p_{i+1}}", squiggly, from=1-2, to=1-3]
\end{tikzcd}\)
is the identity, since $p_{i+1}$ is a canonical $\Sg$-path, $q_i$ is the reverse of a canonical $\Sg$-path, and canonical $\Sg$-paths departing from the same $\Sg$-scheme are equivalent.
 Consequently,  we conclude that the given $\Sg$-path is indeed equivalent to
\begin{equation}\label{eq:S-path}
\begin{tikzcd}
	{S_1} && \Can && {S_n}
	\arrow["{p_1}", dashed, from=1-1, to=1-3]
	\arrow["{q_{n-1}}", dashed, from=1-3, to=1-5]
\end{tikzcd}\,.
\end{equation}
Since all canonical $\Sg$-paths from a given $\Sg$-scheme to $\Can$ are equivalent, we conclude that all $\Sg$-paths of interest from $S_1$ to $S_n$ are equivalent to \eqref{eq:S-path}.
\end{proof}

\section{Appendix: Full proof of the universal property}\label{sec:appendixB}

Here we give a detailed account of the omitted calculations in the proof of the main result of \cref{sec:universal}. We will make extensive use of string diagrams. We provide a very basic introduction here to fix conventions, but see \cite{marsden2014stringdiagrams} for more details.

Recall that in a string diagram 1-morphisms are represented by wires, 2-morphisms are represented by dots and objects are represented by regions. Our string diagrams should be read from bottom to top (for vertical composition) and left to right (for horizontal composition). We will not label the regions of the diagrams, which should be clear from context.

Since identity morphisms are suppressed in string diagrams, the units and counits for an adjunction take a particularly appealing form, appearing to allow us `turn corners'. The triangle identities then say we can `pull the wires straight'.
\begin{center}
\vspace{-3pt}
\begin{minipage}{0.49\textwidth}
\begin{equation*}
\begin{tikzpicture}[scale=0.5,baseline={([yshift=-0.5ex]current bounding box.center)}]
\path coordinate[dot, label=above:$\eta$] (eta) ++(1,1) coordinate (a) ++(1,1) coordinate[dot, label=below:$\epsilon$] (epsilon)
 ++(1,-1) coordinate (b) ++(0,-2) coordinate[label=below:$f$] (br)
 (eta) ++(-1,1) coordinate (c) ++(0,2) coordinate[label=above:$f$] (tl);
\draw (tl) -- (c) to[out=-90, in=180] (eta) to[out=0, in=-90] (a) to[out=90, in=180] (epsilon) to[out=0, in=90] (b) -- (br);
\end{tikzpicture}
\enspace=\enspace
\begin{tikzpicture}[scale=0.5,baseline={([yshift=-0.5ex]current bounding box.center)}]
\path coordinate[label=below:$f$] (b) ++(0,4) coordinate[label=above:$f$] (t);
\draw (b) -- (t);
\end{tikzpicture}
\end{equation*}
\end{minipage}
\begin{minipage}{0.49\textwidth}
\begin{equation*}
\begin{tikzpicture}[scale=0.5,baseline={([yshift=-0.5ex]current bounding box.center)}]
\path coordinate[dot, label=below:$\epsilon$] (epsilon) ++(1,-1) coordinate (a) ++(1,-1) coordinate[dot, label=above:$\eta$] (eta)
 ++(1,1) coordinate (b) ++(0,2) coordinate[label=above:$f_*$] (tr)
 (epsilon) ++(-1,-1) coordinate (c) ++(0,-2) coordinate[label=below:$f_*$] (bl);
\draw (bl) -- (c) to[out=90, in=180] (epsilon) to[out=0, in=90] (a) to[out=-90, in=180] (eta) to[out=0, in=-90] (b) -- (tr);
\end{tikzpicture}
\enspace=\enspace
\begin{tikzpicture}[scale=0.5,baseline={([yshift=-0.5ex]current bounding box.center)}]
\path coordinate[label=below:$f_*$] (b) ++(0,4) coordinate[label=above:$f_*$] (t);
\draw (b) -- (t);
\end{tikzpicture}
\end{equation*}
\end{minipage}
\end{center}
In what follows we will suppress the dots and labels for units and counits as they can be readily understood from context.

\begin{lemma} \label{lem:transformation_at_adjoint}
 Let $H_1, H_2\colon \cata \to \catc$ be pseudofunctors and let $\upsilon\colon H_1 \to H_2$ be a pseudonatural transformation. Suppose $r\colon B \to I$ is a morphism in $\cata$ that has a right adjoint $r_*\colon I \to B$ (with unit $\eta$ and counit $\epsilon$). Then the 2-morphism $\upsilon_{r_*}\colon H_2(r_*) \circ \upsilon_I \to \upsilon_B \circ H_1(r_*)$ in the pseudonaturality square for $r_*$ is the inverse of the mate of 2-morphism $\upsilon_r$ in the pseudonaturality square for $r$.
\end{lemma}
\begin{proof}
 First note that $H_{1,2}(r_*)$ is right adjoint to $H_{1,2}(r)$ with unit $\gamma^{H_{1,2}\,-1}_{r_*,r} \circ H_{1,2}(\eta) \circ \iota^{H_{1,2}}_I$ and counit $\iota^{H_{1,2}\,-1}_B \circ H_{1,2}(\epsilon) \circ \gamma^{H_{1,2}}_{r,r_*}$.

 So we can express the mate of $\upsilon_r$ as follows.
 \ctikzfig{pseudonatural_at_adjoint}

 Now consider the composition $(\upsilon_r)_* \circ \upsilon_{r_*}$. We have
 \begin{equation*}
  \input{pseudonatural_at_adjoint_compose1.tikz} \quad = \quad
  \input{pseudonatural_at_adjoint_compose2.tikz}
  \quad = \quad
  \input{pseudonatural_at_adjoint_compose3.tikz}
 \end{equation*}
 where we have used the compositor, 2-dimensional naturality and unitor axioms of the pseudonatural transformation $\upsilon$ in turn for the first equality and the triangle identity for the second equality.
 Thus, we have shown $(\upsilon_r)_* \circ \upsilon_{r_*} = 1_{H_2(r_*) \upsilon_I}$. In a similar way, one can show $\upsilon_{r_*} \circ (\upsilon_r)_* = 1_{\upsilon_B H_1(r_*)}$ and hence these are inverse 2-morphisms, as required.
\end{proof}

\begin{theorem}\label{thm:universal_prop_append}
 Let $\catx$ be a 2-category admitting a  calculus of left lax fractions for $\Sigma$. The pseudofunctor $P_\Sigma\colon \catx \to \catx[\Sigma_*]$ as defined in \cref{def:Psigma} is the universal (strictly unital) pseudofunctor %
 that satisfies the conditions of \cref{prop:P_gives_laris}.
 More precisely, we have
 \begin{enumerate}[label=\normalfont(\alph*)]
  \item If $F\colon \catx \to \catc$ is a pseudofunctor sending $\Sigma$-morphisms to laris and $\Sigma$-squares to BC squares, then there is a pseudofunctor $H\colon \catx[\Sigma_*] \to \catc$ such that $F \simeq H \circ P_\Sigma$.
  \item If $H,H'\colon \catx[\Sigma_*] \to \catc$ are pseudofunctors and $\xi\colon H \circ P_\Sigma \to H' \circ P_\Sigma$ is a pseudonatural transformation for which the pseudonaturality squares for $r\colon B \to I$ in $\catx$ are BC squares whenever $r \in \Sigma$, then there is a pseudonatural transformation $\upsilon\colon H \to H'$ such that $\xi \cong \upsilon \circ P_\Sigma$.
  \item If $H,H'\colon \catx[\Sigma_*] \to \catc$ are pseudofunctors, $\upsilon, \upsilon' \colon H \to H'$ are pseudonatural transformations, and $\aleph \colon \upsilon \circ P_\Sigma \to \upsilon' \circ P_\Sigma$ is a modification, then there is a unique modification $\beth\colon \upsilon \to \upsilon'$ such that $\aleph = \beth \circ P_\Sigma$.
\end{enumerate}
\end{theorem}

We break up the proof into three parts.

\begin{proof}[Proof of (a)]
 Let $\catc$ be a bicategory and let $F\colon \catx \to \catc$ be a pseudofunctor that sends 1-cells in $\Sigma$ to laris in $\catc$ and $\Sigma$-squares to BC squares.
 We may assume the unitors of $\catc$ are identities and that $F$ is strictly unitary.
 For each lari $f$ of $\catc$ we choose a right adjoint $f_*$ (and associated unit/counits $\eta^f$ and $\epsilon^f$) such that $(1_X)_* = 1_X$ for every identity 1-cell (and $\eta^f = \epsilon^f = \id$ in this case).
 There is a pseudofunctor $H\colon \catx[\Sigma_*] \to \catc$ defined as follows.

 \begin{itemize} %
  \item On objects, $H(X) = F(X)$,
  \item On morphisms, $H( (f,r) ) = (F r)_* (F f)$,
  \item On 2-morphisms, $H\left(\begin{tikzcd}
	A & I & B \\
	& X & B \\
	A & J & B
	\arrow["f", from=1-1, to=1-2]
	\arrow[equals, from=1-1, to=3-1]
	\arrow[""{name=0, anchor=center, inner sep=0}, "{x_1}"', from=1-2, to=2-2]
	\arrow["\alpha"', shorten <=15pt, shorten >=15pt, Rightarrow, from=1-2, to=3-1]
	\arrow["r"', from=1-3, to=1-2]
	\arrow[""{name=1, anchor=center, inner sep=0}, equals, from=1-3, to=2-3]
	\arrow["{x_3}"', from=2-3, to=2-2]
	\arrow["g"', from=3-1, to=3-2]
	\arrow[""{name=2, anchor=center, inner sep=0}, "{x_2}", from=3-2, to=2-2]
	\arrow[""{name=3, anchor=center, inner sep=0}, equals, from=3-3, to=2-3]
	\arrow["s", from=3-3, to=3-2]
	\arrow["{\SIGMA^{\delta_1}}"{marking, allow upside down}, draw=none, from=0, to=1]
	\arrow["{\SIGMA^{\delta_2}}"{marking, allow upside down}, draw=none, from=2, to=3]
  \end{tikzcd}\right) = \!\!\!
  \begin{tikzcd}
	{F A} & {F I} & {F B} \\
	& {F X} & {F B} \\
	{F A} & {F J} & {F B}
	\arrow[""{name=0, anchor=center, inner sep=0}, "Ff", from=1-1, to=1-2]
	\arrow[""{name=1, anchor=center, inner sep=0}, "\text{\tiny $F(x_1 f)$}"'{yshift=2pt,xshift=2pt}, from=1-1, to=2-2, gray]
	\arrow[equals, from=1-1, to=3-1]
	\arrow[""{name=2, anchor=center, inner sep=0}, "{F x_1}"{pos=0.3,xshift=-1pt}, from=1-2, to=2-2]
	\arrow["(Fr)_*", to=1-3, from=1-2]
	\arrow[""{name=3, anchor=center, inner sep=0}, equals, from=1-3, to=2-3]
	\arrow["F(\delta_1)_*"{xshift=-4pt}, shorten <=6pt, shorten >=6pt, Rightarrow, from=1-3, to=2-2]
	\arrow["F(x_3)_*"'{pos=0.4, yshift=1pt}, to=2-3, from=2-2]
	\arrow["F(\delta_2)_*^{-1}"{xshift=-5pt}, shorten <=6pt, shorten >=6pt, Rightarrow, from=2-2, to=3-3, yshift=-3pt,xshift=-1pt]
	\arrow[""{name=4, anchor=center, inner sep=0}, "\text{\tiny $F(x_2 g)$}"{yshift=-2pt,xshift=2pt}, from=3-1, to=2-2, gray]
	\arrow[""{name=5, anchor=center, inner sep=0}, "{F g}"', from=3-1, to=3-2]
	\arrow[""{name=6, anchor=center, inner sep=0}, "{F x_2}"'{pos=0.3,xshift=-1pt}, from=3-2, to=2-2]
	\arrow[""{name=7, anchor=center, inner sep=0}, equals, from=3-3, to=2-3]
	\arrow["(Fs)_*"', to=3-3, from=3-2]
	\arrow["{F(\alpha)}"', shorten <=22pt, shorten >=22pt, Rightarrow, from=0, to=5]
	\arrow["\text{\tiny $\gamma^F_{x_1,f}$}"'{xshift=2pt}, shorten >=2pt, Rightarrow, from=1-2, to=1, xshift=4pt, yshift=-2pt, gray]
	\arrow["\text{\tiny $\gamma^F_{x_2,g}{}^{\!\!\!\!-1}$}"'{yshift=-2pt,xshift=3pt}, shorten <=2pt, Rightarrow, from=4, to=3-2, xshift=4pt,yshift=3pt, gray]
  \end{tikzcd}$, \\
  where $F(\delta_1)_*$ and $F(\delta_2)_*$ denote the mate of $F(\delta_1)$ and $F(\delta_2)$, respectively.
  \item The unitors $\iota^H_X\colon 1_{H(X)} \to H(1_X)$ are identites,
  \item The compositors $\gamma^H_{(g,s),(f,r)} \colon H( (g,s) ) \circ H( (f,r) ) \to H( (g,s) \circ (f,r) )$ are given by the composite
  {\tiny\begin{align*}
   & {( F(s)_* F(g) ) \circ ( F(r)_* F(f) )}
    \xrightarrow{\alpha^{\catc\; -1}_{F(s)_*F(g),F(r)_*,F(f)}}
    {((F(s)_* F(g)) \circ F(r)_*) \circ F(f)}
    \xrightarrow{\alpha^\catc_{F(s)_*,F(g),F(r)_*} \circ F(f)} \\
   & {(F(s)_* \circ (F(g) F(r)_*)) \circ F(f)}
    \xrightarrow{(F(s)_* \circ F(\dot{\beta})_*) \circ F(f)}
    {(F(s)_* \circ (F(\dot{r})_* F(\dot{g}))) \circ F(f)}
    \xrightarrow{\alpha^{\catc\; -1}_{F(s)_*,F(\dot{r}),F(\dot{g})}\circ F(f)} \\
   & {((F(s)_* F(\dot{r})_*) \circ F(\dot{g}))) \circ F(f)}
    \xrightarrow{\alpha^{\catc}_{F(s)_*F(\dot{r}),F(\dot{g}),F(f)}}
    {(F(s)_* F(\dot{r})_*) \circ (F(\dot{g}) F(f))}
    \xrightarrow{(F(s)_* F(\dot{r})_*) \circ \gamma^F_{\dot{g},f}} \\
   & {(F(s)_* F(\dot{r})_*) \circ F(\dot{g} f)}
    \xrightarrow{\sigma \circ F(\dot{g} f)}
    {(F(\dot{r}) F(s))_* \circ F(\dot{g} f)}
    \xrightarrow{(\gamma^{F\; -1}_{\dot{r},s})_* \circ F(\dot{g} f)}
    {F(\dot{r} s)_* F(\dot{g} f)}
  \end{align*}}
  where $\dot{g}$, $\dot{r}$ and $\dot{\beta}$ are given by the composite $(g,s) \circ (f,r) = (\dot{g}f, \dot{r}s)$ from

  \begin{equation}\label{eq:compositor_beta_dot_append}
   \begin{tikzcd}
	A & I & B \\
	& {\dot{B}} & J & C\rlap{,}
	\arrow["f", from=1-1, to=1-2]
	\arrow[""{name=0, anchor=center, inner sep=0}, "{\dot{g}}"', from=1-2, to=2-2]
	\arrow["r"', from=1-3, to=1-2]
	\arrow[""{name=1, anchor=center, inner sep=0}, "g", from=1-3, to=2-3]
	\arrow["{\dot{r}}", from=2-3, to=2-2]
	\arrow["s", from=2-4, to=2-3]
	\arrow["{\SIGMA^{\dot{\beta}}}"{marking, allow upside down}, draw=none, from=0, to=1]
   \end{tikzcd}
  \end{equation}
  and where $\alpha^\catc_{\bullet,\bullet,\bullet}$ denotes the associator for $\catc$, $\gamma^F_{\bullet,\bullet}$ denotes the compositor for $F$, $F(\dot{\beta})_*$ is the mate of $F(\Sigma^{\dot{\beta}})$, $\sigma\colon F(s)_* F(\dot{r})_* \to (F(\dot{r}) F(s))_*$ is the canonical isomorphism given by composition of adjoints and $(\gamma^{F\; -1}_{\dot{r},s})_*\colon (F(\dot{r}) F(s))_* \to F(\dot{r} s)_*$ is the mate of the inverse of the compositor.
 \end{itemize}

 We show this is indeed a pseudofunctor. Here it will be useful to make use of string diagrams.

 (a1) First we must show that $H$ is well-defined on 2-morphisms.
 Let us consider a $\Sigma$-extension of the 2-morphism $(\alpha, x_1, x_2, x_3, \delta_1 , \delta_2)$:
 \[\begin{tikzcd}
	A && I && B \\
	& X & {} & X & B \\
	&& D && B \\
	& X & {} & X & B \\
	A && J && B\rlap{.}
	\arrow[""{name=0, anchor=center, inner sep=0}, "f", from=1-1, to=1-3]
	\arrow[Rightarrow, no head, from=1-1, to=5-1]
	\arrow["{{{{{x_1}}}}}"{description}, curve={height=6pt}, from=1-3, to=2-2]
	\arrow[""{name=1, anchor=center, inner sep=0}, "{{{{x_1}}}}"{description}, curve={height=-6pt}, from=1-3, to=2-4]
	\arrow["{{{z_1}}}"{description}, from=1-3, to=3-3]
	\arrow["{{{r}}}"', from=1-5, to=1-3]
	\arrow[""{name=2, anchor=center, inner sep=0}, Rightarrow, no head, from=1-5, to=2-5]
	\arrow["{d}"{description}, curve={height=6pt}, from=2-2, to=3-3]
	\arrow["{{{{{\theta_1^{-1}}}}}}"', Rightarrow, from=2-3, to=2-2]
	\arrow["{{{{{\theta_1}}}}}"', Rightarrow, from=2-4, to=2-3]
	\arrow[""{name=3, anchor=center, inner sep=0}, "{d}"{description}, curve={height=-6pt}, from=2-4, to=3-3]
	\arrow["{{{{{x_3}}}}}"', from=2-5, to=2-4]
	\arrow[""{name=4, anchor=center, inner sep=0}, Rightarrow, no head, from=2-5, to=3-5]
	\arrow["z_3"', from=3-5, to=3-3] %
	\arrow[""{name=5, anchor=center, inner sep=0}, Rightarrow, no head, from=3-5, to=4-5]
	\arrow["{d}"{description}, curve={height=-6pt}, from=4-2, to=3-3]
	\arrow["{{{{{\theta_2}}}}}", Rightarrow, from=4-2, to=4-3]
	\arrow[""{name=6, anchor=center, inner sep=0}, "{d}"{description}, curve={height=6pt}, from=4-4, to=3-3]
	\arrow["{{{{{\theta_2}}}}}"', Rightarrow, from=4-4, to=4-3]
	\arrow["{{{{{x_3}}}}}"', from=4-5, to=4-4]
	\arrow[""{name=7, anchor=center, inner sep=0}, Rightarrow, no head, from=4-5, to=5-5]
	\arrow[""{name=8, anchor=center, inner sep=0}, "g"', from=5-1, to=5-3]
	\arrow["{{{z_2}}}"{description}, from=5-3, to=3-3]
	\arrow["{{{{{x_2}}}}}"{description}, curve={height=-6pt}, from=5-3, to=4-2]
	\arrow[""{name=9, anchor=center, inner sep=0}, "{{{{x_2}}}}"{description}, curve={height=6pt}, from=5-3, to=4-4]
	\arrow["{{{s}}}", from=5-5, to=5-3]
	\arrow["{d\circ\alpha}"', shift right=5, shorten <=26pt, shorten >=26pt, Rightarrow, from=0, to=8]
	\arrow["{{{{{\SIGMA^{\delta_1}}}}}}"{description}, draw=none, from=1, to=2]
	\arrow["{{{{{\SIGMA^{\psi}}}}}}"{description}, draw=none, from=3, to=4]
	\arrow["{{{{\SIGMA^{\psi}}}}}"{description}, draw=none, from=6, to=5]
	\arrow["{{{{{\SIGMA^{\delta_2}}}}}}"{description}, draw=none, from=9, to=7]
 \end{tikzcd}\]
 Applying $H$ to $(\alpha, x_1, x_2, x_3, \delta_1 , \delta_2)$ gives
 \[\begin{tikzcd}
	{F A} & {F I} & {F B} \\
	& {F X} & {F B} \\
	{F A} & {F J} & {F B}\rlap{,}
	\arrow[""{name=0, anchor=center, inner sep=0}, "Ff", from=1-1, to=1-2]
	\arrow[""{name=1, anchor=center, inner sep=0}, "\text{\tiny $F(x_1 f)$}"'{yshift=2pt,xshift=2pt}, from=1-1, to=2-2, gray]
	\arrow[equals, from=1-1, to=3-1]
	\arrow[""{name=2, anchor=center, inner sep=0}, "{F x_1}"{pos=0.3,xshift=-1pt}, from=1-2, to=2-2]
	\arrow["(Fr)_*", to=1-3, from=1-2]
	\arrow[""{name=3, anchor=center, inner sep=0}, equals, from=1-3, to=2-3]
	\arrow["F(\delta_1)_*"{xshift=-4pt}, shorten <=6pt, shorten >=6pt, Rightarrow, from=1-3, to=2-2]
	\arrow["F(x_3)_*"'{pos=0.4, yshift=1pt}, to=2-3, from=2-2]
	\arrow["F(\delta_2)_*^{-1}"{xshift=-5pt}, shorten <=6pt, shorten >=6pt, Rightarrow, from=2-2, to=3-3, yshift=-3pt,xshift=-1pt]
	\arrow[""{name=4, anchor=center, inner sep=0}, "\text{\tiny $F(x_2 g)$}"{yshift=-2pt,xshift=2pt}, from=3-1, to=2-2, gray]
	\arrow[""{name=5, anchor=center, inner sep=0}, "{F g}"', from=3-1, to=3-2]
	\arrow[""{name=6, anchor=center, inner sep=0}, "{F x_2}"'{pos=0.3,xshift=-1pt}, from=3-2, to=2-2]
	\arrow[""{name=7, anchor=center, inner sep=0}, equals, from=3-3, to=2-3]
	\arrow["(Fs)_*"', to=3-3, from=3-2]
	\arrow["{F(\alpha)}"', shorten <=22pt, shorten >=22pt, Rightarrow, from=0, to=5]
	\arrow["\text{\tiny $\gamma^F_{x_1,f}$}"'{xshift=2pt}, shorten >=2pt, Rightarrow, from=1-2, to=1, xshift=4pt, yshift=-2pt, gray]
	\arrow["\text{\tiny $\gamma^F_{x_2,g}{}^{\!\!\!\!-1}$}"'{yshift=-2pt,xshift=3pt}, shorten <=2pt, Rightarrow, from=4, to=3-2, xshift=4pt,yshift=3pt, gray]
  \end{tikzcd}\]
  while applying $H$ to the above $\Sigma$-extension gives
  \[\begin{tikzcd}[sep=large]
	{F A} & {F I} & {F B} \\
	& {F X} & {F B} \\
	{F A} & {F J} & {F B}\rlap{.}
	\arrow[""{name=0, anchor=center, inner sep=0}, "Ff", from=1-1, to=1-2]
	\arrow[""{name=1, anchor=center, inner sep=0}, "\text{\tiny $F(z_1 f)$}"'{yshift=2pt,xshift=2pt}, from=1-1, to=2-2, gray]
	\arrow[equals, from=1-1, to=3-1]
	\arrow[""{name=2, anchor=center, inner sep=0}, "{F z_1}"{pos=0.2,xshift=-1pt}, from=1-2, to=2-2]
	\arrow["(Fr)_*", to=1-3, from=1-2]
	\arrow[""{name=3, anchor=center, inner sep=0}, equals, from=1-3, to=2-3]
	\arrow["\text{\tiny $\substack{F(\theta_1r \cdot d\delta_1 \\ {} \cdot \psi)_*}$}"{pos=0.5,xshift=-2pt}, shorten <=10pt, shorten >=10pt, Rightarrow, from=1-3, to=2-2, yshift=5pt,xshift=-5pt]
	\arrow["F(z_3)_*"'{pos=0.4, yshift=1pt}, to=2-3, from=2-2]
	\arrow["\text{\tiny $\substack{F(\theta_2r \cdot d\delta_2 \\ {} \cdot \psi)_*^{-1}}$}"{pos=0.6, xshift=-5pt}, shorten <=10pt, shorten >=10pt, Rightarrow, from=2-2, to=3-3, yshift=-5pt,xshift=-5pt]
	\arrow[""{name=4, anchor=center, inner sep=0}, "\text{\tiny $F(z_2 g)$}"{yshift=-2pt,xshift=2pt}, from=3-1, to=2-2, gray]
	\arrow[""{name=5, anchor=center, inner sep=0}, "{F g}"', from=3-1, to=3-2]
	\arrow[""{name=6, anchor=center, inner sep=0}, "{F z_2}"'{pos=0.2,xshift=-1pt}, from=3-2, to=2-2]
	\arrow[""{name=7, anchor=center, inner sep=0}, equals, from=3-3, to=2-3]
	\arrow["(Fs)_*"', to=3-3, from=3-2]
	\arrow["\text{\tiny $\substack{F(\theta_2g \cdot d\alpha \\ {} \cdot \theta_1^{-1}f)}$}"', shorten <=30pt, shorten >=30pt, Rightarrow, from=0, to=5, xshift=5pt]
	\arrow["\text{\tiny $\gamma^F_{z_1,f}$}"'{xshift=2pt}, shorten >=2pt, Rightarrow, from=1-2, to=1, xshift=4pt, yshift=-2pt, gray]
	\arrow["\text{\tiny $\gamma^F_{z_2,g}{}^{\!\!\!\!-1}$}"'{yshift=-2pt,xshift=3pt}, shorten <=2pt, Rightarrow, from=4, to=3-2, xshift=4pt,yshift=3pt, gray]
  \end{tikzcd}\]
  To show these are equal we express both in terms of string diagrams.
  The former gives
  \begin{equation}\label{eq:H(alpha)}
  \input{H_alpha_.tikz}
  \end{equation}
  where the blue box depicts the inverse of the 2-morphism represented in the following diagram.
  \begin{equation}\label{eq:inv_of_mate}
  \input{mate_of_alpha.tikz}
  \end{equation}

  The latter gives
  \begin{equation}\label{eq:H(extension)}
  \input{H_extension_.tikz}
  \end{equation}
  (Here the three boxes represent the three rectangles in the diagram for $H$ applied to the $\Sigma$-extension.)
  Also, $F(\theta_2 r \cdot d\delta_2 \cdot \psi)_*$ is represented by
  \begin{equation}\label{eq:inv_of_mate_ext}
  \input{mate_of_extension.tikz} \!.
  \end{equation}

  In the string diagram \ref{eq:H(extension)}, $F(\theta_1)$ (together with its nearby compositor) cancels with $F(\theta_1^{-1})$ (and its compositor), and $F(\theta_2)$ (and its compositor) cancels with the $F(\theta_2)^{-1}$ (and its compositor) from the inverse of $F(\theta_2 r \cdot d\delta_2 \cdot \psi)_*$.

  Now we compose \cref{eq:H(extension)} and \cref{eq:H(alpha)} on the top with
  \begin{equation}\label{eq:fragment1}
   \input{fragment1.tikz}
  \end{equation}
  (which is an isomorphism since $\eta^{F(z_3)}$ is and $F(\delta_2)$ forms a BC square).
  Part of this cancels with the blue box in \cref{eq:H(alpha)} and another part partially cancels with the blue box in \cref{eq:H(extension)}.
  So what was \cref{eq:H(alpha)} becomes
  \ctikzfig{result1}
  and what was \cref{eq:H(extension)} yields
  \begin{center}\adjustbox{scale=0.8}{\input{result2.tikz}}\end{center}
  where $\widetilde{\delta}^{-1}$ denotes the inverse of
  \ctikzfig{fragment2}
  Observe that if we compose both of these with $\epsilon^{F(x_3)}$ (whiskered with the appropriate morphisms) on the top and apply one of the triangle identities we arrive at the same result in each case.
  While $\epsilon^{F(x_3)}$ is not an isomorphism, its resulting composition with \cref{eq:fragment1} is one by the BC condition for $F(\delta_2)$ and the form of the resulting inverse of its mate. So we have proved the equality of the two original 2-cells, and hence $H$ is well-defined on 2-cells.

  (a2) It is easy to see $H$ sends identity 2-cells to identity 2-cells by direct computation. %
  We now show that $H$ preserves vertical composition of 2-cells. Consider 2-morphisms $\bar{\alpha} = (\alpha, x_1, x_2, x_3, \delta_1, \delta_2)\colon (f,r) \Rightarrow (g,s)$ and $\bar{\beta} = (\beta, y_1, y_2, y_3, \epsilon_1, \epsilon_2)\colon (g,s) \Rightarrow (h,t)$. By applying Rule 4' to $\Sigma^{\delta_2}$ and $\Sigma^{\epsilon_1}$ we obtain
  \begin{equation*}
   \begin{tikzcd}
    B & J \\
    B & X & Y & {} \\
    B & D
    \arrow["s", from=1-1, to=1-2]
    \arrow[""{name=0, anchor=center, inner sep=0}, Rightarrow, no head, from=1-1, to=2-1]
    \arrow[""{name=1, anchor=center, inner sep=0}, "{{x_2}}", from=1-2, to=2-2]
    \arrow["{y_1}", curve={height=-6pt}, from=1-2, to=2-3]
    \arrow["{x_3}", from=2-1, to=2-2]
    \arrow[""{name=2, anchor=center, inner sep=0}, Rightarrow, no head, from=2-1, to=3-1]
    \arrow["\gamma", Rightarrow, from=2-2, to=2-3]
    \arrow["\cong"', draw=none, from=2-2, to=2-3]
    \arrow[""{name=3, anchor=center, inner sep=0}, "{{d_x}}", from=2-2, to=3-2]
    \arrow["{d_y}", curve={height=-6pt}, from=2-3, to=3-2]
    \arrow["u"', from=3-1, to=3-2]
    \arrow["{\SIGMA^{\delta_2}}"{description}, draw=none, from=0, to=1]
    \arrow["{\SIGMA^{\phi_x}}"{description}, draw=none, from=2, to=3]
   \end{tikzcd}
   =\qquad
   \begin{tikzcd}
    B & J \\
    B & Y \\
    B & D \rlap{\,.}
    \arrow["s", from=1-1, to=1-2]
    \arrow[""{name=0, anchor=center, inner sep=0}, Rightarrow, no head, from=1-1, to=2-1]
    \arrow[""{name=1, anchor=center, inner sep=0}, "{{y_1}}", from=1-2, to=2-2]
    \arrow["{{y_3}}", dashed, from=2-1, to=2-2]
    \arrow[""{name=2, anchor=center, inner sep=0}, Rightarrow, no head, from=2-1, to=3-1]
    \arrow[""{name=3, anchor=center, inner sep=0}, "{{d_y}}", from=2-2, to=3-2]
    \arrow["u"', from=3-1, to=3-2]
    \arrow["{\SIGMA^{\epsilon_1}}"{description}, draw=none, from=0, to=1]
    \arrow["{\SIGMA^{\phi_y}}"{description}, draw=none, from=2, to=3]
   \end{tikzcd}
  \end{equation*}
  We may use $\Sigma^{\phi_x}$ and $\gamma$ to obtain an $\Sigma$-extension of $\bar{\alpha}$ and $\Sigma^{\phi_y}$ to obtain a $\Sigma$-extension of $\bar{\beta}$, which agree on their abutting $\Sigma$-squares. By well-definedness $H$ we may replace $\bar{\alpha}$ and $\bar{\beta}$ by these. Thus, we may assume without loss of generality that $\Sigma^{\delta_2} = \Sigma^{\epsilon_1}$. Such 2-morphisms compose in a particularly simple way (since the necessary $\Sigma$-squares can be chosen to be identities). In particular, we may simply compose the $\alpha$ and $\beta$ and remove the matching $\Sigma$-squares. It is now easy to see that the repeated parts also cancel after applying $H$, since the mate and the inverse of the mate of the same morphism end up adjacent to each other and cancel. Thus, $H$ preserves vertical composition.

  (a3) Now we show naturality of the compositors $\gamma^H$.
  We can represent $\gamma^H_{(g,s),(f,r)}$ using string diagrams (omitting the associators, which we already know to be natural) as follows:
  \ctikzfig{compositor}
  where $\Sigma^{\dot{\beta}}$ refers to the canonical $\Sigma$-square for $r$ and $g$  as in (\ref{eq:compositor_beta_dot_append}) above.
  We can immediately simplify this using one of the triangle identities.

  \begin{equation}\label{eq:compositor}
   \input{compositor2.tikz}
  \end{equation}

  Now given horizontally composable 2-morphisms $\bar{\alpha}=(\alpha, x_1,x_2,x_3, \delta_1, \delta_2)\colon (f_1,r_1) \to (f_2,r_2)$ and $\bar{\beta}=(\beta, y_1,y_2,y_3, \epsilon_1, \epsilon_2)\colon (g_1,s_1) \to (g_2,s_2)$ we must prove
  \begin{equation}\label{eq:naturality_of_gamma_append}
   \gamma_{(g_2,s_2),(f_2,r_2)} \cdot (H(\bar{\beta})\circ H(\bar{\alpha})) = H(\bar{\beta} \circ \bar{\alpha}) \cdot \gamma_{(g_1,s_1),(f_1,r_1)}.
  \end{equation}

  Expressing the left-hand side of \cref{eq:naturality_of_gamma_append}
  in terms of string diagrams we have:
  \begin{equation}\label{eq:naturality_LHS}
   \adjustbox{scale=0.9}{\input{naturality_LHS.tikz}}
  \end{equation}
  For the right-hand side of \cref{eq:naturality_of_gamma_append} we need to find the horizontal composite $\bar{\beta} \circ \bar{\alpha}$.
  This is given by some $\Omega$ 2-cells composed with the core 2-cell
  \[\begin{tikzcd}
		A & {I_1} & X & V & Y & C \\
		&&& V & Y & C \\
		A & {I_2} & X & V & Y & C
		\arrow["{f_1}", from=1-1, to=1-2]
		\arrow[Rightarrow, no head, from=1-1, to=3-1]
		\arrow["{{x_1}}", from=1-2, to=1-3]
		\arrow["{{y_1'}}", from=1-3, to=1-4]
		\arrow["\alpha"', shorten <=11pt, shorten >=11pt, Rightarrow, from=1-3, to=3-1]
		\arrow[Rightarrow, no head, from=1-3, to=3-3]
		\arrow["v", tail reversed, no head, from=1-4, to=1-5]
		\arrow[""{name=0, anchor=center, inner sep=0}, Rightarrow, no head, from=1-4, to=2-4]
		\arrow["{{\beta'}}"', shorten <=11pt, shorten >=8pt, Rightarrow, from=1-4, to=3-3]
		\arrow["{y_3}", tail reversed, no head, from=1-5, to=1-6]
		\arrow[""{name=1, anchor=center, inner sep=0}, Rightarrow, no head, from=1-6, to=2-6]
		\arrow["v", from=2-5, to=2-4]
		\arrow["{y_3}", from=2-6, to=2-5]
		\arrow[""{name=2, anchor=center, inner sep=0}, Rightarrow, no head, from=2-6, to=3-6]
		\arrow["{f_2}"', from=3-1, to=3-2]
		\arrow["{{x_2}}"', from=3-2, to=3-3]
		\arrow["{{y'_2}}"', from=3-3, to=3-4]
		\arrow[""{name=3, anchor=center, inner sep=0}, Rightarrow, no head, from=3-4, to=2-4]
		\arrow["v", from=3-5, to=3-4]
		\arrow["{y_3}", from=3-6, to=3-5]
		\arrow["{\SIGMA^{\id}}"{description}, draw=none, from=0, to=1]
		\arrow["{\SIGMA^{\id}}"{description}, draw=none, from=3, to=2]
	\end{tikzcd}\]
	where $\beta'$ is obtained from the equality
	\begin{equation}\label{eq:defining_beta_prime_append}
	\begin{tikzcd}
		B & X &&& B & X \\
		{J_1} &&& {J_1} & {J_2} \\
		Y & V &&& Y & V
		\arrow["{{x_3}}", from=1-1, to=1-2]
		\arrow["{{g_1}}"', from=1-1, to=2-1]
		\arrow[""{name=0, anchor=center, inner sep=0}, "{{y'_1}}", from=1-2, to=3-2]
		\arrow[""{name=1, anchor=center, inner sep=0}, "{{y'_2}}"{pos=0.6}, curve={height=-30pt}, from=1-2, to=3-2]
		\arrow["{{x_3}}", from=1-5, to=1-6]
		\arrow["{{g_1}}"', curve={height=6pt}, from=1-5, to=2-4]
		\arrow["{{g_2}}", from=1-5, to=2-5]
		\arrow[""{name=2, anchor=center, inner sep=0}, "{{y'_2}}", from=1-6, to=3-6]
		\arrow["{{y_1}}"', from=2-1, to=3-1]
		\arrow["\beta", Rightarrow, from=2-4, to=2-5]
		\arrow["{{y_1}}"', curve={height=6pt}, from=2-4, to=3-5]
		\arrow["{{y_2}}", from=2-5, to=3-5]
		\arrow["v", from=3-1, to=3-2]
		\arrow["v", from=3-5, to=3-6]
		\arrow["{{\text{\Large =}}}"{description, pos=0.6}, draw=none, from=1, to=2-4]
		\arrow["{{\beta'}}"{pos=0.6}, shorten <=12pt, shorten >=6pt, Rightarrow, from=0, to=1]
		\arrow["{{\SIGMA^{\xi_1}}}"{description}, draw=none, from=2-1, to=0]
		\arrow["{{\SIGMA^{\xi_2}}}"{description}, draw=none, from=2-5, to=2]
	\end{tikzcd}
	\end{equation}
	as in \cref{eq:(A12)}.
	Now we find the $\Omega$ 2-cells $\Omega_i$ as in \cref{eq:Omega-i} for $i = 1,2$. The $d$-type $\Sigma$-step uses the basic $\Omega$ 2-cell
	\[\begin{tikzcd}
	{\dot{B}_i} & {\dot{B}_i} & {J_i} & C \\
	& {\tilde{J}_i} & Y & C \\
	{\dot{B}_i} & {\tilde{J}_i} & Y & C \rlap{\,.}
	\arrow["", equals, from=1-1, to=1-2]
	\arrow[equals, from=1-1, to=3-1]
	\arrow[""{name=0, anchor=center, inner sep=0}, "{\tilde{y}_i}"', from=1-2, to=2-2]
	\arrow["\id"', shorten <=15pt, shorten >=15pt, Rightarrow, from=1-2, to=3-1]
	\arrow["{\dot{r}_i}"', from=1-3, to=1-2]
	\arrow[""{name=1, anchor=center, inner sep=0}, from=1-3, to=2-3]
	\arrow["y_i", from=1-3, to=2-3]
	\arrow["{s_i}"', from=1-4, to=1-3]
	\arrow[""{name=2, anchor=center, inner sep=0}, equals, from=1-4, to=2-4]
	\arrow["{\tilde{r}_i}"{description}, from=2-3, to=2-2]
	\arrow["{y_3}"{description}, from=2-4, to=2-3]
	\arrow[""{name=3, anchor=center, inner sep=0}, equals, from=2-4, to=3-4]
	\arrow["{\tilde{y}_i}"', from=3-1, to=3-2]
	\arrow[""{name=4, anchor=center, inner sep=0}, equals, from=3-2, to=2-2]
	\arrow["{\tilde{r}_i}", from=3-3, to=3-2]
	\arrow["{y_3}", from=3-4, to=3-3]
	\arrow["{\SIGMA^{\tilde{\alpha}_i}}"{description}, draw=none, from=0, to=1]
	\arrow["{\SIGMA^{\epsilon_i}}"{description}, draw=none, from=1, to=2]
	\arrow["{\SIGMA^\id}"{description}, draw=none, from=4, to=3]
  \end{tikzcd}\] %
  For the $u$-type $\Sigma$-step we use Rule 4' to form the equality
  \begin{equation}\label{eq:defining_thetas_append}
  \begin{tikzcd}
	B & {I_i} \\
	{J_i} & {\dot{B}_i} \\
	Y & {\tilde{J}_i} \\
	Y & {D_i}
	\arrow["{r_i}", from=1-1, to=1-2]
	\arrow[""{name=0, anchor=center, inner sep=0}, "{g_i}"', from=1-1, to=2-1]
	\arrow[""{name=1, anchor=center, inner sep=0}, "{\dot{g}_i}"', from=1-2, to=2-2]
	\arrow["d^2_i y'_ix_i"{name=2}, curve={height=-30pt}, from=1-2, to=4-2]
	\arrow["{\dot{r}_i}", from=2-1, to=2-2]
	\arrow[""{name=3, anchor=center, inner sep=0}, "{y_i}"', from=2-1, to=3-1]
	\arrow[""{name=4, anchor=center, inner sep=0}, "{\tilde{y}_i}"', from=2-2, to=3-2]
	\arrow["{\tilde{r}_i}", from=3-1, to=3-2]
	\arrow[""{name=5, anchor=center, inner sep=0}, equals, from=3-1, to=4-1]
	\arrow[""{name=6, anchor=center, inner sep=0}, "{d_i^1}"', from=3-2, to=4-2]
	\arrow["{d_i}"', from=4-1, to=4-2]
	\arrow["\SIGMAc"{marking, allow upside down}, draw=none, from=0, to=1]
	\arrow["{\SIGMA^{\tilde{\alpha}_i}}"{marking, allow upside down}, draw=none, from=3, to=4]
	\arrow["{\phi_i}", shorten <=4pt, shorten >=4pt, Rightarrow, from=4, to=2]
	\arrow["{\SIGMA^{\theta^1_i}}"{marking, allow upside down}, draw=none, from=5, to=6]
  \end{tikzcd}
  \ =\ \ %
  \begin{tikzcd}
	B & {I_i} \\
	{J_i} \\
	Y & V \\
	Y & {D_i}
	\arrow["{r_i}", from=1-1, to=1-2]
	\arrow["{g_i}"', from=1-1, to=2-1]
	\arrow[""{name=0, anchor=center, inner sep=0}, "{y'_ix_i}", from=1-2, to=3-2]
	\arrow["{y_i}"', from=2-1, to=3-1]
	\arrow["v", from=3-1, to=3-2]
	\arrow[""{name=1, anchor=center, inner sep=0}, equals, from=3-1, to=4-1]
	\arrow[""{name=2, anchor=center, inner sep=0}, "{d_i^2}", from=3-2, to=4-2]
	\arrow["{d_i}"', from=4-1, to=4-2]
	\arrow["{\SIGMA^{\xi_i \odot \delta_i}}"{marking, allow upside down}, draw=none, from=2-1, to=0]
	\arrow["{\SIGMA^{\theta^2_i}}"{marking, allow upside down}, draw=none, from=1, to=2]
  \end{tikzcd}
  \end{equation}
  giving the basic $\Omega$ 2-cell
  \[\begin{tikzcd}
	{I_i} & {\tilde{J}_i} & Y \\
	& {D_i} & Y \\
	{I_i} & V & Y \rlap{\,.}
	\arrow["{\tilde{y}_i \dot{g}_i}", from=1-1, to=1-2]
	\arrow[equals, from=1-1, to=3-1]
	\arrow[""{name=0, anchor=center, inner sep=0}, "{d^1_i}"', from=1-2, to=2-2]
	\arrow["{\phi_i}"', shorten <=15pt, shorten >=15pt, Rightarrow, from=1-2, to=3-1]
	\arrow["{\tilde{r}_i}"', from=1-3, to=1-2]
	\arrow[""{name=1, anchor=center, inner sep=0}, equals, from=1-3, to=2-3]
	\arrow["{d_i}"{description}, from=2-3, to=2-2]
	\arrow[""{name=2, anchor=center, inner sep=0}, equals, from=2-3, to=3-3]
	\arrow["{y'_ix_i}"', from=3-1, to=3-2]
	\arrow[""{name=3, anchor=center, inner sep=0}, "{d^2_i}", from=3-2, to=2-2]
	\arrow["v", from=3-3, to=3-2]
	\arrow["{\SIGMA^{\theta^1_i}}"{description}, draw=none, from=0, to=1]
	\arrow["{\SIGMA^{\theta^2_i}}"{description}, draw=none, from=3, to=2]
  \end{tikzcd}\]
  Now we can use these to form the composite $\Omega$ 2-cells. We have
  \[\begin{tikzcd}
	{I_i} & {\dot{B}_i} & {\dot{B}_i} & {J_i} & C \\
	&& {\tilde{J}_i} & Y & C \\
	{I_i} & {\dot{B}_i} & {\tilde{J}_i} & Y & C \\
	&& {D_i} & Y & C \\
	{I_i} && V & Y & C
	\arrow["{\dot{g}_i}", from=1-1, to=1-2]
	\arrow[equals, from=1-1, to=3-1]
	\arrow[equals, from=1-2, to=1-3]
	\arrow["\id"', between={0.3}{0.7}, Rightarrow, from=1-2, to=3-1]
	\arrow[equals, from=1-2, to=3-2]
	\arrow[""{name=0, anchor=center, inner sep=0}, "{\tilde{y}_i}"', from=1-3, to=2-3]
	\arrow["\id"', between={0.3}{0.7}, Rightarrow, from=1-3, to=3-2]
	\arrow["{\dot{r}_i}"', from=1-4, to=1-3]
	\arrow["{s_i}"', from=1-5, to=1-4]
	\arrow[""{name=1, anchor=center, inner sep=0}, equals, from=1-5, to=2-5]
	\arrow["{\tilde{r}_i}"', from=2-4, to=2-3]
	\arrow["{y_3}"', from=2-5, to=2-4]
	\arrow["{\dot{g}_i}", from=3-1, to=3-2]
	\arrow[equals, from=3-1, to=5-1]
	\arrow["{\tilde{y}_i}", from=3-2, to=3-3]
	\arrow[""{name=2, anchor=center, inner sep=0}, equals, from=3-3, to=2-3]
	\arrow[""{name=3, anchor=center, inner sep=0}, "{d^1_i}"', from=3-3, to=4-3]
	\arrow["{\phi_i}"', between={0.35}{0.65}, Rightarrow, from=3-3, to=5-1]
	\arrow["{\tilde{r}_i}"', from=3-4, to=3-3]
	\arrow[""{name=4, anchor=center, inner sep=0}, equals, from=3-4, to=4-4]
	\arrow[""{name=5, anchor=center, inner sep=0}, equals, from=3-5, to=2-5]
	\arrow["{y_3}"', from=3-5, to=3-4]
	\arrow[""{name=6, anchor=center, inner sep=0}, equals, from=3-5, to=4-5]
	\arrow["{d_i}", from=4-4, to=4-3]
	\arrow[""{name=7, anchor=center, inner sep=0}, equals, from=4-4, to=5-4]
	\arrow["{y_3}", from=4-5, to=4-4]
	\arrow["{y'_ix_i}"', from=5-1, to=5-3]
	\arrow[""{name=8, anchor=center, inner sep=0}, "{d^2_i}", from=5-3, to=4-3]
	\arrow["v", from=5-4, to=5-3]
	\arrow[""{name=9, anchor=center, inner sep=0}, equals, from=5-5, to=4-5]
	\arrow["{y_3}", from=5-5, to=5-4]
	\arrow["{\SIGMA^{\epsilon_i \oplus \tilde{\alpha}_i}}"{marking, allow upside down}, draw=none, from=0, to=1]
	\arrow["{\SIGMA^\id}"{marking, allow upside down}, draw=none, from=2, to=5]
	\arrow["{\SIGMA^{\theta^1_i}}"{description}, draw=none, from=3, to=4]
	\arrow["{\SIGMA^\id}"{marking, allow upside down}, draw=none, from=4, to=6]
	\arrow["{\SIGMA^\id}"{marking, allow upside down}, draw=none, from=7, to=9]
	\arrow["{\SIGMA^{\theta^2_i}}"{description}, draw=none, from=8, to=7]
  \end{tikzcd}\]
  and using the obvious choices of $\Sigma$-square to compute the composite %
  and after whiskering with $f_i$ we arrive at $\Omega_i$ for $i = 1,2$.
  \[\begin{tikzcd}
	A & {I_i} & {\dot{B}_i} & C \\
	&& {\tilde{J}_i} & C \\
	&& {D_i} & C \\
	A & {I_i} & V & C
	\arrow["{f_i}", from=1-1, to=1-2]
	\arrow[equals, from=1-1, to=4-1]
	\arrow["{\dot{g}_i}", from=1-2, to=1-3]
	\arrow["\id"', between={0.4}{0.6}, Rightarrow, from=1-2, to=4-1]
	\arrow[equals, from=1-2, to=4-2]
	\arrow[""{name=0, anchor=center, inner sep=0}, "{\tilde{y}_i}"', from=1-3, to=2-3]
	\arrow["{\phi_i}"', between={0.35}{0.65}, Rightarrow, from=1-3, to=4-2]
	\arrow["{\dot{r}_i s_i}"', from=1-4, to=1-3]
	\arrow[""{name=1, anchor=center, inner sep=0}, equals, from=1-4, to=2-4]
	\arrow[""{name=2, anchor=center, inner sep=0}, "{d^1_i}"', from=2-3, to=3-3]
	\arrow["{\tilde{r}_i y_3}"', from=2-4, to=2-3]
	\arrow[""{name=3, anchor=center, inner sep=0}, equals, from=2-4, to=3-4]
	\arrow["{d_i y_3}"', from=3-4, to=3-3]
	\arrow["{f_i}"', from=4-1, to=4-2]
	\arrow["{y'_ix_i}"', from=4-2, to=4-3]
	\arrow[""{name=4, anchor=center, inner sep=0}, "{d^2_i}", from=4-3, to=3-3]
	\arrow[""{name=5, anchor=center, inner sep=0}, equals, from=4-4, to=3-4]
	\arrow["{v y_3}", from=4-4, to=4-3]
	\arrow["{\SIGMA^{\epsilon_i \oplus \tilde{\alpha}_i}}"{marking, allow upside down}, draw=none, from=0, to=1]
	\arrow["{\SIGMA^{\theta^1_i y_3}}"{description}, draw=none, from=2, to=3]
	\arrow["{\SIGMA^{\theta^2_i y_3}}"{description}, draw=none, from=4, to=5]
  \end{tikzcd}\]
  Finally, we will obtain the horizontal composite by vertically composing the core 2-cell from before with $\Omega_1$ and $\Omega_2^{-1}$. The first composite with $\Omega_1$ gives %
  \[\begin{tikzcd}
	A && {\dot{B}_1} & C \\
	&& {\tilde{J}_1} & C \\
	A && {D_1} & C \\
	A && V & C \rlap{\,.}
	\arrow["{\dot{g}_1 f_1}", from=1-1, to=1-3]
	\arrow[equals, from=1-1, to=3-1]
	\arrow[""{name=0, anchor=center, inner sep=0}, "{\tilde{y}_1}"', from=1-3, to=2-3]
	\arrow["{\phi_1 f_1}"', between={0.35}{0.65}, Rightarrow, from=1-3, to=3-1]
	\arrow["{\dot{r}_1 s_1}"', from=1-4, to=1-3]
	\arrow[""{name=1, anchor=center, inner sep=0}, equals, from=1-4, to=2-4]
	\arrow[""{name=2, anchor=center, inner sep=0}, "{d^1_1}"', from=2-3, to=3-3]
	\arrow["{\tilde{r}_1 y_3}"', from=2-4, to=2-3]
	\arrow[""{name=3, anchor=center, inner sep=0}, equals, from=2-4, to=3-4]
	\arrow["{d^2_1 y'_1 x_1 f_1}", from=3-1, to=3-3]
	\arrow[equals, from=3-1, to=4-1]
	\arrow["{d^2_1(\beta' \circ \alpha)}"', xshift=5pt, between={0.2}{0.8}, Rightarrow, from=3-3, to=4-1]
	\arrow["{d_1 y_3}"', from=3-4, to=3-3]
	\arrow["{f_2}"', from=4-1, to=4-3]
	\arrow[""{name=4, anchor=center, inner sep=0}, "{d^2_1}", from=4-3, to=3-3]
	\arrow[""{name=5, anchor=center, inner sep=0}, equals, from=4-4, to=3-4]
	\arrow["{v y_3}", from=4-4, to=4-3]
	\arrow["{\SIGMA^{\epsilon_1 \oplus \tilde{\alpha}_1}}"{marking, allow upside down}, draw=none, from=0, to=1]
	\arrow["{\SIGMA^{\theta^1_1 y_3}}"{description}, draw=none, from=2, to=3]
	\arrow["{\SIGMA^{\theta^2_1 y_3}}"{description}, draw=none, from=4, to=5]
  \end{tikzcd}\]
  Then for the second composite with $\Omega_2^{-1}$ (which is obtained from $\Omega_2$ by inverting $\phi_i$ and swapping the top and bottom $\Sigma$-squares) we use the equality
  \begin{equation}\label{eq:defining_psis_append}
  \begin{tikzcd}
	C & V \\
	C & {D_1} & {D_2} \\
	C & E
	\arrow["{v y_3}", from=1-1, to=1-2]
	\arrow[""{name=0, anchor=center, inner sep=0}, equals, from=1-1, to=2-1]
	\arrow[""{name=1, anchor=center, inner sep=0}, "{d^2_1}", from=1-2, to=2-2]
	\arrow["{d^2_2}", curve={height=-6pt}, from=1-2, to=2-3]
	\arrow["{d_1 y_3}", from=2-1, to=2-2]
	\arrow[""{name=2, anchor=center, inner sep=0}, equals, from=2-1, to=3-1]
	\arrow["\zeta", between={0.2}{0.8}, Rightarrow, from=2-2, to=2-3]
	\arrow[""{name=3, anchor=center, inner sep=0}, "{e_1}", from=2-2, to=3-2]
	\arrow["{e_2}", curve={height=-6pt}, from=2-3, to=3-2]
	\arrow["e"', from=3-1, to=3-2]
	\arrow["{\SIGMA^{\theta^2_1 y_3}}"{marking, allow upside down}, draw=none, from=0, to=1]
	\arrow["{\SIGMA^{\psi_1}}"{marking, allow upside down}, draw=none, from=2, to=3]
  \end{tikzcd}
  \ =\ %
  \begin{tikzcd}
	C & V \\
	C & {D_2} \\
	C & E
	\arrow["{v y_3}", from=1-1, to=1-2]
	\arrow[""{name=0, anchor=center, inner sep=0}, equals, from=1-1, to=2-1]
	\arrow[""{name=1, anchor=center, inner sep=0}, "{d^2_2}", from=1-2, to=2-2]
	\arrow["{d_2 y_3}", from=2-1, to=2-2]
	\arrow[""{name=2, anchor=center, inner sep=0}, equals, from=2-1, to=3-1]
	\arrow[""{name=3, anchor=center, inner sep=0}, "{e_2}", from=2-2, to=3-2]
	\arrow["e"', from=3-1, to=3-2]
	\arrow["{\SIGMA^{\theta^2_2 y_3}}"{marking, allow upside down}, draw=none, from=0, to=1]
	\arrow["{\SIGMA^{\psi_2}}"{marking, allow upside down}, draw=none, from=2, to=3]
  \end{tikzcd}
  \end{equation}
  to obtain the final composite
  \[\begin{tikzcd}
	A && {\dot{B}_1} & C \\
	&& {\tilde{J}_1} & C \\
	&& {D_1} & C \\
	A & V & E & C \\
	&& {D_2} & C \\
	&& {\tilde{J}_2} & C \\
	A && {\dot{B}_2} & C \rlap{\,.}
	\arrow["{\dot{g}_1 f_1}", from=1-1, to=1-3]
	\arrow[equals, from=1-1, to=4-1]
	\arrow[""{name=0, anchor=center, inner sep=0}, "{\tilde{y}_1}"', from=1-3, to=2-3]
	\arrow["{\text{\tiny $\begin{aligned} e_1 d^2_1(\beta' \circ \alpha) \\ {} \cdot e_1 \phi_1 f_1 \end{aligned}$}}"', shift left=4, between={0.3}{0.7}, Rightarrow, from=1-3, to=4-1]
	\arrow["{\dot{r}_1 s_1}"', from=1-4, to=1-3]
	\arrow[""{name=1, anchor=center, inner sep=0}, equals, from=1-4, to=2-4]
	\arrow[""{name=2, anchor=center, inner sep=0}, "{d^1_1}"', from=2-3, to=3-3]
	\arrow["{\tilde{r}_1 y_3}"', from=2-4, to=2-3]
	\arrow[""{name=3, anchor=center, inner sep=0}, equals, from=2-4, to=3-4]
	\arrow[""{name=4, anchor=center, inner sep=0}, "{e_1}", from=3-3, to=4-3]
	\arrow["{d_1 y_3}"', from=3-4, to=3-3]
	\arrow["{y'_2x_2f_2}"', from=4-1, to=4-2]
	\arrow[equals, from=4-1, to=7-1]
	\arrow[""{name=5, anchor=center, inner sep=0}, "{e_1 d_1^2}", curve={height=-12pt}, from=4-2, to=4-3]
	\arrow[""{name=6, anchor=center, inner sep=0}, "{e_2 d^2_2}"', curve={height=12pt}, from=4-2, to=4-3]
	\arrow["{e_2 \phi_2^{-1} f_2}"', between={0.3}{0.7}, Rightarrow, from=4-3, to=7-1]
	\arrow[""{name=7, anchor=center, inner sep=0}, equals, from=4-4, to=3-4]
	\arrow["e", from=4-4, to=4-3]
	\arrow[""{name=8, anchor=center, inner sep=0}, "{e_2}"', from=5-3, to=4-3]
	\arrow[""{name=9, anchor=center, inner sep=0}, equals, from=5-4, to=4-4]
	\arrow["{d_2 y_3}", from=5-4, to=5-3]
	\arrow[""{name=10, anchor=center, inner sep=0}, "{d^1_2}", from=6-3, to=5-3]
	\arrow[""{name=11, anchor=center, inner sep=0}, equals, from=6-4, to=5-4]
	\arrow["{\tilde{r}_2 y_3}", from=6-4, to=6-3]
	\arrow["{\dot{g}_2 f_2}"', from=7-1, to=7-3]
	\arrow[""{name=12, anchor=center, inner sep=0}, "{\tilde{y}_2}", from=7-3, to=6-3]
	\arrow[""{name=13, anchor=center, inner sep=0}, equals, from=7-4, to=6-4]
	\arrow["{\dot{r}_2 s_2}", from=7-4, to=7-3]
	\arrow["{\SIGMA^{\epsilon_1 \oplus \tilde{\alpha}_1}}"{marking, allow upside down}, draw=none, from=0, to=1]
	\arrow["{\SIGMA^{\theta^1_1 y_3}}"{description}, draw=none, from=2, to=3]
	\arrow["{\SIGMA^{\psi_1}}"{description}, draw=none, from=4, to=7]
	\arrow["\zeta"', between={0.2}{0.8}, Rightarrow, from=5, to=6]
	\arrow["{\SIGMA^{\psi_2}}"{marking, allow upside down}, draw=none, from=8, to=9]
	\arrow["{\SIGMA^{\theta^1_2 y_3}}"{marking, allow upside down}, shift right, draw=none, from=10, to=11]
	\arrow["{\SIGMA^{\epsilon_2 \oplus \tilde{\alpha}_2}}"{marking, allow upside down}, draw=none, from=12, to=13]
  \end{tikzcd}\]

   The result of applying $H$ to this is represented by the
   string diagram \eqref{eq:below} below,
   where $\Xi_*^{-1}$ denotes the inverse of the following 2-cell:
  \begin{equation}\label{eq:Xi_star}
   \input{Xi_part_of_composite.tikz}
  \end{equation}

  \begin{equation}\label{eq:below}\input{H_of_horiz_composite.tikz}
   \end{equation}%

  We now obtain the right-hand side of \cref{eq:naturality_of_gamma_append} by composing this with $\gamma_{(g_1,s_1),(f_1,r_1)}$.
  Using a triangle inequality this right-hand side can be represented by the following string diagram.

  \begin{equation}\label{eq:naturality_RHS}
   \adjustbox{scale=0.9}{\input{naturality_RHS_simplified1.tikz}} %
  \end{equation}

  To show the string diagrams \eqref{eq:naturality_LHS} and \eqref{eq:naturality_RHS} give the same 2-cell, we may first compose each with (an appropriately whiskered form of) the isomorphism $\Xi_*$ (\eqref{eq:Xi_star}). This cancels with $\Xi_*^{-1}$ part of the right-hand side. Next we can `bend the wire' $F(e)_*$ by composing with $\epsilon^{F(e)}$ in a similar way. (Note that this is again an invertible operation.) Thus, we arrive at a modified right-hand side (where we have also cancelled a compositor and its inverse involving $\dot{r}_1 s_1$)

  \begin{equation}\label{eq:naturality_RHS2}
   \adjustbox{scale=0.9}{\input{naturality_RHS_modified2.tikz}}
  \end{equation}
  and a modified left-hand side (after applying a triangle identity and cancelling inverse compositors)

  \begin{equation}\label{eq:naturality_LHS2}
   \adjustbox{scale=0.9}{\input{naturality_LHS_modified1.tikz}}
  \end{equation}

  To show these are equal we will need to use \cref{eq:defining_beta_prime_append,,eq:defining_psis_append,eq:defining_thetas_append}. %
  Applying $F$ each of these we obtain the following string diagrams.
  \begin{equation}\label{eq:defining_beta_prime2}
   \input{F_beta_prime_equation_LHS.tikz}
   \quad=\quad
   \input{F_beta_prime_equation_RHS.tikz}
  \end{equation}

  \begin{equation}\label{eq:defining_psis2}
   \adjustbox{scale=1.1}{\input{F_psi_equation_LHS.tikz}}
   \quad=\quad
   \adjustbox{scale=1.1}{\input{F_psi_equation_RHS.tikz}}
  \end{equation}

  \begin{equation}\label{eq:defining_thetas2}
   \adjustbox{scale=1.1}{\input{F_theta_equation_LHS.tikz}}
   \quad=\quad
   \adjustbox{scale=1.1}{\input{F_theta_equation_RHS.tikz}}
  \end{equation}

  First can find the left-hand side of \cref{eq:defining_thetas2} (for $i=1$) in the modified right-hand side \cref{eq:naturality_RHS2}.
  Replacing it with the right-hand side of \cref{eq:defining_thetas2} then gives the following (where the highlighted section is the part that has been changed).

  \begin{equation}\label{eq:naturality_RHS3}
   \input{naturality_RHS_modified3.tikz}
  \end{equation}

  Next we apply \cref{eq:defining_psis2} to give the following.
  \begin{equation}\label{eq:naturality_RHS4}
   \input{naturality_RHS_modified4.tikz}
  \end{equation}

  Now we use \cref{eq:defining_beta_prime2} to obtain
  \begin{equation}\label{eq:naturality_RHS5}
   \adjustbox{scale=0.95}{\input{naturality_RHS_modified5.tikz}}
  \end{equation}

  We will now use \cref{eq:defining_thetas2} once more. Composing both sides of \cref{eq:defining_thetas2} for $i=2$ with $F(\phi_2)^{-1}$ together with the appropriate compositors with have

  \begin{equation}\label{eq:defining_thetas3}
   \adjustbox{scale=0.8}{\input{F_theta_equation_LHS2.tikz}}
   \quad=\quad
   \adjustbox{scale=0.8}{\input{F_theta_equation_RHS2.tikz}}
  \end{equation}

  The left-hand side of this equation now appears in the left-hand side (\eqref{eq:naturality_LHS2}) of the equation we are trying to prove.
  We arrive at the following.
  \begin{equation}\label{eq:naturality_LHS3}
   \adjustbox{scale=0.80}{\input{naturality_LHS_modified2.tikz}}
  \end{equation}

  Now by using the triangle identities to introduce zig-zags (and consolodating compositors) we can identify copies of $F(\delta_2)_*$ and $F(\epsilon_2)_*$ in this diagram.
  \begin{equation}\label{eq:naturality_LHS4}
   \input{naturality_LHS_modified2_highlight.tikz}
  \end{equation}
  These then cancel with the blue boxes and after simplifying the resulting diagram we arrive at

  \begin{equation}\label{eq:naturality_LHS5}
   \adjustbox{scale=0.85}{\input{naturality_LHS_modified4.tikz}}
  \end{equation}
  This is now exactly the same as Diagram \eqref{eq:naturality_RHS5} and thus we have shown \cref{eq:naturality_of_gamma_append} and the compositors are indeed natural.

  (a4) It is easy to see that the unit coherence condition for the compositor holds. We now show the associativity condition.
  Suppose we have the triple composite $(h,t) \circ (g,s) \circ (f,r)$. If we associate to the right we have %
  \[\begin{tikzcd}
	\bullet & \bullet & \bullet \\
	& \bullet & \bullet & \bullet
	\arrow["f", from=1-1, to=1-2]
	\arrow[""{name=0, anchor=center, inner sep=0}, "{\dot{g}}"', from=1-2, to=2-2]
	\arrow["r"', from=1-3, to=1-2]
	\arrow[""{name=1, anchor=center, inner sep=0}, "g", from=1-3, to=2-3]
	\arrow["{\dot{r}}", from=2-3, to=2-2]
	\arrow["s", from=2-4, to=2-3]
	\arrow["{\SIGMA^{\dot{\alpha}}}"{description}, draw=none, from=0, to=1]
  \end{tikzcd}\]
  and
  \[\begin{tikzcd}
	\bullet & \bullet & \bullet & \bullet & \bullet \\
	&& \bullet && \bullet & \bullet \rlap{\,.}
	\arrow["f", from=1-1, to=1-2]
	\arrow["{{\dot{g}}}", from=1-2, to=1-3]
	\arrow[""{name=0, anchor=center, inner sep=0}, "{{\ddot{h}}}"', from=1-3, to=2-3]
	\arrow["{{\dot{r}}}"', from=1-4, to=1-3]
	\arrow["s"', from=1-5, to=1-4]
	\arrow[""{name=1, anchor=center, inner sep=0}, "h", from=1-5, to=2-5]
	\arrow["{{\dot{r}'}}", from=2-5, to=2-3]
	\arrow["t", from=2-6, to=2-5]
	\arrow["{{\SIGMA^{\ddot{\alpha}}}}"{description}, draw=none, from=0, to=1]
  \end{tikzcd}\]
  Associating to the left we have

  \[\begin{tikzcd}
	\bullet & \bullet & \bullet \\
	& \bullet & \bullet & \bullet
	\arrow["g", from=1-1, to=1-2]
	\arrow[""{name=0, anchor=center, inner sep=0}, "{\dot{h}}"', from=1-2, to=2-2]
	\arrow["s"', from=1-3, to=1-2]
	\arrow[""{name=1, anchor=center, inner sep=0}, "h", from=1-3, to=2-3]
	\arrow["{\dot{s}}", from=2-3, to=2-2]
	\arrow["t", from=2-4, to=2-3]
	\arrow["{\SIGMA^{\dot{\beta}}}"{description}, draw=none, from=0, to=1]
  \end{tikzcd}\]
  and
  \[\adjustbox{scale=0.9}{\begin{tikzcd}
	\bullet & \bullet & \bullet \\
	&& \bullet \\
	& \bullet & \bullet & \bullet & {\bullet \rlap{\,.}}
	\arrow["f", from=1-1, to=1-2]
	\arrow[""{name=0, anchor=center, inner sep=0}, "{\dot{h}'}"', from=1-2, to=3-2]
	\arrow["r"', from=1-3, to=1-2]
	\arrow["g", from=1-3, to=2-3]
	\arrow["{\dot{h}}", from=2-3, to=3-3]
	\arrow["{\ddot{r}}", from=3-3, to=3-2]
	\arrow["{\dot{s}}", from=3-4, to=3-3]
	\arrow["t", from=3-5, to=3-4]
	\arrow["{\SIGMA^{\ddot{\beta}}}"{description}, draw=none, from=0, to=2-3]
  \end{tikzcd}}\]
  To show the associativity condition we must show the following equality of string diagrams.
  \begin{equation}\label{eq:compositor_assoc}
   \adjustbox{scale=0.85}{\input{assoc_LHS.tikz}}
   \quad =\quad
   \adjustbox{scale=0.85}{\input{assoc_RHS.tikz}}
  \end{equation}
  We now need to expand the 2-cell $H(\Assoc_{\bar{h},\bar{g},\bar{h}})$.
  Recall that $\Assoc_{\bar{h},\bar{g},\bar{h}}$ is given by the $\Sigma$-path
  \[\adjustbox{scale=0.9}{\begin{tikzcd}
	& \bullet & \bullet &&& \bullet & \bullet &&& \bullet & \bullet \\
	\bullet & \bullet & {} && \bullet & \bullet & \bullet && \bullet & \bullet & \bullet \\
	\bullet & \bullet & \bullet && \bullet & \bullet & \bullet && \bullet && \bullet {\rlap{\;.}}
	\arrow["r", from=1-2, to=1-3]
	\arrow["g"', from=1-2, to=2-2]
	\arrow[""{name=0, anchor=center, inner sep=0}, "{\dot{h}'}", from=1-3, to=3-3]
	\arrow["r", from=1-6, to=1-7]
	\arrow[""{name=1, anchor=center, inner sep=0}, "g"', from=1-6, to=2-6]
	\arrow[""{name=2, anchor=center, inner sep=0}, "{\dot{g}}", from=1-7, to=2-7]
	\arrow["r", from=1-10, to=1-11]
	\arrow[""{name=3, anchor=center, inner sep=0}, "g"', from=1-10, to=2-10]
	\arrow[""{name=4, anchor=center, inner sep=0}, "{\dot{g}}", from=1-11, to=2-11]
	\arrow["s", from=2-1, to=2-2]
	\arrow[""{name=5, anchor=center, inner sep=0}, "h"', from=2-1, to=3-1]
	\arrow[""{name=6, anchor=center, inner sep=0}, "{\dot{h}}", from=2-2, to=3-2]
	\arrow["{u}", between={0.3}{0.7}, squiggly, from=2-3, to=2-5]
	\arrow["s", from=2-5, to=2-6]
	\arrow[""{name=7, anchor=center, inner sep=0}, "h"', from=2-5, to=3-5]
	\arrow["{\dot{r}}", from=2-6, to=2-7]
	\arrow[""{name=8, anchor=center, inner sep=0}, "{\dot{h}}", from=2-6, to=3-6]
	\arrow["{d}", between={0.3}{0.7}, squiggly, from=2-7, to=2-9]
	\arrow[""{name=9, anchor=center, inner sep=0}, "{\tilde{h}}", from=2-7, to=3-7]
	\arrow["s", from=2-9, to=2-10]
	\arrow[""{name=10, anchor=center, inner sep=0}, "h"', from=2-9, to=3-9]
	\arrow["{\dot{r}}", from=2-10, to=2-11]
	\arrow[""{name=11, anchor=center, inner sep=0}, "{\ddot{h}}", from=2-11, to=3-11]
	\arrow["{\dot{s}}"', from=3-1, to=3-2]
	\arrow["{\ddot{r}}"', from=3-2, to=3-3]
	\arrow["{\dot{s}}"', from=3-5, to=3-6]
	\arrow["{\tilde{r}}"', from=3-6, to=3-7]
	\arrow["{\dot{r}'}"', from=3-9, to=3-11]
	\arrow["{\SIGMA^{\dot{\alpha}}}"{description}, draw=none, from=1, to=2]
	\arrow["{\SIGMA^{\dot{\alpha}}}"{description}, draw=none, from=3, to=4]
	\arrow["{\SIGMA^{\dot{\beta}}}"{description}, between={0.2}{0.8}, no body, from=5, to=6]
	\arrow["{\SIGMA^{\ddot{\beta}}}"{description}, draw=none, from=2-2, to=0]
	\arrow["{\SIGMA^{\dot{\beta}}}"{description}, between={0.2}{0.8}, no body, from=7, to=8]
	\arrow["{\SIGMA^{\dot{\gamma}}}"{description}, draw=none, from=8, to=9]
	\arrow["{\SIGMA^{\ddot{\alpha}}}"{description}, draw=none, from=10, to=11]
  \end{tikzcd}}\]
  Using Rule 4' we have the equalities
  \begin{equation}\label{eq:assoc_cell_eq_1}
   \begin{tikzcd}
	\bullet & \bullet \\
	\bullet & {} \\
	\bullet & \bullet \\
	\bullet & \bullet
	\arrow["r", from=1-1, to=1-2]
	\arrow["g"', from=1-1, to=2-1]
	\arrow[""{name=0, anchor=center, inner sep=0}, "{\dot{h}'}", from=1-2, to=3-2]
	\arrow[""{name=1, anchor=center, inner sep=0}, "{d_2 \tilde{h} \dot{g}}", curve={height=-30pt}, from=1-2, to=4-2]
	\arrow["{\dot{h}}"', from=2-1, to=3-1]
	\arrow[""{name=2, anchor=center, inner sep=0}, draw=none, from=2-2, to=3-2]
	\arrow["{\ddot{r}}", from=3-1, to=3-2]
	\arrow[""{name=3, anchor=center, inner sep=0}, equals, from=3-1, to=4-1]
	\arrow[""{name=4, anchor=center, inner sep=0}, "{d_1}", from=3-2, to=4-2]
	\arrow["d"', from=4-1, to=4-2]
	\arrow["{\SIGMA^{\ddot{\beta}}}"{description}, draw=none, from=2-1, to=0]
	\arrow["\phi", between={0.2}{0.8}, Rightarrow, from=2, to=1]
	\arrow["{\SIGMA^{\delta_1}}"{description}, draw=none, from=3, to=4]
   \end{tikzcd}
   \quad = \quad
   \begin{tikzcd}
	\bullet & \bullet \\
	\bullet & \bullet \\
	\bullet & \bullet \\
	\bullet & \bullet
	\arrow["r", from=1-1, to=1-2]
	\arrow[""{name=0, anchor=center, inner sep=0}, "g"', from=1-1, to=2-1]
	\arrow[""{name=1, anchor=center, inner sep=0}, "{\dot{g}}", from=1-2, to=2-2]
	\arrow["{\dot{r}}", from=2-1, to=2-2]
	\arrow[""{name=2, anchor=center, inner sep=0}, "{\dot{h}}"', from=2-1, to=3-1]
	\arrow[""{name=3, anchor=center, inner sep=0}, "{\tilde{h}}", from=2-2, to=3-2]
	\arrow["{\tilde{r}}", from=3-1, to=3-2]
	\arrow[""{name=4, anchor=center, inner sep=0}, equals, from=3-1, to=4-1]
	\arrow[""{name=5, anchor=center, inner sep=0}, "{d_2}", from=3-2, to=4-2]
	\arrow["d"', from=4-1, to=4-2]
	\arrow["{\SIGMA^{\dot{\alpha}}}"{description}, draw=none, from=0, to=1]
	\arrow["{\SIGMA^{\dot{\gamma}}}"{description}, draw=none, from=2, to=3]
	\arrow["{\SIGMA^{\delta_2}}"{description}, draw=none, from=4, to=5]
   \end{tikzcd}
  \end{equation}
  and
  \begin{equation}\label{eq:assoc_cell_eq_2}
   \begin{tikzcd}
	\bullet & \bullet & \bullet \\
	\bullet & \bullet & \bullet \\
	\bullet && \bullet
	\arrow["s", from=1-1, to=1-2]
	\arrow[""{name=0, anchor=center, inner sep=0}, "h"', from=1-1, to=2-1]
	\arrow["{{\dot{r}}}", from=1-2, to=1-3]
	\arrow[""{name=1, anchor=center, inner sep=0}, "{{\dot{h}}}"', from=1-2, to=2-2]
	\arrow[""{name=2, anchor=center, inner sep=0}, "{{\tilde{h}}}", from=1-3, to=2-3]
	\arrow[""{name=3, anchor=center, inner sep=0}, "{{e_2 \ddot{h}}}", curve={height=-30pt}, from=1-3, to=3-3]
	\arrow["{{\dot{s}}}"', from=2-1, to=2-2]
	\arrow[""{name=4, anchor=center, inner sep=0}, equals, from=2-1, to=3-1]
	\arrow["{{\tilde{r}}}"', from=2-2, to=2-3]
	\arrow[""{name=5, anchor=center, inner sep=0}, "{{e_1}}", from=2-3, to=3-3]
	\arrow["e"', from=3-1, to=3-3]
	\arrow["{{\SIGMA^{\dot{\beta}}}}"{description}, draw=none, from=0, to=1]
	\arrow["{{\SIGMA^{\dot{\gamma}}}}"{description}, draw=none, from=1, to=2]
	\arrow["{{\SIGMA^{\epsilon_1}}}"{description}, draw=none, from=4, to=5]
	\arrow["\psi", between={0}{0.8}, Rightarrow, from=2-3, to=3]
   \end{tikzcd}
   \quad = \quad
   \begin{tikzcd}
	\bullet & \bullet & \bullet \\
	\bullet && \bullet \\
	\bullet && \bullet \rlap{\;.}
	\arrow["s", from=1-1, to=1-2]
	\arrow[""{name=0, anchor=center, inner sep=0}, "h"', from=1-1, to=2-1]
	\arrow["{\dot{r}}", from=1-2, to=1-3]
	\arrow[""{name=1, anchor=center, inner sep=0}, "{\ddot{h}}", from=1-3, to=2-3]
	\arrow["{\dot{r}'}"', from=2-1, to=2-3]
	\arrow[""{name=2, anchor=center, inner sep=0}, equals, from=2-1, to=3-1]
	\arrow[""{name=3, anchor=center, inner sep=0}, "{e_2}", from=2-3, to=3-3]
	\arrow["e"', from=3-1, to=3-3]
	\arrow["{\SIGMA^{\ddot{\alpha}}}"{description}, draw=none, from=0, to=1]
	\arrow["{\SIGMA^{\epsilon_2}}"{description}, draw=none, from=2, to=3]
   \end{tikzcd}
  \end{equation}
  \Cref{eq:assoc_cell_eq_1} and \cref{eq:assoc_cell_eq_2} give, respectively,  the $\Omega$ 2-cells
  \[\begin{tikzcd}
	\bullet && \bullet & \bullet & \bullet \\
	&& \bullet & \bullet & \bullet \\
	\bullet & \bullet & \bullet & \bullet & \bullet
	\arrow["{\dot{h}'}", from=1-1, to=1-3]
	\arrow[equals, from=1-1, to=3-1]
	\arrow[""{name=0, anchor=center, inner sep=0}, "{d_1}"', from=1-3, to=2-3]
	\arrow["\phi"', between={0.3}{0.7}, Rightarrow, from=1-3, to=3-1]
	\arrow["{\ddot{r}}"', from=1-4, to=1-3]
	\arrow[""{name=1, anchor=center, inner sep=0}, equals, from=1-4, to=2-4]
	\arrow["{\dot{s}}"', from=1-5, to=1-4]
	\arrow[""{name=2, anchor=center, inner sep=0}, equals, from=1-5, to=2-5]
	\arrow["d"', from=2-4, to=2-3]
	\arrow["{\dot{s}}"', from=2-5, to=2-4]
	\arrow["{\dot{g}}"', from=3-1, to=3-2]
	\arrow["{\tilde{h}}"', from=3-2, to=3-3]
	\arrow[""{name=3, anchor=center, inner sep=0}, "{d_2}", from=3-3, to=2-3]
	\arrow[""{name=4, anchor=center, inner sep=0}, equals, from=3-4, to=2-4]
	\arrow["{\tilde{r}}", from=3-4, to=3-3]
	\arrow[""{name=5, anchor=center, inner sep=0}, equals, from=3-5, to=2-5]
	\arrow["{\dot{s}}", from=3-5, to=3-4]
	\arrow["{\SIGMA^{\delta_1}}"{description}, between={0.2}{0.8}, no body, two heads, from=0, to=1]
	\arrow["{\SIGMA^\id}"{description}, draw=none, from=1, to=2]
	\arrow["{\SIGMA^{\delta_2}}"{description}, draw=none, from=3, to=4]
	\arrow["{\SIGMA^\id}"{description}, draw=none, from=4, to=5]
  \end{tikzcd}
  \;\;\;\;
  \text{and}
  \;\;\;\;
  \begin{tikzcd}
	\bullet & \bullet & \bullet & \bullet & \bullet \\
	&& \bullet && \bullet \\
	\bullet & \bullet & \bullet && {\bullet \rlap{\;.}}
	\arrow["{\dot{g}}", from=1-1, to=1-2]
	\arrow[equals, from=1-1, to=3-1]
	\arrow["{\tilde{h}}", from=1-2, to=1-3]
	\arrow["\id"', between={0.3}{0.7}, Rightarrow, from=1-2, to=3-1]
	\arrow[equals, from=1-2, to=3-2]
	\arrow[""{name=0, anchor=center, inner sep=0}, "{e_1}", from=1-3, to=2-3]
	\arrow["\psi"', between={0.3}{0.7}, Rightarrow, from=1-3, to=3-2]
	\arrow["{\tilde{r}}"', from=1-4, to=1-3]
	\arrow["{\dot{s}}"', from=1-5, to=1-4]
	\arrow[""{name=1, anchor=center, inner sep=0}, equals, from=1-5, to=2-5]
	\arrow["e"', from=2-5, to=2-3]
	\arrow["{\dot{g}}"', from=3-1, to=3-2]
	\arrow["{\ddot{h}}"', from=3-2, to=3-3]
	\arrow[""{name=2, anchor=center, inner sep=0}, "{e_2}"', from=3-3, to=2-3]
	\arrow[""{name=3, anchor=center, inner sep=0}, equals, from=3-5, to=2-5]
	\arrow["{\dot{r}'}", from=3-5, to=3-3]
	\arrow["{\SIGMA^{\epsilon_1}}"{description}, draw=none, from=0, to=1]
	\arrow["{\SIGMA^{\epsilon_2}}"{description}, draw=none, from=2, to=3]
  \end{tikzcd}\]
  We can then compose these using
  \begin{equation}\label{eq:assoc_eq_3}
   \begin{tikzcd}
	\bullet & \bullet & \bullet \\
	\bullet & \bullet & \bullet \\
	\bullet && \bullet
	\arrow["{\dot{s}}", from=1-1, to=1-2]
	\arrow[""{name=0, anchor=center, inner sep=0}, equals, from=1-1, to=2-1]
	\arrow["{\tilde{r}}", from=1-2, to=1-3]
	\arrow[""{name=1, anchor=center, inner sep=0}, equals, from=1-2, to=2-2]
	\arrow[""{name=2, anchor=center, inner sep=0}, "{d_2}", from=1-3, to=2-3]
	\arrow[""{name=3, anchor=center, inner sep=0}, "{e_1 z_2}", curve={height=-24pt}, from=1-3, to=3-3]
	\arrow["{\dot{s}}"', from=2-1, to=2-2]
	\arrow[""{name=4, anchor=center, inner sep=0}, equals, from=2-1, to=3-1]
	\arrow["d"', from=2-2, to=2-3]
	\arrow[""{name=5, anchor=center, inner sep=0}, "{z_1}", from=2-3, to=3-3]
	\arrow["z"', from=3-1, to=3-3]
	\arrow["{\SIGMA^\id}"{description}, draw=none, from=0, to=1]
	\arrow["{\SIGMA^{\delta_2}}"{description}, draw=none, from=1, to=2]
	\arrow["{\SIGMA^{\zeta_1}}"{description}, draw=none, from=4, to=5]
	\arrow["\theta", between={0}{0.8}, Rightarrow, from=2-3, to=3]
   \end{tikzcd}
   \quad = \quad
   \begin{tikzcd}
	\bullet & \bullet & \bullet \\
	\bullet && \bullet \\
	\bullet && \bullet
	\arrow["{\dot{s}}", from=1-1, to=1-2]
	\arrow[""{name=0, anchor=center, inner sep=0}, equals, from=1-1, to=2-1]
	\arrow["{\tilde{r}}", from=1-2, to=1-3]
	\arrow[""{name=1, anchor=center, inner sep=0}, "{e_1}", from=1-3, to=2-3]
	\arrow["e"', from=2-1, to=2-3]
	\arrow[""{name=2, anchor=center, inner sep=0}, equals, from=2-1, to=3-1]
	\arrow[""{name=3, anchor=center, inner sep=0}, "{z_2}", from=2-3, to=3-3]
	\arrow["z"', from=3-1, to=3-3]
	\arrow["{\SIGMA^{\epsilon_1}}"{description}, draw=none, from=0, to=1]
	\arrow["{\SIGMA^{\zeta_2}}"{description}, draw=none, from=2, to=3]
   \end{tikzcd}
  \end{equation}
  to give (after whiskering with $f$ and $t$)

  \begin{equation}\label{eq:Assoc}
   \adjustbox{scale=0.9}{\begin{tikzcd}
	\bullet & \bullet &&& \bullet & \bullet & \bullet & \bullet \\
	&&&& \bullet & \bullet & \bullet \\
	& \bullet & \bullet & \bullet & \bullet && \bullet & \bullet \\
	&&&& \bullet && \bullet \\
	\bullet & \bullet && \bullet & \bullet && \bullet & \bullet \rlap{\;.}
	\arrow["f", from=1-1, to=1-2]
	\arrow[equals, from=1-1, to=5-1]
	\arrow["{\dot{h}'}", from=1-2, to=1-5]
	\arrow[equals, from=1-2, to=3-2]
	\arrow[""{name=0, anchor=center, inner sep=0}, "{d_1}"', from=1-5, to=2-5]
	\arrow["{z_1 \phi}"', between={0.3}{0.7}, Rightarrow, from=1-5, to=3-2]
	\arrow["{\ddot{r}}"', from=1-6, to=1-5]
	\arrow[""{name=1, anchor=center, inner sep=0}, equals, from=1-6, to=2-6]
	\arrow["{\dot{s}}"', from=1-7, to=1-6]
	\arrow[""{name=2, anchor=center, inner sep=0}, equals, from=1-7, to=2-7]
	\arrow["t"', from=1-8, to=1-7]
	\arrow[equals, from=1-8, to=3-8]
	\arrow[""{name=3, anchor=center, inner sep=0}, "{z_1}", from=2-5, to=3-5]
	\arrow["d", from=2-6, to=2-5]
	\arrow["{\dot{s}}", from=2-7, to=2-6]
	\arrow[""{name=4, anchor=center, inner sep=0}, equals, from=2-7, to=3-7]
	\arrow["{\dot{g}}", from=3-2, to=3-3]
	\arrow["{z_2 \psi \dot{g}}"', between={0.3}{0.7}, Rightarrow, from=3-2, to=5-5]
	\arrow["{\tilde{h}}", from=3-3, to=3-4]
	\arrow[""{name=5, anchor=center, inner sep=0}, "{z_1 d_2}", curve={height=-12pt}, from=3-4, to=3-5]
	\arrow[""{name=6, anchor=center, inner sep=0}, "{z_2 e_1}"', curve={height=12pt}, from=3-4, to=3-5]
	\arrow["z", from=3-7, to=3-5]
	\arrow["t"', from=3-8, to=3-7]
	\arrow[""{name=7, anchor=center, inner sep=0}, "{z_2}"', from=4-5, to=3-5]
	\arrow[""{name=8, anchor=center, inner sep=0}, equals, from=4-7, to=3-7]
	\arrow["e", from=4-7, to=4-5]
	\arrow["f"', from=5-1, to=5-2]
	\arrow[equals, from=5-2, to=3-2]
	\arrow["{\dot{g}}"', from=5-2, to=5-4]
	\arrow["{\ddot{h}}"', from=5-4, to=5-5]
	\arrow[""{name=9, anchor=center, inner sep=0}, "{e_2}", from=5-5, to=4-5]
	\arrow[""{name=10, anchor=center, inner sep=0}, equals, from=5-7, to=4-7]
	\arrow["{\dot{r}'}", from=5-7, to=5-5]
	\arrow[equals, from=5-8, to=3-8]
	\arrow["t", from=5-8, to=5-7]
	\arrow["{\SIGMA^{\delta_1}}"{description}, draw=none, from=0, to=1]
	\arrow["{\SIGMA^\id}"{description}, draw=none, from=1, to=2]
	\arrow["{\SIGMA^{\zeta_1}}"{description}, draw=none, from=3, to=4]
	\arrow["\theta"', between={0.2}{0.8}, Rightarrow, from=5, to=6]
	\arrow["{\SIGMA^{\zeta_2}}"{description}, draw=none, from=7, to=8]
	\arrow["{\SIGMA^{\epsilon_2}}"{description}, draw=none, from=9, to=10]
   \end{tikzcd}}
  \end{equation}
  This is $\Assoc_{\bar{h},\bar{g},\bar{h}}$.
  Now we apply $H$ to this 2-cell. Expressing the resulting 2-cell in terms of string diagrams we arrive at

  \begin{equation}
   \adjustbox{scale=0.85}{\input{H_assoc.tikz}}
  \end{equation}
  where $(F(\zeta_2 \odot \epsilon_2) \circ F(t)_*)_*^{-1}$ is given by
  \begin{equation}
   \adjustbox{scale=0.8}{\input{H_assoc_box.tikz}}
  \end{equation}

  Getting rid of some bends and cancelling some compositors, the right-hand side of \cref{eq:compositor_assoc} becomes \eqref{eq:assoc_a} below. Then, composing both sides of the desired equality with (inverses of) compositors (and their mates), using triangle identites, and pealing off the functors $F(f)$ and $F(t)_*$ from either side, we find it suffices to show the equality \eqref{eq:compositor_assoc2} on the next page.%
  \begin{equation}\label{eq:assoc_a}
   \adjustbox{scale=0.8}{\input{assoc_RHS2.tikz}}
  \end{equation}

  \begin{equation}\label{eq:compositor_assoc2}
   \adjustbox{scale=0.8}{\input{assoc_LHS2.tikz}}
   \!\!\quad=\quad\!\!
   \adjustbox{scale=0.7}{\input{assoc_RHS3.tikz}}
  \end{equation}

  Now recall \cref{eq:assoc_cell_eq_1,eq:assoc_cell_eq_2,eq:assoc_eq_3}. Applying $F$ to these we obtain the following string diagrams.
  \begin{equation}\label{eq:assoc_cell_eq_1_diag}
   \input{F_assoc_cell_1_LHS.tikz}
   \quad=\quad
   \input{F_assoc_cell_1_RHS.tikz}
  \end{equation}
  \begin{equation}\label{eq:assoc_cell_eq_2_diag}
   \adjustbox{scale=1.1}{\input{F_assoc_cell_2_LHS.tikz}}
   \quad=\quad
   \adjustbox{scale=1.1}{\input{F_assoc_cell_2_RHS.tikz}}
  \end{equation}
  \begin{equation}\label{eq:assoc_cell_eq_3_diag}
   \adjustbox{scale=1.1}{\input{F_assoc_cell_3_LHS.tikz}}
   \quad=\quad
   \adjustbox{scale=1.1}{\input{F_assoc_cell_3_RHS.tikz}}
  \end{equation}

  We can use these to simplify the right-hand side of \cref{eq:compositor_assoc2}. Applying \cref{eq:assoc_cell_eq_1_diag} we have
  \ctikzfig{assoc_RHS4}

  Next using \cref{eq:assoc_cell_eq_3_diag} we find
  \ctikzfig{assoc_RHS5}

  Now using \cref{eq:assoc_cell_eq_2_diag} we obtain
  \ctikzfig{assoc_RHS6}

  Finally, (after using a triangle identity) we notice that the part highlighted below is the mate of $F(\zeta_2 \cdot \epsilon_2)$ and hence will cancel with the box.
  \ctikzfig{assoc_RHS7}

  Now after combining the compositors we find this is precisely the left-hand side of \cref{eq:compositor_assoc2} and so we have shown the desired equality. Hence, the associativity coherence law holds.

  Thus, we have shown that $H$ is a pseudofunctor.
  It is then easy to see that $F = H \circ P_\Sigma$. %
  The hardest part is checking that the compositor of $H \circ P_\Sigma$ agrees with that of $F$. This holds since the 2-cell $\dot{\beta}$ in the canonical $\Sigma$-square in \cref{eq:compositor_beta_dot_append} for $\gamma^H_{(g,1),(f,1)}$ %
  is chosen to be the identity and due to strict unitarity of the bicategories involved.
\end{proof}
\begin{proof}[Proof of (b)]
  Now we show the `2-dimensional' universality condition. Assume $H, H'\colon \catx[\Sigma_*] \to \catc$ are strictly unitary pseudofunctors and $\xi\colon H \circ P_\Sigma \to H' \circ P_\Sigma$ is a pseudonatural transformation for which the pseudonaturality squares involving $\xi_r$ are BC squares whenever $r \in \Sigma$. We define $\upsilon\colon H \to H'$ by
  \begin{itemize}
   \item $\upsilon_X = \xi_X$ for objects $X \in \catx$,
   \item $\upsilon_{(f,r)}$ for $(f,r) \in \catx[\Sigma_*]$ is given by the composite
   \[\begin{tikzcd}
	{H(A)} & {H(I)} & {H(B)} \\
	{H'(A)} & {H'(I)} & {H'(B)}
	\arrow["{H(f)}", from=1-1, to=1-2]
	\arrow["{\xi_A}"', from=1-1, to=2-1]
	\arrow["{H(r)_*}", from=1-2, to=1-3]
	\arrow["{\xi_I}", from=1-2, to=2-2]
	\arrow["{\xi_B}", from=1-3, to=2-3]
	\arrow["{\xi_f}"', shorten <=4pt, shorten >=4pt, Rightarrow, from=2-1, to=1-2]
	\arrow["{H'(f)}"', from=2-1, to=2-2]
	\arrow["{(\xi_r)_*^{-1}}"', shift left, shorten <=4pt, shorten >=4pt, Rightarrow, from=2-2, to=1-3]
	\arrow["{H'(r)_*}"', from=2-2, to=2-3]
   \end{tikzcd}\]
   where we write $H(f)$ for $H((f,\id)) = H(P_\Sigma(f))$ and
   where $(\xi_r)_*^{-1}$ denotes the inverse of the mate of $\xi_r$. (Note that this inverse exists by the BC assumption.)
  \end{itemize}

  We start by showing $\upsilon$ is a pseudonatural transformation. The unit condition follows from the unit condition for $\xi$ and strict unitarity (so that $\xi_{1_X} = \id_{\xi_X}$).

    The compositor condition is
  \begin{equation}\label{eq:pseudonatural_mult_cond}
   \input{pseudonaturality_mult_LHS.tikz}
   \qquad=\qquad
   \input{pseudonaturality_mult_RHS.tikz}
  \end{equation}
  Expanding the definitions this desired equality becomes
  \begin{equation}\label{eq:pseudonatural_mult_cond_2}
   \adjustbox{scale=0.9}{\input{pseudonaturality_mult_LHS2.tikz}}
   \quad=\quad
   \adjustbox{scale=0.9}{\input{pseudonaturality_mult_RHS2.tikz}}
  \end{equation}

  Now we can compose both sides with the mates of $\xi_{\dot{r}s}$, $\xi_r$ and $\xi_s$, and aslso $\xi_f^{-1}$ and $\xi_g^{-1}$), as appropriate in order to cancel the boxes, yielding (after using a triangle identity on the right-hand side).
  \begin{equation}\label{eq:pseudonatural_mult_cond_3}
   \adjustbox{scale=0.7}{\input{pseudonaturality_mult_LHS3.tikz}}
   \ =\
   \adjustbox{scale=0.75}{\input{pseudonaturality_mult_RHS3.tikz}}
  \end{equation}

  To show this equality we will simplify the left-hand side. First we can use the compositor condition for pseudonaturality of $\xi$ to move $\xi_{\dot{g}f}$ past the compositor and partially cancel it with $\xi_f^{-1}$. Then we can add in $\xi_r$ and its inverse near the bottom right of the diagram.

  \ctikzfig{pseudonaturality_mult_LHS4}

  Using a triangle identity on the left and the compositor pseudonaturality condition of $\xi$ on the bottom right we then have

  \ctikzfig{pseudonaturality_mult_LHS5}

  Next we can use the compositor condition pseudonatural transformation $\xi$ we can move $\xi_g^{-1}$ and $\xi_{\dot{r}}^{-1}$ up past the compositor, then using the 2-dimensional naturality condition for $\xi$ we can move $\xi_{\dot{r}g}^{-1}$ past $H'(\dot{\beta})$.
  Finally, moving past the compositor again we can cancel with the $\xi_r$ and $\xi_{\dot{g}}$. In this way we arrive at the string diagram from the right-hand side and thus the equality holds.

  The final condition to show pseudonaturality is the 2-dimensional naturality condition.
  \begin{equation}\label{eq:pseudonatural_nat_cond}
   \input{pseudonaturality_nat_LHS.tikz}
   \qquad=\qquad
   \input{pseudonaturality_nat_RHS.tikz}
  \end{equation}

  Expanding this out, we see that the desired equality becomes
  \begin{equation}\label{eq:pseudonatural_nat_cond2}
   \input{pseudonaturality_nat_LHS2.tikz}
   \qquad=\qquad
   \input{pseudonaturality_nat_RHS2.tikz}
  \end{equation}

  As before we now compose both sides with isomorphisms such as $(\xi_{r_2})_*$, $\xi_{f_1}^{-1}$, $(\xi_{r_1})_*$ and $H'(\delta_2)_*$ so as to cancel with most of the boxes (and apply a triangle identity on the right-hand side). We are left with
  \begin{equation}\label{eq:pseudonatural_nat_cond3}
   \input{pseudonaturality_nat_LHS5.tikz}
   \quad=\quad
   \input{pseudonaturality_nat_RHS5.tikz}
  \end{equation}
  Again we can use the pseudonaturality of $\xi$ on the left-hand side to `pull the $\xi$ wire' to the left of the diagram.
  \begin{equation}\label{eq:pseudonatural_nat_cond4_LHS}
   \input{pseudonaturality_nat_LHS6.tikz}
  \end{equation}
  Similarly, on the right-hand side we can move the $\xi$ wire and then use a triangle inequality to obtain
  \begin{equation}\label{eq:pseudonatural_nat_cond4_RHS}
   \input{pseudonaturality_nat_RHS6.tikz}
  \end{equation}
  The highlighted segment is the mate of $H(\delta_2)$ and hence cancels with the box. After using a triangle equality, this gives exactly the diagram corresponding the the left-hand side found in \cref{eq:pseudonatural_nat_cond4_LHS}.

  Thus, we have shown the 2-dimensional naturality condition and $\upsilon$ is indeed a pseudonatural transformation.
  Simply by expanding the definition it is now also easy to see that $\xi = \upsilon \circ P_\Sigma$.
\end{proof}

\begin{proof}[Proof of (c)]
  Finally, we show the `3-dimensional' universality condition. For data as in the statement of the theorem we simply set $\beth_X = \aleph_X$.

 Let us show that $\beth$ is indeed a modification. The necessary naturality condition on $\beth$ holds for 1-morphisms of the form $(f,1)$ by same condition on $\aleph$. It then suffices to show the condition for 1-morphisms of the form $(1,r)$, since every morphism of $\catx[\Sigma_*]$ is a composite of these two kinds.

 We want to show
 \begin{equation}\label{eq:modification_at_r_adjoint}
   \input{modification_cond_LHS.tikz}
   \qquad=\qquad
   \input{modification_cond_RHS.tikz}
 \end{equation}
 Taking inverses, this is equivalent to
 \begin{equation}\label{eq:modification_at_r_adjoint2}
   \input{modification_cond_LHS2.tikz}
   \qquad=\qquad
   \input{modification_cond_RHS2.tikz}
 \end{equation}
 Now from \cref{lem:transformation_at_adjoint} we know that the inverse of $\upsilon_{r_*}$ is the mate of $\upsilon_{r}$ (and similarly for $\upsilon'_{r}$).
 Replacing $(\upsilon_{r_*})^{-1}$ and $(\upsilon'_{r_*})^{-1}$ with $(\upsilon_{r})_*$ and $(\upsilon'_{r})_*$  in the diagram above we arrive at
 \begin{equation}\label{eq:modification_at_r_adjoint3}
   \input{modification_cond_LHS3.tikz}
   \qquad=\qquad
   \input{modification_cond_RHS3.tikz}
 \end{equation}
 and these are equal since the original $\aleph$ is a modification.

 Hence $\beth$ is a modification. It is also clearly the only modification $\beth$ such that $\aleph = \beth \circ P_\Sigma$.
\end{proof}

\bibliographystyle{abbrv}
\bibliography{references}

\vspace{2mm}

\begin{tabular}{p{6.7cm}p{6.7cm}}

\noindent Graham \textsc{Manuell} \newline
Stellenbosch University\newline
Stellenbosch, South Africa\newline
\href{mailto:graham@manuell.me}{\sf graham@manuell.me}

&

\noindent Lurdes \textsc{Sousa} \newline
Polythecnic of Viseu, Portugal \newline
CMUC, University of Coimbra, Portugal\newline
\href{mailto:sousa@estv.ipv.pt}{\sf sousa@estv.ipv.pt}

\end{tabular}

\end{document}